%% file: main.tex
\documentclass[10pt, reqno]{amsart}

\usepackage{mathdots}
\usepackage{fancyhdr, lastpage, ifpdf}
\usepackage{yhmath}
\usepackage{enumitem}
\usepackage{mathrsfs}
\usepackage{bbm}
\usepackage{mathtools}
\usepackage{tensor}
\usepackage{ulem}
\usepackage{comment}
\usepackage{booktabs}

\usepackage{multirow}
\usepackage{array}

\usepackage[colorlinks=true, allcolors=blue]{hyperref}

\if@aom@screen@mode
  \hypersetup{breaklinks=true, linkcolor=blue, citecolor=blue, urlcolor=blue}
\else\if@aom@manuscript@mode
  \hypersetup{linkcolor=blue, citecolor=blue, urlcolor=blue}
\else
  \hypersetup{linkcolor=black, citecolor=black, urlcolor=black}
\fi\fi

\usepackage{environ}

\usepackage{tikz}
\usetikzlibrary{matrix, fit, arrows, backgrounds}
\usetikzlibrary{calc}
\usetikzlibrary{cd,positioning}
\tikzset{|/.tip={Bar[width=.8ex,round]}}
\usetikzlibrary{bending}
\usetikzlibrary{hobby}
\usetikzlibrary{decorations.pathreplacing}
\usepackage[english]{babel}
\usepackage{setspace}

\usepackage[letterpaper,top=2cm,bottom=2.25cm,left=2.25cm,right=2cm,marginparwidth=1.75cm]{geometry}

\usepackage{amsmath}
\usepackage{amssymb}
\usepackage{graphicx}
\usepackage{float}
\usepackage{caption}
\usepackage[export]{adjustbox}
\usepackage{parskip}
\usepackage{subcaption}
\usetikzlibrary{shapes.geometric}
\usetikzlibrary{patterns}
\usetikzlibrary{patterns.meta}
\usetikzlibrary{hobby}
\usetikzlibrary{shapes.symbols,shapes.misc}
\usepackage{adjustbox}
\usepackage{extarrows}
\usepackage{ragged2e}
\usepackage{enumitem}
\usepackage{nicematrix}
\usepackage{hyperref}

\usepackage{etoolbox}
\makeatletter
\patchcmd{\@sect}{\@svsec}{\color{blue}\@svsec}{}{}
\makeatother

\hypersetup{
    colorlinks=true,
    linkcolor=blue,   
    citecolor=blue,
    urlcolor=blue,
    linktoc=page      
}

\newcommand{\sh}[1]{\mathscr{#1}}
\newcommand{\fl}[1]{\mathcal{#1}^{\,\bullet}}
\newcommand{\ccs}[2]{H^{\bullet}(\mathcal{S}h_{#1}(\Lambda(#2), \mathbb{K})_{0}) }
\newcommand{\lk}[1]{\Lambda(#1)}
\newcommand{\sfl}{\mathcal{F}^{\bullet}_{\mathrm{std}}}
\newcommand{\asfl}{\mathcal{F}^{\bullet}_{\mathrm{astd}}}

\newcommand{\flag}[1]{ \fl{#1}=\big\{ {#1}^{(0)} \subset \cdots \subset {#1}^{(n)}\big\}  }

\newcommand{\flagstd}[2]{ \fl{#1}_{\mathrm{std}}\,[\hat{\mathbf{#2}}]:=\big\{ {#1}^{(0)}_{\mathrm{std}} \subset \cdots \subset {#1}^{(n)}_{\mathrm{std}}\big\} }
\newcommand{\flagastd}[2]{ \fl{#1}_{\mathrm{astd}}\,[\hat{\mathbf{#2}}]:=\big\{ {#1}^{(0)}_{\mathrm{astd}} \subset \cdots \subset {#1}^{(n)}_{\mathrm{astd}}\big\} }

\newcommand{\elmg}[3]{ \tensor[_{\gamma_{\sh{#2}}}]{ \langle \smash{\gamma_{\scriptscriptstyle \sh{#3}}} \rangle }{_{\gamma_{\sh{#1}}}} }
\newcommand{\be}[3]{ \tensor[_{#2}]{ \langle \smash{#3} \rangle }{_{#1 } } }
\newcommand{\elmG}{\tensor[_{\gamma_{\sh{G}}}]{ \langle \smash{\gamma_{\scriptscriptstyle \sh{H}}} \rangle }{_{\gamma_{\sh{F}}}}}

\newenvironment{myequ*}{\begin{equation*}\textstyle}{\end{equation*}}

\newtheorem{theorem}{Theorem}
\newtheorem{lemma}[theorem]{Lemma}
\newtheorem{claim}[theorem]{Claim}

\newtheorem{proposition}[theorem]{Proposition}

\newtheorem{definition}[theorem]{Definition}
\newtheorem{remark}[theorem]{Remark}
\newtheorem{notation}[theorem]{Notation}

\newtheorem{construction}[theorem]{Construction}

\newtheorem{observation}[theorem]{Observation}
\newtheorem{setup}[theorem]{Setup}

\numberwithin{equation}{section}

\usepackage{amsthm}

\newtheorem{maintheorem}{Theorem}         

\DeclareMathAlphabet{\mathbbold}{U}{bbold}{m}{n}

\makeatletter
\renewenvironment{bmatrix}
{\left[\mkern3mu\env@matrix}
{\endmatrix\mkern3mu\right]}
\renewcommand{\arraystretch}{0.8}

\usepackage{parskip}

\setitemize{itemsep=25pt}
\setitemize{labelsep=25pt}
\setitemize{leftmargin=25pt}

\usepackage{afterpage}

\allowdisplaybreaks

\usepackage{mdframed}
\usepackage{xcolor}

\hypersetup{
    colorlinks=true,
    linkcolor=blue,
    citecolor=blue,
    urlcolor=blue
}

\definecolor{myblue}{RGB}{33,85,205}

\newmdtheoremenv[
  linewidth=1pt,
  roundcorner=4pt,
  linecolor=myblue,
  backgroundcolor=white, 
  innerleftmargin=10pt,
  innerrightmargin=10pt,
  innertopmargin=8pt,
  innerbottommargin=8pt,
]{blueboxthm}{Theorem}

\author{\'Angel Rodr\'iguez--L\'opez}
\address{University of California Davis, Dept.~of Mathematics, USA}
\email{arodriguezlopez@ucdavis.edu}

\title{Computable Sheaf Invariants for Legendrian Rainbow Closures}

\begin{document}

\fontdimen14\textfont2=6pt
\fontdimen16\textfont2=2pt
\fontdimen17\textfont2=3pt

\setlength{\parindent}{24pt}
\setlength{\parskip}{8pt}
\setlength{\abovedisplayskip}{5pt}
\setlength{\belowdisplayskip}{5pt}
\setlength{\abovedisplayshortskip}{5pt}
\setlength{\belowdisplayshortskip}{5pt}
\setlength{\abovecaptionskip}{10pt}

\begin{abstract}
For any Legendrian link in $\displaystyle \mathbb{R}^{3}$ given by the rainbow closure of a positive braid word, we develop an explicit and computable description of a Legendrian isotopy invariant associated with it, namely the cohomological category of compactly supported, microlocal rank-one sheaves with singular support on the Legendrian link. In particular, we parametrize the objects of the category by points of a braid variety, and for any pair of objects, we provide a linear map that algebraically characterizes their possible non-trivial graded morphism spaces. In addition, we provide combinatorial rules governing the compositions of graded morphisms in the category under consideration. Finally, we present several applications of our results, highlighting the structural features captured by the categorical invariant of interest.
\end{abstract}

\maketitle
\tableofcontents

\input{sec1}
\input{sec2}
\input{sec3}
\input{sec4}
\input{sec5}
\input{sec6}
\input{appendix}

\bibliographystyle{amsalpha}   
\bibliography{references}

\end{document}

%% file: sec1.tex
\section{Introduction}\label{sec:introduction}

\noindent
In this article, we investigate a broad family of Legendrian links in the standard contact three-dimensional space through the lens of microlocal sheaf theory. More precisely, for any Legendrian link obtained as the rainbow closure of a positive braid word, we analyze in depth a categorical invariant associated with it, namely a category of sheaves whose singular support lies on the Legendrian link under study. The main contribution of this work is the development of an explicit sheaf-theoretic framework that, for any Legendrian link in the family of interest, provides an algebraic, combinatorial, and computable characterization of the objects, morphisms, and compositions of its associated category. Furthermore, to illustrate our methods, we present several examples demonstrating the range and effectiveness of our approach.

\subsection{Scientific Context} In recent years, microlocal sheaf-theoretic methods~\cite{KS1} have emerged as a powerful framework in contact and symplectic topology~\cite{GKS1, G1, NZ1, Nad1, GPS1, JT1}, providing new categorical invariants~\cite{STZ1} and novel approaches to long-standing problems in the theory of Legendrian links~\cite{CG1, CZ1}, thereby revealing rich connections with cluster theory, algebraic geometry, and representation theory~\cite{STZ2, CW1, TZ1, CGGLSS1, CGGS1, CGGS2}. 

For a Legendrian link, the study of its associated microlocal-sheaf categorical invariants has evolved along two main directions. The first explores the relationship between these invariants and those arising from Floer-theoretic constructions~\cite{STZ1, NRSSZ1, CNS1}, while the second focuses on the algebraic and geometric structures of their moduli spaces of objects and their applications to the problem of existence and classification of exact Lagrangian fillings of the underlying Legendrian link~\cite{CG1, CZ1, CW1, CL1, CL2, SW1, HJ1, HJ2}. Despite the substantial progress achieved along these two research directions, several aspects of these categorical invariants remain largely unexplored, particularly those concerning their graded morphism spaces and their compositions. In particular, from the general framework developed by Nadler and Zaslow~\cite{NZ1} and by Nadler~\cite{Nad1}, one can deduce that the structure of the graded morphism spaces in these invariants encodes significant geometric and topological information about the underlying Legendrian link~\cite{GPS1}. For example, in~\cite{STZ1}, Shende, Treumann, and Zaslow studied a concrete microlocal-sheaf categorical invariant for the Chekanov pair, a pair of Legendrian knots with identical classical invariants, and used the structure of the associated graded morphism spaces to show that the two Chekanov knots are not Legendrian isotopic. Apart from this result, the work of Chantraine, Ng, and Sivek on Legendrian $(2,m)$ torus links~\cite{CNS1} constitutes essentially the only study in the literature where the structure of the graded morphism spaces and their compositions has been treated in detail for a microlocal-sheaf categorical invariant.

Bearing this in mind, the main goal in this paper is to investigate in detail a microlocal-sheaf categorical invariant for a broad family of Legendrian links and to shed light on the rich algebraic, combinatorial, geometric, and topological structures encoded by the graded morphism spaces and their compositions within the categorical invariant of interest.

\subsection{Main Results} 
Before stating our main results, we briefly introduce the notation and conventions necessary for their formulation. Let $(x,y,z)$ be coordinates on $\mathbb{R}^{3}$. The standard contact structure $\xi_{\mathrm{std}}$ on $\mathbb{R}^{3}$ is given by $\xi_{\mathrm{std}}:=\mathrm{ker}\,(dz-ydx)$~\cite{EN1}. Let $n\geq 2$ be an integer. We denote by $\mathrm{Br}_{n}$ the braid group on $n$ strands with Artin generators $\sigma_1, \dots, \sigma_{n-1}$, and by $\mathrm{S}_{n}$ its associated Coxeter group---the symmetric group on $n$ elements---with generators $s_{1}, \dots, s_{n-1}$, the adjacent transpositions~\cite{CGGS1}. Accordingly, we denote by $\mathrm{Br}^{+}_{n}\subset \mathrm{Br}_{n}$ the monoid generated by the positive powers of the Artin generators, and call its elements positive braid words. Finally, let $\mathbb{K}$ be a ground field. For any integer $m\geq 1$, we denote by $\mathbb{K}^{m}$ the $m$-dimensional Cartesian vector space over $\mathbb{K}$ with no basis specified, and by $\mathbb{K}^{m}_{\mathrm{std}}$ the same vector space equipped with its standard ordered basis.

Let $\beta \in \mathrm{Br}^{+}_{n}$ be a positive braid word. Following~\cite{KT1, KT2, STZ1}, we associate to it the Legendrian link $\Lambda(\beta)$ in $(\mathbb{R}^{3}, \xi_{\mathrm{std}})$---the rainbow closure of $\beta$---whose front projection $\Pi_{x,z}(\Lambda(\beta)) \subset \mathbb{R}^2$ is depicted in Figure~\ref{Front diagram of the rainbow closure of a braid in n strands}. Subsequently, we denote by $\mathcal{S}h_{1}(\Lambda(\beta), \mathbb{K})_{0}$ the dg--category of microlocal-rank-one, compactly supported, constructible sheaves of $\mathbb{K}$--modules on $\mathbb{R}^{2}$ whose singular support lies on $\Lambda(\beta)$~\cite{STZ1}. By the work of Guillermou, Kashiwara, and Schapira~\cite{GKS1}, this category is an invariant of the Legendrian isotopy class of $\Lambda(\beta)$, and in~\cite{STZ1}, Shende, Treumann, and Zaslow provided an explicit local combinatorial description of its objects. In particular, building on these foundations and the seminal work of Chantraine, Ng, and Sivek~\cite{CNS1}, the main results in this paper lead to a concrete and computable characterization of a categorical invariant of $\Lambda(\beta)$ closely associated with $\mathcal{S}h_{1}(\Lambda(\beta), \mathbb{K})_{0}$, namely its cohomological category $H^{\bullet}\left(\mathcal{S}h_{1}(\Lambda(\beta),\mathbb{K})_{0}\right)$. 

\begin{definition}[Cohomological Category for Rainbow Closures]
Let $\beta \in\mathrm{Br}^{+}_{n}$ be a positive braid word. According to~\cite{STZ1, CNS1}, the category $H^{\bullet}\left(\mathcal{S}h_{1}(\Lambda(\beta),\mathbb{K})_{0}\right)$ admits the following presentation:
\begin{itemize}
\item \textbf{Objects}: The objects of the category $H^{\bullet}\left(\mathcal{S}h_{1}(\Lambda(\beta),\mathbb{K})_{0}\right)$ are those of the category $\mathcal{S}h_{1}(\Lambda(\beta), \mathbb{K})_{0}$. 
\item \textbf{Graded Morphisms}: Let $\sh{F}, \, \sh{G}$ be objects of the category $H^\bullet(\mathcal{S}h_1(\Lambda(\beta), \mathbb{K})_0)$. The graded morphism space associated with the pair $(\sh{F},\sh{G})$ is given by the Yoneda graded vector space:
\begin{equation*}
\mathrm{Ext}^{\bullet}(\sh{F},\sh{G})= \bigoplus_{p\geq0} \mathrm{Ext}^{\,p}(\sh{F},\sh{G})\, .    
\end{equation*} 
\item \textbf{Graded Composition}: Let $\sh{F}, \, \sh{G}, \,\sh{H}$ be objects of the category $H^\bullet(\mathcal{S}h_1(\Lambda(\beta), \mathbb{K})_0)$, and let \,$p,\,q\geq 0$ be integers. The graded composition \,$\circ : \mathrm{Ext}^{\,p}(\sh{G},\sh{H})\times \mathrm{Ext}^{\,q}(\sh{F},\sh{G})\to \mathrm{Ext}^{\,p+q}(\sh{F},\sh{H})$ is given by the Yoneda composition (see, for instance,~\cite{HS1}). 
\end{itemize}
\end{definition}

\noindent
In recent years, the category $H^{\bullet}\left(\mathcal{S}h_{1}(\Lambda(\beta),\mathbb{K})_{0}\right)$ has played a pivotal role in addressing several long-standing problems in symplectic topology and algebraic geometry. For instance, the study of its moduli space of objects has yielded fundamental results on the existence and classification of the exact Lagrangian fillings of $\Lambda(\beta)$~\cite{STZ2, CG1, CL1} and led to the definition of the braid varieties~\cite{CGGS1, CGGS2, CGGLSS1}, thereby revealing deep connections between the theory of Legendrian links, cluster theory, algebraic geometry, and representation theory~\cite{CW1, HJ1, HJ2}.

\begin{definition}[A-type Braid Variety]\label{Def: A-type braid variety}
Let $n\geq 2$ be an integer, and let $\beta=\sigma_{i_{1}}\dots \sigma_{i_{\ell}}\in \mathrm{Br}_{n}^{+}$ be a positive braid word. The braid variety $X(\beta,\mathbb{K})\subset \mathbb{K}^{\ell}_{\mathrm{std}}$ associated with $\beta$ is the affine variety defined by
\begin{equation*}
X(\beta,\mathbb{K}):=\Big\{\,\vec{x}\in \mathbb{K}^{\ell}_{\mathrm{std}} \;\big|\; \text{$P_{\beta}(\vec{x})$ admits an LU decomposition} \,\Big\}\, ,
\end{equation*}
where, for any $\vec{x}=(x_{1},\dots,x_{\ell})\in \mathbb{K}_{\mathrm{std}}^{\ell}$, we denote by $P_{\beta}(\vec{x}) \in \mathrm{GL}(n,\mathbb{K})$ the ordered matrix product of the $n$-dimensional braid matrices associated with the generators of $\beta$ and the entries of $\vec{x}$:
\begin{equation*}
P_{\beta}(\vec{x}) := B^{(n)}_{i_{1}}(x_{1})\,\cdots\,B^{(n)}_{i_{\ell}}(x_{\ell})\, .    
\end{equation*}
Specifically, given an Artin generator $\sigma_{k}\in \mathrm{Br}^{+}_{n}$, with $k\in [1,n-1]$, and a parameter $x\in \mathbb{K}$, the $n$-dimensional braid matrix $B^{(n)}_{k}(x)\in \mathrm{GL}\big(n,\, \mathbb{K}\big)$ associated with $\sigma_{k}$ and $x$ is defined by
\begin{equation*}
\Big[\,B^{(n)}_{k}(x)\,\Big]_{i,j}\,:=\,\begin{cases}
~1\,, & \text{if $~i=j\,$ and $\,i\neq k,\,k+1$}\, ,\\
~1\,, & \text{if $~(i,j)=(k,k+1)\,$ or $\,(k+1,k)$}\, ,\\
~x\,, & \text{if $~i=j=k$}\, ,\\ 
~0\,, & \text{otherwise}\, ,
\end{cases}  \qquad   i,j\in[1, n]\, .
\end{equation*}
\end{definition}

\noindent
In~\cite{STZ1}, Shende, Treumann, and Zaslow showed that the objects of the category $\ccs{1}{\beta}$ can be geometrically described by linear configurations of complete flags in $\mathbb{K}^{n}$ that satisfy certain transversality conditions determined by the positive braid word $\beta$. By the work of Casals, Gorsky, Gorsky, and Simental~\cite{CGGS1, CGGS2}, it is known that such configurations of flags are parametrized by points in the braid variety $X(\beta, \mathbb{K})$. In particular, combining these results yields the following algebraic characterization of the objects of the category $\ccs{1}{\beta}$.

\begin{theorem}[Algebraic Characterization of the Objects,~\cite{STZ1, CGGS1}]\label{Theorem: sheaves as points in the braid variety}
Let $\beta=\sigma_{i_{1}}\cdots\sigma_{i_{\ell}}\in\mathrm{Br}_{n}^{+}$ be a positive braid word, and let $\sh{F}$ be an object of the category $\ccs{1}{\beta}$. Then $\sh{F}$ is algebraically parametrized by a basis \,$\hat{\mathbf{f}}^{(n)}$ for \,$\mathbb{K}^{n}$ and a point $\vec{x}$ in the braid variety $X(\beta,\mathbb{K})$.
\end{theorem}

\noindent
Prior to this work, the only known results regarding the graded morphism spaces and their compositions in the category $\ccs{1}{\beta}$ were those due to Chantraine, Ng, and Sivek~\cite{CNS1}. Specifically, they provided an explicit characterization of the category for Legendrian $(2,m)$ torus links---positive braid words on two strands---and proved the hereditary-type property in full generality.  

\begin{theorem}[Hereditary-Type Property,~\cite{CNS1}]\label{Theorem: hereditary-type property}
Let $\beta \in \mathrm{Br}_n^+$ be a positive braid word, and let $\sh{F}, \sh{G}$ be objects of the category $\ccs{1}{\beta}$. Then, for any $p \geq 2$, 
\begin{equation*}
\mathrm{Ext}^{p}(\sh{F}, \sh{G}) = 0.    
\end{equation*}
\end{theorem}

Building on the above foundations, we now present the first main result of this paper: a theorem providing a precise algebraic description of the lower-degree morphism spaces in the category $\ccs{1}{\beta}$. To this end, the following definition plays a central role.

\begin{definition}[$\delta$--map]\label{Def: delta map}
Let $\beta=\sigma_{i_{1}}\cdots\sigma_{i_{\ell}}\in\mathrm{Br}^{+}_{n}$ be a positive braid word, and let $\sh{F}$, $\sh{G}$ be objects of the category $\ccs{1}{\beta}$. Let \,$\mathbf{\hat{f}}^{(n)}$, \,$\mathbf{\hat{g}}^{(n)}$ be bases for \,$\mathbb{K}^{n}$, and let \,$\vec{x}=(x_{1},\dots, x_{\ell}),\;\vec{y}=(y_{1},\dots, y_{\ell})\in X(\beta,\mathbb{K})$ be points such that the pairs \,$\big(\,\mathbf{\hat{f}}^{(n)},\, \vec{x}\, \big)$ and \,$\big(\,\mathbf{\hat{g}}^{(n)},\, \vec{y}\, \big)$ algebraically characterize $\sh{F}$ and $\sh{G}$ according to Theorem~\eqref{Theorem: sheaves as points in the braid variety}, respectively. Then, we assign to the pair $(\sh{F}, \sh{G})$ the linear map \,$\delta_{\sh{F},\sh{G}}:\mathbb{K}^{n}_{\mathrm{std}}\rightarrow \mathbb{K}^{\ell}_{\mathrm{std}}$\, defined by
\begin{equation*}
\delta_{\sh{F},\sh{G}}(\vec{u}\,):=\big(\delta_{1}(\vec{u}\,),\dots, \delta_{\ell}(\vec{u}\,)\big)\in \mathbb{K}^{\ell}_{\mathrm{std}} \,, \qquad \vec{u}=(u_{1}, \dots, u_{n})\in \mathbb{K}^{n}_{\mathrm{std}}\, ,
\end{equation*}
where, for each $j\in [1,\ell]$,
\begin{equation*}
\delta_j(\vec{u}\,) := \Big[\, \big(B^{(n)}_{i_j}(y_j) \big)^{-1}D\big(\pi_{\beta_{j-1}}(\vec{u}\,)\big) \,B^{(n)}_{i_j}(x_j)\, \Big]_{i_j+1,\,i_j}\, .
\end{equation*}
The components appearing in this formula are defined as follows:
\begin{itemize}
\item $\pi_{\beta_{j-1}}\in \mathrm{S}_{n}$ is the permutation associated with $\beta_{j-1}=\sigma_{i_{1}}\cdots\sigma_{i_{j-1}}\in \mathrm{Br}^{+}_{n}$, the truncation of $\beta$ at the $(j-1)$-th crossing; in particular, $\pi_{\beta_{0}}$ is the trivial permutation. 
\item $\pi_{\beta_{j-1}}(\vec{u}\,):=(u_{\pi_{\beta_{j-1}}(1)},\dots,u_{\pi_{\beta_{j-1}}(n)} )\in \mathbb{K}^{n}_{\mathrm{std}}$ is the vector obtained by permuting the entries of $\vec{u}$ via $\pi_{\beta_{j-1}}$; accordingly, $\pi_{\beta_{0}}(\vec{u}\,)=\vec{u}$.
\item $D(\pi_{\beta_{j-1}}(\vec{u}\,)):=\mathrm{diag}\big\{ u_{\pi_{\beta_{j-1}}(1)},\dots,u_{\pi_{\beta_{j-1}}(n)} \big\}\in \mathrm{M}(n,\mathbb{K})$ is the $n\times n$ diagonal matrix whose diagonal entries are given by $\pi_{\beta_{j-1}}(\vec{u}\,)$. 
\item $B^{(n)}_{i_{j}}(x_{j}),\, B^{(n)}_{i_{j}}(y_{j})\in \mathrm{GL}(n,\mathbb{K})$ are $n$-dimensional braid matrices associated with $\sigma_{i_{j}}$---the $j$-th crossing of $\beta$---and $x_{j},\,y_{j}\in \mathbb{K}$---the $j$-th entries of $\vec{x}$ and $\vec{y}$---respectively. 
\end{itemize}
\end{definition}

\begin{maintheorem}[Algebraic Description of the Lower-Degree Morphism Spaces]\label{Theorem: main result 1}
Let $\beta=\sigma_{i_{1}}\cdots\sigma_{i_{\ell}}\in\mathrm{Br}^{+}_{n}$ be a positive braid word, and let $\sh{F}$, $\sh{G}$ be objects of the category $\ccs{1}{\beta}$. Let \,$\mathbf{\hat{f}}^{(n)},\,\mathbf{\hat{g}}^{(n)}$ be bases for \,$\mathbb{K}^{n}$, and let \,$\vec{x},\, \vec{y}\in X(\beta,\mathbb{K})$ be points such that the pairs $\big(\,\mathbf{\hat{f}}^{(n)},\,\vec{x}\, \big)$ and $\big(\,\mathbf{\hat{g}}^{(n)},\ \vec{y}\, \big)$ algebraically characterize $\sh{F}$ and $\sh{G}$ according to Theorem~\eqref{Theorem: sheaves as points in the braid variety}, respectively.

Following Definition~\eqref{Def: delta map}, let \,$\delta_{\sh{F},\sh{G}}:\mathbb{K}^{n}_{\mathrm{std}}\to \mathbb{K}^{\ell}_{\mathrm{std}}$\, be the linear map associated with the pair $(\sh{F},\sh{G})$. Then there are isomorphisms of vector spaces
\begin{equation*}
\begin{array}{rcl}
\mathrm{Ext}^{0}(\sh{F},\sh{G})&\cong & \mathrm{ker}\, \delta_{\sh{F},\sh{G}}\, ,   \\[6pt]
\mathrm{Ext}^{1}(\sh{F},\sh{G})&\cong & \mathrm{coker}\, \delta_{\sh{F},\sh{G}} \, .
\end{array}
\end{equation*}    
\end{maintheorem}

Having established this, we proceed to present the second main result of this paper: a theorem providing a precise algebraic and combinatorial description of the composition between the lower-degree morphism spaces in the category $\ccs{1}{\beta}$. In particular, the following definition is fundamental to its formulation.

\begin{definition}[Braided Compositions]\label{Def: main braided graded composition}
Let $\beta = \sigma_{i_1} \cdots \sigma_{i_\ell} \in \mathrm{Br}_n^+$ be a positive braid word. We introduce three bilinear operations associated with $\beta$: the \textbf{Hadamard} $\odot: \mathbb{K}^n_{\mathrm{std}} \times \mathbb{K}^n_{\mathrm{std}} \to \mathbb{K}^n_{\mathrm{std}}$, the \textbf{left braided} $\circ_{\beta_{\mathrm{L}}}: \mathbb{K}^n_{\mathrm{std}} \times \mathbb{K}^\ell_{\mathrm{std}} \to \mathbb{K}^\ell_{\mathrm{std}}$, and the \textbf{right braided} $\circ_{\beta_{\mathrm{R}}}: \mathbb{K}^\ell_{\mathrm{std}} \times \mathbb{K}^n_{\mathrm{std}} \to \mathbb{K}^\ell_{\mathrm{std}}$ compositions.

Specifically, for any $\vec{u} = (u_1, \dots, u_n),\; \vec{v} = (v_1, \dots, v_n) \in \mathbb{K}^n_{\mathrm{std}}$, and any $\vec{p} = (p_1, \dots, p_\ell),\; \vec{q} = (q_1, \dots, q_\ell) \in \mathbb{K}^\ell_{\mathrm{std}}$, we define:
\begin{equation*}
\begin{array}{rcl}
\vec{v} \odot \vec{u} &:=& (v_1 u_1, \dots, v_n u_n)\in \mathbb{K}^{n}_{\mathrm{std}}\,, \\[6pt]
\vec{v}\, \circ_{\beta_{\mathrm{L}}} \vec{p} &:=& (v_{\pi_{\beta_1}(i_1+1)} p_1, \dots, v_{\pi_{\beta_\ell}(i_\ell+1)} p_\ell)\in \mathbb{K}^{\ell}_{\mathrm{std}}\, ,\\[6pt]
\vec{q}\, \circ_{\beta_{\mathrm{R}}} \vec{u} &:=& (q_1 u_{\pi_{\beta_1}(i_1)}, \dots, q_\ell u_{\pi_{\beta_\ell}(i_\ell)})\in \mathbb{K}^{\ell}_{\mathrm{std}}\,,     
\end{array}    
\end{equation*}
where, for each $j \in [1,\ell]$:  
\begin{itemize}
\item $\pi_{\beta_j} \in \mathrm{S}_n$ is the permutation associated with $\beta_j = \sigma_{i_1} \dots \sigma_{i_j} \in \mathrm{Br}^{+}_{n}$, the truncation of $\beta$ at the $j$-th crossing.  
\item $u_{\pi_{\beta_j}(i_j)}\in \mathbb{K}$ and $v_{\pi_{\beta_j}(i_j+1)}\in\mathbb{K}$ are the $i_j$-th and $(i_j+1)$-th entries of $\pi_{\beta_j}(\vec{u}\,) \in \mathbb{K}^{n}_{\mathrm{std}}$ and $\pi_{\beta_j}(\vec{v}\,) \in \mathbb{K}^{n}_{\mathrm{std}}$---the vectors obtained by permuting the entries of $\vec{u}$ and $ \vec{v}$ via $\pi_{\beta_j}$---respectively, where $i_j \in [1,n-1]$ denotes the index of $\sigma_{i_j}$---the $j$-th crossing of $\beta$.
\end{itemize}
\end{definition}

\begin{maintheorem}[Algebraic Description of the Graded Composition]\label{Theorem: main result 2}
Let $\beta=\sigma_{i_{1}}\dots \sigma_{i_{\ell}}\in \mathrm{Br}^{+}_{n}$ be a positive braid word, and let $\sh{F}$, $\sh{G}$, $\sh{H}$ be objects of the category $\ccs{1}{\beta}$. Let \,$\mathbf{\hat{f}}^{(n)}, \,\mathbf{\hat{g}}^{(n)},\, \mathbf{\hat{h}}^{(n)}$ be bases for \,$\mathbb{K}^{n}$, and let \,$\vec{x},\, \vec{y},\, \vec{z}\in X(\beta, \mathbb{K})$ be points such that the pairs $(\,\mathbf{\hat{f}}^{(n)}\,, \vec{x}\,),\, (\,\mathbf{\hat{g}}^{(n)}\,, \vec{y}\,),\, (\,\mathbf{\hat{h}}^{(n)}\,, \vec{z}\,)$ algebraically characterize $\sh{F}$, $\sh{G}$, $\sh{H}$ according to Theorem~\eqref{Theorem: sheaves as points in the braid variety}, respectively. 

Following Definition~\eqref{Def: delta map}, let $\delta_{\sh{F},\sh{G}},\; \delta_{\sh{G},\sh{H}},\; \delta_{\sh{F},\sh{H}} : \mathbb{K}^{n}_{\mathrm{std}}\to \mathbb{K}^{\ell}_{\mathrm{std}}$ be the linear maps associated with the pairs $(\sh{F}, \sh{G}),\,(\sh{G}, \sh{H}),\,(\sh{F}, \sh{H})$. In light of Theorem~\eqref{Theorem: main result 1}, let $\mu\in \mathrm{Ext}^{0}(\sh{F},\sh{G})$, $\nu\in \mathrm{Ext}^{0}(\sh{G},\sh{H})$, $\Theta\in \mathrm{Ext}^{1}(\sh{F},\sh{G})$, $\Phi\in \mathrm{Ext}^{1}(\sh{G},\sh{H})$, and suppose that, for some representatives $\vec{p},\,\vec{q}\in \mathbb{K}^{\ell}_{\mathrm{std}}$, the vectors $\vec{u}\in \mathrm{ker}\,\delta_{\sh{F},\sh{G}}\subseteq \mathbb{K}^{n}_{\mathrm{std}}$, $\vec{v}\in \mathrm{ker}\,\delta_{\sh{G},\sh{H}}\subseteq \mathbb{K}^{n}_{\mathrm{std}}$, and the classes $[\,\vec{p}\,]\in \mathrm{coker}\, \delta_{\sh{F},\sh{G}}$, $[\,\vec{q}\,]\in \mathrm{coker}\, \delta_{\sh{G},\sh{H}}$ determine $\mu$, $\nu$, $\Theta$, $\Phi$, respectively, under the isomorphisms
\begin{equation*}
\begin{array}{rclcrcl}
\mathrm{Ext}^{0}(\sh{F},\sh{G})&\cong& \mathrm{ker}\, \delta_{\sh{F},\sh{G}}\, ,   & \qquad \qquad &  \mathrm{Ext}^{0}(\sh{G},\sh{H})&\cong& \mathrm{ker}\, \delta_{\sh{G},\sh{H}}\, ,\\[8pt]
\mathrm{Ext}^{1}(\sh{F},\sh{G})&\cong& \mathrm{coker}\, \delta_{\sh{F},\sh{G}}\, , & \qquad \qquad & \mathrm{Ext}^{1}(\sh{G},\sh{H})&\cong& \mathrm{coker}\, \delta_{\sh{G},\sh{H}}\, .\\[4pt]
\end{array}    
\end{equation*}
Let $i,j\geq 0$ be integers such that $i+j\in \{0,1\}$. Then, building on Definition~\eqref{Def: main braided graded composition}, the graded composition $\circ :\mathrm{Ext}^{i}(\sh{G}, \sh{H})\times \mathrm{Ext}^{j}(\sh{F}, \sh{G})\to \mathrm{Ext}^{i+j}(\sh{F}, \sh{H})$ admits the following description: 
\begin{itemize}
\item  \textbf{$(0,0)$-degree composition}: Under the isomorphism $\mathrm{Ext}^{0}(\sh{F},\sh{H})\cong \mathrm{ker}\, \delta_{\sh{F},\sh{H}}$, the graded composition $\nu\circ \mu \in \mathrm{Ext}^{0}(\sh{F},\sh{H})$ is determined by $\vec{v}\odot\vec{u}\in\mathrm{ker}\,\delta_{\sh{F},\sh{H}}$.
\item \textbf{$(0,1)$-degree composition}: Under the isomorphism $\mathrm{Ext}^{1}(\sh{F},\sh{H})\cong \mathrm{coker}\, \delta_{\sh{F},\sh{H}}$, the graded composition $\nu\circ \Theta \in \mathrm{Ext}^{1}(\sh{F},\sh{H})$ is determined by the class  $[\,\vec{v}\,\circ_{\beta_{\mathrm{L}}}\vec{p}\;]\in\mathrm{coker}\,\delta_{\sh{F},\sh{H}}$.
\item \textbf{$(1,0)$-degree composition}: Under the isomorphism $\mathrm{Ext}^{1}(\sh{F},\sh{H})\cong \mathrm{coker}\, \delta_{\sh{F},\sh{H}}$, the graded composition $\Phi\circ \mu \in \mathrm{Ext}^{1}(\sh{F},\sh{H})$ is determined by the class $[\,\vec{q}\,\circ_{\beta_{\mathrm{R}}}\vec{u}\;]\in\mathrm{coker}\,\delta_{\sh{F},\sh{H}}$.
\end{itemize}    
\end{maintheorem}

\noindent
The above results are of fundamental relevance, as together with Theorems~\eqref{Theorem: sheaves as points in the braid variety} and~\eqref{Theorem: hereditary-type property}, they provide an explicit and computable characterization of the category $\ccs{1}{\beta}$. Furthermore, together with the work of Chantraine, Ng, and Sivek~\cite{CNS1} on Legendrian $(2,m)$ torus links, our results constitute one of the first comprehensive studies of a microlocal-sheaf categorical invariant for a broad family of Legendrian links.

Having presented the main results of this paper, we proceed to describe the structure of the paper.

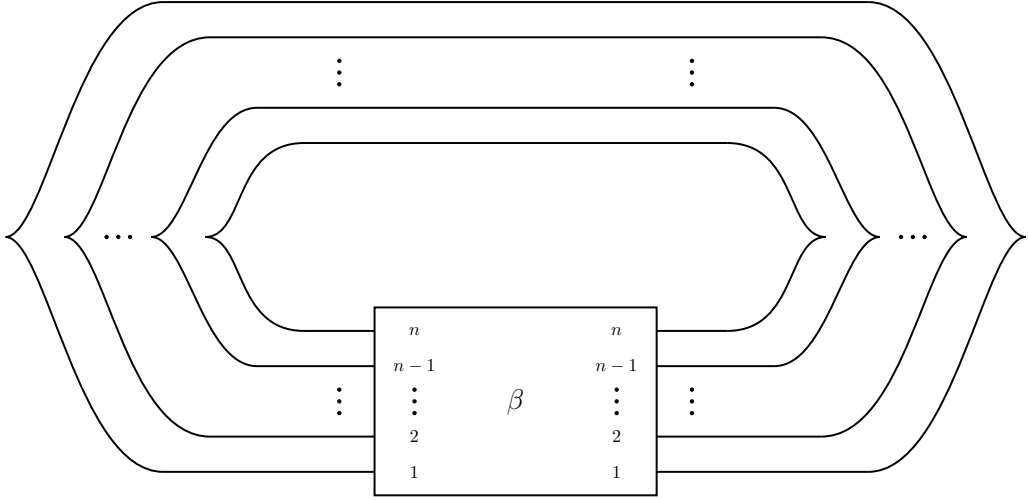
\begin{figure}
\centering
\begin{tikzpicture}
\useasboundingbox (-4,-3.75) rectangle (4,4);
\scope[transform canvas={scale=0.625}]

\draw[very thick] (-4-0.5,2) -- (4+0.5, 2);
\draw[very thick] (-4-0.5,-2) -- (-3, -2);
\draw[very thick] (3,-2) -- (4+0.5, -2);
\draw[very thick] (4+0.5,2) .. controls (4+0.5+1.5,2) and (4+0.5+2-0.75,0) .. (4+0.5+2+0.1,0);
\draw[very thick] (4+0.5,-2) .. controls (4+0.5+1.5,-2) and (4+0.5+2-0.75,0) .. (4+0.5+2+0.1,0);
\draw[very thick] (-4-0.5-2-0.1,0) .. controls (-4-0.5-2+0.75,0) and (-4-0.5-1.5,2) .. (-4-0.5,2);
\draw[very thick] (-4-0.5-2-0.1,0) .. controls (-4-0.5-2+0.75,0) and (-4-0.5-1.5,-2) .. (-4-0.5,-2);


\draw[very thick] (-4-0.5-1,2+0.75) -- (4+0.5+1, 2+0.75);
\draw[very thick] (-4-0.5-1,-2-0.75) -- (-3, -2-0.75);
\draw[very thick] (3,-2-0.75) -- (4+0.5+1, -2-0.75);
\draw[very thick] (4+0.5+1,2+0.75) .. controls (4+0.5+1+1.15,2+0.75) and (4+0.5+1+2-0.5,0) .. (4+0.5+1+2+0.25,0);
\draw[very thick] (4+0.5+1,-2-0.75) .. controls (4+0.5+1+1.15,-2-0.75) and (4+0.5+1+2-0.5,0) .. (4+0.5+1+2+0.25,0);
\draw[very thick] (-4-0.5-1-2-0.25,0) .. controls (-4-0.5-1-2+0.5,0) and (-4-0.5-1-1.15,2+0.75) .. (-4-0.5-1,2+0.75);
\draw[very thick] (-4-0.5-1-2-0.25,0) .. controls (-4-0.5-1-2+0.5,0) and (-4-0.5-1-1.15,-2-0.75) .. (-4-0.5-1,-2-0.75);


\draw[very thick] (-4-0.5-1-1,+2+3.5-0.5-0.75) -- (4+0.5+1+1,+2+3.5-0.5-0.75);
\draw[very thick] (-4-0.5-1-1,-2-3.5+0.5+0.75) -- (-3,-2-3.5+0.5+0.75);
\draw[very thick] (3,-2-3.5+0.5+0.75) -- (4+0.5+1+1,-2-3.5+0.5+0.75);
\draw[very thick] (4+0.5+1+1,+2+3.5-0.5-0.75) .. controls (4+0.5+1+1+1.15+0.65,+2+3.5-0.5-0.75) and (4+0.5+1+1+2-0.5+0.75,0) .. (4+0.5+1+1+2+0.75+0.35,0);
\draw[very thick] (4+0.5+1+1,-2-3.5+0.5+0.75) .. controls (4+0.5+1+1+1.15+0.65,-2-3.5+0.5+0.75) and (4+0.5+1+1+2-0.5+0.75,0) .. (4+0.5+1+1+2+0.75+0.35,0);
\draw[very thick] (-4-0.5-1-1-2-0.75-0.35,0) .. controls (-4-0.5-1-1-2+0.5-0.75,0) and (-4-0.5-1-1-1.15-0.65,+2+3.5-0.5-0.75) .. (-4-0.5-1-1,+2+3.5-0.5-0.75);
\draw[very thick] (-4-0.5-1-1-2-0.75-0.35,0) .. controls (-4-0.5-1-1-2+0.5-0.75,0) and (-4-0.5-1-1-1.15-0.65,-2-3.5+0.5+0.75) .. (-4-0.5-1-1,-2-3.5+0.5+0.75);


\draw[very thick] (-4-0.5-1-1-1,+2+3.5-0.5) -- (4+0.5+1+1+1,+2+3.5-0.5);
\draw[very thick] (-4-0.5-1-1-1,-2-3.5+0.5) -- (-3,-2-3.5+0.5);
\draw[very thick] (3,-2-3.5+0.5) -- (4+0.5+1+1+1,-2-3.5+0.5);
\draw[very thick] (4+0.5+1+1+1,+2+3.5-0.5) .. controls (4+0.5+1+1+1+1.15+0.65,+2+3.5-0.5) and (4+0.5+1+1+1+2-0.5+0.75+0.25,0) .. (4+0.5+1+1+1+2+0.75+0.25+0.35,0);
\draw[very thick] (4+0.5+1+1+1,-2-3.5+0.5) .. controls (4+0.5+1+1+1+1.15+0.65,-2-3.5+0.5) and (4+0.5+1+1+1+2-0.5+0.75+0.25,0) .. (4+0.5+1+1+1+2+0.75+0.25+0.35,0);
\draw[very thick] (-4-0.5-1-1-1-2-0.75-0.25-0.35,0) .. controls (-4-0.5-1-1-1-2+0.5-0.75-0.25,0) and (-4-0.5-1-1-1-1.15-0.65,+2+3.5-0.5) .. (-4-0.5-1-1-1,+2+3.5-0.5);
\draw[very thick] (-4-0.5-1-1-1-2-0.75-0.25-0.35,0) .. controls (-4-0.5-1-1-1-2+0.5-0.75-0.25,0) and (-4-0.5-1-1-1-1.15-0.65,-2-3.5+0.5) .. (-4-0.5-1-1-1,-2-3.5+0.5);


\draw[very thick] (-3,-2-3.5) rectangle (3,2-3.5);

\node at (-3+0.85,-2) {\large$n$};
\node at (-3+0.85,-2-0.75) {\large$n-1$};
\node at (-3+0.85,-2-3.5+0.5+0.75) {\large$2$};
\node at (-3+0.85,-2-3.5+0.5) {\large$1$};

\node at (3-0.85,-2) {\large$n$};
\node at (3-0.85,-2-0.75) {\large$n-1$};
\node at (3-0.85,-2-3.5+0.5+0.75) {\large$2$};
\node at (3-0.85,-2-3.5+0.5) {\large$1$};

\node at  (0,-3.5) {\huge$\beta$};

\filldraw[black] (-3+0.85,-3.5+0.25) circle (1pt);
\filldraw[black] (-3+0.85,-3.5) circle (1pt);
\filldraw[black] (-3+0.85,-3.5-0.25) circle (1pt);

\filldraw[black] (3-0.85,-3.5+0.25) circle (1pt);
\filldraw[black] (3-0.85,-3.5) circle (1pt);
\filldraw[black] (3-0.85,-3.5-0.25) circle (1pt);

\filldraw[black] (-3-0.75,-3.5+0.25) circle (1pt);
\filldraw[black] (-3-0.75,-3.5) circle (1pt);
\filldraw[black] (-3-0.75,-3.5-0.25) circle (1pt);

\filldraw[black] (3+0.75,-3.5+0.25) circle (1pt);
\filldraw[black] (3+0.75,-3.5) circle (1pt);
\filldraw[black] (3+0.75,-3.5-0.25) circle (1pt);

\filldraw[black] (-3-0.75,-3.5+0.25+7) circle (1pt);
\filldraw[black] (-3-0.75,-3.5+7) circle (1pt);
\filldraw[black] (-3-0.75,-3.5-0.25+7) circle (1pt);

\filldraw[black] (3+0.75,-3.5+0.25+7) circle (1pt);
\filldraw[black] (3+0.75,-3.5+7) circle (1pt);
\filldraw[black] (3+0.75,-3.5-0.25+7) circle (1pt);

\filldraw[black] (-8.45-0.25,0) circle (1pt);
\filldraw[black] (-8.45,0) circle (1pt);
\filldraw[black] (-8.45+0.25,0) circle (1pt);

\filldraw[black] (8.45-0.25,0) circle (1pt);
\filldraw[black] (8.45,0) circle (1pt);
\filldraw[black] (8.45+0.25,0) circle (1pt);

\node at (10,5.75) {\LARGE$\mathbb{R}^{2}_{x,z}$};

\endscope
\end{tikzpicture}
\caption{Front projection $\Pi_{x,z}(\Lambda(\beta))$ of the Legendrian link $\Lambda(\beta)$.}
\label{Front diagram of the rainbow closure of a braid in n strands}
\end{figure}

\subsection{Structure of the Manuscript}
In this subsection, we briefly outline the structure of the paper. In particular, the remainder of Section \ref{sec:introduction} introduces the notation and conventions used throughout the manuscript.

\begin{itemize}
\item Section \ref{sec:rainbow closures} presents the basic geometric constructions relevant to our study and gives a concise definition of the family of Legendrian links of interest.
\item Section \ref{sec:microlocal theory} briefly reviews some aspects of the microlocal theory of Legendrian links developed by Shende, Treumann, and Zaslow~\cite{STZ1}.
\item Section \ref{sec:objects} provides a detailed analysis of the objects of the category under study and derives Theorem~\eqref{Theorem: sheaves as points in the braid variety} from first principles.
\item Section \ref{sec:morphisms and compositions} discusses the graded morphism spaces and their compositions in the category of interest, and provides the proofs of Theorems~\eqref{Theorem: main result 1} and~\eqref{Theorem: main result 2}.
\item Section \ref{sec:applications} presents applications of our main results and illustrates the scope and effectiveness of our methods through a detailed study of selected examples.
\end{itemize}

\noindent {\bf Acknowledgements}. First and foremost, I would like to express my sincere gratitude to my PhD advisor, Roger Casals, for his guidance, support, and many inspiring discussions, which have culminated in this manuscript. In particular, I am grateful to him for introducing me to the microlocal theory of sheaves and its applications to symplectic and contact topology. I would also like to thank Eugene Gorsky for his insightful questions and comments at an early stage of this work. Last but not least, I would like to thank Lenny Ng, Honghao Gao, Orsola Capovilla--Searle, Wenyuan Li, and James Hughes for their interest in this project and for helpful discussions on sheaf- and Floer-theoretic approaches to study of Legendrian links.

\subsection{Notational Conventions} To ensure clarity and consistency, this subsection sets forth the notation and conventions used throughout the manuscript.

Let $a,b\in \mathbb{N}$ with $a<b$. We define $[a,b]:=\{c\in\mathbb{N}\mid a \leq c \leq b \}$. Let $n\geq 2$ be an integer. We denote by $\mathrm{Br}_{n}$ the braid group on $n$ strands; more precisely, the group on $n-1$ generators whose presentation is given by:
\begin{equation*}
\mathrm{Br}_{n}:=\left\langle \sigma_{1}, \dots, \sigma_{n-1}~ \left| ~ 
\begin{aligned}
~\sigma_{i}\sigma_{i+1}\sigma_{i}&=\sigma_{i+1}\sigma_{i}\sigma_{i+1}\, , \quad && \text{for all $i\in [1,n-2]$}\,\\
~\sigma_{i}\sigma_{j}&=\sigma_{j}\sigma_{i}\, , \quad && \text{for all $i, j\in [1,n-1]$ such that $|i-j|\geq2$}\,     
\end{aligned}\right. ~ \right\rangle\, ,
\end{equation*}
where $\sigma_{i}\in \mathrm{Br}_{n}$ corresponds to the standard $i$-th Artin generator for all $i\in [1,n-1]$. Accordingly, we denote by $e_{n}$ the identity element of $\mathrm{Br}_{n}$, and by $\mathrm{Br}^{+}_{n}\subset \mathrm{Br}_{n}$ the monoid generated by the positive powers of the Artin generators. Furthermore, a product expression of the form $\beta = \sigma_{i_{1}} \cdots \sigma_{i_{\ell}} \in \mathrm{Br}^{+}_{n}$ is referred to as a \emph{positive braid word} of length $\ell \in \mathbb{N}$.

When drawing a braid diagram on $n$ strands, our convention is that strands are enumerated from bottom to top by $1,\dots, n$. In addition, we multiply braids from left to right, and depict the braid diagram of a positive braid word by reading its generators from left to right, drawing the corresponding crossings in the same order. For example, the braid diagram on three strands corresponding to the positive braid word $\sigma_{1}\sigma_{2}\in \mathrm{Br}^{+}_{3}$ is depicted in Figure~\eqref{Example of a braid diagram with n=3}. 

\begin{figure}[ht]
\centering
\begin{tikzpicture}
\useasboundingbox (-2.5,-1) rectangle (6,3);
\scope[transform canvas={scale=1}]

\draw[thick] (0,0) .. controls (0.5,0) and (0.5,0).. (1,0).. controls (1 +0.50, 0) and (1+0.50, 1)..(2,1).. controls (2+0.5,1) and (2+0.5,1).. (3,1).. controls (3+0.50, 1) and (3+0.50,2).. (4,2).. controls (4+0.5,2) and (4+0.5, 2).. (5,2);

\draw[thick] (0,1) .. controls (0.5,1) and (0.5,1).. (1,1).. controls (1 +0.50, 1) and (1+0.50, 0)..(2,0).. controls (3,0) and (4,0).. (5,0);

\draw[thick] (0,2) .. controls (1,2) and (2,2).. (3,2).. controls (3 +0.50, 2) and (3+0.50, 1)..(4,1).. controls (4.5,1) and (4.5,1).. (5,1);

\node at (-0.5,0) {\footnotesize $1$};
\node at (-0.5,1) {\footnotesize $2$};
\node at (-0.5,2) {\footnotesize $3$};

\node at (1.5,-0.75) {\footnotesize\large $\sigma_{1}$};
\node at (3.5,-0.75) {\footnotesize\large $\sigma_{2}$};

\node at (-1.75, 1) {\footnotesize\large $\sigma_{1}\,\sigma_{2}~=$};

\endscope
\end{tikzpicture}
\caption{Braid diagram on three strands of the positive braid word $\sigma_{1}\sigma_{2}\in\mathrm{Br}^{+}_{3}$.}
\label{Example of a braid diagram with n=3}
\end{figure}
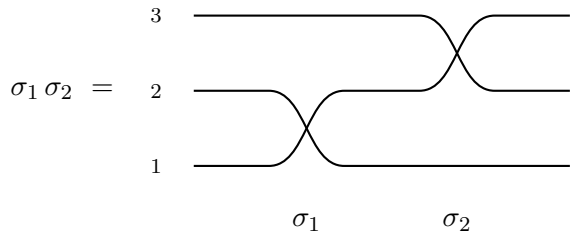

Next, we denote by $\mathrm{S}_{n}$ the Coxeter group associated with $\mathrm{Br}_{n}$---the symmetric group on $n$ elements. More precisely,
\begin{equation*}
\mathrm{S}_{n}:=\left\langle s_{1}, \dots, s_{n-1}~ \left| ~ 
\begin{aligned}
s_{i}^{2}&=1\, , \quad && \text{for all $i\in [1,n-1]$}\,\\ 
s_{i}s_{i+1}s_{i}&=s_{i+1}s_{i}s_{i+1}\, , \quad && \text{for all $i\in [1,n-2]$}\,\\ 
s_{i}s_{j}&=s_{j}s_{i}\, , \quad && \text{for all $i, j\in [1,n-1]$ such that $|i-j|\geq2$}\,     
\end{aligned}\right. ~ \right\rangle\, ,
\end{equation*}
where $s_{i}\in \mathrm{S}_{n}$ denotes the $i$-th adjacent transposition for all $i\in [1,n-1]$. In particular, by a slight abuse of notation, we also use $e_{n}$ to denote the identity element in $\mathrm{S}_{n}$. Throughout this paper, we compose permutations from left to right, so that the map $\mathrm{Br}_{n}\to\mathrm{S}_{n}$ given by $\sigma_{i}\mapsto s_{i}$ defines a group homomorphism. In other words, for any $\pi_{1},\pi_{2}\in \mathrm{S}_{n}$ and any $k\in[1,n]$, we set $(\pi_{1}\pi_{2})(k):=\pi_{2}(\pi_{1}(k))$. For instance, the permutation $s_{1}s_{2}\in \mathrm{S}_{3}$ is given by $s_{1}s_{2}(1)=3$, $s_{1}s_{2}(2)=1$, $s_{1}s_{2}(3)=2$. Furthermore, for any positive braid word $\beta=\sigma_{i_{1}}\cdots\sigma_{i_{\ell}}\in \mathrm{Br}^{+}_{n}$, we denote by $\pi_{\beta}:=s_{i_{1}}\cdots s_{i_{\ell}}\in \mathrm{S}_{n}$ the permutation associated with $\beta$.
 
Let $m,n\in \mathbb{N}$, and let $R$ be a commutative ring. We denote by $\mathrm{M}(n,m,R)$ the set of $n\times m$ matrices over $R$, by $\mathrm{M}(n,R)$ the ring of $n\times n$ matrices over $R$, and by $\mathrm{GL}(n,R)$ the group of invertible $n\times n$ matrices over $R$. Moreover, given a matrix $A\in \mathrm{M}(n,m, R)$, we denote by $A_{i,j}\in R$ the $(i,j)$-entry of $A$, for all $i\in[1,n]$ and $j\in[1,m]$. In certain contexts, we also write
\begin{equation*}
A = 
\left[
\begin{array}{ccc}
c_1 & \cdots & c_m
\end{array}
\right]\, ,
\quad \text{and} \quad
A =
\begin{bmatrix}
r_1 \\
\vdots \\
r_n
\end{bmatrix},
\end{equation*}
where $c_{i}$ and $r_{j}$ denote the $i$-th column and $j$-th row vectors of $A$, respectively, for all $i\in[1,m]$ and $j\in[1,n]$.

Throughout the manuscript, we identify $\mathbb{K}$ with an arbitrary ground field. Accordingly, for any integer $m\geq 1$, we denote by $\mathbb{K}^{m}$ the $m$-dimensional Cartesian vector space over $\mathbb{K}$ with no basis specified, and by $\mathbb{K}^{m}_{\mathrm{std}}$ the same vector space equipped with its standard ordered basis. Now, let $X$ and $Y$ be finite-dimensional vector spaces over $\mathbb{K}$ such that $\mathrm{dim}_{\,\mathbb{K}}X=p$ and $\mathrm{dim}_{\,\mathbb{K}}Y=q$, for some $p,q\in \mathbb{N}$. Let $T: X\to Y$ be a linear map, and let $\hat{\mathbf{x}}:=\big\{\hat{x}_{i}\big\}_{i=1}^{p}$ and $\hat{\mathbf{y}}:=\big\{\hat{y}_{i}\big\}_{i=1}^{q}$ be bases for $X$ and $Y$, respectively. Then, we denote by $\tensor[_{\hat{\mathbf{y}}}]{ \big[\, T \,\big] }{_{\hat{\mathbf{x}}}}\in\mathrm{M}(q,p, \mathbb{K})$ the matrix representing $T$ with respect to the bases $\hat{\mathbf{x}}$ and $\hat{\mathbf{y}}$. 

Finally, let $\mathcal{M}$ be a smooth manifold, and let $\sh{F}$ be a sheaf on $\mathcal{M}$ with values in an abelian category $\mathcal{C}$ \cite{Hart1, KS1, HS1}. For any point $x\in \mathcal{M}$, we denote by $\sh{F}_{x}$ the stalk of $\sh{F}$ at $x$. Furthermore, we denote by $SS(\sh{F})\subseteq T^{*}\mathcal{M}$ the singular support of $\sh{F}$. For the definition of singular support of a sheaf and a detailed treatment of its geometric and topological properties, see~\cite{KS1}.
   
With our notation and conventions in place, we proceed to introduce the family of Legendrian links in $\mathbb{R}^{3}$ that constitutes the essential geometric foundation for the main categorical invariant under study in this article.

%% file: sec2.tex
\section{The Rainbow Closure of Positive Braids}\label{sec:rainbow closures}
\noindent
This section is devoted to defining the \emph{rainbow closure} of positive braid words---that is, a geometric construction that yields a distinguished family of Legendrian links in $\mathbb{R}^{3}$.

Let $(x,y,z)$ be coordinates on $\mathbb{R}^{3}$. In this setting, the \textit{standard contact structure} $\xi_{\mathrm{std}}$ on $\mathbb{R}^{3}$ is defined by $\xi_{\mathrm{std}}:=\mathrm{ker}\,(dz-ydx)$. Given this contact distribution, a \textit{Legendrian knot} in $(\mathbb{R}^{3},\xi_{\mathrm{std}})$ is a knot in $\mathbb{R}^{3}$ with a smooth parametrization $\gamma:\mathbb{S}^{1}\rightarrow \mathbb{R}^{3}$ such that, for all $t\in \mathbb{S}^{1}$, the tangent vector $\gamma'(t)$ lies in $\xi_{\mathrm{std}}$ at the point $\gamma(t)$. Extending this notion, a \textit{Legendrian link} in $(\mathbb{R}^{3},\xi_{\mathrm{std}})$ is a finite collection of disjoint Legendrian knots~\cite{EN1}. 

Next, we introduce a powerful representation for visualizing Legendrian knots in $(\mathbb{R}^{3},\xi_{\mathrm{std}})$, namely the \emph{front projection}. To this end, let $\Lambda\subset (\mathbb{R}^{3},\xi_{\mathrm{std}})$ be a Legendrian knot, and let $\Pi_{x,z}:\mathbb{R}^{3}\rightarrow \mathbb{R}^{2}$ be the smooth map given by $\Pi_{x,z}(x,y,z):=(x,z)$, for all $(x,y,z)\in \mathbb{R}^{3}$. Accordingly, the front projection of $\Lambda$, denoted $\Pi_{x,z}(\Lambda)\subset\mathbb{R}^{2}$, is defined as the image of $\Lambda$ under the map $\Pi_{x,z}$. Notably, the front projection fully characterizes $\Lambda$. More precisely, since the $1$-form $dz-ydx$ vanishes along $\Lambda$, the $y$-coordinate can be recovered from the front projection $\Pi_{x,z}(\Lambda)$ via the relation $y=dz/dx$. The front projection possesses several powerful properties; in particular, any immersed curve in $\mathbb{R}^{2}$ with no vertical tangencies lifts to a unique Legendrian link in $(\mathbb{R}^{3},\xi_{\mathrm{std}})$~\cite{EN1}. Consequently, certain distinguished families of Legendrian links in $(\mathbb{R}^{3},\xi_{\mathrm{std}})$ can be constructed via the process of \textit{cusping off} braid diagrams of positive braid words on $n$ strands. Following~\cite{KT1, KT2, STZ1}, we introduce the following definition.

\begin{definition}\label{Def: rainbow closure of a positive braid}
Let $\beta \in\mathrm{Br}^{+}_{n}$ be a positive braid word. We denote by $\Lambda(\beta)$ the Legendrian link in $(\mathbb{R}^{3}, \xi_{\mathrm{std}})$ whose front projection $\Pi_{x,z}(\Lambda(\beta))\subset\mathbb{R}^{2}$ is illustrated in Figure~\eqref{Front diagram of the rainbow closure of a braid in n strands}. We refer to $\Lambda(\beta)$ as the \textit{rainbow closure} of $\beta$.    
\end{definition}

\begin{remark}
Let $e_{n}\in \mathrm{Br}^{+}_{n}$ be the trivial braid word. The Legendrian link $\Lambda(e_{n})\subset  (\mathbb{R}^{3}, \xi_{\mathrm{std}})$ corresponds to the Legendrian unlink on $n$ strands. In particular, the front projection $\Pi_{x,z}(\Lambda(e_{n}))\subset \mathbb{R}^{2}$ of $\Lambda(e_{n})$ is depicted in Figure~\eqref{fig: unlink on n-strands}.  
\end{remark}

\begin{figure}
\centering
\begin{tikzpicture}
\useasboundingbox (-8,-3.75) rectangle (8,3.85);
\scope[transform canvas={scale=0.625}]

\draw[very thick] (-2,1) -- (2, 1);
\draw[very thick] (-2,-1) -- (2, -1);
\draw[very thick] (2,1) .. controls (3-0.15,1) and (3-0.25,0) .. (3.5,0);
\draw[very thick] (2,-1) .. controls (3-0.15,-1) and (3-0.25,0) .. (3.5,0);
\draw[very thick] (-3.5, 0) .. controls (-3+0.25,0) and (-3+0.15,1) .. (-2,1);
\draw[very thick] (-3.5, 0) .. controls (-3+0.25,0) and (-3+0.15,-1) .. (-2,-1);

\draw[very thick] (-2-1,1+1+0.25) -- (2+1,1+1+0.25);
\draw[very thick] (-2-1,-1-1-0.25) -- (2+1,-1-1-0.25);
\draw[very thick] (2+1,1+1+0.25) .. controls (3+1-0.15,1+1+0.25) and (3+1-0.15,0+0.2) .. (3.5+1.5+0.15,0);
\draw[very thick] (2+1,-1-1-0.25) .. controls (3+1-0.15,-1-1-0.25) and (3+1-0.15,0-0.2) .. (3.5+1.5+0.15,0);
\draw[very thick] (-3.5-1.5-0.15, 0) .. controls (-3-1+0.15,0+0.2) and (-3-1+0.15,1+1+0.25) .. (-2-1,1+1+0.25);
\draw[very thick] (-3.5-1.5-0.15, 0) .. controls (-3-1+0.15,0-0.2) and (-3-1+0.15,-1-1-0.25) .. (-2-1,-1-1-0.25);

\draw[very thick] (-2-2.5,1+2.5) -- (2+2.5, 1+2.5);
\draw[very thick] (-2-2.5,-1-2.5) -- (2+2.5, -1-2.5);
\draw[very thick] (2+2.5,1+2.5) .. controls (3+3,1+2.5) and (3+3,0) .. (3.5+4,0);
\draw[very thick] (2+2.5,-1-2.5) .. controls (3+3,-1-2.5) and (3+3,0) .. (3.5+4,0);
\draw[very thick] (-3.5-4, 0) .. controls (-3-3,0) and (-3-3,1+2.5) .. (-2-2.5,1+2.5);
\draw[very thick] (-3.5-4, 0) .. controls (-3-3,0) and (-3-3,-1-2.5) .. (-2-2.5,-1-2.5);

\draw[very thick] (-2-2.5-1,1+2.5+1+0.25) -- (2+2.5+1, 1+2.5+1+0.25);
\draw[very thick] (-2-2.5-1,-1-2.5-1-0.25) -- (2+2.5+1, -1-2.5-1-0.25);
\draw[very thick] (2+2.5+1, 1+2.5+1+0.25) .. controls (3+3+1+0.5,1+2.5+1+0.25) and (3+3+1+0.5,0) .. (3.5+4+2,0);
\draw[very thick] (2+2.5+1,-1-2.5-1-0.25) .. controls (3+3+1+0.5,-1-2.5-1-0.25) and (3+3+1+0.5,0) .. (3.5+4+2,0);
\draw[very thick] (-3.5-4-2, 0) .. controls (-3-3-1-0.5,0) and (-3-3-1-0.5,1+2.5+1+0.25) .. (-2-2.5-1,1+2.5+1+0.25);
\draw[very thick] (-3.5-4-2, 0) .. controls (-3-3-1-0.5,0) and (-3-3-1-0.5,-1-2.5-1-0.25) .. (-2-2.5-1,-1-2.5-1-0.25);

\filldraw[black] (2,2.9) circle (1pt);
\filldraw[black] (2,2.9+0.25) circle (1pt);
\filldraw[black] (2,2.9-0.25) circle (1pt);

\filldraw[black] (-2,2.9) circle (1pt);
\filldraw[black] (-2,2.9+0.25) circle (1pt);
\filldraw[black] (-2,2.9-0.25) circle (1pt);

\filldraw[black] (2,-2.9) circle (1pt);
\filldraw[black] (2,-2.9-0.25) circle (1pt);
\filldraw[black] (2,-2.9+0.25) circle (1pt);

\filldraw[black] (-2,-2.9) circle (1pt);
\filldraw[black] (-2,-2.9-0.25) circle (1pt);
\filldraw[black] (-2,-2.9+0.25) circle (1pt);

\filldraw[black] (-6,0) circle (1pt);
\filldraw[black] (-6-0.25,0) circle (1pt);
\filldraw[black] (-6+0.25,0) circle (1pt);

\filldraw[black] (6,0) circle (1pt);
\filldraw[black] (6+0.25,0) circle (1pt);
\filldraw[black] (6-0.25,0) circle (1pt);

\node at (0,-1-2.5-1-0.25-0.5) {\large$1$};
\node at (0,-1-2.5-0.5) {\large$2$};
\node at (0,-1-1-0.25-0.5) {\large$n-1$};
\node at (0,-1-0.5) {\large$n$};

\node at (7.5,5.5) {\Large$\mathbb{R}^{2}_{x,z}$};

\endscope
\end{tikzpicture}
\captionof{figure}{Front projection $\Pi_{x,z}(\Lambda(e_{n}))\subset \mathbb{R}^{2}$ of the Legendrian unlink $\Lambda(e_{n})\subset(\mathbb{R}^{3}, \xi_{\mathrm{std}})$ on $n$ strands.}
\label{fig: unlink on n-strands}
\end{figure}
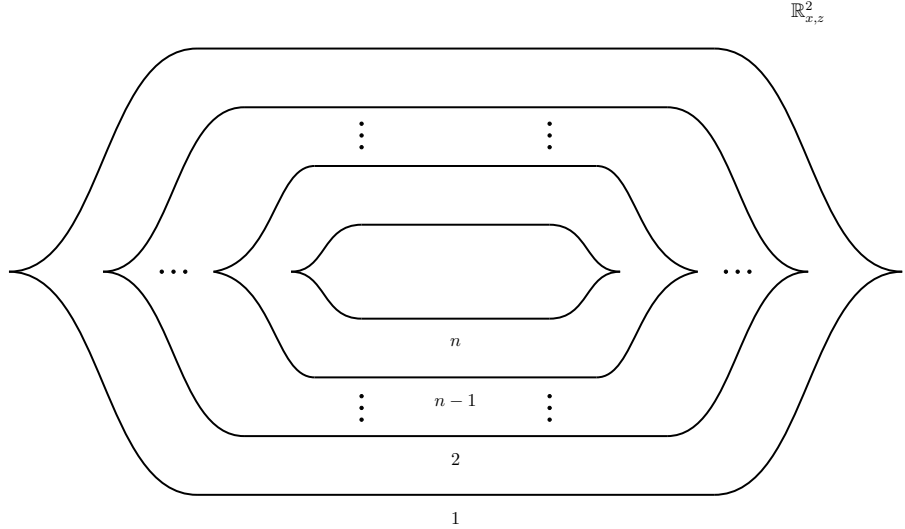

\noindent 
The family of Legendrian links introduced in Definition \eqref{Def: rainbow closure of a positive braid} constitutes the essential geometric foundation for the main categorical invariant under study in this manuscript. This family of Legendrian links has also been studied from a variety of algebraic and geometric perspectives in the literature; see, for instance,~\cite{KT1, KT2, STZ1, STZ2, CG1, HJ1, CG2, CN1, GSW1, SW1}.

We conclude this section by recalling a construction from~\cite{STZ1} that associates to any Legendrian link $\Lambda\subset (\mathbb{R}^{3}, \xi_{\mathrm{std}})$ a closed conic Lagrangian $L(\Lambda)\subset (T^{*}\mathbb{R}^{2},\omega_{\mathrm{std}})$, where $\omega_{\mathrm{std}}$ denotes the standard symplectic structure on $T^{*}\mathbb{R}^{2}$.  

\begin{construction}\label{Def: closed conic Lagrangian associated with a Legendrian}
Let $(x,z)$ be coordinates on $\mathbb{R}^{2}$, and let $(x,z, p_{x}, p_{z})$ be induced coordinates on $T^{*}\mathbb{R}^{2}$. In this setting, $T^{*}\mathbb{R}^{2}$ carries the standard symplectic structure $\omega_{\mathrm{std}}:=d\,\theta_{\mathrm{std}}$, where the standard Liouville $1$-form $\theta_{\mathrm{std}}$ on $T^{*}\mathbb{R}^{2}$ is given by $\theta_{\mathrm{std}}:=-(p_{x}\,dx+p_{z}\,dz)$.  

Let $(x,y,z)$ be coordinates on $\mathbb{R}^{3}$, and recall that the standard contact structure $\xi_{\mathrm{std}}$ on $\mathbb{R}^{3}$ is defined by $\xi_{\mathrm{std}}=\mathrm{ker}\,(dz-y\,dx)$. Now, let \,$\kappa:\mathbb{R}^{3}\to T^{*}\mathbb{R}^{2}$ be the smooth map given by $\kappa(x,y,z):=(x,z,y,-1)$, for all $(x,y,z)\in \mathbb{R}^{3}$. By construction, $\,\xi_{\mathrm{std}}=\mathrm{ker}\,\big (\kappa^{*}\theta_{\mathrm{std}}\big)$, where $\kappa^{*}\theta_{\mathrm{std}}$ denotes the pull-back of $\theta_{\mathrm{std}}$ via $\kappa$. It follows that, since $\kappa$ is injective, has an injective differential, and $\,\mathbb{R}^{3}$ is diffeomorphic to $\kappa(\mathbb{R}^{3})$, $\kappa$ is an embedding~\cite{STZ1}. As a result, $\kappa$ realizes $(\mathbb{R}^{3}, \xi_{\mathrm{std}})$ as an embedded contact submanifold of $(T^{*}\mathbb{R}^{2}, \omega_{\mathrm{std}})$. 

Let $\Lambda \subset (\mathbb{R}^{3}, \xi_{\mathrm{std}})$ be a Legendrian link. We denote by $\tilde{\Lambda}\subset T^{*}\mathbb{R}^{2}$ the image of $\Lambda$ under $\kappa$, namely $\tilde{\Lambda}:=\kappa(\Lambda)$, and introduce $L(\Lambda)\subset (T^{*}\mathbb{R}^{2}, \omega_{\mathrm{std}})$ to denote the closed conic Lagrangian defined by $L(\Lambda):=\mathbb{R}_{> 0}\,\tilde{\Lambda}\;\cup\; 0_{\,\mathbb{R}^{2}}$, where $\mathbb{R}_{>0}\,\tilde{\Lambda} \subset T^{*}\mathbb{R}^{2}$ is the positive cone over $\tilde{\Lambda}$ and $0_{\,\mathbb{R}^{2}} \subset T^{*}\mathbb{R}^{2}$ is the zero section.         
\end{construction}

\begin{notation}
Let $\beta\in \mathrm{Br}^{+}_{n}$ be a positive braid word, and let $\Lambda(\beta)\subset (\mathbb{R}^{3}, \xi_{\mathrm{std}})$ be its associated Legendrian link. Building on Construction~\eqref{Def: closed conic Lagrangian associated with a Legendrian}, we introduce $L(\Lambda(\beta))\subset (T^{*}\mathbb{R}^{2}, \omega_{\mathrm{std}})$ to denote the closed conic Lagrangian associated with $\Lambda(\beta)$.   
\end{notation}

Having introduced the distinguished family of Legendrian links and the main contact and symplectic constructions relevant to our study, we now turn to the definition of the associated microlocal-sheaf categorical invariant of interest in this article.

%% file: sec3.tex
\section{Microlocal Theory of Legendrian Links}\label{sec:microlocal theory}

\noindent
In this section, we briefly review the microlocal theory of Legendrian links developed by Shende, Treumann, and Zaslow~\cite{STZ1}. More precisely, given a ground field $\mathbb{K}$, an integer $r\geq 1$, and a Legendrian link $\Lambda\subset (\mathbb{R}^{3},\xi_{\mathrm{std}})$ whose front projection $\Pi_{x,z}(\Lambda)\subset\mathbb{R}^{2}$ is generic and carries a binary Maslov potential, we focus on the combinatorial description of the category $\mathcal{S}h_{r}(\Lambda,\mathbb{K})_{0}$---that is, the dg-derived category of microlocal rank $r$, compactly supported, constructible sheaves of $\mathbb{K}$-modules on $\mathbb{R}^{2}$ whose singular support is contained in the closed conic Lagrangian $L(\Lambda)\subset (T^{*}\mathbb{R}^{2}, \omega_{\mathrm{std}})$ associated with $\Lambda$ (see Construction~\eqref{Def: closed conic Lagrangian associated with a Legendrian}). In addition, we define the main object of interest in this manuscript: the cohomological category $H^{\bullet}\big(\mathcal{S}h_{r}(\Lambda,\mathbb{K})_{0}\big)$. To this end, we begin by describing certain stratifications of $\mathbb{R}^{2}$, which are essential for the combinatorial characterization of the objects of the categories under consideration.

\subsection{Regular Stratifications of \texorpdfstring{$\mathbb{R}^{2}$}{R2} Induced by Legendrian Links} As previously mentioned, we aim to describe a specific dg-derived category of constructible sheaves of $\mathbb{K}$-modules on $\mathbb{R}^{2}$. Accordingly, the main goal of this subsection is to introduce the stratifications of $\mathbb{R}^{2}$ relevant to our study.

Let $\Lambda\subset (\mathbb{R}^{3},\xi_{\mathrm{std}})$ be a Legendrian link. We say that the front projection $\Pi_{x,z}(\Lambda)\subset \mathbb{R}^{2}$ of $\Lambda$ is \textit{generic} if $\Pi_{x,z}(\Lambda)$ has only cusps and crossings as singularities. In particular, when $\Lambda$ has a generic front projection, it induces a regular stratification $\mathcal{S}_{\Lambda}$ of $\mathbb{R}^{2}$ defined as follows~\cite{STZ1}:  
\begin{itemize}
\justifying
    \item [(\textit{i})] The $0$-dimensional strata correspond to the cusps and crossings in $\Pi_{x,z}(\Lambda)$. 
    \item [(\textit{ii})] The $1$-dimensional strata are given by the arcs in $\Pi_{x,z}(\Lambda)$. 
    \item [(\textit{iii})] The $2$-dimensional strata consist of the disjoint open subsets of $\mathbb{R}^{2}$ whose union is equal to the complement of $\Pi_{x,z}(\Lambda)$ in $\mathbb{R}^{2}$. Consequently, there is only one unbounded $2$-dimensional stratum in the stratification $\mathcal{S}_{\Lambda}$. 
\end{itemize}
Next, let $a\in\mathcal{S}_{\Lambda}$ be a stratum. We define the star of $a$, denoted $s(a)\subset\mathcal{S}_{\Lambda}$, as the union of all the strata of $\mathcal{S}_{\Lambda}$ whose closure contains $a$. By construction, $\mathcal{S}_{\Lambda}$ forms a regular cell complex; that is, each stratum $a\in \mathcal{S}_{\Lambda}$ and its star $s(a)$ are contractible in $\mathbb{S}^{2}=\mathbb{R}^{2}\cup\{\infty\}$, the one-point compactification of $\mathbb{R}^{2}$. 

It is important to emphasize that, for any positive braid word $\beta\in\mathrm{Br}_{n}^{+}$, the front projection $\Pi_{x,z}(\Lambda(\beta))$ of the associated Legendrian link $\Lambda(\beta)\subset(\mathbb{R}^{3},\xi_{\mathrm{std}})$ is generic. Bearing this in mind, we henceforth restrict our discussion to Legendrian links with generic front projections.

Having introduced the stratifications of $\mathbb{R}^{2}$ relevant to our discussion, we now turn to the definition of a binary \emph{Maslov potential} for any Legendrian link in $(\mathbb{R}^{3},\xi_{\mathrm{std}})$ arising as the rainbow closure of a positive braid word. In particular, this binary Maslov potential will constitute an essential ingredient in the combinatorial description of the objects of the categories of interest.

\subsection{A Binary Maslov Potential for the Rainbow Closure of Positive Braids} Let $\beta\in\mathrm{Br}^{+}_{n}$ be a positive braid word, and let $\Lambda(\beta)\subset(\mathbb{R}^{3},\xi_{\mathrm{std}})$ be its associated Legendrian link. In this subsection, our goal is to equip $\Pi_{x,z}(\Lambda(\beta))\subset\mathbb{R}^{2}$ with a binary Maslov potential; specifically, a locally constant function $\mu_{\mathrm{Maslov}}:\Pi_{x,z}(\Lambda(\beta))\rightarrow \mathbb{Z}$ satisfying certain combinatorial properties.

To begin, observe that, given a generic positive braid word $\beta\in\mathrm{Br}^{+}_{n}$, the front projection $\Pi_{x,z}(\Lambda(\beta))\subset \mathbb{R}^{2}$ of the Legendrian link $\Lambda(\beta)\subset(\mathbb{R}^{3},\xi_{\mathrm{std}})$ is illustrated in Figure~\eqref{Front diagram of the rainbow closure of a braid in n strands}. In this article, we call \emph{strands} the smooth curves connecting a left cusp to a right cusp in $\Pi_{x,z}(\Lambda(\beta))$. It follows that, for each pair of left and right cusps $c_{l}$ and $c_{r}$ in $\Pi_{x,z}(\Lambda(\beta))$, exactly two strands meet at this pair of cusps. Consequently, we denote by $L_{c_{l},c_{r}}^{-}$ and $L_{c_{l},c_{r}}^{+}$ the bottom and top strands meeting at $c_{l}$ and $c_{r}$, respectively. Having established these notions, we define a binary Maslov potential $\mu_{\mathrm{Maslov}}:\Pi_{x,z}(\Lambda(\beta))\rightarrow \mathbb{Z}$ via the assignment $\mu_{\mathrm{Maslov}}(L_{c_{l},c_{r}}^{-})=0$ and $\mu_{\mathrm{Maslov}}(L_{c_{l},c_{r}}^{+})=1$, for every pair of left and right cusps $c_{l}$ and $c_{r}$ in $\Pi_{x,z}(\Lambda(\beta))$. 

In particular, by closely inspecting the front projection $\Pi_{x,z}(\Lambda(\beta))$ of the Legendrian link $\Lambda(\beta)$, illustrated in Figure~\eqref{Front diagram of the rainbow closure of a braid in n strands}, we observe that $\Pi_{x,z}(\Lambda(\beta))$ has $n$ left cusps, $n$ right cusps, and $2n$ strands. Then, according to our previous discussion, we have that: 
\begin{itemize}
\justifying
\item  The $n$ strands forming the braid diagram of $\beta$ in the front projection $\Pi_{x,z}(\Lambda(\beta))$ have Maslov potential $0$. These strands are referred to as the \textit{bottom strands} in $\Pi_{x,z}(\Lambda(\beta))$.

\item  The $n$ strands comprising the \textit{rainbow-like shape} that closes the braid diagram of $\beta$ to form the front projection $\Pi_{x,z}(\Lambda(\beta))$ have Maslov potential $1$. These strands are referred to as the \textit{top strands} in $\Pi_{x,z}(\Lambda(\beta))$.
\end{itemize}

Having introduced a binary Maslov potential for the family of Legendrian links under study, we now review the combinatorial description of the category $\mathcal{S}h_{r}(\Lambda, \mathbb{K})_{0}$, which constitutes a robust Legendrian isotopy invariant of a given Legendrian link $\Lambda\subset(\mathbb{R}^{3},\xi_{\mathrm{std}})$~\cite{STZ1}. 

\subsection{Combinatorial Description of the Category \texorpdfstring{$\mathcal{S}h_{r}(\Lambda,\mathbb{K})_{0}$}{Sh}} Let $\Lambda\subset (\mathbb{R}^{3},\xi_{\mathrm{std}})$ be a Legendrian link. Given a field $\mathbb{K}$ and an integer $r\geq 1$, we denote by $\mathcal{S}h_{r}(\Lambda, \mathbb{K})_{0}$ the dg-derived category of microlocal rank $r$, compactly supported, constructible sheaves of $\mathbb{K}$-modules on $\mathbb{R}^{2}$ whose singular support is contained in the closed conic Lagrangian $L(\Lambda)\subset (T^{*}\mathbb{R}^{2}, \omega_{\mathrm{std}})$ associated with $\Lambda$ (see Construction~\eqref{Def: closed conic Lagrangian associated with a Legendrian}). From the perspective of contact topology, the category $\mathcal{S}h_{r}(\Lambda,\mathbb{K})_{0}$ is of fundamental relevance: more precisely, the work of Guillermou, Kashiwara, and Schapira~\cite{GKS1} establishes that it is a Legendrian isotopy invariant of $\Lambda$. Moreover, Shende, Treumann, and Zaslow~\cite{STZ1} proved that when the front projection $\Pi_{x,z}(\Lambda)\subset\mathbb{R}^{2}$ is generic and equipped with a binary Maslov potential, the category $\mathcal{S}h_{r}(\Lambda,\mathbb{K})_{0}$ admits an explicit local combinatorial description. In this subsection, our goal is to briefly review this combinatorial characterization of such a microlocal-sheaf categorical invariant of $\Lambda$. Accordingly, we henceforth assume that $\Lambda$ is a Legendrian link whose front projection $\Pi_{x,z}(\Lambda)$ is generic and equipped with a binary Maslov potential. 

To begin, let $\mathcal{S}_{\Lambda}$ denote the regular stratification of $\mathbb{R}^{2}$ induced by $\Lambda$. As shown in \cite{STZ1}, when $\Lambda$ carries a binary Maslov potential, the objects of the category $\mathcal{S}h_{r}(\Lambda,\mathbb{K})_{0}$ are quasi-isomorphic to their zeroth cohomology sheaf, and therefore can be represented by honest sheaves of $\mathbb{K}$-modules. This simplification is particularly convenient, as it enables us to work with sheaves of $\mathbb{K}$-modules rather than complexes of such sheaves. Next, let $\sh{F}$ be an object of the category $\mathcal{S}h_{r}(\Lambda,\mathbb{K})_{0}$. Then, according to \cite{STZ1}, $\sh{F}$ is fully characterized by the \emph{microlocal support conditions}---namely, the set of conditions relative to the stratification $\mathcal{S}_{\Lambda}$ that determine $\sh{F}$ according to the following local models: 

\begin{itemize}
\justifying
\item[a)] \textit{Arcs}: Let $a$ be an arc in the stratification $\mathcal{S}_{\Lambda}$. Locally near $a$, $\mathcal{S}_{\Lambda}$ consists of the arc $a$, an upper two-dimensional stratum $U$, and a lower two-dimensional stratum $D$, as illustrated in Sub-figure~\eqref{sub-fig: Strata near an arc}. With respect to this local configuration, the behavior of $\sh{F}$ is described by the diagram in Sub-figure~\eqref{sub-fig: microsupport near an arc}.

\begin{figure}[ht]
\centering
\begin{subfigure}[b]{.45\textwidth}
\centering
\begin{tikzpicture}[framed, background rectangle/.style={fill=white,inner sep=2pt, draw=white}, thick]
     \useasboundingbox (-2,-1.5) rectangle (4,4);
     \scope[transform canvas={scale=0.7}]
     
\draw[very thick] (-2,1) .. controls (-1.2,1) and (-0.2,2) .. (1.5,2) .. controls (3.2,2) and (4.2,1) .. (5,1);
\node[below] at (1.5, 2.8) {\footnotesize\Large $a$};
\node[below] at (1.5, 4.8) {\footnotesize\Large $U$};
\node[below] at (1.5, 0.5) {\footnotesize\Large $D$};
     
     \endscope
\end{tikzpicture}
\vspace{10pt}
\caption{Strata configuration near an arc.}
\label{sub-fig: Strata near an arc}
\end{subfigure}%
\begin{subfigure}[b]{.45\textwidth}
\centering
\begin{tikzpicture}[framed, background rectangle/.style={fill=white,inner sep=2pt, draw=white}, thick]
\useasboundingbox (-3,-3) rectangle (3,2.5);
\scope[transform canvas={scale=0.7}]
\centering
\node[below] (a) at (0, 3) {\footnotesize\Large $\sh{F}(U)$};
\node[below] (b) at (0, 0) {\footnotesize\Large $\sh{F}(s(a))$};
\node[below] (c) at (0, -3) {\footnotesize\Large $\sh{F}(D)$};
\draw[->] (b) -- (a);
\draw[->] (b) -- (c) node[midway, right] {\footnotesize\Large $\cong$};
\endscope
\end{tikzpicture}
\vspace{10pt}
\caption{Micro-support condition near an arc.}
\label{sub-fig: microsupport near an arc}
\end{subfigure}
\caption{Strata configuration (left) and micro-support condition (right) near an arc.}
\end{figure}
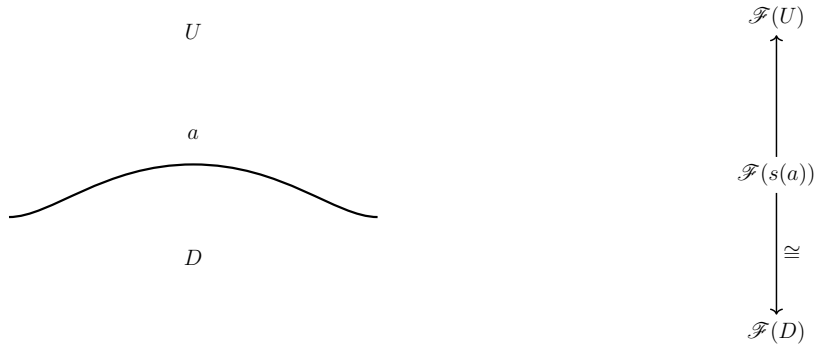   

\item[b)] \textit{Cusps}: Let $c$ be a cusp in the stratification $\mathcal{S}_{\Lambda}$. Locally near $c$, $\mathcal{S}_{\Lambda}$ consists of the cusp $c$, an upper arc $u$, a lower arc $d$, an inside region $I$, and an outside region $O$, as shown in Sub-figure~\eqref{sub-fig: strata near a cusp}. With respect to this local configuration, the behavior of $\sh{F}$ is characterized by the commutative diagram in Sub-figure~\eqref{sub-fig: microsupport near a cusp}.

\begin{figure}[ht]
    \centering
\begin{subfigure}[b]{.4\textwidth}
\centering
\begin{tikzpicture}[framed, background rectangle/.style={fill=white,inner sep=2pt, draw=white}, thick]
\useasboundingbox (-3.5,-2.5) rectangle (3,3);
\scope[transform canvas={scale=0.5}]
\draw[very thick] (-2.5,0) .. controls (1,1) and (1,4) .. (2.5,5);
\draw[very thick] (-2.5,0) .. controls (-0.5,-0.5) and (-0.5,-3.1) .. (1,-3);
\node[below] at (-3, 0.3) {\footnotesize\LARGE $c$};
\node[below] at (-5.5, 0.4) {\footnotesize\huge $O$};
\node[below] at (0.5, 0.4) {\footnotesize\huge $I$};
\node[below] at (0.5, 4) {\footnotesize\LARGE $u$};
\node[below] at (0.5, -4) {\footnotesize\LARGE $d$};
\endscope
\end{tikzpicture}    
\vspace{10pt}
\caption{Strata configuration near a cusp.}
\label{sub-fig: strata near a cusp}
\end{subfigure}%
\begin{subfigure}[b]{0.5\textwidth}
\centering
\begin{tikzpicture}[framed, background rectangle/.style={fill=white,inner sep=2pt, draw=white}, thick]
\useasboundingbox (-6,-3) rectangle (1.1,2.5);
\scope[transform canvas={scale=0.65}]
     
\node[below] (a) at (0, 3) {\footnotesize\Large $\sh{F}(s(u))$};
\node[below] (b) at (0, 0) {\footnotesize\Large $\sh{F}(I)$};
\node[below] (c) at (0, -3) {\footnotesize\Large $\sh{F}(s(d))$};
\node[below] (d) at (-4, 0) {\footnotesize\Large $\sh{F}(s(c))$};
\node[below] (e) at (-8, 0) {\footnotesize\Large $\sh{F}(O)$};
\draw[->] (a) -- (b) node[midway, right] {\footnotesize\Large $\cong$};
\draw[->] (c) -- (b) ;
\draw[->] (d) -- (a) ;
\draw[->] (d) -- (b) ;
\draw[->] (d) -- (c) node[midway, above right] {\footnotesize\Large $\cong$};
\draw[->] (d) -- (e) node[midway, above] {\footnotesize\Large $\cong$};
\draw[->] (-1,2.6) .. controls (-2.5,2.7) and (-6,2.3) .. (-8,0);
\draw[->] (-1,-3.4) .. controls (-2.5,-3.5) and (-6,-3.1) .. (-8,-0.8);
\node[below] at (-5,-3) {\footnotesize\Large $\cong$};          
\endscope
\end{tikzpicture}
\vspace{10pt}
\caption{Micro-support condition near a cusp.}
\label{sub-fig: microsupport near a cusp}
\end{subfigure}
\caption{Strata configuration (left) and micro-support condition (right) near a cusp.}
\label{fig: sheaf near a cusp}
\end{figure}
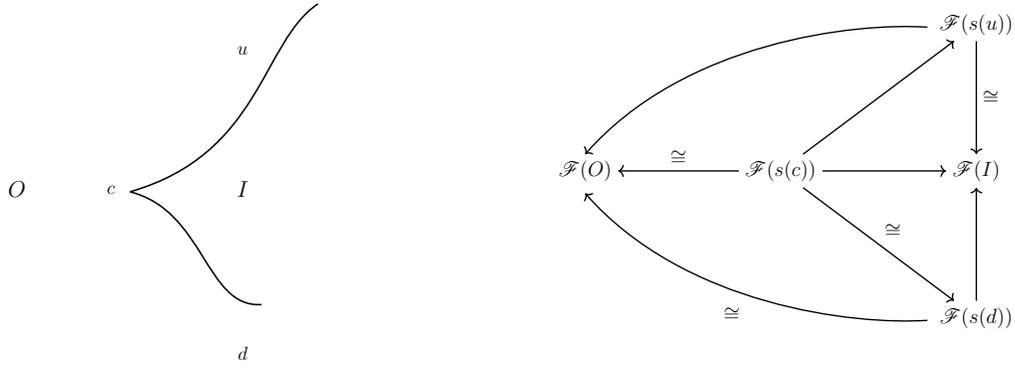

\item[c)] \textit{Crossings}: Let $x$ be a crossing in the stratification $\mathcal{S}_{\Lambda}$. Locally near $x$, $\mathcal{S}_{\Lambda}$ consists of the crossing $x$, four two-dimensional strata $N$, $S$, $E$, and $W$, along with four arcs $nw$, $ne$, $sw$, and $se$, as depicted in Sub-figure~\eqref{sub-fig: strata near a crossing}. With respect to this local configuration, the behavior of $\sh{F}$ is described by the commutative diagram in Sub-figure~\eqref{sub-fig: microsupport near a crossing}. Moreover, the sequence 
\begin{equation}\nonumber
0\longrightarrow\sh{F}(s(x)) \longrightarrow \sh{F}(s(nw))\oplus \sh{F}(s(ne)) \longrightarrow  \sh{F}(N)\longrightarrow 0,
\end{equation}
is required to be short exact.  

\begin{remark}
In the diagrams of Sub-figures~\eqref{sub-fig: microsupport near an arc}, \eqref{sub-fig: microsupport near a cusp}, and~\eqref{sub-fig: microsupport near a crossing}, all maps labeled with the symbol “$\,\cong\,$” are required to be isomorphisms. In this article, we further require these maps to be the identity maps, a convenient convention often adopted in the literature when working with concrete examples (see, for instance,~\cite{STZ1, CNS1, CW1}). 
\end{remark}

\begin{figure}[ht]
\centering
\begin{subfigure}[b]{.4\textwidth}
\centering    
\begin{tikzpicture}[framed, background rectangle/.style={fill=white,inner sep=4pt, draw=white}, thick]
     \useasboundingbox (-4,-4) rectangle (1,1);
     \scope[transform canvas={scale=0.5}]
\draw[very thick] (-6,-6) .. controls (-3,-6.1) and (-3,0.1) .. (0,0);
\draw[very thick] (-6,0) .. controls (-3,0.1) and (-3,-6.1) .. (0,-6);
\node[below] at (-3, -4) {\footnotesize\LARGE $x$};
\node[below] at (-3, 1.5) {\footnotesize\LARGE $N$};
\node[below] at (-3, -7) {\footnotesize\LARGE $S$};
\node[below] at (-6, -2.5) {\footnotesize\LARGE $W$};
\node[below] at (0, -2.5) {\footnotesize\LARGE $E$};
\node[below] at (-4, 0) {\footnotesize\LARGE $nw$};
\node[below] at (-2, 0) {\footnotesize\LARGE $ne$};
\node[below] at (-4, -5.5) {\footnotesize\LARGE $sw$};
\node[below] at (-2, -5.5) {\footnotesize\LARGE $se$};
\endscope
\end{tikzpicture}
\vspace{10pt}
\caption{Strata configuration near a crossing.}
\label{sub-fig: strata near a crossing}
\end{subfigure}%
\begin{subfigure}[b]{.5\textwidth}
\centering
\begin{tikzpicture}[framed, background rectangle/.style={fill=white,inner sep=2pt, draw=white}, thick]
\useasboundingbox (-3.8,-3.5) rectangle (3.8,3);
\scope[transform canvas={scale=0.5}]
\node[below] (a) at (0, 0) {\footnotesize\Large $\sh{F}(s(x))$};
\node[below] (b) at (0, 6) {\footnotesize\Large $\sh{F}(N)$};
\node[below] (c) at (0, -6) {\footnotesize\Large $\sh{F}(S)$};
\node[below] (d) at (-6, 0) {\footnotesize\Large $\sh{F}(W)$};
\node[below] (e) at (6, 0) {\footnotesize\Large $\sh{F}(E)$};
\node[below] (f) at (-3, 3) {\footnotesize\Large $\sh{F}(s(nw))$};
\node[below] (g) at (3, 3) {\footnotesize\Large $\sh{F}(s(ne))$};
\node[below] (h) at (-3, -3) {\footnotesize\Large $\sh{F}(s(sw))$};
\node[below] (i) at (3, -3) {\footnotesize\Large $\sh{F}(s(se))$};
\draw[->] (a)--(b);
\draw[->] (a)--(c) node[midway, right] {\footnotesize\Large $\cong$};
\draw[->] (a)--(d);
\draw[->] (a)--(e);
\draw[->] (a)--(f);
\draw[->] (a)--(g);
\draw[->] (a)--(h) node[midway, above left] {\footnotesize\Large $\cong$};
\draw[->] (a)--(i) node[midway, above right] {\footnotesize\Large $\cong$};
\draw[->] (f)--(b);
\draw[->] (f)--(d) node[midway, above left] {\footnotesize\Large $\cong$};
\draw[->] (g)--(b);
\draw[->] (g)--(e) node[midway, above right] {\footnotesize\Large $\cong$};
\draw[->] (h)--(d);
\draw[->] (h)--(c) node[midway, below left] {\footnotesize\Large $\cong$};
\draw[->] (i)--(e);
\draw[->] (i)--(c) node[midway, below right] {\footnotesize\Large $\cong$};
\endscope
\end{tikzpicture}
\vspace{10pt}
\caption{Micro-support condition near a crossing.}
\label{sub-fig: microsupport near a crossing}
\end{subfigure}
\caption{Strata configuration (left) and micro-support condition (right) near a crossing.}
\label{fig:sheaf near a crossing}
\end{figure}
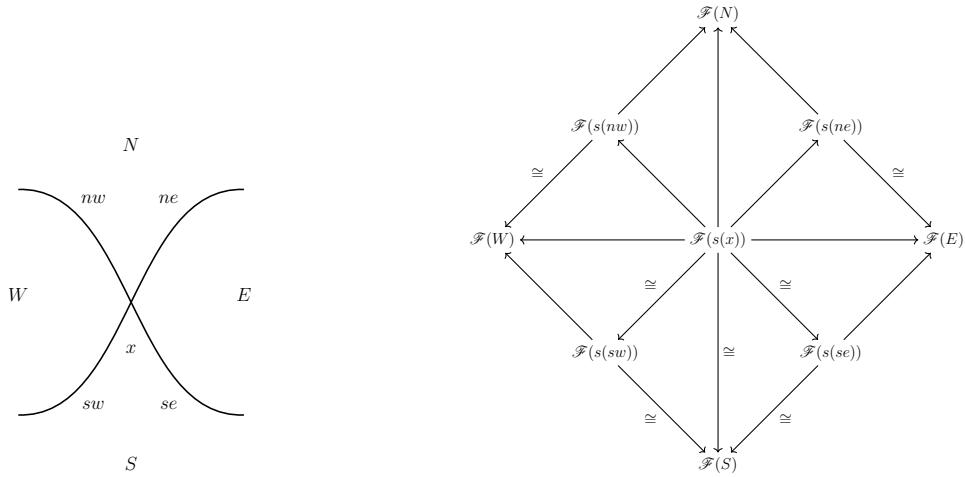

\item [d)] \textit{Microlocal Rank for Binary Maslov Potentials}: Let $r\geq 1$ be an integer, and let $a$ be an arc in the stratification $\mathcal{S}_{\Lambda}$. In particular, observe that locally near $a$, $\mathcal{S}_{\Lambda}$ consists of the arc $a$, an upper two-dimensional stratum $U$, and a lower two-dimensional stratum $D$, as illustrated in Sub-figure~\eqref{sub-fig: Strata near an arc}. Given this local configuration, let $p\in U$ and $q\in D$ be arbitrary points, and denote by $\sh{F}_{p}$ and $\sh{F}_{q}$ the stalks of $\sh{F}$ at $p$ and $q$, respectively. Then, depending on the binary Maslov potential, the microlocal-rank-$r$ condition impose the following constraints on the dimensions of these stalks:
\begin{itemize}
\justifying
\item If the arc $a$ belongs to a strand with Maslov potential $0$, we require $\mathrm{dim}_{\,\mathbb{K}}\,\sh{F}_{p}-\mathrm{dim}_{\,\mathbb{K}}\,\sh{F}_{q}=r$.
\item If the arc $a$ belongs to a strand with Maslov potential $1$, we require $\mathrm{dim}_{\,\mathbb{K}}\,\sh{F}_{q}-\mathrm{dim}_{\,\mathbb{K}}\,\sh{F}_{p}=r$. 
\end{itemize} 
In addition, let $U_{0}$ denote the unique unbounded two-dimensional stratum of $\mathcal{S}_{\Lambda}$. Then, the compact support condition establishes that for any $x\in U_{0}$, the stalk of $\sh{F}$ at $x$ must be trivial; that is, $\sh{F}_{x}=0$.  
\end{itemize}

With the local combinatorial description of the objects of the category $\mathcal{S}h_{r}(\Lambda,\mathbb{K})_{0}$ in place, we now introduce one of the core concepts of this manuscript: the cohomological category $H^{\bullet}(\mathcal{S}h_{r}(\Lambda,\mathbb{K})_{0})$. To this end, let $\sh{F}$ and $\sh{G}$ be objects of the category $\mathcal{S}h_{r}(\Lambda,\mathbb{K})_{0}$. Following~\cite{KS1}, we define $\mathrm{Ext}^{i}(\sh{F},\sh{G})$ to be the $i$-th sheaf cohomology of the right-derived internal Hom sheaf $R\sh{H}om(\sh{F},\sh{G})$; that is,
\begin{equation}
\mathrm{Ext}^{i}(\sh{F},\sh{G}):=H^{i}\big(R\Gamma(\mathbb{R}^{2};R\sh{H}om(\sh{F},\sh{G}))\big)\, ,
\end{equation}
for all $i\in\mathbb{Z}$. In particular, as shown in~\cite{STZ1}, when the front projection $\Pi_{x,z}(\Lambda)$ is generic and equipped with a binary Maslov potential, $\mathrm{Ext}^{i}(\sh{F},\sh{G})=0$ for all $i<0$. In this setting, the category $H^{\bullet}(\mathcal{S}h_{r}(\Lambda,\mathbb{K})_{0})$ admits the following presentation: 

\begin{definition}
Let $\Lambda\subset(\mathbb{R}^{3},\xi_{\mathrm{std}})$ be a Legendrian link whose front projection $\Pi_{x,z}(\Lambda)\subset\mathbb{R}^{2}$ is generic and equipped with a binary Maslov potential. Then, building on~\cite{STZ1}, the cohomological category $H^{\bullet}(\mathcal{S}h_{r}(\Lambda,\mathbb{K})_{0})$ is characterized by the following data:   
\begin{itemize}
\justifying
\item \textit{Objects}: The objects of the category $H^{\bullet}(\mathcal{S}h_{r}(\Lambda,\mathbb{K})_{0})$ are those of the category $\mathcal{S}h_{r}(\Lambda, \mathbb{K})_{0}$. 

\item \textit{Morphisms}: The morphism spaces are positively graded $\mathbb{K}$-modules. More precisely, let $\sh{F}$, $\sh{G}$ be objects of the category $H^{\bullet}(\mathcal{S}h_{r}(\Lambda,\mathbb{K})_{0})$. Then,
\begin{equation*}
\textit{Hom}_{\,H^{\bullet}(\mathcal{S}h_{r}(\Lambda,\mathbb{K})_{0})}(\sh{F},\sh{G}):=H^{\bullet}\big(\,\textit{Hom}_{\,\mathcal{S}h_{r}(\Lambda,\mathbb{K})_{0}}(\sh{F},\sh{G})\,\big)=\mathrm{Ext}^{\bullet}(\sh{F},\sh{G})=\bigoplus_{i\geq 0}\mathrm{Ext}^{i}(\sh{F},\sh{G})\, .     
\end{equation*}

\item \textit{Composition}: The morphism spaces are equipped with a graded composition. Specifically, let $\sh{F}$, $\sh{G}$, $\sh{H}$ be objects of the category $H^{\bullet}(\mathcal{S}h_{r}(\Lambda,\mathbb{K})_{0})$, and let $i,\,j\geq 0$ be integers. Then, the graded composition
\begin{equation*}
\circ:\textit{Hom}^{i}_{\,H^{\bullet}(\mathcal{S}h_{r}(\Lambda,\mathbb{K})_{0})}(\sh{G},\sh{H})\times \textit{Hom}^{j}_{\,H^{\bullet}(\mathcal{S}h_{r}(\Lambda,\mathbb{K})_{0})}(\sh{F},\sh{G})\longrightarrow \textit{Hom}^{i+j}_{\,H^{\bullet}(\mathcal{S}h_{r}(\Lambda,\mathbb{K})_{0})}(\sh{F},\sh{H}) \, ,       
\end{equation*}
is defined by the Yoneda composition on \textrm{Ext} groups (see, for instance,~\cite{HS1}). 
\end{itemize} 
\end{definition}

With the above definition at hand, we conclude our brief review of the microlocal theory of Legendrian links and turn our attention to the main goal of this manuscript. More precisely, given a positive braid word  $\beta\in \mathrm{Br}^{+}_{n}$, the remainder of the paper is devoted to providing an explicit and computable characterization of the category $H^{\bullet}(\mathcal{S}h_{1}(\Lambda(\beta),\mathbb{K})_{0})$. In the next section, we present both a geometric and an algebraic description of the objects of this category.

%% file: sec4.tex
\section{\texorpdfstring{The Objects of the Category $\ccs{1}{\beta}$}{The Objects of the Category HS}}\label{sec:objects}
\noindent
Let $\beta\in\mathrm{Br}_{n}^{+}$ be a positive braid word. Then, given a fixed ground field $\mathbb{K}$, the main purpose of this section is to provide both a geometric and an algebraic characterization of the objects of the category $H^{\bullet}(\mathcal{S}h_{1}(\Lambda(\beta),\mathbb{K})_{0})$. To this end, we first establish some technical definitions and lemmas concerning linear maps and complete flags, which will be instrumental in the discussion ahead. 

\subsection{Some Technical Results on Linear Maps and Complete Flags}
Let $\beta\in\mathrm{Br}_{n}^{+}$ be a positive braid word. In this subsection, we collect several technical lemmas concerning linear maps subject to specific constraints. Furthermore, we introduce some relevant definitions and properties of complete flags in finite-dimensional vector spaces. The importance of the concepts we present below lies in the fact that they provide the basic framework for the geometric and algebraic description of the objects of the category $\ccs{1}{\beta}$.

We begin by introducing the following lemma, which will be fundamental in studying the objects of the category $\ccs{1}{\beta}$ near the cusps in the front projection of the Legendrian link $\Lambda(\beta)$.

\begin{lemma}[Cusp Condition]\label{Lemma: cusp condition}
Let $S$ and $N$ be finite-dimensional vector spaces over $\mathbb{K}$ such that $\mathrm{dim}_{\mathbb{K}}\,S=p$ and $\mathrm{dim}_{\mathbb{K}}\,N=q$, where $p$ and $q$ are two positive integers subject to the constraints $1\leq p \leq q-1$ and $q\geq 2$. Let $\phi:S\to N$ and $\psi:N\to S$ be linear maps such that $\psi\circ\phi=\mathrm{id}_{S}$. Then, the following statements hold:
\begin{itemize}
\justifying
\item[(\textit{i})] $\psi$ is surjective.
\item[(\textit{ii})] $\phi$ is injective.
\item[(\textit{iii})] $N=\mathrm{ker}\,\psi\,\oplus\,\mathrm{im}\,\phi$. 
\end{itemize}
\end{lemma}
\begin{proof}
First, we prove part (\textit{i}). To this end, let $s\in S$. Then, we obtain that $s=\mathrm{id}_{S}(s)=\psi(\phi(s))$, which shows that $s\in \mathrm{im}\,\psi$. As a result, we conclude that $S=\mathrm{im}\,\psi$. Next, we verify part (\textit{ii}). For this purpose, let $x\in\mathrm{ker}\,\phi\subset S$. Hence, we have that $x=\mathrm{id}_{S}(x)= \psi(\phi(x))=\psi(0)=0$. It follows that $\mathrm{ker}\,\phi=0$. Finally, we prove part (\textit{iii}). To this end, let $n\in\mathrm{ker}\,\psi \,\cap\, \mathrm{im}\,\phi \subset N$. By definition, we know that there exists $s\in S$ such that $n=\phi(s)$. In particular, we observe that $s=\mathrm{id}_{S}(s)=\psi(\phi(s))=\psi(n)=0$, which implies that $\mathrm{ker}\,\psi\,\cap\,\mathrm{im}\,\phi=0$. Furthermore, by parts (\textit{i}) and (\textit{ii}) above, we know that $\psi$ is surjective and $\phi$ is injective. Therefore, by the rank-nullity theorem, we get that $\mathrm{dim}_{\,\mathbb{K}}\,\mathrm{ker}\,\psi=q-p$ and $\mathrm{dim}_{\,\mathbb{K}}\,\mathrm{im}\,\phi=p$. Bearing this in mind, we deduce that $N=\mathrm{ker}\,\psi\,\oplus\,\mathrm{im}\,\phi$.  
\end{proof}

We now present a lemma that will play a key role in the description of the objects of the category $\ccs{1}{\beta}$ near the crossings in the front projection of the Legendrian link $\Lambda(\beta)$.

\begin{lemma}[Crossing Condition]\label{Lemma: Crossing condition for linear maps}
Let $S$, $W$, $E$, and $N$ be vector spaces over $\mathbb{K}$. In addition, suppose that $\alpha_{1}:S\to W$, $\beta_{1}:W\to N$, $\alpha_{2}:S\to E$, and $\beta_{2}:E\to N$ are a collection of linear maps subject to the following constraints:
\begin{itemize}
\item The diagram in Figure \eqref{fig: Commutative diagram for the crossing condition} commutes.

\begin{figure}
\centering
\begin{tikzpicture}
\useasboundingbox (-2.25,-2.25) rectangle (2.25,2.25);
\scope[transform canvas={scale=0.75}]

\node[above] at (0, 2) {\Large$N$};
\node[below] at (0, -2) {\Large$S$};
\node at (-1.5, 0) {\Large$W$};
\node at (1.5, 0) {\Large$E$};

\draw[->, shorten <=0.5cm,  shorten >=0.4cm, line width=0.03cm] (-1.75, 0) -- (0,2.25);
\draw[->, shorten <=0.5cm,  shorten >=0.4cm, line width=0.03cm] (1.75, 0) -- (0,2.25);
\draw[->, shorten <=0.4cm,  shorten >=0.5cm, line width=0.03cm] (0,-2.25) -- (-1.75, 0);
\draw[->, shorten <=0.4cm,  shorten >=0.5cm, line width=0.03cm] (0,-2.25) -- (1.75, 0);

\node at (-1.4, 1.3) {\Large$\beta_{1}$};
\node at (1.4, 1.3) {\Large$\beta_{2}$};
\node at (-1.4, -1.3) {\Large$\alpha_{1}$};
\node at (1.4, -1.3) {\Large$\alpha_{2}$};

\endscope
\end{tikzpicture}
\caption{Commutative diagram for a collection of linear maps $\alpha_{1}$, $\beta_{1}$, $\alpha_{2}$, and $\beta_{2}$ satisfying the crossing condition.}
\label{fig: Commutative diagram for the crossing condition}
\end{figure}
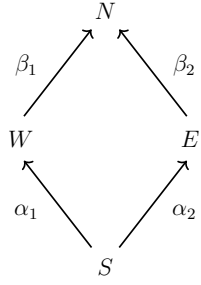

\item The sequence in Equation \eqref{Eq: exact sequence for the crossing condition} is short exact.
\begin{equation}\label{Eq: exact sequence for the crossing condition}
\begin{tikzpicture}
\useasboundingbox (-3,-0.25) rectangle (3,1);
\scope[transform canvas={scale=0.95}]
\node at (-5.5-1,0) {\large$0$}; 
\node at (-3-0.5,0) {\large$S$};    
\node at (0,0) {\large$W\oplus E$};   
\node at (3+0.5,0) {\large$N$};
\node at (5.5+1,0) {\large$0$}; 

\draw[->, shorten <=0.8cm,  shorten >=0.3cm, line width=0.02cm] (-6-1, 0) -- (-3-0.5,0);
\draw[->, shorten <=0.3cm,  shorten >=0.8cm, line width=0.02cm] (-3-0.5, 0) -- (0,0);
\draw[->, shorten <=0.8cm,  shorten >=0.3cm, line width=0.02cm] (0, 0) -- (3+0.5,0);
\draw[->, shorten <=0.3cm,  shorten >=0.8cm, line width=0.02cm] (3+0.5, 0) -- (6+1,0);

\node at (-1.5-0.25-0.25, 0.6) {\large$(\alpha_{1},\alpha_{2})$};
\node at (1.5+0.25+0.25, 0.6) {\large$\beta_{1}\oplus\big(-\beta_{2}\,\big)$};

\node at (5.8+1,-0.10) {\large$.$}; 

\endscope
\end{tikzpicture}    
\end{equation}  
\end{itemize}

\noindent
Then, the following statements hold:
\begin{itemize}
\justifying
\item[(\textit{i})] $\beta_{1}\circ \alpha_{1}=\beta_{2}\circ \alpha_{2}$.
\item[(\textit{ii})] $\mathrm{im}\big(\beta_{1}\circ \alpha_{1}\big)=\mathrm{im}\,\beta_{1}\cap\,\mathrm{im}\,\beta_{2}=\mathrm{im}\big(\beta_{2}\circ\alpha_{2}\big)$.
\item[(\textit{iii})] $N=\mathrm{im}\,\beta_{1}\,+\,\mathrm{im}\,\beta_{2}$.
\item[(\textit{iv})] If the maps $\alpha_{1}$ and $\beta_{1}$ are injective, then the maps $\alpha_{2}$ and $\beta_{2}$ are also injective.
\end{itemize}
\end{lemma}
\begin{proof}
First, we prove part (\textit{i}). To this end, observe that, by assumption, the diagram in Figure \eqref{fig: Commutative diagram for the crossing condition} commutes. In consequence, we obtain that $\beta_{1}\circ \alpha_{1}=\beta_{2}\circ \alpha_{2}$.

Next, we verify part (\textit{ii}). For this purpose, let $n\in \mathrm{im}\,\beta_{1}\cap\,\mathrm{im}\,\beta_{2}\subset N$. Then, there exist $w\in W$ and $e\in E$ such that $n=\beta_{1}(w)$ and $n=\beta_{2}(e)$. In particular, observe that $\beta_{1}(w)-\beta_{2}(e)=n-n=0$, which implies that $(w,e)\in \mathrm{ker}\,\big(\beta_{1}\oplus(-\beta_{2}\,)\big)$. Consequently, since the sequence in Equation~\eqref{Eq: exact sequence for the crossing condition} is short exact, there exists $s\in S$ such that $(w,e)=(\alpha_{1}(s), \alpha_{2}(s))$, which shows that $n=\beta_{1}(\alpha_{1}(s))$ and $n=\beta_{2}(\alpha_{2}(s))$. Therefore, by part (\textit{i}) above, we conclude that $\mathrm{im}\big(\beta_{1}\circ \alpha_{1}\big)=\mathrm{im}\,\beta_{1}\cap\,\mathrm{im}\,\beta_{2}=\mathrm{im}\big(\beta_{2}\circ\alpha_{2}\big)$.

Now, we prove part (\textit{iii}). To this end, let $n\in N$. Then, since the sequence in Equation~\eqref{Eq: exact sequence for the crossing condition} is short exact, there exist $w\in W$ and $e\in E$ such that $n=\beta_{1}(w)-\beta_{2}(e)$. It follows that $N=\mathrm{im}\,\beta_{1}\,+\,\mathrm{im}\,\beta_{2}$.

Finally, we verify part (\textit{iv}). For this purpose, let $x\in \mathrm{ker}\,\alpha_{2}\subset S$. Then, by part (\textit{i}) above, we have that $\beta_{1}(\alpha_{1}(x))=\beta_{2}(\alpha_{2}(x))=\beta_{2}(0)=0$. In particular, since the linear maps $\alpha_{1}$ and $\beta_{1}$ are assumed to be injective, we conclude that $x=0$, and therefore $\mathrm{ker}\,\alpha_{2}=0$. In addition, let $y\in \mathrm{ker}\,\beta_{2}\subset E$. Then, we have that $\beta_{1}(0)-\beta_{2}(y)=0-0=0$, which implies that $(0,y)\in\mathrm{ker}\big(\beta_{1}\oplus(-\beta_{2})\big)$. Thus, since the sequence in Equation~\eqref{Eq: exact sequence for the crossing condition} is short exact, there exists $s\in S$ such that $(\alpha_{1}(s), \alpha_{2}(s))=(0,y)$. Consequently, since the linear map $\alpha_{1}$ is assumed to be injective, we obtain that $s=0$. It follows that $y=\alpha_{2}(s)=\alpha_{2}(0)=0$. Bearing this in mind, we deduce that $\mathrm{ker}\,\beta_{2}=0$. This completes the proof.   
\end{proof}

Next, we introduce some key definitions and lemmas concerning complete flags in finite-dimensional vector spaces.

\begin{definition}\label{Def: flags}
Let $V$ be a finite-dimensional vector space over $\mathbb{K}$ with $\mathrm{dim}_{\,\mathbb{K}}V=n$, for some integer $n\geq 1$. Then, a complete flag $\flag{F}$ in $V$ is defined to be a nested sequence of vector subspaces $F^{(p)}$ of $\,V$ such that $\mathrm{dim}_{\,\mathbb{K}}F^{(p)}=p$ for all $p\in[0, n]$. With this definition in place, we adopt the following terminology:   
\begin{itemize}
\justifying    
\item Two complete flags $\flag{F}$ and $\flag{G}$ in $V$ are said to be completely opposite if and only if 
\begin{equation*}
V=F^{(p)}\oplus G^{(n-p)}\, ,    
\end{equation*}
for each $p\in[0, n]$.

\item Let $k\in[1,n-1]$ be an integer. Two complete flags $\flag{F}$ and $\flag{G}$ in $V$ are said to be in $s_{k}$-relative position if and only if $F^{(k)}\neq G^{(k)}$ and $F^{(p)}=G^{(p)}$ for each $p\neq k$, where $p\in[0,n]$. 

\item Let $\hat{\mathbf{x}}:=\big\{\hat{x}_{i}\big\}_{i=1}^{n}$ be a basis for $V$. We define the standard flag in $V$ relative to the basis $\hat{\mathbf{x}}$ to be the complete flag $\flagstd{F}{x}$ given by
\begin{equation*}
F_{\mathrm{std}}^{(p)}[ \hat{\mathbf{x}}]:=\begin{cases}
\hfil 0\, , & \text{if $~p=0$}\, \\
\big<\hat{x}_{1},\dots, \hat{x}_{p}\big>\, , & \text{if $~p\in[1, n]$}
\end{cases}\quad.    
\end{equation*}
Similarly, we define the anti-standard flag in $V$ relative to the basis $\hat{\mathbf{x}}$ to be the complete flag $\flagastd{F}{x}$ given by
\begin{equation*}
F_{\mathrm{astd}}^{(p)} [\hat{\mathbf{x}}]:=\begin{cases}
\hfil 0\, , & \text{if $~p=0$}\, ,\\
\big<\hat{x}_{n},\dots, \hat{x}_{n+1-p}\big>\, , & \text{if $~p\in[1, n]$}
\end{cases}\quad.      
\end{equation*}

\item Let $\hat{\mathbf{x}}:=\big\{\hat{x}_{i}\big\}_{i=1}^{n}$ be a basis for $V$, and let $A\in\mathrm{GL}(n, \mathbb{K})$ be an $n\times n$ invertible matrix over $\mathbb{K}$. We introduce $\big\{\vec{a}_{j}\big\}_{j=1}^{n}$ to denote the collection of vectors in $V$ given by $\vec{a}_{j}:=\sum_{i=1}^{n}A_{i,j}\,\hat{x}_{i}$ for each $j\in[1,n]$. In particular, let $\flag{F}$ be a complete flag in $V$ such that
\begin{equation*}
F^{(p)}=\begin{cases}
\hfil 0\, , & \text{if $~p=0$}\, ,\\
\big<\vec{a}_{1},\dots,\vec{a}_{p}\big>\, , & \text{if $~p\in[1,n]$}
\end{cases}\quad .    
\end{equation*}
Then, we say that, relative to the basis $\hat{\mathbf{x}}$ for $V$, the complete flag $\fl{F}$ is represented by the matrix $A$.
\end{itemize}
\end{definition}

\begin{notation}
Let $n,\,m\geq 1$ be integers. Throughout this manuscript, we implement the following standard matrices:
\begin{itemize}
\justifying
\item $\mathbf{1}_{n} \in \mathrm{GL}(n, \mathbb{K})$: the identity matrix of size $n\times n$.
\item $\mathbf{w}_{n}\in \mathrm{GL}(n, \mathbb{K})$: the anti-diagonal identity matrix of size $n\times n$; that is, the matrix with ones on the anti-diagonal and zeros elsewhere. 
\item $\mathbf{0}_{n\times m}\in M(n,m,\mathbb{K})$: the zero matrix of size $n\times m$.
\end{itemize}
Explicitly, 
\begin{equation*}
\mathbf{1}_{n}:=\begin{bmatrix}
1 & \cdots & 0\\
\vdots & \ddots & \vdots\\
0 & \cdots & 1
\end{bmatrix}_{n\times n}\, ,    \qquad \quad \mathbf{w}_{n}:=\begin{bmatrix}
0 & \cdots & 1\\
\vdots & \iddots & \vdots\\
1 & \cdots & 0
\end{bmatrix}_{n\times n}\, , \qquad \quad \mathbf{0}_{n\times m}:=\begin{bmatrix}
0 & \cdots & 0\\
\vdots & \ddots & \vdots\\
0 & \cdots & 0
\end{bmatrix}_{n\times m }\, .       
\end{equation*}    

In addition, for any pair of integers $p,q\geq 1$ with $q>p$, we denote by $\iota^{(q,p)}\in M(q,p,\mathbb{K})$ the standard inclusion matrix, and by $\pi^{(p,q)}\in M(p,q,\mathbb{K})$ the standard projection matrix, defined by 
\begin{equation*}
\iota^{(q,p)}:=\left[
\begin{array}{c}
\mathbf{1}_{p} \\
\hline
\mathbf{0}_{(q - p)\times p}
\end{array}
\right]\, , \qquad \qquad \pi^{(p,q)}:=\left[
\begin{array}{c|c}
\mathbf{1}_{p} & \mathbf{0}_{p \times (q - p)}
\end{array}
\right]\, .
\end{equation*}
Accordingly, for all $k\geq 2$, the following identities hold:
\begin{equation*}
\pi^{(1,2)}\cdots\pi^{(k-1,k)}=\pi^{(1,k)}\, , \qquad\qquad  \pi^{(p,q)}\cdot \iota^{(q,p)}=\mathbf{1}_{p}\, , \qquad\qquad \iota^{(k,k-1)}\cdots\iota^{(2,1)}=\iota^{(k,1)}\, .     
\end{equation*}
\end{notation}

\begin{remark}
Let $V$ be a finite-dimensional vector space over $\mathbb{K}$ with $\mathrm{dim}_{\,\mathbb{K}}V=n$, for some integer $n\geq 1$. Furthermore, suppose that $\hat{\mathbf{x}}:=\big\{\hat{x}_{i}\big\}_{i=1}^{n}$ is a basis for $V$. Then, according to Definition~\eqref{Def: flags}, we have that:
\begin{itemize}
\justifying
\item Relative to the basis $\hat{\mathbf{x}}$ for $V$, the standard flag $\sfl[\hat{\mathbf{x}}]$ in $V$ is represented by the matrix $\mathbf{1}_{n}\in \mathrm{GL}(n,\mathbb{K})$.
\item Relative to the basis $\hat{\mathbf{x}}$ for $V$, the anti-standard flag $\asfl[\hat{\mathbf{x}}]$ in $V$ is represented by the matrix $\mathbf{w}_{n}\in \mathrm{GL}(n,\mathbb{K})$.
\end{itemize}
\end{remark}

With the basic notions of complete flags and standard matrices in place, we now formalize how complete flags arise naturally from collections of vector spaces and linear maps. Moreover, we introduce the notion of bases adapted to such collections, which provide concrete matrix realizations of the resulting flags.

\begin{definition}\label{Def:flags and adapted bases}
Let $n\geq 2$ be an integer, and let $\big\{V^{(i)}\big\}_{i=1}^{n}$ be a collection of vector spaces over $\mathbb{K}$ such that $\dim_{\,\mathbb{K}}V^{(i)}=i$ for all $i\in[1,n]$. Let $\big\{\phi^{(i)}:V^{(i)}\to V^{(i+1)}\big\}_{i=1}^{n-1}$ and $\big\{\psi^{(i)}:V^{(i+1)}\to V^{(i)}\big\}_{i=1}^{n-1}$ be collections of injective and surjective linear maps, respectively. Then, we introduce the following definitions: 

\begin{enumerate}[leftmargin=*]
\justifying
\item \label{Def: type I flag} \textbf{\textit{Type $\mathcal{I}$ Flag}}: Consider the collection $\big\{\phi^{(i)}\big\}_{i=1}^{n-1}$. We associate to this data the filtration in $V^{(n)}$ given by
\begin{equation*}
\prescript{}{\mathcal{I}\,}{\fl{F}}(\phi^{(1)},\dots,\phi^{(n-1)}):=\big\{ F^{(0)}\subset \cdots \subset F^{(n)} \big\}\, ,    
\end{equation*}
where
\begin{equation}
F^{(p)}:=\begin{cases}
\hfil 0, & \text{if $~p=0$}\,,\\
\mathrm{im}\big( \phi^{(n-1)}\circ\cdots\circ\phi^{(p)} \big), & \text{if $~p\in[1,n-1]$}\,,\\
\hfil V^{(n)}, & \text{if $~p=n$}\, .
\end{cases}
\end{equation}
We refer to $\prescript{}{\mathcal{I}\,}{\fl{F}}(\phi^{(1)},\dots,\phi^{(n-1)})$ as the \emph{type $\mathcal{I}$ flag} in $V^{(n)}$ associated with $\big\{\phi^{(i)}\big\}_{i=1}^{n-1}$. 

\item \label{Def: type K flag} \textbf{\textit{Type $\mathcal{K}$ Flag}}: Consider the collection $\big\{\psi^{(i)}\big\}_{i=1}^{n-1}$. We associate to this data the filtration in $V^{(n)}$ given by
\begin{equation*}
\prescript{}{\mathcal{K}\,}{\fl{F} }(\psi^{(1)},\dots,\psi^{(n-1)}):=\big\{ F^{(0)}\subset \cdots \subset F^{(n)} \big\}\, ,    
\end{equation*} 
where
\begin{equation}
F^{(p)}:=\begin{cases}
\hfil 0, & \text{if $p=0$}\,,\\
\mathrm{ker}\big( \psi^{(n-p)}\circ\cdots\circ\psi^{(n-1)} \big), & \text{if $p\in[1, n-1]$}\,,\\
\hfil V^{(n)}, & \text{if $p=n$}\, .
\end{cases}
\end{equation}
We refer to $\prescript{}{\mathcal{K}\,}{\fl{F}}(\psi^{(1)},\dots,\psi^{(n-1)})$ as the \emph{type $\mathcal{K}$ flag} in $V^{(n)}$ associated with $\big\{\psi^{(i)}\big\}_{i=1}^{n-1}$. 

\item \label{Def: adapted bases I} \textbf{\textit{Adapted System of Bases I}}: Consider the collection $\big\{\phi^{(i)}\big\}_{i=1}^{n-1}$. For each $i\in[1,n]$, let $\hat{\mathbf{f}}^{(i)}:=\big\{\hat{f}^{(i)}_{k}\big\}_{k=1}^{i}$ be a basis for $V^{(i)}$, and denote by $\big\{\hat{\mathbf{f}}^{(i)}\big\}_{i=1}^{n}$ the collection of such bases. We say that $\big\{\hat{\mathbf{f}}^{(i)}\big\}_{i=1}^{n}$ is a \emph{system of bases adapted} to $\big\{\phi^{(i)}\big\}_{i=1}^{n-1}$ if, for each $i\in[1,n-1]$, the matrix $\tensor[_{\hat{\mathbf{f}}^{(i+1)}}]{ \big[\, \phi^{\,(i)}\,\big] }{_{\hat{\mathbf{f}}^{(i)}}}\in M(i+1,i,\mathbb{K})$ representing $\phi^{(i)}$ is the standard inclusion matrix; namely, 
\begin{equation*}
\tensor[_{\hat{\mathbf{f}}^{(i+1)}}]{ \big[\, \phi^{\,(i)}\,\big] }{_{\hat{\mathbf{f}}^{(i)}}}=\iota^{(i+1,i)}\, .    
\end{equation*}

\item \label{Def: adapted bases II} \textbf{\textit{Adapted System of Bases II}}: Consider the collection $\big\{\psi^{(i)}\big\}_{i=1}^{n-1}$. For each $i\in[1,n]$, let $\hat{\mathbf{f}}^{(i)}:=\big\{\hat{f}^{(i)}_{k}\big\}_{k=1}^{i}$ be a basis for $V^{(i)}$, and denote by $\big\{\hat{\mathbf{f}}^{(i)}\big\}_{i=1}^{n}$ the collection of such bases. We say that $\big\{\hat{\mathbf{f}}^{(i)}\big\}_{i=1}^{n}$ is a \emph{system of bases adapted} to $\big\{\psi^{(i)}\big\}_{i=1}^{n-1}$ if, for each $i\in[1,n-1]$, the matrix $\tensor[_{\hat{\mathbf{f}}^{(i)}}]{ \big[\, \psi^{\,(i)}\,\big] }{_{\hat{\mathbf{f}}^{(i+1)}}}\in M(i,i+1,\mathbb{K})$ representing $\psi^{(i)}$ is the standard projection matrix; specifically, 
\begin{equation*}
\tensor[_{\hat{\mathbf{f}}^{(i)}}]{ \big[\, \psi^{\,(i)}\,\big] }{_{\hat{\mathbf{f}}^{(i+1)}}}=\pi^{(i,i+1)}\, .
\end{equation*}
\end{enumerate}
\end{definition}

\begin{lemma}\label{Lemma: type I flag}
Let $n\geq2$ be an integer, and let $\big\{V^{(i)}\big\}_{i=1}^{n}$ be a collection of vector spaces over $\mathbb{K}$ such that $\dim_{\,\mathbb{K}}V^{(i)}=i$ for all $i\in[1,n]$. In addition, let $\big\{\phi^{(i)}:V^{(i)}\to V^{(i+1)}\big\}_{i=1}^{n-1}$ be a collection of injective linear maps. Then the type $\mathcal{I}$ flag $\prescript{}{\mathcal{I}\,}{\fl{F}}(\phi^{(1)},\dots,\phi^{(n-1)})$ in $V^{(n)}$ associated with $\big\{\phi^{(i)}\big\}_{i=1}^{n-1}$ is a complete flag.  
\end{lemma}
\begin{proof}
By Definition~\eqref{Def:flags and adapted bases}--\eqref{Def: type I flag}, the filtration $\prescript{}{\mathcal{I}\,}{\fl{F}}(\phi^{(1)},\dots,\phi^{(n-1)}):=\big\{ F^{(0)}\subset \cdots \subset F^{(n)} \big\}$ in $V^{(n)}$ is given by 
\begin{equation*}
F^{(p)}:=\begin{cases}
\hfil 0, & \text{if $~p=0$}\,,\\
\mathrm{im}\big( \phi^{(n-1)}\circ\cdots\circ\phi^{(p)} \big), & \text{if $~p\in[1,n-1]$}\,,\\
\hfil V^{(n)}, & \text{if $~p=n$}\, .
\end{cases}
\end{equation*}
In particular, observe that since each $\phi^{(i)}$ is injective, the composition $\phi^{(n-1)}\circ \cdots\circ \phi^{(p)}:V^{(p)}\to V^{(n)}$ is also injective for every $p\in[1,n-1]$. Consequently, by the rank-nullity theorem, we deduce that 
\begin{equation*}
\mathrm{dim}_{\,\mathbb{K}}\,\mathrm{im}\big(\phi^{(n-1)}\circ \cdots\circ \phi^{(p)}\big)=p\, ,    
\end{equation*}
for all $p\in[1,n-1]$, which shows that $\prescript{}{\mathcal{I}\,}{\fl{F}}(\phi^{(1)},\dots,\phi^{(n-1)})$ is a complete flag.     
\end{proof}

\begin{lemma}\label{Lemma: type K flag}
Let $n\geq 2$ be an integer, and let $\big\{V^{(i)}\big\}_{i=1}^{n}$ be a collection of vector spaces over $\mathbb{K}$ such that $\dim_{\,\mathbb{K}}V^{(i)}=i$ for all $i\in[1,n]$. In addition, let $\big\{\psi^{(i)}:V^{(i+1)}\to V^{(i)}\big\}_{i=1}^{n-1}$ be a collection of surjective linear maps. Then the type $\mathcal{K}$ flag $\prescript{}{\mathcal{K}\,}{\fl{F}}(\psi^{(1)},\dots,\psi^{(n-1)})$ in $V^{(n)}$ associated with $\big\{\psi^{(i)}\big\}_{i=1}^{n-1}$ is a complete flag.
\end{lemma}
\begin{proof}
By Definition~\eqref{Def:flags and adapted bases}--\eqref{Def: type K flag}, the filtration $\prescript{}{\mathcal{K}\,}{\fl{F}}(\psi^{(1)},\dots,\psi^{(n-1)}):=\big\{ F^{(0)}\subset \cdots \subset F^{(n)} \big\}$ in $V^{(n)}$ is given by
\begin{equation}
F^{(p)}:=\begin{cases}
\hfil 0, & \text{if $p=0$}\,,\\
\mathrm{ker}\big( \psi^{(n-p)}\circ\cdots\circ\psi^{(n-1)} \big), & \text{if $p\in[1, n-1]$}\,,\\
\hfil V^{(n)}, & \text{if $p=n$}\, .
\end{cases}
\end{equation}
In addition, note that since each $\psi^{(i)}$ is surjective, the composition $\psi^{(n-p)}\circ \cdots\circ \psi^{(n-1)}:V^{(n)}\to V^{(n-p)}$ is also surjective for every $p\in[1,n-1]$. Consequently, by the rank-nullity theorem, we deduce that
\begin{equation*}
\mathrm{dim}_{\,\mathbb{K}}\,\mathrm{ker}\big(\psi^{(n-p)}\circ \cdots\circ \psi^{(n-1)}\big)=p\, ,    
\end{equation*}
for all $p\in[1,n-1]$, which confirms that $\prescript{}{\mathcal{K}\,}{\fl{F}}(\psi^{(1)},\dots,\psi^{(n-1)})$ is a complete flag.     
\end{proof}

\begin{lemma}\label{Lemma: completely opposite flags from linear maps}
Let $n\geq 2$ be an integer, and let $\big\{V^{(i)}\big\}_{i=1}^{n}$ be a collection of vector spaces over $\mathbb{K}$ such that $\dim_{\,\mathbb{K}}V^{(i)}=i$ for all $i\in[1,n]$. Let $\big\{\psi^{(i)}:V^{(i+1)}\to V^{(i)}\big\}_{i=1}^{n-1}$ be a collection of surjective linear maps, and let $\big\{\phi^{(i)}:V^{(i)}\to V^{(i+1)}\big\}_{i=1}^{n-1}$ be a collection of injective linear maps satisfying:
\begin{equation*}
\psi^{(i)}\circ \phi^{(i)}=\mathrm{id}_{V^{(i)}}\, ,    
\end{equation*}
for each $i\in[1,n-1]$. In particular, denote by 
\begin{equation*}
\begin{aligned}
\fl{F}_{0}&:=\prescript{}{\mathcal{K}\,}{\fl{F}}(\psi^{(1)},\dots,\psi^{(n-1)})\, ,   \\[6pt] 
\fl{F}_{1}&:=\prescript{}{\mathcal{I}\,}{\fl{F}}(\phi^{(1)},\dots,\phi^{(n-1)})\, ,
\end{aligned}
\end{equation*}
the type $\mathcal{K}$ and type $\mathcal{I}$ flags in $V^{(n)}$ associated with $\big\{\psi^{(i)}\big\}_{i=1}^{n-1}$ and $\big\{\phi^{(i)}\big\}_{i=1}^{n-1}$, respectively. Then, $\fl{F}_{0}$ and $\fl{F}_{1}$ are completely opposite flags. 
\end{lemma}
\begin{proof}
By assumption, $\psi^{(i)}\circ \phi^{(i)}=\mathrm{id}_{V^{(i)}}$ for all $i\in [1,n-1]$. Thus, for each $i\in [1,n-1]$, the compositions $\psi^{(i)}\circ \cdots \circ \psi^{(n-1)}:V^{(n)}\to V^{(i)}$ and $\phi^{(n-1)}\circ \cdots \circ \phi^{(i)}:V^{(i)}\to V^{(n)}$ satisfy
\begin{equation*}
\Big(\psi^{(i)}\circ \cdots \circ \psi^{(n-1)}\Big) \circ \Big(\phi^{(n-1)}\circ \cdots \circ \phi^{(i)}\Big)=\mathrm{id}_{V^{(i)}}\, .     
\end{equation*}
Consequently, a direct application of Lemma~\eqref{Lemma: cusp condition} yields, for each $i\in[1,n-1]$, the direct sum decomposition
\begin{equation}\label{Eq: direct sum decomposition}
V^{(n)}=\mathrm{ker} \Big(\psi^{(i)}\circ \cdots \circ \psi^{(n-1)}\Big)\oplus \mathrm{im} \Big(\phi^{(n-1)}\circ \cdots \circ \phi^{(i)}\Big)\, .
\end{equation}

Now, consider
\begin{equation*}
\begin{aligned}
\fl{F}_{0}=\prescript{}{\mathcal{K}\,}{\fl{F}}(\psi^{(1)},\dots,\psi^{(n-1)})=\big\{ F^{(0)}_{0}\subset  \cdots \subset F^{(n)}_{0} \big\}\, , \\[4pt]
\fl{F}_{1}=\prescript{}{\mathcal{I}\,}{\fl{F}}(\phi^{(1)},\dots,\phi^{(n-1)})=\big\{ F^{(0)}_{1} \subset \cdots \subset F^{(n)}_{1} \big\}\, .  
\end{aligned}    
\end{equation*}
the type $\mathcal{K}$ and type $\mathcal{I}$ flags in $V^{(n)}$ associated with $\big\{\psi^{(i)}\big\}_{i=1}^{n-1}$ and $\big\{\phi^{(i)}\big\}_{i=1}^{n-1}$, respectively. By Definition~\eqref{Def:flags and adapted bases}--\eqref{Def: type I flag}--\eqref{Def: type K flag}, we have that
\begin{equation*}
\begin{aligned}
F^{(p)}_{0}&:=\begin{cases}
\hfil 0, & \text{if $p=0$}\,,\\
\mathrm{ker}\big( \psi^{(n-p)}\circ\cdots\circ\psi^{(n-1)} \big), & \text{if $p\in[1, n-1]$}\,,\\
\hfil V^{(n)}, & \text{if $p=n$}\, .
\end{cases}\\[8pt]  
F^{(p)}_{1}&:=\begin{cases}
\hfil 0, & \text{if $~p=0$}\,,\\
\mathrm{im}\big( \phi^{(n-1)}\circ\cdots\circ\phi^{(p)} \big), & \text{if $~p\in[1,n-1]$}\,,\\
\hfil V^{(n)}, & \text{if $~p=n$}\, .
\end{cases}
\end{aligned}    
\end{equation*}
Therefore, by the direct sum decompositions~\eqref{Eq: direct sum decomposition}, we deduce that
\begin{equation*}
V^{(n)}=F^{(p)}_{0}\oplus F^{(n-p)}_{1}\, ,    
\end{equation*}
for all $p\in[1,n-1]$, which shows that $\fl{F}_{0}$ and $\fl{F}_{1}$ are completely opposite flags, as desired. 
\end{proof}

\begin{lemma}\label{Lemma: I type flag in adapted bases for injective maps}
Let $n\geq 2$ be an integer, and let $\big\{V^{(i)}\big\}_{i=1}^{n}$ be a collection of vector spaces over $\mathbb{K}$ such that $\dim_{\,\mathbb{K}}V^{(i)}=i$ for all $i\in[1,n]$. Let $\big\{ \phi^{(i)}:V^{(i)}\to V^{(i+1)}\big\}_{i=1}^{n-1}$ be a collection of injective linear maps, and for each $i\in[1,n]$, let $\,\hat{\mathbf{f}}^{(i)}:=\big\{\hat{f}^{(i)}_{k}\big\}_{k=1}^{i}$ be a basis for $V^{(i)}$ such that $\big\{\hat{\mathbf{f}}^{(i)}\big\}_{i=1}^{n}$ is a system of bases adapted to $\big\{\phi^{(i)}\big\}_{i=1}^{n-1}$ (see Definition~\eqref{Def:flags and adapted bases}--\eqref{Def: adapted bases I}). Then, relative to the basis $\hat{\mathbf{f}}^{(n)}$ for $V^{(n)}$, the type $\mathcal{I}$ flag $\prescript{}{\mathcal{I}\,}{\fl{F} }(\phi^{(1)},\dots,\phi^{(n-1)})$ in $V^{(n)}$ associated with $\big\{\phi^{(i)}\big\}_{i=1}^{n-1}$ coincides with the standard flag, $\prescript{}{\mathcal{I}\,}{\fl{F} }(\phi^{(1)},\dots,\phi^{(n-1)})\big[\, \hat{\mathbf{f}}^{(n)}\,\big]=\sfl \big[\, \hat{\mathbf{f}}^{(n)}\,\big]$, and is therefore represented by the matrix $\mathbf{1}_{n}$.  
\end{lemma} 
\begin{proof}
By Definition~\eqref{Def:flags and adapted bases}--\eqref{Def: type I flag}, the filtration $\prescript{}{\mathcal{I}\,}{\fl{F} }(\phi^{(1)},\dots,\phi^{(n-1)}):=\big\{ F^{(0)}\subset \cdots \subset  F^{(n)} \big\}$ in $V^{(n)}$ is given by
\begin{equation*}
F^{(p)}:=\begin{cases}
\hfil 0, & \text{if $~p=0$}\,,\\
\mathrm{im}\big( \phi^{(n-1)}\circ\cdots\circ\phi^{(p)} \big), & \text{if $~p\in[1,n-1]$}\,,\\
\hfil V^{(n)}, & \text{if $~p=n$}\, .
\end{cases}
\end{equation*}
Now, since $\big\{\hat{\mathbf{f}}^{(i)}\big\}_{i=1}^{n}$ is a system of bases adapted to $\big\{\phi^{(i)}\big\}_{i=1}^{n-1}$, Definition~\eqref{Def:flags and adapted bases}--\eqref{Def: adapted bases I} asserts that for each $i\in[1,n-1]$, the matrix $\tensor[_{\hat{\mathbf{f}}^{(i+1)}}]{ \big[\, \phi^{\,(i)}\,\big] }{_{\hat{\mathbf{f}}^{(i)}}}\in M(i+1,i,\mathbb{K})$ representing $\phi^{\,(i)}$ is the standard inclusion matrix; that is, $\tensor[_{\hat{\mathbf{f}}^{(i+1)}}]{ \big[\, \phi^{\,(i)}\,\big] }{_{\hat{\mathbf{f}}^{(i)}}}=\iota^{(i+1,i)}$. As a result, we deduce that
\begin{equation*}
F^{(p)}=\left<\hat{f}^{(n)}_{1},\dots, \hat{f}^{(n)}_{p}\right>\, ,
\end{equation*}
for all $p\in [1,n-1]$, which confirms that $\prescript{}{\mathcal{I}\,}{\fl{F} }(\phi^{(1)},\dots,\phi^{(n-1)})\big[\, \hat{\mathbf{f}}^{(n)}\,\big]=\sfl \big[\, \hat{\mathbf{f}}^{(n)} \,\big]$, as claimed.   
\end{proof}

\begin{remark}
Let $n\geq 2$ be an integer, and let $\big\{V^{(i)}\big\}_{i=1}^{n}$ be a collection of vector spaces over $\mathbb{K}$ such that $\dim_{\,\mathbb{K}}V^{(i)}=i$ for all $i\in[1,n]$. Let $\big\{ \phi^{(i)}:V^{(i)}\to V^{(i+1)}\big\}_{i=1}^{n-1}$ be a collection of injective linear maps, and let $\,\hat{\mathbf{f}}^{(n)}:=\big\{\hat{f}^{(n)}_{k}\big\}_{k=1}^{n}$ be a basis for $V^{(n)}$ such that the type $\mathcal{I}$ flag $\prescript{}{\mathcal{I}\,}{\fl{F} }(\phi^{(1)},\dots,\phi^{(n-1)})$ in $V^{(n)}$ associated with $\big\{\phi^{(i)}\big\}_{i=1}^{n-1}$ coincides with the standard flag, $\prescript{}{\mathcal{I}\,}{\fl{F} }(\phi^{(1)},\dots,\phi^{(n-1)})\big[\, \hat{\mathbf{f}}^{(n)}\,\big]=\sfl \big[\, \hat{\mathbf{f}}^{(n)}\,\big]$. Then, an inductive argument establishes that, for each $i\in[1,n-1]$, there is a unique basis $\hat{\mathbf{f}}^{(i)}:=\big\{\hat{f}^{(i)}_{k}\big\}_{k=1}^{i}$ for $V^{(i)}$ such that the collection of bases $\big\{\hat{\mathbf{f}}^{(i)}\big\}_{i=1}^{n}$ is a system of bases for $\big\{V^{(i)}\big\}_{i=1}^{n}$ adapted to $\big\{\phi^{(i)}\big\}_{i=1}^{n-1}$ (see Definition~\eqref{Def:flags and adapted bases}--\eqref{Def: adapted bases I}). In particular, this observation will play a key role in our discussion ahead.
\end{remark}

\begin{lemma}\label{Lemma: K type flag in adapted bases for surjective maps}
Let $n\geq 2$ be an integer, and let $\big\{V^{(i)}\big\}_{i=1}^{n}$ be a collection of vector spaces over $\mathbb{K}$ such that $\dim_{\,\mathbb{K}}V^{(i)}=i$ for all $i\in[1,n]$.  Let $\big\{ \psi^{(i)}:V^{(i+1)}\to V^{(i)}\big\}_{i=1}^{n-1}$ be a collection of surjective linear maps, and for each $i\in[1,n]$, let $\,\hat{\mathbf{f}}^{(i)}:=\big\{\hat{f}^{(i)}_{k}\big\}_{k=1}^{i}$ be a basis for $V^{(i)}$ such that $\big\{\hat{\mathbf{f}}^{(i)}\big\}_{i=1}^{n}$ is a system of bases adapted to $\big\{\psi^{(i)}\big\}_{i=1}^{n-1}$ (see Definition~\eqref{Def:flags and adapted bases}--\eqref{Def: adapted bases II}). Then, relative to the basis $\hat{\mathbf{f}}^{(n)}$ for $V^{(n)}$, the type $\mathcal{K}$ flag~$\prescript{}{\mathcal{K}\,}{\fl{F} }(\psi^{(1)},\dots,\psi^{(n-1)})$ in $V^{(n)}$ associated with $\big\{\psi^{(i)}\big\}_{i=1}^{n-1}$ coincides with the anti-standard flag, $\prescript{}{\mathcal{K}\,}{\fl{F} }(\psi^{(1)},\dots,\psi^{(n-1)})\big[\, \hat{\mathbf{f}}^{(n)}\,\big]=\asfl \big[\, \hat{\mathbf{f}}^{(n)}\,\big]$, and is therefore represented by the matrix $\,\mathbf{w}_{n}$.  
\end{lemma} 
\begin{proof}
By Definition~\eqref{Def:flags and adapted bases}--\eqref{Def: type K flag}, the filtration $\prescript{}{\mathcal{K}\,}{\fl{F} }(\psi^{(1)},\dots,\psi^{(n-1)}):=\big\{ F^{(0)}\subset  \cdots \subset F^{(n)} \big\}$ in $V^{(n)}$ is given by 
\begin{equation*}
F^{(p)}:=\begin{cases}
\hfil 0, & \text{if $~p=0$}\,,\\
\mathrm{ker}\big( \psi^{(n-p)}\circ\cdots\circ\psi^{(n-1)} \big), & \text{if $~p\in[1,n-1]$}\,,\\
\hfil V^{(n)}, & \text{if $~p=n$}\, .
\end{cases}
\end{equation*}
Now, since $\big\{\hat{\mathbf{f}}^{(i)}\big\}_{i=1}^{n}$ is a system of bases adapted to $\big\{\psi^{(i)}\big\}_{i=1}^{n-1}$, Definition~\eqref{Def:flags and adapted bases}--\eqref{Def: adapted bases II} asserts that for each $i\in[1,n-1]$, the matrix $\tensor[_{\hat{\mathbf{f}}^{(i)}}]{ \big[\, \psi^{\,(i)}\,\big] }{_{\hat{\mathbf{f}}^{(i+1)}}}\in M(i,i+1,\mathbb{K})$ representing $\psi^{\,(i)}$ is the standard projection matrix; that is, $\tensor[_{\hat{\mathbf{f}}^{(i)}}]{ \big[\, \psi^{\,(i)}\,\big] }{_{\hat{\mathbf{f}}^{(i+1)}}}=\pi^{(i,i+1)}$. Consequently, we obtain that
\begin{equation*}
F^{(p)}=\left<\hat{f}^{(n)}_{n},\dots, \hat{f}^{(n)}_{n-p+1}\right>\, ,
\end{equation*}
for all $p\in [1,n-1]$, which establishes that $\prescript{}{\mathcal{K}\,}{\fl{F} }(\psi^{(1)},\dots,\psi^{(n-1)})\big[\, \hat{\mathbf{f}}^{(n)}\,\big]=\sfl \big[\, \hat{\mathbf{f}}^{(n)}\,\big]$, as desired.   
\end{proof}

\begin{remark}
Let $n\geq 2$ be an integer, and let $\big\{V^{(i)}\big\}_{i=1}^{n}$ be a collection of vector spaces over $\mathbb{K}$ such that $\dim_{\,\mathbb{K}}V^{(i)}=i$ for all $i\in[1,n]$. Let $\big\{ \psi^{(i)}:V^{(i+1)}\to V^{(i)}\big\}_{i=1}^{n-1}$ be a collection of surjective linear maps, and let $\,\hat{\mathbf{f}}^{(n)}:=\big\{\hat{f}^{(n)}_{k}\big\}_{k=1}^{n}$ be a basis for $V^{(n)}$ such that the type $\mathcal{K}$ flag $\prescript{}{\mathcal{K}\,}{\fl{F} }(\psi^{(1)},\dots,\psi^{(n-1)})$ in $V^{(n)}$ associated with $\big\{\psi^{(i)}\big\}_{i=1}^{n-1}$ coincides with the anti-standard flag, $\prescript{}{\mathcal{K}\,}{\fl{F} }(\psi^{(1)},\dots,\psi^{(n-1)})\big[\, \hat{\mathbf{f}}^{(n)}\,\big]=\asfl \big[\, \hat{\mathbf{f}}^{(n)}\,\big]$. Then, an inductive argument shows that, for each $i\in[1,n-1]$, there is a unique basis $\hat{\mathbf{f}}^{(i)}:=\big\{\hat{f}^{(i)}_{k}\big\}_{k=1}^{i}$ for $V^{(i)}$ such that the collection of bases $\big\{\hat{\mathbf{f}}^{(i)}\big\}_{i=1}^{n}$ is a system of bases for $\big\{V^{(i)}\big\}_{i=1}^{n}$ adapted to $\big\{\psi^{(i)}\big\}_{i=1}^{n-1}$ (see Definition~\eqref{Def:flags and adapted bases}--\eqref{Def: adapted bases II}). In particular, this observation will play a key role in our discussion ahead.
\end{remark}

Before proceeding further, we state a lemma that describes how the matrix representation of a complete flag in a vector space transforms under a change of basis.

\begin{lemma}\label{lemma: Matrices that represent a flag in two different bases}
Let $V$ be a finite-dimensional vector space over $\mathbb{K}$ with $\mathrm{dim}_{\,\mathbb{K}}V=n$ for some integer $n\geq 1$, and let $\hat{\mathbf{x}}:=\big\{\hat{x}_{i}\big\}_{i=1}^{n}$ and $\hat{\mathbf{y}}:=\big\{\hat{y}_{j}\big\}_{j=1}^{n}$ be bases for $V$ related via the change-of-basis matrix $M\in \mathrm{GL}(n,\mathbb{K})$; that is, $\hat{y}_{j}=\sum_{i=1}^{n}M_{i,j}\,\hat{x}_{i}$ for each $j\in[1, n]$. Furthermore, let $\fl{F}:=\big\{F^{0}\subset \cdots \subset F^{n}\big\}$ be a complete flag in $V$, and suppose that relative to the basis $\hat{\mathbf{y}}$ for $V$, the flag is represented by the matrix $A\in \mathrm{GL}(n,\mathbb{K})$. Then, relative to the basis $\hat{\mathbf{x}}$ for $V$, the complete flag $\fl{F}$ is represented by the matrix product $M\cdot A\in \mathrm{GL}(n,\mathbb{K})$. 
\end{lemma}
\begin{proof}
To begin with, let $\big\{\vec{v}_{k}\big\}_{k=1}^{n}$ be the collection of vectors in $V$ given by $\vec{v}_{k}=\sum_{j=1}^{n}A_{j,k}\,\hat{y}_{j}$ for each $k\in[1,n]$. Thus, since relative to the basis $\hat{\mathbf{y}}$ for $V$, the flag $\fl{F}$ is represented by the matrix $A$, we have that 
\begin{equation*}
F^{(p)}=\begin{cases}
\hfil 0\, , & \text{if $p=0$}\, ,\\
\big< \vec{v}_{1}, \dots, \vec{v}_{p} \big>\, , & \text{if $p=1,\dots, n$}\, .
\end{cases}    
\end{equation*}
In addition, recall that $\hat{y}_{j}=\sum_{i=1}^{n}M_{i,j}\,\hat{x}_{i}$ for each $j\in[1,n]$. As a result, we obtain that 
\begin{equation*}
\begin{aligned}
\vec{v}_{k}&=\sum_{j=1}^{n}\sum_{i=1}^{n}A_{j,k}M_{i,j}\,\hat{x}_{i}\, ,\\
&=\sum_{i=1}^{n}\Bigg(\sum_{j=1}^{n}M_{i,j}A_{j,k}\Bigg)\,\hat{x}_{i}\, ,\\
&=\sum_{i=1}^{n}\big(M\cdot A\big)_{i,k}\,\hat{x}_{i}\, ,\\
\end{aligned}
\end{equation*}
for each $k\in[1,n]$. Bearing this in mind, we conclude that, relative to the basis $\hat{\mathbf{x}}$ for $V$, the flag $\fl{F}$ is represented by the matrix product $M\cdot A$. This completes the proof.  
\end{proof}

We now introduce a collection of matrices that parametrize complete flags in relative position, the so-called \emph{braid matrices}. These matrices are of central importance, as they will play a key role in the algebraic description of the objects of the category $\ccs{1}{\beta}$. 

\begin{definition}[Braid Matrices]\label{Def: braid matrices}
Let $n\geq 2$ be an integer. Given an Artin generator $\sigma_{k}\in \mathrm{Br}^{+}_{n}$, with $k\in [1,n-1]$, and a parameter $z\in \mathbb{K}$, the $n$-dimensional braid matrix $B^{(n)}_{k}(z)\in \mathrm{GL}\big(n,\, \mathbb{K}\big)$ associated with $\sigma_{k}$ and $z$ is defined by
\begin{equation*}
\Big[\,B^{(n)}_{k}(z)\,\Big]_{i,j}\,:=\,\begin{cases}
~1\,, & \text{if $~i=j\,$ and $\,i\neq k,\,k+1$}\, ,\\
~1\,, & \text{if $~(i,j)=(k,k+1)\,$ or $\,(k+1,k)$}\, ,\\
~z\,, & \text{if $~i=j=k$}\, ,\\ 
~0\,, & \text{otherwise}\, ,
\end{cases}  \qquad   i,j\in[1, n]\, .
\end{equation*}
More precisely,
\begin{equation*}
\centering    
\begin{tikzpicture}
\matrix (m) [matrix of math nodes, ampersand replacement=\&, left delimiter={[}, right delimiter={]}, nodes in empty cells, row sep=1pt, column sep=1pt]
{
\begin{array}{ccc}
    1 & \cdots & 0 \\
    \vdots & \ddots & \vdots \\
    0 & \cdots & 1
\end{array} \& \begin{array}{cc}
    0 & 0 \\
    \vdots & \vdots \\
    0 & 0
\end{array} \& \begin{array}{ccc}
    0 & \cdots & 0 \\
    \vdots & \ddots & \vdots \\
    0 & \cdots & 0
\end{array} \\
\begin{array}{ccc}
    0 & \cdots & 0 \\
    0 & \cdots & 0 
\end{array} \& \begin{array}{cc}
    z & 1 \\
    1 & 0 
\end{array} \& \begin{array}{ccc}
    0 & \cdots & 0 \\
    0 & \cdots & 0 
\end{array} \\
\begin{array}{ccc}
    0 & \cdots & 0 \\
    \vdots & \ddots & \vdots \\
    0 & \cdots & 0
\end{array} \& \begin{array}{cc}
    0 & 0 \\
    \vdots & \vdots \\
    0 & 0
\end{array} \& \begin{array}{ccc}
    1 & \cdots & 0 \\
    \vdots & \ddots & \vdots \\
    0 & \cdots & 1
\end{array} \\
};

\node[draw=none, fit=(m-1-1) (m-1-1), inner sep=2pt] (Abox) {};
\node[draw=none, fit=(m-1-2) (m-1-2), inner sep=2pt] (Bbox) {};
\node[draw=none, fit=(m-1-3) (m-1-3), inner sep=2pt] (Cbox) {};
\node[draw=none, fit=(m-2-3) (m-2-3), inner sep=2pt] (Dbox) {};
\node[draw=none, fit=(m-3-3) (m-3-3), inner sep=2pt] (Ebox) {};
\node[draw=none, fit=(m-2-1) (m-2-1), inner sep=2pt] (Fbox) {};

\draw[<->] ([yshift=0.2cm,xshift=0.15cm]Abox.north west) -- node[above=0.15cm] {\scriptsize $k$} ([yshift=0.2cm,xshift=-0.15cm]Abox.north east);
\draw[<->] ([yshift=0.2cm,xshift=0.15cm]Bbox.north west) -- node[above=0.15cm] {\scriptsize $2$} ([yshift=0.2cm,xshift=-0.15cm]Bbox.north east);
\draw[<->] ([yshift=0.2cm,xshift=0.15cm]Cbox.north west) -- node[above=0.15cm] {\scriptsize $n-(k+1)$} ([yshift=0.2cm,xshift=-0.15cm]Cbox.north east);

\draw[<->] ([yshift=-0.3cm,xshift=0.45cm]Cbox.north east) -- node[right=0.15cm] {\scriptsize $k$} ([yshift=0.3cm,xshift=0.45cm]Cbox.south east);
\draw[<->] ([yshift=-0.3cm,xshift=0.45cm]Dbox.north east) -- node[right=0.15cm] {\scriptsize $2$} ([yshift=0.3cm,xshift=0.45cm]Dbox.south east);
\draw[<->] ([yshift=-0.3cm,xshift=0.45cm]Ebox.north east) -- node[right=0.15cm] {\scriptsize $n-(k+1)$} ([yshift=0.3cm,xshift=0.45cm]Ebox.south east);

\node[left=2.5cm] at (m) {$B^{(n)}_{k}(z):=$};

\node[draw=none, fit=(m-1-1) (m-1-1), inner sep=2pt] (a) {};
\node[draw=none, fit=(m-1-2) (m-1-2), inner sep=2pt] (b) {};
\node[draw=none, fit=(m-1-3) (m-1-3), inner sep=2pt] (c) {};
\node[draw=none, fit=(m-2-1) (m-2-1), inner sep=2pt] (d) {};
\node[draw=none, fit=(m-2-2) (m-2-2), inner sep=2pt] (e) {};
\node[draw=none, fit=(m-2-3) (m-2-3), inner sep=2pt] (f) {};
\node[draw=none, fit=(m-3-1) (m-3-1), inner sep=2pt] (g) {};
\node[draw=none, fit=(m-3-2) (m-3-2), inner sep=2pt] (h) {};
\node[draw=none, fit=(m-3-3) (m-3-3), inner sep=2pt] (k) {};

\draw[-, line width=0.02cm] ([yshift=-0.95cm, xshift=-0.75cm]a.center) -- ([yshift=-0.95cm, xshift=0.75cm]c.center);
\draw[-, line width=0.02cm] ([yshift=0.95cm, xshift=-0.75cm]g.center) -- ([yshift=0.95cm, xshift=0.75cm]k.center);

\draw[-, line width=0.02cm] ([yshift=0.75cm, xshift=0.85cm]a.center) -- ([yshift=-0.75cm, xshift=0.85cm]g.center);
\draw[-, line width=0.02cm] ([yshift=0.75cm, xshift=-0.85cm]c.center) -- ([yshift=-0.75cm, xshift=-0.85cm]k.center);

\end{tikzpicture}
\end{equation*}
In this manuscript, we refer to $B^{(n)}_{k}(z)$ as the $k$-th braid matrix of dimension $n$ associated with $\sigma_{k}$ and $z$. 
\end{definition}

\noindent
The braid matrices have been extensively studied and employed in the literature (see, for instance,~\cite{KT2, CGGS1, CGGS2, CGGS1, GSW1, CN1, CNS1}). In particular, in Appendix A, we list some of their properties, focusing on those most relevant for this manuscript.

\begin{lemma}\label{Lemma for flags and braid matrices}
Let $V$ be a finite-dimensional vector space over $\mathbb{K}$ with $\mathrm{dim}_{\,\mathbb{K}}V=n$ for some integer $n \geq 1$. Let $\flag{F}$ and $\flag{G}$ be a pair of complete flags in $V$ such that $\fl{G}$ is in $s_{k}$-relative position with respect to $\fl{F}$ for some $k\in [1, n-1]$. Furthermore, let $\hat{\mathbf{x}}:=\big\{\hat{x}_{i}\big\}_{i=1}^{n}$ be a basis for $V$, and suppose that relative to the basis $\hat{\mathbf{x}}$, the flag $\fl{F}$ is represented by the matrix $A\in \mathrm{GL}(n,\mathbb{K})$. Then there exists $z\in\mathbb{K}$ such that, relative to the basis $\hat{\mathbf{x}}$, the flag $\fl{G}$ is represented by the matrix product $A\cdot B^{(n)}_{k}(z)$, where $B^{(n)}_{k}(z)\in\mathrm{GL}(n,\mathbb{K})$ denotes a $k$-th braid matrix of dimension $n$. 
\end{lemma}
\begin{proof}
Let $\big\{\vec{v}_{j}\big\}_{j=1}^{n}$ be the collection of vectors in $V$ given by $\vec{v}_{j}=\sum_{i=1}^{n}A_{i,j}\,\hat{x}_{i}$ for each $j\in[1,n]$. By hypothesis, we know that, relative to the basis $\hat{\mathbf{x}}$, the flag $\fl{F}$ is represented by the matrix $A$. In other words, we have that
\begin{equation*}
F^{(p)}=\begin{cases}
\hfil 0\, , & \text{if $p=0$}\, ,\\
\big< \vec{v}_{1}, \dots, \vec{v}_{p} \big>\, , & \text{if $p=1,\dots, n$}\, .
\end{cases}    
\end{equation*}
Moreover, recall that the flag $\fl{G}$ is in $s_{k}$-relative position with respect to $\fl{F}$. To be more precise, we know that $F^{(k)}\neq G^{(k)}$ and $F^{(p)}=G^{(p)}$ for all $p\neq k$, where $p\in[0,n]$. Here, observe that for each $p\in[1,n]$, the set $\big\{\vec{v}_{1},\dots, \vec{v}_{p}\big\}$ is a basis for $F^{(p)}$. Thus, since $F^{(k-1)}\subset G^{(k)}$, we can extend the basis $\{\vec{v}_{1},\dots, \vec{v}_{k-1}\}$ for $F^{(k-1)}$ to a basis $\big\{\vec{v}_{1},\dots, \vec{v}_{k-1},\vec{g}\big\}$ for $G^{(k)}$, for some non-zero vector $\vec{g}\in G^{(k)}$. Moreover, since $\big\{\vec{v_{1}},\dots, \vec{v}_{k+1}\big\}$ spans $F^{(k+1)}$ and $G^{(k)}\subset F^{(k+1)}$, we can write $\vec{g}=\alpha_{1}\,\vec{v}_{1}+\cdots + \alpha_{k+1}\,\vec{v}_{k+1}$ for some $\alpha_{j}\in\mathbb{K}$, where $j\in[1, k+1]$. In particular, since $G^{(k)}\neq F^{(k)}$, we know that $\vec{g}$ does not belong to $F^{(k)}$, and therefore $\alpha_{k+1}\neq 0$. Thus, we obtain that $\vec{g}= \alpha_{1}\,\vec{v}_{1}+\dots+\alpha_{k-1}\,\vec{v}_{k-1} +\alpha_{k+1}\big(\vec{v}_{k+1}+z\,\vec{v}_{k}\big)$, where $z:=\alpha_{k+1}^{-1}\alpha_{k}\in \mathbb{K}$. As a result, we can see that the set $\big\{\vec{v}_{1},\dots, \vec{v}_{k-1}, \vec{v}_{k+1}+ z\,\vec{v}_{k}\big\}$ spans $G^{(k)}$. Bearing this in mind, we conclude that, there is $z\in \mathbb{K}$ such that, relative to the basis $\hat{\mathbf{x}}$ for $V$, the flag $\fl{G}$ is represented by the matrix whose column vectors are given by $\big[\vec{v}_{1},\dots,\vec{v}_{k-1}, \vec{v}_{k+1}+z\,\vec{v}_{k},\vec{v}_{k}, \vec{v}_{k+2}, \dots, \vec{v}_{n}\big]$. Finally, in light of Claim~\eqref{braid matrix product from right}, we readily verify that
\begin{equation*}
A\cdot B^{(n)}_{k}(z)=\big[\vec{v}_{1},\dots,\vec{v}_{k-1}, \vec{v}_{k+1}+z\,\vec{v}_{k},\vec{v}_{k}, \vec{v}_{k+2}, \dots, \vec{v}_{n}\big]\, ,    
\end{equation*}
where $B^{(n)}_{k}(z)\in \mathrm{GL}(n,\mathbb{K})$ denotes a $k$-th braid matrix of dimension $n$. Therefore, we deduce that, relative to the basis $\hat{\mathbf{x}}$, the flag $\fl{G}$ is represented by the matrix product $A\cdot B^{(n)}_{k}(z)$, for some $z\in\mathbb{K}$.
\end{proof}

Finally, we introduce the notion of braid-transformed bases, which captures how a positive braid word $\beta=\sigma_{i_{1}}\cdots \sigma_{i_{\ell}}\in\mathrm{Br}^{+}_{n}$, together with a tuple $\vec{z}\in\mathbb{K}^{\ell}_{\mathrm{std}}$, acts on a collection of bases for a family of vector spaces and produces a new set of bases for the corresponding family.

\begin{definition}[Braid-Transformed Bases]\label{Def: braid transformation of bases}
Let $n\geq 2$ be an integer, and let $\big\{V^{(i)}\big\}_{i=1}^{n}$ be a collection of vector spaces over $\mathbb{K}$ such that $\dim_{\,\mathbb{K}}V^{(i)}=i$ for all $i\in[1,n]$. For each $i\in [1,n]$, let $\,\hat{\mathbf{f}}^{(i)}:=\big\{\hat{f}^{(i)}_{k}\big\}_{k=1}^{i}$ be basis for $V^{(i)}$, and denote by $\big\{\hat{\mathbf{f}}^{(i)}\big\}_{i=1}^{n}$ the collection of such bases.

Let $\sigma_{p}\in\mathrm{Br}^{+}_{n}$ be the $p$-th Artin generator for some $p\in[1,n-1]$, and let $z\in\mathbb{K}$ be a fixed parameter. For each $i\in [1,n]$, we define a new basis for $V^{(i)}$, $\hat{\mathbf{f}}^{(i)}[\sigma_{p},z]:=\big\{\hat{f}^{(i)}_{k}[\sigma_{p},z]\big\}_{k=1}^{i}$, as follows: for each $k\in[1,i]$, we set  
\begin{equation*}
\hat{f}^{(i)}_{k}[\sigma_{p},z]:=\begin{cases}
\hfil\hat{f}^{(i)}_{k}\, , &~ \text{if $~i\in[1,p]$}\, ,\\
\sum_{j=1}^{i}\big(B^{(i)}_{p}(z)\big)_{j,k}\,\hat{f}^{(i)}_{j} \, , &~ \text{if $~i\in[p+1,n]$}\, ,
\end{cases}    
\end{equation*}
where $B^{(i)}_{p}(z)$ denotes the $p$-th braid matrix of dimension $i$ associated with $\sigma_{p}$ and $z$.  In other words, for each $i\in [1,n]$, the new basis for $V^{(i)}$ is obtained from the old basis via the change-of-basis matrix $B^{(i)}_{p}(z)$ if $i\geq p+1$; otherwise, the new and old bases agree.

Accordingly, let $\beta=\sigma_{i_{1}}\dots \sigma_{i_{\ell}}\in\mathrm{Br}^{+}_{n}$ be a positive braid word, and let $\vec{z}=(z_{1},\dots, z_{\ell})\in \mathbb{K}^{\ell}_{\mathrm{std}}$ be a fixed tuple. Then, we define the \emph{braid-transformed bases} $\big\{\,\hat{\mathbf{f}}^{(i)}[\,\beta,\vec{z}\,]\,\big\}_{i=1}^{n}$, a new collection of bases for the family of vector spaces $\big\{V^{(i)}\big\}_{i=1}^{n}$, by recursively applying, in order, the single-step transformations corresponding to each Artin generator in $\beta$ and the associated entry of $\vec{z}$ to the old bases $\big\{\hat{\mathbf{f}}^{(i)}\big\}_{i=1}^{n}$, namely: 
\begin{equation*}
\hat{\mathbf{f}}^{(i)}[\,\beta,\vec{z}\,]:=  \big(\cdots\big(\hat{\mathbf{f}}^{(i)}[\sigma_{i_{1}},z_{1}]\big)\cdots \big)[\sigma_{i_{\ell}},z_{\ell}]\, , \quad i\in [1,n]. 
\end{equation*}
\end{definition}

Having established the necessary technical background, we now turn to the explicit characterization of the objects of the category $\ccs{1}{\beta}$ in the case $\beta=e_{n}$.

\subsection{The Case of the Trivial Braid} 
Let $e_{n}\in\mathrm{Br}^{+}_{n}$ be the trivial braid word. In this subsection, we study the objects of the category $\ccs{1}{e_{n}}$. In particular, this concrete example, corresponding to the Legendrian unlink $\Lambda(e_{n})\subset (\mathbb{R}^{3}, \xi_{\mathrm{std}})$ on $n$ strands, will serve as a prototype for the general theory developed later in this section.

Let $\sh{F}$ be an object of the category $\ccs{1}{e_{n}}$. Next, we construct an open cover of $\mathbb{R}^{2}$ adapted to the front projection $\Pi_{x,z}(\Lambda(e_{n}))\subset \mathbb{R}^{2}$ of the Legendrian link $\Lambda(e_{n})\subset (\mathbb{R}^{3},\xi_{\mathrm{std}})$. The purpose of this construction is to decompose the plane into manageable regions where the local behavior of $\sh{F}$ can be analyzed independently and later integrated into a concise global description.

\begin{construction}\label{Const: finite open cover of R^2 for the unlink}
Let $e_{n}\in \mathrm{Br}^{+}_{n}$ be the trivial braid word. Given the front projection $\Pi_{x,z}(\Lambda(e_{n}))\subset \mathbb{R}^{2}$ of the Legendrian link $\Lambda(e_{n})\subset (\mathbb{R}^{3},\xi_{\mathrm{std}})$, we denote by $\mathcal{U}_{\Lambda(e_{n})}:=\big\{ U_{0}, U_{\mathrm{B}}, U_{\mathrm{T}}\big\}$ the finite open cover of $\mathbb{R}^{2}$ illustrated in Figure~\eqref{Fig: open cover for the unlink}. Specifically, we assume that:     
\begin{itemize}
\justifying
\item $U_{0}$ is an unbounded open subset of $\mathbb{R}^{2}$ such that $U_{0}\,\cap\,\Pi_{x,z}(\Lambda(e_{n}))=\emptyset$. In Figure~\eqref{Fig: open cover for the unlink}, the region $U_{0}$ is depicted in purple. 

\item $U_{\mathrm{B}}$ is a bounded open subset of $\mathbb{R}^{2}$ such that the intersection $U_{\mathrm{B}}\,\cap\,\Pi_{x,z}(\Lambda(e_{n}))$ consists of the $n$ strands at the bottom of the front projection $\Pi_{x,z}(\Lambda(e_{n}))$. In Figure~\eqref{Fig: open cover for the unlink}, the region $U_{\mathrm{B}}$ is shown in pink.  

\item $U_{\mathrm{T}}$ is a bounded open subset of $\mathbb{R}^{2}$ such that the intersection $U_{\mathrm{T}}\,\cap\, \Pi_{x,z}(\Lambda(e_{n}))$ consists the $n$ strands at the top of the front projection $\Pi_{x,z}(\Lambda(e_{n}))$. In Figure~\eqref{Fig: open cover for the unlink}, the region $U_{\mathrm{T}}$ is illustrated in green. 
\end{itemize}

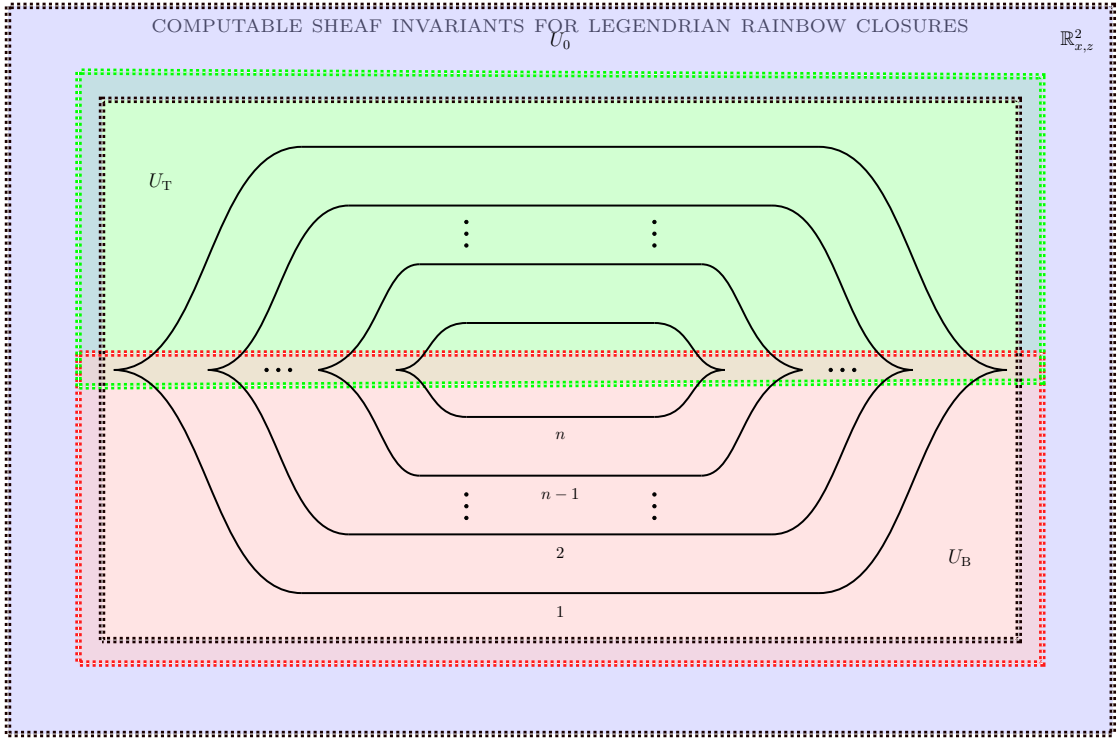
\begin{figure}
\centering
\begin{tikzpicture}
\useasboundingbox (-8,-4.85) rectangle (8,4);
\scope[transform canvas={scale=0.625}]

\def\regionone{ (-11.75+2-0.5,+0.35-1+1) -- (-11.75+2-0.5,-6-1.25+1) -- (11.75-2+0.5,-6-1.25+1) -- (+11.75-2+0.5,+0.35-1+1) -- (-11.75+2-0.5,+0.35-1+1) }
\def\regiontwo{ (-11.75+2-0.5,-0.35+1-1) -- (-11.75+2-0.5,6+1.35-1) -- (11.75-2+0.5,6+1.25-1) -- (+11.75-2+0.5,-0.25+1-1) -- (-11.75+2-0.5,-0.35+1-1) }
\def\regionthree{(-13.75+2,-6.25-2+0.5) rectangle (13.75-2, 6.25+2-0.5) (-11.25+1.5,-6.25+0.5) rectangle (11.25-1.5, 6.25-0.5)}

\fill[green!25, fill opacity=0.7] \regiontwo;
\fill[blue!30, even odd rule, fill opacity=0.4] \regionthree;
\fill[pink!70, fill opacity=0.6] \regionone;

\draw[dotted, line width=0.065cm, color=pink!360, double distance=0.2pt] \regionone;
\draw[dotted, line width=0.065cm, color=green!360, double distance=0.2pt] \regiontwo; 
\draw[dotted, line width=0.065cm, color=purple!360, double distance=0.2pt] \regionthree;

\draw[very thick] (-2,1) -- (2, 1);
\draw[very thick] (-2,-1) -- (2, -1);
\draw[very thick] (2,1) .. controls (3-0.15,1) and (3-0.25,0) .. (3.5,0);
\draw[very thick] (2,-1) .. controls (3-0.15,-1) and (3-0.25,0) .. (3.5,0);
\draw[very thick] (-3.5, 0) .. controls (-3+0.25,0) and (-3+0.15,1) .. (-2,1);
\draw[very thick] (-3.5, 0) .. controls (-3+0.25,0) and (-3+0.15,-1) .. (-2,-1);

\draw[very thick] (-2-1,1+1+0.25) -- (2+1,1+1+0.25);
\draw[very thick] (-2-1,-1-1-0.25) -- (2+1,-1-1-0.25);
\draw[very thick] (2+1,1+1+0.25) .. controls (3+1-0.15,1+1+0.25) and (3+1-0.15,0+0.2) .. (3.5+1.5+0.15,0);
\draw[very thick] (2+1,-1-1-0.25) .. controls (3+1-0.15,-1-1-0.25) and (3+1-0.15,0-0.2) .. (3.5+1.5+0.15,0);
\draw[very thick] (-3.5-1.5-0.15, 0) .. controls (-3-1+0.15,0+0.2) and (-3-1+0.15,1+1+0.25) .. (-2-1,1+1+0.25);
\draw[very thick] (-3.5-1.5-0.15, 0) .. controls (-3-1+0.15,0-0.2) and (-3-1+0.15,-1-1-0.25) .. (-2-1,-1-1-0.25);

\draw[very thick] (-2-2.5,1+2.5) -- (2+2.5, 1+2.5);
\draw[very thick] (-2-2.5,-1-2.5) -- (2+2.5, -1-2.5);
\draw[very thick] (2+2.5,1+2.5) .. controls (3+3,1+2.5) and (3+3,0) .. (3.5+4,0);
\draw[very thick] (2+2.5,-1-2.5) .. controls (3+3,-1-2.5) and (3+3,0) .. (3.5+4,0);
\draw[very thick] (-3.5-4, 0) .. controls (-3-3,0) and (-3-3,1+2.5) .. (-2-2.5,1+2.5);
\draw[very thick] (-3.5-4, 0) .. controls (-3-3,0) and (-3-3,-1-2.5) .. (-2-2.5,-1-2.5);

\draw[very thick] (-2-2.5-1,1+2.5+1+0.25) -- (2+2.5+1, 1+2.5+1+0.25);
\draw[very thick] (-2-2.5-1,-1-2.5-1-0.25) -- (2+2.5+1, -1-2.5-1-0.25);
\draw[very thick] (2+2.5+1, 1+2.5+1+0.25) .. controls (3+3+1+0.5,1+2.5+1+0.25) and (3+3+1+0.5,0) .. (3.5+4+2,0);
\draw[very thick] (2+2.5+1,-1-2.5-1-0.25) .. controls (3+3+1+0.5,-1-2.5-1-0.25) and (3+3+1+0.5,0) .. (3.5+4+2,0);
\draw[very thick] (-3.5-4-2, 0) .. controls (-3-3-1-0.5,0) and (-3-3-1-0.5,1+2.5+1+0.25) .. (-2-2.5-1,1+2.5+1+0.25);
\draw[very thick] (-3.5-4-2, 0) .. controls (-3-3-1-0.5,0) and (-3-3-1-0.5,-1-2.5-1-0.25) .. (-2-2.5-1,-1-2.5-1-0.25);

\filldraw[black] (2,2.9) circle (1pt);
\filldraw[black] (2,2.9+0.25) circle (1pt);
\filldraw[black] (2,2.9-0.25) circle (1pt);

\filldraw[black] (-2,2.9) circle (1pt);
\filldraw[black] (-2,2.9+0.25) circle (1pt);
\filldraw[black] (-2,2.9-0.25) circle (1pt);

\filldraw[black] (2,-2.9) circle (1pt);
\filldraw[black] (2,-2.9-0.25) circle (1pt);
\filldraw[black] (2,-2.9+0.25) circle (1pt);

\filldraw[black] (-2,-2.9) circle (1pt);
\filldraw[black] (-2,-2.9-0.25) circle (1pt);
\filldraw[black] (-2,-2.9+0.25) circle (1pt);

\filldraw[black] (-6,0) circle (1pt);
\filldraw[black] (-6-0.25,0) circle (1pt);
\filldraw[black] (-6+0.25,0) circle (1pt);

\filldraw[black] (6,0) circle (1pt);
\filldraw[black] (6+0.25,0) circle (1pt);
\filldraw[black] (6-0.25,0) circle (1pt);

\node at (0,-1-2.5-1-0.25-0.5+0.1) {\normalsize$1$};
\node at (0,-1-2.5-0.5+0.1) {\normalsize$2$};
\node at (0,-1-1-0.25-0.5+0.1) {\normalsize$n-1$};
\node at (0,-1-0.5+0.1) {\normalsize$n$};

\node at (0,5.5+1.5) {\Large$U_{0}$};
\node at (-8.5,4) {\Large$U_{\mathrm{T}}$};
\node at (8.5,-4) {\Large$U_{\mathrm{B}}$};

\node at (7.5+3.5,5.5+1.5) {\Large$\mathbb{R}^{2}_{x,z}$};

\endscope
\end{tikzpicture}
\captionof{figure}{The front projection $ \Pi_{x,z}(\Lambda(e_n)) \subset \mathbb{R}^2 $ of the Legendrian link $\Lambda(e_n) \subset (\mathbb{R}^3, \xi_{\mathrm{std}})$, along with the finite open cover $\mathcal{U}_{\Lambda(e_n)} = \big\{ U_0, U_{\mathrm{B}}, U_{\mathrm{T}} \big\}$ of $\mathbb{R}^2$. The regions $ U_0 $, $U_{\mathrm{B}} $, and $ U_{\mathrm{T}}$ are depicted in purple, pink, and green, respectively.}
\label{Fig: open cover for the unlink}
\end{figure}
\end{construction}

Next, by applying the sheaf axioms to the open cover $\mathcal{U}_{\Lambda(e_{n})}$ of $\mathbb{R}^{2}$ introduced in Construction~\eqref{Const: finite open cover of R^2 for the unlink}, we provide a global characterization of $\sh{F}$ in terms of its local data and the associated gluing conditions. 

\begin{proposition}\label{Prop: linear map description of a sheaf for the unlink on n strands}
Let $e_{n}\in \mathrm{Br}_{n}^{+}$ be the trivial braid word, $\sh{F}$ an object of the category $\ccs{1}{e_{n}}$, and ~$\mathcal{U}_{\Lambda(e_{n})}=\big\{ U_{0}, U_{\mathrm{B}}, U_{\mathrm{T}}\big\}$ the open cover of~$\mathbb{R}^{2}$ introduced in Construction~\eqref{Const: finite open cover of R^2 for the unlink}. Then, the following statements hold: 
\begin{itemize}
\justifying
\item  On $U_{0}$, $\sh{F}$ is identically zero.
\item  On $U_{\mathrm{T}}$, $\sh{F}$ is specified by a collection of $n-1$ surjective linear maps $\big\{ \psi^{(i)}_{\sh{F}}:\mathbb{K}^{i+1}\to \mathbb{K}^{i} \big\}_{i=1}^{n-1}$, as illustrated in Figure~\eqref{Fig: an object F on the region U_T for the unlink}. 
\item  On $U_{\mathrm{B}}$, $\sh{F}$ is specified by a collection of $n-1$ injective linear maps $\big\{ \phi^{(i)}_{\sh{F}}:\mathbb{K}^{i}\to \mathbb{K}^{i+1}\big\}_{i=1}^{n-1}$, as illustrated in Figure~\eqref{Fig: an object F on the region U_B for the unlink}.
\item  \textbf{Compatibility conditions}: For each $i\in[1,n-1]$, 
\begin{equation*}
\psi^{(i)}_{\sh{F}}\circ \phi^{(i)}_{\sh{F}}=\mathrm{id}_{\mathbb{K}^{i}}\, .   
\end{equation*}
\end{itemize}
\end{proposition}
\begin{proof} To begin, consider the front projection $\Pi_{x,z}(\Lambda(e_{n}))\subset \mathbb{R}^{2}$ of the Legendrian link $\lk{e_{n}}\subset (\mathbb{R}^{3},\xi_{\mathrm{std}})$, which is depicted in Figure~\eqref{fig: unlink on n-strands}. In particular, recall that in $\Pi_{x,z}(\Lambda(e_{n}))$, the strands at the bottom have Maslov potential $0$ and the strands at the top have Maslov potential $1$. Thus, by the microlocal rank conditions and the microlocal support conditions near the arcs, we obtain that:
\begin{itemize}
\justifying
\item For all $q\in U_{0}$, the stalk of $\sh{F}$ at $q$ is trivial; that is, $\sh{F}_{q}=0$. In other words, $\sh{F}$ is identically zero on $U_{0}$.
\item On $U_{\mathrm{T}}$, $\sh{F}$ is defined by a collections of $n-1$ linear maps $\big\{\psi^{(i)}_{\sh{F}}:\mathbb{K}^{i+1}\to \mathbb{K}^{i}\big\}_{i=1}^{n-1}$, as shown in Figure~\eqref{Fig: an object F on the region U_T for the unlink}.
\item On $U_{\mathrm{B}}$, $\sh{F}$ is defined by a collections of $n-1$ linear maps $\big\{\phi^{(i)}_{\sh{F}}:\mathbb{K}^{i}\to \mathbb{K}^{i+1}\big\}_{i=1}^{n-1}$, as shown in Figure~\eqref{Fig: an object F on the region U_B for the unlink}.
\end{itemize}

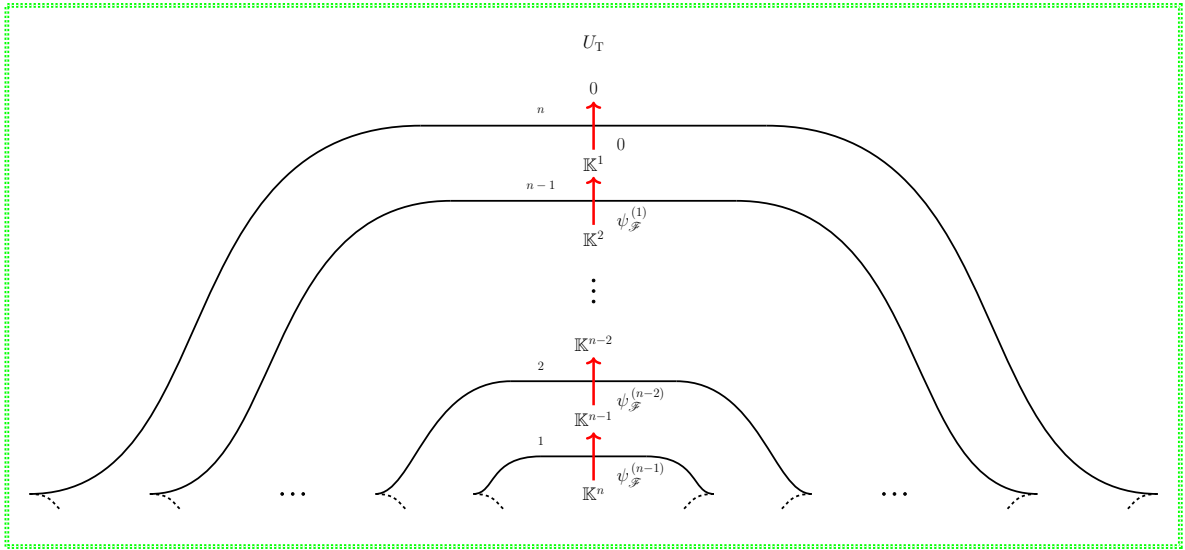
\begin{figure}
\begin{center}
\begin{tikzpicture}
\useasboundingbox (-7.5,+3.75) rectangle (7.5, -3.75);
\scope[transform canvas={scale=0.4}]

\draw[line width=0.06cm] (+1.75,-3-2.5) .. controls (+1.75+1.5+0.25,-3-2.5) and (+1.75+1.5, -3-2.5-1.25) .. (+1.75+2+0.25,-3-2.5-1.25);
\draw[line width=0.06cm] (-1.75,-3-2.5) .. controls (-1.75-1.5-0.25,-3-2.5) and (-1.75-1.5, -3-2.5-1.25) .. (-1.75-2-0.25,-3-2.5-1.25);

\draw[line width=0.06cm] (+1.75+1,-3) .. controls (+1.75+1+2+0.25+0.5,-3) and (+1.75+1+2+1.25,-3-2.5-1.25) .. (+1.75+2+3+0.5,-3-2.5-1.25);
\draw[line width=0.06cm] (-1.75-1,-3) .. controls (-1.75-1-2-0.25-0.5,-3) and (-1.75-1-2-1.25,-3-2.5-1.25) .. (-1.75-2-3-0.5,-3-2.5-1.25);

\draw[line width=0.06cm] (+1.75+3,+3) .. controls (+1.75+1.5+2+2+3+1,+3) and (+1.75+1.5+2+2+1+2,-3-2.5-1.25) .. (+1.75+2+3+5+3,-3-2.5-1.25);
\draw[line width=0.06cm] (-1.75-3,+3) .. controls (-1.75-1.5-2-2-3-1,+3) and (-1.75-1.5-2-2-1-2,-3-2.5-1.25) .. (-1.75-2-3-5-3,-3-2.5-1.25);

\draw[line width=0.06cm] (+1.75+4,+3+2.5) .. controls (+1.75+1.5+2+2+3+2+2,+3+2.5) and (+1.75+1.5+2+2+1+3+1,-3-2.5-1.25) .. (+1.75+2+3+5+3+2+2,-3-2.5-1.25);
\draw[line width=0.06cm] (-1.75-4,+3+2.5) .. controls (-1.75-1.5-2-2-3-2-2,+3+2.5) and (-1.75-1.5-2-2-1-3-1,-3-2.5-1.25) .. (-1.75-2-3-5-3-2-2,-3-2.5-1.25);

\draw[line width=0.06cm, dashed] (+1.75+2+0.25, -3-2.5-1.25) .. controls (+1.75+2+0.25-1+0.35,-3-2.5-1.25) and (+1.75+2+0.25-1+0.25,-3-2.5-1.25-0.25) .. (+1.75+2+0.25-1,-3-2.5-1.25-0.5);
\draw[line width=0.06cm, dashed] (+1.75+2+3+0.5, -3-2.5-1.25) .. controls (+1.75+2+3+0.5-1+0.35,-3-2.5-1.25) and (+1.75+2+3+0.5-1+0.25,-3-2.5-1.25-0.25) .. (+1.75+2+3+0.5-1,-3-2.5-1.25-0.5);
\draw[line width=0.06cm, dashed] (+1.75+2+3+5+3, -3-2.5-1.25) .. controls (+1.75+2+3+5+3-1+0.35,-3-2.5-1.25) and (+1.75+2+3+5+3-1+0.25,-3-2.5-1.25-0.25) .. (+1.75+2+3+5+3-1,-3-2.5-1.25-0.5);
\draw[line width=0.06cm, dashed] (+1.75+2+3+5+3+2+2, -3-2.5-1.25) .. controls (+1.75+2+3+5+3+2+2-1+0.35,-3-2.5-1.25) and (+1.75+2+3+5+3+2+2-1+0.25,-3-2.5-1.25-0.25) .. (+1.75+2+3+5+3+2+2-1,-3-2.5-1.25-0.5);

\draw[line width=0.06cm, dashed] (-1.75-2-0.25, -3-2.5-1.25) .. controls (-1.75-2-0.25+1-0.35,-3-2.5-1.25) and (-1.75-2-0.25+1-0.25,-3-2.5-1.25-0.25) .. (-1.75-2-0.25+1,-3-2.5-1.25-0.5);
\draw[line width=0.06cm, dashed] (-1.75-2-3-0.5, -3-2.5-1.25) .. controls (-1.75-2-3-0.5+1-0.35,-3-2.5-1.25) and (-1.75-2-3-0.5+1-0.25,-3-2.5-1.25-0.25) .. (-1.75-2-3-0.5+1,-3-2.5-1.25-0.5);
\draw[line width=0.06cm, dashed] (-1.75-2-3-5-3, -3-2.5-1.25) .. controls (-1.75-2-3-5-3+1-0.35,-3-2.5-1.25) and (-1.75-2-3-5-3+1-0.25,-3-2.5-1.25-0.25) .. (-1.75-2-3-5-3+1,-3-2.5-1.25-0.5);
\draw[line width=0.06cm, dashed] (-1.75-2-3-5-3-2-2, -3-2.5-1.25) .. controls (-1.75-2-3-5-3-2-2+1-0.35,-3-2.5-1.25) and (-1.75-2-3-5-3-2-2+1-0.25,-3-2.5-1.25-0.25) .. (-1.75-2-3-5-3-2-2+1,-3-2.5-1.25-0.5);

\draw[shorten <=0.5cm,  shorten >=0.5cm, line width=0.06cm] (-2.25-4, 3+2.5) -- (2.25+4, 3+2.5);
\draw[shorten <=0.5cm,  shorten >=0.5cm, line width=0.06cm] (-2.25-2-1, 3) -- (2.25+2+1, 3);
\draw[shorten <=0.5cm,  shorten >=0.5cm, line width=0.06cm] (-2.25-1, -3) -- (2.25+1, -3);
\draw[shorten <=0.5cm,  shorten >=0.5cm, line width=0.06cm] (-2.25, -3-2.5) -- (2.25, -3-2.5);

\node at (0,3+2.5+1.25) {\huge$0$};
\node at (0,3+1.25) {\huge$\mathbb{K}^{1}$};
\node at (0,3-1.25) {\huge$\mathbb{K}^{2}$};

\node at (0,-3+1.25) {\huge$\mathbb{K}^{n-2}$};
\node at (0,-3-1.25) {\huge$\mathbb{K}^{n-1}$};
\node at (0,-3-2.5-1.25) {\huge$\mathbb{K}^{n}$};

\draw[->, shorten <=0.45cm,  shorten >=0.45cm, line width=0.085cm, red] (0,3+2.5-1.25) -- (0,3+2.5+1.25);
\draw[->, shorten <=0.45cm,  shorten >=0.45cm, line width=0.085cm, red] (0,3-1.25) -- (0,3+1.25);
\draw[->, shorten <=0.45cm,  shorten >=0.45cm, line width=0.085cm, red] (0, -3-1.25) -- (0, -3+1.25);
\draw[->, shorten <=0.45cm,  shorten >=0.45cm, line width=0.085cm, red] (0, -3-2.5-1.25) -- (0, -3-2.5+1.25);

\node[right] at (+0.65,+3+2.5-0.6) {\huge$0$};
\node[right] at (+0.65,+3-0.6) {\huge$\psi^{\,(1)}_{\sh{F}}$};

\node[right] at (+0.65,-3-0.6) {\huge$\psi^{\,(n-2)}_{\sh{F}}$};
\node[right] at (+0.65,-3-2.5-0.6) {\huge$\psi^{\,(n-1)}_{\sh{F}}$};

\filldraw[black] (0,0+0.35) circle (1.25pt);
\filldraw[black] (0,0) circle (1.25pt);
\filldraw[black] (0,0-0.35) circle (1.25pt);

\filldraw[black] (-10-0.35,-3-2.5-1.25) circle (1.25pt);
\filldraw[black] (-10,-3-2.5-1.25) circle (1.25pt);
\filldraw[black] (-10+0.35,-3-2.5-1.25) circle (1.25pt);

\filldraw[black] (+10-0.35,-3-2.5-1.25) circle (1.25pt);
\filldraw[black] (+10,-3-2.5-1.25) circle (1.25pt);
\filldraw[black] (+10+0.35,-3-2.5-1.25) circle (1.25pt);

\node at (-1.75,3+2.5+0.5) {\Large$n$};
\node at (-1.75,3+0.5) {\Large$n-1$};
\node at (-1.75,-3+0.5) {\Large$2$};
\node at (-1.75,-3-2.5+0.5) {\Large$1$};

\draw[dotted, line width=0.075cm, color=green!360, double distance=0.2pt] (-19.5,-8.5) rectangle (19.5, 9.5);

\node at (0,8.25) {\huge$U_{\mathrm{T}}$};

\endscope    
\end{tikzpicture}
\end{center}
\caption{An object $\sh{F}$ of the category $\ccs{1}{e_{n}}$ on the region $U_{\mathrm{T}}$.}
\label{Fig: an object F on the region U_T for the unlink}
\end{figure}

\begin{figure}
\begin{center}
\begin{tikzpicture}
\useasboundingbox (-7.5,-3.75) rectangle (7.5, 3.75);
\scope[transform canvas={scale=0.4}]

\draw[line width=0.06cm] (+1.75,+3+2.5) .. controls (+1.75+1.5+0.25,+3+2.5) and (+1.75+1.5, +3+2.5+1.25) .. (+1.75+2+0.25,+3+2.5+1.25);
\draw[line width=0.06cm] (-1.75,+3+2.5) .. controls (-1.75-1.5-0.25,+3+2.5) and (-1.75-1.5, +3+2.5+1.25) .. (-1.75-2-0.25,+3+2.5+1.25);

\draw[line width=0.06cm] (+1.75+1,+3) .. controls (+1.75+1+2+0.25+0.5,+3) and (+1.75+1+2+1.25,+3+2.5+1.25) .. (+1.75+2+3+0.5,+3+2.5+1.25);
\draw[line width=0.06cm] (-1.75-1,+3) .. controls (-1.75-1-2-0.25-0.5,+3) and (-1.75-1-2-1.25,+3+2.5+1.25) .. (-1.75-2-3-0.5,+3+2.5+1.25);

\draw[line width=0.06cm] (+1.75+3,-3) .. controls (+1.75+1.5+2+2+3+1,-3) and (+1.75+1.5+2+2+1+2,+3+2.5+1.25) .. (+1.75+2+3+5+3,+3+2.5+1.25);
\draw[line width=0.06cm] (-1.75-3,-3) .. controls (-1.75-1.5-2-2-3-1,-3) and (-1.75-1.5-2-2-1-2,+3+2.5+1.25) .. (-1.75-2-3-5-3,+3+2.5+1.25);

\draw[line width=0.06cm] (+1.75+4,-3-2.5) .. controls (+1.75+1.5+2+2+3+2+2,-3-2.5) and (+1.75+1.5+2+2+1+3+1,+3+2.5+1.25) .. (+1.75+2+3+5+3+2+2,+3+2.5+1.25);
\draw[line width=0.06cm] (-1.75-4,-3-2.5) .. controls (-1.75-1.5-2-2-3-2-2,-3-2.5) and (-1.75-1.5-2-2-1-3-1,+3+2.5+1.25) .. (-1.75-2-3-5-3-2-2,+3+2.5+1.25);

\draw[line width=0.06cm, dashed] (+1.75+2+0.25, +3+2.5+1.25) .. controls (+1.75+2+0.25-1+0.35,+3+2.5+1.25) and (+1.75+2+0.25-1+0.25,+3+2.5+1.25+0.25) .. (+1.75+2+0.25-1,+3+2.5+1.25+0.5);
\draw[line width=0.06cm, dashed] (+1.75+2+3+0.5, +3+2.5+1.25) .. controls (+1.75+2+3+0.5-1+0.35,+3+2.5+1.25) and (+1.75+2+3+0.5-1+0.25,+3+2.5+1.25+0.25) .. (+1.75+2+3+0.5-1,+3+2.5+1.25+0.5);
\draw[line width=0.06cm, dashed] (+1.75+2+3+5+3, +3+2.5+1.25) .. controls (+1.75+2+3+5+3-1+0.35,+3+2.5+1.25) and (+1.75+2+3+5+3-1+0.25,+3+2.5+1.25+0.25) .. (+1.75+2+3+5+3-1,+3+2.5+1.25+0.5);
\draw[line width=0.06cm, dashed] (+1.75+2+3+5+3+2+2, +3+2.5+1.25) .. controls (+1.75+2+3+5+3+2+2-1+0.35,+3+2.5+1.25) and (+1.75+2+3+5+3+2+2-1+0.25,+3+2.5+1.25+0.25) .. (+1.75+2+3+5+3+2+2-1,+3+2.5+1.25+0.5);

\draw[line width=0.06cm, dashed] (-1.75-2-0.25, +3+2.5+1.25) .. controls (-1.75-2-0.25+1-0.35,+3+2.5+1.25) and (-1.75-2-0.25+1-0.25,+3+2.5+1.25+0.25) .. (-1.75-2-0.25+1,+3+2.5+1.25+0.5);
\draw[line width=0.06cm, dashed] (-1.75-2-3-0.5, +3+2.5+1.25) .. controls (-1.75-2-3-0.5+1-0.35,+3+2.5+1.25) and (-1.75-2-3-0.5+1-0.25,+3+2.5+1.25+0.25) .. (-1.75-2-3-0.5+1,+3+2.5+1.25+0.5);
\draw[line width=0.06cm, dashed] (-1.75-2-3-5-3, +3+2.5+1.25) .. controls (-1.75-2-3-5-3+1-0.35,+3+2.5+1.25) and (-1.75-2-3-5-3+1-0.25,+3+2.5+1.25+0.25) .. (-1.75-2-3-5-3+1,+3+2.5+1.25+0.5);
\draw[line width=0.06cm, dashed] (-1.75-2-3-5-3-2-2, +3+2.5+1.25) .. controls (-1.75-2-3-5-3-2-2+1-0.35,+3+2.5+1.25) and (-1.75-2-3-5-3-2-2+1-0.25,+3+2.5+1.25+0.25) .. (-1.75-2-3-5-3-2-2+1,+3+2.5+1.25+0.5);

\draw[shorten <=0.5cm,  shorten >=0.5cm, line width=0.06cm] (-2.25-4, -3-2.5) -- (2.25+4, -3-2.5);
\draw[shorten <=0.5cm,  shorten >=0.5cm, line width=0.06cm] (-2.25-2-1, -3) -- (2.25+2+1, -3);
\draw[shorten <=0.5cm,  shorten >=0.5cm, line width=0.06cm] (-2.25-1, +3) -- (2.25+1, +3);
\draw[shorten <=0.5cm,  shorten >=0.5cm, line width=0.06cm] (-2.25, +3+2.5) -- (2.25, +3+2.5);

\node at (0,-3-2.5-1.25) {\huge$0$};
\node at (0,-3-1.25) {\huge$\mathbb{K}^{1}$};
\node at (0,-3+1.25) {\huge$\mathbb{K}^{2}$};

\node at (0,+3-1.25) {\huge$\mathbb{K}^{n-2}$};
\node at (0,+3+1.25) {\huge$\mathbb{K}^{n-1}$};
\node at (0,+3+2.5+1.25) {\huge$\mathbb{K}^{n}$};

\draw[->, shorten <=0.45cm,  shorten >=0.45cm, line width=0.085cm, blue] (0,-3-2.5-1.25) -- (0,-3-2.5+1.25);
\draw[->, shorten <=0.45cm,  shorten >=0.45cm, line width=0.085cm, blue] (0,-3-1.25) -- (0,-3+1.25);
\draw[->, shorten <=0.45cm,  shorten >=0.45cm, line width=0.085cm, blue] (0, +3-1.25) -- (0, +3+1.25);
\draw[->, shorten <=0.45cm,  shorten >=0.45cm, line width=0.085cm, blue] (0, +3+2.5-1.25) -- (0, +3+2.5+1.25);

\node[right] at (+0.65,-3-2.5-0.6) {\huge$0$};
\node[right] at (+0.65,-3-0.6) {\huge$\phi^{\,(1)}_{\sh{F}}$};

\node[right] at (+0.65,+3-0.6) {\huge$\phi^{\,(n-2)}_{\sh{F}}$};
\node[right] at (+0.65,+3+2.5-0.6) {\huge$\phi^{\,(n-1)}_{\sh{F}}$};

\filldraw[black] (0,0-0.35) circle (1.25pt);
\filldraw[black] (0,0) circle (1.25pt);
\filldraw[black] (0,0+0.35) circle (1.25pt);

\filldraw[black] (-10-0.35,+3+2.5+1.25) circle (1.25pt);
\filldraw[black] (-10,+3+2.5+1.25) circle (1.25pt);
\filldraw[black] (-10+0.35,+3+2.5+1.25) circle (1.25pt);

\filldraw[black] (+10-0.35,+3+2.5+1.25) circle (1.25pt);
\filldraw[black] (+10,+3+2.5+1.25) circle (1.25pt);
\filldraw[black] (+10+0.35,+3+2.5+1.25) circle (1.25pt);

\node at (-1.75,-3-2.5-0.5) {\Large$1$};
\node at (-1.75,-3-0.5) {\Large$2$};
\node at (-1.75,+3-0.5) {\Large$n-1$};
\node at (-1.75,+3+2.5-0.5) {\Large$n$};

\draw[dotted, line width=0.075cm, color=pink!360, double distance=0.2pt] (-19.5,+8.5) rectangle (19.5, -9.5);

\node at (0,-8.25) {\huge$U_{\mathrm{B}}$};

\endscope    
\end{tikzpicture}
\end{center}
\caption{An object $\sh{F}$ of the category $\ccs{1}{e_{n}}$ on the region $U_{\mathrm{B}}$.}
\label{Fig: an object F on the region U_B for the unlink}
\end{figure}
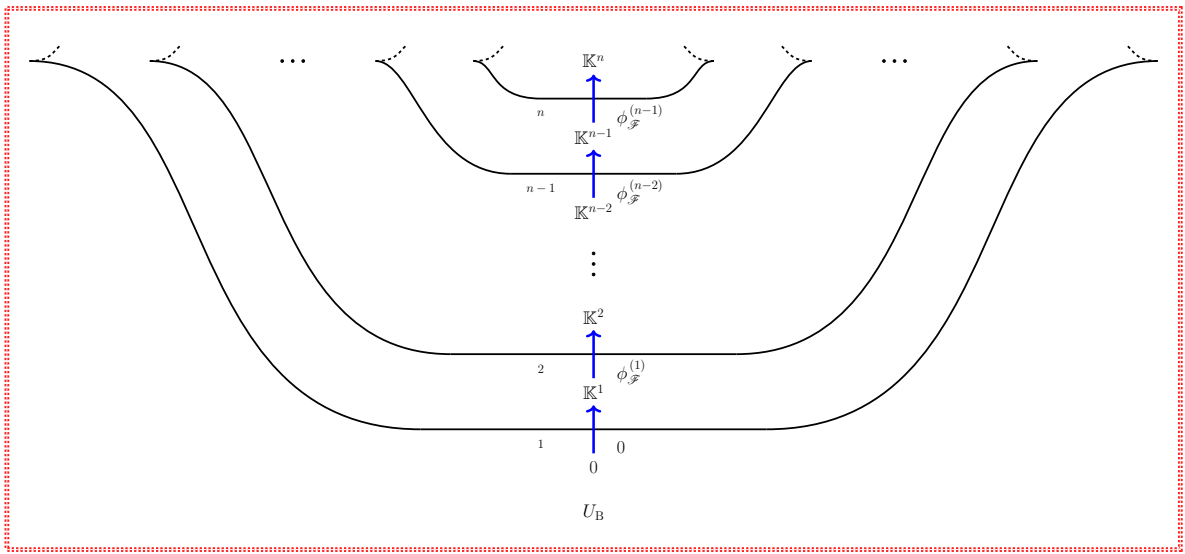

Now, note that the intersection $U_{\mathrm{B}}\,\cap\,U_{\mathrm{T}}\,\cap\,\Pi_{x,z}(\Lambda(e_{n}))$ consists of the $n$ left cusps and the $n$ right cusps in $\Pi_{x,z}(\Lambda(e_{n}))$. Thus, on $U_{\mathrm{B}}\,\cap\,U_{\mathrm{T}}$, $\sh{F}$ is described by the diagram in Figure~\eqref{Fig: an object F on the intersection of U_B and U_T for the unlink}. Bearing this in mind, the microlocal support conditions near the cusps establish that:
\begin{equation}\label{cusp conditions for the unlink}
\psi_{\sh{F}}^{(i)}\circ \phi_{\sh{F}}^{(i)}=\mathrm{id}_{\mathbb{K}^{i}}\, ,
\end{equation}
for each $i\in[1, n-1]$. Consequently, a direct application of Lemma~\eqref{Lemma: cusp condition} yields that the maps $\big\{\psi_{\sh{F}}^{(i)} \big\}_{i=1}^{n-1}$ are surjective, whereas the maps $\big\{\phi_{\sh{F}}^{(i)} \big\}_{i=1}^{n-1}$ are injective. This concludes the proof.

\begin{figure}
\begin{center}
\begin{tikzpicture}
\useasboundingbox (-8.25,-2) rectangle (8.25, 2);
\scope[transform canvas={scale=0.35}]

\draw[line width=0.05cm] (-2.5-2.5,0) .. controls (-1-0.75-2.5,0) and (-1+0.5-2.5,0.75) .. (0-2.5,1.5);
\draw[line width=0.05cm] (-2.5-2.5,0) .. controls (-1-0.75-2.5,-0) and (-1+0.5-2.5,-0.75) .. (0-2.5,-1.5);

\node[right] at (0-2.5,1.5+0.5) {\Large$1$};
\node[right] at (0-2.5,-1.5-0.5) {\Large$n$};

\node[right] at (0-2.5,0) {\huge$\mathbb{K}^{n}$};
\node[right] at (-4.25-2.5,0) {\huge$\mathbb{K}^{n-1}$};

\draw[->, shorten <=0.85cm,  shorten >=0.85cm, line width=0.075cm, red] (0+0.65-2.5,0) -- (-4+0.65-2.5, 1.5);
\draw[->, shorten <=0.85cm,  shorten >=0.85cm, line width=0.075cm, blue] (-4+0.65-2.5,-1.5) -- (0+0.65-2.5,0);

\node at (-1.75+0.5-2.5,+1.75) {\huge$\psi^{\,(n-1)}$};
\node at (-1.75+0.5-2.5,-1.75) {\huge$\phi^{\,(n-1)}_{\sh{F}}$};


\draw[line width=0.05cm] (-2.5-4.25-2.5,0) .. controls (-1-0.75-4.25-2.5,0) and (-1+0.5-4.25-2.5,0.75) .. (0-4.25-2.5,1.5);
\draw[line width=0.05cm] (-2.5-4.25-2.5,0) .. controls (-1-0.75-4.25-2.5,-0) and (-1+0.5-4.25-2.5,-0.75) .. (0-4.25-2.5,-1.5);

\node[right] at (0-4.25-2.5,1.5+0.5) {\Large$2$};
\node[right] at (0-4.25-2.5,-1.5-0.5) {\Large$n-1$};

\node at (-4-4.25-2.5,0) {\huge$\mathbb{K}^{n-2}$};

\draw[->, shorten <=0.85cm,  shorten >=0.85cm, line width=0.075cm, red] (0-4.25+0.65-2.5,0) -- (-4-4.25+0.65-2.5, 1.5);
\draw[->, shorten <=0.85cm,  shorten >=0.85cm, line width=0.075cm, blue] (-4-4.25+0.65-2.5,-1.5) -- (0-4.25+0.65-2.5,0);

\node at (-1.75-4.25+0.5-2.5,+1.75) {\huge$\psi^{\,(n-2)}$};
\node at (-1.75-4.25+0.5-2.5,-1.75) {\huge$\phi^{\,(n-2)}_{\sh{F}}$};


\filldraw[black] (-10.25-0.35-2.5,0+0) circle (1.25pt);
\filldraw[black] (-10.25-2.5,0) circle (1.25pt);
\filldraw[black] (-10.25+0.35-2.5,0) circle (1.25pt);


\draw[line width=0.05cm] (-2.5-12.25-2.5,0) .. controls (-1-0.75-12.25-2.5,0) and (-1+0.5-12.25-2.5,0.75) .. (0-12.25-2.5,1.5);
\draw[line width=0.05cm] (-2.5-12.25-2.5,0) .. controls (-1-0.75-12.25-2.5,-0) and (-1+0.5-12.25-2.5,-0.75) .. (0-12.25-2.5,-1.5);

\node[right] at (0-12.25-2.5,1.5+0.5) {\Large$n-1$};
\node[right] at (0-12.25-2.5,-1.5-0.5) {\Large$2$};

\node[right] at (0-12.25-2.5,0) {\huge$\mathbb{K}^{2}$};
\node[right] at (-4.25-12.25-2.5,0) {\huge$\mathbb{K}^{1}$};

\draw[->, shorten <=0.85cm,  shorten >=0.85cm, line width=0.075cm, red] (0-12.25+0.65-2.5,0) -- (-4-12.25+0.65-2.5, 1.5);
\draw[->, shorten <=0.85cm,  shorten >=0.85cm, line width=0.075cm, blue] (-4-12.25+0.65-2.5,-1.5) -- (0-12.25+0.65-2.5,0);

\node at (-1.75-12.25+0.5-2.5,+1.75) {\huge$\psi^{\,(1)}$};
\node at (-1.75-12.25+0.5-2.5,-1.75) {\huge$\phi^{\,(1)}_{\sh{F}}$};


\draw[line width=0.05cm] (-2.5-12.25-4.25-2.5,0) .. controls (-1-0.75-12.25-4.25-2.5,0) and (-1+0.5-12.25-4.25-2.5,0.75) .. (0-12.25-4.25-2.5,1.5);
\draw[line width=0.05cm] (-2.5-12.25-4.25-2.5,0) .. controls (-1-0.75-12.25-4.25-2.5,-0) and (-1+0.5-12.25-4.25-2.5,-0.75) .. (0-12.25-4.25-2.5,-1.5);

\node[right] at (0-12.25-4.25-2.5,1.5+0.5) {\Large$n$};
\node[right] at (0-12.25-4.25-2.5,-1.5-0.5) {\Large$1$};

\node[right] at (-4-12.25-4.25-2.5,0) {\huge$0$};

\draw[->, shorten <=0.85cm,  shorten >=0.85cm, line width=0.075cm, red] (0-12.25-4.25+0.65-2.5,0) -- (-4-12.25-4.25+0.65-2.5, 1.5);
\draw[->, shorten <=0.85cm,  shorten >=0.85cm, line width=0.075cm, blue] (-4-12.25-4.25+0.65-2.5,-1.5) -- (0-12.25-4.25+0.65-2.5,0);

\node at (-1.75-12.25-4.25+0.5-2.5,+1.75) {\huge$0$};
\node at (-1.75-12.25-4.25+0.5-2.5,-1.75) {\huge$0$};


\draw[dotted, line width=0.075cm, color=green!360, double distance=0.2pt] (-22-2.5,-5) rectangle (22+2.5, 5);
\draw[dotted, line width=0.075cm, color=yellow!360, double distance=0.2pt] (-22+0.25-2.5,-5+0.25) rectangle (22-0.25+2.5, 5-0.25);

\node at (0,3.75) {\huge$U_{\mathrm{B}}\,\cap\, U_{\mathrm{T}}$};

\filldraw[black] (0+0.35,0+0) circle (1.25pt);
\filldraw[black] (0,0) circle (1.25pt);
\filldraw[black] (0-0.35,0) circle (1.25pt);


\draw[line width=0.05cm] (+2.5+2.5,0) .. controls (+1+0.75+2.5,0) and (+1-0.5+2.5,0.75) .. (0+2.5,1.5);
\draw[line width=0.05cm] (+2.5+2.5,0) .. controls (+1+0.75+2.5,-0) and (+1-0.5+2.5,-0.75) .. (0+2.5,-1.5);

\node[left] at (0+2.5,1.5+0.5) {\Large$1$};
\node[left] at (0+2.5,-1.5-0.5) {\Large$n$};

\node[left] at (0+2.5,0) {\huge$\mathbb{K}^{n}$};
\node[left] at (+4.25+2.5,0) {\huge$\mathbb{K}^{n-1}$};

\draw[->, shorten <=0.85cm,  shorten >=0.85cm, line width=0.075cm, red] (0-0.65+2.5,0) -- (+4-0.65+2.5, 1.5);
\draw[->, shorten <=0.85cm,  shorten >=0.85cm, line width=0.075cm, blue] (+4-0.65+2.5,-1.5) -- (0-0.65+2.5,0);

\node at (+1.75-0.5+2.5,+1.75) {\huge$\psi^{\,(n-1)}$};
\node at (+1.75-0.5+2.5,-1.75) {\huge$\phi^{\,(n-1)}_{\sh{F}}$};


\draw[line width=0.05cm] (+2.5+4.25+2,0) .. controls (+1+0.75+4.25+2.5,0) and (+1-0.5+4.25+2.5,0.75) .. (0+4.25+2.5,1.5);
\draw[line width=0.05cm] (+2.5+4.25+2,0) .. controls (+1+0.75+4.25+2.5,-0) and (+1-0.5+4.25+2.5,-0.75) .. (0+4.25+2.5,-1.5);

\node[left] at (0+4.25+2.5,1.5+0.5) {\Large$2$};
\node[left] at (0+4.25+2.5,-1.5-0.5) {\Large$n-1$};

\node at (+4+4.25+2.5,0) {\huge$\mathbb{K}^{n-2}$};

\draw[->, shorten <=0.85cm,  shorten >=0.85cm, line width=0.075cm, red] (0+4.25-0.65+2.5,0) -- (+4+4.25-0.65+2.5, 1.5);
\draw[->, shorten <=0.85cm,  shorten >=0.85cm, line width=0.075cm, blue] (+4+4.25-0.65+2.5,-1.5) -- (0+4.25-0.65+2.5,0);

\node at (+1.75+4.25-0.5+2.5,+1.75) {\huge$\psi^{\,(n-2)}$};
\node at (+1.75+4.25-0.5+2.5,-1.75) {\huge$\phi^{\,(n-2)}_{\sh{F}}$};


\filldraw[black] (+10.25+0.35+2.5,0+0) circle (1.25pt);
\filldraw[black] (+10.25+2.5,0) circle (1.25pt);
\filldraw[black] (+10.25-0.35+2.5,0) circle (1.25pt);


\draw[line width=0.05cm] (+2.5+12.25+2.5,0) .. controls (+1+0.75+12.25+2.5,0) and (+1-0.5+12.25+2.5,0.75) .. (0+12.25+2.5,1.5);
\draw[line width=0.05cm] (+2.5+12.25+2.5,0) .. controls (+1+0.75+12.25+2.5,-0) and (+1-0.5+12.25+2.5,-0.75) .. (0+12.25+2.5,-1.5);

\node[left] at (0+12.25+2.5,1.5+0.5) {\Large$n-1$};
\node[left] at (0+12.25+2.5,-1.5-0.5) {\Large$2$};

\node[left] at (0+12.25+2.5,0) {\huge$\mathbb{K}^{2}$};
\node[left] at (+4+12.25+2.5,0) {\huge$\mathbb{K}^{1}$};

\draw[->, shorten <=0.85cm,  shorten >=0.85cm, line width=0.075cm, red] (0+12.25-0.65+2.5,0) -- (+4+12.25-0.65+2.5, 1.5);
\draw[->, shorten <=0.85cm,  shorten >=0.85cm, line width=0.075cm, blue] (+4+12.25-0.65+2.5,-1.5) -- (0+12.25-0.65+2.5,0);

\node at (+1.75+12.25-0.5+2.5,+1.75) {\huge$\psi^{\,(1)}$};
\node at (+1.75+12.25-0.5+2.5,-1.75) {\huge$\phi^{\,(1)}_{\sh{F}}$};


\draw[line width=0.05cm] (+2.5+12.25+4.25+2.5,0) .. controls (+1+0.75+12.25+4.25+2.5,0) and (+1-0.5+12.25+4.25+2.5,0.75) .. (0+12.25+4.25+2.5,1.5);
\draw[line width=0.05cm] (+2.5+12.25+4.25+2.5,0) .. controls (+1+0.75+12.25+4.25+2.5,-0) and (+1-0.5+12.25+4.25+2.5,-0.75) .. (0+12.25+4.25+2.5,-1.5);

\node[left] at (0+12.25+4.25+2.5,1.5+0.5) {\Large$n$};
\node[left] at (0+12.25+4.25+2.5,-1.5-0.5) {\Large$1$};

\node[left] at (+4+12.25+4.25+2.5,0) {\huge$0$};

\draw[->, shorten <=0.85cm,  shorten >=0.85cm, line width=0.075cm, red] (0+12.25+4.25-0.65+2.5,0) -- (+4+12.25+4.25-0.65+2.5, 1.5);
\draw[->, shorten <=0.85cm,  shorten >=0.85cm, line width=0.075cm, blue] (+4+12.25+4.25-0.65+2.5,-1.5) -- (0+12.25+4.25-0.65+2.5,0);

\node at (+1.75+12.25+4.25-0.5+2.5,+1.75) {\huge$0$};
\node at (+1.75+12.25+4.25-0.5+2.5,-1.75) {\huge$0$};

\endscope    
\end{tikzpicture}
\end{center}
\caption{An object $\sh{F}$ of the category $\ccs{1}{e_{n}}$ on the intersection $U_{\mathrm{B}}\,\cap\,U_{\mathrm{T}}$.}
\label{Fig: an object F on the intersection of U_B and U_T for the unlink}
\end{figure}
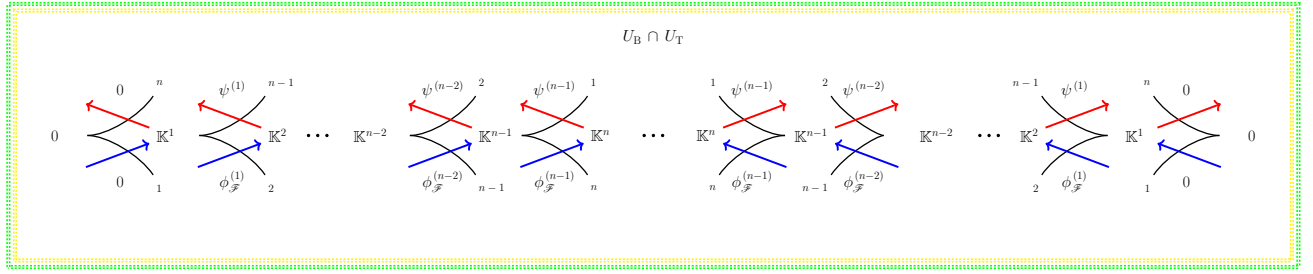
\end{proof}

Finally, we formalize how the description of $\sh{F}$ in terms of linear maps, as established in Proposition~\eqref{Prop: linear map description of a sheaf for the unlink on n strands}, naturally gives rise to a geometric characterization via complete flags in $\mathbb{K}^{n}$. 

\begin{proposition}\label{Sheaves for the unlink on n strands}
Let $e_{n}\in \mathrm{Br}_{n}^{+}$ be the trivial braid word, $\sh{F}$ an object of the category $\ccs{1}{e_{n}}$, and $\mathcal{U}_{\Lambda(e_{n})}=\big\{ U_{0}, U_{\mathrm{B}}, U_{\mathrm{T}}\big\}$ the open cover of~$\mathbb{R}^{2}$ introduced in Construction~\eqref{Const: finite open cover of R^2 for the unlink}. Then, the following statements hold: 
\begin{itemize}
\item On $U_{\mathrm{0}}$, $\sh{F}$ is identically zero.
\item On $U_{\mathrm{T}}$, $\sh{F}$ is characterized by a complete flag $\fl{F}_{0}$ in $\mathbb{K}^{n}$. 
\item On $U_{\mathrm{B}}$, $\sh{F}$ is characterized by a complete flag $\fl{F}_{1}$ in $\mathbb{K}^{n}$. 
\item \textbf{Compatibility condition}: $\fl{F}_{0}$ and $\fl{F}_{1}$ are completely opposite. 
\end{itemize}
\end{proposition}
\begin{proof}
By Proposition~\eqref{Prop: linear map description of a sheaf for the unlink on n strands}, we know that:
\begin{itemize}
\justifying
\item  On $U_{0}$, $\sh{F}$ is identically zero.
\item  On $U_{\mathrm{T}}$, $\sh{F}$ is specified by a collection of $n-1$ surjective linear maps $\big\{ \psi^{(i)}_{\sh{F}}:\mathbb{K}^{i+1}\to \mathbb{K}^{i} \big\}_{i=1}^{n-1}$.
\item  On $U_{\mathrm{B}}$, $\sh{F}$ is specified by a collection of $n-1$ injective linear maps $\big\{ \phi^{(i)}_{\sh{F}}:\mathbb{K}^{i}\to \mathbb{K}^{i+1}\big\}_{i=1}^{n-1}$.
\item  \textbf{Compatibility conditions}: For each $i\in[1,n-1]$, 
\begin{equation*}
\psi^{(i)}_{\sh{F}}\circ \phi^{(i)}_{\sh{F}}=\mathrm{id}_{\mathbb{K}^{i}}\, .   
\end{equation*}
\end{itemize}
To verify the result, it suffices to show that the above data determines a pair of completely opposite flags $\fl{F}_{0}$ and $\fl{F}_{1}$ in $\mathbb{K}^{n}$, thereby providing a geometric description of $\sh{F}$. To this end, relying on the notions of type $\mathcal{I}$ and type $\mathcal{K}$ flags introduced in Definition~\eqref{Def:flags and adapted bases}--\eqref{Def: type I flag}--\eqref{Def: type K flag}, we proceed as follows:

\begin{enumerate}
\justifying
\item Since the maps $\big\{\psi^{(i)}_{\sh{F}}\big\}_{i=1}^{n-1}$ are surjective, we assign to $\sh{F}$ on $U_{\mathrm{T}}$ the complete flag $\fl{F}_{0}$ in $\mathbb{K}^{n}$ given by
\begin{equation*}
\fl{F}_{0}:=\prescript{}{\mathcal{K}\,}{\fl{F}}\big(\,\psi^{(1)}_{\sh{F}},\dots,\psi^{(n-1)}_{\sh{F}}\,\big)\, .
\end{equation*}

\item Since the maps $\big\{\phi^{(i)}_{\sh{F}}\big\}_{i=1}^{n-1}$ are injective, we assign to $\sh{F}$ on $U_{\mathrm{B}}$ the complete flag $\fl{F}_{1}$ in $\mathbb{K}^{n}$ given by
\begin{equation*}
\fl{F}_{1}:=\prescript{}{\mathcal{I}\,}{\fl{F}}\big(\,\phi^{(1)}_{\sh{F}},\dots,\phi^{(n-1)}_{\sh{F}}\,\big)\, .
\end{equation*}
\end{enumerate}

Finally, observe that the compatibility conditions on the maps $\big\{\psi^{(i)}_{\sh{F}}\big\}_{i=1}^{n-1}$ and $\big\{\phi^{(i)}_{\sh{F}}\big\}_{i=1}^{n-1}$allow a direct application of Lemma~\eqref{Lemma: completely opposite flags from linear maps}, which shows that $\fl{F}_{0}$ and $\fl{F}_{1}$ are completely opposite, as desired.
\end{proof}

With the above results at hand, we consider the analysis of the objects of the category $\ccs{1}{e_{n}}$ complete. Next, we extend our discussion to general positive braid words $\beta\in \mathrm{Br}^{+}_{n}$, aiming to establish a geometric characterization of the objects of the category $\ccs{1}{\beta}$.

\subsection{A Geometric Characterization of the Objects for General Positive Braids}
Let $\beta:=\sigma_{i_{1}}\cdots\sigma_{i_{\ell}}\in\mathrm{Br}^{+}_{n}$ be a positive braid word. In this subsection, we focus our attention on the study of the objects of the category $\ccs{1}{\beta}$. In particular, our goal is to describe these objects geometrically, thereby extending the approach previously developed for the Legendrian unlink $\Lambda(e_{n})\subset (\mathbb{R}^{3},\xi_{\mathrm{std}})$ on $n$ strands. Following~\cite{STZ1}, we provide a characterization of the objects of the category $\ccs{1}{\beta}$ using configurations of complete flags in $\mathbb{K}^{n}$ subject to certain compatibility conditions determined by the singularities in the front projection $\Pi_{x,z}(\Lambda(\beta))\subset \mathbb{R}^{2}$ of the Legendrian link $\Lambda(\beta)\subset (\mathbb{R}^{3},\xi_{\mathrm{std}})$. 

Let $\sh{F}$ be an object of the category $\ccs{1}{\beta}$. Next, we construct an open cover of $\,\mathbb{R}^{2}$ adapted to the front projection $\Pi_{x,z}(\Lambda(\beta))$. The purpose of this construction is to decompose the plane into simple regions where the local behavior of $\sh{F}$ can be examined independently and subsequently assembled into a concise global description.

\begin{construction}\label{Cons: Finite open cover for R^2}
Let $\beta:=\sigma_{i_{1}}\cdots\sigma_{i_{\ell}} \in \mathrm{Br}^{+}_{n}$ be a positive braid word. Given the front projection $\Pi_{x,z}(\Lambda(\beta))\subset \mathbb{R}^{2}$ of the Legendrian link $\Lambda(\beta)\subset(\mathbb{R}^{3}, \xi_{\mathrm{std}})$, we introduce $\mathcal{U}_{\Lambda(\beta)}:=\big\{U_{0}, U_{\mathrm{B}}, U_{\mathrm{L}}, U_{\mathrm{R}}, U_{\mathrm{T}} \big\}$ to denote the finite open cover of $\,\mathbb{R}^{2}$ illustrated in Figure~\eqref{Front diagram decomposed into two regions}. Specifically, we assume that:
\begin{itemize}
\justifying
\item $U_{0}$ is an unbounded open subset of $\,\mathbb{R}^{2}$ such that $U_{0}\,\cap\,\Pi_{x,z}(\Lambda(\beta))=\emptyset$. In Figure \eqref{Front diagram decomposed into two regions}, the region $U_{0}$ is depicted in purple. 

\item $U_{\mathrm{B}}$ is a bounded open subset of $\,\mathbb{R}^{2}$ such that the intersection $U_{\mathrm{B}}\,\cap\, \Pi_{x,z}(\Lambda(\beta))$ comprises the braid diagram of $\beta$. In Figure \eqref{Front diagram decomposed into two regions}, the region $U_{\mathrm{B}}$ is illustrated in pink. 

\item $U_{\mathrm{L}}$ is a bounded open subset of $\,\mathbb{R}^{2}$ such that the intersection $U_{\mathrm{L}}\,\cap\, \Pi_{x,z}(\Lambda(\beta))$ consists of the $n$ arcs on the left side of the braid diagram of $\beta$. In Figure \eqref{Front diagram decomposed into two regions}, the region $U_{\mathrm{L}}$ is shown in yellow.

\item $U_{\mathrm{R}}$ is a bounded open subset of $\,\mathbb{R}^{2}$ such that the intersection $U_{\mathrm{R}}\,\cap\, \Pi_{x,z}(\Lambda(\beta))$ consists of the $n$ arcs on the right side of the braid diagram of $\beta$. In Figure \eqref{Front diagram decomposed into two regions}, the region $U_{\mathrm{R}}$ is also depicted in yellow. 

\item $U_{\mathrm{T}}$ is a bounded open subset of $\,\mathbb{R}^{2}$ such that the intersection $U_{\mathrm{T}}\,\cap\, \Pi_{x,z}(\Lambda(\beta))$ consists of the rainbow on $n$ strands that closes the braid diagram of $\beta$ in $\mathbb{R}^{2}$ to form the front projection $\Pi_{x,z}(\Lambda(\beta))$. In Figure \eqref{Front diagram decomposed into two regions}, the region $U_{\mathrm{T}}$ is illustrated in green. 
\end{itemize}

\begin{figure}
\centering
\begin{tikzpicture}
\useasboundingbox (-8.25,-5.25) rectangle (8.25, 5);
\scope[transform canvas={scale=0.6}]

\def\regionone{ (-4.25-0.25,-6.5-0.5) -- (-4.25-0.25, -0.5+1) -- (4.25+0.25,-0.5+1) -- (4.25+0.25,-6.5-0.5) -- (-4.25-0.25,-6.5-0.5)}

\def\regiontwo{ (-11.75,-0.5+1-1) -- (-11.75,6+0.5) -- (11.75,6+0.5) -- (+11.75,-0.5+1-1) -- (-11.75,-0.5+1-1) }

\def\regionthree{(-13.75,-6.5-2) rectangle (13.75, 6+2) (-11.25,-6.5) rectangle (11.25, 6)}

\def\regionfour{(-11.25-0.5,-6.5-0.5) rectangle (-11.25+7.25+0.25, 0+0.5)}

\def\regionfive{(+11.25-7.25-0.25, 0+0.5) rectangle (+11.25+0.5,-6.5-0.5)}





\fill[green!25, fill opacity=0.7] \regiontwo;
\fill[blue!30, even odd rule, fill opacity=0.4] \regionthree;
\fill[pink!70, fill opacity=0.6] \regionone;
\fill[yellow!55, fill opacity=0.6] \regionfour;
\fill[yellow!55, fill opacity=0.6] \regionfive;

\draw[dotted, line width=0.065cm, color=pink!360, double distance=0.2pt] \regionone;
\draw[dotted, line width=0.065cm, color=green!360, double distance=0.2pt] \regiontwo; 
\draw[dotted, line width=0.065cm, color=purple!360, double distance=0.2pt] \regionthree;
\draw[dotted, line width=0.065cm, color=yellow!360, double distance=0.2pt] \regionfour; 
\draw[dotted, line width=0.065cm, color=yellow!360, double distance=0.2pt] \regionfive;

\draw[very thick] (-4-0.5,2) -- (4+0.5, 2);
\draw[very thick] (-4-0.5,-2) -- (-3, -2);
\draw[very thick] (3,-2) -- (4+0.5, -2);
\draw[very thick] (4+0.5,2) .. controls (4+0.5+1.5,2) and (4+0.5+2-0.75,0) .. (4+0.5+2+0.1,0);
\draw[very thick] (4+0.5,-2) .. controls (4+0.5+1.5,-2) and (4+0.5+2-0.75,0) .. (4+0.5+2+0.1,0);
\draw[very thick] (-4-0.5-2-0.1,0) .. controls (-4-0.5-2+0.75,0) and (-4-0.5-1.5,2) .. (-4-0.5,2);
\draw[very thick] (-4-0.5-2-0.1,0) .. controls (-4-0.5-2+0.75,0) and (-4-0.5-1.5,-2) .. (-4-0.5,-2);


\draw[very thick] (-4-0.5-1,2+0.75) -- (4+0.5+1, 2+0.75);
\draw[very thick] (-4-0.5-1,-2-0.75) -- (-3, -2-0.75);
\draw[very thick] (3,-2-0.75) -- (4+0.5+1, -2-0.75);
\draw[very thick] (4+0.5+1,2+0.75) .. controls (4+0.5+1+1.15,2+0.75) and (4+0.5+1+2-0.5,0) .. (4+0.5+1+2+0.25,0);
\draw[very thick] (4+0.5+1,-2-0.75) .. controls (4+0.5+1+1.15,-2-0.75) and (4+0.5+1+2-0.5,0) .. (4+0.5+1+2+0.25,0);
\draw[very thick] (-4-0.5-1-2-0.25,0) .. controls (-4-0.5-1-2+0.5,0) and (-4-0.5-1-1.15,2+0.75) .. (-4-0.5-1,2+0.75);
\draw[very thick] (-4-0.5-1-2-0.25,0) .. controls (-4-0.5-1-2+0.5,0) and (-4-0.5-1-1.15,-2-0.75) .. (-4-0.5-1,-2-0.75);


\draw[very thick] (-4-0.5-1-1,+2+3.5-0.5-0.75) -- (4+0.5+1+1,+2+3.5-0.5-0.75);
\draw[very thick] (-4-0.5-1-1,-2-3.5+0.5+0.75) -- (-3,-2-3.5+0.5+0.75);
\draw[very thick] (3,-2-3.5+0.5+0.75) -- (4+0.5+1+1,-2-3.5+0.5+0.75);
\draw[very thick] (4+0.5+1+1,+2+3.5-0.5-0.75) .. controls (4+0.5+1+1+1.15+0.65,+2+3.5-0.5-0.75) and (4+0.5+1+1+2-0.5+0.75,0) .. (4+0.5+1+1+2+0.75+0.35,0);
\draw[very thick] (4+0.5+1+1,-2-3.5+0.5+0.75) .. controls (4+0.5+1+1+1.15+0.65,-2-3.5+0.5+0.75) and (4+0.5+1+1+2-0.5+0.75,0) .. (4+0.5+1+1+2+0.75+0.35,0);
\draw[very thick] (-4-0.5-1-1-2-0.75-0.35,0) .. controls (-4-0.5-1-1-2+0.5-0.75,0) and (-4-0.5-1-1-1.15-0.65,+2+3.5-0.5-0.75) .. (-4-0.5-1-1,+2+3.5-0.5-0.75);
\draw[very thick] (-4-0.5-1-1-2-0.75-0.35,0) .. controls (-4-0.5-1-1-2+0.5-0.75,0) and (-4-0.5-1-1-1.15-0.65,-2-3.5+0.5+0.75) .. (-4-0.5-1-1,-2-3.5+0.5+0.75);


\draw[very thick] (-4-0.5-1-1-1,+2+3.5-0.5) -- (4+0.5+1+1+1,+2+3.5-0.5);
\draw[very thick] (-4-0.5-1-1-1,-2-3.5+0.5) -- (-3,-2-3.5+0.5);
\draw[very thick] (3,-2-3.5+0.5) -- (4+0.5+1+1+1,-2-3.5+0.5);
\draw[very thick] (4+0.5+1+1+1,+2+3.5-0.5) .. controls (4+0.5+1+1+1+1.15+0.65,+2+3.5-0.5) and (4+0.5+1+1+1+2-0.5+0.75+0.25,0) .. (4+0.5+1+1+1+2+0.75+0.25+0.35,0);
\draw[very thick] (4+0.5+1+1+1,-2-3.5+0.5) .. controls (4+0.5+1+1+1+1.15+0.65,-2-3.5+0.5) and (4+0.5+1+1+1+2-0.5+0.75+0.25,0) .. (4+0.5+1+1+1+2+0.75+0.25+0.35,0);
\draw[very thick] (-4-0.5-1-1-1-2-0.75-0.25-0.35,0) .. controls (-4-0.5-1-1-1-2+0.5-0.75-0.25,0) and (-4-0.5-1-1-1-1.15-0.65,+2+3.5-0.5) .. (-4-0.5-1-1-1,+2+3.5-0.5);
\draw[very thick] (-4-0.5-1-1-1-2-0.75-0.25-0.35,0) .. controls (-4-0.5-1-1-1-2+0.5-0.75-0.25,0) and (-4-0.5-1-1-1-1.15-0.65,-2-3.5+0.5) .. (-4-0.5-1-1-1,-2-3.5+0.5);


\draw[very thick] (-3,-2-3.5) rectangle (3,2-3.5);

\node at (-3+0.85,-2) {\large$n$};
\node at (-3+0.85,-2-0.75) {\large$n-1$};
\node at (-3+0.85,-2-3.5+0.5+0.75) {\large$2$};
\node at (-3+0.85,-2-3.5+0.5) {\large$1$};

\node at (3-0.85,-2) {\large$n$};
\node at (3-0.85,-2-0.75) {\large$n-1$};
\node at (3-0.85,-2-3.5+0.5+0.75) {\large$2$};
\node at (3-0.85,-2-3.5+0.5) {\large$1$};

\node at  (0,5.5) {\Large$U_{\mathrm{T}}$};
\node at  (0,-3.5) {\huge$\beta$};
\node at  (0,-3.5-2.5) {\Large$U_{\mathrm{B}}$};
\node at  (0-7.5,-3.5-2.5) {\Large$U_{\mathrm{L}}$};
\node at  (0+7.5,-3.5-2.5) {\Large$U_{\mathrm{R}}$};

\filldraw[black] (-3+0.85,-3.5+0.25) circle (1pt);
\filldraw[black] (-3+0.85,-3.5) circle (1pt);
\filldraw[black] (-3+0.85,-3.5-0.25) circle (1pt);

\filldraw[black] (3-0.85,-3.5+0.25) circle (1pt);
\filldraw[black] (3-0.85,-3.5) circle (1pt);
\filldraw[black] (3-0.85,-3.5-0.25) circle (1pt);

\filldraw[black] (-3-0.5-1.5,-3.5+0.25) circle (1pt);
\filldraw[black] (-3-0.5-1.5,-3.5) circle (1pt);
\filldraw[black] (-3-0.5-1.5,-3.5-0.25) circle (1pt);

\filldraw[black] (3+0.5+1.5,-3.5+0.25) circle (1pt);
\filldraw[black] (3+0.5+1.5,-3.5) circle (1pt);
\filldraw[black] (3+0.5+1.5,-3.5-0.25) circle (1pt);

\filldraw[black] (-3-0.5-1.5,-3.5+0.25+7) circle (1pt);
\filldraw[black] (-3-0.5-1.5,-3.5+7) circle (1pt);
\filldraw[black] (-3-0.5-1.5,-3.5-0.25+7) circle (1pt);

\filldraw[black] (3+0.5+1.5,-3.5+0.25+7) circle (1pt);
\filldraw[black] (3+0.5+1.5,-3.5+7) circle (1pt);
\filldraw[black] (3+0.5+1.5,-3.5-0.25+7) circle (1pt);

\filldraw[black] (-8.45-0.25,0) circle (1pt);
\filldraw[black] (-8.45,0) circle (1pt);
\filldraw[black] (-8.45+0.25,0) circle (1pt);

\filldraw[black] (8.45-0.25,0) circle (1pt);
\filldraw[black] (8.45,0) circle (1pt);
\filldraw[black] (8.45+0.25,0) circle (1pt);

\node at (12.75, 7.25) {\LARGE$\mathbb{R}^{2}_{x,z}$};

\node at (0, 7.25) {\Large$U_{0}$};




\endscope
\end{tikzpicture}
\caption{The front projection $\Pi_{x,z}(\Lambda(\beta)) \subset \mathbb{R}^2$ of the Legendrian link $\Lambda(\beta) \subset (\mathbb{R}^3, \xi_{\mathrm{std}})$, along with the finite open cover $\mathcal{U}_{\Lambda(\beta)}=\big\{U_{0}, U_{\mathrm{B}}, U_{\mathrm{L}}, U_{\mathrm{R}}, U_{\mathrm{T}} \big\}$ of $\mathbb{R}^{2}$. The regions $U_{\mathrm{L}}$ and $U_{\mathrm{R}}$ are shown in yellow, while $U_0$, $U_{\mathrm{B}}$, and $U_{\mathrm{T}}$ are depicted in purple, pink, and green, respectively.}
\label{Front diagram decomposed into two regions}
\end{figure}
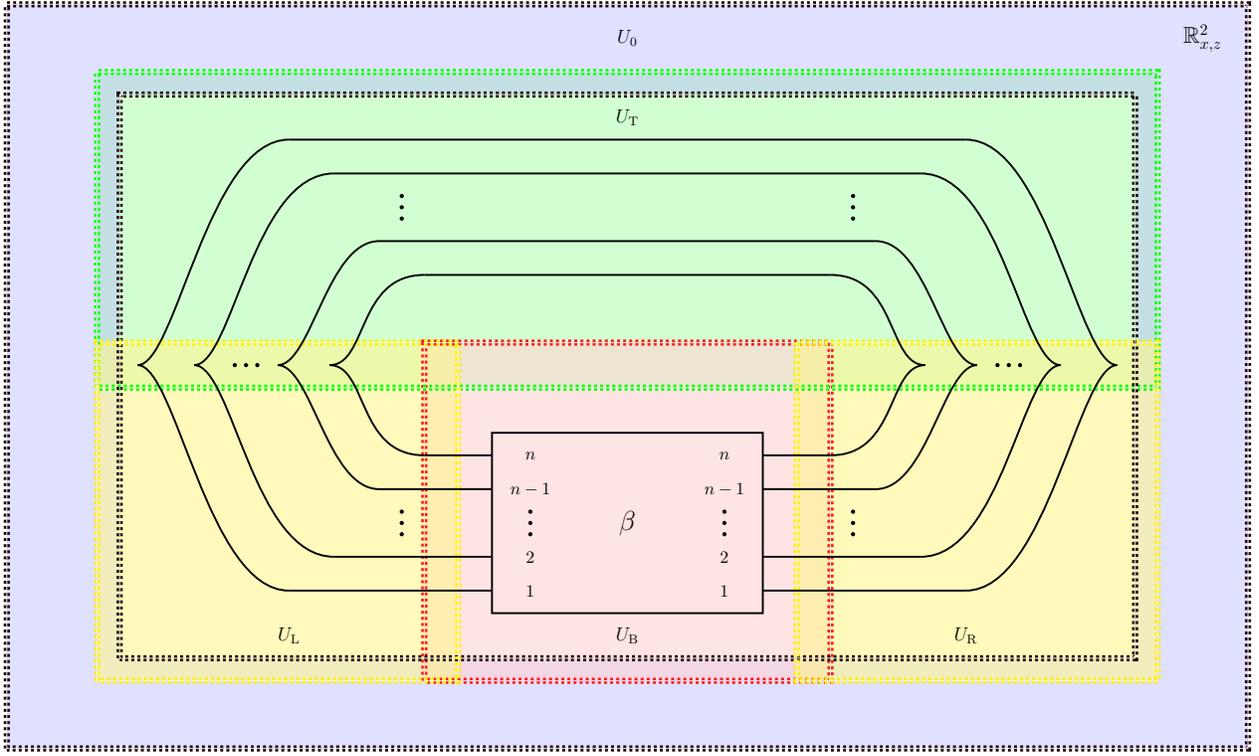
\end{construction}

Next, by applying the sheaf axioms to the open cover $\mathcal{U}_{\Lambda(\beta)}=\big\{U_{0}, U_{\mathrm{B}}, U_{\mathrm{L}}, U_{\mathrm{R}}, U_{\mathrm{T}} \big\}$ of $\mathbb{R}^{2}$ introduced in Construction~\eqref{Cons: Finite open cover for R^2}, we provide a global description of $\sh{F}$ in terms of its local data and the associated gluing conditions. Among the subsets of the cover $\mathcal{U}_{\Lambda(\beta)}$, the region $U_{\mathrm{B}}$ is particularly important, as it contains all the crossings in the front projection $\Pi_{x,z}(\Lambda(\beta))$, and consequently, the behavior of $\sh{F}$ on $U_{\mathrm{B}}$ is more intricate and warrants a detailed analysis. We therefore begin by analyzing $\sh{F}$ on the remaining subsets, postponing the study of its behavior on $U_{\mathrm{B}}$ until afterward. With this strategy in place, we now present the following result.

\begin{lemma}\label{Lemma: linear map description of an object on the regions U_T, U_L, and U_R}
Let $\beta:=\sigma_{i_{1}}\cdots \sigma_{i_{\ell}}\in\mathrm{Br}_{n}^{+}$ be a positive braid word, $\sh{F}$ an object of the category $\ccs{1}{\beta}$, and $\mathcal{U}_{\Lambda(\beta)}=\big\{U_{0}, U_{\mathrm{B}}, U_{\mathrm{L}}, U_{\mathrm{R}}, U_{\mathrm{T}} \big\}$ the open cover of $\mathbb{R}^{2}$ introduced in Construction~\eqref{Cons: Finite open cover for R^2}. Then, the following statements hold:    
\begin{itemize}
\justifying
\item On $U_{0}$, $\sh{F}$ is identically zero.
\item On $U_{\mathrm{T}}$, $\sh{F}$ is specified by a collection of $n-1$ surjective linear maps $\big\{\psi^{(i)}_{\sh{F}}:\mathbb{K}^{i+1}\to \mathbb{K}^{i}\big\}_{i=1}^{n-1}$, as illustrated in Figure~\eqref{Fig: an object F in the region U_T}.
\item On $U_{\mathrm{L}}$, $\sh{F}$ is specified by a collection of $n-1$ injective linear maps $\big\{\alpha^{(i)}_{\sh{F}}:\mathbb{K}^{i}\to \mathbb{K}^{i+1}\big\}_{i=1}^{n-1}$, as illustrated in Figure~\eqref{Fig: an object F in the region U_L}. 
\item On $U_{\mathrm{R}}$, $\sh{F}$ is specified by a collection of $n-1$ injective linear maps $\big\{\beta^{(i)}_{\sh{F}}:\mathbb{K}^{i}\to \mathbb{K}^{i+1}\big\}_{i=1}^{n-1}$, as illustrated in Figure~\eqref{Fig: an object F in the region U_R}. 
\item \textbf{Compatibility conditions}: For each $i\in[1,n-1]$, 
\begin{equation*}
\psi^{(i)}_{\sh{F}}\circ \alpha^{(i)}_{\sh{F}}=\mathrm{id}_{\mathbb{K}^{i}}\, , \quad \text{and} \quad \psi^{(i)}_{\sh{F}}\circ \beta^{(i)}_{\sh{F}}=\mathrm{id}_{\mathbb{K}^{i}}\, .    
\end{equation*}
\end{itemize}
\end{lemma}
\begin{proof}
To begin, consider the front projection $\Pi_{x,z}(\Lambda(\beta))\subset \mathbb{R}^{2}$ of the Legendrian link $\Lambda(\beta)\subset (\mathbb{R}^{3}, \xi_{\mathrm{std}})$, which is depicted in Figure~\eqref{Front diagram of the rainbow closure of a braid in n strands}. In particular, recall that in $\Pi_{x,z}(\Lambda(\beta))$, the strands at bottom have Maslov potential $0$ and the strands at the top have Maslov potential $1$. Thus, by the microlocal rank conditions and the microlocal support conditions near the arcs, we obtain that:
\begin{itemize}
\item For all $q\in U_{0}$, the stalk of $\sh{F}$ at $q$ is trivial; that is, $\sh{F}_{q}=0$. In other words, $\sh{F}$ is identically zero on $ U_{0}$.

\item On $U_{\mathrm{T}}$, $\sh{F}$ is defined by a collection of $n-1$ linear maps $\big\{\psi^{(i)}_{\sh{F}}:\mathbb{K}^{i+1}\to \mathbb{K}^{i}\big\}_{i=1}^{n-1}$, as shown in Figure~\eqref{Fig: an object F in the region U_T}. 

\item On $U_{\mathrm{L}}$, $\sh{F}$ is defined by a collection of $n-1$ linear maps $\big\{\alpha^{(i)}_{\sh{F}}:\mathbb{K}^{i}\to \mathbb{K}^{i+1}\big\}_{i=1}^{n-1}$, as shown in Figure~\eqref{Fig: an object F in the region U_L}. 

\item On $U_{\mathrm{R}}$, $\sh{F}$ is defined by a collection of $n-1$ linear maps $\big\{\beta^{(i)}_{\sh{F}}:\mathbb{K}^{i}\to \mathbb{K}^{i+1}\big\}_{i=1}^{n-1}$, as shown in Figure~\eqref{Fig: an object F in the region U_R}. 
\end{itemize}

\begin{figure}
\begin{center}
\begin{tikzpicture}
\useasboundingbox (-8,-4.25) rectangle (2.25, 4.25);
\scope[transform canvas={scale=0.55}]

\draw[line width=0.05cm] (-3, 3+2.5+1.25) .. controls (-1-1.25,3+2.5+1.25) and (-1-1.25,3+2.5) .. (-1.5,3+2.5);
\draw[line width=0.05cm] (-6, 3+2.5+1.25) .. controls (-3-1.25,3+2.5+1.25) and (-3-1.25,3) .. (-1.5,3);
\draw[line width=0.05cm] (-11, 3+2.5+1.25) .. controls (-5.5-1.75,3+2.5+1.25) and (-5.5-1.75, -3) .. (-1.5, -3);
\draw[line width=0.05cm] (-14, 3+2.5+1.25) .. controls (-7-3-0.25,3+2.75+1.25) and (-7-3.5, -3-2.5) .. (-1.5, -3-2.5);

\draw[line width=0.05cm, dashed] (-3, 3+2.5+1.25) .. controls (-2-0.35,3+2.5+1.25) and (-2-0.25,3+2.5+1.25+0.25) .. (-2,3+2.5+1.25+0.5);
\draw[line width=0.05cm, dashed] (-3-3, 3+2.5+1.25) .. controls (-2-0.35-3,3+2.5+1.25) and (-2-0.25-3,3+2.5+1.25+0.25) .. (-2-3,3+2.5+1.25+0.5);
\draw[line width=0.05cm, dashed] (-3-8, 3+2.5+1.25) .. controls (-2-0.35-8,3+2.5+1.25) and (-2-0.25-8,3+2.5+1.25+0.25) .. (-2-8,3+2.5+1.25+0.5);
\draw[line width=0.05cm, dashed] (-3-11, 3+2.5+1.25) .. controls (-2-0.35-11,3+2.5+1.25) and (-2-0.25-11,3+2.5+1.25+0.25) .. (-2-11,3+2.5+1.25+0.5);

\draw[shorten <=0.5cm,  shorten >=0.5cm, line width=0.05cm] (-2, 3+2.5) -- (2.25, 3+2.5);
\draw[shorten <=0.5cm,  shorten >=0.5cm, line width=0.05cm] (-2, 3) -- (2.25, 3);
\draw[shorten <=0.5cm,  shorten >=0.5cm, line width=0.05cm] (-2, -3) -- (2.25, -3);
\draw[shorten <=0.5cm,  shorten >=0.5cm, line width=0.05cm] (-2, -3-2.5) -- (2.25, -3-2.5);

\node at (0,3+2.5+1.25) {\LARGE$\mathbb{K}^{n}$};
\node at (0,3+1.25) {\LARGE$\mathbb{K}^{n-1}$};
\node at (0,3-1.25) {\LARGE$\mathbb{K}^{n-2}$};

\node at (0,-3+1.25) {\LARGE$\mathbb{K}^{2}$};
\node at (0,-3-1.25) {\LARGE$\mathbb{K}^{1}$};
\node at (0,-3-2.5-1.25) {\LARGE$0$};

\draw[->, shorten <=0.35cm,  shorten >=0.35cm, line width=0.055cm, blue] (0,3+2.5-1.25) -- (0,3+2.5+1.25);
\draw[->, shorten <=0.35cm,  shorten >=0.35cm, line width=0.055cm, blue] (0,3-1.25) -- (0,3+1.25);

\draw[->, shorten <=0.35cm,  shorten >=0.35cm, line width=0.055cm, blue] (0, -3-1.25) -- (0, -3+1.25);
\draw[->, shorten <=0.35cm,  shorten >=0.35cm, line width=0.055cm, blue] (0, -3-2.5-1.25) -- (0, -3-2.5+1.25);

\node[right] at (+0.25,+3+2.5-0.5) {\LARGE$\alpha^{\,(n-1)}_{\sh{F}}$};
\node[right] at (+0.25,+3-0.5) {\LARGE$\alpha^{\,(n-2)}_{\sh{F}}$};

\node[right] at (+0.25,-3-0.5) {\LARGE$\alpha^{\,(1)}_{\sh{F}}$};
\node[right] at (+0.25,-3-2.5-0.5) {\LARGE$0$};

\filldraw[black] (0,0+0.35) circle (1pt);
\filldraw[black] (0,0) circle (1pt);
\filldraw[black] (0,0-0.35) circle (1pt);

\filldraw[black] (+2.75,0+0.35) circle (1pt);
\filldraw[black] (+2.75,0) circle (1pt);
\filldraw[black] (+2.75,0-0.35) circle (1pt);

\filldraw[black] (-8-0.35,3+2.5+1.25) circle (1pt);
\filldraw[black] (-8,3+2.5+1.25) circle (1pt);
\filldraw[black] (-8+0.35,3+2.5+1.25) circle (1pt);

\node at (+2.75,3+2.5) {\large$n$};
\node at (+2.75,3) {\large$n-1$};
\node at (+2.75,-3) {\large$2$};
\node at (+2.75,-3-2.5) {\large$1$};

\draw[dotted, line width=0.065cm, color=yellow!360, double distance=0.2pt] (-14.5,-7.5) rectangle (4,7.5);

\node at (-7.25,-3-2.5-0.5) {\LARGE$U_{\mathrm{L}}$};

\endscope    
\end{tikzpicture}
\end{center}
\caption{An object $\sh{F}$ of the cohomological category $\ccs{1}{\beta}$ on the region $U_{\mathrm{L}}$.}
\label{Fig: an object F in the region U_L}
\end{figure}
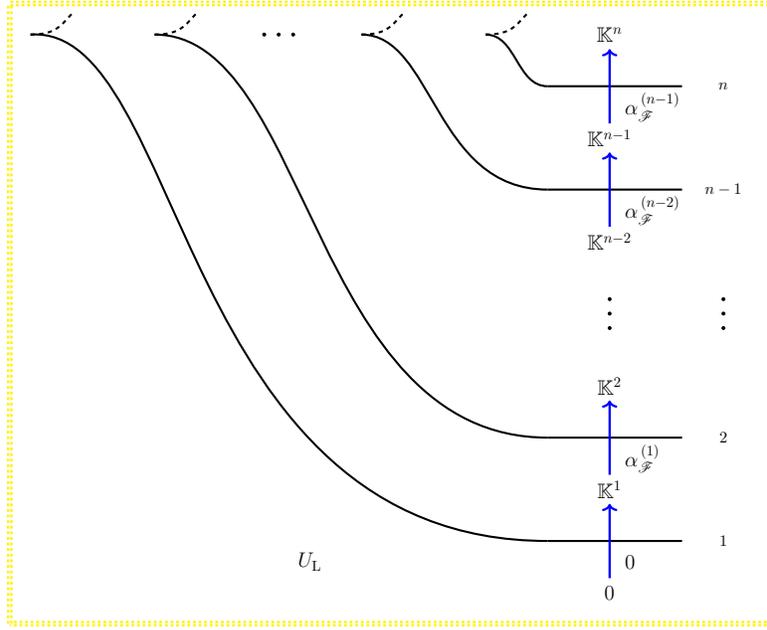

\begin{figure}
\begin{center}
\begin{tikzpicture}
\useasboundingbox (-2.25,-4.25) rectangle (8, 4.25);
\scope[transform canvas={scale=0.55}]

\draw[line width=0.05cm] (+3, 3+2.5+1.25) .. controls (+1+1.25,3+2.5+1.25) and (+1+1.25,3+2.5) .. (+1.5,3+2.5);
\draw[line width=0.05cm] (+6, 3+2.5+1.25) .. controls (+3+1.25,3+2.5+1.25) and (+3+1.25,3) .. (+1.5,3);
\draw[line width=0.05cm] (+11, 3+2.5+1.25) .. controls (+5.5+1.75,3+2.5+1.25) and (+5.5+1.75, -3) .. (+1.5, -3);
\draw[line width=0.05cm] (+14, 3+2.5+1.25) .. controls (+7+3+0.25,3+2.75+1.25) and (+7+3.5, -3-2.5) .. (+1.5, -3-2.5);

\draw[line width=0.05cm, dashed] (+3, 3+2.5+1.25) .. controls (+2+0.35,3+2.5+1.25) and (+2+0.25,3+2.5+1.25+0.25) .. (+2,3+2.5+1.25+0.5);
\draw[line width=0.05cm, dashed] (+3+3, 3+2.5+1.25) .. controls (+2+0.35+3,3+2.5+1.25) and (+2+0.25+3,3+2.5+1.25+0.25) .. (+2+3,3+2.5+1.25+0.5);
\draw[line width=0.05cm, dashed] (+3+8, 3+2.5+1.25) .. controls (+2+0.35+8,3+2.5+1.25) and (+2+0.25+8,3+2.5+1.25+0.25) .. (+2+8,3+2.5+1.25+0.5);
\draw[line width=0.05cm, dashed] (+3+11, 3+2.5+1.25) .. controls (+2+0.35+11,3+2.5+1.25) and (+2+0.25+11,3+2.5+1.25+0.25) .. (+2+11,3+2.5+1.25+0.5);

\draw[shorten <=0.5cm,  shorten >=0.5cm, line width=0.05cm] (-2.25, 3+2.5) -- (2, 3+2.5);
\draw[shorten <=0.5cm,  shorten >=0.5cm, line width=0.05cm] (-2.25, 3) -- (2, 3);

\draw[shorten <=0.5cm,  shorten >=0.5cm, line width=0.05cm] (-2.25, -3) -- (2, -3);
\draw[shorten <=0.5cm,  shorten >=0.5cm, line width=0.05cm] (-2.25, -3-2.5) -- (2, -3-2.5);

\node at (0,3+2.5+1.25) {\LARGE$\mathbb{K}^{n}$};
\node at (0,3+1.25) {\LARGE$\mathbb{K}^{n-1}$};
\node at (0,3-1.25) {\LARGE$\mathbb{K}^{n-2}$};

\node at (0,-3+1.25) {\LARGE$\mathbb{K}^{2}$};
\node at (0,-3-1.25) {\LARGE$\mathbb{K}^{1}$};
\node at (0,-3-2.5-1.25) {\LARGE$0$};

\draw[->, shorten <=0.45cm,  shorten >=0.45cm, line width=0.055cm, blue] (0,3+2.5-1.25) -- (0,3+2.5+1.25);
\draw[->, shorten <=0.45cm,  shorten >=0.45cm, line width=0.055cm, blue] (0,3-1.25) -- (0,3+1.25);

\draw[->, shorten <=0.45cm,  shorten >=0.45cm, line width=0.055cm, blue] (0, -3-1.25) -- (0, -3+1.25);
\draw[->, shorten <=0.45cm,  shorten >=0.45cm, line width=0.055cm, blue] (0, -3-2.5-1.25) -- (0, -3-2.5+1.25);

\node[right] at (+0.25,+3+2.5-0.5) {\LARGE$\beta^{\,(n-1)}_{\sh{F}}$};
\node[right] at (+0.25,+3-0.5) {\LARGE$\beta^{\,(n-2)}_{\sh{F}}$};

\node[right] at (+0.25,-3-0.5) {\LARGE$\beta^{\,(1)}_{\sh{F}}$};
\node[right] at (+0.25,-3-2.5-0.5) {\LARGE$0$};

\filldraw[black] (0,0+0.35) circle (1pt);
\filldraw[black] (0,0) circle (1pt);
\filldraw[black] (0,0-0.35) circle (1pt);

\filldraw[black] (-2.75,0+0.35) circle (1pt);
\filldraw[black] (-2.75,0) circle (1pt);
\filldraw[black] (-2.75,0-0.35) circle (1pt);

\filldraw[black] (+8-0.35,3+2.5+1.25) circle (1pt);
\filldraw[black] (+8,3+2.5+1.25) circle (1pt);
\filldraw[black] (+8+0.35,3+2.5+1.25) circle (1pt);

\node at (-2.75,3+2.5) {\large$n$};
\node at (-2.75,3) {\large$n-1$};
\node at (-2.75,-3) {\large$2$};
\node at (-2.75,-3-2.5) {\large$1$};

\draw[dotted, line width=0.065cm, color=yellow!360, double distance=0.2pt] (-4,-7.5) rectangle (14.5,7.5);

\node at (+7.25,-3-2.5-0.5) {\LARGE$U_{\mathrm{R}}$};

\endscope    
\end{tikzpicture}
\end{center}
\caption{An object $\sh{F}$ of the cohomological category $\ccs{1}{\beta}$ on the region $U_{\mathrm{R}}$.}
\label{Fig: an object F in the region U_R}
\end{figure}
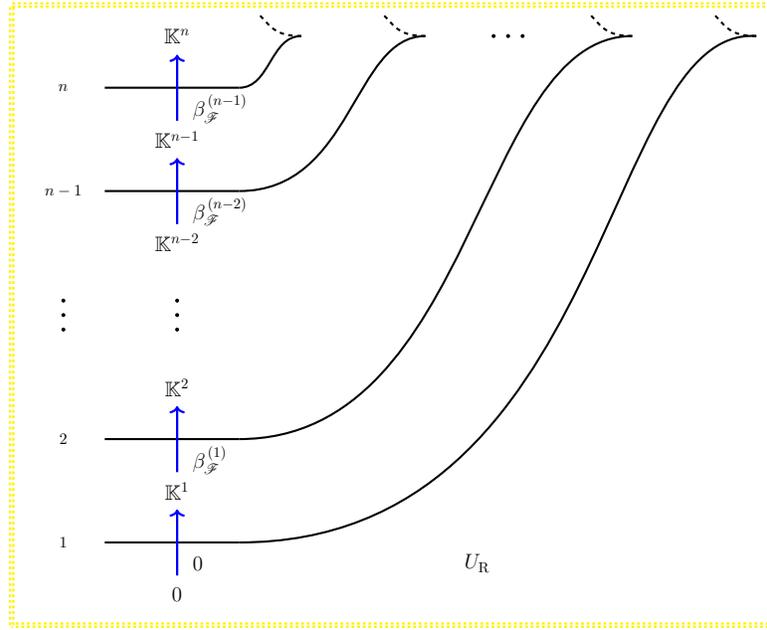

\begin{figure}
\begin{center}
\begin{tikzpicture}
\useasboundingbox (-7.5,-3.25) rectangle (7.5, 3.75);
\scope[transform canvas={scale=0.4}]

\draw[line width=0.06cm] (+1.75,-3-2.5) .. controls (+1.75+1.5+0.25,-3-2.5) and (+1.75+1.5, -3-2.5-1.25) .. (+1.75+2+0.25,-3-2.5-1.25);
\draw[line width=0.06cm] (-1.75,-3-2.5) .. controls (-1.75-1.5-0.25,-3-2.5) and (-1.75-1.5, -3-2.5-1.25) .. (-1.75-2-0.25,-3-2.5-1.25);

\draw[line width=0.06cm] (+1.75+1,-3) .. controls (+1.75+1+2+0.25+0.5,-3) and (+1.75+1+2+1.25,-3-2.5-1.25) .. (+1.75+2+3+0.5,-3-2.5-1.25);
\draw[line width=0.06cm] (-1.75-1,-3) .. controls (-1.75-1-2-0.25-0.5,-3) and (-1.75-1-2-1.25,-3-2.5-1.25) .. (-1.75-2-3-0.5,-3-2.5-1.25);

\draw[line width=0.06cm] (+1.75+3,+3) .. controls (+1.75+1.5+2+2+3+1,+3) and (+1.75+1.5+2+2+1+2,-3-2.5-1.25) .. (+1.75+2+3+5+3,-3-2.5-1.25);
\draw[line width=0.06cm] (-1.75-3,+3) .. controls (-1.75-1.5-2-2-3-1,+3) and (-1.75-1.5-2-2-1-2,-3-2.5-1.25) .. (-1.75-2-3-5-3,-3-2.5-1.25);

\draw[line width=0.06cm] (+1.75+4,+3+2.5) .. controls (+1.75+1.5+2+2+3+2+2,+3+2.5) and (+1.75+1.5+2+2+1+3+1,-3-2.5-1.25) .. (+1.75+2+3+5+3+2+2,-3-2.5-1.25);
\draw[line width=0.06cm] (-1.75-4,+3+2.5) .. controls (-1.75-1.5-2-2-3-2-2,+3+2.5) and (-1.75-1.5-2-2-1-3-1,-3-2.5-1.25) .. (-1.75-2-3-5-3-2-2,-3-2.5-1.25);

\draw[line width=0.06cm, dashed] (+1.75+2+0.25, -3-2.5-1.25) .. controls (+1.75+2+0.25-1+0.35,-3-2.5-1.25) and (+1.75+2+0.25-1+0.25,-3-2.5-1.25-0.25) .. (+1.75+2+0.25-1,-3-2.5-1.25-0.5);
\draw[line width=0.06cm, dashed] (+1.75+2+3+0.5, -3-2.5-1.25) .. controls (+1.75+2+3+0.5-1+0.35,-3-2.5-1.25) and (+1.75+2+3+0.5-1+0.25,-3-2.5-1.25-0.25) .. (+1.75+2+3+0.5-1,-3-2.5-1.25-0.5);
\draw[line width=0.06cm, dashed] (+1.75+2+3+5+3, -3-2.5-1.25) .. controls (+1.75+2+3+5+3-1+0.35,-3-2.5-1.25) and (+1.75+2+3+5+3-1+0.25,-3-2.5-1.25-0.25) .. (+1.75+2+3+5+3-1,-3-2.5-1.25-0.5);
\draw[line width=0.06cm, dashed] (+1.75+2+3+5+3+2+2, -3-2.5-1.25) .. controls (+1.75+2+3+5+3+2+2-1+0.35,-3-2.5-1.25) and (+1.75+2+3+5+3+2+2-1+0.25,-3-2.5-1.25-0.25) .. (+1.75+2+3+5+3+2+2-1,-3-2.5-1.25-0.5);

\draw[line width=0.06cm, dashed] (-1.75-2-0.25, -3-2.5-1.25) .. controls (-1.75-2-0.25+1-0.35,-3-2.5-1.25) and (-1.75-2-0.25+1-0.25,-3-2.5-1.25-0.25) .. (-1.75-2-0.25+1,-3-2.5-1.25-0.5);
\draw[line width=0.06cm, dashed] (-1.75-2-3-0.5, -3-2.5-1.25) .. controls (-1.75-2-3-0.5+1-0.35,-3-2.5-1.25) and (-1.75-2-3-0.5+1-0.25,-3-2.5-1.25-0.25) .. (-1.75-2-3-0.5+1,-3-2.5-1.25-0.5);
\draw[line width=0.06cm, dashed] (-1.75-2-3-5-3, -3-2.5-1.25) .. controls (-1.75-2-3-5-3+1-0.35,-3-2.5-1.25) and (-1.75-2-3-5-3+1-0.25,-3-2.5-1.25-0.25) .. (-1.75-2-3-5-3+1,-3-2.5-1.25-0.5);
\draw[line width=0.06cm, dashed] (-1.75-2-3-5-3-2-2, -3-2.5-1.25) .. controls (-1.75-2-3-5-3-2-2+1-0.35,-3-2.5-1.25) and (-1.75-2-3-5-3-2-2+1-0.25,-3-2.5-1.25-0.25) .. (-1.75-2-3-5-3-2-2+1,-3-2.5-1.25-0.5);

\draw[shorten <=0.5cm,  shorten >=0.5cm, line width=0.06cm] (-2.25-4, 3+2.5) -- (2.25+4, 3+2.5);
\draw[shorten <=0.5cm,  shorten >=0.5cm, line width=0.06cm] (-2.25-2-1, 3) -- (2.25+2+1, 3);
\draw[shorten <=0.5cm,  shorten >=0.5cm, line width=0.06cm] (-2.25-1, -3) -- (2.25+1, -3);
\draw[shorten <=0.5cm,  shorten >=0.5cm, line width=0.06cm] (-2.25, -3-2.5) -- (2.25, -3-2.5);

\node at (0,3+2.5+1.25) {\huge$0$};
\node at (0,3+1.25) {\huge$\mathbb{K}^{1}$};
\node at (0,3-1.25) {\huge$\mathbb{K}^{2}$};

\node at (0,-3+1.25) {\huge$\mathbb{K}^{n-2}$};
\node at (0,-3-1.25) {\huge$\mathbb{K}^{n-1}$};
\node at (0,-3-2.5-1.25) {\huge$\mathbb{K}^{n}$};

\draw[->, shorten <=0.45cm,  shorten >=0.45cm, line width=0.085cm, red] (0,3+2.5-1.25) -- (0,3+2.5+1.25);
\draw[->, shorten <=0.45cm,  shorten >=0.45cm, line width=0.085cm, red] (0,3-1.25) -- (0,3+1.25);
\draw[->, shorten <=0.45cm,  shorten >=0.45cm, line width=0.085cm, red] (0, -3-1.25) -- (0, -3+1.25);
\draw[->, shorten <=0.45cm,  shorten >=0.45cm, line width=0.085cm, red] (0, -3-2.5-1.25) -- (0, -3-2.5+1.25);

\node[right] at (+0.55,+3+2.5-0.6) {\huge$0$};
\node[right] at (+0.55,+3-0.6) {\huge$\psi^{\,(1)}_{\sh{F}}$};

\node[right] at (+0.55,-3-0.6) {\huge$\psi^{\,(n-2)}_{\sh{F}}$};
\node[right] at (+0.55,-3-2.5-0.6) {\huge$\psi^{\,(n-1)}_{\sh{F}}$};

\filldraw[black] (0,0+0.35) circle (1.25pt);
\filldraw[black] (0,0) circle (1.25pt);
\filldraw[black] (0,0-0.35) circle (1.25pt);

\filldraw[black] (-10-0.35,-3-2.5-1.25) circle (1.25pt);
\filldraw[black] (-10,-3-2.5-1.25) circle (1.25pt);
\filldraw[black] (-10+0.35,-3-2.5-1.25) circle (1.25pt);

\filldraw[black] (+10-0.35,-3-2.5-1.25) circle (1.25pt);
\filldraw[black] (+10,-3-2.5-1.25) circle (1.25pt);
\filldraw[black] (+10+0.35,-3-2.5-1.25) circle (1.25pt);

\node at (-1.75,3+2.5+0.5) {\Large$n$};
\node at (-1.75,3+0.5) {\Large$n-1$};
\node at (-1.75,-3+0.5) {\Large$2$};
\node at (-1.75,-3-2.5+0.5) {\Large$1$};

\draw[dotted, line width=0.075cm, color=green!360, double distance=0.2pt] (-19.5,-8.5) rectangle (19.5, 9.5);

\node at (0,8.25) {\huge$U_{\mathrm{T}}$};

\endscope    
\end{tikzpicture}
\end{center}
\caption{An object $\sh{F}$ of the cohomological category $\ccs{1}{\beta}$ on the region $U_{\mathrm{T}}$.}
\label{Fig: an object F in the region U_T}
\end{figure}

Now, note that the intersection $U_{\mathrm{L}}\,\cap\,U_{\mathrm{T}}\,\cap\, \Pi_{x,z}(\Lambda(\beta))$ consists of the $n$ left cusps in $\Pi_{x,z}(\Lambda(\beta))$. Thus, on $U_{\mathrm{L}}\,\cap\,U_{\mathrm{T}}$, $\sh{F}$ is described by the diagram in Figure~\eqref{Fig: an object F in the intersection of U_L and U_T}. Similarly, the intersection $U_{\mathrm{R}}\,\cap\,U_{\mathrm{T}}\,\cap\, \Pi_{x,z}(\Lambda(\beta))$ consists of the $n$ right cusps in $\Pi_{x,z}(\Lambda(\beta))$, and hence, on $U_{\mathrm{R}}\,\cap\,U_{\mathrm{T}}$, $\sh{F}$ is described by the diagram in Figure~\eqref{Fig: an object F in the intersection of U_R and U_T}. Bearing this in mind, the microlocal support conditions near the cusps establish that:
\begin{equation}\label{cusp conditions for lin maps in the general case}
\psi_{\sh{F}}^{(i)}\circ \alpha_{\sh{F}}^{(i)}=\mathrm{id}_{\mathbb{K}^{i}}\, , \quad \text{and} \quad \psi_{\sh{F}}^{(i)}\circ \beta_{\sh{F}}^{(i)}=\mathrm{id}_{\mathbb{K}^{i}}\, ,  
\end{equation}
for each $i\in[1, n-1]$. Consequently, a straightforward application of Lemma~\eqref{Lemma: cusp condition} yields that the maps $\big\{\psi^{(i)}_{\sh{F}}\big\}_{i=1}^{n-1}$ are surjective, whereas the maps $\big\{\alpha^{(i)}_{\sh{F}}\big\}_{i=1}^{n-1}$ and $\big\{\beta^{(i)}_{\sh{F}}\big\}_{i=1}^{n-1}$ are injective. This completes the proof.
\end{proof}

\begin{figure}
\begin{center}
\begin{tikzpicture}
\useasboundingbox (-8.75,-1.75) rectangle (0.75, 2);
\scope[transform canvas={scale=0.4}]

\draw[line width=0.05cm] (-2.5,0) .. controls (-1-0.75,0) and (-1+0.5,0.75) .. (0,1.5);
\draw[line width=0.05cm] (-2.5,0) .. controls (-1-0.75,-0) and (-1+0.5,-0.75) .. (0,-1.5);


\node at (0,0) {\huge$\mathbb{K}^{n}$};
\node at (-4,0) {\huge$\mathbb{K}^{n-1}$};

\draw[->, shorten <=0.85cm,  shorten >=0.85cm, line width=0.075cm, red] (0+0.25,0) -- (-4+0.25, 1.5);
\draw[->, shorten <=0.85cm,  shorten >=0.85cm, line width=0.075cm, blue] (-4+0.25,-1.5) -- (0+0.25,0);

\node at (-1.75+0.35,+1.75) {\huge$\psi^{\,(n-1)}$};
\node at (-1.75+0.35,-1.75) {\huge$\alpha^{\,(n-1)}_{\sh{F}}$};


\draw[line width=0.05cm] (-2.5-4.25,0) .. controls (-1-0.75-4.25,0) and (-1+0.5-4.25,0.75) .. (0-4.25,1.5);
\draw[line width=0.05cm] (-2.5-4.25,0) .. controls (-1-0.75-4.25,-0) and (-1+0.5-4.25,-0.75) .. (0-4.25,-1.5);


\node at (-4-4.25,0) {\huge$\mathbb{K}^{n-2}$};

\draw[->, shorten <=0.85cm,  shorten >=0.85cm, line width=0.075cm, red] (0-4.25+0.25,0) -- (-4-4.25+0.25, 1.5);
\draw[->, shorten <=0.85cm,  shorten >=0.85cm, line width=0.075cm, blue] (-4-4.25+0.25,-1.5) -- (0-4.25+0.25,0);

\node at (-1.75-4.25+0.35,+1.75) {\huge$\psi^{\,(n-2)}$};
\node at (-1.75-4.25+0.35,-1.75) {\huge$\alpha^{\,(n-2)}_{\sh{F}}$};


\filldraw[black] (-10.25-0.35,0+0) circle (1.25pt);
\filldraw[black] (-10.25,0) circle (1.25pt);
\filldraw[black] (-10.25+0.35,0) circle (1.25pt);


\draw[line width=0.05cm] (-2.5-12.25,0) .. controls (-1-0.75-12.25,0) and (-1+0.5-12.25,0.75) .. (0-12.25,1.5);
\draw[line width=0.05cm] (-2.5-12.25,0) .. controls (-1-0.75-12.25,-0) and (-1+0.5-12.25,-0.75) .. (0-12.25,-1.5);


\node at (0-12.25,0) {\huge$\mathbb{K}^{2}$};
\node at (-4-12.25,0) {\huge$\mathbb{K}^{1}$};

\draw[->, shorten <=0.85cm,  shorten >=0.85cm, line width=0.075cm, red] (0-12.25+0.25,0) -- (-4-12.25+0.25, 1.5);
\draw[->, shorten <=0.85cm,  shorten >=0.85cm, line width=0.075cm, blue] (-4-12.25+0.25,-1.5) -- (0-12.25+0.25,0);

\node at (-1.75-12.25+0.35,+1.75) {\huge$\psi^{\,(1)}$};
\node at (-1.75-12.25+0.35,-1.75) {\huge$\alpha^{\,(1)}_{\sh{F}}$};


\draw[line width=0.05cm] (-2.5-12.25-4.25,0) .. controls (-1-0.75-12.25-4.25,0) and (-1+0.5-12.25-4.25,0.75) .. (0-12.25-4.25,1.5);
\draw[line width=0.05cm] (-2.5-12.25-4.25,0) .. controls (-1-0.75-12.25-4.25,-0) and (-1+0.5-12.25-4.25,-0.75) .. (0-12.25-4.25,-1.5);


\node at (-4-12.25-4.25,0) {\huge$0$};

\draw[->, shorten <=0.85cm,  shorten >=0.85cm, line width=0.075cm, red] (0-12.25-4.25+0.25,0) -- (-4-12.25-4.25+0.25, 1.5);
\draw[->, shorten <=0.85cm,  shorten >=0.85cm, line width=0.075cm, blue] (-4-12.25-4.25+0.25,-1.5) -- (0-12.25-4.25+0.25,0);

\node at (-1.75-12.25-4.25+0.35,+1.75) {\huge$0$};
\node at (-1.75-12.25-4.25+0.35,-1.75) {\huge$0$};


\draw[dotted, line width=0.075cm, color=green!360, double distance=0.2pt] (-22,-4) rectangle (2, 5);
\draw[dotted, line width=0.075cm, color=yellow!360, double distance=0.2pt] (-22+0.25,-4+0.25) rectangle (2-0.25, 5-0.25);

\node at (-10.25,3.75) {\huge$U_{\mathrm{L}}\,\cap\, U_{\mathrm{T}}$};

\endscope    
\end{tikzpicture}
\end{center}
\caption{An object $\sh{F}$ of the cohomological category $\ccs{1}{\beta}$ on the intersection $U_{\mathrm{L}}\,\cap\,U_{\mathrm{T}}$.}
\label{Fig: an object F in the intersection of U_L and U_T}
\end{figure}

\begin{figure}
\begin{center}
\begin{tikzpicture}
\useasboundingbox (-0.75,-1.75) rectangle (+8.75, 2);
\scope[transform canvas={scale=0.4}]

\draw[line width=0.05cm] (+2.5,0) .. controls (+1+0.75,0) and (+1-0.5,0.75) .. (0,1.5);
\draw[line width=0.05cm] (+2.5,0) .. controls (+1+0.75,-0) and (+1-0.5,-0.75) .. (0,-1.5);


\node at (0,0) {\huge$\mathbb{K}^{n}$};
\node at (+4,0) {\huge$\mathbb{K}^{n-1}$};

\draw[->, shorten <=0.85cm,  shorten >=0.85cm, line width=0.075cm, red] (0-0.25,0) -- (+4-0.25, 1.5);
\draw[->, shorten <=0.85cm,  shorten >=0.85cm, line width=0.075cm, blue] (+4-0.25,-1.5) -- (0-0.25,0);

\node at (+1.75-0.35,+1.75) {\huge$\psi^{\,(n-1)}$};
\node at (+1.75-0.35,-1.75) {\huge$\beta^{\,(n-1)}_{\sh{F}}$};


\draw[line width=0.05cm] (+2.5+4.25,0) .. controls (+1+0.75+4.25,0) and (+1-0.5+4.25,0.75) .. (0+4.25,1.5);
\draw[line width=0.05cm] (+2.5+4.25,0) .. controls (+1+0.75+4.25,-0) and (+1-0.5+4.25,-0.75) .. (0+4.25,-1.5);


\node at (+4+4.25,0) {\huge$\mathbb{K}^{n-2}$};

\draw[->, shorten <=0.85cm,  shorten >=0.85cm, line width=0.075cm, red] (0+4.25-0.25,0) -- (+4+4.25-0.25, 1.5);
\draw[->, shorten <=0.85cm,  shorten >=0.85cm, line width=0.075cm, blue] (+4+4.25-0.25,-1.5) -- (0+4.25-0.25,0);

\node at (+1.75+4.25-0.35,+1.75) {\huge$\psi^{\,(n-2)}$};
\node at (+1.75+4.25-0.35,-1.75) {\huge$\beta^{\,(n-2)}_{\sh{F}}$};


\filldraw[black] (+10.25+0.35,0+0) circle (1.25pt);
\filldraw[black] (+10.25,0) circle (1.25pt);
\filldraw[black] (+10.25-0.35,0) circle (1.25pt);


\draw[line width=0.05cm] (+2.5+12.25,0) .. controls (+1+0.75+12.25,0) and (+1-0.5+12.25,0.75) .. (0+12.25,1.5);
\draw[line width=0.05cm] (+2.5+12.25,0) .. controls (+1+0.75+12.25,-0) and (+1-0.5+12.25,-0.75) .. (0+12.25,-1.5);


\node at (0+12.25,0) {\huge$\mathbb{K}^{2}$};
\node at (+4+12.25,0) {\huge$\mathbb{K}^{1}$};

\draw[->, shorten <=0.85cm,  shorten >=0.85cm, line width=0.075cm, red] (0+12.25-0.25,0) -- (+4+12.25-0.25, 1.5);
\draw[->, shorten <=0.85cm,  shorten >=0.85cm, line width=0.075cm, blue] (+4+12.25-0.25,-1.5) -- (0+12.25-0.25,0);

\node at (+1.75+12.25-0.35,+1.75) {\huge$\psi^{\,(1)}$};
\node at (+1.75+12.25-0.35,-1.75) {\huge$\beta^{\,(1)}_{\sh{F}}$};


\draw[line width=0.05cm] (+2.5+12.25+4.25,0) .. controls (+1+0.75+12.25+4.25,0) and (+1-0.5+12.25+4.25,0.75) .. (0+12.25+4.25,1.5);
\draw[line width=0.05cm] (+2.5+12.25+4.25,0) .. controls (+1+0.75+12.25+4.25,-0) and (+1-0.5+12.25+4.25,-0.75) .. (0+12.25+4.25,-1.5);


\node at (+4+12.25+4.25,0) {\huge$0$};

\draw[->, shorten <=0.85cm,  shorten >=0.85cm, line width=0.075cm, red] (0+12.25+4.25-0.25,0) -- (+4+12.25+4.25-0.25, 1.5);
\draw[->, shorten <=0.85cm,  shorten >=0.85cm, line width=0.075cm, blue] (+4+12.25+4.25-0.25,-1.5) -- (0+12.25+4.25-0.25,0);

\node at (+1.75+12.25+4.25-0.35,+1.75) {\huge$0$};
\node at (+1.75+12.25+4.25-0.35,-1.75) {\huge$0$};


\draw[dotted, line width=0.075cm, color=green!360, double distance=0.2pt] (-2,-4) rectangle (22, 5);
\draw[dotted, line width=0.075cm, color=yellow!360, double distance=0.2pt] (-2+0.25,-4+0.25) rectangle (22-0.25, 5-0.25);

\node at (+10.25,3.75) {\huge$U_{\mathrm{R}}\,\cap\, U_{\mathrm{T}}$};

\endscope    
\end{tikzpicture}
\end{center}
\caption{An object $\sh{F}$ of the cohomological category $\ccs{1}{\beta}$ on the intersection $U_{\mathrm{R}}\,\cap\,U_{\mathrm{T}}$.}
\label{Fig: an object F in the intersection of U_R and U_T}
\end{figure}

Next, we address how the description of $\sh{F}$ on the regions $U_{\mathrm{T}}$, $U_{\mathrm{L}}$, and $U_{\mathrm{R}}$ in terms of linear maps, as specified in Lemma~\eqref{Lemma: linear map description of an object on the regions U_T, U_L, and U_R}, naturally leads to a geometric characterization via complete flags in $\mathbb{K}^{n}$. 

\begin{lemma}\label{lemma for F in the region U_{R}}
Let $\beta:=\sigma_{i_{1}}\cdots \sigma_{i_{\ell}}\in\mathrm{Br}_{n}^{+}$ be a positive braid word, $\sh{F}$ an object of the category $\ccs{1}{\beta}$, and $\mathcal{U}_{\Lambda(\beta)}=\big\{U_{0}, U_{\mathrm{B}}, U_{\mathrm{L}}, U_{\mathrm{R}}, U_{\mathrm{T}} \big\}$ the open cover of $\mathbb{R}^{2}$ introduced in Construction~\eqref{Cons: Finite open cover for R^2}. Then, the following statements hold: 
\begin{itemize}
\justifying
\item On $U_{0}$, $\sh{F}$ is identically zero.
\item On $U_{\mathrm{T}}$, $\sh{F}$ is characterized by a complete flag $\fl{F}_{0}$ in $\mathbb{K}^{n}$.
\item On $U_{\mathrm{L}}$, $\sh{F}$ is characterized by a complete flag $\fl{F}_{\mathrm{left}}$ in $\mathbb{K}^{n}$. 
\item On $U_{\mathrm{R}}$, $\sh{F}$ is characterized by a complete flag $\fl{F}_{\mathrm{right}}$ in $\mathbb{K}^{n}$.
\item \textbf{Compatibility conditions}: $\fl{F}_{0}$ is completely opposite to both $\fl{F}_{\mathrm{left}}$ and $\fl{F}_{\mathrm{right}}$.  
\end{itemize}
\end{lemma}
\begin{proof}
By Lemma~\eqref{Lemma: linear map description of an object on the regions U_T, U_L, and U_R}, we know that: 
\begin{itemize}
\justifying
\item On $U_{0}$, $\sh{F}$ is identically zero.
\item On $U_{\mathrm{T}}$, $\sh{F}$ is defined by a collection of $n-1$ surjective linear maps $\big\{\psi^{(i)}_{\sh{F}}:\mathbb{K}^{i+1}\to \mathbb{K}^{i}\big\}_{i=1}^{n-1}$.
\item On $U_{\mathrm{L}}$, $\sh{F}$ is defined by a collection of $n-1$ injective linear maps $\big\{\alpha^{(i)}_{\sh{F}}:\mathbb{K}^{i}\to \mathbb{K}^{i+1}\big\}_{i=1}^{n-1}$. 
\item On $U_{\mathrm{R}}$, $\sh{F}$ is defined by a collection of $n-1$ injective linear maps $\big\{\beta^{(i)}_{\sh{F}}:\mathbb{K}^{i}\to \mathbb{K}^{i+1}\big\}_{i=1}^{n-1}$. 
\item \textbf{Compatibility conditions}: For each $i\in[1,n-1]$, 
\begin{equation*}
\psi^{(i)}_{\sh{F}}\circ \alpha^{(i)}_{\sh{F}}=\mathrm{id}_{\mathbb{K}^{i}}\, , \quad \text{and} \quad \psi^{(i)}_{\sh{F}}\circ \beta^{(i)}_{\sh{F}}=\mathrm{id}_{\mathbb{K}^{i}}\, .    
\end{equation*}
\end{itemize}
To prove the claim, it suffices to show that the above data determines three complete flags $\fl{F}_{0}$, $\fl{F}_{\mathrm{left}}$, and $\fl{F}_{\mathrm{right}}$ in $\mathbb{K}^{n}$, with $\fl{F}_{0}$ completely opposite to both $\fl{F}_{\mathrm{left}}$ and $\fl{F}_{\mathrm{right}}$, thereby providing a geometric description of $\sh{F}$. To this end, relying on the notions of type $\mathcal{I}$ and type $\mathcal{K}$ flags introduced in Definition~\eqref{Def:flags and adapted bases}--\eqref{Def: type I flag}--\eqref{Def: type K flag}, we proceed as follows:

\begin{enumerate}
\justifying
\item Since the maps $\big\{\psi^{(i)}_{\sh{F}}\big\}_{i=1}^{n-1}$ are surjective, we assign to $\sh{F}$ on $U_{\mathrm{T}}$  the complete flag $\fl{F}_{0}$ in $ \mathbb{K}^{n}$ given by
\begin{equation*}
\fl{F}_{0}:=\prescript{}{\mathcal{K}\,}{\fl{F}}\big(\,\psi^{(1)}_{\sh{F}},\dots,\psi^{(n-1)}_{\sh{F}}\, \big)\, .    
\end{equation*}

\item Since the maps $\big\{\alpha^{(i)}_{\sh{F}}\big\}_{i=1}^{n-1}$ are injective, we assign to $\sh{F}$ on $U_{\mathrm{L}}$ the complete flag $\fl{F}_{\mathrm{left}}$ in $ \mathbb{K}^{n}$ given by
\begin{equation*}
\fl{F}_{\mathrm{left}}:=\prescript{}{\mathcal{I}\,}{\fl{F}}\big(\,\alpha^{(1)}_{\sh{F}},\dots,\alpha^{(n-1)}_{\sh{F}}\,\big)\, .    
\end{equation*}

\item Since the maps $\big\{\beta^{(i)}_{\sh{F}}\big\}_{i=1}^{n-1}$ are injective, we assign to $\sh{F}$ on $U_{\mathrm{R}}$ the complete flag $\fl{F}_{\mathrm{right}}$ in $\mathbb{K}^{n}$ given by
\begin{equation*}
\fl{F}_{\mathrm{right}}:=\prescript{}{\mathcal{I}\,}{\fl{F}}\big(\,\beta^{(1)}_{\sh{F}},\dots,\beta^{(n-1)}_{\sh{F}}\,\big)\, .    
\end{equation*}
\end{enumerate}

Finally, the compatibility conditions on the maps $\big\{\psi^{(i)}_{\sh{F}}\big\}_{i=1}^{n-1}$, $\big\{\alpha^{(i)}_{\sh{F}}\big\}_{i=1}^{n-1}$, and $\big\{\beta^{(i)}_{\sh{F}}\big\}_{i=1}^{n-1}$ allow a direct application of Lemma~\eqref{Lemma: completely opposite flags from linear maps}, which shows that $\fl{F}_{0}$ is completely opposite to both $\fl{F}_{\mathrm{left}}$ and $\fl{F}_{\mathrm{right}}$, as desired. 
\end{proof}

We now turn to the analysis of $\sh{F}$ on the region $U_{\mathrm{B}}$, the last remaining element of the open cover $\mathcal{U}_{\Lambda(\beta)}$ of $\mathbb{R}^{2}$ introduced in Construction \eqref{Cons: Finite open cover for R^2}. As previously mentioned, $U_{\mathrm{B}}$ contains all the crossings in the front projection $\Pi_{x,z}(\Lambda(\beta))$, and consequently, the behavior of $\sh{F}$ on this region is inherently more intricate. To facilitate the exposition, we construct a finite open cover of $U_{\mathrm{B}}$ adapted to the front projection $\Pi_{x,z}(\Lambda(\beta))$, allowing us to study $\sh{F}$ on this region in terms of tractable local data and the corresponding gluing conditions. 

\begin{construction}\label{Cons: Definition of the vertical straps}
Let $\beta:=\sigma_{i_{1}}\cdots \sigma_{i_{\ell}}\in \mathrm{Br}^{+}_{n}$ be a positive braid word, and let ~$\mathcal{U}_{\Lambda(\beta)}=\big\{U_{0}, U_{\mathrm{B}}, U_{\mathrm{L}}, U_{\mathrm{R}}, U_{\mathrm{T}}\big\}$ be the open cover of $\mathbb{R}^{2}$ introduced in Construction~\eqref{Cons: Finite open cover for R^2}. Then, we decompose the region $U_{\mathrm{B}}\subset \mathbb{R}^{2}$ into a collection $\mathcal{R}_{\Lambda(\beta)}:=\big\{R_{j}\big\}_{j=1}^{\ell}$ of $\ell$ open vertical straps $R_{j}\subset\mathbb{R}^{2}$ as illustrated in Figure~\eqref{fig: Decomposition of the region U_B into sub-regions R_i}. In particular, given $\Pi_{x,z}(\Lambda(\beta))\subset \mathbb{R}^{2}$ the front projection of the Legendrian link $\Lambda(\beta)\subset (\mathbb{R}^{3},\xi_{\mathrm{std}})$, we assume that:
\begin{itemize}
\justifying
\item For each $j\in[1,\ell]$, the intersection $R_{j}\,\cap\,\Pi_{x,z}(\Lambda(\beta))$ consists of the braid on $n$ strands associated with $\sigma_{i_{j}}$---the $j$-th crossing of $\beta$---as shown in Figure~\eqref{fig: A sub-regions R_j}. 

\item For each $j\in[1,\ell-1]$, the intersection $R_{j}\,\cap\, R_{j+1}\,\cap\, \Pi_{x,z}(\Lambda(\beta))$ consists of the trivial braid on $n$ strands as depicted in Figure~\eqref{fig: Intersection of two consecutive sub-regions R_j and R_{j+1}}. 
\end{itemize} 

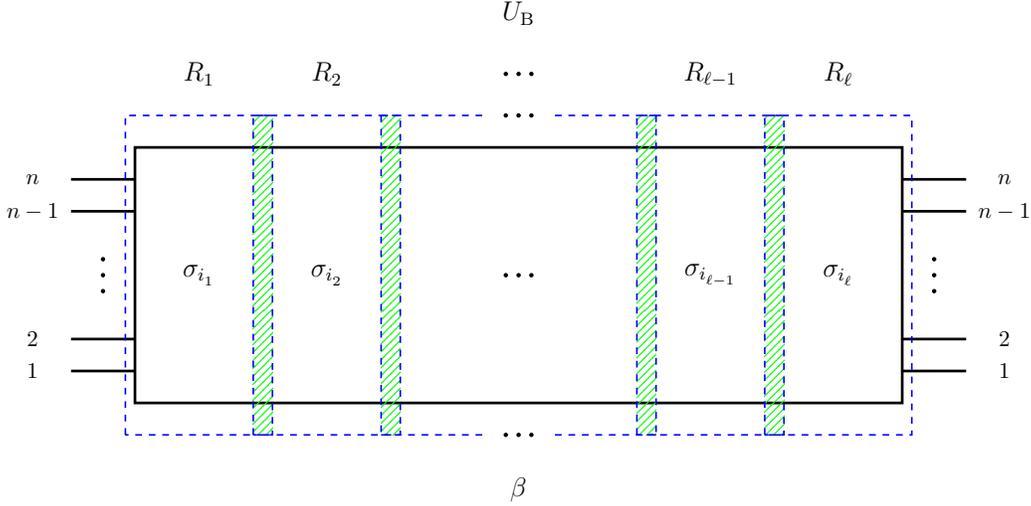
\begin{figure}
\centering
\begin{tikzpicture}
\useasboundingbox (-7,-2.85) rectangle (7,3.75);
\scope[transform canvas={scale=0.85}]

\draw[very thick] (-6,-2) rectangle (6,2);

\draw[very thick] (-7, 2-0.5) -- (-6, 2-0.5);
\draw[very thick] (-7, 1) -- (-6, 1);
\draw[very thick] (-7, -1) -- (-6,-1);
\draw[very thick] (-7, -2+0.5) -- (-6, -2+0.5);

\filldraw[black] (-6-0.5,0+0.25) circle (0.75pt);
\filldraw[black] (-6-0.5,0) circle (0.75pt);
\filldraw[black] (-6-0.5,0-0.25) circle (0.75pt);

\draw[very thick] (6, 2-0.5) -- (7, 2-0.5);
\draw[very thick] (6, 1) -- (7, 1);
\draw[very thick] (6, -1) -- (7,-1);
\draw[very thick] (6, -2+0.5) -- (7, -2+0.5);

\filldraw[black] (6+0.5,0+0.25) circle (0.75pt);
\filldraw[black] (6+0.5,0) circle (0.75pt);
\filldraw[black] (6+0.5,0-0.25) circle (0.75pt);

\node at (-7-0.6,2-0.5) {$n$};
\node at (-7-0.6,1) {$n-1$};
\node at (-7-0.6,-1) {$2$};
\node at (-7-0.6,-2+0.5) {$1$};

\node at (7+0.6,2-0.5) {$n$};
\node at (7+0.6,1) {$n-1$};
\node at (7+0.6,-1) {$2$};
\node at (7+0.6,-2+0.5) {$1$};


\node at (0,+4.10) {\Large$U_{\mathrm{B}}$};
\node at (0,-3.35) {\Large$\beta$};

\node at (-6+1,0) {\Large$\sigma_{i_{1}}$};
\node at (6-1,0) {\Large$\sigma_{i_{\ell}}$};

\node at (-6+3,0) {\Large$\sigma_{i_{2}}$};
\node at (6-3,0) {\Large$\sigma_{i_{\ell-1}}$};


\draw[draw=none, dashed, pattern=north east lines, pattern color=green] (-4-0.15,-2-0.5) rectangle (-4+0.15,+2+0.5);
\draw[draw=none, dashed, pattern=north east lines, pattern color=green] (+4-0.15,-2-0.5) rectangle (+4+0.15,+2+0.5);

\draw[draw=none, dashed, pattern=north east lines, pattern color=green] (-2-0.15,-2-0.5) rectangle (-2+0.15,+2+0.5);
\draw[draw=none, dashed, pattern=north east lines, pattern color=green] (+2-0.15,-2-0.5) rectangle (+2+0.15,+2+0.5);

\draw[dashed, blue, line width=0.025cm] (-6-0.15,-2-0.5) rectangle (-4+0.15,+2+0.5);
\draw[dashed, blue, line width=0.025cm] (-4-0.15,-2-0.5) rectangle (-2+0.15,+2+0.5);

\draw[dashed, blue, line width=0.025cm] (2-0.15,-2-0.5) rectangle (4+0.15,+2+0.5);
\draw[dashed, blue, line width=0.025cm] (4-0.15,-2-0.5) rectangle (6+0.15,+2+0.5);

\draw[dashed, blue, line width=0.025cm] (-2-0.15,2+0.5) -- (-2-0.15,-2-0.5);
\draw[dashed, blue, line width=0.025cm] (-2-0.15,-2-0.5) -- (0-0.5,-2-0.5);
\draw[dashed, blue, line width=0.025cm] (-2-0.15,+2+0.5) -- (0-0.5,+2+0.5);

\draw[dashed, blue, line width=0.025cm] (+2+0.15,2+0.5) -- (+2+0.15,-2-0.5);
\draw[dashed, blue, line width=0.025cm] (+2+0.15,-2-0.5) -- (0+0.5,-2-0.5);
\draw[dashed, blue, line width=0.025cm] (+2+0.15,+2+0.5) -- (0+0.5,+2+0.5);


\filldraw[black] (0-0.2,-2-0.5) circle (0.75pt);
\filldraw[black] (0,-2-0.5) circle (0.75pt);
\filldraw[black] (0+0.2,-2-0.5) circle (0.75pt);

\filldraw[black] (0-0.2,2+0.5) circle (0.75pt);
\filldraw[black] (0,2+0.5) circle (0.75pt);
\filldraw[black] (0+0.2,2+0.5) circle (0.75pt);

\filldraw[black] (0-0.2,0) circle (0.75pt);
\filldraw[black] (0,0) circle (0.75pt);
\filldraw[black] (0+0.2,0) circle (0.75pt);

\filldraw[black] (0-0.2,3.15) circle (0.75pt);
\filldraw[black] (0,3.15) circle (0.75pt);
\filldraw[black] (0+0.2,3.15) circle (0.75pt);


\node at (-6+1,3.15) {\Large$R_{1}$};
\node at (6-1,3.15) {\Large$R_{\ell}$};
\node at (-6+3,3.15) {\Large$R_{2}$};
\node at (6-3,3.15) {\Large$R_{\ell-1}$};

\endscope
\end{tikzpicture}
\caption{Region $U_{\mathrm{B}}$ in $\mathbb{R}^{2}$ decomposed into a collection $\mathcal{R}_{\Lambda(\beta)}:=\big\{R_{i}\big\}_{i=1}^{\ell}$ of $\ell$ open vertical straps $R_{j}$ in $ \mathbb{R}^{2}$.}
\label{fig: Decomposition of the region U_B into sub-regions R_i}
\end{figure}

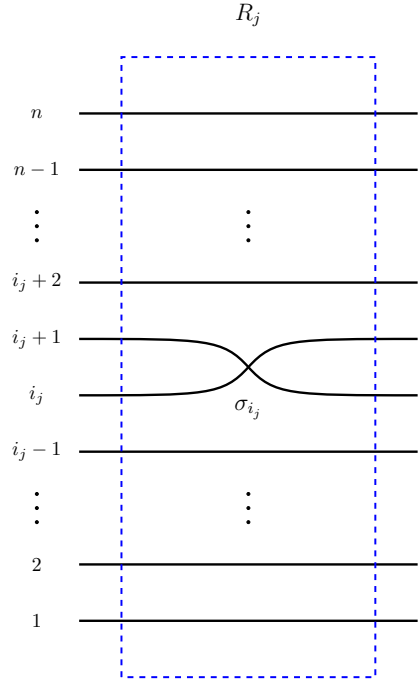
\begin{figure}
\centering
\begin{tikzpicture}
\useasboundingbox (-4,-4.25) rectangle (4,5);
\scope[transform canvas={scale=0.75}]

\draw[very thick] (-3,1+0.5) -- (3,1+0.5);
\draw[very thick] (-3,0.5) .. controls (-0.75,0.5) and (-0.45,0.5) .. (0,0) .. controls (0.45,-0.5) and (0.75,-0.5) .. (3,-0.5);
\draw[very thick] (-3,-0.5) .. controls (-0.75,-0.5) and (-0.45,-0.5) .. (0,0) .. controls (0.45,0.5) and (0.75,0.5) .. (3,0.5);
\draw[very thick] (-3,-1-0.5) -- (3,-1-0.5);

\filldraw[black] (0,2.5+0.25) circle (0.75pt);
\filldraw[black] (0,2.5) circle (0.75pt);
\filldraw[black] (0,2.5-0.25) circle (0.75pt);

\filldraw[black] (0,-2.5+0.25) circle (0.75pt);
\filldraw[black] (0,-2.5) circle (0.75pt);
\filldraw[black] (0,-2.5-0.25) circle (0.75pt);

\draw[very thick] (-3,3.5) -- (3,3.5);
\draw[very thick] (-3,-3.5) -- (3,-3.5);

\draw[very thick] (-3,4.5) -- (3,4.5);
\draw[very thick] (-3,-4.5) -- (3,-4.5);

\node at (-3-0.75, 4.5) {$n$};
\node at (-3-0.75, 3.5) {$n-1$};
\node at (-3-0.75, 1+0.5) {$i_{j}+2$};
\node at (-3-0.75, +0.5) {$i_{j}+1$};
\node at (-3-0.75, -0.5) {$i_{j}$};
\node at (-3-0.75, -1-0.5) {$i_{j}-1$};
\node at (-3-0.75, -3.5) {$2$};
\node at (-3-0.75, -4.5) {$1$};

\filldraw[black] (-3-0.75,2.5+0.25) circle (0.75pt);
\filldraw[black] (-3-0.75,2.5) circle (0.75pt);
\filldraw[black] (-3-0.75,2.5-0.25) circle (0.75pt);

\filldraw[black] (-3-0.75,-2.5+0.25) circle (0.75pt);
\filldraw[black] (-3-0.75,-2.5) circle (0.75pt);
\filldraw[black] (-3-0.75,-2.5-0.25) circle (0.75pt);

\draw[dashed, blue, line width=0.035cm] (-2.25,-5.5) rectangle (2.25,5.5);

\node at (0,-0.75) {\Large$\sigma_{i_{j}}$};
\node at (0,6.25) {\Large$R_{j}$};

\endscope
\end{tikzpicture}
\caption{Intersection of an open vertical strap $R_{j}$ in $\mathbb{R}^{2}$ with the front projection $\Pi_{x,z}(\Lambda(\beta))\subset \mathbb{R}^{2}$.}
\label{fig: A sub-regions R_j}
\end{figure}

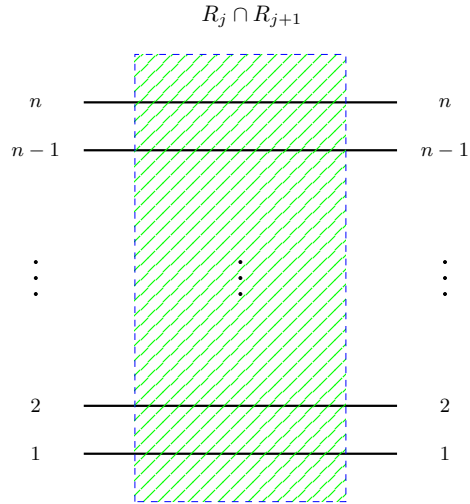
\begin{figure}
\centering
\begin{tikzpicture}
\useasboundingbox (-4,-3.15) rectangle (4,4);
\scope[transform canvas={scale=0.85}]

\draw[line width=0.035cm]  (-2.45,2.5+0.25) -- (2.45,2.5+0.25);
\draw[line width=0.035cm]  (-2.45, 1.75+0.25) -- (2.45, 1.75+0.25);

\draw[line width=0.035cm]  (-2.45, -1.75-0.25) -- (2.45, -1.75-0.25);
\draw[line width=0.035cm] (-2.45,-2.5-0.25) -- (2.45,-2.5-0.25);

\filldraw[black] (0,0+0.25) circle (0.65pt);
\filldraw[black] (0,0) circle (0.65pt);
\filldraw[black] (0,0-0.25) circle (0.65pt);

\filldraw[black] (-3.20,0+0.25) circle (0.65pt);
\filldraw[black] (-3.20,0) circle (0.65pt);
\filldraw[black] (-3.20,0-0.25) circle (0.65pt);

\node at (-3.20,2.5+0.25) {\small $n$};
\node at (-3.20,1.75+0.25) {\small $n-1$};
\node at (-3.20,-1.75-0.25) {\small $2$};
\node at (-3.20,-2.5-0.25) {\small $1$};

\filldraw[black] (+3.20,0+0.25) circle (0.65pt);
\filldraw[black] (+3.20,0) circle (0.65pt);
\filldraw[black] (+3.20,0-0.25) circle (0.65pt);

\node at (+3.20,2.5+0.25) {\small $n$};
\node at (+3.20,1.75+0.25) {\small $n-1$};
\node at (+3.20,-1.75-0.25) {\small $2$};
\node at (+3.20,-2.5-0.25) {\small $1$};

\draw[blue, dashed, pattern={Lines[angle=45,distance=4pt]}, pattern color=green, line width=0.015cm] (-1.65,-3.25-0.25) rectangle (1.65,3.25+0.25);

\node at (0.18, 3.85+0.25) {\normalsize $R_{j}\cap R_{j+1}$};

\endscope
\end{tikzpicture}
\caption{Intersection of the overlap $R_{j}\cap R_{j+1}$ of two adjacent straps $R_{j}$ and $R_{j+1}$ in $\mathbb{R}^{2}$ with the front projection $\Pi_{x,z}(\Lambda(\beta))\subset \mathbb{R}^{2}$.}
\label{fig: Intersection of two consecutive sub-regions R_j and R_{j+1}}
\end{figure}
\end{construction}

Next, by applying the sheaf axioms to the open cover $\mathcal{R}_{\Lambda(\beta)}=\big\{R_{j}\big\}_{j=1}^{\ell}$ of $U_{\mathrm{B}}$ introduced in Construction~\eqref{Cons: Definition of the vertical straps}, we provide an explicit characterization of $\sh{F}$ on the region $U_{\mathrm{B}}$ as follows: First, we give a geometric description of $\sh{F}$ on each vertical strap $R_{j}$ independently, for every $j\in[1,\ell]$. Second, we proceed by gluing the local data that characterizes $\sh{F}$ across adjacent straps $R_{j}$ and $R_{j+1}$, for every $j\in[1,\ell-1]$. Accordingly, we present the following results. 

\begin{lemma}\label{lemma: description of an object F in a region R_j in terms of flags for k=1}
Let $\beta=\sigma_{i_{1}}\cdots \sigma_{i_{\ell}}\in \mathrm{Br}^{+}_{n}$ be a positive braid word, and let $\sh{F}$ be an object of the category $\ccs{1}{\beta}$. Fix $j\in[1,\ell]$, and let $R_{j}\subset \mathbb{R}^{2}$ be the open vertical strap whose intersection with the front projection $\Pi_{x,z}(\Lambda(\beta))\subset \mathbb{R}^{2}$ consists of the braid diagram on $n$ strands associated with $\sigma_{i_{j}}$, i.e., the $j$-th crossing of $\beta$ (see Figure~\eqref{fig: A sub-regions R_j}). Denote by $k:=i_{j}\in[1, n-1]$ the index of $\sigma_{i_{j}}$, and suppose that $k=1$. Then, on the region $R_{j}$, the sheaf $\sh{F}$ is specified by a collection of $n$ linear maps 
\begin{equation*}
\big\{\phi^{(i)}_{\sh{F}}:\mathbb{K}^{i}\to \mathbb{K}^{i+1}\big\}_{i=1}^{n-1}\,\cup\, \big\{\widetilde{\phi}^{\,(1)}_{\sh{F}}:\mathbb{K}^{1}\to \mathbb{K}^{2} \big\}\, , 
\end{equation*}
as illustrated in Figure~\eqref{fig: a sheaf in the vertical strap R_j containing a crossing sigma_1}. 

\textbf{Compatibility conditions}:  The maps $\phi^{(1)}_{\sh{F}}$ and $\widetilde{\phi}^{\,(1)}_{\sh{F}}$ are injective and have complementary image in $\mathbb{K}^{2}$; that is, $\mathbb{K}^{2}=\mathrm{im}\,\phi^{\,(1)}_{\sh{F}}\,\oplus\,\mathrm{im}\,\widetilde{\phi}^{\,(1)}_{\sh{F}}$. 

In particular, if we assume that the maps $\big\{\phi^{(i)}_{\sh{F}}\big\}_{i=2}^{n-1}$ are injective, we further obtain that, on $R_{j}$, $\sh{F}$ is characterized by two complete flags $\fl{F}_{j}$ and $\fl{F}_{j+1}$ in $\mathbb{K}^{n}$, which are in $s_{1}$--relative position. 
\end{lemma}
\begin{proof}
To begin, consider the front projection $\Pi_{x,z}(\Lambda(\beta))\subset \mathbb{R}^{2}$ of the Legendrian link $\Lambda(\beta)\subset (\mathbb{R}^{3}, \xi_{\mathrm{std}})$, which is depicted in Figure~\eqref{Front diagram of the rainbow closure of a braid in n strands}. In particular, recall that in $\Pi_{x,z}(\Lambda(\beta))$, the strands at the bottom have Maslov potential $0$. Thus, since $k=1$, the microlocal rank conditions and the microlocal support conditions near the arcs imply that, on $R_{j}$, $\sh{F}$ is defined by a collection of $n$ linear maps 
\begin{equation*}
\big\{\phi^{(i)}_{\sh{F}}:\mathbb{K}^{i}\to \mathbb{K}^{i+1}\big\}_{i=1}^{n-1}\,\cup\, \big\{\widetilde{\phi}^{\,(1)}_{\sh{F}}:\mathbb{K}^{1}\to \mathbb{K}^{2} \big\}\, , 
\end{equation*}
as shown in Figure~\eqref{fig: a sheaf in the vertical strap R_j containing a crossing sigma_1}. 

Now, let $U_{j}\subset \mathbb{R}^{2}$ be the small open ball---centered at the crossing $\sigma_{i_{j}}$---highlighted in yellow in Figure~\eqref{fig: a sheaf in the vertical strap R_j containing a crossing sigma_1}, where $k=i_{j}=1$. Then, by the microlocal support conditions near the crossings, we obtain that the maps $\phi^{(1)}_{\sh{F}}$ and $\widetilde{\phi}^{\,(1)}_{\sh{F}}$ defining $\sh{F}$ on $U_{j}$ satisfy the following compatibility conditions:
\begin{itemize}
\justifying
\item[(\textit{i})] The diagram in Figure~\eqref{fig: an object F on an open ball at a crossing sigma 1} commutes.
\item[(\textit{ii})]  The sequence in Equation~\eqref{Eq: short exact sequence for the crossing sigma_1} is short exact.
\begin{equation}\label{Eq: short exact sequence for the crossing sigma_1}
\begin{tikzpicture}
\useasboundingbox (-4.5,-0.25) rectangle (4.5,1);
\scope[transform canvas={scale=0.95}]
\node at (-5.5-1.5,0) {\large$0$}; 
\node at (-3-1,0) {\large$0$};    
\node at (0,0) {\large$\mathbb{K}^{1}\oplus \mathbb{K}^{1}$};   
\node at (3+1,0) {\large$\mathbb{K}^{2}$};
\node at (5.5+1.5,0) {\large$0$}; 

\draw[->, shorten <=0.85cm,  shorten >=0.35cm, line width=0.02cm] (-8+0.5, 0) -- (-4,0);
\draw[->, shorten <=0.35cm,  shorten >=0.85cm, line width=0.02cm] (-4, 0) -- (0,0);
\draw[->, shorten <=0.85cm,  shorten >=0.35cm, line width=0.02cm] (0, 0) -- (4,0);
\draw[->, shorten <=0.35cm,  shorten >=0.85cm, line width=0.02cm] (4, 0) -- (8-0.5,0);

\node at (-1.5-0.75, 0.55) {\small$(0,0)$};
\node at (1.5+0.75, 0.55) {\small$\phi^{\,(1)}_{\sh{F}} \oplus\Big(-\widetilde{\phi}^{\,(1)}_{\sh{F}}\,\Big)$};

\node at (5.8+2-0.5,-0.10) {\large$.$}; 

\endscope
\end{tikzpicture}    
\end{equation}  
\end{itemize}
Consequently, since $0\leq \mathrm{dim}_{\,\mathbb{K}}\,\mathrm{im}\, \phi^{(1)}_{\sh{F}}\leq 1$ and $ 0\leq \mathrm{dim}_{\,\mathbb{K}}\,\mathrm{im}\,\widetilde{\phi}^{\,(1)}_{\sh{F}}\leq 1$, the short exact sequence in Equation~\eqref{Eq: short exact sequence for the crossing sigma_1} guarantees that the maps $\phi^{(1)}_{\sh{F}}$ and $\widetilde{\phi}^{\,(1)}_{\sh{F}}$ are injective and their images are complementary subspaces in $\mathbb{K}^{2}$; that is, $\mathbb{K}^{2}=\mathrm{im}\,\phi^{\,(1)}_{\sh{F}}\,\oplus\,\mathrm{im}\,\widetilde{\phi}^{\,(1)}_{\sh{F}}$. 

Finally, assuming the maps $\big\{\phi^{(i)}_{\sh{F}}\big\}_{i=2}^{n-1}$ are injective, we show that the linear maps defining $\sh{F}$ on $R_{j}$ determine two complete flags $\fl{F}_{j}$ and $\fl{F}_{j+1}$ in $\mathbb{K}^{n}$, which are in $s_{1}$--relative position, thereby providing a geometric characterization of $\sh{F}$ on $R_{j}$. To this end, relying on the notion of type $\mathcal{I}$ flag introduced in Definition~\eqref{Def:flags and adapted bases}--\eqref{Def: type I flag}, we proceed as follows: 
\begin{enumerate}
\item Since the maps $\big\{\phi^{(i)}_{\sh{F}}\big\}_{i=1}^{n-1}$ are injective, we assign to $\sh{F}$ on $R_{j}$ the complete flag $\fl{F}_{j}$ in $\mathbb{K}^{n}$ given by
\begin{equation*}
\fl{F}_{j}:=\prescript{}{\mathcal{I}\,}{\fl{F}}\big(\phi_{\sh{F}}^{(1)},  \dots, \phi_{\sh{F}}^{(n-1)}\big)\, .    
\end{equation*}

\item Since the maps $\big\{\widetilde{\phi}^{\,(1)}_{\sh{F}}, \phi^{(2)}_{\sh{F}}, \dots, \phi^{(n-1)}_{\sh{F}} \big\}$ are injective, we assign to $\sh{F}$ on $R_{j}$ the complete flag $\fl{F}_{j+1}$ in $\mathbb{K}^{n}$ given by
\begin{equation*}
\fl{F}_{j+1}:=\prescript{}{\mathcal{I}\,}{\fl{F}}\big( \widetilde{\phi}^{\,(1)}_{\sh{F}}, \phi^{(2)}_{\sh{F}}, \dots, \phi^{(n-1)}_{\sh{F}} \big)\, .   
\end{equation*}
\end{enumerate}

In particular, by Definition~\eqref{Def:flags and adapted bases}--\eqref{Def: type I flag}, we know that, as filtrations of vector subspaces in $\mathbb{K}^{n}$, the flags 
\begin{equation*}
\begin{aligned}
\fl{F}_{j}&=\big\{ F^{(0)}_{j}\subset \cdots \subset F^{(n)}_{j} \big\}\, , \\[4pt]   
\fl{F}_{j+1}&=\big\{ F^{(0)}_{j+1}\subset \cdots \subset F^{(n)}_{j+1} \big\}\, ,       
\end{aligned}
\end{equation*}
are given by
\begin{align*}
F^{(p)}_{j}&:=\begin{cases}
\hfil 0\, , & \text{if $~p=0$}\, , \\[2pt]
\hfil\mathrm{im}\Big(\phi^{\,(n-1)}_{\sh{F}}\circ\phi^{\,(n-2)}_{\sh{F}}\circ\cdots\circ\phi^{\,(2)}_{\sh{F}}\circ\phi^{\,(1)}_{\sh{F}}\Big)\, , & \text{if $~p=1$}\, , \\[2pt]
\mathrm{im}\Big(\phi^{\,(n-1)}_{\sh{F}}\circ\phi^{\,(n-2)}_{\sh{F}}\circ\cdots\circ\phi^{\,(p+1)}_{\sh{F}}\circ\phi^{\,(p)}_{\sh{F}}\Big)\, , & \text{if $~p\in[2, n-1]$}\, , \\[2pt]
\hfil \mathbb{K}^{n}\, , & \text{if $~p=n$}\, .
\end{cases}\\[6pt]
F^{(p)}_{j+1}&:=\begin{cases}
\hfil 0\, , & \text{if $~p=0$}\, , \\[2pt]
\hfil\mathrm{im}\Big(\phi^{\,(n-1)}_{\sh{F}}\circ\phi^{\,(n-2)}_{\sh{F}}\circ\cdots\circ\phi^{\,(2)}_{\sh{F}}\circ\widetilde{\phi}^{\,(1)}_{\sh{F}}\Big)\, , & \text{if $~p=1$}\, , \\[2pt]
\mathrm{im}\Big(\phi^{\,(n-1)}_{\sh{F}}\circ\phi^{\,(n-2)}_{\sh{F}}\circ\cdots\circ\phi^{\,(p+1)}_{\sh{F}}\circ\phi^{\,(p)}_{\sh{F}}\Big)\, , & \text{if $~p\in[2, n-1$]}\, , \\[2pt]
\hfil \mathbb{K}^{n}\, , & \text{if $~p=n$}\, .
\end{cases} 
\end{align*}
Thus, since $\mathrm{im}\,\phi^{\,(1)}_{\sh{F}}\,\neq\,\mathrm{im}\,\widetilde{\phi}^{\,(1)}_{\sh{F}}$ and the composition $\phi^{\,(n-1)}_{\sh{F}}\circ \cdots \circ \phi^{\,(2)}_{\sh{F}}:\mathbb{K}^{2}\longrightarrow \mathbb{K}^{n}$ is injective, we deduce that $F^{(1)}_{j}\neq F^{(1)}_{j+1}$ and $F^{(p)}_{j}= F^{(p)}_{j+1}$ for all $p\in[0,n]$ with $p\neq 1$, confirming that $\fl{F}_{j}$ and $\fl{F}_{j+1}$ are in $s_{1}$--relative position. This concludes the proof. 
\end{proof}

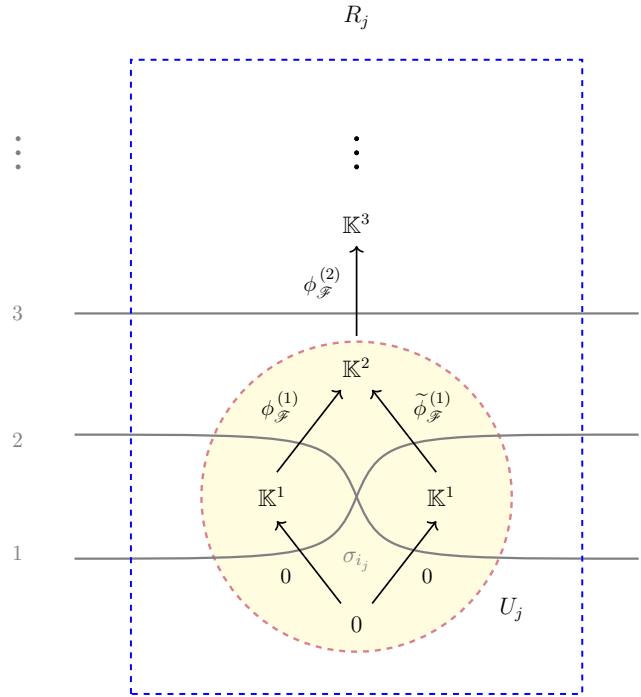
\begin{figure}
\centering
\begin{tikzpicture}
\useasboundingbox (-5.75,-2.75) rectangle (5.75,6.75);
\scope[transform canvas={scale=0.75}]

\begin{scope}[opacity=0.5, transparency group]

\filldraw[color=purple!120, fill=yellow!35, fill opacity=0.9, very thick, dashed] (0,0) circle (2.75); 

\draw[shorten <=0.5cm,  shorten >=0.5cm, line width=0.04cm] (-4.5-1, 2.35+1-0.1) -- (4.5+1,2.35+1-0.1);

\draw[shorten <=0.5cm,  shorten >=0.5cm, line width=0.04cm] (-4.5-1,1.1) .. controls (-0.75,1.1) and (-0.45,1.1) .. (0,0) .. controls (0.45,-1.1) and (0.75,-1.1) .. (4.5+1,-1.1);
\draw[shorten <=0.5cm,  shorten >=0.5cm, line width=0.04cm] (-4.5-1,-1.1) .. controls (-0.75,-1.1) and (-0.45,-1.1) .. (0,0) .. controls (0.45,1.1) and (0.75,1.1) .. (4.5+1,1.1);


\node at (-4-1-1, 2.35+1-0.1) {\large$3$};
\node at (-4-1-1, 1) {\large$2$};
\node at (-4-1-1, -1) {\large$1$};

\filldraw[black] (-5-1,4.35+1+0.25+1-0.25) circle (1pt);
\filldraw[black] (-5-1,4.35+1+1-0.25) circle (1pt);
\filldraw[black] (-5-1,4.35+1-0.25+1-0.25) circle (1pt);


\node at (0, -1.15) {\Large$\sigma_{i_{j}}$};

\end{scope}


\node at (0, 4+0.35+0.5) {\Large$\mathbb{K}^{3}$};

\node[above] at (0, 2) {\Large$\mathbb{K}^{2}$};
\node at (-1.5, 0) {\Large$\mathbb{K}^{1}$};
\node at (1.5, 0) {\Large$\mathbb{K}^{1}$};
\node[below] at (0, -2) {\Large$0$};


\draw[->, shorten <=0.6cm,  shorten >=0.3cm, line width=0.03cm] (0, 2+0.25) -- (0,4+0.25+0.5);

\draw[->, shorten <=0.55cm,  shorten >=0.45cm, line width=0.03cm] (-1.75, 0) -- (0,2.25);
\draw[->, shorten <=0.55cm,  shorten >=0.45cm, line width=0.03cm] (1.75, 0) -- (0,2.25);
\draw[->, shorten <=0.45cm,  shorten >=0.55cm, line width=0.03cm] (0,-2.25) -- (-1.75, 0);
\draw[->, shorten <=0.45cm,  shorten >=0.55cm, line width=0.03cm] (0,-2.25) -- (1.75, 0);


\filldraw[black] (0,4.35+1+0.25+1-0.25) circle (1pt);
\filldraw[black] (0,4.35+1+1-0.25) circle (1pt);
\filldraw[black] (0,4.35+1-0.25+1-0.25) circle (1pt);


\node at (-1.35, 1.6) {\large$\phi^{\,(1)}_{\sh{F}}$};
\node at (1.35, 1.6) {\large$\widetilde{\phi}^{\,(1)}_{\sh{F}}$};
\node at (-1.25, -1.4) {\large$0$};
\node at (1.25, -1.4) {\large$0$};

\node at (-0.6, 3.25+0.55) {\large$\phi^{\,(2)}_{\sh{F}}$};


\draw[dashed, blue, line width=0.035cm] (-3-1,-3.5) rectangle (3+1,7.75);

\node at (0,8.5) {\Large$R_{j}$};

\node at (2.75,-2) {\Large$U_{j}$};

\endscope
\end{tikzpicture}
\caption{An object $\sh{F}$ on the vertical strap $R_{j}$ containing a crossing $\sigma_{i_{j}}$ in the case $k=i_{j}=1$.}
\label{fig: a sheaf in the vertical strap R_j containing a crossing sigma_1}
\end{figure}

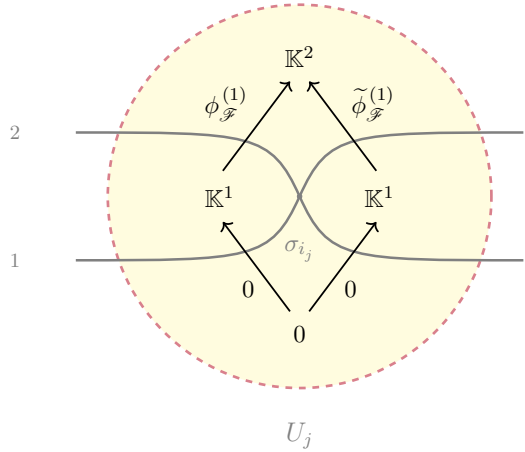
\begin{figure}
\centering
\begin{tikzpicture}
\useasboundingbox (-4,-3.25) rectangle (4,3);
\scope[transform canvas={scale=0.85}]
\begin{scope}[opacity=0.5, transparency group]
\filldraw[color=purple!120, fill=yellow!35, fill opacity=0.9, very thick, dashed] (0,0) circle (3);  
\filldraw[black] (0,0) circle (1pt);
\draw[shorten <=0.5cm,  shorten >=0.5cm, line width=0.04cm] (-4,1) .. controls (-0.75,1) and (-0.45,1) .. (0,0) .. controls (0.45,-1) and (0.75,-1) .. (4,-1);
\draw[shorten <=0.5cm,  shorten >=0.5cm, line width=0.04cm] (-4,-1) .. controls (-0.75,-1) and (-0.45,-1) .. (0,0) .. controls (0.45,1) and (0.75,1) .. (4,1);
\node[left] at (-4.25,1) {\small $2$};
\node[left] at (-4.25,-1) {\small $1$};
\node at (0,-0.85) {\large $\sigma_{i_{j}}$};
\node at (0,-3.75) {\Large $U_{j}$};    
\end{scope}

\node[above] at (0, 1.90) {\large$\mathbb{K}^{2}$};
\node[below] at (0, -1.90) {\large$0$};
\node at (-1.25, 0) {\large$\mathbb{K}^{1}$};
\node at (1.25, 0) {\large$\mathbb{K}^{1}$};

\draw[->, shorten <=0.5cm,  shorten >=0.25cm, thick] (-1.5, 0) -- (0,2);
\draw[->, shorten <=0.5cm,  shorten >=0.25cm, thick] (1.5, 0) -- (0,2);
\draw[->, shorten <=0.25cm,  shorten >=0.5cm, thick] (0,-2) -- (-1.5, 0);
\draw[->, shorten <=0.25cm,  shorten >=0.5cm, thick] (0,-2) -- (1.5, 0);

\node at (-1.15, 1.45) {\large$\phi^{\,(1)}_{\sh{F}}$};
\node at (1.15, 1.45) {\large$\widetilde{\phi}^{\,(1)}_{\sh{F}}$};
\node at (-0.8, -1.45) {\large$0$};
\node at (0.8, -1.45) {\large$0$};

\endscope
\end{tikzpicture}
\caption{Commutative diagram for the linear maps defining a sheaf $\sh{F}$ on the region $U_{j}$ in the case $k=i_{j}=1$.}
\label{fig: an object F on an open ball at a crossing sigma 1}
\end{figure}

\begin{lemma}\label{lemma: description of an object F in a region R_j in terms of flags for k geq 2}
Let $\beta=\sigma_{i_{1}}\cdots \sigma_{i_{\ell}}\in \mathrm{Br}^{+}_{n}$ be a positive braid word, and let $\sh{F}$ be an object of the category $\ccs{1}{\beta}$. Fix $j\in[1,\ell]$, and let $R_{j}\subset \mathbb{R}^{2}$ be the open vertical strap whose intersection with the front projection $\Pi_{x,z}(\Lambda(\beta))\subset \mathbb{R}^{2}$ consists of the braid diagram on $n$ strands associated with $\sigma_{i_{j}}$, i.e., the $j$-th crossing of $\beta$ (see Figure~\eqref{fig: A sub-regions R_j}). Denote by $k:=i_{j}\in[1, n-1]$ the index of $\sigma_{i_{j}}$, and suppose that $k\geq 2$. Then, on the region $R_{j}$, the sheaf $\sh{F}$ is specified by a collection of $n+1$ linear maps 
\begin{equation*}
\Big\{\phi^{\,(i)}_{\sh{F}}:\mathbb{K}^{i}\to \mathbb{K}^{i+1}\big\}_{i=1}^{n-1}\,\cup\,\big\{ \widetilde{\phi}^{\,(k-1)}_{\sh{F}}:\mathbb{K}^{k-1}\to \mathbb{K}^{k}\,,~\widetilde{\phi}^{\,(k)}_{\sh{F}}:\mathbb{K}^{k}\to \mathbb{K}^{k+1}\,\Big\}\, ,   
\end{equation*}
as illustrated in Figure~\eqref{fig: a sheaf in the vertical strap R_j containing a crossing sigma_k}. 

\textbf{Compatibility conditions}: The maps $\phi^{\,(k)}_{\sh{F}}$, $\phi^{\,(k-1)}_{\sh{F}}$, $\widetilde{\phi}^{\,(k)}_{\sh{F}}$, and $\widetilde{\phi}^{\,(k-1)}_{\sh{F}}$ satisfy the following properties: 
\begin{itemize}
\justifying
\item $\phi^{\,(k)}_{\sh{F}}\circ \phi^{\,(k-1)}_{\sh{F}}=\widetilde{\phi}^{\,(k)}_{\sh{F}}\circ\widetilde{\phi}^{\,(k-1)}_{\sh{F}}$. 
\item $\mathrm{im}\Big(\phi^{\,(k)}_{\sh{F}}\circ \phi^{\,(k-1)}_{\sh{F}}\Big)=\mathrm{im}\,\phi^{\,(k)}_{\sh{F}} \cap\,\mathrm{im}\,\widetilde{\phi}^{\,(k)}_{\sh{F}} =\mathrm{im}\Big(\widetilde{\phi}^{\,(k)}_{\sh{F}}\circ \widetilde{\phi}^{\,(k-1)}_{\sh{F}}\Big)$.
\item $\mathbb{K}^{k+1}=\mathrm{im}\,\phi^{\,(k)}_{\sh{F}}\,+\,\mathrm{im}\,\widetilde{\phi}^{\,(k)}_{\sh{F}}$. 
\end{itemize}

In particular, if we assume that the maps $\big\{\phi^{(i)}_{\sh{F}}\big\}_{i=1}^{n-1}$ are injective, we further obtain that:
\begin{itemize}
\justifying
\item The maps $\widetilde{\phi}^{\,(k-1)}_{\sh{F}}$ and $\widetilde{\phi}^{\,(k)}_{\sh{F}}$ are injective.
\item On $R_{j}$, $\sh{F}$ is characterized by two complete flags $\fl{F}_{j}$ and $\fl{F}_{j+1}$ in $\mathbb{K}^{n}$, which are in $s_{k}$--relative position. 
\end{itemize}
\end{lemma}
\begin{proof}
To begin, consider the front projection $\Pi_{x,z}(\Lambda(\beta))\subset \mathbb{R}^{2}$ of the Legendrian link $\Lambda(\beta)\subset (\mathbb{R}^{3}, \xi_{\mathrm{std}})$, which is depicted in Figure~\eqref{Front diagram of the rainbow closure of a braid in n strands}. In particular, recall that in $\Pi_{x,z}(\Lambda(\beta))$, the strands at the bottom have Maslov potential $0$. Thus, since $k\geq 2$, the microlocal rank conditions and the microlocal support conditions near the arcs establish that, on $R_{j}$, $\sh{F}$ is defined by a collection of $n+1$ linear maps     
\begin{equation*}
\big\{\phi^{\,(i)}_{\sh{F}}:\mathbb{K}^{i}\to \mathbb{K}^{i+1}\big\}_{i=1}^{n-1}\,\cup\,\big\{ \widetilde{\phi}^{\,(k-1)}_{\sh{F}}:\mathbb{K}^{k-1}\to \mathbb{K}^{k}\,,~\widetilde{\phi}^{\,(k)}_{\sh{F}}:\mathbb{K}^{k}\to \mathbb{K}^{k+1}\,\big\}\, ,   
\end{equation*}
as shown in Figure~\eqref{fig: a sheaf in the vertical strap R_j containing a crossing sigma_k}. 

Now, let $U_{j}\subset \mathbb{R}^{2}$ be the small open ball---centered at the crossing $\sigma_{i_{j}}$---highlighted in yellow in Figure~\eqref{fig: a sheaf in the vertical strap R_j containing a crossing sigma_k}. Then, by the microlocal support conditions near the crossings, we obtain that the maps $\phi^{\,(k)}_{\sh{F}}$, $\phi^{\,(k-1)}_{\sh{F}}$, $\widetilde{\phi}^{\,(k)}_{\sh{F}}$, and $\widetilde{\phi}^{\,(k-1)}_{\sh{F}}$ defining $\sh{F}$ on $U_{j}$ satisfy the following compatibility conditions:
\begin{itemize}
\justifying
\item[(\textit{i})] The diagram in Figure~\eqref{fig: an object F on an open ball at a crossing sigma k} commutes.
\item[(\textit{ii})] The sequence in Equation~\eqref{Eq: short exact sequence for the crossing sigma_k} is short exact.
\begin{equation}\label{Eq: short exact sequence for the crossing sigma_k}
\begin{tikzpicture}
\useasboundingbox (-4.5,-0.25) rectangle (4.5,1.15);
\scope[transform canvas={scale=0.95}]
\node at (-5.5-2,0.05) {\large$0$}; 
\node at (-3-1-0.5,0.05) {\large$\mathbb{K}^{k-1}$};    
\node at (0,0.05) {\large$\mathbb{K}^{k}\oplus \mathbb{K}^{k}$};   
\node at (3+1+0.5,0.05) {\large$\mathbb{K}^{k+1}$};
\node at (5.5+2,0.05) {\large$0$}; 

\draw[->, shorten <=0.85cm,  shorten >=0.6cm, line width=0.02cm] (-8, 0) -- (-4-0.5,0);
\draw[->, shorten <=0.6cm,  shorten >=1.0cm, line width=0.02cm] (-4-0.5, 0) -- (0,0);
\draw[->, shorten <=1.0cm,  shorten >=0.6cm, line width=0.02cm] (0, 0) -- (4+0.5,0);
\draw[->, shorten <=0.6cm,  shorten >=0.85cm, line width=0.02cm] (4+0.5, 0) -- (8,0);

\node at (-1.5-0.75-0.25, 0.7) {\normalsize$\big(\,\phi^{\,(k-1)}_{\sh{F}},\,\widetilde{\phi}^{\,(k-1)}_{\sh{F}}\,\big)$};
\node at (1.5+0.75+0.25, 0.7) {\normalsize$\phi^{\,(k)}_{\sh{F}} \oplus\big(-\widetilde{\phi}^{\,(k)}_{\sh{F}}\,\big)$};

\node at (5.8+2,-0.025) {\large$.$}; 

\endscope
\end{tikzpicture}    
\end{equation}  
\end{itemize}
Consequently, a straightforward application of Lemma~\eqref{Lemma: Crossing condition for linear maps} ensures that: 
\begin{itemize}
\item \textit{(i)} $\phi^{\,(k)}_{\sh{F}}\circ \phi^{\,(k-1)}_{\sh{F}}=\widetilde{\phi}^{\,(k)}_{\sh{F}}\circ\widetilde{\phi}^{\,(k-1)}_{\sh{F}}$.
\item \textit{(ii)} $\mathrm{im}\big(\phi^{\,(k)}_{\sh{F}}\circ \phi^{\,(k-1)}_{\sh{F}}\big)=\mathrm{im}\,\phi^{\,(k)}_{\sh{F}} \cap\,\mathrm{im}\,\widetilde{\phi}^{\,(k)}_{\sh{F}} =\mathrm{im}\big(\widetilde{\phi}^{\,(k)}_{\sh{F}}\circ \widetilde{\phi}^{\,(k-1)}_{\sh{F}}\big)$.
\item \textit{(iii)} $\mathbb{K}^{k+1}=\mathrm{im}\,\phi^{\,(k)}_{\sh{F}}\,+\,\mathrm{im}\,\widetilde{\phi}^{\,(k)}_{\sh{F}}$.    
\end{itemize}

Finally, assume that the maps $\big\{\phi^{(i)}_{\sh{F}}\big\}_{i=1}^{n-1}$ are injective. Then, the compatibility conditions for the maps $\phi^{\,(k)}_{\sh{F}}$, $\phi^{\,(k-1)}_{\sh{F}}$, $\widetilde{\phi}^{\,(k)}_{\sh{F}}$, and $\widetilde{\phi}^{\,(k-1)}_{\sh{F}}$ further imply, via Lemma~\eqref{Lemma: Crossing condition for linear maps}, that $\widetilde{\phi}^{\,(k-1)}_{\sh{F}}$ and $\widetilde{\phi}^{\,(k)}_{\sh{F}}$ are also injective. 

Bearing this in mind, we now show that, under the above injectivity assumptions, the linear maps defining $\sh{F}$ on $R_{j}$ determine two complete flags $\fl{F}_{j}$ and $\fl{F}_{j+1}$ in $\mathbb{K}^{n}$, which are in $s_{k}$--relative position, thereby providing a geometric characterization of $\sh{F}$ on $R_{j}$. To this end, relying on the notion of type $\mathcal{I}$ flag introduced in Definition~\eqref{Def:flags and adapted bases}--\eqref{Def: type I flag}, we proceed as follows: 
\begin{enumerate}
\justifying
\item Since the maps $\big\{\phi^{(i)}_{\sh{F}}\big\}_{i=1}^{n-1}$ are injective, we assign to $\sh{F}$ on $R_{j}$ the complete flag $\fl{F}_{j}$ in $\mathbb{K}^{n}$ given by
\begin{equation*}
\fl{F}_{j}:=\prescript{}{\mathcal{I}\,}{\fl{F}}\big( \phi^{(1)}_{\sh{F}}, \dots, \phi^{(n-1)}_{\sh{F}} \big)\,.    
\end{equation*}

\item Since the maps $\big\{\phi^{(1)}_{\sh{F}}, \dots, \phi^{(k-2)}_{\sh{F}}, \widetilde{\phi}^{\,(k-1)}_{\sh{F}}, \widetilde{\phi}^{\,(k)}_{\sh{F}}, \phi^{(k+1)}_{\sh{F}},\dots, \phi^{(n-1)}_{\sh{F}} \big\}$ are injective, we assign to $\sh{F}$ on $R_{j}$ the complete flag $\fl{F}_{j+1}$ in $\mathbb{K}^{n}$ given by
\begin{equation*}
\fl{F}_{j+1}:=\prescript{}{\mathcal{I}\,}{\fl{F}}\big( \phi^{(1)}_{\sh{F}}, \dots, \phi^{(k-2)}_{\sh{F}}, \widetilde{\phi}^{\,(k-1)}_{\sh{F}}, \widetilde{\phi}^{\,(k)}_{\sh{F}}, \phi^{(k+1)}_{\sh{F}},\dots, \phi^{(n-1)}_{\sh{F}} \big)\,.    
\end{equation*}
\end{enumerate}

In particular, by Definition~\eqref{Def:flags and adapted bases}--\eqref{Def: type I flag}, we know that, as filtrations of vector subspaces in $\mathbb{K}^{n}$, the flags
\begin{equation*}
\begin{aligned}
\fl{F}_{j}&=\big\{ F^{(0)}_{j}\subset \cdots \subset F^{(n)}_{j}\big\}\, ,\\[6pt]
\fl{F}_{j+1}&=\big\{ F^{(0)}_{j+1}\subset \cdots \subset F^{(n)}_{j+1}\big\}\, ,
\end{aligned}
\end{equation*}
are given by 
\begin{equation*}
\begin{aligned}
F^{(p)}_{j}&:=\begin{cases}
\hfil \{0\}\,, & \text{if $p=0$}\,, \\[3pt]
\mathrm{im}\left(\, \phi^{\,(n-1)}_{\sh{F}}\circ \cdots \circ \phi^{\,(k+1)}_{\sh{F}}\circ\phi^{\,(k)}_{\sh{F}}\circ\phi^{\,(k-1)}_{\sh{F}}\circ\phi^{\,(k-2)}_{\sh{F}}\circ\cdots\circ\phi^{\,(p)}_{\sh{F}} \,\right)\, , & \text{if $p\in[1,n-1]$}\,, \\[3pt]
\hfil \mathbb{K}^{n}\,, & \text{if $p=n$}\, .
\end{cases}\\[4pt]    
F^{(p)}_{j+1}&:=\begin{cases}
\hfil \{0\}\,, & \text{if $p=0$}\,, \\[3pt]    
\mathrm{im}\left(\, \phi^{\,(n-1)}_{\sh{F}}\circ \cdots \circ \phi^{\,(k+1)}_{\sh{F}}\circ\widetilde{\phi}^{\,(k)}_{\sh{F}}\circ\widetilde{\phi}^{\,(k-1)}_{\sh{F}}\circ\phi^{\,(k-2)}_{\sh{F}}\circ\cdots\circ\phi^{\,(p)}_{\sh{F}} \,\right)\, , & \text{if $p\in[1, n-1]$}\,, \\[3pt]    
\hfil \mathbb{K}^{n}\,, & \text{if $p=n$}\, .
\end{cases}       
\end{aligned}    
\end{equation*}
Thus, since $\phi^{\,(k)}_{\sh{F}}\circ \phi^{\,(k-1)}_{\sh{F}}=\widetilde{\phi}^{\,(k)}_{\sh{F}}\circ\widetilde{\phi}^{\,(k-1)}_{\sh{F}}$, we have that  $F^{(p)}_{j}=F^{(p)}_{j+1}$ for all $p\in[0,n]$ with $p\neq k$. Moreover, since $\mathrm{im}\,\phi^{\,(k)}_{\sh{F}}\,\neq\,\mathrm{im}\,\widetilde{\phi}^{\,(k)}_{\sh{F}}$ and the composition map $\phi^{\,(n-1)}_{\sh{F}}\circ \cdots \circ\phi^{\,(k+1)}_{\sh{F}}:\mathbb{K}^{k+1}\longrightarrow \mathbb{K}^{n}$ is injective, we deduce that $F^{(k)}_{j}\neq F^{(k)}_{j+1}$, confirming that $\fl{F}_{j+1}$ and $\fl{F}_{j}$ are in $s_{k}$-relative position. This concludes the proof. 
\end{proof}

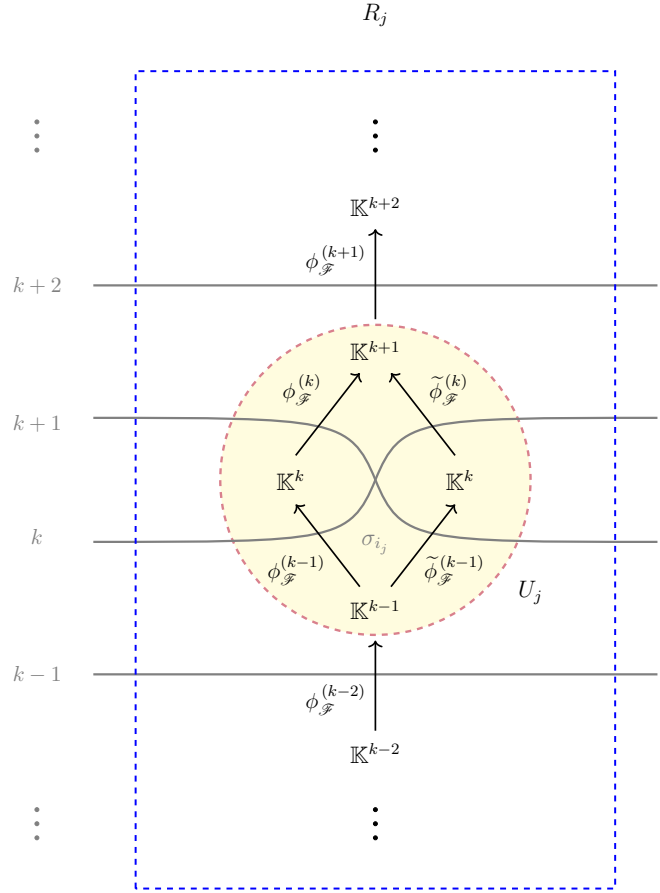
\begin{figure}
\centering
\begin{tikzpicture}
\useasboundingbox (-6,-5.5) rectangle (6,6.5);
\scope[transform canvas={scale=0.75}]

\begin{scope}[opacity=0.5, transparency group]

\filldraw[color=purple!120, fill=yellow!35, fill opacity=0.9, very thick, dashed] (0,0) circle (2.75); 

\draw[shorten <=0.5cm,  shorten >=0.5cm, line width=0.04cm] (-4.5-1, 2.35+1+0.1) -- (4.5+1,2.35+1+0.1);

\draw[shorten <=0.5cm,  shorten >=0.5cm, line width=0.04cm] (-4.5-1,1.1) .. controls (-0.75,1.1) and (-0.45,1.1) .. (0,0) .. controls (0.45,-1.1) and (0.75,-1.1) .. (4.5+1,-1.1);
\draw[shorten <=0.5cm,  shorten >=0.5cm, line width=0.04cm] (-4.5-1,-1.1) .. controls (-0.75,-1.1) and (-0.45,-1.1) .. (0,0) .. controls (0.45,1.1) and (0.75,1.1) .. (4.5+1,1.1);

\draw[shorten <=0.5cm,  shorten >=0.5cm, line width=0.04cm] (-4.5-1, -2.35-1-0.1) -- (4.5+1,-2.35-1-0.1);

\node at (-4-1-1, 2.35+1+0.1) {\large$k+2$};
\node at (-4-1-1, 1) {\large$k+1$};
\node at (-4-1-1, -1) {\large$k$};
\node at (-4-1-1, -2.35-1-0.1) {\large$k-1$};

\filldraw[black] (-5-1,4.35+1+0.25+1-0.25) circle (1pt);
\filldraw[black] (-5-1,4.35+1+1-0.25) circle (1pt);
\filldraw[black] (-5-1,4.35+1-0.25+1-0.25) circle (1pt);

\filldraw[black] (-5-1,-4.35-1+0.25-1+0.25) circle (1pt);
\filldraw[black] (-5-1,-4.35-1-1+0.25) circle (1pt);
\filldraw[black] (-5-1,-4.35-1-0.25-1+0.25) circle (1pt);

\node at (0, -1.15) {\Large$\sigma_{i_{j}}$};

\end{scope}


\node at (0, 4+0.35+0.5) {\Large$\mathbb{K}^{k+2}$};

\node[above] at (0, 2) {\Large$\mathbb{K}^{k+1}$};
\node at (-1.5, 0) {\Large$\mathbb{K}^{k}$};
\node at (1.5, 0) {\Large$\mathbb{K}^{k}$};
\node[below] at (0, -2) {\Large$\mathbb{K}^{k-1}$};

\node at (0, -4-0.35-0.5) {\Large$\mathbb{K}^{k-2}$};

\draw[->, shorten <=0.6cm,  shorten >=0.3cm, line width=0.03cm] (0, 2+0.25) -- (0,4+0.25+0.5);

\draw[->, shorten <=0.55cm,  shorten >=0.45cm, line width=0.03cm] (-1.75, 0) -- (0,2.25);
\draw[->, shorten <=0.55cm,  shorten >=0.45cm, line width=0.03cm] (1.75, 0) -- (0,2.25);
\draw[->, shorten <=0.45cm,  shorten >=0.55cm, line width=0.03cm] (0,-2.25) -- (-1.75, 0);
\draw[->, shorten <=0.45cm,  shorten >=0.55cm, line width=0.03cm] (0,-2.25) -- (1.75, 0);

\draw[->, shorten <=0.3cm,  shorten >=0.6cm, line width=0.03cm] (0, -4-0.25-0.5) -- (0,-2-0.25);

\filldraw[black] (0,4.35+1+0.25+1-0.25) circle (1pt);
\filldraw[black] (0,4.35+1+1-0.25) circle (1pt);
\filldraw[black] (0,4.35+1-0.25+1-0.25) circle (1pt);

\filldraw[black] (0,-4.35-1+0.25-1+0.25) circle (1pt);
\filldraw[black] (0,-4.35-1-1+0.25) circle (1pt);
\filldraw[black] (0,-4.35-1-0.25-1+0.25) circle (1pt);

\node at (-1.3, 1.6) {\large$\phi^{\,(k)}_{\sh{F}}$};
\node at (1.3, 1.6) {\large$\widetilde{\phi}^{\,(k)}_{\sh{F}}$};
\node at (-1.4, -1.6) {\large$\phi^{(k-1)}_{\sh{F}}$};
\node at (1.4, -1.6) {\large$\widetilde{\phi}^{\,(k-1)}_{\sh{F}}$};

\node at (-0.7, 3.25+0.65) {\large$\phi^{\,(k+1)}_{\sh{F}}$};
\node at (-0.7, -3.25-0.65) {\large$\phi^{\,(k-2)}_{\sh{F}}$};



\draw[dashed, blue, line width=0.035cm] (-4.25,-7.25) rectangle (4.25,7.25);

\node at (0,8.25) {\Large$R_{j}$};

\node at (2.75,-2) {\Large$U_{j}$};

\endscope
\end{tikzpicture}
\caption{An object $\sh{F}$ on the vertical strap $R_{j}$ containing the crossing $\sigma_{i_{j}}$ in the case $k=i_{j}\geq 2$.}
\label{fig: a sheaf in the vertical strap R_j containing a crossing sigma_k}
\end{figure}

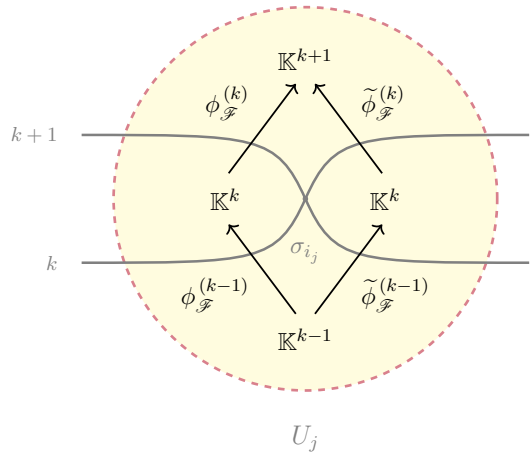
\begin{figure}
\centering
\begin{tikzpicture}
\useasboundingbox (-4,-3.25) rectangle (4,3.25);
\scope[transform canvas={scale=0.85}]
\begin{scope}[opacity=0.5, transparency group]
\filldraw[color=purple!120, fill=yellow!35, fill opacity=0.9, very thick, dashed] (0,0) circle (3);  
\filldraw[black] (0,0) circle (1pt);
\draw[shorten <=0.5cm,  shorten >=0.5cm, line width=0.04cm] (-4,1) .. controls (-0.75,1) and (-0.45,1) .. (0,0) .. controls (0.45,-1) and (0.75,-1) .. (4,-1);
\draw[shorten <=0.5cm,  shorten >=0.5cm, line width=0.04cm] (-4,-1) .. controls (-0.75,-1) and (-0.45,-1) .. (0,0) .. controls (0.45,1) and (0.75,1) .. (4,1);

\node[left] at (-3.75,1) {\small $k+1$};
\node[left] at (-3.75,-1) {\small $k$};
\node at (0,-0.85) {\large $\sigma_{i_{j}}$};
\node at (0,-3.75) {\Large $U_{j}$};    
\end{scope}

\node[above] at (0, 1.90) {\large $\mathbb{K}^{k+1}$};
\node[below] at (0, -1.90) {\large $\mathbb{K}^{k-1}$};
\node at (-1.25, 0) {\large $\mathbb{K}^{k}$};
\node at (1.25, 0) {\large $\mathbb{K}^{k}$};

\draw[->, shorten <=0.5cm,  shorten >=0.25cm, thick] (-1.5, 0) -- (0,2);
\draw[->, shorten <=0.5cm,  shorten >=0.25cm, thick] (1.5, 0) -- (0,2);
\draw[->, shorten <=0.25cm,  shorten >=0.5cm, thick] (0,-2) -- (-1.5, 0);
\draw[->, shorten <=0.25cm, shorten >=0.5cm, thick] (0,-2) -- (1.5, 0);

\node[left] at (-0.75, 1.5) {\large$\phi^{\,(k)}_{\sh{F}}$};
\node[right] at (0.75, 1.5) {\large$\widetilde{\phi}^{\,(k)}_{\sh{F}}$};
\node[left] at (-0.75, -1.5) {\large$\phi^{\,(k-1)}_{\sh{F}}$};
\node[right] at (0.75, -1.5) {\large$\widetilde{\phi}^{\,(k-1)}_{\sh{F}}$};

\endscope
\end{tikzpicture}
\caption{Commutative diagram for the linear maps that define a sheaf $\sh{F}$ on the region $U_{j}$ in the case $k=i_{j}\geq 2$.}
\label{fig: an object F on an open ball at a crossing sigma k}
\end{figure}

Building on the previous results, we now present a concise geometric description of $\sh{F}$ on the region $U_{\mathrm{B}}$. 

\begin{lemma}\label{Lemma for F in the region U_{B}}
Let $\beta:=\sigma_{i_{1}}\cdots \sigma_{i_{\ell}}\in\mathrm{Br}_{n}^{+}$ be a positive braid word, $\sh{F}$ an object of the category $\ccs{1}{\beta}$, and $\mathcal{U}_{\Lambda(\beta)}=\big\{U_{0}, U_{\mathrm{B}}, U_{\mathrm{L}}, U_{\mathrm{R}}, U_{\mathrm{T}}\big\}$ the open cover of $\mathbb{R}^{2}$ introduced in Construction~\eqref{Cons: Finite open cover for R^2}. Then, on the region $U_{\mathrm{B}}$, $\sh{F}$ is characterized by a sequence of complete flags $\big\{\fl{F}_{j}\big\}_{j=1}^{\ell+1}$ in $\mathbb{K}^{n}$, such that for each $j\in[1,\ell]$, $\fl{F}_{j+1}$ is in $s_{i_{j}}$-relative position with respect to $\fl{F}_{j}$.
\end{lemma}
\begin{proof}
Let $\mathcal{R}_{\Lambda(\beta)}=\big\{R_{j}\big\}_{j=1}^{\ell}$ be the collection of $\ell$ open vertical straps in $\mathbb{R}^{2}$ that partition the region $U_{\mathrm{B}}\subset \mathbb{R}^{2}$, as specified in Construction~\eqref{Cons: Definition of the vertical straps}. Specifically, for each $j\in [1,\ell]$, $R_{j}\subset \mathbb{R}^{2}$ denotes the open vertical strap whose intersection with the front projection $\Pi_{x,z}(\Lambda(\beta))$ consists of the braid diagram on $n$ strands associated with $\sigma_{i_{j}}$, i.e., the $j$-th crossing of $\beta$, as illustrated in Figure~\eqref{fig: A sub-regions R_j}. 

Next, we develop a recursive argument to analyze the behavior of $\sh{F}$ on each successive vertical strap $R_{j}$, thereby obtaining a geometric characterization of $\sh{F}$ on $U_{\mathrm{B}}$. With this strategy in place, we begin by studying $\sh{F}$ on the first region $R_{1}$. Let $k:=i_{1}\in[1,n-1]$ denote the index of $\sigma_{i_{1}}$---the first crossing of $\beta$ and the only crossing contained in $R_{1}$. Accordingly, we distinguish two cases:

{$\bullet$ \textit{Case 1}}: Suppose that $k=1$. Then, on $R_{1}$, $\sh{F}$ is defined by a collection of $n$ linear maps 
\begin{equation*}
\big\{\phi^{(i)}_{\sh{F}}:\mathbb{K}^{i}\to \mathbb{K}^{i+1}\big\}_{i=1}^{n-1}\,\cup\, \big\{\widetilde{\phi}^{\,(1)}_{\sh{F}}:\mathbb{K}^{1}\to \mathbb{K}^{2} \big\}\, , 
\end{equation*}
as shown in Figure~\eqref{fig: a sheaf in the vertical strap R_j containing a crossing sigma_1}, where $j=1$. In particular, these maps satisfy the compatibility conditions stated in Lemma~\eqref{lemma: description of an object F in a region R_j in terms of flags for k=1}. 

{$\bullet$ \textit{Case 2}}: Suppose that $k\geq 2$. Then, on $R_{j}$, $\sh{F}$ is defined by a collection of $n+1$ linear maps 
\begin{equation*}
\big\{\phi^{\,(i)}_{\sh{F}}:\mathbb{K}^{i}\to \mathbb{K}^{i+1}\big\}_{i=1}^{n-1}\,\cup\,\big\{ \widetilde{\phi}^{\,(k-1)}_{\sh{F}}:\mathbb{K}^{k-1}\to \mathbb{K}^{k}\,,~\widetilde{\phi}^{\,(k)}_{\sh{F}}:\mathbb{K}^{k}\to \mathbb{K}^{k+1}\,\big\}\, ,   
\end{equation*}
as shown in Figure~\eqref{fig: a sheaf in the vertical strap R_j containing a crossing sigma_k}, where $j=1$. Specifically, these maps satisfy the compatibility conditions stated in Lemma~\eqref{lemma: description of an object F in a region R_j in terms of flags for k geq 2}. 

Now, let $U_{\mathrm{L}}$ be the region of the open cover $\mathcal{U}_{\Lambda(\beta)}$ of $\mathbb{R}^{2}$ intersecting $U_{\mathrm{B}}$ on the left, as illustrated in Figure~\eqref{Front diagram decomposed into two regions}. By construction, $R_{1}$ is the only vertical strap in the partition $\mathcal{R}_{\Lambda(\beta)}$ of $U_{\mathrm{B}}$ whose intersection with $U_{\mathrm{L}}$ is nontrivial. Consequently, in either of the above cases, constructibility and the sheaf axioms imply that, on $U_{\mathrm{L}}$, $\sh{F}$ is determined by the collection of linear maps $\big\{\phi^{\,(i)}_{\sh{F}} \big\}_{i=1}^{n-1}$, and hence, by Lemma~\eqref{Lemma: linear map description of an object on the regions U_T, U_L, and U_R}, these maps are injective. Bearing this in mind, depending on the value of $k$, we can apply Lemma~\eqref{lemma: description of an object F in a region R_j in terms of flags for k=1} or Lemma~\eqref{lemma: description of an object F in a region R_j in terms of flags for k geq 2} to obtain that:
\begin{itemize}
\justifying
\item If $k=1$, the map $\widetilde{\phi}^{\,(1)}_{\sh{F}}$ is injective.
\item If $k\geq 2$, the maps $\widetilde{\phi}^{\,(k-1)}_{\sh{F}}$ and $\widetilde{\phi}^{\,(k)}_{\sh{F}}$ are injective.
\item On $R_{1}$, $\sh{F}$ is characterized by two complete flags $\fl{F}_{1}$ and $\fl{F}_{2}$ in $\mathbb{K}^{n}$, which are in $s_{k}$-relative position (recalling that $k=i_{1}$ as fixed earlier).
\end{itemize}
In particular, note that regardless of the value of $i_{1}$, all the linear maps defining $\sh{F}$ on $R_{1}$ are injective. 

Next, consider the region $R_{2}$ in the partition $\mathcal{R}_{\Lambda(\beta)}$ of $U_{\mathrm{B}}$; that is, the open vertical strap adjacent to $R_{1}$ on the right. By construction, the intersection $R_{1}\,\cap\,R_{2}\,\cap\,\Pi_{x,z}(\Lambda(\beta))$ consists of the trivial braid on $n$ strands, as illustrated in Figure~\eqref{fig: Intersection of two consecutive sub-regions R_j and R_{j+1}}. Thus, since all the linear maps defining $\sh{F}$ on $R_{1}$ are injective, the sheaf axioms ensure that the linear maps defining $\sh{F}$ on the intersection $R_{1}\,\cap\,R_{2}$  are also injective. Consequently, depending on the value of the index $i_{2}$ of the second crossing $\sigma_{i_{2}}$ of $\beta$, we can apply Lemma~\eqref{lemma: description of an object F in a region R_j in terms of flags for k=1} or Lemma~\eqref{lemma: description of an object F in a region R_j in terms of flags for k geq 2} to conclude that: 
\begin{itemize}
\justifying
\item The linear maps determining $\sh{F}$ on $R_{2}$ are injective.
\item On $R_{2}$, $\sh{F}$ is characterized by two complete flags $\fl{F}_{2}$ and $\fl{F}_{3}$ in $\mathbb{K}^{n}$, which are in $s_{i_{2}}$--relative position. 
\end{itemize}
In particular, by applying the above argument recursively to each vertical strap $R_{j}$ in the partition $\mathcal{R}_{\Lambda(\beta)}$ of $U_{\mathrm{B}}$, we deduce that on $U_{\mathrm{B}}$, the object $\sh{F}$ is determined by a sequence of complete flags $\big\{\fl{F}_{j}\big\}_{j=1}^{\ell+1}$ in $\mathbb{K}^{n}$, such that for each $j\in[1,\ell]$, $\fl{F}_{j+1}$ is in $s_{i_{j}}$-relative position with respect to $\fl{F}_{j}$.
\end{proof}

With the preceding results in place, we now present a global geometric characterization of $\sh{F}$, a result which also follows from~\cite{STZ1} and is included here for completeness.

\begin{theorem}\label{Flags and constructible sheaves}
Let $\beta:=\sigma_{i_{1}}\cdots\sigma_{i_{\ell}}\in\mathrm{Br}_{n}^{+}$ be a positive braid word, and let $\sh{F}$ be an object of the category $\ccs{1}{\beta}$. Then, as illustrated in Figure~\eqref{fig: sheaves and flags}, $\sh{F}$ is determined by a sequence of complete flags $\big\{\fl{F}_{j}\big\}_{j=0}^{\ell+1}$ in $\mathbb{K}^{n}$ satisfying the following compatibility conditions:
\begin{itemize}
\justifying
\item For each $j\in[1,\ell]$, $\fl{F}_{j+1}$ is in $s_{i_{j}}$-relative position with respect to $\fl{F}_{j}$. 
\item $\fl{F}_{0}$ is completely opposite to both $\fl{F}_{1}$ and $\fl{F}_{\ell+1}$.
\end{itemize}
\end{theorem}
\begin{proof}
Let $\mathcal{U}_{\Lambda(\beta)}=\big\{U_{0}, U_{\mathrm{B}}, U_{\mathrm{L}}, U_{\mathrm{R}}, U_{\mathrm{T}} \big\}$ denote the finite open cover of $\mathbb{R}^{2}$ introduced in Construction \eqref{Cons: Finite open cover for R^2}. By Lemma \eqref{lemma for F in the region U_{R}}, we know that:
\begin{itemize}
\justifying
\item On $U_{0}$, $\sh{F}$ is identically zero.
\item On $U_{\mathrm{T}}$, $\sh{F}$ is characterized by a complete flag $\fl{F}_{0}$ in $\mathbb{K}^{n}$. 
\item On $U_{\mathrm{L}}$, $\sh{F}$ is characterized by a complete flag $\fl{F}_{\mathrm{left}}$ in $\mathbb{K}^{n}$. 
\item On $U_{\mathrm{R}}$, $\sh{F}$ is characterized by a complete flag $\fl{F}_{\mathrm{right}}$ in $\mathbb{K}^{n}$. 
\item \textbf{Compatibility conditions}: $\fl{F}_{0}$ is completely opposite to both $\fl{F}_{\mathrm{left}}$ and $\fl{F}_{\mathrm{right}}$.
\end{itemize}

Furthermore, by Lemma \eqref{Lemma for F in the region U_{B}}, we have that on $U_{\mathrm{B}}$, $\sh{F}$ is characterized by a sequence of complete flags $\big\{\fl{F}_{j}\big\}_{j=1}^{\ell+1}$ in $\mathbb{K}^{n}$ such that, for each $j\in[1,\ell]$, $\fl{F}_{j+1}$ is in $s_{i_{j}}$-relative position with respect to $\fl{F}_{j}$. In particular, observe that on the left and right edges of the region $U_{\mathrm{B}}$, $\sh{F}$ is determined by the complete flags $\fl{F}_{1}$ and $\fl{F}_{\ell+1}$, respectively. Hence, by applying the sheaf axioms to the intersections $U_{\mathrm{L}}\,\cap\,U_{\mathrm{B}}$ and $U_{\mathrm{R}}\,\cap\, U_{\mathrm{B}}$, we obtain that $\fl{F}_{\mathrm{left}}=\fl{F}_{1}$ and $\fl{F}_{\mathrm{right}}=\fl{F}_{\ell+1}$. Bearing this in mind, we conclude that $\sh{F}$ is determined by a sequence of complete flags $\big\{\fl{F}_{j}\big\}_{j=0}^{\ell+1}$ in $\mathbb{K}^{n}$ satisfying the compatibility conditions stated in the theorem. 
\end{proof}

\begin{figure}
\centering
\begin{tikzpicture}
\useasboundingbox (-8,-5.25) rectangle (8,5);
\scope[transform canvas={scale=0.675}]

\draw[very thick] (-4-0.5,2) -- (4+0.5, 2);
\draw[very thick] (-4-0.5,-2) -- (-3, -2);
\draw[very thick] (3,-2) -- (4+0.5, -2);
\draw[very thick] (4+0.5,2) .. controls (4+0.5+1.5,2) and (4+0.5+2-0.75,0) .. (4+0.5+2+0.1,0);
\draw[very thick] (4+0.5,-2) .. controls (4+0.5+1.5,-2) and (4+0.5+2-0.75,0) .. (4+0.5+2+0.1,0);
\draw[very thick] (-4-0.5-2-0.1,0) .. controls (-4-0.5-2+0.75,0) and (-4-0.5-1.5,2) .. (-4-0.5,2);
\draw[very thick] (-4-0.5-2-0.1,0) .. controls (-4-0.5-2+0.75,0) and (-4-0.5-1.5,-2) .. (-4-0.5,-2);


\draw[very thick] (-4-0.5-1,2+0.75) -- (4+0.5+1, 2+0.75);
\draw[very thick] (-4-0.5-1,-2-0.75) -- (-3, -2-0.75);
\draw[very thick] (3,-2-0.75) -- (4+0.5+1, -2-0.75);
\draw[very thick] (4+0.5+1,2+0.75) .. controls (4+0.5+1+1.15,2+0.75) and (4+0.5+1+2-0.5,0) .. (4+0.5+1+2+0.25,0);
\draw[very thick] (4+0.5+1,-2-0.75) .. controls (4+0.5+1+1.15,-2-0.75) and (4+0.5+1+2-0.5,0) .. (4+0.5+1+2+0.25,0);
\draw[very thick] (-4-0.5-1-2-0.25,0) .. controls (-4-0.5-1-2+0.5,0) and (-4-0.5-1-1.15,2+0.75) .. (-4-0.5-1,2+0.75);
\draw[very thick] (-4-0.5-1-2-0.25,0) .. controls (-4-0.5-1-2+0.5,0) and (-4-0.5-1-1.15,-2-0.75) .. (-4-0.5-1,-2-0.75);


\draw[very thick] (-4-0.5-1-1,+2+3.5-0.5-0.75) -- (4+0.5+1+1,+2+3.5-0.5-0.75);
\draw[very thick] (-4-0.5-1-1,-2-3.5+0.5+0.75) -- (-3,-2-3.5+0.5+0.75);
\draw[very thick] (3,-2-3.5+0.5+0.75) -- (4+0.5+1+1,-2-3.5+0.5+0.75);
\draw[very thick] (4+0.5+1+1,+2+3.5-0.5-0.75) .. controls (4+0.5+1+1+1.15+0.65,+2+3.5-0.5-0.75) and (4+0.5+1+1+2-0.5+0.75,0) .. (4+0.5+1+1+2+0.75+0.35,0);
\draw[very thick] (4+0.5+1+1,-2-3.5+0.5+0.75) .. controls (4+0.5+1+1+1.15+0.65,-2-3.5+0.5+0.75) and (4+0.5+1+1+2-0.5+0.75,0) .. (4+0.5+1+1+2+0.75+0.35,0);
\draw[very thick] (-4-0.5-1-1-2-0.75-0.35,0) .. controls (-4-0.5-1-1-2+0.5-0.75,0) and (-4-0.5-1-1-1.15-0.65,+2+3.5-0.5-0.75) .. (-4-0.5-1-1,+2+3.5-0.5-0.75);
\draw[very thick] (-4-0.5-1-1-2-0.75-0.35,0) .. controls (-4-0.5-1-1-2+0.5-0.75,0) and (-4-0.5-1-1-1.15-0.65,-2-3.5+0.5+0.75) .. (-4-0.5-1-1,-2-3.5+0.5+0.75);


\draw[very thick] (-4-0.5-1-1-1,+2+3.5-0.5) -- (4+0.5+1+1+1,+2+3.5-0.5);
\draw[very thick] (-4-0.5-1-1-1,-2-3.5+0.5) -- (-3,-2-3.5+0.5);
\draw[very thick] (3,-2-3.5+0.5) -- (4+0.5+1+1+1,-2-3.5+0.5);
\draw[very thick] (4+0.5+1+1+1,+2+3.5-0.5) .. controls (4+0.5+1+1+1+1.15+0.65,+2+3.5-0.5) and (4+0.5+1+1+1+2-0.5+0.75+0.25,0) .. (4+0.5+1+1+1+2+0.75+0.25+0.35,0);
\draw[very thick] (4+0.5+1+1+1,-2-3.5+0.5) .. controls (4+0.5+1+1+1+1.15+0.65,-2-3.5+0.5) and (4+0.5+1+1+1+2-0.5+0.75+0.25,0) .. (4+0.5+1+1+1+2+0.75+0.25+0.35,0);
\draw[very thick] (-4-0.5-1-1-1-2-0.75-0.25-0.35,0) .. controls (-4-0.5-1-1-1-2+0.5-0.75-0.25,0) and (-4-0.5-1-1-1-1.15-0.65,+2+3.5-0.5) .. (-4-0.5-1-1-1,+2+3.5-0.5);
\draw[very thick] (-4-0.5-1-1-1-2-0.75-0.25-0.35,0) .. controls (-4-0.5-1-1-1-2+0.5-0.75-0.25,0) and (-4-0.5-1-1-1-1.15-0.65,-2-3.5+0.5) .. (-4-0.5-1-1-1,-2-3.5+0.5);


\draw[very thick] (-3,-2-3.5) rectangle (3,2-3.5);

\node at (-3+0.6,-2) {\normalsize$n$};
\node at (-3+0.6,-2-0.75) {\normalsize$n-1$};
\node at (-3+0.6,-2-3.5+0.5+0.75) {\normalsize$2$};
\node at (-3+0.6,-2-3.5+0.5) {\normalsize$1$};

\node at (3-0.6,-2) {\normalsize$n$};
\node at (3-0.6,-2-0.75) {\normalsize$n-1$};
\node at (3-0.6,-2-3.5+0.5+0.75) {\normalsize$2$};
\node at (3-0.6,-2-3.5+0.5) {\normalsize$1$};

\node at  (0,-3.5) {\huge$\beta$};

\filldraw[black] (-3+0.6,-3.5+0.25) circle (1pt);
\filldraw[black] (-3+0.6,-3.5) circle (1pt);
\filldraw[black] (-3+0.6,-3.5-0.25) circle (1pt);

\filldraw[black] (3-0.6,-3.5+0.25) circle (1pt);
\filldraw[black] (3-0.6,-3.5) circle (1pt);
\filldraw[black] (3-0.6,-3.5-0.25) circle (1pt);

\filldraw[black] (-3-0.5,-3.5+0.25) circle (1pt);
\filldraw[black] (-3-0.5,-3.5) circle (1pt);
\filldraw[black] (-3-0.5,-3.5-0.25) circle (1pt);

\filldraw[black] (3+0.5,-3.5+0.25) circle (1pt);
\filldraw[black] (3+0.5,-3.5) circle (1pt);
\filldraw[black] (3+0.5,-3.5-0.25) circle (1pt);

\filldraw[black] (-3-0.5,-3.5+0.25+7) circle (1pt);
\filldraw[black] (-3-0.5,-3.5+7) circle (1pt);
\filldraw[black] (-3-0.5,-3.5-0.25+7) circle (1pt);

\filldraw[black] (3+0.5,-3.5+0.25+7) circle (1pt);
\filldraw[black] (3+0.5,-3.5+7) circle (1pt);
\filldraw[black] (3+0.5,-3.5-0.25+7) circle (1pt);

\filldraw[black] (-8.45-0.25,0) circle (1pt);
\filldraw[black] (-8.45,0) circle (1pt);
\filldraw[black] (-8.45+0.25,0) circle (1pt);

\filldraw[black] (8.45-0.25,0) circle (1pt);
\filldraw[black] (8.45,0) circle (1pt);
\filldraw[black] (8.45+0.25,0) circle (1pt);

\node at (10,5.75) {\LARGE$\mathbb{R}^{2}_{x,z}$};


\node at (-2.5, -6.15) {\LARGE$\sigma_{i_{1}}$};
\node at (2.5, -6.15) {\LARGE$\sigma_{i_{\ell}}$};

\node[left] at (0-0.4, -6.15) {\LARGE$\sigma_{i_{2}}$};
\node[right] at (0+0.4, -6.15) {\LARGE$\sigma_{i_{\ell-1}}$};

\draw[line width=0.06cm, blue] (0, 1) -- (0, 6);
\draw[line width=0.06cm, blue] (0-4, -6.75) -- (0-4, -0.25);
\draw[line width=0.06cm, blue] (0+4, -6.75) -- (0+4, -0.25);
\draw[line width=0.06cm, blue] (0-1.75, -6.75) -- (0-1.75, -0.25);
\draw[line width=0.06cm, blue] (0+1.75, -6.75) -- (0+1.75, -0.25);

\node[blue, tape, draw, minimum height=0.25pt, minimum width=14pt, fill=blue, transform shape] at (0+0.25,6-0.25) {};
\node[blue, tape, draw, minimum height=0.25pt, minimum width=14pt, fill=blue, transform shape] at (-4+0.25,-0.25-0.25) {};
\node[blue, tape, draw, minimum height=0.25pt, minimum width=14pt, fill=blue, transform shape] at (4+0.25,-0.25-0.25) {};
\node[blue, tape, draw, minimum height=0.25pt, minimum width=14pt, fill=blue, transform shape] at (-1.75+0.25,-0.25-0.25) {};
\node[blue, tape, draw, minimum height=0.25pt, minimum width=14pt, fill=blue, transform shape] at (1.75+0.25,-0.25-0.25) {};

\node at (0-4, -7.5) {\LARGE $\fl{F}_{1}$}; 
\node at (0-1.75, -7.5) {\LARGE $\fl{F}_{2}$}; 
\node at (0+1.75, -7.5) {\LARGE $\fl{F}_{\ell}$}; 
\node at (0+4, -7.5) {\LARGE $\fl{F}_{\ell+1}$}; 
\node at (0, 6.75) {\LARGE $\fl{F}_{0}$}; 

\filldraw[black] (0,-6.15) circle (1pt);
\filldraw[black] (0-0.25,-6.15) circle (1pt);
\filldraw[black] (0+0.25,-6.15) circle (1pt);

\filldraw[black] (0,-7.25) circle (1pt);
\filldraw[black] (0-0.25,-7.25) circle (1pt);
\filldraw[black] (0+0.25,-7.25) circle (1pt);

\filldraw[black] (0,-0.85) circle (1pt);
\filldraw[black] (0-0.25,-0.85) circle (1pt);
\filldraw[black] (0+0.25,-0.85) circle (1pt);

\endscope
\end{tikzpicture}
\caption{An object $\sh{F}$ of the category $\ccs{1}{\beta}$.}
\label{fig: sheaves and flags}
\end{figure}

Having established a geometric description of the objects of the category $\ccs{1}{\beta}$, we now turn to an algebraic characterization, which is particularly well-suited for concrete computations and will play a key role in the analysis of the graded morphism spaces and their compositions in the category under study.

\subsection{An Algebraic Characterization of the Objects for General Positive Braids}
Let $\beta=\sigma_{i_{1}}\cdots\sigma_{i_{\ell}}\in\mathrm{Br}^{+}_{n}$ be a positive braid word. Building on the previous subsection, we now develop a correspondence between the objects of the category $\ccs{1}{\beta}$ and the points of the so-called \emph{braid variety} $X(\beta, \mathbb{K})$, thereby providing an algebraic description of these objects. With this goal in mind, we begin by introducing the following definitions. 

\begin{definition}
Let $\beta=\sigma_{i_{1}}\cdots \sigma_{i_{\ell}}\in \mathrm{Br}_{n}^{+}$ be a positive braid word. For each $j\in [1,\ell]$, we define $\beta_{j}:=\sigma_{i_{1}}\cdots\sigma_{i_{j}}\in\mathrm{Br}^{+}_{n}$ to be the truncation of $\beta$ at its $j$-th crossing. In particular, we have that $\beta_{\ell}=\beta$.    
\end{definition}

\begin{definition}
Let $\beta=\sigma_{i_{1}}\cdots \sigma_{i_{\ell}}\in \mathrm{Br}_{n}^{+}$ be a positive braid word, and let $\vec{z}=(z_{1}, \dots, z_{\ell})\in \mathbb{K}^{\ell}_{\mathrm{std}}$ be a fixed tuple. Then, the \emph{path matrix} $P_{\beta}(z_{1},\dots, z_{\ell})\in\mathrm{GL}\big(n,\mathbb{K}\big)$ associated with $\beta$ and $\vec{z}$ is defined by
\begin{equation*}
 P_{\beta}(z_{1},\dots, z_{\ell}):=B^{(n)}_{i_{1}}(z_{1})\cdots B^{(n)}_{i_{\ell}}(z_{\ell}),
\end{equation*}
where, for each $j\in[1,\ell]$, the matrix $B^{(n)}_{i_{j}}(z_{j})\in \mathrm{GL}\big(n,\mathbb{K}\big)$ denotes the $i_{j}$-th braid matrix of dimension $n$ associated with the $j$-th crossing $\sigma_{i_{j}}$ of $\beta$ and the $j$-th entry $z_{j}$ of $\vec{z}$. 
\end{definition}

\begin{remark}\label{Remark: path matrices and truncated braid words}
Let $\beta=\sigma_{i_{1}}\cdots \sigma_{i_{\ell}}\in \mathrm{Br}_{n}^{+}$ be a positive braid word. For each $k\in [1,\ell-1]$, the truncations satisfy $\beta_{k+1}=\beta_{k}\cdot \sigma_{i_{k+1}}$. Consequently, the associated path matrices satisfy $P_{\beta_{k+1}}(z_{1},\dots, z_{k+1})=P_{\beta_{k}}(z_{1},\dots, z_{k})\cdot B^{(n)}_{i_{k+1}}(z_{k+1})$.   
\end{remark}

\begin{definition}\label{Def: braid variety}
Let $\beta=\sigma_{i_{1}}\cdots \sigma_{i_{\ell}}\in \mathrm{Br}_{n}^{+}$ be a positive braid word. The \emph{braid variety}  $X(\beta,\mathbb{K})$ associated with $\beta$ corresponds to the affine variety defined by
\begin{equation}
X(\beta,\mathbb{K}):=\Big\{\, (z_{1},\dots, z_{\ell})\in \mathbb{K}^{\ell} \,\big|\, \text{$P_{\beta}(z_{1}, \dots, z_{\ell})$ admits an LU decomposition} \,\Big\}\, .    
\end{equation}
\end{definition}

\noindent
Braid varieties have been the focus of considerable research (see, for instance,~\cite{KT2, CN1, GSW1}). In particular, Casals, Gorsky, Gorsky, and Simental, along with their collaborators, have explored their cluster-algebraic structures and their connections with classical varieties in algebraic geometry, including positroid varieties~\cite{CGGS1, CGGLSS1, CGGS2}.

Next, we proceed to reformulate Theorem~\eqref{Flags and constructible sheaves} in terms of the braid matrices. To facilitate this, we present the following lemma.

\begin{lemma}\label{Lemma: sequences of flags and path matrices}
Let $\beta=\sigma_{i_{1}}\cdots\sigma_{i_{\ell}}\in\mathrm{Br}_{n}^{+}$ be a positive braid word, and let $\big\{\fl{F}_{j}\big\}_{j=1}^{\ell+1}$ be a sequence of complete flags in $\mathbb{K}^{n}$ such that, for each $j\in[1,\ell]$, $\fl{F}_{j+1}$ is in $s_{i_{j}}$-relative position with respect to $\fl{F}_{j}$. 

Let $\,\hat{\mathbf{f}}^{(n)}:=\big\{\hat{f}^{(n)}_{i}\big\}_{i=1}^{n}$ be a basis for $\mathbb{K}^{n}$, and suppose that, relative to this basis, $\fl{F}_{1}$ is the standard flag, that is, $\fl{F}_{1}\big[\,\hat{\mathbf{f}}^{(n)}\,\big]=\sfl\big[\,\hat{\mathbf{f}}^{(n)}\,\big]$. Then there exists a tuple $\vec{z}=(z_{1},\dots, z_{\ell})\in\mathbb{K}^{\ell}_{\mathrm{std}}$ such that, relative to the basis $\hat{\mathbf{f}}^{(n)}$ for $\mathbb{K}^{n}$, the flag $\fl{F}_{j+1}$ is represented by the path matrix $P_{\beta_{j}}(\vec{z}_{j})=B^{(n)}_{i_{1}}(z_{1})\cdots B^{(n)}_{i_{j}}(z_{j})\;\in \mathrm{GL}(n, \mathbb{K})$ associated with the truncated braid word $\beta_{j}=\sigma_{i_{1}}\cdots \sigma_{i_{j}}\in \mathrm{Br}^{+}_{n}$ and the truncated tuple $\vec{z}_{j}=(z_{1},\dots, z_{j})\in \mathbb{K}^{j}_{\mathrm{std}}$, for each $j\in [1,\ell]$. 
\end{lemma}
\begin{proof}
To prove the claim, we proceed recursively. To begin, observe that $\hat{\mathbf{f}}^{(n)}$ is a basis for $\mathbb{K}^{n}$ such that $\fl{F}_{1}\big[\,\hat{\mathbf{f}}^{(n)}\,\big]=\sfl\big[\,\hat{\mathbf{f}}^{(n)}\,\big]$. Hence, relative to this basis, the flag $\fl{F}_{1}$ is represented by the matrix $\mathbf{1}_{n}$. 

Now, consider the flag $\fl{F}_{2}$. By assumption, $\fl{F}_{2}$ is in $s_{i_{1}}$--relative position with respect to $\fl{F}_{1}$. Then, by Lemma~\eqref{Lemma for flags and braid matrices}, there exists $z_{1}\in \mathbb{K}$ such that, relative to the basis $\hat{\mathbf{f}}^{(n)}$, the flag $\fl{F}_{2}$ is represented by the matrix product $\mathbf{1}_{n}\cdot B^{(n)}_{i_{1}}(z_{1})$. In particular, since $\beta_{1}=\sigma_{i_{1}}$, we have that $P_{\beta_{1}}(z_{1})=B^{(n)}_{i_{1}}(z_{1})$, implying that the flag $\fl{F}_{2}$ is represented by the path matrix $P_{\beta_{1}}(z_{1})$. 

Next, consider the flag $\fl{F}_{3}$. By hypothesis, $\fl{F}_{3}$ is in $s_{i_{2}}$--relative position with respect to $\fl{F}_{2}$. Then, by Lemma~\eqref{Lemma for flags and braid matrices}, there exists $z_{2}\in \mathbb{K}$ such that, relative to the basis $\hat{\mathbf{f}}^{(n)}$, the flag $\fl{F}_{3}$ is represented by the matrix product $P_{\beta_{1}}(z_{1})\cdot B^{(n)}_{i_{2}}(z_{2})$. Furthermore, by Remark~\eqref{Remark: path matrices and truncated braid words}, we know that $P_{\beta_{2}}(z_{1},z_{2})=P_{\beta_{1}}(z_{1})\cdot B^{(n)}_{i_{2}}(z_{2})$, confirming that the flag $\fl{F}_{3}$ is represented by the path matrix $P_{\beta_{2}}(z_{1},z_{2})$. 

Continuing this process recursively, we conclude that there exists a tuple $\vec{z}=(z_{1},\dots, z_{\ell})\in\mathbb{K}^{\ell}_{\mathrm{std}}$ such that, relative to the basis $\hat{\mathbf{f}}^{(n)}$ for $\mathbb{K}^{n}$, the flag $\fl{F}_{j+1}$ is represented by the path matrix $P_{\beta_{j}}(\vec{z}_{j})=B^{(n)}_{i_{1}}(z_{1})\cdots B^{(n)}_{i_{j}}(z_{j})\,\in \mathrm{GL}(n, \mathbb{K})$ associated with the truncated braid word $\beta_{j}=\sigma_{i_{1}}\cdots \sigma_{i_{j}}\in \mathrm{Br}^{+}_{n}$ and the truncated tuple $\vec{z}_{j}=(z_{1},\dots, z_{j})\in\mathbb{K}^{j}_{\mathrm{std}}$, for each $j\in[1,\ell]$. This completes the proof.
\end{proof}

Building on the above preliminary result, we now present the following theorem, which algebraically characterizes an object $\sh{F}$ of the category $\ccs{1}{\beta}$ in terms of a basis for $\mathbb{K}^{n}$ and a point in the braid variety $X(\beta,\mathbb{K})$. 

\begin{theorem}\label{Prop. for sheaves and braid matrices}
Let $\beta=\sigma_{i_{1}}\cdots\sigma_{i_{\ell}}\in\mathrm{Br}_{n}^{+}$ be a positive braid word, and let $\sh{F}$ be an object of the category $\ccs{1}{\beta}$. In accordance with Theorem~\eqref{Flags and constructible sheaves}, we denote by $\big\{\fl{F}_{j}\big\}_{j=0}^{\ell+1}$ the sequence of complete flags in $\mathbb{K}^{n}$ that geometrically characterize $\sh{F}$. 

Let $\,\hat{\mathbf{f}}^{(n)}:=\big\{\hat{f}^{(n)}_{i}\big\}_{i=1}^{n}$ be a basis for $\mathbb{K}^{n}$, and suppose that, relative to this basis, $\fl{F}_{0}$ and $\fl{F}_{1}$ are the anti-standard and standard flags, respectively, that is, $\fl{F}_{0}\big[\,\hat{\mathbf{f}}^{(n)}\,\big]=\asfl\big[\,\hat{\mathbf{f}}^{(n)}\,\big] $ and $\fl{F}_{1}\big[\,\hat{\mathbf{f}}^{(n)}\,\big]=\sfl\big[\,\hat{\mathbf{f}}^{(n)}\,\big]$. 

Then $\sh{F}$ is algebraically parametrized by the basis $\hat{\mathbf{f}}^{(n)}$ and a point $\vec{z}=(z_{1},\dots, z_{\ell})$ in the braid variety $X(\beta,\mathbb{K})\subset \mathbb{K}^{\ell}_{\mathrm{std}}$. Furthermore, relative to the basis $\hat{\mathbf{f}}^{(n)}$ for $\mathbb{K}^{n}$, the flag $\fl{F}_{j+1}$ is represented by the path matrix $P_{\beta_{j}}(\vec{z}_{j}):=B^{(n)}_{i_{1}}(z_{1})\cdots B^{(n)}_{i_{j}}(z_{j})\;\in \mathrm{GL}(n, \mathbb{K})$ associated with the truncated braid word $\beta_{j}=\sigma_{i_{1}}\cdots \sigma_{i_{j}}\in \mathrm{Br}^{+}_{n}$ and the truncated tuple $\vec{z}_{j}=(z_{1},\dots, z_{j})\in \mathbb{K}^{j}_{\mathrm{std}}$, for each $j\in [1,\ell]$. 
\end{theorem}
\begin{proof}
By Theorem~\eqref{Flags and constructible sheaves}, $\big\{\fl{F}_{j}\big\}_{j=0}^{\ell+1}$ is a sequence of complete flags in $\mathbb{K}^{n}$ such that: $\fl{F}_{0}$ is completely opposite to both $\fl{F}_{1}$ and $\fl{F}_{\ell+1}$, and for each $j\in[1,\ell]$, $\fl{F}_{j+1}$ is in $s_{i_{j}}$-relative position with respect to $\fl{F}_{j}$. 

Observe that, relative to the basis $\hat{\mathbf{f}}^{(n)}$ for $\mathbb{K}^{n}$,  $\fl{F}_{0}$ and $\fl{F}_{1}$ are the anti-standard and standard flags, respectively. Consequently, Lemma~\eqref{Lemma: sequences of flags and path matrices} asserts that there exists a tuple $\vec{z}=(z_{1},\dots, z_{\ell})\in \mathbb{K}^{\ell}_{\mathrm{std}}$ such that, relative to the basis $\hat{\mathbf{f}}^{(n)}$ for $\mathbb{K}^{n}$, the flag $\fl{F}_{j+1}$ is represented by the path matrix $P^{(n)}_{\beta_{j}}(\vec{z}_{j})=B^{(n)}_{i_{1}}(z_{1})\cdots B^{(n)}_{i_{j}}(z_{j})\;\in \mathrm{GL}(n, \mathbb{K})$ associated with the truncated braid word $\beta_{j}=\sigma_{i_{1}}\cdots \sigma_{i_{j}}\in \mathrm{Br}^{+}_{n}$ and the truncated tuple $\vec{z}_{j}=(z_{1},\dots, z_{j})\in \mathbb{K}^{j}$, for each $j\in [1,\ell]$. Thus, since $\beta_{\ell}=\beta$ and $\vec{z}_{\ell}=\vec{z}$, we have that, relative to the basis $\hat{\mathbf{f}}^{(n)}$ for $\mathbb{K}^{n}$, the flag $\fl{F}_{\ell+1}$ is represented by the path matrix $P_{\beta}(\vec{z})$.  
 
Finally, recall that $\fl{F}_{0}$ is completely opposite to $\fl{F}_{\ell+1}$. Then, since relative to the basis $\hat{\mathbf{f}}^{(n)}$ for $\mathbb{K}^{n}$, $\fl{F}_{0}$ is the anti-standard flag, the Bruhat decomposition of $\mathrm{GL}(n,\mathbb{K})$ implies that $P_{\beta}(\vec{z})=\mathbf{w}_{n}\tilde{U}\mathbf{w}_{n}U$ for some upper triangular matrices $\tilde{U},U\in \mathrm{GL}(n,\mathbb{K})$, where $\mathbf{w}_{n}$ corresponds to the permutation matrix associated with the longest element in $\mathrm{S}_{n}$~\cite{CGGS1, CGGLSS1, SW1}. In particular, note that the matrix $L=\mathbf{w}_{n}\tilde{U}\mathbf{w}_{n}\in\mathrm{GL}(n,\mathbb{K})$ is lower triangular, and hence $P_{\beta}(\vec{z})=L\,U$, which shows $P_{\beta}(\vec{z})$ admits an LU factorization. Bearing this in mind, it follows from Definition~\eqref{Def: braid variety} that $\vec{z}$ is a point in the braid variety $X(\beta, \mathbb{K})$, which completes the proof.    
\end{proof}

Lastly, let $\sh{F}$ be an object of the category $\ccs{1}{\beta}$, and consider the open cover $\mathcal{U}_{\Lambda(\beta)}=\big\{U_{0}, U_{\mathrm{B}}, U_{\mathrm{L}}, U_{\mathrm{R}}, U_{\mathrm{T}}\big\}$ of $\mathbb{R}^{2}$ (cf. Construction~\eqref{Cons: Finite open cover for R^2}), together with the partition $\mathcal{R}_{\Lambda(\beta)}=\big\{R_{j}\big\}_{j=1}^{\ell}$ of $U_{\mathrm{B}}$ into $\ell$ open vertical straps (cf. Construction~\eqref{Cons: Definition of the vertical straps}). We conclude this subsection by describing how the global correspondence established in Theorem~\eqref{Prop. for sheaves and braid matrices} manifests locally on each vertical strap $R_{j}$. More precisely, following the approach developed by Chantraine, Ng, and Sivek in~\cite{CNS1}, we construct algebraic local models for $\sh{F}$ on the regions $R_{j}$, employing braid matrices. In particular, these local models will be instrumental in the explicit computation of the graded morphism spaces and their compositions in the category $\ccs{1}{\beta}$, which constitutes the main subject of study in the subsequent section. Accordingly, we present the following results.

\begin{lemma}[Algebraic Local Model I]\label{Lemma: matrix local model for sigma_1} \textbf{Setup}: Let $\beta=\sigma_{i_{1}}\dots \sigma_{i_{\ell}}\in\mathrm{Br}^{+}_{n}$ be a positive braid word, and let $\sh{F}$ be an object of the category $\ccs{1}{\beta}$. Fix $j\in [1,\ell]$, and let $R_{j}$ be the open vertical strap in $\mathbb{R}^{2}$ containing $\sigma_{i_{j}}$---the $j$-th crossing of $\beta$ (see Figure~\eqref{fig: A sub-regions R_j}). 

\vspace{0.5em}
\noindent
$\star$ \emph{Assumption 1}: Let $k:=i_{j}\in[1,n-1]$ denote the index of $\sigma_{i_{j}}$, and suppose that $k=1$.

In this setting, we have that on $R_{j}$, the sheaf $\sh{F}$ is specified by a collection of $n$ injective linear maps
\begin{equation*}
\big\{\phi^{\,(i)}_{\sh{F}}:\mathbb{K}^{i}\to\mathbb{K}^{i+1}\big\}_{i=1}^{n-1}\;\cup\;\big\{\widetilde{\phi}^{\,(1)}_{\sh{F}}:\mathbb{K}^{1}\to \mathbb{K}^{2}\, \big\}\, ,  
\end{equation*} 
as illustrated in Figure~\eqref{fig: a sheaf in the vertical strap R_j containing a crossing sigma_1}. According to Lemma~\eqref{lemma: description of an object F in a region R_j in terms of flags for k=1}, the compatibility conditions for these maps ensure that the induced complete flags in $\mathbb{K}^{n}$,
\begin{equation*}
\fl{F}_{j}:=\prescript{}{\mathcal{I}\,}{\fl{F}}\bigl(\,\phi_{\sh{F}}^{(1)},\phi_{\sh{F}}^{(2)},\dots,\phi_{\sh{F}}^{(n-1)}\,\bigr)\, , \quad \text{and} \quad \fl{F}_{j+1}:=\prescript{}{\mathcal{I}\,}{\fl{F}}\bigl(\,\widetilde{\phi}_{\sh{F}}^{\,(1)},\phi_{\sh{F}}^{(2)},\dots,\phi_{\sh{F}}^{(n-1)}\,\bigr)\, ,
\end{equation*}
are in $s_{1}$-relative position. 

\vspace{0.5em}
\noindent
$\star$ \emph{Assumption 2}:  For each $i\in [1,n]$, let $\hat{\mathbf{f}}^{(i)}:=\big\{ \hat{f}^{(i)}_{j}\big\}_{j=1}^{i}$ be a basis for $\mathbb{K}^{i}$, and suppose that $\big\{\hat{\mathbf{f}}^{(i)}\big\}_{i=1}^{n}$ is a system of bases adapted to $\big\{\phi^{(i)}_{\sh{F}}\big\}_{i=1}^{n-1}$ (cf. Definition~\eqref{Def:flags and adapted bases}--\eqref{Def: adapted bases I}).

Within this framework and relative to the basis $\hat{\mathbf{f}}^{(n)}$ for $\mathbb{K}^{n}$, Lemmas~\eqref{Lemma: I type flag in adapted bases for injective maps} and~\eqref{Lemma for flags and braid matrices} algebraically capture the geometric compatibility between the flags $\fl{F}_{j}$ and $\fl{F}_{j+1}$ as follows:
\begin{itemize}
\justifying
\item $\fl{F}_{j}$ is the standard flag, that is, $\fl{F}_{j}\big[\, \hat{\mathbf{f}}^{(n)} \,\big]=\sfl\big[\, \hat{\mathbf{f}}^{(n)} \,\big] $, and is therefore represented by the matrix $\mathbf{1}_{n}$. 
\item $\fl{F}_{j+1}$ is represented by the matrix product $\mathbf{1}_{n}\cdot B^{(n)}_{1}(z)$, where $B^{(n)}_{1}(z)$ denotes the first braid matrix of dimension $n$, for some $z\in \mathbb{K}$ parameterizing the $s_{1}$-relative position between $\fl{F}_{j}$ and $\fl{F}_{j+1}$. 
\end{itemize}

\vspace{4pt}
\noindent
\textbf{Main Conclusion}: Under the given \textbf{setup}, the following statements hold:
\begin{itemize}
\justifying
\item For each $i\in [1,n-1]$, the matrix $\,\tensor[_{\hat{\mathbf{f}}^{(i+1)}}]{ \big[\,\phi^{\,(i)}_{\sh{F}}\,\big] }{_{\hat{\mathbf{f}}^{(i)}}}\in M(i+1,i,\mathbb{K})$ representing $\phi^{(i)}_{\sh{F}}$ is given by
\begin{equation*}
\tensor[_{\hat{\mathbf{f}}^{(i+1)}}]{ \big[\,\phi^{\,(i)}_{\sh{F}}\,\big] }{_{\hat{\mathbf{f}}^{(i)}}}=\iota^{(i+1,i)}\, .    
\end{equation*}

\item The matrix $\,\tensor[_{\hat{\mathbf{f}}^{(2)}}]{ \big[\,\widetilde{\phi}^{\,(1)}_{\sh{F}}\,\big] }{_{\hat{\mathbf{f}}^{(1)}}}\in \mathrm{M}(2, 1, \mathbb{K})$ representing $\widetilde{\phi}^{\,(1)}_{\sh{F}}$ can be written as
\begin{equation*}
\tensor[_{\hat{\mathbf{f}}^{(2)}}]{ \big[\, \widetilde{\phi}^{\,(1)}_{\sh{F}}\,\big] }{_{\hat{\mathbf{f}}^{(1)}}}=B^{(2)}_{1}(z)\cdot \iota^{(2,1)}
\, ,
\end{equation*}
where $B^{(2)}_{1}(z)$ denotes the first braid matrix of dimension $2$.
\end{itemize}

\vspace{4pt}
\noindent
\textbf{Definition (system of bases adapted to $\sh{F}$ on $R_{j}$)}: For each $i\in [1, n]$, let $\hat{\mathbf{h}}^{(i)}:=\big\{\hat{h}^{(i)}_{j}\big\}_{j=1}^{i}$ be a basis for $\mathbb{K}^{i}$, and denote by $\big\{\hat{\mathbf{h}}^{(i)}\big\}_{i=1}^{n}$ the collection of these bases. 

We say that a pair $\big(\big\{\hat{\mathbf{h}}^{(i)}\big\}_{i=1}^{n}, \lambda\big)$, with $\lambda \in \mathbb{K}$, is a \emph{system of bases adapted to $\sh{F}$ on $R_{j}$} if, with respect to these bases, the linear maps defining $\sh{F}$ on $R_{j}$ have the following matrix representations: 
\begin{equation*}
\begin{aligned}
\tensor[_{\hat{\mathbf{h}}^{(i+1)}}]{ \big[\,\phi^{\,(i)}_{\sh{F}}\,\big] }{_{\hat{\mathbf{h}}^{(i)}}}&=\iota^{(i+1,i)}\, , &&\text{for all $i\in[1,n-1]$}\, ,\\[6pt]   
\tensor[_{\hat{\mathbf{h}}^{(2)}}]{ \big[\,\widetilde{\phi}^{\,(1)}_{\sh{F}}\,\big] }{_{\hat{\mathbf{h}}^{(1)}}}&=B^{(2)}_{1}(\lambda)\cdot\iota^{(2,1)}\, . &&
\end{aligned}   
\end{equation*} 

\noindent
In particular, under the given \textbf{setup}, the \textbf{main conclusion} implies that $\big(\big\{\hat{\mathbf{f}}^{(i)}\big\}_{i=1}^{n}, z\big)$ is a system of bases adapted to $\sh{F}$ on $R_{j}$.
\end{lemma}
\begin{proof}
By assumption, $\big\{\hat{\mathbf{f}}^{(i)}\big\}_{i=1}^{n}$ is a system of bases adapted to $\big\{\phi^{(i)}_{\sh{F}}\big\}_{i=1}^{n-1}$. Then, Definition~\eqref{Def:flags and adapted bases}--\eqref{Def: adapted bases I} asserts that, for each $i\in[1,n-1]$, the matrix $\tensor[_{\hat{\mathbf{f}}^{(i+1)}}]{ \big[\, \phi^{(i)}_{\sh{F}}\,\big] }{_{\hat{\mathbf{f}}^{(i)}}}\in M(i+1,i,\mathbb{K})$ representing $\phi^{(i)}_{\sh{F}}$ is the standard inclusion matrix: 
\begin{equation*}
\tensor[_{\hat{\mathbf{f}}^{(i+1)}}]{ \big[\, \phi^{\,(i)}_{\sh{F}}\,\big] }{_{\hat{\mathbf{f}}^{(i)}}}=\iota^{(i+1,i)}\, .    
\end{equation*}

Here, recall that, relative to the basis $\hat{\mathbf{f}}^{(n)}$ for $\mathbb{K}^{n}$, the flags $\fl{F}_{j}$ and $\fl{F}_{j+1}$ that geometrically characterize $\sh{F}$ on $R_{j}$ are represented by the matrices $\mathbf{1}_{n}$ and $B^{(n)}_{1}(z)$, respectively, where $B^{(n)}_{1}(z)$ denotes the first braid matrix of dimension $n$, for some $z\in\mathbb{K}$ parameterizing the $s_{1}$-relative position between $\fl{F}_{j}$ and $\fl{F}_{j+1}$. In particular, this suggests that the matrix $\tensor[_{\hat{\mathbf{f}}^{(2)}}]{ \big[\,\widetilde{\phi}^{\,(1)}_{\sh{F}}\,\big] }{_{\hat{\mathbf{f}}^{(1)}}}\in \mathrm{M}(2, 1, \mathbb{K})$ representing $\widetilde{\phi}^{\,(1)}_{\sh{F}}$ can be expressed as
\begin{equation*}
\tensor[_{\hat{\mathbf{f}}^{(2)}}]{ \big[\, \widetilde{\phi}^{\,(1)}_{\sh{F}}\,\big] }{_{\hat{\mathbf{f}}^{(1)}}}=B^{(2)}_{1}(z)\cdot \iota^{(2,1)}\, ,
\end{equation*}
where $B^{(2)}_{1}(z)$ denotes the first braid matrix of dimension $2$. 

Next, we verify that the linear map $\widetilde{\phi}^{\,(1)}_{\sh{F}}$, henceforth thought of as defined by the above matrix expression in the given bases, is compatible with the microlocal support conditions for $\sh{F}$ on $R_{j}$ and induces the flag $\fl{F}_{j+1}$ in $\mathbb{K}^{n}$. 

To begin, let $\vec{x}:=\alpha_{1}\hat{f}^{(2)}_{1}+\alpha_{2}\hat{f}^{(2)}_{2}\in \mathbb{K}^{2}$, for some $\alpha_{1},\alpha_{2}\in \mathbb{K}$. Then, we have that
\begin{equation*}
\begin{aligned}
 \vec{x}&=\alpha_{1}\hat{f}^{(2)}_{1} -z\,\alpha_{2}\hat{f}^{(2)}_{1}+ z\,\alpha_{2}\hat{f}^{(2)}_{1}+\alpha_{2}\hat{f}^{(2)}_{2}\, ,\\
 &=\phi^{(1)}_{\sh{F}}\big(\alpha_{1}\hat{f}^{(1)}_{1} -z\,\alpha_{2}\hat{f}^{(1)}_{1}\big)-\widetilde{\phi}^{\,(1)}_{\sh{F}}\big(-\alpha_{2}\hat{f}^{(1)}_{1}\big)\, ,
\end{aligned}
\end{equation*}
which implies that the linear map $\phi^{(1)}_{\sh{F}}\oplus \big(-\widetilde{\phi}^{\,(1)}_{\sh{F}} \big):\mathbb{K}\oplus \mathbb{K}\longrightarrow \mathbb{K}^{2}$ is surjective. Consequently, since the linear maps $\phi^{(1)}_{\sh{F}}$ and $\widetilde{\phi}^{\,(1)}_{\sh{F}}$ are injective, we conclude that: 
\begin{itemize}
\justifying
\item[(\textit{i})] The diagram in Figure~\eqref{fig: an object F on an open ball at a crossing sigma 1} commutes.
\item[(\textit{ii})] The sequence in Equation~\eqref{Eq: short exact sequence for the crossing sigma_1} is short exact.
\end{itemize}
Thus, $\widetilde{\phi}^{\,(1)}_{\sh{F}}$ is compatible with the microlocal support conditions for $\sh{F}$ on $R_{j}$. 

Finally, let $\big\{\hat{\mathbf{f}}^{(i)}[\sigma_{1},z]\big\}_{i=1}^{n}$ be the braid-transformed bases obtained from $\big\{\hat{\mathbf{f}}^{(i)}\big\}_{i=1}^{n}$ via $\sigma_{1}$ and the parameter $z$ (cf, Definition~\eqref{Def: braid transformation of bases}). In particular, observe that: 
\begin{itemize}
\justifying
\item By construction, $\big\{\hat{\mathbf{f}}^{(i)}[\sigma_{1},z]\big\}_{i=1}^{n}$ is a system of bases adapted to $\big\{\widetilde{\phi}^{\,(1)}_{\sh{F}},\phi^{(2)}_{\sh{F}}, \dots, \phi^{(n-1)}_{\sh{F}}\big\}$ in the sense of Definition~\eqref{Def:flags and adapted bases}--\eqref{Def: adapted bases I}.
\item $\fl{F}_{j+1}=\prescript{}{\mathcal{I}\,}{\fl{F}}\bigl(\,\widetilde{\phi}_{\sh{F}}^{\,(1)},\phi_{\sh{F}}^{(2)},\dots,\phi_{\sh{F}}^{(n-1)}\,\bigr)$ is the type $\mathcal{I}$ flag in $\mathbb{K}^{n}$ associated with the indicated collection of maps (cf. Definition~\eqref{Def:flags and adapted bases}--\eqref{Def: type I flag}).
\end{itemize}
Hence, by Lemma~\eqref{Lemma: I type flag in adapted bases for injective maps}, we obtain that, relative to the basis $\hat{\mathbf{f}}^{(n)}[\sigma_{1}, z]$ for $\mathbb{K}^{n}$, $\fl{F}_{j+1}$ is the standard flag, that is, $\fl{F}_{j+1}\big[\,\hat{\mathbf{f}}^{(n)}[\sigma_{1}, z]\, \big]=\sfl\big[\,\hat{\mathbf{f}}^{(n)}[\sigma_{1}, z]\, \big]$, and is therefore represented by the matrix $\mathbf{1}_{n}$.

Furthermore, by Definition~\eqref{Def: braid transformation of bases}, the bases $\hat{\mathbf{f}}^{(n)}$ and $\hat{\mathbf{f}}^{(n)}[\sigma_{1}, z]$ for $\mathbb{K}^{n}$ are related via the change-of-basis matrix $B^{(n)}_{1}(z)$; specifically, for each $j\in[1,n]$,
\begin{equation*}
\hat{f}^{(n)}_{j}[\sigma_{1}, z]=\sum_{q=1}^{n}\big(B^{(n)}_{1}(z)\big)_{q,j}\,\hat{f}^{(n)}_{q}\, . 
\end{equation*}
Therefore, by Lemma~\eqref{lemma: Matrices that represent a flag in two different bases}, we conclude that, relative to the basis $\hat{\mathbf{f}}^{(n)}$ for $\mathbb{K}^{n}$, the flag $\fl{F}_{j+1}$ is represented by the matrix product $B^{(n)}_{1}(z)\cdot\mathbf{1}_{n}$, confirming that the linear map $\widetilde{\phi}^{\,(1)}_{\sh{F}}$ induces the flag $\fl{F}_{j+1}$, as required.
\end{proof}

\begin{lemma}[Algebraic Local Model II]\label{Lemma: matrix local model for sigma_k}
\textbf{Setup}: Let $\beta=\sigma_{i_{1}}\dots \sigma_{i_{\ell}}\in\mathrm{Br}^{+}_{n}$ be a positive braid word, and let $\sh{F}$ be an object of the category $\ccs{1}{\beta}$. Fix $j\in [1,\ell]$, and let $R_{j}$ be the open vertical strap in $\mathbb{R}^{2}$ containing $\sigma_{i_{j}}$---the $j$-th crossing of $\beta$ (see Figure~\eqref{fig: A sub-regions R_j}).

\vspace{0.5em}
\noindent
$\star$ \emph{Assumption 1}: Let $k:=i_{j}\in[1,n-1]$ denote the index of $\sigma_{i_{j}}$, and suppose that $k\geq 2$.

In this setting, we have that on $R_{j}$, the sheaf $\sh{F}$ is specified by a collection of $n+1$ injective linear maps
\begin{equation*}
\big\{\phi^{\,(i)}_{\sh{F}}:\mathbb{K}^{i}\to \mathbb{K}^{i+1}\big\}_{i=1}^{n-1}\,\cup\,\big\{ \widetilde{\phi}^{\,(k-1)}_{\sh{F}}:\mathbb{K}^{k-1}\to \mathbb{K}^{k}\,,~\widetilde{\phi}^{\,(k)}_{\sh{F}}:\mathbb{K}^{k}\to \mathbb{K}^{k+1}\,\big\}\, ,   
\end{equation*}
as illustrated in Figure~\eqref{fig: a sheaf in the vertical strap R_j containing a crossing sigma_k}. According to Lemma~\eqref{lemma: description of an object F in a region R_j in terms of flags for k geq 2}, the compatibility conditions for these maps ensure that the induced complete flags in $\mathbb{K}^{n}$,
\begin{equation*}
\begin{aligned}
\fl{F}_{j}&:=\prescript{}{\mathcal{I}\,}{\fl{F}}\bigl(\, \phi^{(1)}_{\sh{F}}, \dots, \phi^{(k-2)}_{\sh{F}}, \phi^{\,(k-1)}_{\sh{F}}, \phi^{\,(k)}_{\sh{F}}, \phi^{(k+1)}_{\sh{F}},\dots, \phi^{(n-1)}_{\sh{F}} \,\bigr)\,,\\[6pt]  
\fl{F}_{j+1}&:=\prescript{}{\mathcal{I}\,}{\fl{F}}\bigl(\, \phi^{(1)}_{\sh{F}}, \dots, \phi^{(k-2)}_{\sh{F}}, \widetilde{\phi}^{\,(k-1)}_{\sh{F}}, \widetilde{\phi}^{\,(k)}_{\sh{F}}, \phi^{(k+1)}_{\sh{F}},\dots, \phi^{(n-1)}_{\sh{F}} \,\bigr)\,,
\end{aligned}
\end{equation*}
are in $s_{k}$-relative position. 

\noindent
$\star$ \textit{Assumption 2}: For each $i\in [1,n]$, let $\hat{\mathbf{f}}^{(i)}:=\big\{ \hat{f}^{(i)}_{j}\big\}_{j=1}^{i}$ be a basis for $\mathbb{K}^{i}$, and suppose that $\big\{\hat{\mathbf{f}}^{(i)}\big\}_{i=1}^{n}$ is a system of bases adapted to $\big\{\phi^{(i)}_{\sh{F}}\big\}_{i=1}^{n-1}$ (cf. Definition~\eqref{Def:flags and adapted bases}--\eqref{Def: adapted bases I}).

Within this framework and relative to the basis $\hat{\mathbf{f}}^{(n)}$ for $\mathbb{K}^{n}$, Lemmas~\eqref{Lemma: I type flag in adapted bases for injective maps} and~\eqref{Lemma for flags and braid matrices} algebraically capture the geometric compatibility between the flags $\fl{F}_{j}$ and $\fl{F}_{j+1}$ as follows:
\begin{itemize}
\justifying
\item[i)] $\fl{F}_{j}$ is the standard flag, that is, $\fl{F}_{j}\big[\, \hat{\mathbf{f}}^{(n)} \,\big]=\sfl\big[\, \hat{\mathbf{f}}^{(n)} \,\big] $, and is therefore represented by the matrix $\mathbf{1}_{n}$. 
\item[ii)] $\fl{F}_{j+1}$ is represented by the matrix product $\mathbf{1}_{n}\cdot B^{(n)}_{k}(z)$, where $B^{(n)}_{k}(z)$ denotes the $k$-th braid matrix of dimension $n$, for some $z\in \mathbb{K}$ parameterizing the $s_{k}$-relative position between $\fl{F}_{j}$ and $\fl{F}_{j+1}$. 
\end{itemize}

\vspace{4pt}
\noindent
\textbf{Main Result}: Under the given \textbf{setup}, the following statements hold:
\begin{itemize}
\justifying
\item For each $i\in [1,n-1]$, the matrix $\tensor[_{\hat{\mathbf{f}}^{(i+1)}}]{ \big[\,\phi^{\,(i)}_{\sh{F}}\,\big] }{_{\hat{\mathbf{f}}^{(i)}}}\in M(i+1,i,\mathbb{K})$ representing $\phi^{(i)}_{\sh{F}}$ is given by
\begin{equation*}
\tensor[_{\hat{\mathbf{f}}^{(i+1)}}]{ \big[\,\phi^{\,(i)}_{\sh{F}}\,\big] }{_{\hat{\mathbf{f}}^{(i)}}}=\iota^{(i+1,i)}\, .    
\end{equation*}
\item The matrices $\tensor[_{\hat{\mathbf{f}}^{(k)}}]{ \big[\, \widetilde{\phi}^{\,(k-1)}_{\sh{F}}\,\big] }{_{\hat{\mathbf{f}}^{(k-1)}}} \in \mathrm{M}(k, k-1, \mathbb{K})$ and $\tensor[_{\hat{\mathbf{f}}^{(k+1)}}]{ \big[\, \widetilde{\phi}^{\,(k)}_{\sh{F}}\,\big] }{_{\hat{\mathbf{f}}^{(k)}}}\in \mathrm{M}(k+1, k, \mathbb{K})$ representing $\widetilde{\phi}^{\,(k-1)}_{\sh{F}}$ and $\widetilde{\phi}^{\,(k)}_{\sh{F}}$, respectively, can be written as
\begin{equation*}
\tensor[_{\hat{\mathbf{f}}^{(k)}}]{ \big[\, \widetilde{\phi}^{\,(k-1)}_{\sh{F}}\,\big] }{_{\hat{\mathbf{f}}^{(k-1)}}}=\iota^{(k,k-1)}\, ,   \quad \text{and} \quad 
\tensor[_{\hat{\mathbf{f}}^{(k+1)}}]{ \big[\, \widetilde{\phi}^{\,(k)}_{\sh{F}}\,\big] }{_{\hat{\mathbf{f}}^{(k)}}}=B^{(k+1)}_{k}(z)\cdot \iota^{(k+1,k)}\, , 
\end{equation*}
where $B^{(k+1)}_{k}(z)$ denotes the $k$-th braid matrix of dimension $k+1$.
\end{itemize}

\vspace{4pt}
\noindent
\textbf{Definition (system of bases adapted to $\sh{F}$ on $R_{j}$)}: For each $i\in [1, n]$, let $\hat{\mathbf{h}}^{(i)}:=\big\{\hat{h}^{(i)}_{j}\big\}_{j=1}^{i}$ be a basis for $\mathbb{K}^{i}$, and denote by $\big\{\hat{\mathbf{h}}^{(i)}\big\}_{i=1}^{n}$ the collection of these bases. 

We say that a pair $\big(\big\{\hat{\mathbf{h}}^{(i)}\big\}_{i=1}^{n}, \lambda\big)$, with $\lambda \in \mathbb{K}$, is a \emph{system of bases adapted to $\sh{F}$ on $R_{j}$} if, with respect to these bases, the linear maps defining $\sh{F}$ on $R_{j}$ have the following matrix representations: 
\begin{equation*}
\begin{aligned}
\tensor[_{\hat{\mathbf{h}}^{(i+1)}}]{ \big[\,\phi^{\,(i)}_{\sh{F}}\,\big] }{_{\hat{\mathbf{h}}^{(i)}}}&=\iota^{(i+1,i)}\, , \qquad&&\text{for all $i\in[1,n-1]$}\, ,\\[8pt]   
\tensor[_{\hat{\mathbf{h}}^{(k)}}]{ \big[\,\widetilde{\phi}^{\,(k-1)}_{\sh{F}}\,\big] }{_{\hat{\mathbf{h}}^{(k-1)}}}&=\iota^{(k,k-1)}\, , \qquad&&\tensor[_{\hat{\mathbf{h}}^{(k+1)}}]{ \big[\,\widetilde{\phi}^{\,(k)}_{\sh{F}}\,\big] }{_{\hat{\mathbf{h}}^{(k)}}}=B^{(k+1)}_{k}(\lambda)\cdot\iota^{(k+1,k)}\, . \\[6pt]  
\end{aligned}   
\end{equation*}

\noindent
In particular, under the given \textbf{setup}, the \textbf{main conclusion} implies that $\big(\big\{\hat{\mathbf{f}}^{(i)}\big\}_{i=1}^{n}, z\big)$ is a system of bases adapted to $\sh{F}$ on $R_{j}$.
\end{lemma}
\begin{proof}
By assumption, $\big\{\hat{\mathbf{f}}^{(i)}\big\}_{i=1}^{n}$ is a system of bases adapted to $\big\{\phi^{(i)}_{\sh{F}}\big\}_{i=1}^{n-1}$. Then, Definition~\eqref{Def:flags and adapted bases}--\eqref{Def: adapted bases I} asserts that, for each $i\in[1,n-1]$, the matrix $\tensor[_{\hat{\mathbf{f}}^{(i+1)}}]{ \big[\, \phi^{(i)}_{\sh{F}}\,\big] }{_{\hat{\mathbf{f}}^{(i)}}}\in M(i+1,i,\mathbb{K})$ representing $\phi^{(i)}_{\sh{F}}$ is the standard inclusion matrix: 
\begin{equation*}
\tensor[_{\hat{\mathbf{f}}^{(i+1)}}]{ \big[\, \phi^{\,(i)}_{\sh{F}}\,\big] }{_{\hat{\mathbf{f}}^{(i)}}}=\iota^{(i+1,i)}\, .    
\end{equation*} 

Here, recall that, relative to the basis $\hat{\mathbf{f}}^{(n)}$ for $\mathbb{K}^{n}$, the flags $\fl{F}_{j}$ and $\fl{F}_{j+1}$ that geometrically characterize $\sh{F}$ on $R_{j}$ are represented by the matrices $\mathbf{1}_{n}$ and $B^{(n)}_{k}(z)$, respectively, where $B^{(n)}_{k}(z)$ denotes the $k$-th braid matrix of dimension $n$, for some $z\in\mathbb{K}$ parameterizing the $s_{k}$-relative position between $\fl{F}_{j}$ and $\fl{F}_{j+1}$. In particular, this suggests that the matrices $\tensor[_{\hat{\mathbf{f}}^{(k)}}]{ \big[\, \widetilde{\phi}^{\,(k-1)}_{\sh{F}}\,\big] }{_{\hat{\mathbf{f}}^{(k-1)}}} \in \mathrm{M}(k, k-1, \mathbb{K})$ and $\tensor[_{\hat{\mathbf{f}}^{(k+1)}}]{ \big[\, \widetilde{\phi}^{\,(k)}_{\sh{F}}\,\big] }{_{\hat{\mathbf{f}}^{(k)}}}\in \mathrm{M}(k+1, k, \mathbb{K})$ representing $\widetilde{\phi}^{\,(k-1)}_{\sh{F}}$ and $\widetilde{\phi}^{\,(k)}_{\sh{F}}$, respectively, can be expressed as
 \begin{equation*}
\tensor[_{\hat{\mathbf{f}}^{(k)}}]{ \big[\, \widetilde{\phi}^{\,(k-1)}_{\sh{F}}\,\big] }{_{\hat{\mathbf{f}}^{(k-1)}}}=\iota^{(k,k-1)},   \quad \text{and} \quad 
\tensor[_{\hat{\mathbf{f}}^{(k+1)}}]{ \big[\, \widetilde{\phi}^{\,(k)}_{\sh{F}}\,\big] }{_{\hat{\mathbf{f}}^{(k)}}}=B^{(k+1)}_{k}(z)\cdot \iota^{(k+1,k)}\, , 
\end{equation*}
where $B^{(k+1)}_{k}(z)$ denotes the $k$-th braid matrix of dimension $k+1$. 

Next, we verify that the linear maps $\widetilde{\phi}^{\,(k-1)}_{\sh{F}}$ and $\widetilde{\phi}^{\,(k)}_{\sh{F}}$, henceforth thought of as defined by the above matrix expressions
in the given bases, are compatible with the microlocal support conditions for $\sh{F}$ on $R_{j}$ and induce the flag $\fl{F}_{j+1}$ in $\mathbb{K}^{n}$. 

To begin, a direct calculation shows that
\begin{equation*}
\begin{aligned}
\tensor[_{\hat{\mathbf{f}}^{(k+1)}}]{ \big[\, \widetilde{\phi}^{\,(k)}_{\sh{F}}\,\big] }{_{\hat{\mathbf{f}}^{(k)}}}\cdot \tensor[_{\hat{\mathbf{f}}^{(k)}}]{ \big[\, \widetilde{\phi}^{\,(k-1)}_{\sh{F}}\,\big] }{_{\hat{\mathbf{f}}^{(k-1)}}}&=B^{(k+1)}_{k}(z)\cdot \iota^{(k+1,k)}\cdot \iota^{(k,k-1)}  \, , \\
&=B^{(k+1)}_{k}(z)\cdot \iota^{(k+1,k-1)}\, ,\\
&=\iota^{(k+1,k-1)}\, ,
\end{aligned}
\end{equation*}
and similarly,
\begin{equation*}
\begin{aligned}
\tensor[_{\hat{\mathbf{f}}^{(k+1)}}]{ \big[\, \phi^{\,(k)}_{\sh{F}}\,\big] }{_{\hat{\mathbf{f}}^{(k)}}}\cdot \tensor[_{\hat{\mathbf{f}}^{(k)}}]{ \big[\, \phi^{\,(k-1)}_{\sh{F}}\,\big] }{_{\hat{\mathbf{f}}^{(k-1)}}}&=\iota^{(k+1,k)}\cdot \iota^{(k,k-1)}  \, , \\
&=\iota^{(k+1,k-1)}\, .
\end{aligned}
\end{equation*}
It follows that
\begin{equation*}
\phi^{(k)}_{\sh{F}}\circ \phi^{(k-1)}_{\sh{F}}(y)=\widetilde{\phi}^{\,(k)}_{\sh{F}}\circ \widetilde{\phi}^{\,(k-1)}_{\sh{F}}(y)\, ,  
\end{equation*}
for all $y:=\beta_{1}\hat{f}^{(k-1)}_{1}+\cdots+\beta_{k-1}\hat{f}^{(k-1)}_{k-1} \in\mathbb{K}^{k-1}$, with $\beta_{1},\dots, \beta_{k-1}\in \mathbb{K}$. 

Now, let $\vec{x}:=\alpha_{1}\hat{f}^{(k+1)}_{1}+\cdots+\alpha_{k+1}\hat{f}^{(k+1)}_{k+1}\in \mathbb{K}^{k+1}$, for some $\alpha_{1},\dots, \alpha_{k+1}\in \mathbb{K}$. Then, we have that
\begin{equation*}
\begin{aligned}
 \vec{x}&=\alpha_{1}\hat{f}^{(k+1)}_{1} +\cdots + \alpha_{k}\hat{f}^{(k+1)}_{k} -z\,\alpha_{k+1}\hat{f}^{(k+1)}_{k} + z\,\alpha_{k+1}\hat{f}^{(k+1)}_{k}+\alpha_{k+1}\hat{f}^{(k+1)}_{k+1}\, ,\\
 &=\phi^{(k)}_{\sh{F}}\big( \alpha_{1}\hat{f}^{(k)}_{1} +\cdots + \alpha_{k}\hat{f}^{(k)}_{k} -z\,\alpha_{k+1}\hat{f}^{(k)}_{k}\big)-\widetilde{\phi}^{\,(k)}_{\sh{F}}\big(-\alpha_{k+1}\hat{f}^{(k)}_{k}\big)\,  ,
\end{aligned}
\end{equation*}
which implies that the linear map $\phi^{(k)}_{\sh{F}}\oplus (-\widetilde{\phi}^{\,(k)}_{\sh{F}}):\mathbb{K}^{k}\oplus \mathbb{K}^{k}\longrightarrow \mathbb{K}^{k+1}$ is surjective. 

In particular, since the linear maps $\phi^{(k-1)}_{\sh{F}}$, $\phi^{(k)}_{\sh{F}}$, $\widetilde{\phi}^{\,(k-1)}_{\sh{F}}$,  and $\widetilde{\phi}^{\,(k)}_{\sh{F}}$ are injective, the above results guarantee that: 
\begin{itemize}
\justifying
\item[(\textit{i})] The diagram in Figure~\eqref{fig: an object F on an open ball at a crossing sigma k} commutes.
\item[(\textit{ii})] The sequence in Equation~\eqref{Eq: short exact sequence for the crossing sigma_k} is short exact.
\end{itemize}
Hence, $\widetilde{\phi}^{\,(k-1)}_{\sh{F}}$ and $\widetilde{\phi}^{\,(k)}_{\sh{F}}$ are compatible with the microlocal support conditions for $\sh{F}$ on $R_{j}$. 

Finally, let $\big\{\hat{\mathbf{f}}^{(i)}[\sigma_{k},z]\big\}_{i=1}^{n}$ be the braid-transformed bases obtained from $\big\{\hat{\mathbf{f}}^{(i)}\big\}_{i=1}^{n}$ via $\sigma_{k}$ and the parameter $z$ (cf. Definition~\eqref{Def: braid transformation of bases}). In particular, observe that: 
\begin{itemize}
\justifying
\item By construction, $\big\{\hat{\mathbf{f}}^{(i)}[\sigma_{k},z]\big\}_{i=1}^{n}$ is a system of bases adapted to the collection of linear maps
\begin{equation*}
\big\{\phi_{\sh{F}}^{(1)},\dots,\phi_{\sh{F}}^{(k-2)},\widetilde{\phi}_{\sh{F}}^{\,(k-1)},\widetilde{\phi}_{\sh{F}}^{\,(k)}, \phi_{\sh{F}}^{(k+1)},\dots,\phi_{\sh{F}}^{(n-1)}\big\}    
\end{equation*}
in the sense of Definition~\eqref{Def:flags and adapted bases}--\eqref{Def: adapted bases I}.
\item $\fl{F}_{j+1}=\prescript{}{\mathcal{I}\,}{\fl{F}}\bigl(\, \phi^{(1)}_{\sh{F}}, \dots, \phi^{(k-2)}_{\sh{F}}, \widetilde{\phi}^{\,(k-1)}_{\sh{F}}, \widetilde{\phi}^{\,(k)}_{\sh{F}}, \phi^{(k+1)}_{\sh{F}},\dots, \phi^{(n-1)}_{\sh{F}} \,\bigr)$ is the type $\mathcal{I}$ flag in $\mathbb{K}^{n}$ associated with the indicated collection of maps (cf. Definition~\eqref{Def:flags and adapted bases}--\eqref{Def: type I flag}).
\end{itemize}
Hence, by Lemma~\eqref{Lemma: I type flag in adapted bases for injective maps}, we obtain that, relative to the basis $\hat{\mathbf{f}}^{(n)}[\sigma_{k}, z]$ for $\mathbb{K}^{n}$, the flag $\fl{F}_{j+1}$ is the standard flag, that is, $\fl{F}_{j+1}\big[\,\hat{\mathbf{f}}^{(n)}[\sigma_{k}, z]\, \big]=\sfl\big[\,\hat{\mathbf{f}}^{(n)}[\sigma_{k}, z]\, \big]$, and is therefore represented by the matrix $\mathbf{1}_{n}$.

Furthermore, by Definition~\eqref{Def: braid transformation of bases}, the bases $\hat{\mathbf{f}}^{(n)}$ and $\hat{\mathbf{f}}^{(n)}[\sigma_{k}, z]$ for $\mathbb{K}^{n}$ are related via the change-of-basis matrix $B^{(n)}_{k}(z)$; specifically, for each $j\in[1,n]$,
\begin{equation*}
\hat{f}^{(n)}_{j}[\sigma_{k}, z]=\sum_{q=1}^{n}\big(B^{(n)}_{k}(z)\big)_{q,j}\,\hat{f}^{(n)}_{q}\, . 
\end{equation*}
Therefore, by Lemma \eqref{lemma: Matrices that represent a flag in two different bases}, we conclude that, relative to the basis $\hat{\mathbf{f}}^{(n)}$ for $\mathbb{K}^{n}$, the flag $\fl{F}_{j+1}$ is represented by the matrix product $B^{(n)}_{k}(z)\cdot\mathbf{1}_{n}$, confirming that the linear maps $\widetilde{\phi}^{\,(k-1)}_{\sh{F}}$ and $\widetilde{\phi}^{\,(k)}_{\sh{F}}$ induce the flag $\fl{F}_{j+1}$, as required.
\end{proof}

\begin{remark}\label{Remark: intuitive explanation of the algebraic local models}
Let $\beta=\sigma_{i_{1}}\cdots \sigma_{i_{\ell}}\in \mathrm{Br}^{+}_{n}$ be a positive braid word, $\sh{F}$ an object of the category $\ccs{1}{\beta}$, and $\mathcal{R}_{\Lambda(\beta)}=\big\{R_{j}\big\}_{j=1}^{\ell}$ the collection of $\ell$ open vertical straps in $\mathbb{R}^{2}$ introduced in Construction~\eqref{Cons: Definition of the vertical straps}. For each $i\in [1, n]$, let $\,\hat{\mathbf{f}}^{(i)}:=\big\{\hat{f}^{(i)}_{j}\big\}_{j=1}^{i}$ be a basis for $\mathbb{K}^{i}$, and denote by $\big\{\hat{\mathbf{f}}^{(i)}\big\}_{i=1}^{n}$ the collection of these bases. 

Intuitively, Lemmas~\eqref{Lemma: matrix local model for sigma_1} and~\eqref{Lemma: matrix local model for sigma_k} assert the following. If $\big\{\hat{\mathbf{f}}^{(i)}\big\}_{i=1}^{n}$ is a system of bases adapted to the $n-1$ injective linear maps defining $\sh{F}$ to the left of a crossing $\sigma_{i_{j}}$ within a vertical strap $R_{j}$, then the pair $\big(\big\{\hat{\mathbf{f}}^{(i)}\big\}_{i=1}^{n}, z\big)$ defines a system of bases adapted to $\sh{F}$ on $R_{j}$, where $z\in \mathbb{K}$ algebraically encodes, relative to the basis $\hat{\mathbf{f}}^{(n)}$ for $\mathbb{K}^{n}$, the $s_{i_{j}}$-relative position between the complete flags in $\mathbb{K}^{n}$ that geometrically characterize $\sh{F}$ on $R_{j}$. In other words, the behavior of $\sh{F}$ on $R_{j}$ is completely determined by the single parameter $z$ together with the adapted bases $\big\{\hat{\mathbf{f}}^{(i)}\big\}_{i=1}^{n}$ on the left of the crossing $\sigma_{i_{j}}$.
\end{remark}

Building on the local algebraic descriptions established in Lemmas~\eqref{Lemma: matrix local model for sigma_1} and~\eqref{Lemma: matrix local model for sigma_k}, we now state the following result.

\begin{theorem}\label{Theorem: adapted bases for an object at a region R_j}
\textbf{Setup}: Let $\beta=\sigma_{i_{1}}\cdots \sigma_{i_{\ell}}\in\mathrm{Br}^{+}_{n}$ be a positive braid word, and let $\sh{F}$ be an object of the category $\ccs{1}{\beta}$. Let $\,\mathcal{U}_{\Lambda(\beta)}=\big\{U_{0}, U_{\mathrm{B}}, U_{\mathrm{L}}, U_{\mathrm{R}}, U_{\mathrm{T}}\big\}$ be the open cover of $\mathbb{R}^{2}$ from Construction~\eqref{Cons: Finite open cover for R^2}, and let $\,\mathcal{R}_{\Lambda(\beta)}=\big\{R_{j}\big\}_{j=1}^{\ell}$ be the partition of $U_{\mathrm{B}}$ into $\ell$ open vertical straps from Construction~\eqref{Cons: Definition of the vertical straps}.

\vspace{4pt}
\noindent
$\star$ \emph{Local descriptions on $U_{\mathrm{T}}$ and $U_{\mathrm{L}}$}: According to Lemma~\eqref{Lemma: linear map description of an object on the regions U_T, U_L, and U_R}, $\sh{F}$ admits the following local descriptions:
\begin{itemize}
\justifying
\item On $U_{\mathrm{T}}$, $\sh{F}$ is specified by a collection of $n-1$ surjective linear maps $\big\{ \psi_{\sh{F}}^{(i)}:\mathbb{K}^{i+1}\to \mathbb{K}^{i}\big\}_{i=1}^{n-1}$ (see Figure~\eqref{Fig: an object F in the region U_T}).
\item On $U_{\mathrm{L}}$, $\sh{F}$ is specified by a collection of $n-1$ injective linear maps $\big\{\phi_{\sh{F}}^{(i)}:\mathbb{K}^{i}\to \mathbb{K}^{i+1}\big\}_{i=1}^{n-1}$ (see Figure~\eqref{Fig: an object F in the region U_L}).
\item \textbf{Compatibility conditions}: For each $i\in[1,n-1]$,
\begin{equation*}
\psi^{(i)}_{\sh{F}}\circ \phi^{(i)}_{\sh{F}}=\mathrm{id}_{\mathbb{K}^{i}}\, .
\end{equation*}
\end{itemize}

\vspace{4pt}
\noindent
$\star$ \emph{Global flag data}: By Theorem~\eqref{Flags and constructible sheaves}, $\sh{F}$ is geometrically characterized by a sequence of complete flags $\big\{\fl{F}_{j}\big\}_{j=0}^{\ell+1}$ in $\mathbb{K}^{n}$ such that: $\fl{F}_{0}$ is completely opposite to both $\fl{F}_{1}$ and $\fl{F}_{\ell+1}$, and for each $j\in[1,\ell]$, $\fl{F}_{j}$ is in $s_{i_{j}}$-relative position with respect to $\fl{F}_{j+1}$. In particular, by Lemma~\eqref{lemma for F in the region U_{R}}, we know that:
\begin{equation*}
\fl{F}_{0}:=\prescript{}{\mathcal{K}\,}{\fl{F}}\big(\psi_{\sh{F}}^{(1)},\dots, \psi_{\sh{F}}^{(n-1)}\big)\, , \quad \text{and} \quad
\fl{F}_{1}:=\prescript{}{\mathcal{I}\,}{\fl{F}}\big(\phi_{\sh{F}}^{(1)},\dots, \phi_{\sh{F}}^{(n-1)}\big)\, ,
\end{equation*}
are the type $\mathcal{K}$ and type $\mathcal{I}$ flags in $\mathbb{K}^{n}$ associated with $\big\{\psi_{\sh{F}}^{(i)}\big\}_{i=1}^{n-1}$ and $\big\{\phi_{\sh{F}}^{(i)}\big\}_{i=1}^{n-1}$, respectively (cf. Definition~\eqref{Def:flags and adapted bases}--\eqref{Def: type I flag}--\eqref{Def: type K flag}). 

\vspace{4pt}
\noindent
$\star$ \emph{Main assumption} (Adapted bases): Let \,$\hat{\mathbf{f}}^{(n)}$ be a basis for $\mathbb{K}^{n}$, and let $\vec{z}=(z_{1},\dots, z_{\ell})\in X(\beta,\mathbb{K})$ be a point such that the pair $(\,\hat{\mathbf{f}}^{(n)}, \,\vec{z}\, )$ algebraically characterizes $\sh{F}$ according to Theorem~\eqref{Prop. for sheaves and braid matrices}.

In this setting, we have that relative to the basis $\hat{\mathbf{f}}^{(n)}$ for $\mathbb{K}^{n}$, the flags $\fl{F}_{0}$ and $\fl{F}_{1}$ are the anti-standard and standard flags, respectively, and for each $j\in[1,\ell]$, the flag $\fl{F}_{j+1}$ is represented by the path matrix $P_{\beta_{j}}(\vec{z}_{j})=B^{(n)}_{i_{1}}(z_{1})\cdots B^{(n)}_{i_{j}}(z_{j})\,\in\mathrm{GL}(n,\mathbb{K})$ associated with the truncated braid word $\beta_{j}=\sigma_{i_{1}}\cdots \sigma_{i_{j}}\in\mathrm{Br}^{+}_{n}$ and the truncated tuple $\vec{z}_{j}=(z_{1},\dots, z_{j})\in\mathbb{K}^{j}_{\mathrm{std}}$.  

Furthermore, under this assumption, the \textbf{compatibility conditions} and an inductive argument assert that, for each $i\in [1,n-1]$, there is a unique basis $\hat{\mathbf{f}}^{(i)}:=\big\{ \hat{f}^{(i)}_{j}\big\}_{j=1}^{i}$ for $\mathbb{K}^{i}$ such that the collection $\big\{\hat{\mathbf{f}}^{(i)}\big\}_{i=1}^{n}$ is a system of bases adapted to both $\big\{\psi_{\sh{F}}^{(i)}\big\}_{i=1}^{n-1}$ and $\big\{\phi_{\sh{F}}^{(i)}\big\}_{i=1}^{n-1}$ (cf. Definition~\eqref{Def:flags and adapted bases}--\eqref{Def: adapted bases I}--\eqref{Def: adapted bases II}).

\smallskip
\noindent
\textbf{Main conclusion}: For each $j\in[1,\ell]$, let $\big\{\hat{\mathbf{f}}^{(i)}[\beta_{j},\vec{z}_{j} ]\big\}_{i=1}^{n}$ denote the braid-transformed bases obtained from $\big\{\hat{\mathbf{f}}^{(i)} \big\}_{i=1}^{n}$ via the truncated braid word $\beta_{j}$ and the truncated tuple $\vec{z}_{j}$ (cf. Definition~\eqref{Def: braid transformation of bases}). Then: 
\begin{itemize}
\justifying
\item $\big(\big\{\hat{\mathbf{f}}^{(i)} \big\}_{i=1}^{n}, z_{1}\big)$ is a system of bases adapted to $\sh{F}$ on $R_{1}$.
\item For each $j\in[1,\ell-1]$, $\big(\big\{\hat{\mathbf{f}}^{(i)}[\beta_{j},\vec{z}_{j} ]\big\}_{i=1}^{n}, z_{j+1}\big)$ is a system of bases adapted to $\sh{F}$ on $R_{j+1}$.
\end{itemize}
\end{theorem}
\begin{proof}
To begin, let $R_{1}$ be the open vertical strap containing $\sigma_{i_{1}}$, the first crossing of $\beta$, and let $z_{1}$ be the first entry of $\vec{z}$, which parametrizes the $s_{i_{1}}$-relative position between the complete flags $\fl{F}_{1}$ and $\fl{F}_{2}$ in $\mathbb{K}^{n}$ that geometrically characterize $\sh{F}$ on $R_{1}$. By assumption, $\big\{\hat{\mathbf{f}}^{(i)} \big\}_{i=1}^{n}$ is a system of bases adapted to the $n-1$ injective linear maps defining $\sh{F}$ on the intersection $U_{\mathrm{L}}\,\cap\, R_{1}$, namely $\big\{\phi^{(i)}_{\sh{F}}\big\}_{i=1}^{n-1}$. Then, depending on whether $i_{1}=1$ or $i_{1}\geq 2$, we can apply Lemma~\eqref{Lemma: matrix local model for sigma_1} or Lemma~\eqref{Lemma: matrix local model for sigma_k} to deduce that $\big(\big\{\hat{\mathbf{f}}^{(i)} \big\}_{i=1}^{n}, z_{1}\big)$ is a system of bases adapted to $\sh{F}$ on $R_{1}$.

Now, consider $R_{2}$, the open vertical strap containing $\sigma_{i_{2}}$, the second crossing of $\beta$, and let $z_{2}$ be the second entry of $\vec{z}$, which parametrizes the $s_{i_{2}}$-relative position between the complete flags $\fl{F}_{2}$ and $\fl{F}_{3}$ in $\mathbb{K}^{n}$ that geometrically characterize $\sh{F}$ on $R_{2}$. In addition, let $\big\{\hat{\mathbf{f}}^{(i)}[\beta_{1},\vec{z}_{1}] \big\}_{i=1}^{n}$ be the braid-transformed bases obtained from $\big\{\hat{\mathbf{f}}^{(i)} \big\}_{i=1}^{n}$ via the truncated braid word $B_{1}=\sigma_{i_{1}}$ and the truncated tuple $\vec{z}_{1}=z_{1}$ (cf. Definition~\eqref{Def: braid transformation of bases}). By construction, $\big\{\hat{\mathbf{f}}^{(i)}[\beta_{1},\vec{z}_{1}] \big\}_{i=1}^{n}$ is a system of bases adapted to the $n-1$ injective linear maps defining $\sh{F}$ on the intersection $R_{1} \,\cap\, R_{2}$. Then, depending on whether $i_{2}=1$ or $i_{2}\geq 2$, we can apply Lemma~\eqref{Lemma: matrix local model for sigma_1} or Lemma~\eqref{Lemma: matrix local model for sigma_k} to obtain that $\big(\big\{\hat{\mathbf{f}}^{(i)}[\beta_{1},\vec{z}_{1}] \big\}_{i=1}^{n}, z_{2} \big)$ is a system of bases adapted to $\sh{F}$ on $R_{2}$.  

Finally, for each $j\in[1,\ell]$, let $\big\{\hat{\mathbf{f}}^{(i)}[\beta_{j},\vec{z}_{j} ]\big\}_{i=1}^{n}$ be the braid-transformed bases obtained from $\big\{\hat{\mathbf{f}}^{(i)} \big\}_{i=1}^{n}$ via the truncated braid word $B_{j}=\sigma_{i_{1}}\cdots \sigma_{i_{j}}\in \mathrm{Br}^{+}_{n}$ and the truncated tuple $\vec{z}_{j}=(z_{1},\dots, z_{j})\in \mathbb{K}^{j}_{\mathrm{std}}$. Then, iterating the above argument over all the straps $R_{j}$ shows that, for each $j\in[1,\ell-1]$, the pair $\big(\big\{\hat{\mathbf{f}}^{(i)}[\beta_{j},\vec{z}_{j} ]\big\}_{i=1}^{n}, z_{j+1}\big)$ is a system of bases adapted to $\sh{F}$ on $R_{j+1}$, as desired.
\end{proof}

\begin{remark}
Let $\beta=\sigma_{i_{1}}\cdots \sigma_{i_{\ell}}\in \mathrm{Br}^{+}_{n}$ be a positive braid word, $\sh{F}$ an object of the category $\ccs{1}{\beta}$, and $\mathcal{R}_{\Lambda(\beta)}=\big\{R_{j}\big\}_{j=1}^{\ell}$ the collection of $\ell$ open vertical straps in $\mathbb{R}^{2}$ introduced in Construction~\eqref{Cons: Definition of the vertical straps}. Let $\,\hat{\mathbf{f}}^{(n)}$ be a basis for $\mathbb{K}^{n}$, and let $\vec{z}=(z_{1},\dots, z_{\ell})\in X(\beta,\mathbb{K})$ be a point such that the pair $(\,\hat{\mathbf{f}}^{(n)}, \,\vec{z}\, )$ algebraically characterizes $\sh{F}$ according to Theorem~\eqref{Prop. for sheaves and braid matrices}. Then, Theorem~\eqref{Theorem: adapted bases for an object at a region R_j} illustrates how the algebraic data $(\,\hat{\mathbf{f}}^{(n)},\,\vec{z}\, )$ locally determines the behavior of $\sh{F}$ on each vertical strap $R_{j}$ via the notion of braid-transformed bases, which will play a key role in the computations and discussions ahead.  
\end{remark}

Having completed the algebraic characterization of the objects of the category $\ccs{1}{\beta}$, we now turn to the study of its graded morphism spaces and their compositions, which will further elucidate the rich algebraic, geometric, and combinatorial structures of the categorical invariant under consideration.

%% file: sec5.tex
\section{\texorpdfstring{The Morphism Spaces and Their Compositions in the Category $\ccs{1}{\beta}$}{The Morphisms and Their Compositions in the Category HS}}\label{sec:morphisms and compositions}
\noindent
Let $\beta\in \mathrm{Br}_{n}^{+}$ be a positive braid word. In this section, we present an explicit description of the graded morphism spaces and their compositions in the category $\ccs{1}{\beta}$. To lay the groundwork, we begin by reviewing some key results due to Chantraine, Ng, and Sivek~\cite{CNS1}.

\subsection{Higher-Degree Morphism Spaces}
Let $\beta\in \mathrm{Br}_{n}^{+}$ be a positive braid word. In this subsection, we analyze the structure of the higher-degree morphism spaces in the category $\ccs{1}{\beta}$. As a first step, we recall the following fundamental result of Chantraine, Ng, and Sivek~\cite{CNS1}.

\begin{proposition}[[Propositions 6.5 \& 6.6,~\cite{CNS1}]\label{Prop: CNS results on derived and internal hom sheaves}
Let $\beta\in \mathrm{Br}_{n}^{+}$ be a positive braid word. Let $r,s\geq 1$ be integers, and let $\sh{F}$ and $\sh{G}$ be objects of the categories $\ccs{r}{\beta}$ and $\ccs{s}{\beta}$, respectively. Then, the following statements hold:
\begin{itemize}
\justifying
\item[\textit{a)}] For all $p\geq 2$, 
\begin{equation*}
H^{p}\Big(R\Gamma\big(\mathbb{R}^{2};\sh{H}om(\sh{F},\sh{G}) \big)\Big)=0\,. 
\end{equation*}
\item[\textit{b)}] For all $q\geq 1$, 
\begin{equation*}
H^{q}\big(R\sh{H}om(\sh{F},\sh{G})\big)=0\,.        
\end{equation*}
In other words, $R\sh{H}om(\sh{F},\sh{G}) \backsimeq \sh{H}om(\sh{F},\sh{G})$.
\end{itemize}
\end{proposition}

Next, let $\sh{F}$ and $\sh{G}$ be objects of the category $\ccs{1}{\beta}$. The local--to--global $\mathrm{Ext}$ spectral sequence establishes that 
\begin{equation*}
E^{p,q}_{2}=H^{p}\big(R\Gamma(\mathbb{R}^{2};H^{q}(R\sh{H}om(\sh{F},\sh{G})))\big) \Longrightarrow \mathrm{Ext}^{p+q}(\sh{F},\sh{G})\, ,
\end{equation*}
for all $p,q \geq 0$. By Proposition~\eqref{Prop: CNS results on derived and internal hom sheaves}-(\textit{b}), we have that $H^{q}\big(R\sh{H}om(\sh{F},\sh{G})\big)=0$ for all $q\geq 1$, and therefore $E_{2}^{p,q}=0$ whenever $q\geq 1$. It follows that the $E_{2}$-page is supported entirely along the axis $q=0$, and as a result, the spectral sequence collapses at this page. Consequently, we obtain that
\begin{equation*}
\mathrm{Ext}^{p}(\sh{F},\sh{G})=E_{2}^{p,0}= H^{p}\Big(R\Gamma\big(\mathbb{R}^{2};H^{0}(R\sh{H}om(\sh{F},\sh{G}))\big)\Big)=H^{p}\Big( R\Gamma\big(\mathbb{R}^{2};\sh{H}om(\sh{F},\sh{G})\big)\Big)\, ,  
\end{equation*}
for all $p\geq 0$~\cite{CNS1}. Proposition~\eqref{Prop: CNS results on derived and internal hom sheaves}-(\textit{a}) then implies that
\begin{equation*}
\mathrm{Ext}^{p}(\sh{F},\sh{G})=0\, ,
\end{equation*}
for all $p\geq2$. In particular, this shows that the higher-degree morphism spaces in the category $\ccs{1}{\beta}$ are trivial.    

Having established the vanishing of the higher-degree morphism spaces in category $\ccs{1}{\beta}$, we now focus on analyzing the structure of its lower-degree morphism spaces.

\subsection{Lower-Degree Morphism Spaces} Let $\beta\in\mathrm{Br}^{+}_{n}$ be a positive braid word. In this subsection, our goal is to provide an explicit and computable description of the lower-degree morphism spaces in the category $\ccs{1}{\beta}$. To this end, we begin by introducing some notation. 

\begin{notation}
Let $n\geq 1$ be an integer. For any tuple $\vec{u}=(u_{1},\dots, u_{n})\in\mathbb{K}^{n}_{\mathrm{std}}$, we denote by $D(\vec{u})\in \mathrm{M}(n,\mathbb{K})$ the $n\times n$ diagonal matrix given by
\begin{equation}\notag
 D(\vec{u}):=\begin{bmatrix}
     u_{1}& & \\
     &\ddots &\\
     & & u_{n}
\end{bmatrix}\, .
\end{equation}    
\end{notation}

\begin{notation}
Let $n\geq 2$ be an integer, and let $\beta=\sigma_{i_{1}}\cdots\sigma_{i_{\ell}}\in\mathrm{Br}^{+}_{n}$ be a positive braid word. For each $j\in[0,\ell]$, we denote by $\beta_{j}:=\sigma_{i_{1}}\dots\sigma_{i_{j}}\in\mathrm{Br}^{+}_{n}$ the truncation of $\beta$ at the $j$-th crossing, with the convention $\beta_{0} := e_{n}$, where $e_{n}\in \mathrm{Br}^{+}_{\mathrm{std}}$ is the trivial braid on $n$ strands. Bearing this in mind, we adopt the following notation: 
\begin{itemize}
\item For each $j\in[0,\ell]$, we introduce $\pi_{\beta_{j}}:=s_{i_{1}}\cdots s_{i_{j}}\in \mathrm{S}_{n}$ to denote the permutation associated with $\beta_{j}$, with the convention $\pi_{\beta_{0}}:=e_{n}$, where by slight abuse of notation $e_{n}\in \mathrm{S}_{n}$ denotes within this context the trivial permutation on $n$ elements.  

\item For any tuple $\vec{u}=(u_{1},\dots, u_{n})\in\mathbb{K}^{n}_{\mathrm{std}}$, we set $\pi_{\beta_{j}}(\vec{u}):=(u_{\pi_{\beta_{j}}(1)},\dots, u_{\pi_{\beta_{j}}(n)})\in\mathbb{K}^{n}_{\mathrm{std}}$ for each $j\in[0,\ell]$, with the convention $\pi_{\beta_{0}}(\vec{u}):=\vec{u}$. 
\end{itemize}
\end{notation}
\noindent
With the preceding notation in place, we now introduce a linear map that will be central to the explicit description of the lower-degree morphism spaces in the category $\ccs{1}{\beta}$.

\begin{definition}\label{Def: linear map delta}
Let $\beta=\sigma_{i_{1}}\cdots\sigma_{i_{\ell}}\in\mathrm{Br}^{+}_{n}$ be a positive braid word, and let $\sh{F}$, $\sh{G}$ be objects of the category $\ccs{1}{\beta}$. Let \,$\mathbf{\hat{f}}^{(n)},\,\mathbf{\hat{g}}^{(n)}$ be bases for \,$\mathbb{K}^{n}$, and let \,$\vec{z}=(z_{1},\dots, z_{\ell}),\,\vec{z}\,'=(z'_{1},\dots, z'_{\ell})\in X(\beta,\mathbb{K})$ be points such that the pairs $\big(\,\mathbf{\hat{f}}^{(n)},\, \vec{z}\, \big)$ and $\big(\,\mathbf{\hat{g}}^{(n)},\, \vec{z}\,'\, \big)$ algebraically characterize $\sh{F}$ and $\sh{G}$ according to Theorem~\eqref{Prop. for sheaves and braid matrices}, respectively. Then, we associate to the pair $(\sh{F},\sh{G})$ the linear \,$\delta_{\sh{F},\sh{G}}:\mathbb{K}^{n}_{\mathrm{std}}\to \mathbb{K}^{\ell}_{\mathrm{std}}$\, defined by
\begin{equation*}
\delta_{\sh{F},\sh{G}}(\vec{u}):=\big(\delta_{1}(\vec{u}),\dots, \delta_{\ell}(\vec{u})\big)\, , \qquad  \vec{u}=(u_{1},\dots, u_{n})\in \mathbb{K}^{n}_{\mathrm{std}}\, ,
\end{equation*}
where, for each $j\in [1,\ell]$,
\begin{equation*}
 \delta_{j}(\vec{u}):=\Big[\,(B^{(n)}_{i_{j}}(z'_{j}))^{-1}D\big( \pi_{\beta_{j-1}}(\vec{u}) \big)\,B^{(n)}_{i_{j}}(z_{j})\,\Big]_{i_{j}+1, i_{j}} \, .  
\end{equation*}
\end{definition}

Building on the previous definition, we now state one of the main results of this manuscript: a theorem providing an explicit and computable characterization of the lower-degree morphism spaces in the category $\ccs{1}{\beta}$. 

\begin{theorem}\label{Theorem: Lower-degree morphisms}
Let $\beta=\sigma_{i_{1}}\cdots\sigma_{i_{\ell}}\in\mathrm{Br}^{+}_{n}$ be a positive braid word, and let $\sh{F}$, $\sh{G}$ be objects of the category $\ccs{1}{\beta}$. Let \,$\mathbf{\hat{f}}^{(n)},\,\mathbf{\hat{g}}^{(n)}$ be bases for \,$\mathbb{K}^{n}$, and let \,$\vec{z},\, \vec{z}\,'\in X(\beta,\mathbb{K})$ be points such that the pairs $\big(\,\mathbf{\hat{f}}^{(n)},\,\vec{z}\, \big)$ and $\big(\,\mathbf{\hat{g}}^{(n)},\ \vec{z}\,'\, \big)$ algebraically characterize $\sh{F}$ and $\sh{G}$ according to Theorem~\eqref{Theorem: sheaves as points in the braid variety}, respectively.

Following Definition~\eqref{Def: linear map delta}, let \,$\delta_{\sh{F},\sh{G}}:\mathbb{K}^{n}_{\mathrm{std}}\to \mathbb{K}^{\ell}_{\mathrm{std}}$\, be the linear map associated with the pair $(\sh{F},\sh{G})$. Then there are isomorphisms of vector spaces
\begin{align}
\mathrm{Ext}^{0}(\sh{F},\sh{G}) &\cong \mathrm{ker}\, \delta_{\sh{F},\sh{G}}\, , \label{Eq: iso for Ext0} \\
\mathrm{Ext}^{1}(\sh{F},\sh{G}) &\cong \mathrm{coker}\, \delta_{\sh{F},\sh{G}}\, . \label{Eq: iso for Ext1}
\end{align}
\end{theorem}

In particular, the following subsections are devoted to proving the preceding statement. Bearing this in mind, we now turn to the study of the zero-degree morphism spaces in the category $\ccs{1}{\beta}$. 

\subsection{\texorpdfstring{Zero-Degree Morphism Spaces in the Category $\ccs{1}{\beta}$}{Zero-degree Morphism Spaces in the Category HSh}} 
\noindent
Let $\beta \in \mathrm{Br}^{+}_{n}$ be a positive braid word. The main goal of this subsection is to establish the first part of Theorem~\eqref{Theorem: Lower-degree morphisms}, namely, the isomorphism of vector spaces in Equation~\eqref{Eq: iso for Ext0}, thereby providing an explicit algebraic characterization of the zero-degree morphism spaces in the category $\ccs{1}{\beta}$. To this end, we begin by introducing some preliminaries. 

\subsubsection{Technical Background} Let $\beta \in \mathrm{Br}^{+}_{n}$ be a positive braid word. We now collect some preliminaries that will play a fundamental role in the explicit computation of the zero-degree morphism spaces in the category $\ccs{1}{\beta}$. In particular, we open the discussion with the following definition.

\begin{definition}\label{Def: morphism between linear maps}
Let $A$, $B$, $X$, $Y$ be vector spaces over $\mathbb{K}$, and let $\gamma_{\sh{F}}:A\to B$ and $\gamma_{\sh{G}}:X\to Y$ be linear maps. A morphism $\lambda_{\xi,\delta}:=(\lambda_{\xi}, \lambda_{\delta})$ between $\gamma_{\sh{F}}$ and $\gamma_{\sh{G}}$ consists of a pair of linear maps $\lambda_{\xi}:A\to X$ and $\lambda_{\delta}:B\to Y$ such that the diagram in Figure~\eqref{Commutative diagram for morphisms of linear maps} commutes. Equivalently,
\begin{equation*}
\lambda_{\delta}\circ \gamma_{\sh{F}}=\gamma_{\sh{G}}\circ \lambda_{\xi}\, .     
\end{equation*}
\begin{figure}[ht]
\centering
\begin{tikzpicture}
\useasboundingbox (-1.5,-2) rectangle (1.5,2);
\scope[transform canvas={scale=1.25}]

\node at (-1,-1) {\footnotesize $A$};
\node at (-1, 1) {\footnotesize $B$};

\node at (1, -1) {\footnotesize $X$};
\node at (1, 1) {\footnotesize $Y$};

\draw[->] (-1,-1+0.25) -- (-1,1-0.25);
\draw[->] (1,-1+0.25) -- (1,1-0.25);

\draw[->] (-1+0.25,1) -- (1-0.25,1);
\draw[->] (-1+0.25,-1) -- (1-0.25,-1);

\node[left] at (-1,0) {\footnotesize $\gamma_{\sh{F}}$};

\node[right] at (1,0) {\footnotesize $\gamma_{\sh{G}}$};

\node at (0,1.4) {\footnotesize $\lambda_{\delta}$};

\node at (0,-1.4) {\footnotesize $\lambda_{\xi}$};
\endscope
\end{tikzpicture}
\caption{A morphism $\lambda_{\xi,\delta}$ between $\gamma_{\sh{F}}$ by $\gamma_{\sh{G}}$.}
\label{Commutative diagram for morphisms of linear maps}
\end{figure}
\end{definition}

\begin{lemma}\label{Lemma: Ext0 for surjective linear maps}
Let $\big\{\psi_{\sh{F}}^{(i)}:\mathbb{K}^{i+1}\to \mathbb{K}^{i}\big\}_{i=1}^{n-1}$ and $\big\{\psi_{\sh{G}}^{(i)}:\mathbb{K}^{i+1}\to \mathbb{K}^{i}\big\}_{i=1}^{n-1}$ be two collections of surjective linear maps, and let $\big\{ T_{\lambda}^{(i)}:\mathbb{K}^{i}\to \mathbb{K}^{i}\big\}_{i=1}^{n}$ be a collection of linear maps such that for each $i\in[1,n-1]$, the pair $(T^{(i)}_{\lambda},T^{(i+1)}_{\lambda})$ defines a morphism between $\psi^{(i)}_{\sh{F}}$ and $\psi^{(i)}_{\sh{G}}$, i.e.,
\begin{equation}\label{Eq: identity I for Ext0 on top strands}
T^{(i)}_{\lambda}\circ \psi_{\sh{F}}^{(i)}=\psi^{(i)}_{\sh{G}}\circ T_{\lambda}^{(i+1)}\, .
\end{equation}

For each $i\in[1,n]$, let $\hat{\mathbf{f}}^{(i)}:=\big\{\hat{f}^{(i)}_{j}\big\}_{j=1}^{i}$ and $\hat{\mathbf{g}}^{(i)}:=\big\{\hat{g}^{(i)}_{j}\big\}_{j=1}^{i}$ be bases for $\mathbb{K}^{i}$ such that the collections $\big\{\hat{\mathbf{f}}^{(j)}\big\}_{j=1}^{n}$ and $\big\{\hat{\mathbf{g}}^{(j)}\big\}_{j=1}^{n}$ are systems of bases adapted to $\big\{\psi_{\sh{F}}^{(i)}\big\}_{i=1}^{n-1}$ and $\big\{\psi_{\sh{G}}^{(i)}\big\}_{i=1}^{n-1}$, respectively (cf. Definition~\eqref{Def:flags and adapted bases}--\eqref{Def: adapted bases II}). Then:
\begin{itemize}
\justifying
\item The matrix $\tensor[_{\hat{\mathbf{g}}^{(n)}}]{ \big[\, T_{\lambda}^{\,(n)}\,\big] }{_{\hat{\mathbf{f}}^{(n)}}}\in M(n,\mathbb{K})$ representing $T^{(n)}_{\lambda}$ is lower triangular. 
\item For each $i\in[1,n-1]$, the matrix $\tensor[_{\hat{\mathbf{g}}^{(i)}}]{ \big[\, T_{\lambda}^{\,(i)}\,\big] }{_{\hat{\mathbf{f}}^{(i)}}}\in M(i,\mathbb{K})$ representing $T^{(i)}_{\lambda}$ coincides with the principal $i\times i$ submatrix of $\tensor[_{\hat{\mathbf{g}}^{(n)}}]{ \big[\, T_{\lambda}^{\,(n)}\,\big] }{_{\hat{\mathbf{f}}^{(n)}}}$. 
\end{itemize} 
\end{lemma}
\begin{proof}
By assumption, $\big\{\hat{\mathbf{f}}^{(i)}\big\}_{i=1}^{n}$ and $\big\{\hat{\mathbf{g}}^{(i)}\big\}_{i=1}^{n}$ are systems of bases adapted to $\big\{\psi_{\sh{F}}^{(i)}\big\}_{i=1}^{n-1}$ and $\big\{\psi_{\sh{G}}^{(i)}\big\}_{i=1}^{n-1}$, respectively. Then, Definition~\eqref{Def:flags and adapted bases}--\eqref{Def: adapted bases II} asserts that for each $i\in[1,n-1]$, the matrices $\tensor[_{\hat{\mathbf{f}}^{(i)}}]{ \big[\, \psi_{\sh{F}}^{(i)}\,\big] }{_{\hat{\mathbf{f}}^{(i+1)}}}\in M(i,i+1,\mathbb{K})$ and $\tensor[_{\hat{\mathbf{g}}^{(i)}}]{ \big[\, \psi_{\sh{G}}^{(i)}\,\big] }{_{\hat{\mathbf{g}}^{(i+1)}}}\in M(i,i+1,\mathbb{K})$ representing $\psi_{\sh{F}}^{(i)}$ and $\psi_{\sh{G}}^{(i)}$ are the standard projection matrices:
\begin{equation*}
\tensor[_{\hat{\mathbf{f}}^{(i)}}]{ \big[\, \psi_{\sh{F}}^{(i)}\,\big] }{_{\hat{\mathbf{f}}^{(i+1)}}}=\pi^{(i,i+1)}\,,  \quad \text{and} \quad \tensor[_{\hat{\mathbf{g}}^{(i)}}]{ \big[\, \psi_{\sh{G}}^{(i)}\,\big] }{_{\hat{\mathbf{g}}^{(i+1)}}}=\pi^{(i,i+1)}\, .    
\end{equation*}

Now, for each $i\in[1,n]$, denote by $\tensor[_{\hat{\mathbf{g}}^{(i)}}]{ \big[\, T_{\lambda}^{\,(i)}\,\big] }{_{\hat{\mathbf{f}}^{(i)}}}\in M(i,\mathbb{K})$ the matrix representing $T^{(i)}_{\lambda}$. Then, for each $i\in[1,n-1]$, the morphism condition~\eqref{Eq: identity I for Ext0 on top strands} translates into the matrix equation
\begin{equation*}
\tensor[_{\hat{\mathbf{g}}^{(i)}}]{ \big[\, T_{\lambda}^{(i)}\,\big] }{_{\hat{\mathbf{f}}^{(i)}}}\cdot  \tensor[_{\hat{\mathbf{f}}^{(i)}}]{ \big[\, \psi_{\sh{F}}^{(i)}\,\big] }{_{\hat{\mathbf{f}}^{(i+1)}}}=\tensor[_{\hat{\mathbf{g}}^{(i)}}]{ \big[\, \psi_{\sh{G}}^{\,(i)}\,\big] }{_{\hat{\mathbf{g}}^{(i+1)}}}\cdot \tensor[_{\hat{\mathbf{g}}^{(i+1)}}]{ \big[\, T_{\lambda}^{(i+1)}\,\big] }{_{\hat{\mathbf{f}}^{(i+1)}}}\, .
\end{equation*}
Consequently, a direct substitution of the standard projection matrices into the above relation shows that
\begin{equation*}
\left[
\begin{array}{c|c}
~\tensor[_{\hat{\mathbf{g}}^{(i)}}]{ \big[\, T_{\lambda}^{(i)}\,\big] }{_{\hat{\mathbf{f}}^{(i)}}}~ & ~ \mathbf{0}_{i\times 1}~
\end{array}\right]_{i\times (i+1)} =  \parbox{12em}{\center{the upper-left\\$(i,i+1)$ submatrix of}} \tensor[_{\hat{\mathbf{g}}^{(i+1)}}]{ \big[\, T_{\lambda}^{(i+1)}\,\big] }{_{\hat{\mathbf{f}}^{(i+1)}}}\, ,    
\end{equation*}
for each $i\in[1,n-1]$. Hence, by applying the above identity recursively, we conclude that the matrix $\tensor[_{\hat{\mathbf{g}}^{(n)}}]{ \big[\, T_{\lambda}^{(n)}\,\big] }{_{\hat{\mathbf{f}}^{(n)}}}$ representing $T^{(n)}_{\lambda}$ is lower triangular, and that for each $i\in[1,n-1]$, the matrix $\tensor[_{\hat{\mathbf{g}}^{(i)}}]{ \big[\, T_{\lambda}^{(i)}\,\big] }{_{\hat{\mathbf{f}}^{(i)}}}$ representing $T^{(i)}_{\lambda}$ coincides with the principal $i\times i$ submatrix of $\tensor[_{\hat{\mathbf{g}}^{(n)}}]{ \big[\, T_{\lambda}^{(n)}\,\big] }{_{\hat{\mathbf{f}}^{(n)}}}$, as claimed.  
\end{proof}

\begin{lemma}\label{Lemma: Ext0 for injective linear maps}
Let $\big\{\phi_{\sh{F}}^{(i)}:\mathbb{K}^{i}\to \mathbb{K}^{i+1}\big\}_{i=1}^{n-1}$ and $\big\{\phi_{\sh{G}}^{(i)}:\mathbb{K}^{i}\to \mathbb{K}^{i+1}\big\}_{i=1}^{n-1}$ be two collections of injective linear maps, and let $\big\{ B_{\lambda}^{(i)}:\mathbb{K}^{i}\to \mathbb{K}^{i}\big\}_{i=1}^{n}$ be a collection of linear maps such that for each $i\in[1,n-1]$, the pair $(B^{(i+1)}_{\lambda}, B^{(i)}_{\lambda})$ defines a morphism between $\phi^{(i)}_{\sh{F}}$ and $\phi^{(i)}_{\sh{G}}$, i.e., 
\begin{equation}\label{Eq: identity II for Ext0 on bottom strands}
B_{\lambda}^{(i+1)}\circ \phi_{\sh{F}}^{(i)}=\phi^{(i)}_{\sh{G}}\circ B_{\lambda}^{(i)}\, .   
\end{equation}

For each $i\in[1,n]$, let $\hat{\mathbf{f}}^{(i)}:=\big\{\hat{f}^{(i)}_{j}\big\}_{j=1}^{i}$ and $\hat{\mathbf{g}}^{(i)}:=\big\{\hat{g}^{(i)}_{j}\big\}_{j=1}^{i}$ be bases for $\mathbb{K}^{i}$ such that the collections $\big\{\hat{\mathbf{f}}^{(i)}\big\}_{i=1}^{n}$ and $\big\{\hat{\mathbf{g}}^{(i)}\big\}_{i=1}^{n}$ are systems of bases adapted to $\big\{\phi_{\sh{F}}^{(i)}\big\}_{i=1}^{n-1}$ and $\big\{\phi_{\sh{G}}^{(i)}\big\}_{i=1}^{n-1}$, respectively (cf. Definition~\eqref{Def:flags and adapted bases}--\eqref{Def: adapted bases I}). Then:
\begin{itemize}
\justifying
\item The matrix $\tensor[_{\hat{\mathbf{g}}^{(n)}}]{ \big[\, B_{\lambda}^{(n)}\,\big] }{_{\hat{\mathbf{f}}^{(n)}}}\in M(n,\mathbb{K})$ representing $B^{(n)}_{\lambda}$ is upper triangular. 
\item For each $i\in[1,n-1]$, the matrix $\tensor[_{\hat{\mathbf{g}}^{(i)}}]{ \big[\, B_{\lambda}^{(i)}\,\big] }{_{\hat{\mathbf{f}}^{(i)}}}\in M(i,\mathbb{K})$ representing $B^{(i)}_{\lambda}$ coincides with the principal $i\times i$ submatrix of $\tensor[_{\hat{\mathbf{g}}^{(n)}}]{ \big[\, B_{\lambda}^{(n)}\,\big] }{_{\hat{\mathbf{f}}^{(n)}}}$. 
\end{itemize} 
\end{lemma}
\begin{proof}
By assumption, $\big\{\hat{\mathbf{f}}^{(i)}\big\}_{i=1}^{n}$ and $\big\{\hat{\mathbf{g}}^{(i)}\big\}_{i=1}^{n}$ are systems of bases adapted to $\big\{\phi_{\sh{F}}^{(i)}\big\}_{i=1}^{n-1}$ and $\big\{\phi_{\sh{G}}^{(i)}\big\}_{i=1}^{n-1}$, respectively. Then, Definition~\eqref{Def:flags and adapted bases}--\eqref{Def: adapted bases I} asserts that, for each $i\in[1,n-1]$, the matrices $\tensor[_{\hat{\mathbf{f}}^{(i+1)}}]{ \big[\, \phi_{\sh{F}}^{(i)}\,\big] }{_{\hat{\mathbf{f}}^{(i)}}}\in M(i+1,i,\mathbb{K})$ and $\tensor[_{\hat{\mathbf{g}}^{(i+1)}}]{ \big[\, \phi_{\sh{G}}^{(i)}\,\big] }{_{\hat{\mathbf{g}}^{(i)}}}\in M(i+1,i,\mathbb{K})$ representing $\phi_{\sh{F}}^{(i)}$ and $\phi_{\sh{G}}^{(i)}$ are the standard inclusion matrices: 
\begin{equation*}
\tensor[_{\hat{\mathbf{f}}^{(i+1)}}]{ \big[\, \phi_{\sh{F}}^{(i)}\,\big] }{_{\hat{\mathbf{f}}^{(i)}}}=\iota^{(i+1,i)}\, , \quad \text{and} \quad \tensor[_{\hat{\mathbf{g}}^{(i+1)}}]{ \big[\, \phi_{\sh{G}}^{(i)}\,\big] }{_{\hat{\mathbf{g}}^{(i)}}}= \iota^{(i+1,i)}\, .    
\end{equation*}

Now, for each $i\in[1,n]$, denote by $\tensor[_{\hat{\mathbf{g}}^{(i)}}]{ \big[\, B_{\lambda}^{(i)}\,\big] }{_{\hat{\mathbf{f}}^{(i)}}}\in M(i,\mathbb{K})$ the matrix representing $B^{(i)}_{\lambda}$. Then, for each $i\in[1,n-1]$, the morphism condition~\eqref{Eq: identity II for Ext0 on bottom strands} translates into the matrix equation
\begin{equation*}
\tensor[_{\hat{\mathbf{g}}^{(i+1)}}]{ \big[\, B_{\lambda}^{(i+1)}\,\big] }{_{\hat{\mathbf{f}}^{(i+1)}}}\cdot  \tensor[_{\hat{\mathbf{f}}^{(i+1)}}]{ \big[\, \phi_{\sh{F}}^{(i)}\,\big] }{_{\hat{\mathbf{f}}^{(i)}}}=\tensor[_{\hat{\mathbf{g}}^{(i+1)}}]{ \big[\, \phi_{\sh{G}}^{(i)}\,\big] }{_{\hat{\mathbf{g}}^{(i)}}}\cdot \tensor[_{\hat{\mathbf{g}}^{(i)}}]{ \big[\, B_{\lambda}^{(i)}\,\big] }{_{\hat{\mathbf{f}}^{(i)}}}\, .
\end{equation*}
Consequently, a direct substitution of the standard inclusion matrices into the above relation shows that
\begin{equation*}
\parbox{12em}{\center{ the upper-left\\$(i+1,i)$ submatrix of}}  \tensor[_{\hat{\mathbf{g}}^{(i+1)}}]{ \big[\, B_{\lambda}^{(i+1)}\,\big] }{_{\hat{\mathbf{f}}^{(i+1)}}}  = \left[\begin{array}{c}
~ \tensor[_{\hat{\mathbf{g}}^{(i)}}]{ \big[\, B_{\lambda}^{(i)}\,\big] }{_{\hat{\mathbf{f}}^{(i)}}} ~ \vspace{3pt}\\
\hline
\mathbf{0}_{1\times i}
\end{array}\right]_{(i+1)\times i}\, ,
\end{equation*}
for each $i\in[1,n-1]$. Hence, by applying the above identity recursively, we conclude that the matrix $\tensor[_{\hat{\mathbf{g}}^{(n)}}]{ \big[\, B_{\lambda}^{(n)}\,\big] }{_{\hat{\mathbf{f}}^{(n)}}}$ representing $B^{(n)}_{\lambda}$ is upper triangular, and that for each $i\in[1,n-1]$, the matrix $\tensor[_{\hat{\mathbf{g}}^{(i)}}]{ \big[\, B_{\lambda}^{(i)}\,\big] }{_{\hat{\mathbf{f}}^{(i)}}}$ representing $B^{(i)}_{\lambda}$ coincides with the principal $i\times i$ submatrix of $\tensor[_{\hat{\mathbf{g}}^{(n)}}]{ \big[\, B_{\lambda}^{(n)}\,\big] }{_{\hat{\mathbf{f}}^{(n)}}}$, as claimed.
\end{proof}

Having established the necessary groundwork, we conclude this part of the manuscript with the following observation, which sheds light on the local structure of the zero-degree morphism spaces in the category $\ccs{1}{\beta}$. 

\begin{observation}\label{Obs: elements of Ext0 near arcs}
Let $\beta=\sigma_{i_{1}}\cdots \sigma_{i_{\ell}}\in \mathrm{Br}^{+}_{n}$ be a positive braid word, $\mathcal{S}_{\Lambda(\beta)}$ the stratification of $\mathbb{R}^{2}$ induced by $\Lambda(\beta)\subset (\mathbb{R}^{3},\xi_{\mathrm{std}})$, and $\sh{F}$ and $\sh{G}$ objects of the category $\ccs{1}{\beta}$. By the microlocal support conditions, the global structure of $\sh{F}$ and $\sh{G}$ is determined by their local behavior along the arcs in $\mathcal{S}_{\Lambda(\beta)}$, together with compatibility conditions imposed at cusps and crossings. Now, recall that elements of $\,\mathrm{Ext}^{0}(\sh{F},\sh{G})$ are global sheaf homomorphisms between $\sh{F}$ and $\sh{G}$. Hence, the constructibility of $\sh{F}$ and $\sh{G}$ with respect to $\mathcal{S}_{\Lambda(\beta)}$, together with the sheaf axioms, implies that the global structure of the elements of $\mathrm{Ext}^{0}(\sh{F},\sh{G})$ is likewise determined by their local behavior along the arcs in $\mathcal{S}_{\Lambda(\beta)}$, together with compatibility conditions required at cusps and crossings. 

Next, we explicitly describe the structure of the elements of $\,\mathrm{Ext}^{0}(\sh{F},\sh{G})$ near an arc $a$ in $\mathcal{S}_{\Lambda(\beta)}$. To this end, observe that in a neighborhood of $a$, the stratification $\mathcal{S}_{\Lambda(\beta)}$ consists of an upper $2$-dimensional stratum $U$ and a lower $2$-dimensional stratum $D$, as illustrated in Sub-figure~\eqref{sub-fig: Strata near an arc}. Given this local configuration, let $p\in U$ and $q\in D$ be arbitrary points, and denote by $A=\sh{F}_{p}$, $B=\sh{F}_{q}$, $X=\sh{G}_{p}$, and $Y=\sh{G}_{q}$ the stalks of $\sh{F}$ and $\sh{G}$ at these points, respectively. It then follows from the microlocal support conditions near the arcs that, near $a$, $\sh{F}$ and $\sh{G}$ are determined by a pair of linear maps $\gamma_{\sh{F}}:A\to B$ and $\gamma_{\sh{G}}:X\to Y$, respectively. According to our previous discussion, we have that a sheaf homomorphism $\lambda \in \mathrm{Ext}^{0}(\sh{F}, \sh{G})$ between $\sh{F}$ and $\sh{G}$ is specified near $a$ by two linear maps $\lambda_{\xi}: A\to X$ and $\lambda_{\delta}:B \to Y$ such that
\begin{equation*}
\lambda_{\delta}\circ \gamma_{\sh{F}}=\gamma_{\sh{G}}\circ \lambda_{\xi}\, .     
\end{equation*}
In other words, near $a$, $\lambda$ is determined by a morphism $\lambda_{\xi,\delta}:=(\lambda_{\xi}, \lambda_{\delta})$ between $\gamma_{\sh{F}}$ and $\gamma_{\sh{G}}$ (cf. Definition~\eqref{Def: morphism between linear maps}).
\end{observation}

Building on the preceding local description of the elements of the zero-degree morphism spaces in the category $\ccs{1}{\beta}$, we proceed to analyze the global structure of these elements, yielding an explicit characterization of the zero-degree morphism spaces under consideration. In particular, we begin our analysis by considering the case $\beta=e_{n}$.

\subsubsection{The Case of the Trivial Braid} Let $e_{n}\in\mathrm{Br}^{+}_{n}$ be the trivial braid word. Next, we study the zero-degree morphism spaces in the category $\ccs{1}{e_{n}}$. 

To begin, consider the open cover $\mathcal{U}_{\Lambda(e_{n})}=\big\{ U_{0}, U_{\mathrm{B}}, U_{\mathrm{T}}\big\}$ of $\mathbb{R}^{2}$ introduced in Construction~\eqref{Const: finite open cover of R^2 for the unlink}. Let $\sh{F}$ and $\sh{G}$ be objects of the category $\ccs{1}{e_{n}}$, and let $\lambda\in \mathrm{Ext}^{0}(\sh{F},\sh{G})$ be a sheaf homomorphism between them. By Proposition~\eqref{Sheaves for the unlink on n strands}, $\sh{F}$ and $\sh{G}$ are identically zero on $U_{0}$, and as a result, $\lambda$ is also identically zero on this region. It follows that $\lambda$ is entirely determined by its behavior on the regions $U_{\mathrm{B}}$ and $U_{\mathrm{T}}$. Having established this, we present the following proposition.

\begin{proposition}\label{Prop: Ext0 for the unlink}
\textbf{Setup}: Let $e_{n}\in\mathrm{Br}^{+}_{n}$ be the trivial braid word, $\mathcal{U}_{\Lambda(e_{n})}=\big\{ U_{0}, U_{\mathrm{B}}, U_{\mathrm{T}}\big\}$ the open cover of $\mathbb{R}^{2}$ introduced in Construction~\eqref{Const: finite open cover of R^2 for the unlink}, and $\sh{F}$ and $\sh{G}$ objects of the category $\ccs{1}{e_{n}}$. Accordingly, Proposition~\eqref{Prop: linear map description of a sheaf for the unlink on n strands} asserts that: 
\begin{itemize}
\justifying
\item  On $U_{\mathrm{T}}$, $\sh{F}$ and $\sh{G}$ are specified by two collections of surjective linear maps 
\begin{equation*}
\big\{\psi_{\sh{F}}^{(i)}:\mathbb{K}^{i+1}\to \mathbb{K}^{i}\big\}_{i=1}^{n-1}\, , \quad \text{and} \quad \big\{\psi_{\sh{G}}^{(i)}:\mathbb{K}^{i+1}\to \mathbb{K}^{i}\big\}_{i=1}^{n-1}\, ,    
\end{equation*}
respectively (see Figure~\eqref{Fig: an object F on the region U_T for the unlink} for a schematic representation).

\item On $U_{\mathrm{B}}$, $\sh{F}$ and $\sh{G}$ are specified by two collections of injective linear maps 
\begin{equation*}
\big\{\phi_{\sh{F}}^{(i)}:\mathbb{K}^{i}\to \mathbb{K}^{i+1}\big\}_{i=1}^{n-1}\, , \quad \text{and} \quad \big\{\phi_{\sh{G}}^{(i)}:\mathbb{K}^{i}\to \mathbb{K}^{i+1}\big\}_{i=1}^{n-1} \, ,    
\end{equation*}
respectively (see Figure~\eqref{Fig: an object F on the region U_B for the unlink} for a schematic representation).

\item \textbf{Compatibility conditions}: For each $i\in[1,n-1]$, 
\begin{equation*}
\psi^{(i)}_{\sh{F}}\circ \phi^{(i)}_{\sh{F}}=\mathrm{id}_{\mathbb{K}^{i}}\, , \quad \text{and} \quad \psi^{(i)}_{\sh{G}}\circ \phi^{(i)}_{\sh{G}}=\mathrm{id}_{\mathbb{K}^{i}}\, .    
\end{equation*}
\end{itemize}

\noindent
$\star$ \emph{Assumption}: For each $i\in[1,n]$, let $\hat{\mathbf{f}}^{(i)}:=\big\{\hat{f}^{(i)}_{j}\big\}_{j=1}^{i}$ and $\hat{\mathbf{g}}^{(i)}:=\big\{\hat{g}^{(i)}_{j}\big\}_{j=1}^{i}$ be bases for $\mathbb{K}^{i}$. In particular, building on Definition~\eqref{Def:flags and adapted bases}--\eqref{Def: adapted bases I}--\eqref{Def: adapted bases II}, we assume that: 
\begin{itemize}
\justifying
\item $\big\{\hat{\mathbf{f}}^{(i)}\big\}_{i=1}^{n}$ is a system of bases adapted to both $\big\{\psi_{\sh{F}}^{(i)}\big\}_{i=1}^{n-1}$ and $\big\{\phi_{\sh{F}}^{(i)}\,\big\}_{i=1}^{n-1}$.
\item  $\big\{\hat{\mathbf{g}}^{(i)}\big\}_{i=1}^{n}$ is a system of bases adapted to both $\big\{\psi_{\sh{G}}^{(i)}\big\}_{i=1}^{n-1}$ and $\big\{\phi_{\sh{G}}^{(i)}\,\big\}_{i=1}^{n-1}$.
\end{itemize}

\smallskip
\noindent
\textbf{Main Conclusion}: Let $\lambda\in \mathrm{Ext}^{0}(\sh{F},\sh{G})$ be a sheaf homomorphism between $\sh{F}$ and $\sh{G}$. Then, under the given \textbf{setup}, the following statements hold: 
\begin{itemize}
\justifying
\item $\lambda$ is characterized by a collection of $n$ linear maps $\big\{ T^{(i)}_{\lambda}:\mathbb{K}^{i}\to \mathbb{K}^{i} \big\}_{i=1}^{n}$ such that, for each $i\in [1,n-1]$: 
\begin{equation*}
T^{(i)}_{\lambda}\circ \psi^{(i)}_{\sh{F}}=\psi^{(i)}_{\sh{G}}\circ T^{(i+1)}_{\lambda}\, , \quad \text{and} \quad T^{(i+1)}_{\lambda}\circ \phi^{(i)}_{\sh{F}}=\phi^{(i)}_{\sh{G}}\circ T^{(i)}_{\lambda}\, .    
\end{equation*}
\item The matrix $\tensor[_{\hat{\mathbf{g}}^{(n)}}]{ \big[\, T_{\lambda}^{\,(n)}\,\big] }{_{\hat{\mathbf{f}}^{(n)}}}\in M(n,\mathbb{K})$ representing $T^{(n)}_{\lambda}$ is diagonal. 
\item For each $i\in[1,n-1]$, the matrix $\tensor[_{\hat{\mathbf{g}}^{(i)}}]{ \big[\, T_{\lambda}^{(i)}\,\big] }{_{\hat{\mathbf{f}}^{(i)}}}\in M(i,\mathbb{K})$ representing $T^{(i)}_{\lambda}$ coincides with the principal $i\times i$ submatrix of $\tensor[_{\hat{\mathbf{g}}^{(n)}}]{ \big[\, T_{\lambda}^{\,(n)}\,\big] }{_{\hat{\mathbf{f}}^{(n)}}}$.  
\end{itemize}
Consequently, we deduce that
\begin{equation*}
\mathrm{Ext}^{0}(\sh{F},\sh{G})\cong \mathbb{K}^{n}_{\mathrm{std}}\, .    
\end{equation*}
\end{proposition}
\begin{proof}
Building on our previous discussion, we have that:
\begin{itemize}
\justifying
\item On $U_{\mathrm{T}}$, $\lambda$ is specified by a collection of linear maps $\big\{ T_{\lambda}^{(i)}:\mathbb{K}^{i}\to \mathbb{K}^{i}\big\}_{i=1}^{n}$ such that for each $i\in[1,n-1]$, the pair $\big( T^{(i+1)}_{\lambda},T^{(i)}_{\lambda}\big)$ defines a morphism between $\psi_{\sh{F}}^{(i)}$ and $\psi_{\sh{G}}^{(i)}$, i.e., 
\begin{equation}\label{Eq: identity I for Ext0 for the unlink}
T^{(i)}_{\lambda}\circ \psi^{(i)}_{\sh{F}}=\psi^{(i)}_{\sh{G}}\circ T^{(i+1)}_{\lambda}\, .   
\end{equation}
\item On $U_{\mathrm{B}}$, $\lambda$ is specified by a collection of linear maps $\big\{ B_{\lambda}^{(i)}:\mathbb{K}^{i}\to \mathbb{K}^{i}\big\}_{i=1}^{n}$ such that for each $i\in[1,n-1]$, the pair $\big(B^{(i)}_{\lambda},B^{(i+1)}_{\lambda}\big)$ defines a morphism between $\phi_{\sh{F}}^{(i)}$ and $\phi_{\sh{G}}^{(i)}$, i.e., 
\begin{equation}\label{Eq: identity II for Ext0 for the unlink}
B^{(i+1)}_{\lambda}\circ \phi^{(i)}_{\sh{F}}=\phi^{(i)}_{\sh{G}}\circ B^{(i)}_{\lambda}\, .   
\end{equation}
\end{itemize}

In particular, note that a direct application of the sheaf axioms to the intersection $U_{\mathrm{B}}\,\cap\,U_{\mathrm{T}}\subset \mathbb{R}^{2}$ implies that $T_{\lambda}^{(n)}=B_{\lambda}^{(n)}$. Hence, combining the morphism conditions~\eqref{Eq: identity I for Ext0 for the unlink} and~\eqref{Eq: identity II for Ext0 for the unlink} with the \textbf{compatibility conditions} for the maps characterizing $\sh{F}$ and $\sh{G}$ on $U_{\mathrm{B}}$ and $U_{\mathrm{T}}$, we further deduce that $T_{\lambda}^{(i)}=B_{\lambda}^{(i)}$ for every $i\in[1,n-1]$. It follows that $\lambda$ is determined by a collection of linear maps $\big\{T_{\lambda}^{(i)}:\mathbb{K}^{i}\to \mathbb{K}^{i}\big\}_{i=1}^{n}$ satisfying the stated conditions. 

Finally, recall that $\big\{\hat{\mathbf{f}}^{(i)}\big\}_{i=1}^{n}$ and $\big\{\hat{\mathbf{g}}^{(i)}\big\}_{i=1}^{n}$ are systems of bases adapted to the pairs $\bigr( \big\{\psi_{\sh{F}}^{(i)}\big\}_{i=1}^{n-1}, \big\{\phi_{\sh{F}}^{(i)}\big\}_{i=1}^{n-1} \bigl)$ and $\bigr( \big\{\psi_{\sh{G}}^{(i)}\big\}_{i=1}^{n-1}, \big\{\phi_{\sh{G}}^{(i)}\big\}_{i=1}^{n-1} \bigl)$, respectively. Thus, relying on our previous results, Lemmas~\eqref{Lemma: Ext0 for surjective linear maps} and~\eqref{Lemma: Ext0 for injective linear maps} guarantee that:
\begin{itemize}
\justifying
\item The matrix $\tensor[_{\hat{\mathbf{g}}^{(n)}}]{ \big[\, T_{\lambda}^{(n)}\,\big] }{_{\hat{\mathbf{f}}^{(n)}}}\in M(n,\mathbb{K})$ representing $T^{(n)}_{\lambda}$ is both lower and upper triangular, and hence diagonal. 
\item For each $i\in[1,n-1]$, the matrix $\tensor[_{\hat{\mathbf{g}}^{(i)}}]{ \big[\, T_{\lambda}^{(i)}\,\big] }{_{\hat{\mathbf{f}}^{(i)}}}\in M(i,\mathbb{K})$ representing $T^{(i)}_{\lambda}$ coincides with the principal $i\times i$ submatrix of $\tensor[_{\hat{\mathbf{g}}^{(n)}}]{ \big[\, T_{\lambda}^{\,(n)}\,\big] }{_{\hat{\mathbf{f}}^{(n)}}}$.  
\end{itemize}
Bearing this in mind, we conclude that
\begin{equation*}
\mathrm{Ext}^{0}(\sh{F},\sh{G})\cong \mathbb{K}^{n}_{\mathrm{std}}\, .
\end{equation*}
This completes the proof. 
\end{proof}

Having established the preceding result, we conclude our study of the zero-degree morphism spaces in the category $\ccs{1}{e_{n}}$, and turn to their analysis in the category $\ccs{1}{\beta}$ for an arbitrary positive braid word $\beta\in\mathrm{Br}^{+}_{n}$.  

\subsubsection{General Positive Braids} Let $\beta=\sigma_{i_{1}}\cdots \sigma_{i_{\ell}}\in\mathrm{Br}^{+}_{n}$ be a positive braid word. Next, we analyze the zero-degree morphism spaces in the category $\ccs{1}{\beta}$.   

To being, we establish an observation that will facilitate the study the zero-degree morphism spaces in the category $\ccs{1}{\beta}$. 

\begin{observation}\label{Obs: objects on U_T and vertical straps}
Let $\beta=\sigma_{i_{1}}\cdots \sigma_{i_{\ell}}\in \mathrm{Br}^{+}_{n}$ be a positive braid word, $\mathcal{S}_{\Lambda(\beta)}$ the stratification of $\mathbb{R}^{2}$ induced by $\Lambda(\beta)\subset (\mathbb{R}^{3},\xi_{\mathrm{std}})$, and $\sh{F}$ an object of the category $\ccs{1}{\beta}$. Let \,$\mathcal{U}_{\Lambda(\beta)}=\big\{U_{0}, U_{\mathrm{B}}, U_{\mathrm{L}}, U_{\mathrm{R}}, U_{\mathrm{T}}\big\}$ be the open cover of $\mathbb{R}^{2}$ from Construction~\eqref{Cons: Finite open cover for R^2}, and let $\mathcal{R}_{\Lambda(\beta)}=\big\{R_{j}\big\}_{j=1}^{\ell}$ be the partition of $U_{\mathrm{B}}$ into $\ell$ open vertical straps from Construction~\eqref{Cons: Definition of the vertical straps}. 

In particular, observe that the only singularities of $\Pi_{x,z}(\Lambda)$ contained in $U_{\mathrm{L}}$ are the $n$ left cusps, and $R_{1}$ is its only adjacent vertical strap. Hence, since $U_{\mathrm{L}}$ contains no crossings of $\Pi_{x,z}(\Lambda)$, the constructibility of $\sh{F}$ with respect to $\mathcal{S}_{\Lambda(\beta)}$ and the sheaf axioms guarantee that the maps defining $\sh{F}$ on $U_{\mathrm{L}}$ coincide with those specifying it on the intersection $U_{\mathrm{L}}\,\cap\, R_{1}$. It follows that, for the purpose of studying $\sh{F}$, instead of considering the region $U_{\mathrm{L}}$, it suffices to restrict to $R_{1}$, with the understanding that the data characterizing $\sh{F}$ on $R_{1}$ inherits the compatibility conditions associated with the left cusps in $U_{\mathrm{L}}$. 

Finally, note that a completely analogous argument applies to the region $U_{\mathrm{R}}$ and the vertical strap $R_{\ell}$, with the data characterizing $\sh{F}$ on $R_{\ell}$ inheriting the compatibility conditions associated with the right cusps in $U_{\mathrm{R}}$. Consequently, we deduce that the global structure of $\sh{F}$ can be fully recovered from its local description on the region $U_{\mathrm{T}}$ and the collection of vertical straps $\mathcal{R}_{\Lambda(\beta)}$, together with the appropriate compatibility conditions.
\end{observation}

Now, let $\mathcal{U}_{\Lambda(\beta)}=\big\{U_{0}, U_{\mathrm{B}}, U_{\mathrm{L}}, U_{\mathrm{R}}, U_{\mathrm{T}}\big\}$ be the open cover of $\mathbb{R}^{2}$ from Construction~\eqref{Cons: Finite open cover for R^2}, $\mathcal{R}_{\Lambda(\beta)}=\big\{R_{j}\big\}_{j=1}^{\ell}$ be the partition of $U_{\mathrm{B}}$ into $\ell$ open vertical straps from Construction~\eqref{Cons: Definition of the vertical straps}, $\sh{F}$ and $\sh{G}$ objects of the category $\ccs{1}{\beta}$, and $\lambda\in\mathrm{Ext}^{0}(\sh{F},\sh{G})$ a sheaf homomorphism between them. Building on Observation~\eqref{Obs: objects on U_T and vertical straps}, we have that the global structure of $\lambda$ can be fully recovered from its local description on the region $U_{\mathrm{T}}$ and the collection of vertical straps $\mathcal{R}_{\Lambda(\beta)}$, together with the appropriate compatibility conditions. Accordingly, we present the following result.

\begin{lemma}\label{Lemma: Ext0 on U_T and U_L}
\textbf{Setup}: Let $\beta=\sigma_{i_{1}}\cdots \sigma_{i_{\ell}} \in\mathrm{Br}^{+}_{n}$ be a positive braid word, $\mathcal{U}_{\Lambda(\beta)}=\big\{U_{0}, U_{\mathrm{B}}, U_{\mathrm{L}}, U_{\mathrm{R}}, U_{\mathrm{T}}\big\}$ the open cover of $\mathbb{R}^{2}$ from Construction~\eqref{Cons: Finite open cover for R^2}, and $\sh{F}$ and $\sh{G}$ objects of the category $\ccs{1}{\beta}$. According to Lemma~\eqref{Lemma: linear map description of an object on the regions U_T, U_L, and U_R}, $\sh{F}$ and $\sh{G}$ have the following local descriptions: 
\begin{itemize}
\justifying
\item  On $U_{\mathrm{T}}$, $\sh{F}$ and $\sh{G}$ are specified by two collections of surjective linear maps 
\begin{equation*}
\big\{\psi_{\sh{F}}^{(i)}:\mathbb{K}^{i+1}\to \mathbb{K}^{i}\big\}_{i=1}^{n-1}\, , \quad \text{and} \quad \big\{\psi_{\sh{G}}^{(i)}:\mathbb{K}^{i+1}\to \mathbb{K}^{i}\big\}_{i=1}^{n-1}\, ,    
\end{equation*}
respectively. For a schematic illustration of a generic representative of one of these sheaves on $U_{\mathrm{T}}$, see Figure~\eqref{Fig: an object F in the region U_T}.

\item On $U_{\mathrm{L}}$, $\sh{F}$ and $\sh{G}$ are specified by two collections of injective linear maps 
\begin{equation*}
\big\{\phi_{\sh{F}}^{(i)}:\mathbb{K}^{i}\to \mathbb{K}^{i+1}\big\}_{i=1}^{n-1}\, , \quad \text{and} \quad \big\{\phi_{\sh{G}}^{(i)}:\mathbb{K}^{i}\to \mathbb{K}^{i+1}\big\}_{i=1}^{n-1} \, ,    
\end{equation*}
respectively. For a schematic illustration of a generic representative of one of these sheaves on $U_{\mathrm{L}}$, see Figure~\eqref{Fig: an object F in the region U_L}.

\item \textbf{Compatibility conditions}: For each $i\in[1,n-1]$, 
\begin{equation*}
\psi^{(i)}_{\sh{F}}\circ \phi^{(i)}_{\sh{F}}=\mathrm{id}_{\mathbb{K}^{i}}\, , \quad \text{and} \quad \psi^{(i)}_{\sh{G}}\circ \phi^{(i)}_{\sh{G}}=\mathrm{id}_{\mathbb{K}^{i}}\, .    
\end{equation*}
\end{itemize}

\noindent
$\star$ \emph{Assumption}: For each $i\in[1,n]$, let $\hat{\mathbf{f}}^{(i)}:=\big\{\hat{f}^{(i)}_{j}\big\}_{j=1}^{i}$ and $\hat{\mathbf{g}}^{(i)}:=\big\{\hat{g}^{(i)}_{j}\big\}_{j=1}^{i}$ be bases for $\mathbb{K}^{i}$. Based on Definition~\eqref{Def:flags and adapted bases}--\eqref{Def: adapted bases I}--\eqref{Def: adapted bases II}, we assume that: 
\begin{itemize}
\justifying
\item $\big\{\hat{\mathbf{f}}^{(i)}\big\}_{i=1}^{n}$ is a system of bases adapted to both $\big\{\psi_{\sh{F}}^{(i)}\big\}_{i=1}^{n-1}$ and $\big\{\phi_{\sh{F}}^{(i)}\,\big\}_{i=1}^{n-1}$.
\item  $\big\{\hat{\mathbf{g}}^{(i)}\big\}_{i=1}^{n}$ is a system of bases adapted to both $\big\{\psi_{\sh{G}}^{(i)}\big\}_{i=1}^{n-1}$ and $\big\{\phi_{\sh{G}}^{(i)}\,\big\}_{i=1}^{n-1}$.
\end{itemize}

\smallskip
\noindent
\textbf{Main Conclusion}: Let $\lambda\in \mathrm{Ext}^{0}(\sh{F},\sh{G})$. Under the given \textbf{setup}, the following statements hold: 

\noindent
$\star$ \emph{Local characterization of $\lambda$ on $U_{\mathrm{T}}$ and $U_{\mathrm{L}}$}: On $U_{\mathrm{T}}$ and \,$U_{\mathrm{L}}$, $\lambda$ is characterized by a collection of $n$ linear maps $\big\{ T^{(i)}_{\lambda}:\mathbb{K}^{i}\to \mathbb{K}^{i} \big\}_{i=1}^{n}$ such that for each $i\in [1,n-1]$: 
\begin{equation*}
T^{(i)}_{\lambda}\circ \psi^{(i)}_{\sh{F}}=\psi^{(i)}_{\sh{G}}\circ T^{(i+1)}_{\lambda}\, , \quad \text{and} \quad T^{(i+1)}_{\lambda}\circ \phi^{(i)}_{\sh{F}}=\phi^{(i)}_{\sh{G}}\circ T^{(i)}_{\lambda}\, .    
\end{equation*} 

\noindent
$\star$ \emph{Block-diagonal properties}: With respect to the bases $\big\{\hat{\mathbf{f}}^{(i)}\big\}_{i=1}^{n}$ and $\big\{\hat{\mathbf{g}}^{(i)}\big\}_{i=1}^{n}$, we have that: 
\begin{itemize}
\justifying
\item $\tensor[_{\hat{\mathbf{g}}^{(n)}}]{ \big[\, T_{\lambda}^{\,(n)}\,\big] }{_{\hat{\mathbf{f}}^{(n)}}}\in M(n,\mathbb{K})$ is a diagonal matrix. 
\item For each $i\in[1,n-1]$, the matrix $\tensor[_{\hat{\mathbf{g}}^{(i)}}]{ \big[\, T_{\lambda}^{(i)}\,\big] }{_{\hat{\mathbf{f}}^{(i)}}}\in M(i,\mathbb{K})$ coincides with the principal $i\times i$ submatrix of $\tensor[_{\hat{\mathbf{g}}^{(n)}}]{ \big[\, T_{\lambda}^{\,(n)}\,\big] }{_{\hat{\mathbf{f}}^{(n)}}}$.  
\end{itemize}
\end{lemma}
\begin{proof}
The argument parallels that of Proposition~\eqref{Prop: Ext0 for the unlink}. More precisely, given the local descriptions of $\sh{F}$ and $\sh{G}$ on $U_{\mathrm{T}}$ and $U_{\mathrm{L}}$ from Lemma~\eqref{Lemma: linear map description of an object on the regions U_T, U_L, and U_R}, the \emph{local characterization of $\lambda$ on $U_{\mathrm{T}}$ and $U_{\mathrm{L}}$} follows from Observation~\eqref{Obs: elements of Ext0 near arcs}. Furthermore, this characterization, combined with the standing \emph{assumption} and Lemmas~\eqref{Lemma: Ext0 for surjective linear maps} and~\eqref{Lemma: Ext0 for injective linear maps}, yields the \emph{block-diagonal properties}.
\end{proof}

Fix $j\in [1,\ell]$, and let $R_{j}$ be the vertical strap in $\mathbb{R}^2$ containing $\sigma_{i_{j}}$---the $j$-th crossing of $\beta$ (see Figure~\eqref{fig: A sub-regions R_j}). Building on the algebraic local models established in Lemmas~\eqref{Lemma: matrix local model for sigma_1} and~\eqref{Lemma: matrix local model for sigma_k}, we now give an explicit local characterization of the elements of the zero-degree morphism spaces in the category $\ccs{1}{\beta}$ on $R_j$.

\begin{lemma}\label{Lemma: Ext0 on R_j for k=1}
\textbf{Setup}: Let $\beta=\sigma_{i_{1}}\cdots \sigma_{i_{\ell}}\in\mathrm{Br}^{+}_{n}$ be a positive braid word, and let $\sh{F}$ and $\sh{G}$ be objects of the category $\ccs{1}{\beta}$. Fix $j\in[1,\ell]$, and let $R_{j}$ be the vertical strap in $\mathbb{R}^{2}$ containing $\sigma_{i_{j}}$---the $j$-th crossing of $\beta$ (see Figure~\eqref{fig: A sub-regions R_j}). 

\noindent
$\star$ \emph{Assumption 1}: Let $k:=i_{j}\in[1,n-1]$ denote the index of $\sigma_{i_{j}}$, and suppose that $k=1$.

Under this assumption, on $R_{j}$, $\sh{F}$ and $\sh{G}$ are specified by two collections of $n$ injective linear maps 
\begin{equation}\label{Eq: maps defining F and G on R_j for k=1}
\begin{aligned}
\big\{\phi^{\,(i)}_{\sh{F}}:\mathbb{K}^{i}\to \mathbb{K}^{i+1}\,\big\}_{i=1}^{n-1}&\,\cup\,\big\{ \widetilde{\phi}^{\,(1)}_{\sh{F}}:\mathbb{K}^{1}\to \mathbb{K}^{2}\, \big\}\, ,\\[6pt]
\big\{\phi^{\,(i)}_{\sh{G}}:\mathbb{K}^{i}\to \mathbb{K}^{i+1}\,\big\}_{i=1}^{n-1}&\,\cup\,\big\{ \widetilde{\phi}^{\,(1)}_{\sh{G}}:\mathbb{K}^{1}\to \mathbb{K}^{2}\, \big\}\, ,
\end{aligned}    
\end{equation}
respectively. For a schematic illustration of a generic representative of one of these sheaves, see Figure~\eqref{fig: a sheaf in the vertical strap R_j containing a crossing sigma_1}.

\noindent
$\star$ \emph{Assumption 2}: For each $i\in[1,n]$, let $\hat{\mathbf{f}}^{(i)}:=\big\{\hat{f}^{(i)}_{j}\big\}_{j=1}^{i}$ and $\hat{\mathbf{g}}^{(i)}:=\big\{\hat{g}^{(i)}_{j}\big\}_{j=1}^{i}$ be bases for $\mathbb{K}^{i}$. Based on Definition~\eqref{Def:flags and adapted bases}--\eqref{Def: adapted bases I}, we assume that: 
\begin{itemize}
\justifying
\item $\big\{\hat{\mathbf{f}}^{(i)}\big\}_{i=1}^{n}$ is a system of bases adapted to $\big\{\phi_{\sh{F}}^{(i)}\,\big\}_{i=1}^{n-1}$.
\item $\big\{\hat{\mathbf{g}}^{(i)}\big\}_{i=1}^{n}$ is a system of bases adapted to $\big\{\phi_{\sh{G}}^{(i)}\,\big\}_{i=1}^{n-1}$.
\end{itemize}

Under this assumption, Lemma~\eqref{Lemma: matrix local model for sigma_1} ensures that $\big(\big\{\hat{\mathbf{f}}^{(i)}\big\}_{i=1}^{n}, z \big)$ and $\big(\big\{\hat{\mathbf{g}}^{(i)}\big\}_{i=1}^{n}, z' \big)$ are system of bases adapted to $\sh{F}$ and $\sh{G}$ on $R_{j}$, where $z,\,z'\in \mathbb{K}$ algebraically parameterize, relative to the bases $\hat{\mathbf{f}}^{(n)}$ and $\hat{\mathbf{g}}^{(n)}$ for $\mathbb{K}^{n}$, the $s_{1}$-relative position between the pairs of complete flags in $\mathbb{K}^{n}$ that geometrically characterize $\sh{F}$ and $\sh{G}$ on $R_{j}$, respectively. In particular, following Definition~\eqref{Def: braid transformation of bases}, we denote by $\big\{\hat{\mathbf{f}}^{(i)}[\sigma_{1},z]\big\}_{i=1}^{n}$ and $\big\{\hat{\mathbf{g}}^{(i)}[\sigma_{1},z']\big\}_{i=1}^{n}$ the corresponding braid-transformed bases.

\vspace{4pt}
\noindent
\textbf{Main Conclusion}: Let $\lambda\in\mathrm{Ext}^{0}(\sh{F},\sh{G})$. Under the given \textbf{setup}, the following statements hold:

\noindent
$\star$ \emph{Local characterization of $\lambda$ on $R_{j}$}: On $R_{j}$, $\lambda$ is specified by a collection of $n+1$ linear maps
\begin{equation}\label{Eq: maps for lambda in R_j for k=1}
\big\{ T_{\lambda}^{(i)}:\mathbb{K}^{i}\to \mathbb{K}^{i}\big\}_{i=1}^{n}\,\cup\,\big\{ \widetilde{T}_{\lambda}^{\,(1)}:\mathbb{K}^{1}\to \mathbb{K}^{1}\big\}\, ,
\end{equation}
such that
\begin{equation}\label{Eq: morphism conditions for lambda in R_j for k=1}
\begin{aligned}
T^{(i+1)}_{\lambda}\circ \phi_{\sh{F}}^{(i)}=\phi^{(i)}_{\sh{G}}\circ T_{\lambda}^{(i)}\, , \quad \text{and} \quad
T^{(2)}_{\lambda}\circ \widetilde{\phi}_{\sh{F}}^{\,(1)}=\widetilde{\phi}^{\,(1)}_{\sh{G}}\circ \widetilde{T}_{\lambda}^{\,(1)}\, ,
\end{aligned}
\end{equation}
for each $i\in [1,n-1]$. 

\noindent
$\star$ \emph{Block-Triangular Properties I}: With respect to the bases $\big\{\hat{\mathbf{f}}^{(i)}\big\}_{i=1}^{n}$ and $\big\{\hat{\mathbf{g}}^{(i)}\big\}_{i=1}^{n}$, we have that: 
\begin{itemize}
\justifying
\item $\tensor[_{\hat{\mathbf{g}}^{(n)}}]{ \big[\, T_{\lambda}^{(n)}\,\big] }{_{\hat{\mathbf{f}}^{(n)}}}\in M(n,\mathbb{K})$ is an upper-triangular matrix. 
\item For each $i\in[1,n-1]$, the matrix $\tensor[_{\hat{\mathbf{g}}^{(i)}}]{ \big[\, T_{\lambda}^{(i)}\,\big] }{_{\hat{\mathbf{f}}^{(i)}}}\in M(i,\mathbb{K})$ coincides with the principal $i\times i$ submatrix of $\tensor[_{\hat{\mathbf{g}}^{(n)}}]{ \big[\, T_{\lambda}^{\,(n)}\,\big] }{_{\hat{\mathbf{f}}^{(n)}}}$, 
\end{itemize}

\noindent
$\star$ \emph{Block-Triangular Properties II}: With respect to the bases $\big\{\hat{\mathbf{f}}^{(i)}[\sigma_{1},z]\big\}_{i=1}^{n}$ and $\big\{\hat{\mathbf{g}}^{(i)}[\sigma_{1},z']\big\}_{i=1}^{n}$, we have that: 
\begin{itemize}
\justifying
\item $\tensor[_{\hat{\mathbf{g}}^{(n)}[\sigma_{1},z'] }]{ \big[\, T_{\lambda}^{\,(n)}\,\big] }{_{\hat{\mathbf{f}}^{(n)}[\sigma_{1},z]}}\in M(n,\mathbb{K})$ is upper-triangular. 

\item $\tensor[_{\hat{\mathbf{g}}^{(1)}[\sigma_{1},z'] }]{ \big[\, \widetilde{T}_{\lambda}^{\,(1)}\,\big] }{_{\hat{\mathbf{f}}^{(1)}[\sigma_{1},z]}} \in M(1,\mathbb{K})$ coincides with the principal $1\times 1$ submatrix of $\tensor[_{\hat{\mathbf{g}}^{(n)}[\sigma_{1},z'] }]{ \big[\, T_{\lambda}^{\,(n)}\,\big] }{_{\hat{\mathbf{f}}^{(n)}[\sigma_{1},z]}}$. 

\item For each $i\in[2,n-1]$, $\tensor[_{\hat{\mathbf{g}}^{(i)}[\sigma_{1},z'] }]{ \big[\, T_{\lambda}^{\,(i)}\,\big] }{_{\hat{\mathbf{f}}^{(i)}[\sigma_{1},z]}} \in M(i,\mathbb{K})$ coincides with the principal $i\times i$ submatrix of $\tensor[_{\hat{\mathbf{g}}^{(n)}[\sigma_{1},z'] }]{ \big[\, T_{\lambda}^{\,(n)}\,\big] }{_{\hat{\mathbf{f}}^{(n)}[\sigma_{1},z]}}$.
\end{itemize}

\noindent
$\star$ \emph{Refinement}: Suppose that $\tensor[_{\hat{\mathbf{g}}^{(n)}}]{ \big[\, T_{\lambda}^{\,(n)}\,\big] }{_{\hat{\mathbf{f}}^{(n)}}}$ is diagonal, that is,
\begin{equation*}
\tensor[_{\hat{\mathbf{g}}^{(n)}}]{ \big[\, T_{\lambda}^{\,(n)}\,\big] }{_{\hat{\mathbf{f}}^{(n)}}}=D(\vec{u})\, ,    
\end{equation*}
for some tuple $\vec{u}=(u_{1},\dots, u_{n})\in \mathbb{K}^{n}_{\mathrm{std}}$. Then:

\begin{itemize}
\item \textbf{Compatibility condition for $\lambda$}: The $(1+1,1)$-entry of the matrix product
\begin{equation*}
\big(B^{(n)}_{1}(z')\big)^{-1}\cdot D(\vec{u})\cdot B^{(n)}_{1}(z)    
\end{equation*}
must vanish. 

\item Under the above condition,
\begin{equation*}
\tensor[_{\hat{\mathbf{g}}^{(n)}[\sigma_{1},z'] }]{ \big[\, T_{\lambda}^{(n)}\,\big] }{_{\hat{\mathbf{f}}^{(n)}[\sigma_{1},z]}}=D(\pi_{1}(\vec{u}))\,,    
\end{equation*}
where $\pi_{1}(\vec{u})$ denotes the permutation of $\vec{u}$ associated with $\sigma_{1}$.
\end{itemize}
\end{lemma}
\begin{proof}
According to \emph{Assumption 1}, on $R_{j}$, $\sh{F}$ and $\sh{G}$ are defined by the collections of $n$ injective linear maps listed in~\eqref{Eq: maps defining F and G on R_j for k=1}. Hence, the \emph{Local characterization of $\lambda$ on $R_{j}$}, specified by the $n+1$ linear maps listed in~\eqref{Eq: maps for lambda in R_j for k=1} and subject to the morphism conditions~\eqref{Eq: morphism conditions for lambda in R_j for k=1}, follows from Observation~\eqref{Obs: elements of Ext0 near arcs}, which establishes the local description of the elements of $\mathrm{Ext}^{0}(\sh{F},\sh{G})$ near an arc of the stratification $\mathcal{S}_{\Lambda(\beta)}$ of $\mathbb{R}^{2}$ induced by $\Lambda(\beta)\subset (\mathbb{R}^{3},\xi_{\mathrm{std}})$.   

By \emph{Assumption 2}, $\big\{\hat{\mathbf{f}}^{(i)}\big\}_{i=1}^{n}$ and $\big\{\hat{\mathbf{g}}^{(i)}\big\}_{i=1}^{n}$ are systems of bases adapted to $\big\{\phi_{\sh{F}}^{(i)}\big\}_{i=1}^{n-1}$ and $\big\{\phi_{\sh{G}}^{(i)}\big\}_{i=1}^{n-1}$, respectively. Therefore, by virtue of the morphism conditions~\eqref{Eq: morphism conditions for lambda in R_j for k=1}, the \emph{Block-Triangular Properties I} follow directly from Lemma~\eqref{Lemma: Ext0 for injective linear maps}.

Moreover, under \emph{Assumption 2}, Lemma~\eqref{Lemma: matrix local model for sigma_1} asserts that $\big(\big\{\hat{\mathbf{f}}^{(i)}\big\}_{i=1}^{n}, z \big)$ and $\big(\big\{\hat{\mathbf{g}}^{(i)}\big\}_{i=1}^{n}, z' \big)$ are system of bases adapted to $\sh{F}$ and $\sh{G}$ on $R_{j}$, where $z,z'\in \mathbb{K}$ algebraically parameterize, relative to the bases $\hat{\mathbf{f}}^{(n)}$ and $\hat{\mathbf{g}}^{(n)}$ for $\mathbb{K}^{n}$, the $s_{1}$-relative position between the pairs of complete flags in $\mathbb{K}^{n}$ that geometrically characterize $\sh{F}$ and $\sh{G}$ on $R_{j}$, respectively. Accordingly, we have that:  
\begin{itemize}
\item  With respect to the bases $\big\{\hat{\mathbf{f}}^{(i)}\big\}_{i=1}^{n}$, the linear maps determining $\sh{F}$ on $R_{j}$ have the matrix representations: 
\begin{equation*}
\begin{aligned}
\tensor[_{\hat{\mathbf{f}}^{(i+1)}}]{ \big[\,\phi^{\,(i)}_{\sh{F}}\,\big] }{_{\hat{\mathbf{f}}^{(i)}}}=\iota^{(i+1,i)}\, , \quad \text{for all $i\in[1,n-1]$}\, , \quad \text{and} \quad  
\tensor[_{\hat{\mathbf{f}}^{(2)}}]{ \big[\,\widetilde{\phi}^{\,(1)}_{\sh{F}}\,\big] }{_{\hat{\mathbf{f}}^{(1)}}}=B^{(2)}_{1}(z)\cdot\iota^{(2,1)}\, .
\end{aligned}   
\end{equation*} 
\item  With respect to the bases $\big\{\hat{\mathbf{g}}^{(i)}\big\}_{i=1}^{n}$, the linear maps determining $\sh{G}$ on $R_{j}$ have the matrix representations: 
\begin{equation*}
\begin{aligned}
\tensor[_{\hat{\mathbf{g}}^{(i+1)}}]{ \big[\,\phi^{\,(i)}_{\sh{G}}\,\big] }{_{\hat{\mathbf{g}}^{(i)}}}=\iota^{(i+1,i)}\, , \quad \text{for all $i\in[1,n-1]$}\, , \quad \text{and} \quad
\tensor[_{\hat{\mathbf{g}}^{(2)}}]{ \big[\,\widetilde{\phi}^{\,(1)}_{\sh{G}}\,\big] }{_{\hat{\mathbf{g}}^{(1)}}}=B^{(2)}_{1}(z')\cdot\iota^{(2,1)}\, .
\end{aligned}   
\end{equation*} 
\end{itemize}

Next, consider the braid-transformed bases $\big\{\hat{\mathbf{f}}^{(i)}[\sigma_{1},z]\big\}_{i=1}^{n}$ and $\big\{\hat{\mathbf{g}}^{(i)}[\sigma_{1},z']\big\}_{i=1}^{n}$ (cf. Definition~\eqref{Def: braid transformation of bases}). In particular, a direct calculation shows that:
\begin{itemize}
\item  With respect to the bases $\big\{\hat{\mathbf{f}}^{(i)}[\sigma_{1},z]\big\}_{i=1}^{n}$,
\begin{equation*}
\begin{aligned}
\tensor[_{\hat{\mathbf{f}}^{(i+1)}[\sigma_{1}, z] }]{ \big[\,\phi^{\,(i)}_{\sh{F}}\,\big] }{_{\hat{\mathbf{f}}^{(i)}[\sigma_{1},z]}}=\iota^{(i+1,i)}\, , \quad \text{for all $2\in[1,n-1]$}\, ,\quad \text{and} \quad   
\tensor[_{\hat{\mathbf{f}}^{(2)}[\sigma_{1}, z]}]{ \big[\,\widetilde{\phi}^{\,(1)}_{\sh{F}}\,\big] }{_{\hat{\mathbf{f}}^{(1)}[\sigma_{1}, z]}}=\iota^{(2,1)}\, .
\end{aligned}   
\end{equation*} 
\item  With respect to the bases $\big\{\hat{\mathbf{g}}^{(i)}[\sigma_{1},z']\big\}_{i=1}^{n}$,
\begin{equation*}
\begin{aligned}
\tensor[_{\hat{\mathbf{g}}^{(i+1)}[\sigma_{1}, z'] }]{ \big[\,\phi^{\,(i)}_{\sh{G}}\,\big] }{_{\hat{\mathbf{g}}^{(i)}[\sigma_{1},z' ]}}=\iota^{(i+1,i)}\, , \quad \text{for all $2\in[1,n-1]$}\, ,\quad \text{and} \quad   
\tensor[_{\hat{\mathbf{g}}^{(2)}[\sigma_{1}, z']}]{ \big[\,\widetilde{\phi}^{\,(1)}_{\sh{G}}\,\big] }{_{\hat{\mathbf{g}}^{(1)}[\sigma_{1}, z']}}=\iota^{(2,1)}\, . 
\end{aligned}   
\end{equation*} 
\end{itemize}
It follows from Definition~\eqref{Def:flags and adapted bases}--\eqref{Def: adapted bases I} that $\big\{\hat{\mathbf{f}}^{(i)}[\sigma_{1},z]\big\}_{i=1}^{n}$ and $\big\{\hat{\mathbf{g}}^{(i)}[\sigma_{1},z']\big\}_{i=1}^{n}$ are systems of bases adapted to $\big\{ \widetilde{\phi}_{\sh{F}}^{\,(1)},\phi_{\sh{F}}^{(2)},\dots,\phi_{\sh{F}}^{(n-1)}\big\}$ and $\big\{ \widetilde{\phi}_{\sh{G}}^{\,(1)},\phi_{\sh{G}}^{(2)},\dots,\phi_{\sh{G}}^{(n-1)}\big\}$, respectively. Hence, by virtue of the morphism conditions~\eqref{Eq: morphism conditions for lambda in R_j for k=1}, the \emph{Block-Triangular Properties II} follow directly from Lemma~\eqref{Lemma: Ext0 for injective linear maps}.

Finally, assume that $\tensor[_{\hat{\mathbf{g}}^{(n)}}]{ \big[\, T_{\lambda}^{(n)}\,\big] }{_{\hat{\mathbf{f}}^{(n)}}}=D(\vec{u})$ for some tuple $\vec{u}=(u_{1},\dots, u_{n})\in \mathbb{K}^{n}_{\mathrm{std}}$. In particular, observe that the bases $\hat{\mathbf{f}}^{(n)}[\sigma_{1},z]$ and $\hat{\mathbf{f}}^{(n)}$ for $\mathbb{K}^{n}$, as well as the bases $\hat{\mathbf{g}}^{(n)}[\sigma_{1},z']$ and $\hat{\mathbf{g}}^{(n)}$, are related by the change-of-basis matrices $B^{(n)}_{1}(z)$ and $B^{(n)}_{1}(z')$, respectively. Therefore, by the \textit{change-of-basis formula}, we deduce that 
\begin{equation*}
\tensor[_{\hat{\mathbf{g}}^{(n)}[\sigma_{1},z'] }]{ \big[\, T_{\lambda}^{(n)}\,\big] }{_{\hat{\mathbf{f}}^{(n)}[\sigma_{1},z]}}=\big( B^{(n)}_{1}(z')\big)^{-1}\cdot D(\vec{u})\cdot B^{(n)}_{1}(z)\, . 
\end{equation*}
Consequently, since $\tensor[_{\hat{\mathbf{g}}^{(n)}[\sigma_{1},z'] }]{ \big[\, T_{\lambda}^{\,(n)}\,\big] }{_{\hat{\mathbf{f}}^{(n)}[\sigma_{1},z]}}$ must be upper-triangular, a direct application of Lemma~\eqref{lemma for braid matrices and diagonal matrices} yields that:
\begin{itemize}
\justifying
\item The $(1+1,1)$-entry of the matrix $\big(B^{(n)}_{1}(z')\big)^{-1}\cdot D(\vec{u})\cdot B^{(n)}_{1}(z)$ must vanish.
\item Under the above condition, $\tensor[_{\hat{\mathbf{g}}^{(n)}[\sigma_{1},z'] }]{ \big[\, T_{\lambda}^{\,(n)}\,\big] }{_{\hat{\mathbf{f}}^{(n)}[\sigma_{1},z]}}=D(\pi_{1}(\vec{u}))$, 
\end{itemize}
where $\pi_{1}(\vec{u})$ denotes the permutation of $\vec{u}$ associated with $\sigma_{1}$. This completes the proof. 
\end{proof}

\begin{lemma}\label{Lemma: Ext0 for R_j for k geq 2}
\textbf{Setup}: Let $\beta=\sigma_{i_{1}}\cdots \sigma_{i_{\ell}}\in\mathrm{Br}^{+}_{n}$ be a positive braid word, and let $\sh{F}$ and $\sh{G}$ be objects of the category $\ccs{1}{\beta}$. Fix $j\in[1,\ell]$, and let $R_{j}$ be the vertical strap in $\mathbb{R}^{2}$ containing $\sigma_{i_{j}}$---the $j$-th crossing of $\beta$ (see Figure~\eqref{fig: A sub-regions R_j}). 

\noindent
$\star$ \emph{Assumption 1}: Let $k:=i_{j}\in[1,n-1]$ denote the index of $\sigma_{i_{j}}$, and suppose that $k\geq 2$.

Under this assumption, on $R_{j}$, $\sh{F}$ and $\sh{G}$ are specified by two collections of $n+1$ injective linear maps 
\begin{equation}\label{Eq: maps defining F and G on R_j for k geq 2}
\begin{aligned}
\big\{ \phi^{(i)}_{\sh{F}}:\mathbb{K}^{i}\to \mathbb{K}^{i+1} \big\}_{i=1}^{n-1}&\,\cup\,\big\{ \widetilde{\phi}^{\,(k-1)}_{\sh{F}}:\mathbb{K}^{k-1}\to \mathbb{K}^{k}\,, ~ \widetilde{\phi}^{\,(k)}_{\sh{F}}:\mathbb{K}^{k}\to \mathbb{K}^{k+1} \big\}\, ,\\[6pt]
\big\{ \phi^{(i)}_{\sh{G}}:\mathbb{K}^{i}\to \mathbb{K}^{i+1} \big\}_{i=1}^{n-1}&\,\cup\,\big\{ \widetilde{\phi}^{\,(k-1)}_{\sh{G}}:\mathbb{K}^{k-1}\to \mathbb{K}^{k}\,, ~ \widetilde{\phi}^{\,(k)}_{\sh{G}}:\mathbb{K}^{k}\to \mathbb{K}^{k+1} \big\}\, ,
\end{aligned}    
\end{equation}
respectively. For a schematic illustration of a generic representative of one of these sheaves, see Figure~\eqref{fig: a sheaf in the vertical strap R_j containing a crossing sigma_k}.

\noindent
$\star$ \emph{Assumption 2}: For each $i\in[1,n]$, let $\hat{\mathbf{f}}^{(i)}:=\big\{\hat{f}^{(i)}_{j}\big\}_{j=1}^{i}$ and $\hat{\mathbf{g}}^{(i)}:=\big\{\hat{g}^{(i)}_{j}\big\}_{j=1}^{i}$ be bases for $\mathbb{K}^{i}$. Based on Definition~\eqref{Def:flags and adapted bases}--\eqref{Def: adapted bases I}, we assume that: 
\begin{itemize}
\justifying
\item $\big\{\hat{\mathbf{f}}^{(i)}\big\}_{i=1}^{n}$ is a system of bases adapted to $\big\{\phi_{\sh{F}}^{(i)}\,\big\}_{i=1}^{n-1}$.
\item $\big\{\hat{\mathbf{g}}^{(i)}\big\}_{i=1}^{n}$ is a system of bases adapted to $\big\{\phi_{\sh{G}}^{(i)}\,\big\}_{i=1}^{n-1}$.
\end{itemize}

Under this assumption, Lemma~\eqref{Lemma: matrix local model for sigma_k} ensures that $\big(\big\{\hat{\mathbf{f}}^{(i)}\big\}_{i=1}^{n}, z \big)$ and $\big(\big\{\hat{\mathbf{g}}^{(i)}\big\}_{i=1}^{n}, z' \big)$ are system of bases adapted to $\sh{F}$ and $\sh{G}$ on $R_{j}$, where $z,z'\in \mathbb{K}$ algebraically parameterize, relative to the bases $\hat{\mathbf{f}}^{(n)}$ and $\hat{\mathbf{g}}^{(n)}$ for $\mathbb{K}^{n}$, the $s_{k}$-relative position between the pairs of complete flags in $\mathbb{K}^{n}$ that geometrically characterize $\sh{F}$ and $\sh{G}$ on $R_{j}$, respectively. In particular, following Definition~\eqref{Def: braid transformation of bases}, we denote by $\big\{\hat{\mathbf{f}}^{(i)}[\sigma_{k},z]\big\}_{i=1}^{n}$ and $\big\{\hat{\mathbf{g}}^{(i)}[\sigma_{k},z']\big\}_{i=1}^{n}$ the corresponding braid-transformed bases.

\vspace{4pt}
\noindent
\textbf{Main Conclusion}: Let $\lambda\in\mathrm{Ext}^{0}(\sh{F},\sh{G})$. Under the given \textbf{setup}, the following statements hold:

\noindent
$\star$ \emph{Local characterization of $\lambda$ on $R_{j}$}: On $R_{j}$, $\lambda$ is specified by a collection of $n+1$ linear maps
\begin{equation}\label{Eq: maps for lambda in R_j for k geq 2}
\big\{T_{\lambda}^{(i)}:\mathbb{K}^{i}\to \mathbb{K}^{i} \big\}_{i=1}^{n}\,\cup\, \big\{ \widetilde{T}_{\lambda}^{\,(k)}:\mathbb{K}^{k}\to \mathbb{K}^{k} \big\}\, ,
\end{equation}
such that
\begin{equation}\label{Eq: morphism conditions for lambda in R_j for k geq 2}
\begin{aligned}
T^{(i+1)}_{\lambda}\circ \phi_{\sh{F}}^{(i)}&=\phi^{(i)}_{\sh{G}}\circ T_{\lambda}^{(i)}\,,\\[2pt]
\widetilde{T}^{\,(k)}_{\lambda}\circ \widetilde{\phi}_{\sh{F}}^{\,(k-1)}&=\widetilde{\phi}^{\,(k-1)}_{\sh{G}}\circ T_{\lambda}^{(k-1)}\, , \\[2pt]
T^{(k+1)}_{\lambda}\circ \widetilde{\phi}_{\sh{F}}^{\,(k)}&=\widetilde{\phi}^{\,(k)}_{\sh{G}}\circ \widetilde{T}_{\lambda}^{\,(k)}\, ,     
\end{aligned}
\end{equation}
for each $i\in[1,n-1]$. 

\noindent
$\star$ \emph{Block-Triangular Properties I}: With respect to the bases $\big\{\hat{\mathbf{f}}^{(i)}\big\}_{i=1}^{n}$ and $\big\{\hat{\mathbf{g}}^{(i)}\big\}_{i=1}^{n}$, we have that: 
\begin{itemize}
\justifying
\item $\tensor[_{\hat{\mathbf{g}}^{(n)}}]{ \big[\, T_{\lambda}^{(n)}\,\big] }{_{\hat{\mathbf{f}}^{(n)}}}\in M(n,\mathbb{K})$ is upper-triangular. 
\item For each $i\in[1,n-1]$, $\tensor[_{\hat{\mathbf{g}}^{(i)}}]{ \big[\, T_{\lambda}^{(i)}\,\big] }{_{\hat{\mathbf{f}}^{(i)}}}\in M(i,\mathbb{K})$ coincides with the principal $i\times i$ submatrix of $\tensor[_{\hat{\mathbf{g}}^{(n)}}]{ \big[\, T_{\lambda}^{\,(n)}\,\big] }{_{\hat{\mathbf{f}}^{(n)}}}$.
\end{itemize}

\noindent
$\star$ \emph{Block-Triangular Properties II}: With respect to the bases $\big\{\hat{\mathbf{f}}^{(i)}[\sigma_{k},z]\big\}_{i=1}^{n}$ and $\big\{\hat{\mathbf{g}}^{(i)}[\sigma_{k},z']\big\}_{i=1}^{n}$, we have that: 
\begin{itemize}
\justifying
\item $\tensor[_{\hat{\mathbf{g}}^{(n)}[\sigma_{k},z'] }]{ \big[\, T_{\lambda}^{(n)}\,\big] }{_{\hat{\mathbf{f}}^{(n)}[\sigma_{k},z]}}\in M(n,\mathbb{K})$ is upper-triangular. 

\item $\tensor[_{\hat{\mathbf{g}}^{(k)}[\sigma_{k},z'] }]{ \big[\, \widetilde{T}_{\lambda}^{\,(k)}\,\big] }{_{\hat{\mathbf{f}}^{(k)}[\sigma_{k},z]}} \in M(k,\mathbb{K})$ coincides with the principal $k\times k$ submatrix of $\tensor[_{\hat{\mathbf{g}}^{(n)}[\sigma_{k},z'] }]{ \big[\, T_{\lambda}^{(n)}\,\big] }{_{\hat{\mathbf{f}}^{(n)}[\sigma_{k},z]}}$. 

\item For each $i\in[1,n-1]$ with $i\neq k$, $\tensor[_{\hat{\mathbf{g}}^{(i)}[\sigma_{k},z'] }]{ \big[\, T_{\lambda}^{(i)}\,\big] }{_{\hat{\mathbf{f}}^{(i)}[\sigma_{k},z]}} \in M(i,\mathbb{K})$ coincides with the principal $i\times i$ submatrix of $\tensor[_{\hat{\mathbf{g}}^{(n)}[\sigma_{k},z'] }]{ \big[\, T_{\lambda}^{(n)}\,\big] }{_{\hat{\mathbf{f}}^{(n)}[\sigma_{k},z]}}$. 
\end{itemize}

\noindent
$\star$ \emph{Refinement}: Suppose that $\tensor[_{\hat{\mathbf{g}}^{(n)}}]{ \big[\, T_{\lambda}^{\,(n)}\,\big] }{_{\hat{\mathbf{f}}^{(n)}}}$ is diagonal, that is,
\begin{equation*}
\tensor[_{\hat{\mathbf{g}}^{(n)}}]{ \big[\, T_{\lambda}^{\,(n)}\,\big] }{_{\hat{\mathbf{f}}^{(n)}}}=D(\vec{u})\, ,    
\end{equation*}
for some tuple $\vec{u}=(u_{1},\dots, u_{n})\in \mathbb{K}^{n}_{\mathrm{std}}$. Then:

\begin{itemize}
\item \textbf{Compatibility condition for $\lambda$}: The $(k+1,k)$-entry of the matrix product
\begin{equation*}
\big(B^{(n)}_{k}(z')\big)^{-1}\cdot D(\vec{u})\cdot B^{(n)}_{k}(z)    
\end{equation*}
must vanish. 

\item Under the above condition,
\begin{equation*}
\tensor[_{\hat{\mathbf{g}}^{(n)}[\sigma_{k},z'] }]{ \big[\, T_{\lambda}^{(n)}\,\big] }{_{\hat{\mathbf{f}}^{(n)}[\sigma_{k},z]}}=D(\pi_{k}(\vec{u}))\,,    
\end{equation*}
where $\pi_{k}(\vec{u})$ denotes the permutation of $\vec{u}$ associated with $\sigma_{k}$.
\end{itemize}
\end{lemma}
\begin{proof}
According to \emph{Assumption 1}, on $R_{j}$, $\sh{F}$ and $\sh{G}$ are defined by the collections of $n+1$ injective linear maps listed in~\eqref{Eq: maps defining F and G on R_j for k geq 2}. Hence, the \emph{Local characterization of $\lambda$ on $R_{j}$}, specified by the $n+1$ linear maps listed in~\eqref{Eq: maps for lambda in R_j for k geq 2} and subject to the morphism conditions~\eqref{Eq: morphism conditions for lambda in R_j for k geq 2}, follows from Observation~\eqref{Obs: elements of Ext0 near arcs}, which establishes the local description of the elements of $\mathrm{Ext}^{0}(\sh{F},\sh{G})$ near an arc of the stratification $\mathcal{S}_{\Lambda(\beta)}$ of $\mathbb{R}^{2}$ induced by $\Lambda(\beta)\subset (\mathbb{R}^{3},\xi_{\mathrm{std}})$.   

By \emph{Assumption 2}, $\big\{\hat{\mathbf{f}}^{(i)}\big\}_{i=1}^{n}$ and $\big\{\hat{\mathbf{g}}^{(i)}\big\}_{i=1}^{n}$ are systems of bases adapted to $\big\{\phi_{\sh{F}}^{(i)}\big\}_{i=1}^{n-1}$ and $\big\{\phi_{\sh{G}}^{(i)}\big\}_{i=1}^{n-1}$, respectively. Therefore, by virtue of the morphism conditions~\eqref{Eq: morphism conditions for lambda in R_j for k geq 2}, the \emph{Block-Triangular Properties I} follow directly from Lemma~\eqref{Lemma: Ext0 for injective linear maps}.

Moreover, under \emph{Assumption 2}, Lemma~\eqref{Lemma: matrix local model for sigma_k} asserts that $\big(\big\{\hat{\mathbf{f}}^{(i)}\big\}_{i=1}^{n}, z \big)$ and $\big(\big\{\hat{\mathbf{g}}^{(i)}\big\}_{i=1}^{n}, z' \big)$ are system of bases adapted to $\sh{F}$ and $\sh{G}$ on $R_{j}$, where $z,z'\in \mathbb{K}$ algebraically parameterize, relative to the bases $\hat{\mathbf{f}}^{(n)}$ and $\hat{\mathbf{g}}^{(n)}$ for $\mathbb{K}^{n}$, the $s_{k}$-relative position between the pairs of complete flags in $\mathbb{K}^{n}$ that geometrically characterize $\sh{F}$ and $\sh{G}$ on $R_{j}$, respectively. Accordingly, we have that:  
\begin{itemize}
\item  With respect to the bases $\big\{\hat{\mathbf{f}}^{(i)}\big\}_{i=1}^{n}$, the linear maps determining $\sh{F}$ on $R_{j}$ have the matrix representations: 
\begin{equation*}
\begin{aligned}
\tensor[_{\hat{\mathbf{f}}^{(i+1)}}]{ \big[\,\phi^{\,(i)}_{\sh{F}}\,\big] }{_{\hat{\mathbf{f}}^{(i)}}}&=\iota^{(i+1,i)}\, ,  &&\text{for all $i\in[1,n-1]$}\, ,\\
\tensor[_{\hat{\mathbf{f}}^{(k)}}]{ \big[\,\widetilde{\phi}^{\,(k-1)}_{\sh{F}}\,\big] }{_{\hat{\mathbf{f}}^{(k-1)}}}&=i^{(k,k-1)}\,, && \tensor[_{\hat{\mathbf{f}}^{(k+1)}}]{ \big[\,\widetilde{\phi}^{\,(k)}_{\sh{F}}\,\big] }{_{\hat{\mathbf{f}}^{(k)}}}=B^{(k+1)}_{k}(z)\cdot\iota^{(k+1,k)}\, .
\end{aligned}   
\end{equation*} 
\item  With respect to the bases $\big\{\hat{\mathbf{g}}^{(i)}\big\}_{i=1}^{n}$, the linear maps determining $\sh{G}$ on $R_{j}$ have the matrix representations: 
\begin{equation*}
\begin{aligned}
\tensor[_{\hat{\mathbf{g}}^{(i+1)}}]{ \big[\,\phi^{\,(i)}_{\sh{G}}\,\big] }{_{\hat{\mathbf{g}}^{(i)}}}&=\iota^{(i+1,i)}\, ,  &&\text{for all $i\in[1,n-1]$}\, ,\\
\tensor[_{\hat{\mathbf{g}}^{(k)}}]{ \big[\,\widetilde{\phi}^{\,(k-1)}_{\sh{G}}\,\big] }{_{\hat{\mathbf{g}}^{(k-1)}}}&=i^{(k,k-1)}\,, &&\tensor[_{\hat{\mathbf{g}}^{(k+1)}}]{ \big[\,\widetilde{\phi}^{\,(k)}_{\sh{G}}\,\big] }{_{\hat{\mathbf{g}}^{(k)}}}=B^{(k+1)}_{k}(z')\cdot\iota^{(k+1,k)}\, .
\end{aligned}   
\end{equation*} 
\end{itemize}

Next, consider the braid-transformed bases $\big\{\hat{\mathbf{f}}^{(i)}[\sigma_{k},z]\big\}_{i=1}^{n}$ and $\big\{\hat{\mathbf{g}}^{(i)}[\sigma_{k},z']\big\}_{i=1}^{n}$ (cf. Definition~\eqref{Def: braid transformation of bases}). In particular, a direct calculation shows that:
\begin{itemize}
\item  With respect to the bases $\big\{\hat{\mathbf{f}}^{(i)}[\sigma_{k},z]\big\}_{i=1}^{n}$,
\begin{equation*}
\begin{aligned}
\tensor[_{\hat{\mathbf{f}}^{(i+1)}[\sigma_{k}, z] }]{ \big[\,\phi^{\,(i)}_{\sh{F}}\,\big] }{_{\hat{\mathbf{f}}^{(i)}[\sigma_{k},z]}}&=\iota^{(i+1,i)}\, , &&   &&&\text{for all $i\in[1,n-1]$ with $i\neq k-1, k$,}\\[2pt]
\tensor[_{\hat{\mathbf{f}}^{(k)}[\sigma_{k}, z]}]{ \big[\,\widetilde{\phi}^{\,(k-1)}_{\sh{F}}\,\big] }{_{\hat{\mathbf{f}}^{(k-1)}[\sigma_{k}, z]}}&=\iota^{(k,k-1)}\, , &&  &&&\\[2pt]
\tensor[_{\hat{\mathbf{f}}^{(k+1)}[\sigma_{k}, z]}]{ \big[\,\widetilde{\phi}^{\,(k)}_{\sh{F}}\,\big] }{_{\hat{\mathbf{f}}^{(k)}[\sigma_{k}, z]}}&=\iota^{(k+1,k)} \, . &&   &&&
\end{aligned}   
\end{equation*} 
\item  With respect to the bases $\big\{\hat{\mathbf{g}}^{(i)}[\sigma_{k},z']\big\}_{i=1}^{n}$,
\begin{equation*}
\begin{aligned}
\tensor[_{\hat{\mathbf{g}}^{(i+1)}[\sigma_{k}, z'] }]{ \big[\,\phi^{\,(i)}_{\sh{G}}\,\big] }{_{\hat{\mathbf{g}}^{(i)}[\sigma_{k},z']}}&=\iota^{(i+1,i)}\, ,  &&   &&&\text{for all $i\in[1,n-1]$ with $i\neq k-1, k$,}\\[2pt]
\tensor[_{\hat{\mathbf{g}}^{(k)}[\sigma_{k}, z']}]{ \big[\,\widetilde{\phi}^{\,(k-1)}_{\sh{G}}\,\big] }{_{\hat{\mathbf{g}}^{(k-1)}[\sigma_{k}, z']}}&=\iota^{(k,k-1)}\, ,  &&  &&&\\[2pt]
\tensor[_{\hat{\mathbf{g}}^{(k+1)}[\sigma_{k}, z']}]{ \big[\,\widetilde{\phi}^{\,(k)}_{\sh{G}}\,\big] }{_{\hat{\mathbf{g}}^{(k)}[\sigma_{k}, z']}}&=\iota^{(k+1,k)}\, .  &&   &&&
\end{aligned}   
\end{equation*} 
\end{itemize}
It follows from Definition~\eqref{Def:flags and adapted bases}--\eqref{Def: adapted bases I} that $\big\{\hat{\mathbf{f}}^{(i)}[\sigma_{k},z]\big\}_{i=1}^{n}$ and $\big\{\hat{\mathbf{g}}^{(i)}[\sigma_{k},z']\big\}_{i=1}^{n}$ are systems of bases adapted to the collections
\begin{equation*}
\begin{aligned}
&\big\{ \phi_{\sh{F}}^{\,(1)},\dots, \phi^{\,(k-2)}_{\sh{F}},\widetilde{\phi}^{\,(k-1)}_{\sh{F}},\widetilde{\phi}^{\,(k)}_{\sh{F}},\phi_{\sh{F}}^{\,(k+1)},\dots,\phi_{\sh{F}}^{(n-1)}\big\}\, ,\\[6pt]
&\big\{ \phi_{\sh{G}}^{\,(1)},\dots, \phi^{\,(k-2)}_{\sh{G}},\widetilde{\phi}^{\,(k-1)}_{\sh{G}},\widetilde{\phi}^{\,(k)}_{\sh{G}},\phi_{\sh{G}}^{\,(k+1)},\dots,\phi_{\sh{G}}^{(n-1)}\big\}\, , 
\end{aligned}    
\end{equation*}
respectively. Hence, by virtue of the morphism conditions~\eqref{Eq: morphism conditions for lambda in R_j for k geq 2}, the \emph{Block-Triangular Properties II} follow directly from Lemma~\eqref{Lemma: Ext0 for injective linear maps}.

Finally, assume that $\tensor[_{\hat{\mathbf{g}}^{(n)}}]{ \big[\, T_{\lambda}^{(n)}\,\big] }{_{\hat{\mathbf{f}}^{(n)}}}=D(\vec{u})$ for some tuple $\vec{u}=(u_{1},\dots, u_{n})\in \mathbb{K}^{n}_{\mathrm{std}}$. In particular, observe that the bases $\hat{\mathbf{f}}^{(n)}[\sigma_{k},z]$ and $\hat{\mathbf{f}}^{(n)}$ for $\mathbb{K}^{n}$, as well as the bases $\hat{\mathbf{g}}^{(n)}[\sigma_{k},z']$ and $\hat{\mathbf{g}}^{(n)}$, are related by the change-of-basis matrices $B^{(n)}_{k}(z)$ and $B^{(n)}_{k}(z')$, respectively. Therefore, by the \textit{change-of-basis formula}, we deduce that 
\begin{equation*}
\tensor[_{\hat{\mathbf{g}}^{(n)}[\sigma_{k},z'] }]{ \big[\, T_{\lambda}^{(n)}\,\big] }{_{\hat{\mathbf{f}}^{(n)}[\sigma_{k},z]}}=\big( B^{(n)}_{k}(z')\big)^{-1}\cdot D(\vec{u})\cdot B^{(n)}_{k}(z)\, . 
\end{equation*}
Consequently, since $\tensor[_{\hat{\mathbf{g}}^{(n)}[\sigma_{k},z'] }]{ \big[\, T_{\lambda}^{\,(n)}\,\big] }{_{\hat{\mathbf{f}}^{(n)}[\sigma_{k},z]}}$ must be upper-triangular, a direct application of Lemma~\eqref{lemma for braid matrices and diagonal matrices} yields that:
\begin{itemize}
\justifying
\item The $(k+1,k)$-entry of the matrix product $\big(B^{(n)}_{k}(z')\big)^{-1}\cdot D(\vec{u})\cdot B^{(n)}_{k}(z)$ must vanish.
\item Under the above condition, $\tensor[_{\hat{\mathbf{g}}^{(n)}[\sigma_{k},z'] }]{ \big[\, T_{\lambda}^{\,(n)}\,\big] }{_{\hat{\mathbf{f}}^{(n)}[\sigma_{k},z]}}=D(\pi_{k}(\vec{u}))$, 
\end{itemize}
where $\pi_{k}(\vec{u})$ denotes the permutation of $\vec{u}$ associated with $\sigma_{k}$. This completes the proof. 
\end{proof}

Having established the preceding lemmas, we are now in a position to state and prove the first part of one of our main theorems.

\begin{theorem}\label{Theorem: Ext0 as the kernel of delta_F,G}
\textbf{Setup}: Let $\beta=\sigma_{i_{1}}\cdots\sigma_{i_{\ell}}\in\mathrm{Br}^{+}_{n}$ be a positive braid word, $\mathcal{U}_{\Lambda(\beta)}=\big\{U_{0}, U_{\mathrm{B}}, U_{\mathrm{L}}, U_{\mathrm{R}}, U_{\mathrm{T}}\big\}$ the open cover of $\mathbb{R}^{2}$ from Construction~\eqref{Cons: Finite open cover for R^2}, and $\sh{F}$ and $\sh{G}$ objects of the category $\ccs{1}{\beta}$. 

\noindent
$\star$ \emph{Local descriptions on $U_{\mathrm{T}}$ and $U_{\mathrm{L}}$}: According to Lemma~\eqref{Lemma: linear map description of an object on the regions U_T, U_L, and U_R}, $\sh{F}$ and $\sh{G}$ have the following local descriptions: 
\begin{itemize}
\justifying
\item  On $U_{\mathrm{T}}$, $\sh{F}$ and $\sh{G}$ are specified by two collections of $n-1$ surjective linear maps 
\begin{equation*}
\big\{\psi_{\sh{F}}^{(i)}:\mathbb{K}^{i+1}\to \mathbb{K}^{i}\big\}_{i=1}^{n-1}\, , \quad \text{and} \quad \big\{\psi_{\sh{G}}^{(i)}:\mathbb{K}^{i+1}\to \mathbb{K}^{i}\big\}_{i=1}^{n-1}\, ,    
\end{equation*}
respectively. For a schematic illustration of a generic representative of one of these sheaves on $U_{\mathrm{T}}$, see Figure~\eqref{Fig: an object F in the region U_T}.

\item On $U_{\mathrm{L}}$, $\sh{F}$ and $\sh{G}$ are specified by two collections of $n-1$ injective linear maps 
\begin{equation*}
\big\{\phi_{\sh{F}}^{(i)}:\mathbb{K}^{i}\to \mathbb{K}^{i+1}\big\}_{i=1}^{n-1}\, , \quad \text{and} \quad \big\{\phi_{\sh{G}}^{(i)}:\mathbb{K}^{i}\to \mathbb{K}^{i+1}\big\}_{i=1}^{n-1} \, ,    
\end{equation*}
respectively. For a schematic illustration of a generic representative of one of these sheaves on $U_{\mathrm{L}}$, see Figure~\eqref{Fig: an object F in the region U_L}.

\item \textbf{Compatibility conditions}: For each $i\in[1,n-1]$, 
\begin{equation*}
\psi^{(i)}_{\sh{F}}\circ \phi^{(i)}_{\sh{F}}=\mathrm{id}_{\mathbb{K}^{i}}\, , \quad \text{and} \quad \psi^{(i)}_{\sh{G}}\circ \phi^{(i)}_{\sh{G}}=\mathrm{id}_{\mathbb{K}^{i}}\, .    
\end{equation*}
\end{itemize}

\noindent
$\star$ \emph{Global flag data}: By Theorem~\eqref{Flags and constructible sheaves}, $\sh{F}$ and $\sh{G}$ are geometrically characterized by two sequence of complete flags $\big\{\fl{F}_{j}\big\}_{j=0}^{\ell+1}$ and $\big\{\fl{G}_{j}\big\}_{j=0}^{\ell+1}$ in $\mathbb{K}^{n}$, respectively, such that: 
\begin{itemize}
\item $\fl{F}_{0}$ is completely opposite to both $\fl{F}_{1}$ and $\fl{F}_{\ell+1}$, and for each $j\in[1,\ell]$, $\fl{F}_{j}$ is in $s_{i_{j}}$-relative position with respect to $\fl{F}_{j+1}$.
\item $\fl{G}_{0}$ is completely opposite to both $\fl{G}_{1}$ and $\fl{G}_{\ell+1}$, and for each $j\in[1,\ell]$, $\fl{G}_{j}$ is in $s_{i_{j}}$-relative position with respect to $\fl{G}_{j+1}$.
\end{itemize}
In particular, by Lemma~\eqref{lemma for F in the region U_{R}}, we know that:
\begin{equation*}
\begin{aligned}
\fl{F}_{0}:=\prescript{}{\mathcal{K}\,}{\fl{F}}\big(\psi_{\sh{F}}^{(1)},\dots, \psi_{\sh{F}}^{(n-1)}\big)\, , \quad \text{and} \quad
\fl{F}_{1}:=\prescript{}{\mathcal{I}\,}{\fl{F}}\big(\phi_{\sh{F}}^{(1)},\dots, \phi_{\sh{F}}^{(n-1)}\big)\, ,\\[4pt]    
\fl{G}_{0}:=\prescript{}{\mathcal{K}\,}{\fl{F}}\big(\psi_{\sh{G}}^{(1)},\dots, \psi_{\sh{G}}^{(n-1)}\big)\, , \quad \text{and} \quad
\fl{G}_{1}:=\prescript{}{\mathcal{I}\,}{\fl{F}}\big(\phi_{\sh{G}}^{(1)},\dots, \phi_{\sh{G}}^{(n-1)}\big)\, ,\\
\end{aligned}
\end{equation*}
are the type $\mathcal{K}$ and type $\mathcal{I}$ flags in $\mathbb{K}^{n}$ associated with $\big\{\psi_{\sh{F}}^{(i)}\big\}_{i=1}^{n-1}$, $\big\{\phi_{\sh{F}}^{(i)}\big\}_{i=1}^{n-1}$, $\big\{\psi_{\sh{G}}^{(i)}\big\}_{i=1}^{n-1}$, and $\big\{\phi_{\sh{G}}^{(i)}\big\}_{i=1}^{n-1}$, respectively (cf. Definition~\eqref{Def:flags and adapted bases}--\eqref{Def: type I flag}--\eqref{Def: type K flag}). 

\vspace{4pt}
\noindent
$\star$ \emph{Main assumption} (Adapted bases): Let \,$\hat{\mathbf{f}}^{(n)}, \, \hat{\mathbf{g}}^{(n)} $ be bases for $\mathbb{K}^{n}$, and let $\vec{z}=(z_{1},\dots, z_{\ell}), \; \vec{z}\,'=(z'_{1},\dots, z'_{\ell}) \in X(\beta,\mathbb{K})$ be points such that the pairs $(\,\hat{\mathbf{f}}^{(n)}, \,\vec{z}\, )$ and $(\,\hat{\mathbf{g}}^{(n)}, \,\vec{z}\,'\, )$ algebraically characterizes $\sh{F}$ and $\sh{G}$ according to Theorem~\eqref{Prop. for sheaves and braid matrices}, respectively. In this setting, we have that: 
\begin{itemize}
\item Relative to the basis $\hat{\mathbf{f}}^{(n)}$ for $\mathbb{K}^{n}$: $\fl{F}_{0}$ and $\fl{F}_{1}$ are the anti-standard and standard flags, respectively.  For each $j\in[1,\ell]$, the flag $\fl{F}_{j+1}$ is represented by the path matrix $P_{\beta_{j}}(\vec{z}_{j})=B^{(n)}_{i_{1}}(z_{1})\cdots B^{(n)}_{i_{j}}(z_{j})\,\in\mathrm{GL}(n,\mathbb{K})$ associated with the truncated braid word $\beta_{j}=\sigma_{i_{1}}\cdots \sigma_{i_{j}}\in\mathrm{Br}^{+}_{n}$ and the truncated tuple $\vec{z}_{j}=(z_{1},\dots, z_{j})\in\mathbb{K}^{j}_{\mathrm{std}}$.  

\item Relative to the basis $\hat{\mathbf{g}}^{(n)}$ for $\mathbb{K}^{n}$: $\fl{G}_{0}$ and $\fl{G}_{1}$ are the anti-standard and standard flags, respectively. For each $j\in[1,\ell]$, the flag $\fl{G}_{j+1}$ is represented by the path matrix $P_{\beta_{j}}(\vec{z}\,'_{j})=B^{(n)}_{i_{1}}(z'_{1})\cdots B^{(n)}_{i_{j}}(z'_{j})\,\in\mathrm{GL}(n,\mathbb{K})$ associated with the truncated braid word $\beta_{j}$ and the truncated tuple $\vec{z}\,'_{j}=(z'_{1},\dots, z'_{j})\in\mathbb{K}^{j}_{\mathrm{std}}$.  
\end{itemize}

Furthermore, under this assumption, the \textbf{compatibility conditions} and an inductive argument ensure that, for each $i\in [1,n-1]$, there are unique bases $\hat{\mathbf{f}}^{(i)}:=\big\{ \hat{f}^{(i)}_{j}\big\}_{j=1}^{i}$ and $\hat{\mathbf{g}}^{(i)}:=\big\{ \hat{g}^{(i)}_{j}\big\}_{j=1}^{i}$ for $\mathbb{K}^{i}$ such that (cf. Definition~\eqref{Def:flags and adapted bases}--\eqref{Def: adapted bases I}--\eqref{Def: adapted bases II}): 
\begin{itemize}
\item The collection $\big\{\hat{\mathbf{f}}^{(i)}\big\}_{i=1}^{n}$ is a system of bases adapted to both $\big\{\psi_{\sh{F}}^{(i)}\big\}_{i=1}^{n-1}$ and $\big\{\phi_{\sh{F}}^{(i)}\big\}_{i=1}^{n-1}$.
\item The collection $\big\{\hat{\mathbf{g}}^{(i)}\big\}_{i=1}^{n}$ is a system of bases adapted to both $\big\{\psi_{\sh{G}}^{(i)}\big\}_{i=1}^{n-1}$ and $\big\{\phi_{\sh{G}}^{(i)}\big\}_{i=1}^{n-1}$.
\end{itemize}
Building on this, for each $j\in[1,\ell]$, we denote by $\big\{\hat{\mathbf{f}}^{(i)}[\beta_{j},\vec{z}_{j} ]\big\}_{i=1}^{n}$ and $\big\{\hat{\mathbf{g}}^{(i)}[\beta_{j},\vec{z}\,'_{j} ]\big\}_{i=1}^{n}$ the braid-transformed bases obtained from $\big\{\hat{\mathbf{f}}^{(i)} \big\}_{i=1}^{n}$ and $\big\{\hat{\mathbf{g}}^{(i)} \big\}_{i=1}^{n}$ via the truncated braid word $\beta_{j}$ and the truncated tuples $\vec{z}_{j}$ and $\vec{z}\,'_{j}$, respectively (cf. Definition~\eqref{Def: braid transformation of bases}). Accordingly, by Theorem~\eqref{Theorem: adapted bases for an object at a region R_j}, we have that: 
\begin{itemize}
\justifying
\item $\big(\big\{\hat{\mathbf{f}}^{(i)} \big\}_{i=1}^{n}, z_{1}\big)$ is a system of bases adapted to $\sh{F}$ on $R_{1}$.
\item For each $j\in[1,\ell-1]$, $\big(\big\{\hat{\mathbf{f}}^{(i)}[\beta_{j},\vec{z}_{j} ]\big\}_{i=1}^{n}, z_{j+1}\big)$ is a system of bases adapted to $\sh{F}$ on $R_{j+1}$.
\item $\big(\big\{\hat{\mathbf{g}}^{(i)} \big\}_{i=1}^{n}, z'_{1}\big)$ is a system of bases adapted to $\sh{G}$ on $R_{1}$.
\item For each $j\in[1,\ell-1]$, $\big(\big\{\hat{\mathbf{g}}^{(i)}[\beta_{j},\vec{z}\,'_{j} ]\big\}_{i=1}^{n}, z'_{j+1}\big)$ is a system of bases adapted to $\sh{G}$ on $R_{j+1}$.
\end{itemize}

\vspace{4pt}
\noindent
$\star$ \emph{Main Conclusion}: Following Definition~\eqref{Def: linear map delta}, let \,$\delta_{\sh{F},\sh{G}}:\mathbb{K}^{n}_{\mathrm{std}}\to\mathbb{K}^{\ell}_{\mathrm{std}}$\, be the linear map associated with the pair $(\sh{F},\sh{G})$. Then, under the given \textbf{setup}, there is an isomorphism of vector spaces 
\begin{equation*}
\mathrm{Ext}^{0}(\sh{F},\sh{G})\cong \ker \delta_{\sh{F},\sh{G}}\, .
\end{equation*}
\end{theorem}
\begin{proof}
To begin, let $\mathcal{S}_{\Lambda(\beta)}$ be the stratification of $\mathbb{R}^{2}$ induced by $\Lambda(\beta)\subset \big(\mathbb{R}^{3},\xi_{\mathrm{std}}\big)$, $\mathcal{U}_{\Lambda(\beta)}=\big\{U_{0}, U_{\mathrm{B}}, U_{\mathrm{L}}, U_{\mathrm{R}}, U_{\mathrm{T}}\big\}$ the open cover of $\mathbb{R}^{2}$ from Construction~\eqref{Cons: Finite open cover for R^2}, and $\mathcal{R}_{\Lambda(\beta)}=\big\{R_{j}\big\}_{j=1}^{\ell}$ the partition of $U_{\mathrm{B}}$ into $\ell$ open vertical straps from Construction~\eqref{Cons: Definition of the vertical straps}.    

Let $\lambda\in\mathrm{Ext}^{0}(\sh{F},\sh{G})$. By Lemma~\eqref{Lemma: Ext0 on U_T and U_L}, on $U_{\mathrm{T}}$ and $U_{\mathrm{L}}$, $\lambda$ is characterized by a collection of $n$ linear maps $\big\{ T^{(i)}_{\lambda}:\mathbb{K}^{i}\to \mathbb{K}^{i} \big\}_{i=1}^{n}$ such that for each $i\in [1,n-1]$: 
\begin{equation*}
T^{(i)}_{\lambda}\circ \psi^{(i)}_{\sh{F}}=\psi^{(i)}_{\sh{G}}\circ T^{(i+1)}_{\lambda}\, , \quad \text{and} \quad T^{(i+1)}_{\lambda}\circ \phi^{(i)}_{\sh{F}}=\phi^{(i)}_{\sh{G}}\circ T^{(i)}_{\lambda}\, .    
\end{equation*} 

By assumption, $\big\{\mathbf{\hat{f}}^{(i)} \big\}_{i=1}^{n}$ and $\big\{\mathbf{\hat{g}}^{(i)} \big\}_{i=1}^{n}$ are systems of bases adapted to the pairs $\big(\big\{\psi_{\sh{F}}^{(i)}\, \big\}_{i=1}^{n-1}\,,\, \big\{\phi_{\sh{F}}^{(i)}\, \big\}_{i=1}^{n-1}\big)$ and $\big( \big\{\psi_{\sh{G}}^{(i)}\, \big\}_{i=1}^{n-1}\,,\, \big\{\phi_{\sh{G}}^{(i)}\, \big\}_{i=1}^{n-1}\big)$, respectively. Hence, by Lemma~\eqref{Lemma: Ext0 on U_T and U_L}, we obtain that, $\tensor[_{\hat{\mathbf{g}}^{(n)}}]{ \big[\, T_{\lambda}^{\,(n)}\,\big] }{_{\hat{\mathbf{f}}^{(n)}}}\in M(n,\mathbb{K})$ is diagonal, that is, $\tensor[_{\hat{\mathbf{g}}^{(n)}}]{ \big[\, T_{\lambda}^{\,(n)}\,\big] }{_{\hat{\mathbf{f}}^{(n)}}}=D(\vec{u})$
for some tuple $\vec{u}=(u_{1},\dots, u_{n})\in\mathbb{K}^{n}_{\mathrm{std}}$. 

Next, consider $R_{1}$, the vertical strap in $\mathbb{R}^{2}$ containing $\sigma_{i_{1}}$---the first crossing of $\beta$. In particular, we denote by $k=i_{1}\in [1,n-1]$ the index of $\sigma_{i_{1}}$. By construction, $R_{1}$ is the vertical strap immediately to the right of $U_{\mathrm{L}}$. Thus, by the constructibility of $\sh{F}$ and $\sh{G}$ with respect to $\mathcal{S}_{\Lambda(\beta)}$, we obtain one of the following two cases:
\begin{itemize}
\justifying
\item \textit{Case 1}: Suppose that $k=1$. Then, on $R_{1}$, $\sh{F}$ and $\sh{G}$ are described by two collections of $n$ injective linear maps: 
\begin{equation*}
\begin{aligned}
&\big\{\phi^{\,(i)}_{\sh{F}}:\mathbb{K}^{i}\to \mathbb{K}^{i+1}\,\big\}_{i=1}^{n-1}\,\cup\,\big\{\widetilde{\phi}^{\,(1)}_{\sh{F}}:\mathbb{K}^{1}\to \mathbb{K}^{2}\, \big\}\, ,\\[4pt]
&\big\{\phi^{\,(i)}_{\sh{G}}:\mathbb{K}^{i}\to \mathbb{K}^{i+1}\,\big\}_{i=1}^{n-1}\,\cup\,\big\{\widetilde{\phi}^{\,(1)}_{\sh{G}}:\mathbb{K}^{1}\to \mathbb{K}^{2}\, \big\} \, ,
\end{aligned}    
\end{equation*}
respectively. In this local configuration, Lemma~\eqref{Lemma: Ext0 on R_j for k=1} asserts that $\lambda$ is described by a collection of $n+1$ linear maps
\begin{equation*}
\big\{ T_{\lambda}^{(i)}:\mathbb{K}^{i}\to \mathbb{K}^{i}\big\}_{i=1}^{n}\,\cup\, \big\{ \widetilde{T}_{\lambda}^{\,(1)}:\mathbb{K}^{1}\to \mathbb{K}^{1}\big\}\, .
\end{equation*}

\item \textit{Case 2}: Suppose that $k\geq 2$. Then, on $R_{1}$, $\sh{F}$ and $\sh{G}$ are characterized by two collections of $n+1$ injective linear maps: 
\begin{equation*}
\begin{aligned}
&\big\{ \phi^{\,(i)}_{\sh{F}}:\mathbb{K}^{i}\to \mathbb{K}^{i+1}\,\big\}_{i=1}^{n-1}\,\cup\,\big\{ \widetilde{\phi}^{\,(k-1)}_{\sh{F}}:\mathbb{K}^{k-1}\to\mathbb{K}^{k}\,, ~\widetilde{\phi}^{\,(k)}_{\sh{F}}:\mathbb{K}^{k}\to \mathbb{K}^{k+1}\, \big\}\, ,\\[4pt]
&\big\{ \phi^{\,(i)}_{\sh{G}}:\mathbb{K}^{i}\to \mathbb{K}^{i+1}\,\big\}_{i=1}^{n-1}\,\cup\,\big\{ \widetilde{\phi}^{\,(k-1)}_{\sh{G}}:\mathbb{K}^{k-1}\to \mathbb{K}^{k}\,, ~\widetilde{\phi}^{\,(k)}_{\sh{G}}:\mathbb{K}^{k}\to \mathbb{K}^{k+1}\, \big\}\, ,
\end{aligned}    
\end{equation*}
respectively. In this local configuration, Lemma~\eqref{Lemma: Ext0 for R_j for k geq 2} asserts that $\lambda$ is described by a collection of $n+1$ linear maps 
\begin{equation*}
\big\{ T_{\lambda}^{(i)}:\mathbb{K}^{i}\to \mathbb{K}^{i}\big\}_{i=1}^{n}\,\cup\, \big\{\widetilde{T}_{\lambda}^{\,(k)}:\mathbb{K}^{k}\to \mathbb{K}^{k}\big\}\, .
\end{equation*}
\end{itemize}
In any of the above cases, $\big\{\phi^{(i)}_{\sh{F}}\big\}_{i=1}^{n-1}$, $\big\{\phi^{(i)}_{\sh{G}}\big\}_{i=1}^{n-1}$, and $\big\{T^{(i)}_{\lambda}\big\}_{i=1}^{n}$ are precisely the linear maps characterizing $\sh{F}$, $\sh{G}$, and $\lambda$ on $U_{\mathrm{L}}$.

Now, let $\vec{z}=(z_{1},\dots, z_{\ell})$ and $\vec{z}\,'=(z'_{1},\dots, z'_{\ell})$ be the points in the braid variety $X(\beta,\mathbb{K})$ such that the pairs $\big(\,\mathbf{\hat{f}}^{(n)}\,, \vec{z}\, \big)$ and $\big(\,\mathbf{\hat{g}}^{(n)}\,, \vec{z}\,'\, \big)$ algebraically characterize $\sh{F}$ and $\sh{G}$ according to Theorem~\eqref{Prop. for sheaves and braid matrices}, respectively. By Theorem~\eqref{Theorem: adapted bases for an object at a region R_j}, we know that: 
\begin{itemize}
\justifying
\item $\big(\big\{\hat{\mathbf{f}}^{(i)} \big\}_{i=1}^{n}, z_{1}\big)$ is a system of bases adapted to $\sh{F}$ on $R_{1}$.
\item $\big(\big\{\hat{\mathbf{g}}^{(i)} \big\}_{i=1}^{n}, z'_{1}\big)$ is a system of bases adapted to $\sh{G}$ on $R_{1}$.
\end{itemize}
Then, a direct application of either Lemma~\eqref{Lemma: Ext0 on R_j for k=1} or Lemma~\eqref{Lemma: Ext0 for R_j for k geq 2} depending on the value of $k$, considering that $k=i_{1}$, implies that:
\begin{itemize}
\justifying    
\item \textbf{Compatibility condition for $\lambda$ at $\sigma_{i_{1}}$}: The $(i_{1}+1,i_{1})$-entry of $\big(B^{(n)}_{i_{1}}(z_{1}')\big)^{-1} D(\vec{u})\, B^{(n)}_{i_{1}}(z_{1})$ must vanish. 
\item Under the above condition,
\begin{equation*}
\tensor[_{\hat{\mathbf{g}}^{(n)}[\sigma_{i_{1}},z'_{1}]}]{ \big[\, T_{\lambda}^{(n)}\,\big] }{_{\hat{\mathbf{f}}^{(n)}[\sigma_{i_{1}},z_{1}]}}=D(\pi_{i_{1}}(\vec{u}))\,,    
\end{equation*}
where $\pi_{i_{1}}(\vec{u})$ denotes the permutation of $\vec{u}$ associated with $\sigma_{i_{1}}$. 
\end{itemize}

Applying the same reasoning to $R_{2}$, the vertical strap in $\mathbb{R}^{2}$ containing $\sigma_{i_{2}}$---the second crossing of $\beta$---we obtain that: 
\begin{itemize}
\justifying    
\item \textbf{Compatibility condition for $\lambda$ at $\sigma_{i_{2}}$}: The $(i_{2}+1,i_{2})$-entry of $\big(B^{(n)}_{i_{2}}(z_{2}')\big)^{-1} D(\pi_{i_{1}}(\vec{u}))\, B^{(n)}_{i_{2}}(z_{2})$ must vanish. 

\item Under the above condition,
\begin{equation*}
\begin{aligned}
\tensor[_{\hat{\mathbf{g}}^{(n)}[\beta_{2},\vec{z}\,'_{2} ] }]{ \big[\, T_{\lambda}^{(n)}\,\big] }{_{\hat{\mathbf{f}}^{(n)}[\beta_{2} ,\vec{z}_{2}]}}&=D\big(\pi_{i_{2}}(\pi_{i_{1}}(\vec{u}))\big)\, ,\\[4pt]
&=D(\pi_{i_{1}\cdot\, i_{2}}(\vec{u}))\,,\\[4pt]
&=D(\pi_{\beta_{2}}(\vec{u}))\, ,
\end{aligned} 
\end{equation*}
\end{itemize}
where $\pi_{\beta_{2}}(\vec{u})$ denotes the permutation of $\vec{u}$ associated with $\beta_{2}=\sigma_{i_{1}}\sigma_{i_{2}}\in \mathrm{Br}^{+}_{n}$, the truncation of $\beta$ at its second crossing, and 
$\vec{z}_{2}=(z_{1}, z_{2}),\,\vec{z}\,'_{2}=(z'_{1}, z'_{2})\in\mathbb{K}^{2}_{\mathrm{std}}$ denote the truncation of the tuples $\vec{z}$ and $\vec{z}\,'$ at their second entry, respectively. 

Proceeding iteratively through all the vertical straps $R_{j}$, we obtain $\ell$ compatibility conditions for $\lambda$, one for each crossing $\sigma_{i_{j}}$ of $\beta$. More precisely, for each $j\in[1,\ell]$, we deduce that
\begin{equation*}
\Big[\,\big(B^{(n)}_{i_{j}}(z'_{j})\big)^{-1}D(\pi_{\beta_{j-1}}(\vec{u}))\, B^{(n)}_{i_{j}}(z_{j})\, \Big]_{i_{j}+1, i_{j}}=0\, .  
\end{equation*}
In particular, building on Definition~\eqref{Def: linear map delta}, we observe that the above compatibility conditions are precisely those characterizing the elements of the kernel of the linear map $\delta_{\sh{F},\sh{G}}$, which allows us to conclude that
\begin{equation*}
\mathrm{Ext}^{0}(\sh{F},\sh{G})\cong \mathrm{ker}\,\delta_{\sh{F},\sh{G}}\, .    
\end{equation*}
\end{proof}

Having established the previous result, we conclude the analysis of the zero-degree morphism spaces in the category $\ccs{1}{\beta}$ and proceed to study the structure of its one-degree morphism spaces. 

\subsection{\texorpdfstring{One-Degree Morphism Spaces in the Category $\ccs{1}{\beta}$}{One-degree Morphism Spaces in the Category HSh}} 
\noindent
Let $\beta \in \mathrm{Br}^{+}_{n}$ be a positive braid word. The main goal of this subsection is to establish the second part of Theorem~\eqref{Theorem: Lower-degree morphisms}, namely, the isomorphism of vector spaces in Equation~\eqref{Eq: iso for Ext1}, thereby providing an explicit algebraic characterization of the one-degree morphism spaces in the category $\ccs{1}{\beta}$. To this end, we begin by introducing some preliminaries. 

\subsubsection{Technical Background}
Let $\beta \in \mathrm{Br}^{+}_{n}$ be a positive braid word. We now collect some preliminaries that will play a fundamental role in the explicit computation of the one-degree morphism spaces in the category $\ccs{1}{\beta}$. In particular, we begin by establishing some notation.

\begin{notation}
Let $\mathscr{C}$ be a category, and suppose that $A$ is an object of the category $\mathscr{C}$. Then, when drawing diagrams, we introduce $A =\joinrel=\joinrel=\joinrel= \joinrel= A$ to denote the identity morphism between $A$ and $A$.       
\end{notation}

\begin{definition}
Let $\mathcal{M}$ be a smooth manifold, and let $\sh{F}$ and $\sh{G}$ be sheaves of $\mathbb{K}$-modules on $\mathcal{M}$. A sheaf $\sh{H}$ of $\mathbb{K}$-modules on $\mathcal{M}$ is called an \emph{extension of $\sh{F}$ by $\sh{G}$} if there exists a short exact sequence of the form
\begin{equation}
0 \longrightarrow \sh{G} \longrightarrow \sh{H} \longrightarrow \sh{F} \longrightarrow 0\, .   
\end{equation}      

In particular, two extensions $\sh{H}$ and $\sh{H}'$ of $\sh{F}$ by $\sh{G}$ are said to be \emph{equivalent} if there exists a sheaf isomorphism $\lambda:\sh{H}\to \sh{H}'$ such that the diagram in Figure~\eqref{Equivalen extenions of F by F'} commutes in each square.
\begin{figure}[ht]
\centering
\begin{tikzpicture}
\useasboundingbox (-2,-1.5) rectangle (2,1.5);
\scope[transform canvas={scale=1.2}]

\node at (-4,1) {\footnotesize $0$};
\node at (-2,1) {\footnotesize $\sh{G}$};
\node at (0,1) {\footnotesize $\sh{H}$};
\node at (2,1) {\footnotesize $\sh{F}$};
\node at (4,1) {\footnotesize $0$};

\node at (-4,-1) {\footnotesize $0$};
\node at (-2,-1) {\footnotesize $\sh{G}$};
\node at (0,-1) {\footnotesize $\sh{H}'$};
\node at (2,-1) {\footnotesize $\sh{F}$};
\node at (4,-1) {\footnotesize $0$};

\draw[->] (-4+0.3,1) -- (-2-0.30,1);
\draw[->] (-2+0.3,1) -- (-0-0.30,1);
\draw[->] (0+0.3,1) -- (2-0.30,1);
\draw[->] (2+0.3,1) -- (4-0.30,1);

\draw[->] (-4+0.3,-1) -- (-2-0.30,-1);
\draw[->] (-2+0.3,-1) -- (-0-0.30,-1);
\draw[->] (0+0.3,-1) -- (2-0.30,-1);
\draw[->] (2+0.3,-1) -- (4-0.30,-1);

\draw[-] (-2-0.05,1-0.30) -- (-2-0.05,-1+0.30);
\draw[-] (-2+0.05,1-0.30) -- (-2+0.05,-1+0.30);

\draw[-] (2-0.05,1-0.30) -- (2-0.05,-1+0.30);
\draw[-] (2+0.05,1-0.30) -- (2+0.05,-1+0.30);

\draw[->] (0,1-0.30) -- (0,-1+0.30);
\node[right] at (0,0) {\footnotesize $\lambda$};

\endscope
\end{tikzpicture}
\caption{Two equivalent extensions $\sh{H}$ and $\sh{H}'$ of $\sh{F}$ by $\sh{G}$.}
\label{Equivalen extenions of F by F'}
\end{figure}
\end{definition}

Now, recall that the category of sheaves of $\mathbb{K}$-modules on a smooth manifold is an abelian category with enough injectives~\cite{Hart1, KS1}. The following result is standard in classical sheaf theory and also follows from the general theory of abelian categories (see, for instance,~\cite{HS1}).

\begin{lemma}\label{Key lemma for Ext^{1} for sheaves}
Let $\mathcal{M}$ be a smooth manifold, and let $\sh{F}$ and $\sh{G}$ be sheaves of $\mathbb{K}$-modules on $\mathcal{M}$. Then there exists a canonical isomorphism
\begin{equation}
\mathrm{Ext}^{1}(\sh{F}, \sh{G}) \cong \Big\{   \text{equivalence classes of extensions of $\sh{F}$ by $\sh{G}$}  \Big\}\, .    
\end{equation}
\end{lemma}

Let $\sh{F}$ and $\sh{G}$ be objects of the category $\ccs{1}{\beta}$. Building on Lemma~\eqref{Key lemma for Ext^{1} for sheaves}, we proceed to characterize $\mathrm{Ext}^{1}(\sh{F}, \sh{G})$ via the equivalence classes of extensions of $\sh{F}$ by $\sh{G}$. To this end, we introduce several useful definitions and preliminary results. 

\begin{definition}\label{Def: Extensions of linear maps}
Let $A$, $B$, $C$, $X$, $Y$, $Z$ be vector spaces over $\mathbb{K}$, and let $\gamma_{\sh{G}}:A\to X$, $\Gamma_{\sh{H}}:B\to Y$, and $\gamma_{\sh{F}}:C\to Z$ be linear maps. We say that $\Gamma_{\sh{H}}$ is an \emph{extension of $\gamma_{\sh{F}}$ by $\gamma_{\sh{G}}$} if there exists a diagram as in Figure~\eqref{fig: diagram for extensions of linear maps} where the horizontal rows are short exact sequences and each square commutes. 

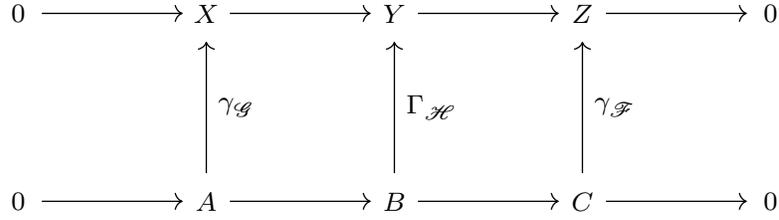
\begin{figure}[ht]
\centering
\begin{tikzpicture}
\useasboundingbox (-2,-1.65) rectangle (2,1.65);
\scope[transform canvas={scale=1.25}]

\node at (-4,1) {\footnotesize $0$};
\node at (-2, 1) {\footnotesize $X$};
\node at (0,1) {\footnotesize $Y$};
\node at (2,1) {\footnotesize $Z$};
\node at (4,1) {\footnotesize $0$};

\node at (-4,-1) {\footnotesize $0$};
\node at (-2, -1) {\footnotesize $A$};
\node at (0,-1) {\footnotesize $B$};
\node at (2,-1) {\footnotesize $C$};
\node at (4,-1) {\footnotesize $0$};

\draw[->] (-4 +0.25,1) -- (-2-0.25,1);
\draw[->] (-2+0.25,1) -- (-0-0.25,1);
\draw[->] (0+0.25,1) -- (2-0.25,1);
\draw[->] (2+0.25,1) -- (4-0.25,1);

\draw[->] (-4 +0.25,-1) -- (-2-0.25,-1);
\draw[->] (-2+0.25,-1) -- (-0-0.25,-1);
\draw[->] (0+0.25,-1) -- (2-0.25,-1);
\draw[->] (2+0.25,-1) -- (4-0.25,-1);

\draw[->] (0,-1+0.30)  -- (0,1-0.30) ;
\node[right] at (0,0) {\footnotesize $\Gamma_{\sh{H}}$};

\draw[->] (-2,-1+0.30) -- (-2,1-0.30);
\node[right] at (-2,0) {\footnotesize $\gamma_{\sh{G}}$};

\draw[->] (2, -1+0.30) -- (2,1-0.30);
\node[right] at (2,0) {\footnotesize $\gamma_{\sh{F}}$};

\endscope
\end{tikzpicture}
\caption{An extension $\Gamma_{\sh{H}}$ of $\gamma_{\sh{F}}$ by $\gamma_{\sh{G}}$.}
\label{fig: diagram for extensions of linear maps}
\end{figure}
\end{definition}

\begin{definition}\label{Def: equivalent extensions of linear maps}
Let $A$, $B$, $B'$, $C$, $X$, $Y$, $Y'$, $Z$ be vector spaces over $\mathbb{K}$, and let $\gamma_{\sh{G}}:A\to X$, $\Gamma_{\sh{H}}:B\to Y$, $\Gamma_{\sh{H'}}:B'\to Y'$, and $\gamma_{\sh{F}}:C\to Z$ be linear maps, with $\Gamma_{\sh{H}}$ and $\Gamma_{\sh{H'}}$ realizing extensions of $\gamma_{\sh{F}}$ by $\gamma_{\sh{G}}$. We say that $\Gamma_{\sh{H}}$ and $\Gamma_{\sh{H'}}$ are \emph{equivalent extensions of $\gamma_{\sh{F}}$ by $\gamma_{\sh{G}}$} if there exist linear isomorphisms $\xi: B \rightarrow B'$ and $\delta: Y \rightarrow Y'$ making the diagram in Figure~\eqref{fig: Commutative diagram for equivalent extensions of linear maps} commute in each square.

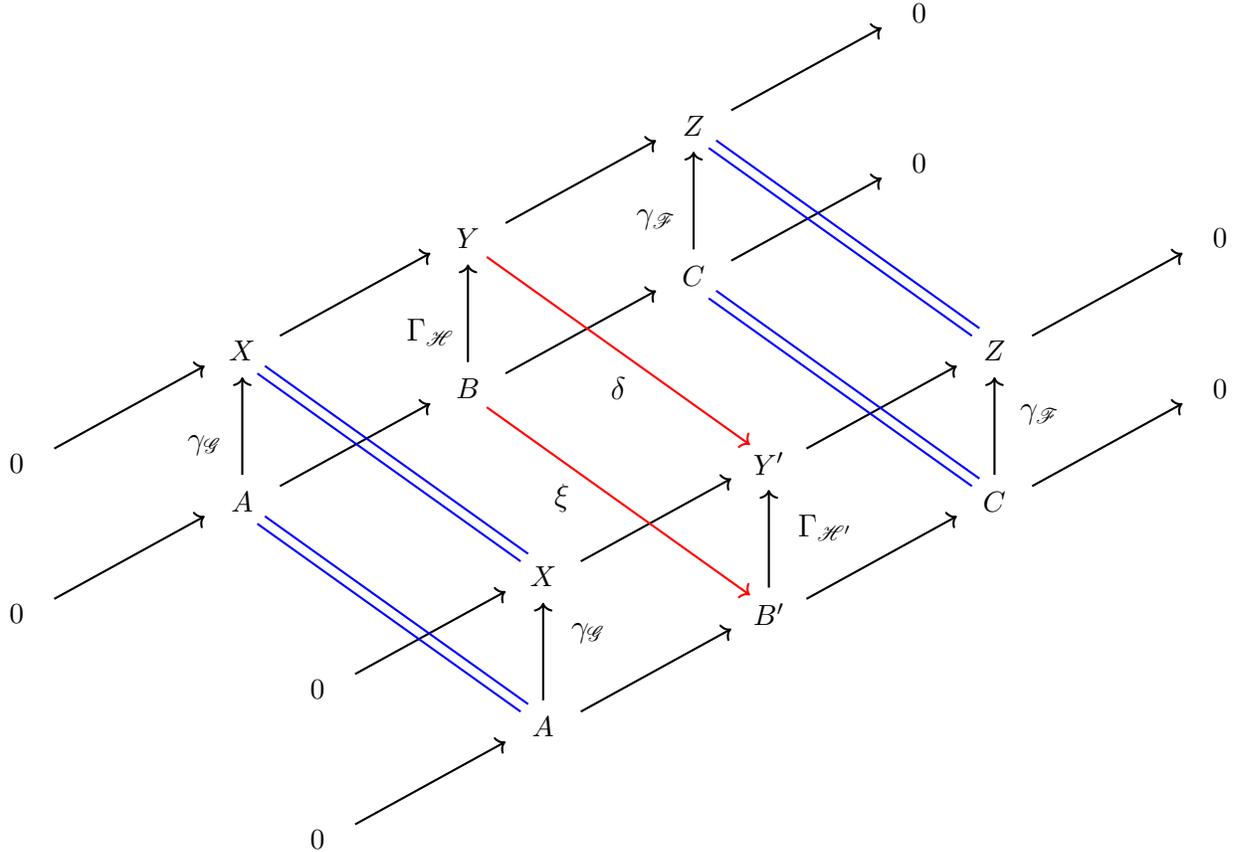
\begin{figure}[ht]
\centering
\begin{tikzpicture}
\useasboundingbox (-6,-5.5) rectangle (6,6.5);
\scope[transform canvas={scale=1}]

\node at (4,6) {\footnotesize\large $0$};
\node at (4,4) {\footnotesize\large $0$};

\node at (1,4.5) {\footnotesize\large $Z$};
\node at (1,2.5) {\footnotesize\large $C$};

\node at (-2,3) {\footnotesize\large $Y$};
\node at (-2,1) {\footnotesize\large $B$};

\node at (-5,1.5) {\footnotesize\large $X$};
\node at (-5,-0.5) {\footnotesize\large $A$};

\node at (-8,0) {\footnotesize\large $0$};
\node at (-8,-2) {\footnotesize\large $0$};

\node at (8,3) {\footnotesize\large $0$};
\node at (8,1) {\footnotesize\large $0$};

\node at (5,1.5) {\footnotesize\large $Z$};
\node at (5,-0.5) {\footnotesize\large $C$};

\node at (2,0) {\footnotesize\large $Y'$};
\node at (2,-2) {\footnotesize\large $B'$};

\node at (-1,-1.5) {\footnotesize\large $X$};
\node at (-1,-3.5) {\footnotesize\large $A$};

\node at (-4,-3) {\footnotesize\large $0$};
\node at (-4,-5) {\footnotesize\large $0$};

\draw[->, thick] (-8 +0.5,0+0.2) -- (-5-0.5,1.5-0.2);
\draw[->, thick] (-5 +0.5,1.5+0.2) -- (-2-0.5,3-0.2);
\draw[->, thick] (-2 +0.5,3+0.2) -- (1-0.5,4.5-0.2);
\draw[->, thick] (1 +0.5,4.5+0.2) -- (4-0.5,6-0.2);

\draw[->, thick] (-8 +0.5,-2+0.2) -- (-5-0.5,-0.5-0.2);
\draw[->, thick] (-5 +0.5,-0.5+0.2) -- (-2-0.5,1-0.2);
\draw[->, thick] (-2 +0.5,1+0.2) -- (1-0.5,2.5-0.2);
\draw[->, thick] (1 +0.5,2.5+0.2) -- (4-0.5,4-0.2);

\draw[->, thick] (-2,1+0.35) -- (-2,3-0.35);
\draw[->, thick] (1,2.5+0.35) -- (1,4.5-0.35);
\draw[->, thick] (-5,-0.5+0.35) -- (-5,1.5-0.35);

\draw[->, thick] (-4+0.5,-3+0.2) -- (-1-0.5,-1.5-0.2);
\draw[->, thick] (-1+0.5,-1.5+0.2) -- (2-0.5,0-0.2);
\draw[->, thick] (2+0.5,0+0.2) -- (5-0.5,1.5-0.2);
\draw[->, thick] (5+0.5,1.5+0.2) -- (8-0.5,3-0.2);

\draw[->, thick] (-4+0.5,-5+0.2) -- (-1-0.5,-3.5-0.2);
\draw[->, thick] (-1+0.5,-3.5+0.2) -- (2-0.5,-2-0.2);
\draw[->, thick] (2+0.5,-2+0.2) -- (5-0.5,-0.5-0.2);
\draw[->, thick] (5+0.5,-0.5+0.2) -- (8-0.5,1-0.2);

\draw[->, thick] (-1,-3.5+0.35) -- (-1,-1.5-0.35);
\draw[->, thick] (2,-2+0.35) -- (2,0-0.35);
\draw[->, thick] (5,-0.5+0.35) -- (5,1.5-0.35);

\draw[thick, blue] (-5+0.25-0.05, 1.5-0.25-0.05) -- (-1-0.25-0.05,-1.5+0.25-0.05);
\draw[thick, blue] (-5+0.25+0.05, 1.5-0.25+0.05) -- (-1-0.25+0.05,-1.5+0.25+0.05);

\draw[thick, blue] (-5+0.25-0.05, -0.5-0.25-0.05) -- (-1-0.25-0.05,-3.5+0.25-0.05);
\draw[thick, blue] (-5+0.25+0.05, -0.5-0.25+0.05) -- (-1-0.25+0.05,-3.5+0.25+0.05);

\draw[->, thick, red] (-2+0.25, 3-0.25) -- (2-0.25,0+0.25);
\draw[->, thick, red] (-2+0.25, 1-0.25) -- (2-0.25,-2+0.25);

\draw[thick, blue] (1+0.25-0.05, 4.5-0.25-0.05) -- (5-0.25-0.05,1.5+0.25-0.05);
\draw[thick, blue] (1+0.25+0.05, 4.5-0.25+0.05) -- (5-0.25+0.05,1.5+0.25+0.05);

\draw[thick, blue] (1+0.25-0.05, 2.5-0.25-0.05) -- (5-0.25-0.05,-0.5+0.25-0.05);
\draw[thick, blue] (1+0.25+0.05, 2.5-0.25+0.05) -- (5-0.25+0.05,-0.5+0.25+0.05);


\node at (-0.75,-0.45) {\Large $\xi$};
\node at (0,1) {\Large $\delta$};


\node at (-5-0.5,0.5-0.25) {\Large $\gamma_{\sh{G}}$};
\node at (-2-0.5,2-0.25) {\Large $\Gamma_{\sh{H}}$};
\node at (1-0.5,3.5-0.25) {\Large $\gamma_{\sh{F}}$};

\node at (-1+0.6,-2.5+0.25) {\Large $\gamma_{\sh{G}}$};
\node at (2+0.75,-1+0.15) {\Large $\Gamma_{\sh{H'}}$};
\node at (5+0.6,0.5+0.15) {\Large $\gamma_{\sh{F}}$};

\endscope
\end{tikzpicture}
\caption{Two equivalent extensions $\Gamma_{\sh{H}}$ and $\Gamma_{\sh{H'}}$ of $\gamma_{\sh{F}}$ by $\gamma_{\sh{G}}$.}
\label{fig: Commutative diagram for equivalent extensions of linear maps}
\end{figure}
\end{definition}

Throughout this subsection, we will implement the following notation. 

\begin{notation}
Let $A$, $B$, $X$, $Y$ be vector spaces over $\mathbb{K}$. Then we know that there exists a canonical isomorphism
\begin{equation*}
\mathrm{Hom}_{\,\mathbb{K}}(A\oplus B, X\oplus Y)\cong \mathrm{Hom}_{\,\mathbb{K}}(A, X)\oplus  \mathrm{Hom}_{\,\mathbb{K}}(B, X)  \oplus \mathrm{Hom}_{\,\mathbb{K}}(A,Y) \oplus \mathrm{Hom}_{\mathbb{K}}(B, Y)\,.     
\end{equation*}
More precisely, we have that, for any linear map $\xi:A \oplus B \to X\oplus Y$, there exist linear maps $\lambda^{(1)}_{\xi}: A\to X$, $\lambda^{(2)}_{\xi}: B\to X$, $\lambda^{(3)}_{\xi}: A\to Y$, and $\lambda^{(4)}_{\xi}: B\to Y$ such that 
\begin{equation}\label{Eq: decomposition for linear maps over direct sums of vector spaces}
\xi (a,b)=\big(\lambda^{(1)}_{\xi}(a)+\lambda^{(2)}_{\xi}(b), \lambda^{(3)}_{\xi}(a)+\lambda^{(4)}_{\xi}(b)\big)\, ,  \quad \text{for all $(a, b)\in A\oplus B$}\, .  
\end{equation}
Then, by abuse of notation, we denote the expression in equation~\eqref{Eq: decomposition for linear maps over direct sums of vector spaces} by
\begin{equation*}
\xi=\begin{bmatrix}
\lambda^{(1)}_{\xi} & \lambda^{(2)}_{\xi}\\
\lambda^{(3)}_{\xi} & \lambda^{(4)}_{\xi}
\end{bmatrix}\, .    
\end{equation*}
\end{notation}

\begin{definition}\label{Def: inclusion and projection maps}
Let $A$ and $B$ be vector spaces over $\mathbb{K}$. We denote by $\iota_A: A \to A \oplus B$ the canonical inclusion $\iota_A(a) := (a,0)$ for all $a \in A$, and by $\pi_B: A \oplus B \to B $ the canonical projection 
$\pi_B(a,b) := b$ for all $(a,b) \in A \oplus B$.
\end{definition}

The following lemma introduces a concept that will play a central role in the subsequent discussion.

\begin{lemma}\label{Lemma: block extensions of linear maps}
Let $A$, $C$, $X$, $Z$ be vector spaces over $\mathbb{K}$, and let  $\gamma_{\sh{G}}:A\to X$, $\gamma_{\sh{H}}:C\to X$, and $\gamma_{\sh{F}}:C\to Z$ be linear maps. Then the linear map $\tensor[_{\gamma_{\sh{G}}}]{ \langle \smash{\gamma_{\scriptscriptstyle \sh{H}}} \rangle }{_{\gamma_{\sh{F}}}}:A\oplus C \to X\oplus Z$ defined by
\begin{equation*}
\tensor[_{\gamma_{\sh{G}}}]{ \langle \smash{\gamma_{\scriptscriptstyle \sh{H}}} \rangle }{_{\gamma_{\sh{F}}}}  :=
\begin{bmatrix}
\gamma_{\sh{G}} & \gamma_{\sh{H}}\\
0 & \gamma_{\sh{F}}
\end{bmatrix}    
\end{equation*}
is an extension of $\gamma_{\sh{F}}$ by $\gamma_{\sh{G}}$, which we call the \emph{block extension of $\gamma_{\sh{F}}$ by $\gamma_{\sh{G}}$ associated with $\gamma_{\sh{H}}$}.
\end{lemma}
\begin{proof}
To begin, consider the canonical inclusion and projection maps $\iota_{A \oplus C}: A \to A \oplus C$, $\iota_{X \oplus Z}: X \to X \oplus Z$, $\rho_{A \oplus C}:A \oplus C\to C$, and $\rho_{X \oplus Z}:X \oplus Z \to Z$ (see Definition~\eqref{Def: inclusion and projection maps}). By construction, we have that: 
\begin{itemize}
\justifying
\item $\iota_{A\oplus C}$ and $\iota_{X\oplus Z}$ are injective.
\item $\rho_{A\oplus C}$ and $\rho_{X\oplus Z}$ are surjective. 
\item $\mathrm{im}\,\iota_{A\oplus C}=\ker\rho_{A\oplus C}$ and $\mathrm{im}\,\iota_{X\oplus Z}=\ker\rho_{X\oplus Z}$. 
\end{itemize}
In addition, a direct calculation shows that
\begin{equation*}
\begin{aligned}
\elmG\circ \iota_{A\oplus C}(a)&=\iota_{X\oplus Z}\circ \gamma_{\sh{G}}(a)\, ,    \quad   \text{for all $a\in A$}\, ,\\
\rho_{X\oplus Z}\circ \elmG (a,c)&=\gamma_{\sh{F}}\circ \rho_{A\oplus C}(a,c)\, ,  \quad \text{for all $(a,c)\in A\oplus C$}\, .     
\end{aligned}    
\end{equation*}

Combining the above results, we observe that each square in the diagram in Figure~\eqref{fig: diagram for the canonical extensions of linear maps} commutes and its horizontal rows are short exact sequences. Hence, building on Definition~\eqref{Def: Extensions of linear maps}, we conclude that $\elmG$ is an extension of $\gamma_{\sh{F}}$ by $\gamma_{\sh{G}}$. 

\begin{figure}[ht]
\centering
\begin{tikzpicture}
\useasboundingbox (-2,-2.35) rectangle (2,2.15);
\scope[transform canvas={scale=1.25}]

\node at (-4-1,1) {\footnotesize $0$};
\node at (-2-0.5, 1) {\footnotesize $X$};
\node at (0,1) {\footnotesize $X \oplus Z$};
\node at (2+0.5,1) {\footnotesize $Z$};
\node at (4+1,1) {\footnotesize $0$};

\node at (-4-1,-1) {\footnotesize $0$};
\node at (-2-0.5, -1) {\footnotesize $A$};
\node at (0,-1) {\footnotesize $A \oplus C$};
\node at (2+0.5,-1) {\footnotesize $C$};
\node at (4+1,-1) {\footnotesize $0$};

\draw[->] (-4-1 +0.25,1) -- (-2-0.5-0.25,1);
\draw[->] (-2-0.5+0.25,1) -- (-0-0.5,1);
\draw[->] (0+0.5,1) -- (2+0.5-0.25,1);
\draw[->] (2+0.5+0.25,1) -- (4+1-0.25,1);

\draw[->] (-4-1 +0.25,-1) -- (-2-0.5-0.25,-1);
\draw[->] (-2-0.5+0.25,-1) -- (-0-0.5,-1);
\draw[->] (0+0.5,-1) -- (2+0.5-0.25,-1);
\draw[->] (2+0.5+0.25,-1) -- (4+1-0.25,-1);

\draw[->] (0,-1+0.30)  -- (0,1-0.30) ;
\node[right] at (0,0) {\footnotesize $\elmG$};

\draw[->] (-2-0.5,-1+0.30) -- (-2-0.5,1-0.30);
\node[right] at (-2-0.5,0) {\footnotesize $\gamma_{\sh{G}}$};

\draw[->] (2+0.5, -1+0.30) -- (2+0.5,1-0.30);
\node[right] at (2+0.5,0) {\footnotesize $\gamma_{\sh{F}}$};

\node at (-1-0.25-0.1,1+0.4) {\scriptsize $\iota_{X\oplus Z}$};
\node at (1+0.25+0.1,1+0.4) {\scriptsize $\rho_{X\oplus Z}$};

\node at (-1-0.25-0.1,-1-0.4) {\scriptsize $\iota_{A\oplus C}$};
\node at (1+0.25+0.1,-1-0.4) {\scriptsize $\rho_{A\oplus C}$};

\endscope
\end{tikzpicture}
\caption{The block extension $\elmG$ of $\gamma_{\sh{F}}$ by $\gamma_{\sh{G}}$ associated with $\gamma_{\sh{H}}$.}
\label{fig: diagram for the canonical extensions of linear maps}
\end{figure}
\end{proof}

The following lemma shows that every extension of linear maps is equivalent to a block extension, which will simplify the analysis of the one-degree morphism spaces in the category $\ccs{1}{\beta}$.

\begin{lemma}\label{Lemma: equivalence of extensions and block extensions}
Let $A$, $B$, $C$, $X$, $Y$, $Z$ be vector spaces over $\mathbb{K}$, and let $\gamma_{\sh{G}}:A\to X$, $\Gamma_{\sh{H}}:B\to Y$, and $\gamma_{\sh{F}}:C\to Z$ be linear maps, with $\Gamma_{\sh{H}}$ realizing an extension of $\gamma_{\sh{F}}$ by $\gamma_{\sh{G}}$. 

Then there exists a linear map $\gamma_{\sh{H}}:C\to X$ such that $\Gamma_{\sh{H}}$ and $\elmG: A \oplus C \to X \oplus Z$ are equivalent extensions of $\gamma_{\sh{F}}$ by $\gamma_{\sh{G}}$, where $\elmG$ denotes the block extension of $\gamma_{\sh{F}}$ by $\gamma_{\sh{G}}$ associated with $\gamma_{\sh{H}}$.
\end{lemma}
\begin{proof}
By assumption, $\Gamma_{\sh{H}}$ is an extension of $\gamma_{\sh{F}}$ by $\gamma_{\sh{G}}$. Hence, according to Definition~\eqref{Def: Extensions of linear maps}, there exists a diagram as in Figure~\eqref{Commutative diagram for the equivalence extensions lemma} where the horizontal rows are short exact sequences and each square commutes. Furthermore, since $A$, $C$, $X$, $Z$ are vector spaces (hence free $\mathbb{K}$-modules), we have that $\mathrm{Ext}^1(A,C)=0$ and $\mathrm{Ext}^1(X,Z)=0$. It then follows that there exist linear isomorphisms $\xi:B\rightarrow A \oplus C$ and $\delta:Y\rightarrow X \oplus Z$ such that the diagrams in Figures~\eqref{fig: diagram for equivalent extensions of the vector spaces A and C} and~\eqref{fig: diagram for equivalent extensions of the vector spaces X and Z} commute in each square. 

\begin{figure}[ht]
\centering
\begin{tikzpicture}
\useasboundingbox (-2,-2) rectangle (2,2);
\scope[transform canvas={scale=1.25}]

\node at (-4,1) {\footnotesize $0$};
\node at (-2, 1) {\footnotesize $X$};
\node at (0,1) {\footnotesize $Y$};
\node at (2,1) {\footnotesize $Z$};
\node at (4,1) {\footnotesize $0$};

\node at (-4,-1) {\footnotesize $0$};
\node at (-2, -1) {\footnotesize $A$};
\node at (0,-1) {\footnotesize $B$};
\node at (2,-1) {\footnotesize $C$};
\node at (4,-1) {\footnotesize $0$};

\draw[->] (-4 +0.25,1) -- (-2-0.25,1);
\draw[->] (-2+0.25,1) -- (-0-0.25,1);
\draw[->] (0+0.25,1) -- (2-0.25,1);
\draw[->] (2+0.25,1) -- (4-0.25,1);

\draw[->] (-4 +0.25,-1) -- (-2-0.25,-1);
\draw[->] (-2+0.25,-1) -- (-0-0.25,-1);
\draw[->] (0+0.25,-1) -- (2-0.25,-1);
\draw[->] (2+0.25,-1) -- (4-0.25,-1);

\node at (-1,-1.35) {\footnotesize $\alpha_{1}$};
\node at (1,-1.35) {\footnotesize $\beta_{1}$};
\node at (-1,+1.35) {\footnotesize $\alpha_{2}$};
\node at (1,+1.35) {\footnotesize $\beta_{2}$};

\draw[->] (0,-1+0.30)  -- (0,1-0.30) ;
\node[right] at (0,0) {\footnotesize $\Gamma_{\sh{H}}$};

\draw[->] (-2,-1+0.30) -- (-2,1-0.30);
\node[right] at (-2,0) {\footnotesize $\gamma_{\sh{G}}$};

\draw[->] (2, -1+0.30) -- (2,1-0.30);
\node[right] at (2,0) {\footnotesize $\gamma_{\sh{F}}$};

\endscope
\end{tikzpicture}
\caption{An extension $\Gamma_{\sh{H}}$ of $\gamma_{\sh{F}}$ by $\gamma_{\sh{G}}$.}
\label{Commutative diagram for the equivalence extensions lemma}
\end{figure}

\begin{figure}[ht]
\centering
\begin{tikzpicture}
\useasboundingbox (-2,-2.35) rectangle (2,2.15);
\scope[transform canvas={scale=1.25}]

\node at (-4-1,1) {\footnotesize $0$};
\node at (-2-0.5, 1) {\footnotesize $A$};
\node at (0,1) {\footnotesize $A \oplus C$};
\node at (2+0.5,1) {\footnotesize $C$};
\node at (4+1,1) {\footnotesize $0$};

\node at (-4-1,-1) {\footnotesize $0$};
\node at (-2-0.5, -1) {\footnotesize $A$};
\node at (0,-1) {\footnotesize $B$};
\node at (2+0.5,-1) {\footnotesize $C$};
\node at (4+1,-1) {\footnotesize $0$};

\draw[->] (-4-1 +0.25,1) -- (-2-0.5-0.25,1);
\draw[->] (-2-0.5+0.25,1) -- (-0-0.5,1);
\draw[->] (0+0.5,1) -- (2+0.5-0.25,1);
\draw[->] (2+0.5+0.25,1) -- (4+1-0.25,1);

\draw[->] (-4-1 +0.25,-1) -- (-2-0.5-0.25,-1);
\draw[->] (-2-0.5+0.25,-1) -- (-0-0.5+0.25,-1);
\draw[->] (0+0.5-0.25,-1) -- (2+0.5-0.25,-1);
\draw[->] (2+0.5+0.25,-1) -- (4+1-0.25,-1);

\draw[->] (0,-1+0.30)  -- (0,1-0.30) ;
\node[right] at (0,0) {\footnotesize $\xi$};

\draw (-2-0.5-0.05,-1+0.30) -- (-2-0.5-0.05,1-0.30);
\draw (-2-0.5+0.05,-1+0.30) -- (-2-0.5+0.05,1-0.30);

\draw (2+0.5-0.05, -1+0.30) -- (2+0.5-0.05,1-0.30);
\draw (2+0.5+0.05, -1+0.30) -- (2+0.5+0.05,1-0.30);

\node at (-1-0.25-0.1,1+0.4) {\scriptsize $\iota_{A\oplus C}$};
\node at (1+0.25+0.1,1+0.4) {\scriptsize $\rho_{A\oplus C}$};

\node at (-1-0.25-0.1,-1-0.4) {\scriptsize $\alpha_{1}$};
\node at (1+0.25+0.1,-1-0.4) {\scriptsize $\beta_{1}$};

\endscope
\end{tikzpicture}
\caption{Two equivalent extensions $B$ and $A\oplus C$ of $A$ by $C$.}
\label{fig: diagram for equivalent extensions of the vector spaces A and C}
\end{figure}

\begin{figure}[ht]
\centering
\begin{tikzpicture}
\useasboundingbox (-2,-2.35) rectangle (2,2.15);
\scope[transform canvas={scale=1.25}]

\node at (-4-1,1) {\footnotesize $0$};
\node at (-2-0.5, 1) {\footnotesize $X$};
\node at (0,1) {\footnotesize $X \oplus Z$};
\node at (2+0.5,1) {\footnotesize $Z$};
\node at (4+1,1) {\footnotesize $0$};

\node at (-4-1,-1) {\footnotesize $0$};
\node at (-2-0.5, -1) {\footnotesize $X$};
\node at (0,-1) {\footnotesize $Y$};
\node at (2+0.5,-1) {\footnotesize $Z$};
\node at (4+1,-1) {\footnotesize $0$};

\draw[->] (-4-1 +0.25,1) -- (-2-0.5-0.25,1);
\draw[->] (-2-0.5+0.25,1) -- (-0-0.5,1);
\draw[->] (0+0.5,1) -- (2+0.5-0.25,1);
\draw[->] (2+0.5+0.25,1) -- (4+1-0.25,1);

\draw[->] (-4-1 +0.25,-1) -- (-2-0.5-0.25,-1);
\draw[->] (-2-0.5+0.25,-1) -- (-0-0.5+0.25,-1);
\draw[->] (0+0.5-0.25,-1) -- (2+0.5-0.25,-1);
\draw[->] (2+0.5+0.25,-1) -- (4+1-0.25,-1);

\draw[->] (0,-1+0.30)  -- (0,1-0.30) ;
\node[right] at (0,0) {\footnotesize $\delta$};

\draw (-2-0.5-0.05,-1+0.30) -- (-2-0.5-0.05,1-0.30);
\draw (-2-0.5+0.05,-1+0.30) -- (-2-0.5+0.05,1-0.30);

\draw (2+0.5-0.05, -1+0.30) -- (2+0.5-0.05,1-0.30);
\draw (2+0.5+0.05, -1+0.30) -- (2+0.5+0.05,1-0.30);

\node at (-1-0.25-0.1,1+0.4) {\scriptsize $\iota_{X\oplus Z}$};
\node at (1+0.25+0.1,1+0.4) {\scriptsize $\rho_{X\oplus Z}$};

\node at (-1-0.25-0.1,-1-0.4) {\scriptsize $\alpha_{2}$};
\node at (1+0.25+0.1,-1-0.4) {\scriptsize $\beta_{2}$};

\endscope
\end{tikzpicture}
\caption{Two equivalent extensions $Y$ and $X\oplus Z$ of $X$ by $Z$.}
\label{fig: diagram for equivalent extensions of the vector spaces X and Z}
\end{figure}

Next, consider the diagram in Figure~\eqref{fig: Commutative diagram for equivalent extensions lemma}, and define the linear map $\Gamma'_{\sh{H}}:A\oplus C\rightarrow X\oplus Z$ by $\Gamma'_{\sh{H}}:=\delta \circ \Gamma_{\sh{H}}\circ \xi^{-1}$. By the commutativity of the diagrams in Figures~\eqref{Commutative diagram for the equivalence extensions lemma},~\eqref{fig: diagram for equivalent extensions of the vector spaces A and C}, and~\eqref{fig: diagram for equivalent extensions of the vector spaces X and Z}, we have that 
\begin{equation*}
\begin{aligned}
\Gamma'_{\sh{H}}\circ \iota_{A\oplus C} &= \delta\circ \Gamma_{\sh{H}}\circ \xi^{-1}\circ \iota_{A\oplus C}\, , \\
&=\delta\circ \Gamma_{\sh{H}}\circ \xi^{-1}\circ\xi\circ \alpha_{1}\, ,\\
&=\delta\circ \Gamma_{\sh{H}}\circ \alpha_{1}\, ,\\
&=\delta\circ \alpha_{2}\circ \gamma_{\sh{G}}\, ,\\
&=\iota_{X\oplus Z}\circ \gamma_{\sh{G}}\, .\\
\end{aligned}    
\end{equation*}
Similarly, we deduce that
\begin{equation*}
\begin{aligned}
\rho_{X\oplus Z}\circ \Gamma'_{\sh{H}}&=\rho_{X\oplus Z}\circ \delta \circ \Gamma_{\sh{H}}\circ \xi^{-1}\, ,\\
&=\beta_{2}\circ \Gamma_{\sh{H}}\circ \xi^{-1}\, ,\\
&=\gamma_{\sh{F}}\circ \beta_{1}\circ \xi^{-1}\, ,\\
&=\gamma_{\sh{F}}\circ \rho_{A\oplus C}\circ \xi\circ \xi^{-1}\, ,\\
&=\gamma_{\sh{F}}\circ \rho_{A\oplus C}\, .
\end{aligned}    
\end{equation*}

Combining the above results, we conclude that there exist linear isomorphisms $\xi:B\rightarrow A \oplus C$ and $\delta:Y\rightarrow X \oplus Z$ such that the diagram in Figure~\eqref{fig: Commutative diagram for equivalent extensions lemma} commutes in each square. It then follows from Definition~\eqref{Def: equivalent extensions of linear maps} that $\Gamma_{\sh{H}}$ and $\Gamma'_{\sh{H}}$ are equivalent extensions of $\gamma_{\sh{F}}$ by $\gamma_{\sh{G}}$. 

Finally, consider the linear map $\Gamma'_{\sh{H}}:A\oplus C\rightarrow X\oplus Z$. By the homomorphism decomposition for direct sums, there exist linear maps $\lambda^{(1)}_{\sh{H}}:A\rightarrow X$,  $\lambda^{(2)}_{\sh{H}}:C\rightarrow X$,  $\lambda^{(3)}_{\sh{H}}:A\rightarrow Z$, and $\lambda^{(4)}_{\sh{H}}:C\rightarrow Z$ such that
\begin{equation*}
\Gamma'_{\sh{H}}= \begin{bmatrix}
\lambda^{(1)}_{\sh{H}} & \lambda^{(2)}_{\sh{H}}\\
\lambda^{(3)}_{\sh{H}} & \lambda^{(4)}_{\sh{H}}
\end{bmatrix}\, .
\end{equation*}
In particular, it follows from the commutativity of the diagram in Figure~\eqref{fig: Commutative diagram for equivalent extensions lemma} that
\begin{equation*}
\begin{aligned}
\lambda^{(1)}_{\sh{H}}&=\gamma_{\sh{G}}\, , \qquad \lambda^{(3)}_{\sh{H}}=0\, , \qquad \lambda^{(4)}_{\sh{H}}=\gamma_{\sh{F}}\, .     
\end{aligned}    
\end{equation*}
Hence, by setting $\gamma_{\sh{H}}:=\lambda^{(2)}_{\sh{H}}:C\rightarrow X$, the above relations ensure that $\Gamma'_{\sh{H}}=\elmG$, where $\elmG:A\oplus C\rightarrow X\oplus Z$ denotes the block extension of $\gamma_{\sh{F}}$ by $\gamma_{\sh{G}}$ associated with $\gamma_{\sh{H}}$. Bearing this in mind, we conclude that there exists a linear map $\gamma_{\sh{H}}:C\rightarrow X$ such that $\Gamma_{H}$ and $\elmG$ are equivalent extensions of $\gamma_{\sh{F}}$ by $\gamma_{\sh{G}}$. This completes the proof. 
 
\begin{figure}[ht]
\centering
\begin{tikzpicture}
\useasboundingbox (-6,-5.5) rectangle (6,6.5);
\scope[transform canvas={scale=1}]

\node at (4,6) {\footnotesize\large $0$};
\node at (4,4) {\footnotesize\large $0$};

\node at (1,4.5) {\footnotesize\large $Z$};
\node at (1,2.5) {\footnotesize\large $C$};

\node at (-2,3) {\footnotesize\large $Y$};
\node at (-2,1) {\footnotesize\large $B$};

\node at (-5,1.5) {\footnotesize\large $X$};
\node at (-5,-0.5) {\footnotesize\large $A$};

\node at (-8,0) {\footnotesize\large $0$};
\node at (-8,-2) {\footnotesize\large $0$};

\node at (8,3) {\footnotesize\large $0$};
\node at (8,1) {\footnotesize\large $0$};

\node at (5,1.5) {\footnotesize\large $Z$};
\node at (5,-0.5) {\footnotesize\large $C$};

\node at (2,0) {\footnotesize\large $X\oplus Z$};
\node at (2,-2) {\footnotesize\large $A\oplus C$};

\node at (-1,-1.5) {\footnotesize\large $X$};
\node at (-1,-3.5) {\footnotesize\large $A$};

\node at (-4,-3) {\footnotesize\large $0$};
\node at (-4,-5) {\footnotesize\large $0$};

\draw[->, thick] (-8 +0.5,0+0.2) -- (-5-0.5,1.5-0.2);
\draw[->, thick] (-5 +0.5,1.5+0.2) -- (-2-0.5,3-0.2);
\draw[->, thick] (-2 +0.5,3+0.2) -- (1-0.5,4.5-0.2);
\draw[->, thick] (1 +0.5,4.5+0.2) -- (4-0.5,6-0.2);

\draw[->, thick] (-8 +0.5,-2+0.2) -- (-5-0.5,-0.5-0.2);
\draw[->, thick] (-5 +0.5,-0.5+0.2) -- (-2-0.5,1-0.2);
\draw[->, thick] (-2 +0.5,1+0.2) -- (1-0.5,2.5-0.2);
\draw[->, thick] (1 +0.5,2.5+0.2) -- (4-0.5,4-0.2);

\draw[->, thick] (-2,1+0.35) -- (-2,3-0.35);
\draw[->, thick] (1,2.5+0.35) -- (1,4.5-0.35);
\draw[->, thick] (-5,-0.5+0.35) -- (-5,1.5-0.35);

\draw[->, thick] (-4+0.5,-3+0.2) -- (-1-0.5,-1.5-0.2);
\draw[->, thick] (-1+0.5,-1.5+0.2) -- (2-0.5-0.2,0-0.2-0.125);
\draw[->, thick] (2+0.5+0.2,0+0.2+0.125) -- (5-0.5,1.5-0.2);
\draw[->, thick] (5+0.5,1.5+0.2) -- (8-0.5,3-0.2);

\draw[->, thick] (-4+0.5,-5+0.2) -- (-1-0.5,-3.5-0.2);
\draw[->, thick] (-1+0.5,-3.5+0.15) -- (2-0.5-0.2,-2-0.2-0.125);
\draw[->, thick] (2+0.5+0.2,-2+0.2+0.125) -- (5-0.5,-0.5-0.2);
\draw[->, thick] (5+0.5,-0.5+0.2) -- (8-0.5,1-0.2);

\draw[->, thick] (-1,-3.5+0.35) -- (-1,-1.5-0.35);
\draw[->, thick] (2,-2+0.35) -- (2,0-0.35);
\draw[->, thick] (5,-0.5+0.35) -- (5,1.5-0.35);

\draw[thick, blue] (-5+0.25-0.05, 1.5-0.25-0.05) -- (-1-0.25-0.05,-1.5+0.25-0.05);
\draw[thick, blue] (-5+0.25+0.05, 1.5-0.25+0.05) -- (-1-0.25+0.05,-1.5+0.25+0.05);

\draw[thick, blue] (-5+0.25-0.05, -0.5-0.25-0.05) -- (-1-0.25-0.05,-3.5+0.25-0.05);
\draw[thick, blue] (-5+0.25+0.05, -0.5-0.25+0.05) -- (-1-0.25+0.05,-3.5+0.25+0.05);

\draw[->, thick, red] (-2+0.25, 3-0.2) -- (2-0.25-0.25,0+0.25+0.25);
\draw[->, thick, red] (-2+0.25, 1-0.2) -- (2-0.25-0.25,-2+0.25+0.25);

\draw[thick, blue] (1+0.25-0.05, 4.5-0.25-0.05) -- (5-0.25-0.05,1.5+0.25-0.05);
\draw[thick, blue] (1+0.25+0.05, 4.5-0.25+0.05) -- (5-0.25+0.05,1.5+0.25+0.05);

\draw[thick, blue] (1+0.25-0.05, 2.5-0.25-0.05) -- (5-0.25-0.05,-0.5+0.25-0.05);
\draw[thick, blue] (1+0.25+0.05, 2.5-0.25+0.05) -- (5-0.25+0.05,-0.5+0.25+0.05);

\node at (-0.75,-0.45) {\Large $\xi$};
\node at (0,1) {\Large $\delta$};

\node at (-5-0.5,0.5-0.25) {\Large $\gamma_{\sh{G}}$};
\node at (-2-0.5,2-0.25) {\Large $\Gamma_{\sh{H}}$};
\node at (1-0.5,3.5-0.25) {\Large $\gamma_{\sh{F}}$};

\node at (-1+0.6,-2.5+0.25) {\Large $\gamma_{\sh{G}}$};
\node at (2+0.75,-1+0.15) {\Large $\Gamma'_{\sh{H}}$};
\node at (5+0.6,0.5+0.15) {\Large $\gamma_{\sh{F}}$};

\node at (-4+0.6,0.5+0.25) {\large $\alpha_{1}$};
\node at (-1+0.6,1.95+0.25) {\large $\beta_{1}$};
\node at (-4+0.4,2.5+0.2) {\large $\alpha_{2}$};
\node at (-1+0.4,3.95+0.2) {\large $\beta_{2}$};

\node at (-4+0.6+4.25,0.5+0.25-3.75-0.12) {\large $\iota_{A\oplus C}$};
\node at (-1+0.6+4.25,1.95+0.25-3.75-0.12) {\large $\rho_{A\oplus C}$};
\node at (-4+0.6+3.75,0.5+0.25-1.95-0.12) {\large $\iota_{X\oplus Z}$};
\node at (-1+0.6+3.75,1.95+0.25-1.95-0.12) {\large $\rho_{X\oplus Z}$};

\endscope
\end{tikzpicture}
\caption{Two equivalent extensions $\Gamma_{\sh{H}}$ and $\Gamma'_{\sh{H}}$ of $\gamma_{\sh{F}}$ by $\gamma_{\sh{G}}$.}
\label{fig: Commutative diagram for equivalent extensions lemma}
\end{figure}
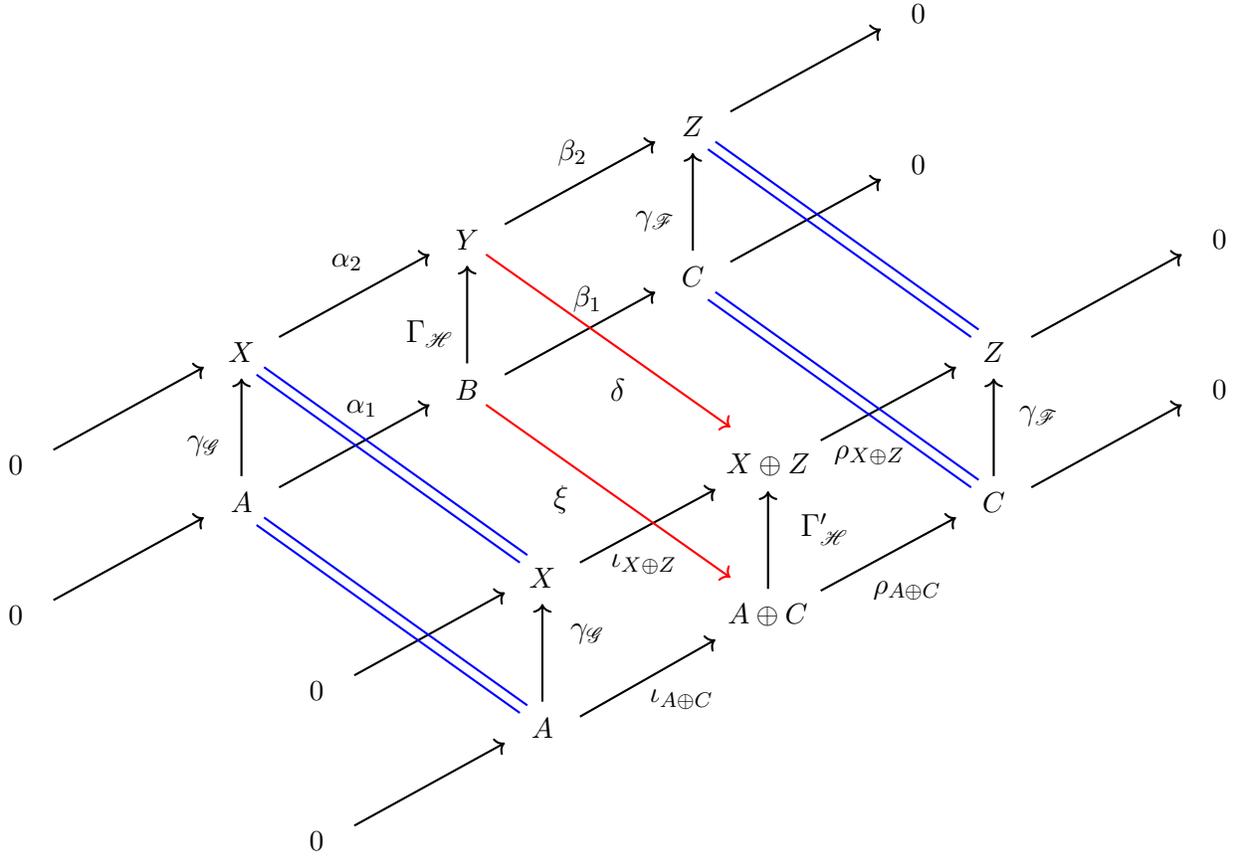
\end{proof}

The following lemma establishes necessary and sufficient conditions for the equivalence of block extensions of linear maps.

\begin{lemma}\label{Lemma: equivalence between block extensions}
Let $A$, $C$, $X$, $Z$ be vector spaces over $\mathbb{K}$, and let $\gamma_{\sh{G}}:A\rightarrow X$, $\gamma_{\sh{H}}:C\rightarrow X$, $\gamma_{\sh{H}'}:C\rightarrow X$, and $\gamma_{\sh{F}}:C\rightarrow Z$ be linear maps. In particular, consider $\elmG:A\oplus C\rightarrow X\oplus Z$ and $\elmg{F}{G}{H'}:A\oplus C \rightarrow X\oplus Z$ the block extensions of $\gamma_{\sh{F}}$ by $\gamma_{\sh{G}}$ associated with $\gamma_{\sh{H}}$ and $\gamma_{\sh{H}'}$, respectively. Then, we have that $\elmG$ and $\elmg{F}{G}{H'}$ are equivalent extensions of $\gamma_{\sh{F}}$ by $\gamma_{\sh{G}}$ if and only if there exist linear maps $\lambda_{\xi}:C\rightarrow A$ and $\lambda_{\delta}:Z\rightarrow X$ such that
\begin{equation*}
\gamma_{\sh{H}'}=\gamma_{\sh{H}}+\lambda_{\delta}\circ \gamma_{\sh{F}} - \gamma_{\sh{G}}\circ \lambda_{\xi}\, .       
\end{equation*}  
\end{lemma}
\begin{proof}
To begin, suppose that $\elmG$ and $\elmg{F}{G}{H'}$ are equivalent extensions of $\gamma_{\sh{F}}$ by $\gamma_{\sh{G}}$. Thus, according to Definition~\eqref{Def: equivalent extensions of linear maps}, there exist linear isomorphisms $\xi: A\oplus C\rightarrow A\oplus C$ and $\delta: X \oplus Z\rightarrow X\oplus Z$ such that the diagram in Figure~\eqref{fig: Commutative diagram for two equivalent canonical extensions of linear maps} commutes in each square. By the homomorphism decomposition for direct sums, there exist linear maps $\lambda^{(1)}_{\xi}:A\rightarrow A$,  $\lambda^{(2)}_{\xi}:C\rightarrow A$,  $\lambda^{(3)}_{\xi}:A\rightarrow C$, and  $\lambda^{(4)}_{\xi}:C\rightarrow C$ such that $\xi$ can be written as
\begin{equation*}
\xi=\begin{bmatrix}
\lambda^{(1)}_{\xi} & \lambda^{(2)}_{\xi}\\
\lambda^{(3)}_{\xi} & \lambda^{(4)}_{\xi}
\end{bmatrix}\, .    
\end{equation*}
Likewise, there exist linear maps $\lambda^{(1)}_{\delta}:X\rightarrow X$,  $\lambda^{(2)}_{\delta}:Z\rightarrow X$,  $\lambda^{(3)}_{\delta}:X\rightarrow Z$, and  $\lambda^{(4)}_{\delta}:Z\rightarrow Z$ such that $\delta$ can be expressed as
\begin{equation*}
\delta=\begin{bmatrix}
\lambda^{(1)}_{\delta} & \lambda^{(2)}_{\delta}\\
\lambda^{(3)}_{\delta} & \lambda^{(4)}_{\delta}
\end{bmatrix}\, .   
\end{equation*}
From these decompositions, the commutativity of the diagram in Figure~\eqref{fig: Commutative diagram for two equivalent canonical extensions of linear maps} implies that
\begin{equation*}
\begin{aligned}
\lambda^{(1)}_{\xi}&=\mathrm{id}_{A}\, , \\
\lambda^{(3)}_{\xi}&=0\, , \\
\lambda^{(4)}_{\xi}&=\mathrm{id}_{C}\, , \\
\lambda^{(1)}_{\delta}&=\mathrm{id}_{X}\, , \\
\lambda^{(3)}_{\delta}&=0\, , \\
\lambda^{(4)}_{\delta}&=\mathrm{id}_{Z}\, , \\
\gamma_{\sh{H}'}+\gamma_{\sh{G}}\circ \lambda^{(2)}_{\xi}&=\gamma_{\sh{H}}+\lambda^{(2)}_{\delta}\circ \gamma_{\sh{F}} \, .    
\end{aligned}    
\end{equation*}
Thus, by setting $\lambda_{\xi}:=\lambda^{(2)}_{\xi}:C\rightarrow A$ and $\lambda_{\delta}:=\lambda^{(2)}_{\delta}:Z\rightarrow X$, the deduce that there exist linear maps $\lambda_{\xi}$ and $\lambda_{\delta}$ such that 
\begin{equation*}
\gamma_{\sh{H}'}=\gamma_{\sh{H}}+\lambda_{\delta}\circ \gamma_{\sh{F}} - \gamma_{\sh{G}}\circ \lambda_{\xi}\, .       
\end{equation*}  

\begin{figure}[ht]
\centering
\begin{tikzpicture}
\useasboundingbox (-6,-5.5) rectangle (6,6.5);
\scope[transform canvas={scale=1}]

\node at (8-4,3+3) {\footnotesize\large $0$};
\node at (8-4,1+3) {\footnotesize\large $0$};

\node at (5-4,1.5+3) {\footnotesize\large $Z$};
\node at (5-4,-0.5+3) {\footnotesize\large $C$};

\node at (2-4,0+3) {\footnotesize\large $X\oplus Z$};
\node at (2-4,-2+3) {\footnotesize\large $A\oplus C$};

\node at (-1-4,-1.5+3) {\footnotesize\large $X$};
\node at (-1-4,-3.5+3) {\footnotesize\large $A$};

\node at (-4-4,-3+3) {\footnotesize\large $0$};
\node at (-4-4,-5+3) {\footnotesize\large $0$};

\node at (8,3) {\footnotesize\large $0$};
\node at (8,1) {\footnotesize\large $0$};

\node at (5,1.5) {\footnotesize\large $Z$};
\node at (5,-0.5) {\footnotesize\large $C$};

\node at (2,0) {\footnotesize\large $X\oplus Z$};
\node at (2,-2) {\footnotesize\large $A\oplus C$};

\node at (-1,-1.5) {\footnotesize\large $X$};
\node at (-1,-3.5) {\footnotesize\large $A$};

\node at (-4,-3) {\footnotesize\large $0$};
\node at (-4,-5) {\footnotesize\large $0$};

\draw[->, thick] (-8 +0.5,0+0.2) -- (-5-0.5,1.5-0.2);
\draw[->, thick] (-5 +0.5,1.5+0.15) -- (-2-0.5-0.15,3-0.2-0.13);
\draw[->, thick] (-2 +0.5+0.2,3+0.2+0.125) -- (1-0.5,4.5-0.2);
\draw[->, thick] (1 +0.5,4.5+0.2) -- (4-0.5,6-0.2);

\draw[->, thick] (-8 +0.5,-2+0.2) -- (-5-0.5,-0.5-0.2);
\draw[->, thick] (-5 +0.5,-0.5+0.15) -- (-2-0.5-0.15,1-0.2-0.11);
\draw[->, thick] (-2 +0.5+0.2,1+0.2+0.125) -- (1-0.5,2.5-0.2);
\draw[->, thick] (1 +0.5,2.5+0.2) -- (4-0.5,4-0.2);

\draw[->, thick] (-2,1+0.35) -- (-2,3-0.35);
\draw[->, thick] (1,2.5+0.35) -- (1,4.5-0.35);
\draw[->, thick] (-5,-0.5+0.35) -- (-5,1.5-0.35);

\draw[->, thick] (-4+0.5,-3+0.2) -- (-1-0.5,-1.5-0.2);
\draw[->, thick] (-1+0.5,-1.5+0.15) -- (2-0.5-0.15,0-0.2-0.13);
\draw[->, thick] (2+0.5+0.2,0+0.2+0.125) -- (5-0.5,1.5-0.2);
\draw[->, thick] (5+0.5,1.5+0.2) -- (8-0.5,3-0.2);

\draw[->, thick] (-4+0.5,-5+0.2) -- (-1-0.5,-3.5-0.2);
\draw[->, thick] (-1+0.5,-3.5+0.15) -- (2-0.5-0.2,-2-0.2-0.125);
\draw[->, thick] (2+0.5+0.2,-2+0.2+0.125) -- (5-0.5,-0.5-0.2);
\draw[->, thick] (5+0.5,-0.5+0.2) -- (8-0.5,1-0.2);

\draw[->, thick] (-1,-3.5+0.35) -- (-1,-1.5-0.35);
\draw[->, thick] (2,-2+0.35) -- (2,0-0.35);
\draw[->, thick] (5,-0.5+0.35) -- (5,1.5-0.35);

\draw[thick, blue] (-5+0.25-0.035, 1.5-0.25-0.035) -- (-1-0.25-0.035,-1.5+0.25-0.035);
\draw[thick, blue] (-5+0.25+0.035, 1.5-0.25+0.035) -- (-1-0.25+0.035,-1.5+0.25+0.035);

\draw[thick, blue] (-5+0.25-0.035, -0.5-0.25-0.035) -- (-1-0.25-0.035,-3.5+0.25-0.035);
\draw[thick, blue] (-5+0.25+0.035, -0.5-0.25+0.035) -- (-1-0.25+0.035,-3.5+0.25+0.035);

\draw[->, thick, red] (-2+0.25+0.25-0.3, 3-0.25-0.0) -- (2-0.25-0.1+0.3,0+0.25+0.15-0.15);
\draw[->, thick, red] (-2+0.25+0.25-0.3, 1-0.25-0.0) -- (2-0.25-0.1+0.2,-2+0.25+0.15-0.15);

\draw[thick, blue] (1+0.25-0.035, 4.5-0.25-0.035) -- (5-0.25-0.035,1.5+0.25-0.035);
\draw[thick, blue] (1+0.25+0.035, 4.5-0.25+0.035) -- (5-0.25+0.035,1.5+0.25+0.035);

\draw[thick, blue] (1+0.25-0.035, 2.5-0.25-0.035) -- (5-0.25-0.035,-0.5+0.25-0.035);
\draw[thick, blue] (1+0.25+0.035, 2.5-0.25+0.035) -- (5-0.25+0.035,-0.5+0.25+0.035);

\node at (-0.75,-0.45) {\large $\xi$};
\node at (0,1) {\large $\delta$};

\node at (-5-0.5,0.5-0.25) {\large $\gamma_{\sh{G}}$};
\node at (-2-0.75,2-0.3) {\large $\elmG$};
\node at (1-0.5,3.5-0.25) {\large $\gamma_{\sh{F}}$};

\node at (-1+0.6,-2.5+0.25) {\large $\gamma_{\sh{G}}$};
\node at (2+0.8,-1+0.15) {\large $\elmg{F}{G}{H'}$};
\node at (5+0.6,0.5+0.15) {\large $\gamma_{\sh{F}}$};

\node at (-4+0.6,0.5+0.45) {\large $\iota_{A\oplus C}$};
\node at (-1+0.6,1.95+0.45) {\large $\rho_{A\oplus C}$};
\node at (-4+0.3,2.5+0.3) {\large $\iota_{X\oplus Z}$};
\node at (-1+0.3,3.95+0.3) {\large $\rho_{X\oplus Z}$};

\node at (-4+0.6+4.25,0.5+0.05-3.75) {\large $\iota_{A\oplus C}$};
\node at (-1+0.6+4.25,1.95+0.2-3.75) {\large $\rho_{A\oplus C}$};
\node at (-4+0.6+3.75,0.5+0.05-1.95) {\large $\iota_{X\oplus Z}$};
\node at (-1+0.6+3.75,1.95+0.2-1.95) {\large $\rho_{X\oplus Z}$};

\endscope
\end{tikzpicture}
\caption{Two equivalent extensions $\elmG$ and $\elmg{F}{G}{H'}$ of $\gamma_{\sh{F}}$ by $\gamma_{\sh{G}}$.}
\label{fig: Commutative diagram for two equivalent canonical extensions of linear maps}
\end{figure}

Conversely, suppose that there exist linear maps $\lambda_{\xi}:C\rightarrow A$ and  $\lambda_{\delta}:Z\rightarrow X$ such that 
\begin{equation*}
\gamma_{\sh{H}'}=\gamma_{\sh{H}}+\lambda_{\delta}\circ \gamma_{\sh{F}} - \gamma_{\sh{G}}\circ \lambda_{\xi}\, .       
\end{equation*}  
Thus, we introduce $\xi: A\oplus C\rightarrow A\oplus C$ and $\delta: X \oplus Z\rightarrow X\oplus Z$ to denote the linear isomorphisms defined by 
\begin{equation*}
\xi:=\begin{bmatrix}
\mathrm{id}_{A} & \lambda_{\xi}\\
0 & \mathrm{id}_{C}
\end{bmatrix}\, , ~~ \qquad ~~  \delta:=\begin{bmatrix}
\mathrm{id}_{X} & \lambda_{\delta}\\
0 & \mathrm{id}_{Z}
\end{bmatrix}\, .
\end{equation*}
In particular, a direct calculation shows the diagram in Figure~\eqref{fig: Commutative diagram for two equivalent canonical extensions of linear maps} commutes in each square. Hence, building on Definition~\eqref{Def: equivalent extensions of linear maps}, we conclude that $\elmG$ and $\elmg{F}{G}{H'}$ are equivalent extensions of $\gamma_{\sh{F}}$ by $\gamma_{\sh{G}}$. This concludes the proof.
\end{proof}

Next, we state a lemma and a related observation that make explicit how extensions of sheaves with singular support in $L(\Lambda(\beta))$ inherit the microlocal support and rank conditions, thus elucidating the structure of the one-degree morphism spaces in the category $\ccs{1}{\beta}$. 

\begin{lemma}\label{Lemma: extensions are microlocal supported sheaves}
Let $\beta=\sigma_{i_{1}}\cdots \sigma_{i_{\ell}}\in \mathrm{Br}^{+}_{n}$ be a positive braid word, and let $L(\Lambda(\beta))\subset T^{*}\mathbb{R}^{2}$ be the closed conic Lagrangian associated with $\Lambda(\beta)\subset(\mathbb{R}^{3},\xi_{\mathrm{std}})$ (see Construction~\eqref{Def: closed conic Lagrangian associated with a Legendrian}). Let $\sh{F}$ and $\sh{G}$ be objects of the category $\ccs{1}{\beta}$, and suppose that $\sh{H}$ is an extension of $\sh{F}$ by $\sh{G}$. Then: 
\begin{itemize}
\justifying
\item The singular support of $\sh{H}$ is contained in $L(\Lambda(\beta))$, namely $SS(\sh{H})\subset L(\Lambda(\beta))$. 
\item $\sh{H}$ is compactly supported. 
\item $\sh{H}$ has microlocal rank $2$.
\end{itemize}
\end{lemma}
\begin{proof}
By assumption, $\sh{H}$ is an extension of $\sh{F}$ by $\sh{G}$. Hence, $\sh{H}$ fits into a short exact sequence of the form
\begin{equation}\label{Eq: short exact sequence of a sheaf extension}
0 \longrightarrow \sh{G} \longrightarrow \sh{H} \longrightarrow \sh{F} \longrightarrow 0\, .      
\end{equation}

Next, we verify that $SS(\sh{H})\subset L(\Lambda(\beta))$. In particular, a direct application of the triangle inequality for the singular support to the short exact sequence in~\eqref{Eq: short exact sequence of a sheaf extension} (see~\cite{KS1})
yields that $SS(\sh{H})\subset SS(\sh{F})\;\cup\; SS(\sh{G})$. Here, recall that $\sh{F}$ and $\sh{G}$ are objects of the category $\ccs{1}{\beta}$, and therefore $SS(\sh{F})\subset L(\Lambda(\beta))$ and $SS(\sh{G})\subset L(\Lambda(\beta))$, which implies that $SS(\sh{H})\subset L(\Lambda(\beta))$. 

Now, we prove that $\sh{H}$ is compactly supported. To begin, observe that for any $x\in \mathbb{R}^{2}$, the short exact sequence in~\eqref{Eq: short exact sequence of a sheaf extension} yields a short exact sequence of $\mathbb{K}$-modules at the stalks: 
\begin{equation*}
0 \longrightarrow \sh{G}_{x} \longrightarrow \sh{H}_{x} \longrightarrow \sh{F}_{x} \longrightarrow 0\, .      
\end{equation*}
Bearing this in mind, we deduce that $\mathrm{supp}(\sh{H})=\mathrm{supp}(\sh{F})\;\cup\; \mathrm{supp}(\sh{G})$. In particular, recall that, since $\sh{F}$ and $\sh{G}$ are objects of the category $\ccs{1}{\beta}$, $\mathrm{supp}(\sh{F})$ and $\mathrm{supp}(\sh{G})$ are compact sets in $\mathbb{R}^{2}$. Then, since the finite union of compact sets is compact, we conclude that $\mathrm{supp}(\sh{H})$ is also a compact set in $\mathbb{R}^{2}$, thereby obtaining that $\sh{H}$ is compactly supported. 

Finally, we show that $\sh{H}$ has microlocal rank $2$. To this end, let $a$ be an arc in $\mathcal{S}_{\Lambda(\beta)}$, the stratification of $\mathbb{R}^{2}$ induced by $\Lambda(\beta)$. Near $a$, $\mathcal{S}_{\Lambda(\beta)}$ consists of an upper $2$-dimensional stratum $U$ and a lower $2$-dimensional stratum $D$, as illustrated in Sub-figure~\eqref{sub-fig: Strata near an arc}. Given this local configuration, let $p\in U$ and $q\in D$ be two arbitrary points. Then, the short exact sequence in~\eqref{Eq: short exact sequence of a sheaf extension} yields short exact sequences of $\mathbb{K}$-modules, for the stalks at $p$ and $q$, of the form 
\begin{equation*}
\begin{aligned}
&0 \longrightarrow \sh{G}_{p} \longrightarrow \sh{H}_{p} \longrightarrow \sh{F}_{p} \longrightarrow 0\, , \\[6pt]
&0 \longrightarrow \sh{G}_{q} \longrightarrow \sh{H}_{q} \longrightarrow \sh{F}_{q} \longrightarrow 0\, . \\
\end{aligned}   
\end{equation*}
Bearing this in mind, we obtain that $\mathrm{dim}_{\,\mathbb{K}}\sh{H}_{p}=\mathrm{dim}_{\,\mathbb{K}}\sh{F}_{p} +\mathrm{dim}_{\,\mathbb{K}}\sh{G}_{p} $ and $\mathrm{dim}_{\,\mathbb{K}}\sh{H}_{q}=\mathrm{dim}_{\,\mathbb{K}}\sh{F}_{q} +\mathrm{dim}_{\,\mathbb{K}}\sh{G}_{q} $. Accordingly, we distinguish two cases: 
\begin{itemize}
\justifying
\item[-] Suppose that $a$ belongs to a strand with Maslov potential $0$. By the microlocal rank conditions, $\mathrm{dim}_{\,\mathbb{K}} \sh{F}_{p} - \mathrm{dim}_{\,\mathbb{K}} \sh{F}_{q}= 1$ and $\mathrm{dim}_{\,\mathbb{K}} \sh{G}_{p} - \mathrm{dim}_{\,\mathbb{K}} \sh{G}_{q}=1$. Hence, in this case $\mathrm{dim}_{\,\mathbb{K}} \sh{H}_{p} - \mathrm{dim}_{\,\mathbb{K}} \sh{H}_{q}=2$. 
\item[-] Suppose that $a$ belongs to a strand with Maslov potential $1$. By the microlocal rank conditions, $\mathrm{dim}_{\,\mathbb{K}} \sh{F}_{q} - \mathrm{dim}_{\,\mathbb{K}} \sh{F}_{p}= 1$ and $\mathrm{dim}_{\,\mathbb{K}} \sh{G}_{q} - \mathrm{dim}_{\,\mathbb{K}} \sh{G}_{p}=1$. Hence, in this case $\mathrm{dim}_{\,\mathbb{K}} \sh{H}_{q} - \mathrm{dim}_{\,\mathbb{K}} \sh{H}_{p}=2$. 
\end{itemize}
Therefore, since $a$ is an arbitrary arc in $\mathcal{S}_{\Lambda(\beta)}$, we conclude that $\sh{H}$ has microlocal rank $2$. This completes the proof. 
\end{proof}

\begin{observation}\label{Obs: elements of Ext1 at the local models}
Let $\beta\in\mathrm{Br}^{+}_{n}$ be a positive braid word, and let $\mathcal{S}_{\Lambda(\beta)}$ denote the stratification of $\mathbb{R}^{2}$ induced by $\Lambda(\beta)$. Let $\sh{F}$ and $\sh{G}$ be objects of the category $\ccs{1}{\beta}$, and suppose that $\sh{H}$ is an extension of $\sh{F}$ by $\sh{G}$.

By Lemma~\eqref{Lemma: extensions are microlocal supported sheaves}, $\sh{H}$ is a compactly supported sheaf of $\mathbb{K}$-modules on $\mathbb{R}^{2}$ of microlocal rank $2$, whose singular support is contained in $L(\Lambda(\beta))$, the closed conic Lagrangian associated with $\Lambda(\beta)$. It then follows from~\cite{STZ1} that $\sh{H}$ is not only constructible with respect to $\mathcal{S}_{\Lambda(\beta)}$, but is also subject to the microlocal support conditions. Building on this, we now provide a detailed characterization of $\sh{H}$ at the key local models: arcs, cusps, and crossings.

\noindent
$\star$ \emph{\textbf{Arcs}}: Let $a$ be an arc in $\mathcal{S}_{\Lambda(\beta)}$. Thus, near $a$, $\mathcal{S}_{\Lambda(\beta)}$ consists of $a$, an upper $2$-dimensional stratum $U$, and a lower $2$-dimensional stratum $D$, as illustrated in Sub-figure~\eqref{sub-fig: Strata near an arc}. Choose arbitrary points $p \in U$ and $q \in D$, and denote the stalks of $\sh{G}$, $\sh{H}$, and $\sh{F}$ at these points by
\begin{equation*}
\begin{aligned}
&X = \sh{G}_p, \qquad Y = \sh{H}_p, \qquad Z = \sh{F}_p, \\
&A = \sh{G}_q, \qquad B = \sh{H}_q, \qquad C = \sh{F}_q.        
\end{aligned}
\end{equation*}
Thus, near $a$, the microlocal support conditions assert that $\sh{G}$, $\sh{H}$, and $\sh{F}$ are specified by linear maps $\gamma_{\sh{G}}:A\to X$, $\Gamma_{\sh{H}}:B\to Y$, and $\gamma_{\sh{F}}:C\to Z$. Moreover, since $\sh{H}$ is an extension of $\sh{F}$ by $\sh{G}$, we deduce that $\Gamma_{\sh{H}}$ is an extension of $\gamma_{\sh{F}}$ by $\gamma_{\sh{G}}$, see Definition~\eqref{Def: Extensions of linear maps}. 

Next, recall that (1) $\mathrm{Ext}^{1}(\sh{F},\sh{G})$ is identified with the set of equivalence classes of extensions of $\sh{F}$ by $\sh{G}$, and (2) every extension of linear maps is equivalent to a block extension, as established in Lemma~\eqref{Lemma: equivalence of extensions and block extensions}. Hence, we deduce that there exists a linear map $\gamma_{\sh{H}}:C\to X$ such that, viewed as an element of $\mathrm{Ext}^{1}(\sh{F},\sh{G})$, $\sh{H}$ is locally specified near $a$ by the block extension $\elmG:A\oplus C \to X\oplus Z$ of $\gamma_{\sh{F}}$ by $\gamma_{\sh{G}}$ associated with $\gamma_{\sh{H}}$ (see Lemma~\eqref{Lemma: block extensions of linear maps}). In particular, we call $\gamma_{\sh{H}}$ the \emph{characteristic map} of $\sh{H}$ near $a$. 

Let $\sh{H}'$ be an extension of $\sh{F}$ by $\sh{G}$ representing the same equivalence class as $\sh{H}$ in $\mathrm{Ext}^{1}(\sh{F},\sh{G})$, so that there exists a sheaf isomorphism $\lambda: \sh{H} \rightarrow \sh{H}'$ such that the diagram in Figure~\eqref{Equivalen extenions of F by F'} commutes in each square. Analogously to $\sh{H}$, we have that there exists a linear map $\gamma_{\sh{H}'}:C\to X$ such that, viewed as an element of $\mathrm{Ext}^{1}(\sh{F},\sh{G})$, $\sh{H}'$ is locally specified near $a$ by the block extension $\elmg{F}{G}{H'}:A\oplus C \to X\oplus Z$ of $\gamma_{\sh{F}}$ by $\gamma_{\sh{G}}$ associated with $\gamma_{\sh{H}'}$ (the characteristic map of $\sh{H}'$ near $a$). Consequently, since $\sh{H}$ and $\sh{H}'$ represent the same equivalence class in $\mathrm{Ext}^{1}(\sh{F},\sh{G})$, we deduce that $\elmG$ and $\elmg{F}{G}{H'}$ must be equivalent extensions of $\gamma_{\sh{F}}$ by $\gamma_{\sh{G}}$. It then follows from Lemma~\eqref{Lemma: equivalence between block extensions} that there exist linear maps $\lambda_{\xi}:C\rightarrow A$ and $\lambda_{\delta}:Z\rightarrow X$
such that 
\begin{equation*}
\gamma_{\sh{H}'}=\gamma_{\sh{H}}+\lambda_{\delta}\circ \gamma_{\sh{F}} - \gamma_{\sh{G}}\circ \lambda_{\xi}\, ,      
\end{equation*}  
with the pair $(\lambda_{\xi}, \lambda_{\delta})$ realizing the sheaf isomorphism $\lambda$ near $a$, providing a concrete description of the relationship between two representatives of the same class in $\mathrm{Ext}^{1}(\sh{F},\sh{G})$ at the level of characteristic maps near an arbitrary arc $a$.   

\noindent
$\star$ \emph{\textbf{Cusps}}: Let $c$ be a cusp in $\mathcal{S}_{\Lambda(\beta)}$. Near $c$, $\sh{F}$ and $\sh{G}$ are specified by two pairs of linear maps between vector spaces, say $\alpha_{\sh{F}}:X_{1}\to X_{2}$, $\beta_{\sh{F}}:X_{2}\to X_{1}$, $\alpha_{\sh{G}}:Z_{1}\to Z_{2}$, and $\beta_{\sh{G}}:Z_{2}\to Z_{1}$ such that: 
\begin{equation*}
\beta_{\sh{F}}\circ \alpha_{\sh{F}}=\mathrm{id}_{X_{1}}\, ,\quad \text{and} \quad \beta_{\sh{G}}\circ \alpha_{\sh{G}}=\mathrm{id}_{Z_{1}}\, .
\end{equation*}
Hence, based on our previous discussion, an extension $\sh{H}$ of $\sh{F}$ by $\sh{G}$ is specified by two characteristic maps $\alpha_{\sh{H}}:X_{1}\to Z_{2}$ and $\beta_{\sh{H}}:X_{2}\to Z_{1}$. 

Now, let $\be{\alpha_{\sh{F}}}{\alpha_{\sh{G}}}{ \alpha_{\scriptscriptstyle \sh{H}}  }:Z_{1}\oplus X_{1}\to Z_{2}\oplus X_{2}$ and $\be{\beta_{\sh{F}}}{\beta_{\sh{G}}}{ \beta_{\scriptscriptstyle \sh{H}}  }: Z_{2}\oplus X_{2}\to Z_{1}\oplus X_{1}$ be the block extensions of $\alpha_{\sh{F}}$ by $\alpha_{\sh{G}}$ and of $\beta_{\sh{F}}$ by $\beta_{\sh{G}}$ associated with the characteristic maps $\alpha_{\sh{H}}$ and $\beta_{\sh{H}}$, respectively. That is, the linear maps describing $\sh{H}$ near $c$, which explicitly read as
\begin{equation*}
\be{\alpha_{\sh{F}}}{\alpha_{\sh{G}}}{ \alpha_{\scriptscriptstyle \sh{H}} }=\begin{bmatrix}
\alpha_{\sh{G}} & \alpha_{\sh{H}}\\
0 & \alpha_{\sh{F}}
\end{bmatrix}\, , \quad \text{and} \quad \be{\beta_{\sh{F}}}{\beta_{\sh{G}}}{ \beta_{\scriptscriptstyle \sh{H}} }=\begin{bmatrix}
\beta_{\sh{G}} & \beta_{\sh{H}}\\
0 & \beta_{\sh{F}}
\end{bmatrix}\, .   
\end{equation*}
Then, by the microlocal support conditions near the cusps, we require the composition $\be{\beta_{\sh{F}}}{\beta_{\sh{G}}}{ \beta_{\scriptscriptstyle \sh{H}} } \circ\be{\alpha_{\sh{F}}}{\alpha_{\sh{G}}}{ \alpha_{\scriptscriptstyle \sh{H}} }$ to be an isomorphism. In particular, observe that
\begin{equation*}
\be{\beta_{\sh{F}}}{\beta_{\sh{G}}}{ \beta_{\scriptscriptstyle \sh{H}} } \circ\be{\alpha_{\sh{F}}}{\alpha_{\sh{G}}}{ \alpha_{\scriptscriptstyle \sh{H}} }= \begin{bmatrix}
\beta_{\sh{G}} \circ \alpha_{\sh{G}} & ~~ & \beta_{\sh{G}} \circ \alpha_{\sh{H}}+ \beta_{\sh{H}} \circ \alpha_{\sh{F}} \\
0 & ~~ & \beta_{\sh{F}} \circ \alpha_{\sh{F}}
\end{bmatrix}\, .
\end{equation*}
Consequently, since $\beta_{\sh{F}}\circ \alpha_{\sh{F}}=\mathrm{id}_{X_{1}}$ and $\beta_{\sh{G}}\circ \alpha_{\sh{G}}=\mathrm{id}_{Z_{1}}$, we conclude that $\be{\beta_{\sh{F}}}{\beta_{\sh{G}}}{ \beta_{\scriptscriptstyle \sh{H}} } \circ\be{\alpha_{\sh{F}}}{\alpha_{\sh{G}}}{ \alpha_{\scriptscriptstyle \sh{H}} }$ is automatically an isomorphism. Therefore, the microlocal support conditions near the cusps impose no additional constraints on extensions of $\sh{F}$ by $\sh{G}$. 

\noindent
$\star$ \emph{\textbf{Crossings}}: Let $x$ be a crossing in $\mathcal{S}_{\Lambda(\beta)}$. Near $x$, each of $\sh{F}$ and $\sh{G}$ is determined by four linear maps between certain vector spaces, namely $\alpha_{\sh{F}}:A\to B$, $\beta_{\sh{F}}:B\to C$, $\widetilde{\alpha}_{\sh{F}}:A\to D$, $\widetilde{\beta}_{\sh{F}}:D\to C$, and $\alpha_{\sh{G}}:X\to Y$, $\beta_{\sh{G}}:Y\to Z$,  $\widetilde{\alpha}_{\sh{G}}:X\to W$, $\widetilde{\beta}_{\sh{G}}:W\to Z$, respectively. By the microlocal support conditions near the crossings, we have that: 
\begin{itemize}
\justifying
\item $\beta_{\sh{F}}\circ \alpha_{\sh{F}} = \widetilde{\beta}_{\sh{F}}\circ\widetilde{\alpha}_{\sh{F}}$ and $\beta_{\sh{G}}\circ \alpha_{\sh{G}} = \widetilde{\beta}_{\sh{G}}\circ\widetilde{\alpha}_{\sh{G}} $.
\item The following sequences are short exact: 
\begin{equation*}
\begin{tikzpicture}
\useasboundingbox (-4.5,-0.25) rectangle (4.5,1.15);
\scope[transform canvas={scale=0.8}]
\node at (-5.5-2,0.05) {\large$0$}; 
\node at (-3-1-0.5,0.05) {\large$ A$};    
\node at (0,0.05) {\large$B \oplus D$};   
\node at (3+1+0.5,0.05) {\large$C$};
\node at (5.5+2,0.05) {\large$0$}; 

\draw[->, shorten <=0.85cm,  shorten >=0.6cm, line width=0.02cm] (-8, 0) -- (-4-0.5,0);
\draw[->, shorten <=0.6cm,  shorten >=1.0cm, line width=0.02cm] (-4-0.5, 0) -- (0,0);
\draw[->, shorten <=1.0cm,  shorten >=0.6cm, line width=0.02cm] (0, 0) -- (4+0.5,0);
\draw[->, shorten <=0.6cm,  shorten >=0.85cm, line width=0.02cm] (4+0.5, 0) -- (8,0);

\node at (-1.5-0.75-0.25, 0.7) {\normalsize$\Big(\, \alpha_{\sh{F}},\, \widetilde{\alpha}_{\sh{F}}\,\Big)$};
\node at (1.5+0.75+0.25, 0.7) {\normalsize$\beta_{\sh{F}} \oplus\Big(-\widetilde{\beta}_{\sh{F}}\,\Big)$};

\node at (5.8+2,-0.025) {\large$.$}; 

\endscope
\end{tikzpicture}  
\end{equation*}
 
\begin{equation*}
\begin{tikzpicture}
\useasboundingbox (-4.5,-0.25) rectangle (4.5,1.15);
\scope[transform canvas={scale=0.8}]
\node at (-5.5-2,0.05) {\large$0$}; 
\node at (-3-1-0.5,0.05) {\large$ X$};    
\node at (0,0.05) {\large$Y \oplus W$};   
\node at (3+1+0.5,0.05) {\large$Z$};
\node at (5.5+2,0.05) {\large$0$}; 

\draw[->, shorten <=0.85cm,  shorten >=0.6cm, line width=0.02cm] (-8, 0) -- (-4-0.5,0);
\draw[->, shorten <=0.6cm,  shorten >=1.0cm, line width=0.02cm] (-4-0.5, 0) -- (0,0);
\draw[->, shorten <=1.0cm,  shorten >=0.6cm, line width=0.02cm] (0, 0) -- (4+0.5,0);
\draw[->, shorten <=0.6cm,  shorten >=0.85cm, line width=0.02cm] (4+0.5, 0) -- (8,0);

\node at (-1.5-0.75-0.25, 0.7) {\normalsize$\Big(\, \alpha_{\sh{G}},\, \widetilde{\alpha}_{\sh{G}}\,\Big)$};
\node at (1.5+0.75+0.25, 0.7) {\normalsize$\beta_{\sh{G}} \oplus\Big(-\widetilde{\beta}_{\sh{G}}\,\Big)$};

\node at (5.8+2,-0.025) {\large$.$}; 

\endscope
\end{tikzpicture}  
\end{equation*}
\end{itemize}
Thus, building on our previous discussion, an extension $\sh{H}$ of $\sh{F}$ by $\sh{G}$ is determined by a collection of four characteristic maps $\alpha_{\sh{H}}:A\to Y$, $\beta_{\sh{H}}:B\to Z$, $\widetilde{\alpha}_{\sh{H}}:A\to W$, and $\widetilde{\beta}_{\sh{H}}:D\to Z$.    

Next, consider the block extensions $\be{\alpha_{\sh{F}}}{\alpha_{\sh{G}}}{ \alpha_{\scriptscriptstyle \sh{H}}  }:X \oplus A\to Y\oplus B$,  $\be{\beta_{\sh{F}}}{\beta_{\sh{G}}}{ \beta_{\scriptscriptstyle \sh{H}}  }:Y \oplus B \to Z\oplus C$, $\be{\widetilde{\alpha}_{\sh{F}}}{\widetilde{\alpha}_{\sh{G}}}{ \widetilde{\alpha}_{\scriptscriptstyle \sh{H}}  }:X \oplus A\to W\oplus D$, and $\be{\widetilde{\beta}_{\sh{F}}}{\widetilde{\beta}_{\sh{G}}}{ \widetilde{\beta}_{\scriptscriptstyle \sh{H}}  }:W \oplus D \to Z\oplus C$. That is, the linear maps determining $\sh{H}$ near $x$, which explicitly read as
\begin{equation*}
\begin{aligned}
\be{\alpha_{\sh{F}}}{\alpha_{\sh{G}}}{ \alpha_{\scriptscriptstyle \sh{H}} }=\begin{bmatrix}
\alpha_{\sh{G}} & \alpha_{\sh{H}}\\
0 & \alpha_{\sh{F}}
\end{bmatrix}\, , \quad \text{and} \quad \be{\beta_{\sh{F}}}{\beta_{\sh{G}}}{ \beta_{\scriptscriptstyle \sh{H}} }=\begin{bmatrix}
\beta_{\sh{G}} & \beta_{\sh{H}}\\
0 & \beta_{\sh{F}}
\end{bmatrix}\, ,  \\[6pt]
\be{\widetilde{\alpha}_{\sh{F}}}{\widetilde{\alpha}_{\sh{G}}}{ \widetilde{\alpha}_{\scriptscriptstyle \sh{H}} }=\begin{bmatrix}
\widetilde{\alpha}_{\sh{G}} & \widetilde{\alpha}_{\sh{H}}\\
0 & \widetilde{\alpha}_{\sh{F}}
\end{bmatrix}\, , \quad \text{and} \quad \be{\widetilde{\beta}_{\sh{F}}}{\widetilde{\beta}_{\sh{G}}}{ \widetilde{\beta}_{\scriptscriptstyle \sh{H}} }=\begin{bmatrix}
\widetilde{\beta}_{\sh{G}} & \widetilde{\beta}_{\sh{H}}\\
0 & \widetilde{\beta}_{\sh{F}}
\end{bmatrix}\, . 
\end{aligned}
\end{equation*}
Then, the microlocal support conditions near the crossings require:
\begin{itemize}
\justifying
\item[(a)] $\be{\beta_{\sh{F}}}{\beta_{\sh{G}}}{ \beta_{\scriptscriptstyle \sh{H}} } \circ \be{\alpha_{\sh{F}}}{\alpha_{\sh{G}}}{ \alpha_{\scriptscriptstyle \sh{H}} }=\be{\widetilde{\beta}_{\sh{F}}}{\widetilde{\beta}_{\sh{G}}}{ \widetilde{\beta}_{\scriptscriptstyle \sh{H}} } \circ \be{\widetilde{\alpha}_{\sh{F}}}{\widetilde{\alpha}_{\sh{G}}}{ \widetilde{\alpha}_{\scriptscriptstyle \sh{H}} }$. 
\item[(b)] The following sequence must be short exact: 
\begin{equation*}
\begin{tikzpicture}
\useasboundingbox (-4.5,-0.25) rectangle (4.5,1.25);
\scope[transform canvas={scale=0.8}]
\node at (-5.5-2,0.05) {\large$0$}; 
\node at (-3-1-0.5,0.05) {\large$ X\oplus A$};    
\node at (0,0.05) {\large$ (Y\oplus B) \oplus (W\oplus D)$};   
\node at (3+1+0.5,0.05) {\large$Z\oplus C$};
\node at (5.5+2,0.05) {\large$0$}; 

\draw[->, shorten <=0.85cm,  shorten >=0.8cm, line width=0.02cm] (-8, 0) -- (-4-0.5,0);
\draw[->, shorten <=0.8cm,  shorten >=1.8cm, line width=0.02cm] (-4-0.5, 0) -- (0,0);
\draw[->, shorten <=1.8cm,  shorten >=0.8cm, line width=0.02cm] (0, 0) -- (4+0.5,0);
\draw[->, shorten <=0.8cm,  shorten >=0.85cm, line width=0.02cm] (4+0.5, 0) -- (8,0);

\node at (-1.5-0.75-0.25, 0.7+0.25) {\normalsize$\big(\,\be{\alpha_{\sh{F}}}{\alpha_{\sh{G}}}{ \alpha_{\scriptscriptstyle \sh{H}} } ,\,  \be{\widetilde{\alpha}_{\sh{F}}}{\widetilde{\alpha}_{\sh{G}}}{ \widetilde{\alpha}_{\scriptscriptstyle \sh{H}} } \,\big)$};
\node at (1.5+0.75+0.25, 0.7+0.25) {\normalsize$ \be{\beta_{\sh{F}}}{\beta_{\sh{G}}}{ \beta_{\scriptscriptstyle \sh{H}} }\oplus\big(- \be{\widetilde{\beta}_{\sh{F}}}{\widetilde{\beta}_{\sh{G}}}{ \widetilde{\beta}_{\scriptscriptstyle \sh{H}} } \,\big)$};

\node at (5.8+2,-0.025) {\large$.$}; 

\endscope
\end{tikzpicture}  
\end{equation*}
\end{itemize}
In fact, once the commutativity condition~(a) holds, the defining properties of the maps characterizing $\sh{F}$ and $\sh{G}$ near $x$ automatically guarantee that the sequence in~(b) is short exact. In other words, the microlocal support conditions near the crossings only impose on $\sh{H}$ the commutativity condition~(a). Concretely, a direct calculation shows that at the level of characteristic maps, the constraint on $\sh{H}$ is given by
\begin{equation*}
\beta_{\sh{G}}\circ\alpha_{\sh{H}}+\beta_{\sh{H}}\circ \alpha_{\sh{F}}=\widetilde{\beta}_{\sh{G}}\circ \widetilde{\alpha}_{\sh{H}}+\widetilde{\beta}_{\sh{H}}\circ \widetilde{\alpha}_{\sh{F}}\, .     
\end{equation*}
\end{observation}

With the above observation at hand, we close this subsection by presenting several technical lemmas that will play a crucial role in the explicit computation of the one-degree morphism spaces in the category $\ccs{1}{\beta}$. 

\begin{lemma}\label{Lemma: system of equations for equivalent collections of injective maps}
Fix an integer $n\geq 2$, let $\big\{ X^{(i)}\in \mathrm{M}(i+1,i, \mathbb{K}) \big\}_{i=1}^{n-1}$ be a collection of matrices, and consider the system of equations 
\begin{equation*}
X'^{\,(i)} = X^{(i)}+ Y^{(i+1)}\cdot \iota^{(i+1,i)}-\iota^{(i+1,i)}\cdot Y^{(i)}\, , \qquad \text{for all $i\in [1,n-1]$}\, ,     
\end{equation*}
with $X'^{\,(i)}\in \mathrm{M}(i+1, i, \mathbb{K})$ and  $Y^{(j)}\in \mathrm{M}(j,\mathbb{K})$ for all $i\in [1,n-1]$ and $j\in [1,n]$. Then: 
\begin{itemize}
\justifying
\item \textbf{General Case}: For any choice of $\big\{ X'^{\,(i)}\big\}_{i=1}^{n-1}$ there exists a collection of matrices $\big\{ Y^{(i)}\big\}_{i=1}^{n}$ solving the system, with the upper-triangular part of $Y^{(n)}$ chosen arbitrarily. More precisely, the solution space of the system is affine over the entries of the matrices $\big\{X'^{\,(i)}\big\}_{i=1}^{n-1}$ and the upper-triangular part of $Y^{(n)}$. 
\item  \textbf{Special case}: if $X^{(i)}=X'^{\,(i)}=\mathbf{0}_{(i+1)\times i}$ for all $i\in[1,n-1]$, then every solution $\big\{ Y^{(i)}\big\}_{i=1}^{n}$ of the system arises from an arbitrary upper-triangular matrix $Y^{(n)}$, with $Y^{(i)}$ given by the principal $i\times i$ submatrix of $Y^{(n)}$, for all $i\in[1,n-1]$. 
\end{itemize}
\end{lemma}
\begin{proof}
The proof follows by induction on $n-1$, keeping track of the entries of $X'^{(i)}$ and the upper-triangular part of $Y^{(n)}$.
\end{proof}

\begin{lemma}\label{Lemma: system of equations for equivalent collections of surjective maps}
Fix an integer $n\geq 2$, let $\big\{ X^{(i)}\in \mathrm{M}(i,i+1, \mathbb{K}) \big\}_{i=1}^{n-1}$ be a collection of matrices, and consider the system of equations 
\begin{equation*}
X'^{\,(i)} = X^{(i)}+ Y^{(i)}\cdot \pi^{(i,i+1)}-\pi^{(i,i+1)}\cdot Y^{(i+1)}\, , \qquad \text{for all $i\in [1,n-1]$}\, ,     
\end{equation*}
with $X'^{\,(i)}\in \mathrm{M}(i, i+1, \mathbb{K})$ and  $Y^{(j)}\in \mathrm{M}(j,\mathbb{K})$ for all $i\in [1,n-1]$ and $j\in [1,n]$. Then: 
\begin{itemize}
\justifying
\item \textbf{General Case}: For any choice of $\big\{ X'^{\,(i)}\big\}_{i=1}^{n-1}$ there exists a collection of matrices $\big\{Y^{(i)}\big\}_{i=1}^{n}$ solving the system, with the lower-triangular part of $Y^{(n)}$ chosen arbitrarily. More precisely, the solution space of the system is affine over the entries of the matrices $\big\{X'^{\,(i)}\big\}_{i=1}^{n-1}$ and the lower-triangular part of $Y^{(n)}$. 
\item  \textbf{Special case}: if $X^{(i)}=X'^{\,(i)}=\mathbf{0}_{i\times (i+1)}$ for all $i\in[1,n-1]$, then every solution $\big\{ Y^{(i)}\big\}_{i=1}^{n}$ of the system arises from an arbitrary lower-triangular matrix $Y^{(n)}$, with $Y^{(i)}$ given by the principal $i\times i$ submatrix of $Y^{(n)}$, for all $i\in[1,n-1]$. 
\end{itemize}
\end{lemma}
\begin{proof}
The proof follows by induction on $n-1$, keeping track of the entries of $X'^{(i)}$ and the lower-triangular part of $Y^{(n)}$.
\end{proof}

\begin{lemma}\label{Lemma: system of equations for equivalent extensions at a crossing sigma_1}
Fix an integer $n\geq 2$, and let $\bigl\{\, Y{^{(i)}}\in \mathrm{M}(i, \mathbb{K}) \,\bigr\}_{i=1}^{n}$ be a collection of matrices such that: 
\begin{equation*}
\begin{aligned}
Y^{(n)}&= D(\vec{u})\, , \quad \text{for some $\vec{u}=(u_{1},\dots, u_{n})\in \mathbb{K}^{n}_{\mathrm{std}}$}\, .   \\[2pt]
Y^{(i)}&= \text{ principal $i\times i$ submatrix of }\,D(\vec{u})\, , \quad \text{for all $i\in [1,n-1]$}\, .
\end{aligned} 
\end{equation*}
Furthermore, let $\widetilde{X}^{(1)}\in \mathrm{M}(2,1, \mathbb{K})$ be a fixed matrix, and let $z, z'\in \mathbb{K}$ be two fixed parameters. Then the system of equations
\begin{equation}\label{Eq: system of equations for k=1}
\widetilde{X}'^{\,(1)}=\widetilde{X}^{(1)} + \Big(\big(B^{(2)}_{1}(z')\big)^{-1}\cdot Y^{(2)}\cdot B^{(2)}_{1}(z)\Big)\cdot \iota^{(2,1)}-\iota^{(2,1)}\cdot \widetilde{Y}^{(1)} \, , 
\end{equation}
with $\widetilde{X}'^{\,(1)}\in \mathrm{M}(2,1,\mathbb{K})$ and $\widetilde{Y}^{(1)}\in \mathrm{M}(1,\mathbb{K})$, has the following properties: 
\begin{itemize}
\justifying
\item[(a)] \textbf{General Solution}: The solution set $\bigl(\widetilde{X}'^{\,(1)}, \widetilde{Y}^{(1)}\bigr)$ of~\eqref{Eq: system of equations for k=1} is an affine line parametrized by the $(1,1)$-entry of $\widetilde{X}'^{\,(1)}$. More precisely, once the $(1,1)$-entry of $\widetilde{X}'^{\,(1)}$ is chosen, the system of equations~\eqref{Eq: system of equations for k=1} uniquely determines $\widetilde{Y}^{(1)}$ and the remaining entry of $\widetilde{X}'^{\,(1)}$.  

\item[(b)] \textbf{Special case}: Consider
\begin{equation*}
\widetilde{X}'^{\,(1)}=\begin{bmatrix}
0\\ 
\tilde{x}'_{2,1}
\end{bmatrix}\, , \quad \text{and} \quad \widetilde{X}^{(1)}=\begin{bmatrix}
0\\ 
\tilde{x}_{2,1}
\end{bmatrix}\, .
\end{equation*} 
Then the solution set $\bigl(\widetilde{X}'^{\,(1)}, \widetilde{Y}^{(1)}\bigr)$ of~\eqref{Eq: system of equations for k=1} reduces to a single point:
\begin{equation*}
\begin{aligned}
~~\widetilde{Y}^{(1)}&= \,\text{principal $1\times 1$ submatrix of } \,\big(B^{(n)}_{1}(z')\big)^{-1}\cdot D(\vec{u})\cdot B^{(n)}_{1}(z)\,  ,\\[2pt]  
~~\tilde{x}_{2,1}'&=\tilde{x}_{2,1} + d(z', \vec{u},z)\, ,
\end{aligned}    
\end{equation*}
where
\begin{equation*}
d(z',\vec{u},z):= \text{the $(2,1)$-entry of } \,\big(B^{(n)}_{1}(z')\big)^{-1}\cdot D(\vec{u})\cdot B^{(n)}_{1}(z)\, .     
\end{equation*}
\end{itemize}
\end{lemma}
\begin{proof}
Set $M:=\big(B^{(2)}_{1}(z')\big)^{-1}\cdot Y^{(2)}\cdot B^{(2)}_{1}(z)$. Then, since $Y^{(2)}$ is the principal $2\times 2$ submatrix of $D(\vec{u})$, the block structure of the braid matrices guarantees that
\begin{equation*}
M= \text{ principal $2\times 2$ submatrix of }~ \big(B^{(n)}_{1}(z')\big)^{-1}\cdot D(\vec{u})\cdot B^{(n)}_{1}(z) \, .
\end{equation*}

Next, let us write
\begin{equation*}
\widetilde{X}'^{\,(1)}=\begin{bmatrix}
\tilde{x}'_{1,1}\\ 
\tilde{x}'_{2,1}
\end{bmatrix}\, , \qquad \widetilde{X}^{(1)}=\begin{bmatrix}
\tilde{x}_{1,1}\\ 
\tilde{x}_{2,1}
\end{bmatrix}\, , \qquad \widetilde{Y}^{(1)}=[\,\tilde{y}_{1,1}\, ]\, .    
\end{equation*}
Then, the system of equations in~\eqref{Eq: system of equations for k=1} reduces to
\begin{equation*}
\begin{aligned}
\tilde{x}'_{1,1}&=\tilde{x}_{1,1}+m_{1,1}-\tilde{y}_{1,1}\, ,   \\[4pt]
\tilde{x}'_{2,1}&=\tilde{x}_{2,1}+m_{2,1}\, ,   \\
\end{aligned} 
\end{equation*}
where $m_{1,1}$ and $m_{2,1}$ denote the $(1,1)$ and $(2,1)$ entries of $M$, respectively. This shows that once $\tilde{x}'_{1,1}$ is chosen, $\tilde{y}_{1,1}$ and $\tilde{x}'_{2,1}$ are uniquely determined. In other words, the solution set $\bigl(\widetilde{X}'^{\,(1)}, \widetilde{Y}^{(1)}\bigr)$ of~\eqref{Eq: system of equations for k=1} is an affine line parametrized by the $(1,1)$-entry of $\widetilde{X}'^{\,(1)}$.

Finally, in the special case $\tilde{x}_{1,1} = \tilde{x}'_{1,1} = 0$, the solution of the system of equations in~\eqref{Eq: system of equations for k=1} is uniquely given by
\begin{equation*}
\begin{aligned}
\tilde{y}_{1,1}&=m_{1,1}\, ,   \\[2pt]
\tilde{x}'_{2,1}&=\tilde{x}_{2,1}+m_{2,1}\, .   \\
\end{aligned} 
\end{equation*}
The result follows once we identify $d(z', \vec{u},z)=m_{2,1}$.  
\end{proof}

\begin{lemma}\label{Lemma: system of equations for the extension crossing constraint}
Fix an integer $k\geq 2$. Let $\widetilde{X}^{(k-1)}\in \mathrm{M}(k,k-1, \mathbb{K})$ and $\widetilde{X}^{(k)}\in \mathrm{M}(k+1,k, \mathbb{K})$ be fixed matrices, and let $z'\in \mathbb{K}$ be a fixed parameter. Consider the system of equations 
\begin{equation}\label{Eq: commutativity conditions for extensions}
\big(B^{(k+1)}_{k}(z')\big)^{-1}\cdot\bigg( \iota^{(k+1, k)}\cdot \widetilde{X}^{(k-1)}+ \widetilde{X}^{(k)}\cdot\iota^{(k, k-1)} \bigg)=  \iota^{(k+1, k)}\cdot \widetilde{X}'^{\,(k-1)}+ \widetilde{X}'^{\,(k)}\cdot\iota^{(k, k-1)}\, ,   
\end{equation}
with $\widetilde{X}'^{\,(k-1)}\in \mathrm{M}(k,k-1, \mathbb{K})$ and $\widetilde{X}'^{\,(k)}\in \mathrm{M}(k+1,k, \mathbb{K})$. Then the solution set $\big(\widetilde{X}'^{\,(k-1)}, \widetilde{X}'^{\,(k)}\big)$ of~\eqref{Eq: commutativity conditions for extensions} is an affine space parametrized by the entries of $\widetilde{X}'^{\,(k-1)}$ and the entries of the $k$-th column of $\widetilde{X}'^{\,(k)}$. More precisely, once $\widetilde{X}'^{\,(k-1)}$ and the $k$-th column of $\widetilde{X}'^{\,(k)}$ are chosen, the system of equations~\eqref{Eq: commutativity conditions for extensions} uniquely determines the first $k-1$ columns of $\widetilde{X}'^{\,(k)}$. 
\end{lemma}
\begin{proof}
The system of equations~\eqref{Eq: commutativity conditions for extensions} is linear in $\widetilde{X}'^{\,(k-1)}$ and $\widetilde{X}'^{\,(k)}$. In particular, since $\iota^{(k,k-1)}$ has only $k-1$ columns, the $k$-th column of $\widetilde{X}'^{\,(k)}$ is not involved in the system of equations, and is therefore free. The remaining columns of $\widetilde{X}'^{\,(k)}$ are uniquely determined once $\widetilde{X}'^{\,(k-1)}$ is chosen. It follows that the solution set $\big(\widetilde{X}'^{\,(k-1)}, \widetilde{X}'^{\,(k)}\big)$ of~\eqref{Eq: commutativity conditions for extensions} is an affine space parametrized by the entries of $\widetilde{X}'^{\,(k-1)}$ and the entries in the $k$-th column of $\widetilde{X}'^{\,(k)}$.
\end{proof}

\begin{lemma}\label{Lemma: system of equations for equivalent extensions at a crossing sigma_k}
Fix an integer $n\geq 2$, and let $\bigl\{\, Y{^{(i)}}\in \mathrm{M}(i, \mathbb{K}) \,\bigr\}_{i=1}^{n}$ be a collection of matrices such that: 
\begin{equation*}
\begin{aligned}
Y^{(n)}&= D(\vec{u})\, , \quad \text{for some $\vec{u}=(u_{1},\dots, u_{n})\in \mathbb{K}^{n}_{\mathrm{std}}$}\, .   \\[2pt]
Y^{(i)}&= \text{ principal $i\times i$ submatrix of }\,D(\vec{u})\, , \quad \text{for all $i\in [1,n-1]$}\, .
\end{aligned} 
\end{equation*}

In addition, fix an integer $k\in [2,n-1]$. Let $\widetilde{X}^{(k-1)}\in \mathrm{M}(k,k-1, \mathbb{K})$ and $\widetilde{X}^{(k)}\in \mathrm{M}(k+1,k, \mathbb{K})$ be fixed matrices, and let $z,z'\in\mathbb K$ be two fixed parameters. Then the system of equations
\begin{subequations}\label{Eq: system of equations for k geq 2}
\begin{align}
\widetilde{X}'^{\,(k-1)} &= \widetilde{X}^{(k-1)} + \widetilde{Y}^{(k)}\cdot\iota^{(k,k-1)} - \iota^{(k,k-1)}\cdot Y^{(k-1)}, \label{Eq: system of equations for k geq 2_a} \\[4pt]
\widetilde{X}'^{\,(k)}   &= \widetilde{X}^{(k)} + \big(B^{(k+1)}_{k}(z')\big)^{-1}\cdot Y^{(k+1)}\cdot B^{(k+1)}_k(z)\cdot\iota^{(k+1,k)} - \iota^{(k+1,k)}\cdot \widetilde{Y}^{(k)}, \label{Eq: system of equations for k geq 2_b}
\end{align}   
\end{subequations}    
with $\widetilde{X}'^{\,(k-1)}\in \mathrm{M}(k,k-1, \mathbb{K})$, $\widetilde{X}'^{\,(k)}\in \mathrm{M}(k+1,k, \mathbb{K})$, and $\widetilde{Y}^{(k)}\in \mathrm{M}(k, \mathbb{K})$, has the following properties:
\begin{itemize}
\justifying
\item[(a)] \textbf{General case}: The solution set $\bigl(\widetilde{X}'^{\,(k-1)}, \widetilde{X}'^{\,(k)}, \widetilde{Y}^{(k)} \bigr)$ of~\eqref{Eq: system of equations for k geq 2} is an affine space parametrized by the entries of $\widetilde{X}'^{\,(k-1)}$ and the entries of the $k$-th column of $\widetilde{X}'^{\,(k)}$ except its last one. More precisely, once $\widetilde{X}'^{\,(k-1)}$ and the $k$-th column of $\widetilde{X}'^{\,(k)}$ except its last entry are chosen, the system of equations~\eqref{Eq: system of equations for k geq 2} uniquely determines $\widetilde{Y}^{(k)}$ and the remaining entries of $\widetilde{X}'^{\,(k)}$. 

\item[(b)] \textbf{Special case}: Consider
\begin{equation*}
\begin{aligned}
\widetilde{X}'^{\,(k-1)}=\mathbf{0}_{k\times (k-1)}\, , \quad \text{and} \quad \widetilde{X}'^{\,(k)}=\begin{bmatrix}
\tilde{x}'_{1,1} & \cdots & \widetilde{x}'_{1,k-1} & 0\\
\vdots     & \ddots &   \vdots     & \vdots\\
\tilde{x}'_{k,1} & \cdots & \widetilde{x}'_{k, k-1} & 0 \\
\tilde{x}'_{k+1,1} & \cdots & \widetilde{x}'_{k+1, k-1} & \widetilde{x}'_{k+1,k}
\end{bmatrix}\, , \\[8pt]
\widetilde{X}^{(k-1)}=\mathbf{0}_{k\times (k-1)}\, , \quad \text{and} \quad  \widetilde{X}^{(k)}=\begin{bmatrix}
x_{1,1} & \cdots & x_{1,k-1} & 0\\
\vdots     & \ddots &   \vdots     & \vdots\\
x_{k,1} & \cdots & x_{k, k-1} & 0 \\
x_{k+1,1} & \cdots & x_{k+1, k-1} & x_{k+1,k}
\end{bmatrix}\, .   
\end{aligned}
\end{equation*}
Then the solution set $\bigl(\widetilde{X}'^{\,(k-1)}, \widetilde{X}'^{\,(k)}, \widetilde{Y}^{(k)} \bigr)$  of~\eqref{Eq: system of equations for k geq 2} reduces to a single point: 
\begin{equation*}
\begin{aligned}
~\widetilde{Y}^{(k)}&= \text{principal $k\times k$ submatrix of } ~\big(B^{(n)}_{k}(z')\big)^{-1}\cdot D(\vec{u}) \cdot B^{(n)}_{k}(z)\,  .\\[4pt]   
~\tilde{x}'_{i,j}&=\tilde{x}_{i,j}\, , \quad \text{for all $i\in[1,k+1]$ and $j\in [1,k-1]$} \, , \\[4pt]  
~\tilde{x}'_{k+1,k}&=\tilde{x}_{k+1,k} + d(z', \vec{u}, z) \, ,
\end{aligned}  
\end{equation*}
where
\begin{equation*}
d(z', \vec{u}, z):=\text{the $(k+1,k)$-entry of } \,\big(B^{(n)}_{k}(z')\big)^{-1}\cdot D(\vec{u}) \cdot B^{(n)}_{k}(z) \, .    
\end{equation*}
\end{itemize}
\end{lemma}
\begin{proof}
Set $M=\big(B^{(k+1)}_{k}(z')\big)^{-1}\cdot Y^{(k+1)}\cdot B^{(k+1)}_{k}(z)$. Then, since $Y^{(k+1)}$ is the principal $(k+1)\times (k+1)$ submatrix of $D(\vec{u})$, the block structure of the braid matrices guarantees that
\begin{equation*}
M= \text{ principal $(k+1)\times (k+1)$ submatrix of }~ \big(B^{(n)}_{k}(z')\big)^{-1}\cdot D(\vec{u})\cdot B^{(n)}_{k}(z).    
\end{equation*}

Next, observe that the set of equations~\eqref{Eq: system of equations for k geq 2_a} allows us to solve for the first $k-1$ columns of $\widetilde{Y}^{(k)}$, while leaving its $k$-th column and all entries of $\widetilde{X}'^{\,(k-1)}$ undetermined. Then, after substituting our partial solution into the set of equations~\eqref{Eq: system of equations for k geq 2_b}, we can exploit the remaining freedom in the $k$-th column of $\widetilde{Y}^{(k)}$ to solve the system completely, while keeping all entries of $\widetilde{X}'^{\,(k-1)}$ and the entries of the $k$-th column of $\widetilde{X}'^{\,(k)}$, except for its very last entry, arbitrary.  

Finally, in the special case where $\widetilde{X}^{(k-1)} = \widetilde{X}'^{\,(k-1)} = \mathbf{0}_{k\times (k-1)}$ and the constraints on the last columns of $\widetilde{X}^{(k)}$ and $\widetilde{X}'^{\,(k)}$ required in the lemma are imposed, the procedure described above yields a unique solution. In particular, the set of equations~\eqref{Eq: system of equations for k geq 2_a} implies that the principal $(k-1)\times(k-1)$ submatrix of $\widetilde{Y}^{(k)}$ is equal to $Y^{(k-1)}$, while its remaining entries outside its last column are zero. 

Furthermore, since $Y^{(k+1)}$ is the principal $(k+1)\times(k+1)$ submatrix of the diagonal matrix $D(\vec{u})$, the block structure of the braid matrices guarantee that the principal $(k-1)\times (k-1)$ sub-matrices of $Y^{(k+1)}$ and $M$ coincide. Hence, under the assumptions of the special case, the set of equations~\eqref{Eq: system of equations for k geq 2_b} further implies that $\widetilde{Y}^{(k)}$ is equal to the principal $k\times k$ submatrix of $M$. 

Lastly, recall that $M$ is almost diagonal, with the only potentially non-zero off-diagonal entry at the $(k+1,k)$ position. Consequently, in the set of equations~\eqref{Eq: system of equations for k geq 2_b}, $\widetilde{Y}^{(k)}$ cancels almost all the contribution of $M$, thus yielding that $\widetilde{X}'^{\,(k)}$ coincides with $\widetilde{X}^{(k)}$ entry-wise, except for the $(k+1,k)$-entry, which differs by the corresponding entry of $M$. This completes the proof.
\end{proof}

Having established the relevant technical lemmas and the local structure of the one-degree morphism spaces in the category $\ccs{1}{\beta}$, we turn to their global analysis, aiming for an explicit description of these morphism spaces. As a first step, we consider the case of the Legendrian unlink on $n$ strands, $\beta = e_n$.

\subsubsection{The Case of the Trivial Braid} Let $e_{n}\in\mathrm{Br}^{+}_{n}$ be the trivial braid word. Next, we study the one-degree morphism spaces in the category $\ccs{1}{e_{n}}$. 

To begin, consider the open cover $\mathcal{U}_{\Lambda(e_{n})}=\big\{ U_{0}, U_{\mathrm{B}}, U_{\mathrm{T}}\big\}$ of $\mathbb{R}^{2}$ introduced in Construction~\eqref{Const: finite open cover of R^2 for the unlink}. Let $\sh{F}$ and $\sh{G}$ be objects of the category $\ccs{1}{e_{n}}$, and let $\sh{H}$ be an extension of $\sh{F}$ by $\sh{G}$. By Proposition~\eqref{Sheaves for the unlink on n strands}, $\sh{F}$ and $\sh{G}$ are identically zero on $U_{0}$, and as a result, $\sh{H}$ is also identically zero on this region. It follows that $\sh{H}$ is entirely determined by its behavior on the regions $U_{\mathrm{B}}$ and $U_{\mathrm{T}}$. Having established this, we present the following proposition.

\begin{proposition}
\textbf{Setup}: Let $e_{n}\in\mathrm{Br}^{+}_{n}$ be the trivial braid word, $\mathcal{U}_{\Lambda(e_{n})}=\big\{ U_{0}, U_{\mathrm{B}}, U_{\mathrm{T}}\big\}$ the open cover of $\mathbb{R}^{2}$ introduced in Construction~\eqref{Const: finite open cover of R^2 for the unlink}, and $\sh{F}$ and $\sh{G}$ objects of the category $\ccs{1}{e_{n}}$. In this setting, Proposition~\eqref{Prop: linear map description of a sheaf for the unlink on n strands} asserts that: 
\begin{itemize}
\justifying
\item  On $U_{\mathrm{T}}$, $\sh{F}$ and $\sh{G}$ are specified by two collections of surjective linear maps 
\begin{equation*}
\big\{\psi_{\sh{F}}^{(i)}:\mathbb{K}^{i+1}\to \mathbb{K}^{i}\big\}_{i=1}^{n-1}\, , \quad \text{and} \quad \big\{\psi_{\sh{G}}^{(i)}:\mathbb{K}^{i+1}\to \mathbb{K}^{i}\big\}_{i=1}^{n-1}\, ,    
\end{equation*}
respectively (see Figure~\eqref{Fig: an object F on the region U_T for the unlink} for a schematic representation).

\item On $U_{\mathrm{B}}$, $\sh{F}$ and $\sh{G}$ are specified by two collections of injective linear maps 
\begin{equation*}
\big\{\phi_{\sh{F}}^{(i)}:\mathbb{K}^{i}\to \mathbb{K}^{i+1}\big\}_{i=1}^{n-1}\, , \quad \text{and} \quad \big\{\phi_{\sh{G}}^{(i)}:\mathbb{K}^{i}\to \mathbb{K}^{i+1}\big\}_{i=1}^{n-1} \, ,    
\end{equation*}
respectively (see Figure~\eqref{Fig: an object F on the region U_B for the unlink} for a schematic representation).

\item \textbf{Compatibility conditions}: For each $i\in[1,n-1]$, 
\begin{equation*}
\psi^{(i)}_{\sh{F}}\circ \phi^{(i)}_{\sh{F}}=\mathrm{id}_{\mathbb{K}^{i}}\, , \quad \text{and} \quad \psi^{(i)}_{\sh{G}}\circ \phi^{(i)}_{\sh{G}}=\mathrm{id}_{\mathbb{K}^{i}}\, .    
\end{equation*}
\end{itemize}

\noindent
$\star$ \emph{Assumption}: For each $i\in[1,n]$, let $\hat{\mathbf{f}}^{(i)}:=\big\{\hat{f}^{(i)}_{j}\big\}_{j=1}^{i}$ and $\hat{\mathbf{g}}^{(i)}:=\big\{\hat{g}^{(i)}_{j}\big\}_{j=1}^{i}$ be bases for $\mathbb{K}^{i}$. Building on Definition~\eqref{Def:flags and adapted bases}--\eqref{Def: adapted bases I}--\eqref{Def: adapted bases II}, we assume that: 
\begin{itemize}
\justifying
\item $\big\{\hat{\mathbf{f}}^{(i)}\big\}_{i=1}^{n}$ is a system of bases adapted to both $\big\{\psi_{\sh{F}}^{(i)}\big\}_{i=1}^{n-1}$ and $\big\{\phi_{\sh{F}}^{(i)}\,\big\}_{i=1}^{n-1}$.
\item $\big\{\hat{\mathbf{g}}^{(i)}\big\}_{i=1}^{n}$ is a system of bases adapted to both $\big\{\psi_{\sh{G}}^{(i)}\big\}_{i=1}^{n-1}$ and $\big\{\phi_{\sh{G}}^{(i)}\,\big\}_{i=1}^{n-1}$.
\end{itemize}

\smallskip
\noindent
\textbf{Main Conclusion}: Let $\sh{H}$ be an extension of $\sh{F}$ by $\sh{G}$. Under the given \textbf{setup}, $\sh{H}$ is equivalent to the trivial extension $\mathbf{0}_{\mathrm{Ext}}$ of $\sh{F}$ by $\sh{G}$. In other words, we have that
\begin{equation*}
    \mathrm{Ext}^{1}(\sh{F},\sh{G})\cong 0 \, .
\end{equation*}
\end{proposition}
\begin{proof}
Let $\sh{H}'$ be an extension of $\sh{F}$ by $\sh{G}$ representing the same equivalence class as $\sh{H}$ in $\mathrm{Ext}^{1}(\sh{F},\sh{G})$, so that there exists a sheaf isomorphism $\lambda: \sh{H}\to \sh{H}'$ such that the diagram in Figure~\eqref{Equivalen extenions of F by F'} commutes in each square. Building on Observation~\eqref{Obs: elements of Ext1 at the local models}, we have that:     
\begin{itemize}
\item On $U_{\mathrm{T}}$, $\sh{H}$ and $\sh{H}'$ are determined by collections of $n-1$ characteristic maps
\begin{equation*}
\big\{ \psi^{(i)}_{\sh{H}}: \mathbb{K}^{i+1}\to \mathbb{K}^{i}\big\}_{i=1}^{n-1}\, , \quad \text{and} \quad \big\{ \psi^{(i)}_{\sh{H}'}: \mathbb{K}^{i+1}\to \mathbb{K}^{i}\big\}_{i=1}^{n-1}\, ,
\end{equation*}
while $\lambda$ is characterized by a collection of $n$ linear maps $\big\{ S^{(i)}_{\lambda}: \mathbb{K}^{i}\to \mathbb{K}^{i}\big\}_{i=1}^{n}$ such that
\begin{equation}\label{Eq: equivalence conditions 1 for the unlink}
 \psi^{(i)}_{\sh{H}'}= \psi^{(i)}_{\sh{H}} + S_{\lambda}^{(i)}\circ\psi^{(i)}_{\sh{F}}-\psi^{(i)}_{\sh{G}}\circ S^{(i+1)}_{\lambda}\, ,    
\end{equation}
for all $i\in[1,n-1]$. 

\item On $U_{\mathrm{B}}$, $\sh{H}$ and $\sh{H}'$ are determined by collections of $n-1$ characteristic maps
\begin{equation*}
\big\{ \phi^{(i)}_{\sh{H}}: \mathbb{K}^{i}\to \mathbb{K}^{i+1}\big\}_{i=1}^{n-1}\, , \quad \text{and} \quad \big\{ \phi^{(i)}_{\sh{H}'}: \mathbb{K}^{i}\to \mathbb{K}^{i+1}\big\}_{i=1}^{n-1}\, ,
\end{equation*}
while $\lambda$ is characterized by a collection of $n$ linear maps $\big\{ T^{(i)}_{\lambda}: \mathbb{K}^{i}\to \mathbb{K}^{i}\big\}_{i=1}^{n}$ such that
\begin{equation}\label{Eq: equivalence conditions 2 for the unlink}
 \phi^{(i)}_{\sh{H}'}= \phi^{(i)}_{\sh{H}} + T_{\lambda}^{(i+1)}\circ\phi^{(i)}_{\sh{F}}-\phi^{(i)}_{\sh{G}}\circ T^{(i)}_{\lambda}\, ,    
\end{equation}
for all $i\in[1,n-1]$.  
\end{itemize}
In particular, applying the sheaf axioms to the intersection $U_{\mathrm{T}}\cap U_{\mathrm{B}}$ yields that $S^{(n)}_{\lambda}=T^{(n)}_{\lambda}$.

By assumption, $\big\{\hat{\mathbf{f}}^{(i)}\big\}_{i=1}^{n}$ and $\big\{\hat{\mathbf{g}}^{(i)}\big\}_{i=1}^{n}$ are systems of bases adapted to the pairs $\big(\big\{\psi_{\sh{F}}^{(i)}\big\}_{i=1}^{n-1}, \, \big\{\phi_{\sh{F}}^{(i)}\big\}_{i=1}^{n-1}\big)$ and $\big(\big\{\psi_{\sh{G}}^{(i)}\big\}_{i=1}^{n-1}, \, \big\{\phi_{\sh{G}}^{(i)}\big\}_{i=1}^{n-1}\big)$, respectively. Consequently, with respect to these bases, the equivalence conditions~\eqref{Eq: equivalence conditions 1 for the unlink} and~\eqref{Eq: equivalence conditions 2 for the unlink} translate into the system of equations
\begin{equation*}
\begin{aligned}
\tensor[_{\hat{\mathbf{g}}^{(i)}}]{ \big[\, \psi^{(i)}_{\sh{H}'}\,\big] }{_{\hat{\mathbf{f}}^{(i+1)}}}&= \tensor[_{\hat{\mathbf{g}}^{(i)}}]{ \big[\, \psi^{(i)}_{\sh{H}}\,\big] }{_{\hat{\mathbf{f}}^{(i+1)}}}+ \tensor[_{\hat{\mathbf{g}}^{(i)}}]{ \big[\, S^{(i)}_{\lambda}\,\big] }{_{\hat{\mathbf{f}}^{(i)}}}\cdot \pi^{(i,i+1)}-\pi^{(i,i+1)}\cdot \tensor[_{\hat{\mathbf{g}}^{(i+1)}}]{ \big[\, S^{(i+1)}_{\lambda}\,\big] }{_{\hat{\mathbf{f}}^{(i+1)}}}\, ,   \\[6pt]       
\tensor[_{\hat{\mathbf{g}}^{(i+1)}}]{ \big[\, \phi^{(i)}_{\sh{H}'}\,\big] }{_{\hat{\mathbf{f}}^{(i)}}}&= \tensor[_{\hat{\mathbf{g}}^{(i+1)}}]{ \big[\, \phi^{(i)}_{\sh{H}}\,\big] }{_{\hat{\mathbf{f}}^{(i)}}}+ \tensor[_{\hat{\mathbf{g}}^{(i+1)}}]{ \big[\, T^{(i+1)}_{\lambda}\,\big] }{_{\hat{\mathbf{f}}^{(i+1)}}}\cdot \iota^{(i+1,i)}-\iota^{(i+1,i)}\cdot \tensor[_{\hat{\mathbf{g}}^{(i)}}]{ \big[\, T^{(i)}_{\lambda}\,\big] }{_{\hat{\mathbf{f}}^{(i)}}}\, ,     
\end{aligned}
\end{equation*}
for all $i\in[1,n-1]$. 

Now, since $\tensor[_{\hat{\mathbf{g}}^{(n)}}]{ \big[\, T^{(n)}_{\lambda}\,\big] }{_{\hat{\mathbf{f}}^{(n)}}}= \tensor[_{\hat{\mathbf{g}}^{(n)}}]{ \big[\, S^{(n)}_{\lambda}\,\big] }{_{\hat{\mathbf{f}}^{(n)}}}$, Lemmas~\eqref{Lemma: system of equations for equivalent collections of surjective maps} and~\eqref{Lemma: system of equations for equivalent collections of injective maps} guarantee that for any choice of matrices $\big\{ \tensor[_{\hat{\mathbf{g}}^{(i)}}]{ \big[\, \psi^{(i)}_{\sh{H}'}\,\big] }{_{\hat{\mathbf{f}}^{(i+1)}}}\big\}_{i=1}^{n-1}$ and $\big\{ \tensor[_{\hat{\mathbf{g}}^{(i+1)}}]{ \big[\, \phi^{(i)}_{\sh{H}}\,\big] }{_{\hat{\mathbf{f}}^{(i)}}} \big\}_{i=1}^{n-1}$ there exist collections of matrices $\big\{ \tensor[_{\hat{\mathbf{g}}^{(i)}}]{ \big[\, S^{(i)}_{\lambda}\,\big] }{_{\hat{\mathbf{f}}^{(i)}}} \big\}_{i=1}^{n}$ and $\big\{  \tensor[_{\hat{\mathbf{g}}^{(i)}}]{ \big[\, T^{(i)}_{\lambda}\,\big] }{_{\hat{\mathbf{f}}^{(i)}}} \big\}_{i=1}^{n}$ solving the system of equations, with the diagonal part of $\tensor[_{\hat{\mathbf{g}}^{(n)}}]{ \big[\, T^{(n)}_{\lambda}\,\big] }{_{\hat{\mathbf{f}}^{(n)}}}$ chosen arbitrarily.

Bearing this in mind, we deduce that $\sh{H}$ is equivalent to any extension of $\sh{F}$ by $\sh{G}$, and in particular, it is equivalent to the trivial extension $\mathbf{0}_{\mathrm{Ext}}$. Hence, since $\sh{H}$ is arbitrary, we conclude that $\mathrm{Ext}^{1}(\sh{F},\sh{G})\cong 0$. 
\end{proof}

With the above result at hand, we conclude our study of the one-degree morphism spaces in the category $\ccs{1}{e_{n}}$, and now turn to their analysis in the category $\ccs{1}{\beta}$ for an arbitrary positive braid word $\beta\in\mathrm{Br}^{+}_{n}$.

\subsubsection{General Positive Braids} Let $\beta:=\sigma_{i_{1}}\cdots \sigma_{i_{\ell}}\in\mathrm{Br}^{+}_{n}$ be a positive braid word. Next, we analyze the one-degree morphism spaces in the category $\ccs{1}{\beta}$.   

Let $\mathcal{U}_{\Lambda(\beta)}=\big\{U_{0}, U_{\mathrm{B}}, U_{\mathrm{L}}, U_{\mathrm{R}}, U_{\mathrm{T}}\big\}$ be the open cover of $\mathbb{R}^{2}$ from Construction~\eqref{Cons: Finite open cover for R^2}, $\mathcal{R}_{\Lambda(\beta)}=\big\{R_{j}\big\}_{j=1}^{\ell}$ the partition of $U_{\mathrm{B}}$ into $\ell$ open vertical straps from Construction~\eqref{Cons: Definition of the vertical straps}, $\sh{F}$ and $\sh{G}$ objects of the category $\ccs{1}{\beta}$, and $\sh{H}$ an extension of $\sh{F}$ by $\sh{G}$. Building on Observations~\eqref{Obs: objects on U_T and vertical straps} and~\eqref{Obs: elements of Ext1 at the local models}, we have that the global structure of $\sh{H}$ can be fully recovered from its local description on the region $U_{\mathrm{T}}$ and the collection of vertical straps $\mathcal{R}_{\Lambda(\beta)}$, together with the appropriate compatibility conditions. Accordingly, we present the following result.

\begin{lemma}\label{Lemma: Ext1 on U_T and U_L}
\textbf{Setup}: Let $\beta=\sigma_{i_{1}}\cdots \sigma_{i_{\ell}} \in\mathrm{Br}^{+}_{n}$ be a positive braid word, $\mathcal{U}_{\Lambda(\beta)}=\big\{U_{0}, U_{\mathrm{B}}, U_{\mathrm{L}}, U_{\mathrm{R}}, U_{\mathrm{T}}\big\}$ the open cover of $\mathbb{R}^{2}$ from Construction~\eqref{Cons: Finite open cover for R^2}, and $\sh{F}$ and $\sh{G}$ objects of the category $\ccs{1}{\beta}$. According to Lemma~\eqref{Lemma: linear map description of an object on the regions U_T, U_L, and U_R}, $\sh{F}$ and $\sh{G}$ have the following local descriptions: 
\begin{itemize}
\justifying
\item  On $U_{\mathrm{T}}$, $\sh{F}$ and $\sh{G}$ are specified by collections of $n-1$ surjective linear maps 
\begin{equation*}
\big\{\psi_{\sh{F}}^{(i)}:\mathbb{K}^{i+1}\to \mathbb{K}^{i}\big\}_{i=1}^{n-1}\, , \quad \text{and} \quad \big\{\psi_{\sh{G}}^{(i)}:\mathbb{K}^{i+1}\to \mathbb{K}^{i}\big\}_{i=1}^{n-1}\, ,    
\end{equation*}
respectively. For a schematic illustration of a generic representative of one of these sheaves on $U_{\mathrm{T}}$, see Figure~\eqref{Fig: an object F in the region U_T}.

\item On $U_{\mathrm{L}}$, $\sh{F}$ and $\sh{G}$ are specified by collections of $n-1$ injective linear maps 
\begin{equation*}
\big\{\phi_{\sh{F}}^{(i)}:\mathbb{K}^{i}\to \mathbb{K}^{i+1}\big\}_{i=1}^{n-1}\, , \quad \text{and} \quad \big\{\phi_{\sh{G}}^{(i)}:\mathbb{K}^{i}\to \mathbb{K}^{i+1}\big\}_{i=1}^{n-1} \, ,    
\end{equation*}
respectively. For a schematic illustration of a generic representative of one of these sheaves on $U_{\mathrm{L}}$, see Figure~\eqref{Fig: an object F in the region U_L}.

\item \textbf{Compatibility conditions}: For each $i\in[1,n-1]$, 
\begin{equation*}
\psi^{(i)}_{\sh{F}}\circ \phi^{(i)}_{\sh{F}}=\mathrm{id}_{\mathbb{K}^{i}}\, , \quad \text{and} \quad \psi^{(i)}_{\sh{G}}\circ \phi^{(i)}_{\sh{G}}=\mathrm{id}_{\mathbb{K}^{i}}\, .    
\end{equation*}
\end{itemize}

\noindent
$\star$ \emph{Assumption}: For each $i\in[1,n]$, let $\hat{\mathbf{f}}^{(i)}:=\big\{\hat{f}^{(i)}_{j}\big\}_{j=1}^{i}$ and $\hat{\mathbf{g}}^{(i)}:=\big\{\hat{g}^{(i)}_{j}\big\}_{j=1}^{i}$ be bases for $\mathbb{K}^{i}$. Building on Definition~\eqref{Def:flags and adapted bases}--\eqref{Def: adapted bases I}--\eqref{Def: adapted bases II}, we assume that: 
\begin{itemize}
\justifying
\item $\big\{\hat{\mathbf{f}}^{(i)}\big\}_{i=1}^{n}$ is a system of bases adapted to both $\big\{\psi_{\sh{F}}^{(i)}\big\}_{i=1}^{n-1}$ and $\big\{\phi_{\sh{F}}^{(i)}\,\big\}_{i=1}^{n-1}$.
\item  $\big\{\hat{\mathbf{g}}^{(i)}\big\}_{i=1}^{n}$ is a system of bases adapted to both $\big\{\psi_{\sh{G}}^{(i)}\big\}_{i=1}^{n-1}$ and $\big\{\phi_{\sh{G}}^{(i)}\,\big\}_{i=1}^{n-1}$.
\end{itemize}

\smallskip
\noindent
\textbf{Main Conclusion}: Under the given \textbf{setup}, the following statements hold: 

\noindent
$\star$ \emph{Local characterization of $\mathrm{Ext}^{1}(\sh{F},\sh{G})$ on $U_{\mathrm{T}}\;\cup\; U_{\mathrm{L}}$}: Any extension $\sh{H}$ of $\sh{F}$ by $\sh{G}$ is locally equivalent to the trivial extension $\mathbf{0}_{\mathrm{Ext}}$ on $U_{\mathrm{T}}\;\cup\;U_{\mathrm{L}}$. 

\noindent
$\star$ \emph{Block-diagonal properties}: Let $\sh{H}$ and $\sh{H}'$ be two extensions of $\sh{F}$ by $\sh{G}$ representing the same equivalence class in $\mathrm{Ext}^{1}(\sh{F},\sh{G})$, and suppose that both $\sh{H}$ and $\sh{H}'$ coincide with the trivial extension $\mathbf{0}_{\mathrm{Ext}}$ on $U_{\mathrm{T}}\;\cup\; U_{\mathrm{L}}$. Let $\lambda:\sh{H}\to \sh{H}'$ be a sheaf isomorphism realizing the equivalence between $\sh{H}$ and $\sh{H}'$. Then, $\lambda$ is characterized by a collection of $n$ linear maps $\big\{ T^{(i)}_{\lambda}:\mathbb{K}^{i}\to \mathbb{K}^{i} \big\}_{i=1}^{n}$, which, with respect to the bases $\big\{\hat{\mathbf{f}}^{(i)}\big\}_{i=1}^{n}$ and $\big\{\hat{\mathbf{g}}^{(i)}\big\}_{i=1}^{n}$, satisfy the following properties:  
\begin{itemize}
\justifying
\item $\tensor[_{\hat{\mathbf{g}}^{(n)}}]{ \Big[\, T_{\lambda}^{\,(n)}\,\Big] }{_{\hat{\mathbf{f}}^{(n)}}}\in M(n,\mathbb{K})$ is diagonal. 
\item For each $i\in[1,n-1]$, $\tensor[_{\hat{\mathbf{g}}^{(i)}}]{ \big[\, T_{\lambda}^{(i)}\,\big] }{_{\hat{\mathbf{f}}^{(i)}}}\in M(i,\mathbb{K})$ is given by the principal $i\times i$ submatrix of $\tensor[_{\hat{\mathbf{g}}^{(n)}}]{ \big[\, T_{\lambda}^{\,(n)}\,\big] }{_{\hat{\mathbf{f}}^{(n)}}}$.  
\end{itemize}
\end{lemma}
\begin{proof}
Let $\sh{H}'$ be an extension of $\sh{F}$ by $\sh{G}$ representing the same equivalence class as $\sh{H}$ in $\mathrm{Ext}^{1}(\sh{F},\sh{G})$, and let $\lambda: \sh{H}\to \sh{H}'$ be a sheaf isomorphism such that the diagram in Figure~\eqref{Equivalen extenions of F by F'} commutes in each square. Building on Observation~\eqref{Obs: elements of Ext1 at the local models}, we have that:     
\begin{itemize}
\item On $U_{\mathrm{T}}$, $\sh{H}$ and $\sh{H}'$ are determined by collections of $n-1$ characteristic maps
\begin{equation*}
\big\{ \psi^{(i)}_{\sh{H}}: \mathbb{K}^{i+1}\to \mathbb{K}^{i}\big\}_{i=1}^{n-1}\, , \quad \text{and} \quad \big\{ \psi^{(i)}_{\sh{H}'}: \mathbb{K}^{i+1}\to \mathbb{K}^{i}\big\}_{i=1}^{n-1}\, ,
\end{equation*}
while $\lambda$ is characterized by a collection of $n$ linear maps $\big\{ S^{(i)}_{\lambda}: \mathbb{K}^{i}\to \mathbb{K}^{i}\big\}_{i=1}^{n}$ such that
\begin{equation}\label{Eq: equivalence conditions on U_T}
 \psi^{(i)}_{\sh{H}'}= \psi^{(i)}_{\sh{H}} + S_{\lambda}^{(i)}\circ\psi^{(i)}_{\sh{F}}-\psi^{(i)}_{\sh{G}}\circ S^{(i+1)}_{\lambda}\, ,    
\end{equation}
for all $i\in[1,n-1]$. 

\item On $U_{\mathrm{L}}$, $\sh{H}$ and $\sh{H}'$ are determined by collections of $n-1$ characteristic maps
\begin{equation*}
\big\{ \phi^{(i)}_{\sh{H}}: \mathbb{K}^{i}\to \mathbb{K}^{i+1}\big\}_{i=1}^{n-1}\, , \quad \text{and} \quad \big\{ \phi^{(i)}_{\sh{H}'}: \mathbb{K}^{i}\to \mathbb{K}^{i+1}\big\}_{i=1}^{n-1}\, ,
\end{equation*}
while $\lambda$ is characterized by a collection of $n$ linear maps $\big\{ T^{(i)}_{\lambda}: \mathbb{K}^{i}\to \mathbb{K}^{i}\big\}_{i=1}^{n}$ such that
\begin{equation}\label{Eq: equivalence conditions on U_L}
 \phi^{(i)}_{\sh{H}'}= \phi^{(i)}_{\sh{H}} + T_{\lambda}^{(i+1)}\circ\phi^{(i)}_{\sh{F}}-\phi^{(i)}_{\sh{G}}\circ T^{(i)}_{\lambda}\, ,    
\end{equation}
for all $i\in[1,n-1]$.  
\end{itemize}
In particular, applying the sheaf axioms to the intersection $U_{\mathrm{T}}\cap U_{\mathrm{L}}$ yields that $S^{(n)}_{\lambda}=T^{(n)}_{\lambda}$.

By assumption, $\big\{\hat{\mathbf{f}}^{(i)}\big\}_{i=1}^{n}$ and $\big\{\hat{\mathbf{g}}^{(i)}\big\}_{i=1}^{n}$ are systems of bases adapted to the pairs $\big(\big\{\psi_{\sh{F}}^{(i)}\big\}_{i=1}^{n-1}, \, \big\{\phi_{\sh{F}}^{(i)}\big\}_{i=1}^{n-1}\big)$ and $\big(\big\{\psi_{\sh{G}}^{(i)}\big\}_{i=1}^{n-1}, \, \big\{\phi_{\sh{G}}^{(i)}\big\}_{i=1}^{n-1}\big)$, respectively. Consequently, with respect to these bases, the equivalence conditions~\eqref{Eq: equivalence conditions on U_T} and~\eqref{Eq: equivalence conditions on U_L} translate into the system of equations
\begin{equation}\label{Eq: system of equivalence equations for U_T and U_L}
\begin{aligned}
\tensor[_{\hat{\mathbf{g}}^{(i)}}]{ \big[\, \psi^{(i)}_{\sh{H}'}\,\big] }{_{\hat{\mathbf{f}}^{(i+1)}}}&= \tensor[_{\hat{\mathbf{g}}^{(i)}}]{ \big[\, \psi^{(i)}_{\sh{H}}\,\big] }{_{\hat{\mathbf{f}}^{(i+1)}}}+ \tensor[_{\hat{\mathbf{g}}^{(i)}}]{ \big[\, S^{(i)}_{\lambda}\,\big] }{_{\hat{\mathbf{f}}^{(i)}}}\cdot \pi^{(i,i+1)}-\pi^{(i,i+1)}\cdot \tensor[_{\hat{\mathbf{g}}^{(i+1)}}]{ \big[\, S^{(i+1)}_{\lambda}\,\big] }{_{\hat{\mathbf{f}}^{(i+1)}}}\, ,   \\[6pt]       
\tensor[_{\hat{\mathbf{g}}^{(i+1)}}]{ \big[\, \phi^{(i)}_{\sh{H}'}\,\big] }{_{\hat{\mathbf{f}}^{(i)}}}&= \tensor[_{\hat{\mathbf{g}}^{(i+1)}}]{ \big[\, \phi^{(i)}_{\sh{H}}\,\big] }{_{\hat{\mathbf{f}}^{(i)}}}+ \tensor[_{\hat{\mathbf{g}}^{(i+1)}}]{ \big[\, T^{(i+1)}_{\lambda}\,\big] }{_{\hat{\mathbf{f}}^{(i+1)}}}\cdot \iota^{(i+1,i)}-\iota^{(i+1,i)}\cdot \tensor[_{\hat{\mathbf{g}}^{(i)}}]{ \big[\, T^{(i)}_{\lambda}\,\big] }{_{\hat{\mathbf{f}}^{(i)}}}\, ,     
\end{aligned}
\end{equation}
for all $i\in[1,n-1]$. 

Now, since $\tensor[_{\hat{\mathbf{g}}^{(n)}}]{ \big[\, T^{(n)}_{\lambda}\,\big] }{_{\hat{\mathbf{f}}^{(n)}}}= \tensor[_{\hat{\mathbf{g}}^{(n)}}]{ \big[\, S^{(n)}_{\lambda}\,\big] }{_{\hat{\mathbf{f}}^{(n)}}}$, Lemmas~\eqref{Lemma: system of equations for equivalent collections of surjective maps} and~\eqref{Lemma: system of equations for equivalent collections of injective maps} guarantee that for any choice of matrices $\big\{ \tensor[_{\hat{\mathbf{g}}^{(i)}}]{ \big[\, \psi^{(i)}_{\sh{H}'}\,\big] }{_{\hat{\mathbf{f}}^{(i+1)}}}\big\}_{i=1}^{n-1}$ and $\big\{ \tensor[_{\hat{\mathbf{g}}^{(i+1)}}]{ \big[\, \phi^{(i)}_{\sh{H}}\,\big] }{_{\hat{\mathbf{f}}^{(i)}}} \big\}_{i=1}^{n-1}$ there exist collections of matrices $\big\{ \tensor[_{\hat{\mathbf{g}}^{(i)}}]{ \big[\, S^{(i)}_{\lambda}\,\big] }{_{\hat{\mathbf{f}}^{(i)}}} \big\}_{i=1}^{n}$ and $\big\{  \tensor[_{\hat{\mathbf{g}}^{(i)}}]{ \big[\, T^{(i)}_{\lambda}\,\big] }{_{\hat{\mathbf{f}}^{(i)}}} \big\}_{i=1}^{n}$ solving the system of equations in~\eqref{Eq: system of equivalence equations for U_T and U_L}, with the diagonal part of $\tensor[_{\hat{\mathbf{g}}^{(n)}}]{ \big[\, T^{(n)}_{\lambda}\,\big] }{_{\hat{\mathbf{f}}^{(n)}}}$ chosen arbitrarily.

Bearing this in mind, we deduce that $\sh{H}$ is locally equivalent to any extension of $\sh{F}$ by $\sh{G}$ on $U_\mathrm{T}\;\cup\; U_{\mathrm{L}}$, and in particular, it is equivalent to the trivial extension $\mathbf{0}_{\mathrm{Ext}}$ on this region, that is, the extension whose characteristic maps are zero. 

Next, suppose that $\sh{H}$ and $\sh{H}'$ are two equivalent extensions of $\sh{F}$ by $\sh{G}$ which coincide with the trivial extension of $U_\mathrm{T}\;\cup\; U_{\mathrm{L}}$. More precisely, assume that $\psi^{(i)}_{\sh{H}}=\psi^{(i)}_{\sh{H}'}=0$ and $\phi^{(i)}_{\sh{H}}=\phi^{(i)}_{\sh{H}'}=0$ for all $i\in [1,n-1]$. Then, the system of equations in~\eqref{Eq: system of equivalence equations for U_T and U_L} reduces to 
\begin{equation}\label{Eq: reduced system of equivalence equations for U_T and U_L}
\begin{aligned}
\mathbf{0}_{i\times(i+1)}&= \mathbf{0}_{i\times(i+1)}+ \tensor[_{\hat{\mathbf{g}}^{(i)}}]{ \big[\, S^{(i)}_{\lambda}\,\big] }{_{\hat{\mathbf{f}}^{(i)}}}\cdot \pi^{(i,i+1)}-\pi^{(i,i+1)}\cdot \tensor[_{\hat{\mathbf{g}}^{(i+1)}}]{ \big[\, S^{(i+1)}_{\lambda}\,\big] }{_{\hat{\mathbf{f}}^{(i+1)}}}\, ,   \\[6pt]       
\mathbf{0}_{(i+1)\times i}&= \mathbf{0}_{(i+1)\times i}+ \tensor[_{\hat{\mathbf{g}}^{(i+1)}}]{ \big[\, T^{(i+1)}_{\lambda}\,\big] }{_{\hat{\mathbf{f}}^{(i+1)}}}\cdot \iota^{(i+1,i)}-\iota^{(i+1,i)}\cdot \tensor[_{\hat{\mathbf{g}}^{(i)}}]{ \big[\, T^{(i)}_{\lambda}\,\big] }{_{\hat{\mathbf{f}}^{(i)}}}\, .     
\end{aligned}
\end{equation}
Finally, recall that $\tensor[_{\hat{\mathbf{g}}^{(n)}}]{ \big[\, T^{(n)}_{\lambda}\,\big] }{_{\hat{\mathbf{f}}^{(n)}}}= \tensor[_{\hat{\mathbf{g}}^{(n)}}]{ \big[\, S^{(n)}_{\lambda}\,\big] }{_{\hat{\mathbf{f}}^{(n)}}}$, and hence Lemmas~\eqref{Lemma: system of equations for equivalent collections of surjective maps} and~\eqref{Lemma: system of equations for equivalent collections of injective maps} ensure that any solution of the reduced system of equations in~\eqref{Eq: reduced system of equivalence equations for U_T and U_L} arises from a diagonal matrix $\tensor[_{\hat{\mathbf{g}}^{(n)}}]{ \Big[\, T_{\lambda}^{\,(n)}\,\Big] }{_{\hat{\mathbf{f}}^{(n)}}}\in M(n,\mathbb{K})$, with $\tensor[_{\hat{\mathbf{g}}^{(i)}}]{ \big[\, S_{\lambda}^{(i)}\,\big] }{_{\hat{\mathbf{f}}^{(i)}}}\in M(i,\mathbb{K})$ and  $\tensor[_{\hat{\mathbf{g}}^{(i)}}]{ \big[\, T_{\lambda}^{(i)}\,\big] }{_{\hat{\mathbf{f}}^{(i)}}}\in M(i,\mathbb{K})$ given by the principal $i\times i$ submatrix of $\tensor[_{\hat{\mathbf{g}}^{(n)}}]{ \big[\, T_{\lambda}^{\,(n)}\,\big] }{_{\hat{\mathbf{f}}^{(n)}}}$ for each $i\in[1,n-1]$. It follows that $S_{\lambda}^{(i)}=T^{(i)}_{\lambda}$ for each $i\in[1,n-1]$, and hence $\lambda$ is characterized by a collection of linear maps $\big\{T^{(i)}_{\lambda}\big\}_{i=1}^{n}$ satisfying the conditions stated in the lemma. 
\end{proof}

The next lemmas will help us analyze the structure of the one-degree morphism spaces in the category of our interest on a region $R_{j}$. In order to state the lemmas, let us first introduce a setup. 

\begin{setup}\label{Setup: Objects on R_j for k=1 and adapted bases}
Let $\beta=\sigma_{i_{1}}\cdots \sigma_{i_{\ell}}\in\mathrm{Br}^{+}_{n}$ be a positive braid word, and let $\sh{F}$ and $\sh{G}$ be objects of the category $\ccs{1}{\beta}$. Fix $j\in[1,\ell]$, and let $R_{j}$ be the vertical strap in $\mathbb{R}^{2}$ containing $\sigma_{i_{j}}$---the $j$-th crossing of $\beta$ (see Figure~\eqref{fig: A sub-regions R_j}). 

\noindent
$\star$ \emph{Assumption 1}: Let $k:=i_{j}\in[1,n-1]$ denote the index of $\sigma_{i_{j}}$, and suppose that $k=1$.

Under this assumption, on $R_{j}$, $\sh{F}$ and $\sh{G}$ are specified by two collections of $n$ injective linear maps 
\begin{equation*}
\begin{aligned}
\big\{\phi^{\,(i)}_{\sh{F}}:\mathbb{K}^{i}\to \mathbb{K}^{i+1}\,\big\}_{i=1}^{n-1}&\,\cup\,\big\{ \widetilde{\phi}^{\,(1)}_{\sh{F}}:\mathbb{K}^{1}\to \mathbb{K}^{2}\, \big\}\, ,\\[6pt]
\big\{\phi^{\,(i)}_{\sh{G}}:\mathbb{K}^{i}\to \mathbb{K}^{i+1}\,\big\}_{i=1}^{n-1}&\,\cup\,\big\{ \widetilde{\phi}^{\,(1)}_{\sh{G}}:\mathbb{K}^{1}\to \mathbb{K}^{2}\, \big\}\, ,
\end{aligned}    
\end{equation*}
respectively. For a schematic illustration of a generic representative of one of these sheaves, see Figure~\eqref{fig: a sheaf in the vertical strap R_j containing a crossing sigma_1}.

\noindent
$\star$ \emph{Assumption 2}: For each $i\in[1,n]$, let $\hat{\mathbf{f}}^{(i)}:=\big\{\hat{f}^{(i)}_{j}\big\}_{j=1}^{i}$ and $\hat{\mathbf{g}}^{(i)}:=\big\{\hat{g}^{(i)}_{j}\big\}_{j=1}^{i}$ be bases for $\mathbb{K}^{i}$. Following Definition~\eqref{Def:flags and adapted bases}--\eqref{Def: adapted bases I}, we assume that: 
\begin{itemize}
\justifying
\item $\big\{\hat{\mathbf{f}}^{(i)}\big\}_{i=1}^{n}$ is a system of bases adapted to $\big\{\phi_{\sh{F}}^{(i)}\,\big\}_{i=1}^{n-1}$.
\item $\big\{\hat{\mathbf{g}}^{(i)}\big\}_{i=1}^{n}$ is a system of bases adapted to $\big\{\phi_{\sh{G}}^{(i)}\,\big\}_{i=1}^{n-1}$.
\end{itemize}

Under this assumption, Lemma~\eqref{Lemma: matrix local model for sigma_1} ensures that $\big(\big\{\hat{\mathbf{f}}^{(i)}\big\}_{i=1}^{n}, z \big)$ and $\big(\big\{\hat{\mathbf{g}}^{(i)}\big\}_{i=1}^{n}, z' \big)$ are system of bases adapted to $\sh{F}$ and $\sh{G}$ on $R_{j}$, where $z,\,z'\in \mathbb{K}$ parameterize, relative to the bases $\hat{\mathbf{f}}^{(n)}$ and $\hat{\mathbf{g}}^{(n)}$ for $\mathbb{K}^{n}$, the $s_{1}$-relative position between the pair of complete flags in $\mathbb{K}^{n}$ that geometrically characterize $\sh{F}$ and $\sh{G}$ on $R_{j}$, respectively. Following Definition~\eqref{Def: braid transformation of bases}, we denote by $\big\{\hat{\mathbf{f}}^{(i)}[\sigma_{1},z]\big\}_{i=1}^{n}$ and $\big\{\hat{\mathbf{g}}^{(i)}[\sigma_{1},z']\big\}_{i=1}^{n}$ the corresponding braid-transformed bases.    
\end{setup}

\begin{lemma}\label{Lemma: Ext1 on R_j for k=1}
Consider the assumptions of Setup~\eqref{Setup: Objects on R_j for k=1 and adapted bases} and fix an equivalence class $\xi\in \mathrm{Ext}^{1}(\sh{F},\sh{G})$. Let $\sh{H}$ and $\sh{H}'$ be two equivalent extensions of $\sh{F}$ by $\sh{G}$ representing $\xi$, and let $\lambda:\sh{H}\to \sh{H}'$ be a sheaf isomorphism realizing their equivalence. 

On $R_{j}$, $\sh{H}$ and $\sh{H}'$ are specified by two collections of $n$ characteristic maps
\begin{equation}\label{Eq: characteristic maps of two extensions on R_j with k=1}
\begin{aligned}
\big\{\phi^{(i)}_{\sh{H}}:\mathbb{K}^{i}\to \mathbb{K}^{i+1}\,\big\}_{i=1}^{n-1}\,&\cup\,\big\{ \widetilde{\phi}^{\,(1)}_{\sh{H}}:\mathbb{K}^{1}\to \mathbb{K}^{2}\, \big\}\, , \\[2pt]
\big\{\phi^{(i)}_{\sh{H}'}:\mathbb{K}^{i}\to \mathbb{K}^{i+1}\,\big\}_{i=1}^{n-1}\,&\cup\,\big\{ \widetilde{\phi}^{\,(1)}_{\sh{H}'}:\mathbb{K}^{1}\to \mathbb{K}^{2}\, \big\}\, , 
\end{aligned}
\end{equation}
while $\lambda$ is characterized by a collection of $n+1$ linear maps
\begin{equation}\label{Eq: maps characterizing lambda on R_j with k=1}
\big\{ T^{(i)}_{\lambda}: \mathbb{K}^{i}\to \mathbb{K}^{i} \big\}_{i=1}^{n}\;\cup\; \big\{ \widetilde{T}^{(1)}:\mathbb{K}^{1}\to \mathbb{K}^{1} \big\}    
\end{equation}
such that:
\begin{subequations}\label{Eq: equivalence relations on R_j with k=1}
\begin{align}
\phi^{(i)}_{\sh{H}'} &= \phi^{(i)}_{\sh{H}} + T^{(i+1)}_{\lambda}\circ \phi^{(i)}_{\sh{F}}- \phi^{(i)}_{\sh{G}}\circ T^{(i)}_{\lambda} \, , \qquad \text{for all $i\in [1,n-1]$,} \label{Eq: equivalence relations for R_j with k=1 a} \\[2pt] 
\widetilde{\phi}^{\,(1)}_{\sh{H}'} &= \widetilde{\phi}^{\,(1)}_{\sh{H}} + T^{(2)}_{\lambda}\circ \widetilde{\phi}^{\,(1)}_{\sh{F}}- \widetilde{\phi}^{\,(1)}_{\sh{G}}\circ \widetilde{T}^{\,(1)}_{\lambda} \, . \label{Eq: equivalence relations on R_j with k=1 b}
\end{align}
\end{subequations}

\noindent
$\star$ \emph{Assumption 1}: Suppose that there are linear maps $\big\{\gamma^{(i)}: \mathbb{K}^{i}\to \mathbb{K}^{i+1} \big\}_{i=1}^{n-1}$ that fully characterize $\xi$, up to equivalence, on the left of the crossing in $R_{j}$. Building on this assumption, we choose representatives $\sh{H}$ and $\sh{H}'$ such that:
\begin{equation*}
\phi^{(i)}_{\sh{H}}=\phi^{(i)}_{\sh{H}'}=\gamma^{(i)}\, , \qquad \text{for all $i\in [1,n-1]$}\, .    
\end{equation*}

\noindent
$\star$ \emph{Assumption 2}: Suppose that, with respect to the bases $\big\{\hat{\mathbf{f}}^{(i)}\big\}_{i=1}^{n}$ and $\big\{\hat{\mathbf{g}}^{(i)}\big\}_{i=1}^{n}$, the matrices representing the maps $\big\{ T^{(i)}_{\lambda} \big\}_{i=1}^{n}$, which characterize $\lambda$ on the left of the crossing in $R_{j}$, satisfy: 
\begin{equation*}
\begin{aligned}
 \tensor[_{\hat{\mathbf{g}}^{(n)}}]{ \big[\, T_{\lambda}^{(n)}\,\big] }{_{\hat{\mathbf{f}}^{(n)}}} &\equiv D(\vec{u}),~\text{ for some $\vec{u}=(u_{1},\dots, u_{n})\in \mathbb{K}^{n}_{\mathrm{srd}}$}\, ,\\[4pt]
\tensor[_{\hat{\mathbf{g}}^{(i)}}]{ \big[\, T_{\lambda}^{(i)}\,\big] }{_{\hat{\mathbf{f}}^{(i)}}}&\equiv \,\text{principal $i\times i$ submatrix of } D(\vec{u})\, , \quad\text{for all  $i\in[1,n-1]$}\, .
\end{aligned}    
\end{equation*}

In particular, under the above assumptions, the equivalence conditions in~\eqref{Eq: equivalence relations for R_j with k=1 a} are automatically satisfied.

\noindent
$\star$ \emph{Main Conclusion}: $\sh{H}$ is equivalent to a representative $\sh{H}'$ such that, with respect to the bases $\hat{\mathbf{g}}^{(2)}[\sigma_{1}, z']$ and $\hat{\mathbf{f}}^{(1)}[\sigma_{1}, z]$, the matrix representing $\widetilde{\phi}^{\,(1)}_{\sh{H}'}$ is given by 
\begin{equation*}
\tensor[_{\hat{\mathbf{g}}^{(2)}[\sigma_{1}, z']}]{ \big[\, \widetilde{\phi}^{\,(1)}_{\sh{H}'} \,\big] }{_{\hat{\mathbf{f}}^{(1)}[\sigma_{1}, z]}}=\begin{bmatrix}
0\\
a'
\end{bmatrix}\, ,
\end{equation*}
for some $a'\in \mathbb{K}$. In other words, the parameter $a'$, together with the data encoded in $\big\{\gamma^{(i)}\big\}_{i=1}^{n-1}$, completely determines $\xi$ on $R_{j}$, up to equivalence. 
\end{lemma}
\begin{proof}
To begin, recall that $\sh{H}$ and $\sh{H}'$ are two extensions of $\sh{F}$ by $\sh{G}$, and hence their local description on $R_{j}$ via the characteristic maps in~\eqref{Eq: characteristic maps of two extensions on R_j with k=1} follows from Observation~\eqref{Obs: elements of Ext1 at the local models}. Moreover, building on \emph{Assumption 1}, we suppose that $\sh{H}$ and $\sh{H}'$ are representatives with $\phi^{(i)}_{\sh{H}}=\phi^{(i)}_{\sh{H}'}=\gamma^{(i)}$, for all $i\in [1,n-1]$. 

Analogously to the case of $\sh{H}$ and $\sh{H}'$, it follows from Observation~\eqref{Obs: elements of Ext1 at the local models} that $\lambda:\sh{H}\to \sh{H}'$, the sheaf isomorphism realizing the equivalence between $\sh{H}$ and $\sh{H}'$, is locally characterized on $R_{j}$ by a collection of linear maps as in~\eqref{Eq: maps characterizing lambda on R_j with k=1}, which are subject to the equivalence conditions~\eqref{Eq: equivalence relations on R_j with k=1}.

Now, given the bases $\big\{\hat{\mathbf{f}}^{(i)}\big\}_{i=1}^{n}$ and $\big\{\hat{\mathbf{g}}^{(i)}\big\}_{i=1}^{n}$, let us define
\begin{equation*}
\begin{aligned}
Y^{(n)}&:=\tensor[_{\hat{\mathbf{g}}^{(n)}}]{ \big[\, T_{\lambda}^{(n)}\,\big] }{_{\hat{\mathbf{f}}^{(n)}}}\in \mathrm{M}(n,\mathbb{K})\, ,  \\[2pt]
Y^{(i)}&:=\tensor[_{\hat{\mathbf{g}}^{(i)}}]{ \big[\, T_{\lambda}^{(i)}\,\big] }{_{\hat{\mathbf{f}}^{(i)}}}\in \mathrm{M}(i,\mathbb{K})\, ,   \quad \text{for all $i\in [1,n-1]$}.
\end{aligned}
\end{equation*}
By \emph{Assumption 2}, we have that $Y^{(n)}=D(\vec{u})$, for some $\vec{u}=(u_{1},\dots, u_{n})\in \mathbb{K}^{n}_{\mathrm{std}}$, and that $Y^{(i)}$ is given by the principal $i\times i$ submatrix of $D(\vec{u})$, for all $i\in [1,n-1]$. 

Next, given the bases $\big\{\hat{\mathbf{f}}^{(i)}[\sigma_{1}, z]\big\}_{i=1}^{n}$ and $\big\{\hat{\mathbf{g}}^{(i)}[\sigma_{1},z']\big\}_{i=1}^{n}$, let us define
\begin{equation*}
\begin{aligned}
\widetilde{X}^{\,(1)}&:=\tensor[_{\hat{\mathbf{g}}^{(2)}[\sigma_{1}, z']}]{ \big[\, \widetilde{\phi}^{\,(1)}_{\sh{H}} \,\big] }{_{\hat{\mathbf{f}}^{(1)}[\sigma_{1}, z]}}=\begin{bmatrix}
\tilde{x}_{1,1}\\
\tilde{x}_{2,1}
\end{bmatrix}\in\mathrm{M}(2,1, \mathbb{K})\, , \\[4pt]
\widetilde{X}'^{\,(1)}&:=\tensor[_{\hat{\mathbf{g}}^{(2)}[\sigma_{1}, z']}]{ \big[\, \widetilde{\phi}^{\,(1)}_{\sh{H}'} \,\big] }{_{\hat{\mathbf{f}}^{(1)}[\sigma_{1}, z]}}= \begin{bmatrix}
\tilde{x}'_{1,1}\\
\tilde{x}'_{2,1}
\end{bmatrix}\in\mathrm{M}(2,1, \mathbb{K})\,, \\[4pt]
\widetilde{Y}^{(1)}&:=\tensor[_{\hat{\mathbf{g}}^{(1)}[\sigma_{1},z']}]{ \big[\, \widetilde{T}^{(1)}_{\lambda}\,\big] }{_{\hat{\mathbf{f}}^{(1)}[\sigma_{1},z]}}=[\tilde{y}_{1,1}]\in \mathrm{M}(1,\mathbb{K})\, .
\end{aligned}
\end{equation*}
Then, with respect to the braid transformed bases, relation~\eqref{Eq: equivalence relations on R_j with k=1 b} translates into the system of equations
\begin{equation}\label{Eq: system of equations for R_j with k=1}
\widetilde{X}'^{\,(1)}= \widetilde{X}^{(1)}+\Big( \big(B^{(2)}_{1}(z')\big)^{-1}\cdot Y^{(2)} \cdot B^{(2)}_{1}(z)\Big) \cdot \iota^{(2,1)} -\iota^{(2,1)}\cdot \widetilde{Y}^{(1)} \, .  
\end{equation}
It follows from Lemma~\eqref{Lemma: system of equations for equivalent extensions at a crossing sigma_1}(a) that, for any choice of $\tilde{x}'_{1,1}$, there exist unique $\tilde{y}_{1,1}$ and $\tilde{x}'_{2,1}$ solving~\eqref{Eq: system of equations for R_j with k=1}, and hence $\sh{H}$ is equivalent to a representative $\sh{H}'$ with $\tilde{x}'_{1,1}=0$ and $\tilde{x}'_{2,1}=a$, for some $a\in \mathbb{K}$, as claimed. 
\end{proof}

\begin{lemma}\label{Lemma: Equivalence conditions for Ext^{1} on R_j with k=1}
Consider the assumptions of Setup~\eqref{Setup: Objects on R_j for k=1 and adapted bases} and fix an equivalence class $\xi\in \mathrm{Ext}^{1}(\sh{F},\sh{G})$. Let $\sh{H}$ and $\sh{H}'$ be two equivalent extensions of $\sh{F}$ by $\sh{G}$ representing $\xi$, and let $\lambda:\sh{H}\to \sh{H}'$ be a sheaf isomorphism realizing their equivalence. 

On $R_{j}$, $\sh{H}$ and $\sh{H}'$ are specified by two collections of $n$ characteristic maps
\begin{equation*}
\begin{aligned}
\big\{\phi^{(i)}_{\sh{H}}:\mathbb{K}^{i}\to \mathbb{K}^{i+1}\,\big\}_{i=1}^{n-1}\,&\cup\,\big\{ \widetilde{\phi}^{\,(1)}_{\sh{H}}:\mathbb{K}^{1}\to \mathbb{K}^{2}\, \big\}\, , \\[2pt]
\big\{\phi^{(i)}_{\sh{H}'}:\mathbb{K}^{i}\to \mathbb{K}^{i+1}\,\big\}_{i=1}^{n-1}\,&\cup\,\big\{ \widetilde{\phi}^{\,(1)}_{\sh{H}'}:\mathbb{K}^{1}\to \mathbb{K}^{2}\, \big\}\, , 
\end{aligned}
\end{equation*}
while $\lambda$ is characterized by a collection of $n+1$ linear maps $\big\{ T^{(i)}_{\lambda}: \mathbb{K}^{i}\to \mathbb{K}^{i} \big\}\;\cup\; \big\{ \widetilde{T}^{(1)}:\mathbb{K}^{1}\to \mathbb{K}^{1} \big\}$ such that
\begin{equation*}
\begin{aligned}
\phi^{(i)}_{\sh{H}'} &= \phi^{(i)}_{\sh{H}} + T^{(i+1)}_{\lambda}\circ \phi^{(i)}_{\sh{F}}- \phi^{(i)}_{\sh{G}}\circ T^{(i)}_{\lambda} \, , \qquad \text{for all $i\in [1,n-1]$,} \\[2pt] 
\widetilde{\phi}^{\,(1)}_{\sh{H}'} &= \widetilde{\phi}^{\,(1)}_{\sh{H}} + T^{(2)}_{\lambda}\circ \widetilde{\phi}^{\,(1)}_{\sh{F}}- \widetilde{\phi}^{\,(1)}_{\sh{G}}\circ \widetilde{T}^{\,(1)}_{\lambda} \, .
\end{aligned}    
\end{equation*}

\noindent
$\star$ \emph{Assumption 1}: Suppose that there are linear maps $\big\{\gamma^{(i)}: \mathbb{K}^{i}\to \mathbb{K}^{i+1} \big\}_{i=1}^{n-1}$ that fully characterize $\xi$, up to equivalence, on the left of the crossing in $R_{j}$. Under this assumption, we choose representatives $\sh{H}$ and $\sh{H}'$ such that
\begin{equation*}
\phi^{(i)}_{\sh{H}}=\phi^{(i)}_{\sh{H}'}=\gamma^{(i)}\, , \qquad \text{for all $i\in [1,n-1]$}\, .    
\end{equation*}

\noindent
$\star$ \emph{Assumption 2}: Suppose that, with respect to the bases $\big\{\hat{\mathbf{f}}^{(i)}\big\}_{i=1}^{n}$ and $\big\{\hat{\mathbf{g}}^{(i)}\big\}_{i=1}^{n}$, the matrices representing the maps $\big\{ T^{(i)}_{\lambda} \big\}_{i=1}^{n}$, which characterize $\lambda$ on the left of the crossing in $R_{j}$, satisfy: 
\begin{equation*}
\begin{aligned}
 \tensor[_{\hat{\mathbf{g}}^{(n)}}]{ \big[\, T_{\lambda}^{(n)}\,\big] }{_{\hat{\mathbf{f}}^{(n)}}} &\equiv D(\vec{u}),~\text{ for some $\vec{u}=(u_{1},\dots, u_{n})\in \mathbb{K}^{n}_{\mathrm{std}}$}\, ,\\[2pt]
\tensor[_{\hat{\mathbf{g}}^{(i)}}]{ \big[\, T_{\lambda}^{(i)}\,\big] }{_{\hat{\mathbf{f}}^{(i)}}}&\equiv \,\text{principal $i\times i$ submatrix of } D(\vec{u})\, , \quad\text{for all  $i\in[1,n-1]$}\, .
\end{aligned}    
\end{equation*}

Under this assumption, Lemma~\eqref{Lemma: Ext1 on R_j for k=1} allows us to choose representatives $\sh{H}$ and $\sh{H}'$ such that, with respect to the bases $\hat{\mathbf{g}}^{(2)}[\sigma_{1}, z']$ and $\hat{\mathbf{f}}^{(1)}[\sigma_{1}, z]$, the matrices representing $\widetilde{\phi}^{\,(1)}_{\sh{H}}$ and $\widetilde{\phi}^{\,(1)}_{\sh{H}'}$ are given by 
\begin{equation*}
\tensor[_{\hat{\mathbf{g}}^{(2)}[\sigma_{1}, z']}]{ \big[\, \widetilde{\phi}^{\,(1)}_{\sh{H}} \,\big] }{_{\hat{\mathbf{f}}^{(1)}[\sigma_{1}, z]}}=\begin{bmatrix}
0\\
a
\end{bmatrix}\, , \quad \text{and} \quad \tensor[_{\hat{\mathbf{g}}^{(2)}[\sigma_{1}, z']}]{ \big[\, \widetilde{\phi}^{\,(1)}_{\sh{H}'} \,\big] }{_{\hat{\mathbf{f}}^{(1)}[\sigma_{1}, z]}}=\begin{bmatrix}
0\\
a'
\end{bmatrix}\, ,
\end{equation*}
for some $a, a'\in \mathbb{K}$. 

\noindent
$\star$ \emph{Main Conclusion}: Under the given assumptions, the following statements hold: 
\begin{itemize}
\justifying
\item \emph{Equivalence Condition}: $a'=a+d(z', \vec{u, z})$, where the residual parameter controlling the equivalence class in $\mathrm{Ext}^{1}(\sh{F},\sh{G})$ after the crossing in $R_{j}$ is given by
\begin{equation*}
d(z', \vec{u}, z):= \text{$(2,1)$-entry of}~~\big(B^{(n)}_{1}(z')\big)^{-1} \cdot D(\vec{u}) \cdot B^{(n)}_{1}(z)\, .   
\end{equation*}
\item With respect to the bases $\big\{\hat{\mathbf{f}}^{(i)}[\sigma_{1},z]\big\}_{i=1}^{n}$ and $\big\{\hat{\mathbf{g}}^{(i)}[\sigma_{1},z']\big\}_{i=1}^{n}$, the matrices representing the maps ${\widetilde{T}^{(1)}} \,\cup\, \big\{ T^{(i)}_{\lambda}\big\}_{i=2}^{n}$, which characterize $\lambda$ on the right of the crossing in $R_{j}$, satisfy:  
\begin{equation*}
\begin{aligned}
\tensor[_{\hat{\mathbf{g}}^{(n)}[\sigma_{1},z'] }]{ \big[\, T_{\lambda}^{\,(n)}\,\big] }{_{\hat{\mathbf{f}}^{(n)}[\sigma_{1},z]}} &\equiv D(\pi_{1}(\vec{u}))\,, \\[2pt]
 \tensor[_{\hat{\mathbf{g}}^{(1)}[\sigma_{1},z'] }]{ \big[\, \widetilde{T}_{\lambda}^{\,(1)}\,\big] }{_{\hat{\mathbf{f}}^{(1)}[\sigma_{1},z]}}&\equiv \,\text{principal $1\times 1$ submatrix of } D(\pi_{1}(\vec{u}))\, ,\\[2pt]
 \tensor[_{\hat{\mathbf{g}}^{(i)}[\sigma_{1},z'] }]{ \big[\, T_{\lambda}^{\,(i)}\,\big] }{_{\hat{\mathbf{f}}^{(i)}[\sigma_{1},z]}}& \equiv \,\text{principal $i\times i$ submatrix of } D(\pi_{1}(\vec{u}))\, , \quad \text{for all $i\in[2,n-1]$,}
\end{aligned}
\end{equation*}
where $\pi_{1}$ denotes the permutation associated with $\sigma_{1}$. 
\end{itemize}
\end{lemma}
\begin{proof}
Consider the bases $\big\{\hat{\mathbf{f}}^{(i)}\big\}_{i=1}^{n}$ and $\big\{\hat{\mathbf{g}}^{(i)}\big\}_{i=1}^{n}$, and set
\begin{equation*}
\begin{aligned}
Y^{(n)}&:=\tensor[_{\hat{\mathbf{g}}^{(n)}}]{ \big[\, T_{\lambda}^{(n)}\,\big] }{_{\hat{\mathbf{f}}^{(n)}}}\in \mathrm{M}(n,\mathbb{K})\, ,  \\[2pt]
Y^{(i)}&:=\tensor[_{\hat{\mathbf{g}}^{(i)}}]{ \big[\, T_{\lambda}^{(i)}\,\big] }{_{\hat{\mathbf{f}}^{(i)}}}\in \mathrm{M}(i,\mathbb{K})\, ,   \qquad \text{for all $i\in [1,n-1]$}.
\end{aligned}
\end{equation*}
By \emph{Assumption 2}, we have that $Y^{(n)}=D(\vec{u})$, and $Y^{(i)}$ is given by the principal $i\times i$ submatrix of $D(\vec{u})$, for all $i\in[1,n-1]$. 

Next, given the bases $\big\{\hat{\mathbf{f}}^{(i)}[\sigma_{1}, z]\big\}_{i=1}^{n}$ and $\big\{\hat{\mathbf{g}}^{(i)}[\sigma_{1},z']\big\}_{i=1}^{n}$, let us define
\begin{equation*}
\begin{aligned}
\widetilde{X}^{\,(1)}&:=\tensor[_{\hat{\mathbf{g}}^{(2)}[\sigma_{1}, z']}]{ \big[\, \widetilde{\phi}^{\,(1)}_{\sh{H}} \,\big] }{_{\hat{\mathbf{f}}^{(1)}[\sigma_{1}, z]}}=\begin{bmatrix}
0\\
a
\end{bmatrix}\in\mathrm{M}(2,1, \mathbb{K})\, , \\[4pt]
\widetilde{X}'^{\,(1)}&:=\tensor[_{\hat{\mathbf{g}}^{(2)}[\sigma_{1}, z']}]{ \big[\, \widetilde{\phi}^{\,(1)}_{\sh{H}'} \,\big] }{_{\hat{\mathbf{f}}^{(1)}[\sigma_{1}, z]}}= \begin{bmatrix}
0\\
a'
\end{bmatrix}\in\mathrm{M}(2,1, \mathbb{K})\,, \\[4pt]
\widetilde{Y}^{(1)}&:=\tensor[_{\hat{\mathbf{g}}^{(1)}[\sigma_{1},z']}]{ \big[\, \widetilde{T}^{(1)}_{\lambda}\,\big] }{_{\hat{\mathbf{f}}^{(1)}[\sigma_{1},z]}}=[\tilde{y}_{1,1}]\in \mathrm{M}(1,\mathbb{K})\, .
\end{aligned}
\end{equation*}
Then, with respect to the braid-transformed bases, the equivalence condition for the maps $\widetilde{\phi}^{\,(1)}_{\sh{H}}$ and $\widetilde{\phi}^{\,(1)}_{\sh{H}'}$ translates into the system of equations
\begin{equation}\label{Eq: special system of equations on R_j with k=1}
\widetilde{X}'^{\,(1)}= \widetilde{X}^{(1)}+\Big( \big(B^{(2)}_{1}(z')\big)^{-1}\cdot Y^{(2)} \cdot B^{(2)}_{1}(z)\Big) \cdot \iota^{(2,1)} -\iota^{(2,1)}\cdot \widetilde{Y}^{(1)} \, .  
\end{equation}
It follows from Lemma~\eqref{Lemma: system of equations for equivalent extensions at a crossing sigma_1}(b) that the system of equations in~\eqref{Eq: special system of equations on R_j with k=1} has a unique solution:
\begin{equation*}
\begin{aligned}
~\widetilde{Y}^{(1)}&= \,\text{principal $1\times 1$ submatrix of } \,\big(B^{(n)}_{1}(z')\big)^{-1}\cdot D(\vec{u})\cdot B^{(n)}_{1}(z)\,  ,\\[4pt]  
~a'&=a + d(z',\vec{u}, z) \, ,
\end{aligned}    
\end{equation*}
where
\begin{equation*}
d(z',\vec{u}, z)= \,\text{the $(2,1)$-entry of } \,\big(B^{(n)}_{1}(z')\big)^{-1}\cdot D(\vec{u})\cdot B^{(n)}_{1}(z)\, .   
\end{equation*}
Specifically, we have that $d(z', \vec{u}, z)$ measures the redundancy in the choice of representatives for the equivalence classes in $\mathrm{Ext}^{1}(\sh{F},\sh{G})$ on the right of the crossing in $R_{j}$, and since redundancies vanish when restricting to $\mathrm{Ext}^{1}(\sh{F},\sh{G})$, we deduce that $d(z', \vec{u}, z)\equiv 0 $.

In particular, observe that
\begin{equation*}
\big(B^{(n)}_{1}(z')\big)^{-1}\cdot D(\vec{u})\cdot B^{(n)}_{1}(z)=D(\pi_{1}(\vec{u}))+d(z',\vec{u}, z)\,E_{2,1}\, ,    
\end{equation*}
where $E_{2,1}\in \mathrm{M}(n,\mathbb{K})$ denotes the elementary matrix with 1 in the position $(2,1)$ and zeros everywhere else. Moreover, since $Y^{(i)}$ is given by the principal $i\times i$ submatrix of $D(\vec{u})$ for all $i\in[2,n-1]$, the block structure of the braid matrices guarantees that
\begin{equation*}
\big(B^{(i)}_{1}(z')\big)^{-1} \cdot Y^{(i)} \cdot B^{(i)}_{1}(z)= \text{ principal $i\times i$ submatrix of } ~\big(B^{(n)}_{1}(z')\big)^{-1}\cdot D(\vec{u})\cdot B^{(n)}_{1}(z)\,  ,   
\end{equation*}    
for all $i\in [2,n-1]$. 

Finally, note that, with respect to the bases $\big\{\hat{\mathbf{f}}^{(i)}[\sigma_{1},z]\big\}_{i=2}^{n}$ and $\big\{\hat{\mathbf{g}}^{(i)}[\sigma_{1},z']\big\}_{i=2}^{n}$, 
\begin{equation*}
\tensor[_{\hat{\mathbf{g}}^{(i)}[\sigma_{1},z'] }]{ \big[\, T_{\lambda}^{\,(i)}\,\big] }{_{\hat{\mathbf{f}}^{(i)}[\sigma_{1},z]}}= \big(B^{(i)}_{1}(z')\big)^{-1} \cdot \tensor[_{\hat{\mathbf{g}}^{(i)}}]{ \big[\, T_{\lambda}^{(i)}\,\big] }{_{\hat{\mathbf{f}}^{(i)}}} \cdot B^{(i)}_{1}(z)\, ,    
\end{equation*}
for all $i\in [2,n-1]$. Hence, when restricting to $\mathrm{Ext}^{1}(\sh{F},\sh{G})$, we obtain that
\begin{equation*}
\begin{aligned}
\tensor[_{\hat{\mathbf{g}}^{(n)}[\sigma_{1},z'] }]{ \big[\, T_{\lambda}^{\,(n)}\,\big] }{_{\hat{\mathbf{f}}^{(n)}[\sigma_{1},z]}} &\equiv D(\pi_{1}(\vec{u}))\,, \\[2pt]
 \tensor[_{\hat{\mathbf{g}}^{(1)}[\sigma_{1},z'] }]{ \big[\, \widetilde{T}_{\lambda}^{\,(1)}\,\big] }{_{\hat{\mathbf{f}}^{(1)}[\sigma_{1},z]}}&\equiv \,\text{principal $1\times 1$ submatrix of } D(\pi_{1}(\vec{u}))\, ,\\[2pt]
 \tensor[_{\hat{\mathbf{g}}^{(i)}[\sigma_{1},z'] }]{ \big[\, T_{\lambda}^{\,(i)}\,\big] }{_{\hat{\mathbf{f}}^{(i)}[\sigma_{1},z]}}& \equiv \,\text{principal $i\times i$ submatrix of } D(\pi_{1}(\vec{u}))\, , \quad \text{for all $i\in[2,n-1]$}. \\[2pt]
\end{aligned}
\end{equation*}
This completes the proof. 
\end{proof}

\begin{setup}\label{Setup: Objects on R_j for k geq 2 and adapted bases}
Let $\beta=\sigma_{i_{1}}\cdots \sigma_{i_{\ell}}\in\mathrm{Br}^{+}_{n}$ be a positive braid word, and let $\sh{F}$ and $\sh{G}$ be objects of the category $\ccs{1}{\beta}$. Fix $j\in[1,\ell]$, and let $R_{j}$ be the vertical strap in $\mathbb{R}^{2}$ containing $\sigma_{i_{j}}$---the $j$-th crossing of $\beta$ (see Figure~\eqref{fig: A sub-regions R_j}). 

\noindent
$\star$ \emph{Assumption 1}: Let $k:=i_{j}\in[1,n-1]$ denote the index of $\sigma_{i_{j}}$, and suppose that $k\geq 2$.

Under this assumption, on $R_{j}$, $\sh{F}$ and $\sh{G}$ are specified by two collections of $n+1$ injective linear maps 
\begin{equation*}
\begin{aligned}
\big\{ \phi^{(i)}_{\sh{F}}:\mathbb{K}^{i}\to \mathbb{K}^{i+1} \big\}_{i=1}^{n-1}&\,\cup\,\big\{ \widetilde{\phi}^{\,(k-1)}_{\sh{F}}:\mathbb{K}^{k-1}\to \mathbb{K}^{k}\,, ~ \widetilde{\phi}^{\,(k)}_{\sh{F}}:\mathbb{K}^{k}\to \mathbb{K}^{k+1} \big\}\, ,\\[6pt]
\big\{ \phi^{(i)}_{\sh{G}}:\mathbb{K}^{i}\to \mathbb{K}^{i+1} \big\}_{i=1}^{n-1}&\,\cup\,\big\{ \widetilde{\phi}^{\,(k-1)}_{\sh{G}}:\mathbb{K}^{k-1}\to \mathbb{K}^{k}\,, ~ \widetilde{\phi}^{\,(k)}_{\sh{G}}:\mathbb{K}^{k}\to \mathbb{K}^{k+1} \big\}\, ,
\end{aligned}    
\end{equation*}
respectively. For a schematic illustration of a generic representative of one of these sheaves, see Figure~\eqref{fig: a sheaf in the vertical strap R_j containing a crossing sigma_k}.

\noindent
$\star$ \emph{Assumption 2}: For each $i\in[1,n]$, let $\hat{\mathbf{f}}^{(i)}:=\big\{\hat{f}^{(i)}_{j}\big\}_{j=1}^{i}$ and $\hat{\mathbf{g}}^{(i)}:=\big\{\hat{g}^{(i)}_{j}\big\}_{j=1}^{i}$ be bases for $\mathbb{K}^{i}$. Following Definition~\eqref{Def:flags and adapted bases}--\eqref{Def: adapted bases I}, we assume that: 
\begin{itemize}
\justifying
\item $\big\{\hat{\mathbf{f}}^{(i)}\big\}_{i=1}^{n}$ is a system of bases adapted to $\big\{\phi_{\sh{F}}^{(i)}\,\big\}_{i=1}^{n-1}$.
\item $\big\{\hat{\mathbf{g}}^{(i)}\big\}_{i=1}^{n}$ is a system of bases adapted to $\big\{\phi_{\sh{G}}^{(i)}\,\big\}_{i=1}^{n-1}$.
\end{itemize}

Under this assumption, Lemma~\eqref{Lemma: matrix local model for sigma_k} ensures that $\big(\big\{\hat{\mathbf{f}}^{(i)}\big\}_{i=1}^{n}, z \big)$ and $\big(\big\{\hat{\mathbf{g}}^{(i)}\big\}_{i=1}^{n}, z' \big)$ are system of bases adapted to $\sh{F}$ and $\sh{G}$ on $R_{j}$, where $z,\,z'\in \mathbb{K}$ parameterize, relative to the bases $\hat{\mathbf{f}}^{(n)}$ and $\hat{\mathbf{g}}^{(n)}$ for $\mathbb{K}^{n}$, the $s_{k}$-relative position between the pairs of complete flags in $\mathbb{K}^{n}$ that geometrically characterize $\sh{F}$ and $\sh{G}$ on $R_{j}$, respectively. Following Definition~\eqref{Def: braid transformation of bases}, we denote by $\big\{\hat{\mathbf{f}}^{(i)}[\sigma_{k},z]\big\}_{i=1}^{n}$ and $\big\{\hat{\mathbf{g}}^{(i)}[\sigma_{k},z']\big\}_{i=1}^{n}$ the corresponding braid-transformed bases.   
\end{setup}

\begin{lemma}\label{Lemma: Ext1 on R_j for k geq 2}
Consider the assumptions of Setup~\eqref{Setup: Objects on R_j for k geq 2 and adapted bases} and fix an equivalence class $\xi\in \mathrm{Ext}^{1}(\sh{F},\sh{G})$. Let $\sh{H}$ and $\sh{H}'$ be two equivalent extensions of $\sh{F}$ by $\sh{G}$ representing $\xi$, and let $\lambda:\sh{H}\to \sh{H}'$ be a sheaf isomorphism realizing their equivalence. 

On $R_{j}$, $\sh{H}$ and $\sh{H}'$ are specified by two collections of $n+1$ characteristic maps
\begin{equation}\label{Eq: characteristic maps of two extensions on R_j with k geq 2}
\begin{aligned}
\big\{ \phi^{(i)}_{\sh{H}}:\mathbb{K}^{i}\to \mathbb{K}^{i+1} \big\}_{i=1}^{n-1}&\,\cup\,\big\{ \widetilde{\phi}^{\,(k-1)}_{\sh{H}}:\mathbb{K}^{k-1}\to \mathbb{K}^{k}\,, ~ \widetilde{\phi}^{\,(k)}_{\sh{H}}:\mathbb{K}^{k}\to \mathbb{K}^{k+1} \big\}\, , \\[2pt]
\big\{ \phi^{(i)}_{\sh{H}'}:\mathbb{K}^{i}\to \mathbb{K}^{i+1} \big\}_{i=1}^{n-1}&\,\cup\,\big\{ \widetilde{\phi}^{\,(k-1)}_{\sh{H}'}:\mathbb{K}^{k-1}\to \mathbb{K}^{k}\,, ~ \widetilde{\phi}^{\,(k)}_{\sh{H}'}:\mathbb{K}^{k}\to \mathbb{K}^{k+1} \big\}\, ,  
\end{aligned}
\end{equation}
while $\lambda$ is characterized by a collection of $n+1$ linear maps 
\begin{equation}\label{Eq: maps characterizing lambda on R_j with k geq 2}
\big\{ T^{(i)}_{\lambda}: \mathbb{K}^{i}\to \mathbb{K}^{i} \big\}_{i=1}^{n}\;\cup\; \big\{ \widetilde{T}^{\,(k)}:\mathbb{K}^{k}\to \mathbb{K}^{k} \big\}    
\end{equation}
such that:
\begin{subequations}\label{Eq: equivalence relations on R_j with k geq 2}
\begin{align}
\phi^{(i)}_{\sh{H}'} &= \phi^{(i)}_{\sh{H}} + T^{(i+1)}_{\lambda}\circ \phi^{(i)}_{\sh{F}}- \phi^{(i)}_{\sh{G}}\circ T^{(i)}_{\lambda} \, , \qquad \text{for all $i\in [1,n-1]$}\, , \label{Eq: equivalence relations for R_j with k geq 2 a}\\[4pt] 
\widetilde{\phi}^{\,(k)}_{\sh{H}'} &= \widetilde{\phi}^{\,(k)}_{\sh{H}} + T^{(k+1)}_{\lambda}\circ \widetilde{\phi}^{\,(k)}_{\sh{F}}- \widetilde{\phi}^{\,(k)}_{\sh{G}}\circ \widetilde{T}^{\,(k)}_{\lambda} \, , \label{Eq: equivalence relations for R_j with k geq 2 b} \\[4pt]
\widetilde{\phi}^{\,(k-1)}_{\sh{H}'} &= \widetilde{\phi}^{\,(k-1)}_{\sh{H}} + \widetilde{T}^{\,(k)}_{\lambda}\circ \widetilde{\phi}^{\,(k-1)}_{\sh{F}}- \widetilde{\phi}^{\,(k-1)}_{\sh{G}}\circ T^{\,(k-1)}_{\lambda} \, . \label{Eq: equivalence relations for R_j with k geq 2 c}
\end{align}    
\end{subequations}

\noindent
$\star$ \emph{Assumption 1}: Suppose that there are linear maps $\big\{\gamma^{(i)}: \mathbb{K}^{i}\to \mathbb{K}^{i+1} \big\}_{i=1}^{n-1}$ that fully characterize $\xi$, up to equivalence, on the left of the crossing in $R_{j}$. Building on this assumption, we choose representatives $\sh{H}$ and $\sh{H}'$ such that:
\begin{equation*}
\phi^{(i)}_{\sh{H}}=\phi^{(i)}_{\sh{H}'}=\gamma^{(i)}\, , \qquad \text{for all $i\in [1,n-1]$}\, .    
\end{equation*}

\noindent
$\star$ \emph{Assumption 2}: Suppose that, for the above choice of representatives and with respect to the bases $\big\{\hat{\mathbf{f}}^{(i)}\big\}_{i=1}^{n}$ and $\big\{\hat{\mathbf{g}}^{(i)}\big\}_{i=1}^{n}$, the matrices representing the maps $\big\{ T^{(i)}_{\lambda} \big\}_{i=1}^{n}$, which characterize $\lambda$ on the left of the crossing in $R_{j}$, satisfy: 
\begin{equation*}
\begin{aligned}
 \tensor[_{\hat{\mathbf{g}}^{(n)}}]{ \big[\, T_{\lambda}^{(n)}\,\big] }{_{\hat{\mathbf{f}}^{(n)}}} &\equiv D(\vec{u}),~\text{ for some $\vec{u}=(u_{1},\dots, u_{n})\in \mathbb{K}^{n}_{\mathrm{std}}$}\, ,\\[4pt]
\tensor[_{\hat{\mathbf{g}}^{(i)}}]{ \big[\, T_{\lambda}^{(i)}\,\big] }{_{\hat{\mathbf{f}}^{(i)}}}&\equiv \,\text{principal $i\times i$ submatrix of } D(\vec{u})\, , \quad\text{for all $i\in[1,n-1]$}\, .
\end{aligned}    
\end{equation*}

In particular, under the above assumptions, the equivalence conditions in~\eqref{Eq: equivalence relations for R_j with k geq 2 a} are automatically satisfied.

\noindent
$\star$ \emph{Main Conclusion}: $\sh{H}$ is equivalent to a representative $\sh{H}'$ such that, with respect to the bases $\big\{\hat{\mathbf{f}}^{(i)}[\sigma_{k}, z]\big\}_{i=1}^{n}$ and $\big\{\hat{\mathbf{g}}^{(i)}[\sigma_{k}, z']\big\}_{i=1}^{n}$, the matrices representing $\widetilde{\phi}^{\,(k-1)}_{\sh{H}'}$ and $\widetilde{\phi}^{\,(k)}_{\sh{H}'}$ are given by 
\begin{equation*}
\tensor[_{\hat{\mathbf{g}}^{(k)}[\sigma_{k}, z']}]{ \big[\, \widetilde{\phi}^{\,(k-1)}_{\sh{H}'} \,\big] }{_{\hat{\mathbf{f}}^{(k-1)}[\sigma_{k}, z]}}= \mathbf{0}_{k\times (k-1)} \, , \quad \text{and} \quad   \tensor[_{\hat{\mathbf{g}}^{(k+1)}[\sigma_{k}, z']}]{ \big[\, \widetilde{\phi}^{\,(k)}_{\sh{H}'} \,\big] }{_{\hat{\mathbf{f}}^{(k)}[\sigma_{k}, z]}}= \begin{bmatrix}
\star & \cdots & \star  & 0\\
     \vdots     & \ddots &         \vdots        & \vdots\\     
\star  & \cdots & \star  & 0\\
\star  & \cdots & \star & a'
\end{bmatrix}\, , \\[6pt]    
\end{equation*}
for some $a' \in \mathbb{K}$, where the starred entries are uniquely determined by the data associated with $\phi^{(k)}_{\sh{G}}$, $\gamma^{(k-1)}$, $\gamma^{(k)}$, and $\phi^{(k-1)}_{\sh{F}}$. In other words, the parameter $a'$, together with the data encoded in $\big\{\gamma^{(i)}\big\}_{i=1}^{n-1}$, completely determines $\xi$ on $R_{j}$, up to equivalence. 
\end{lemma}
\begin{proof}
To begin, recall that $\sh{H}$ and $\sh{H}'$ are two extensions of $\sh{F}$ by $\sh{G}$, and hence their local description on $R_{j}$ via the characteristic maps in~\eqref{Eq: characteristic maps of two extensions on R_j with k geq 2} follows from Observation~\eqref{Obs: elements of Ext1 at the local models}. Moreover, building on \emph{Assumption 1}, we suppose that $\sh{H}$ and $\sh{H}'$ are representatives with $\phi^{(i)}_{\sh{H}}=\phi^{(i)}_{\sh{H}'}=\gamma^{(i)}$, for all $i\in [1,n-1]$. 

Analogously to the case of $\sh{H}$ and $\sh{H}'$, it follows from Observation~\eqref{Obs: elements of Ext1 at the local models} that $\lambda:\sh{H}\to \sh{H}'$, the sheaf isomorphism realizing the equivalence between $\sh{H}$ and $\sh{H}'$, is locally characterized on $R_{j}$ by a collection of linear maps as in~\eqref{Eq: maps characterizing lambda on R_j with k geq 2}, which are subject to the equivalence conditions~\eqref{Eq: equivalence relations on R_j with k geq 2}.

Moreover, building on Observation~\eqref{Obs: elements of Ext1 at the local models}, we know that the microlocal support conditions near the crossing impose the following constraints on the characteristic maps of $\sh{H}$ and $\sh{H}'$
\begin{equation*}
\begin{aligned}
\phi^{(k)}_{\sh{G}}\,\circ \gamma^{(k)} + \gamma^{(k-1)}\,\circ \phi^{(k-1)}_{\sh{F}} = \widetilde{\phi}^{\,(k)}_{\sh{G}}\,\circ \widetilde{\phi}^{\,(k-1)}_{\sh{H}} + \widetilde{\phi}^{\,(k)}_{\sh{H}}\,\circ \widetilde{\phi}^{\,(k-1)}_{\sh{F}} \, ,\\[2pt] 
\phi^{(k)}_{\sh{G}}\,\circ \gamma^{(k)} + \gamma^{(k-1)}\,\circ \phi^{(k-1)}_{\sh{F}} = \widetilde{\phi}^{\,(k)}_{\sh{G}}\,\circ \widetilde{\phi}^{\,(k-1)}_{\sh{H}'} + \widetilde{\phi}^{\,(k)}_{\sh{H}'}\,\circ \widetilde{\phi}^{\,(k-1)}_{\sh{F}} \, ,\\[2pt] 
\end{aligned}
\end{equation*}
where have used that $\phi^{(i)}_{\sh{H}}=\phi^{(i)}_{\sh{H}'}=\gamma^{(i)}$, for all $i\in [1,n-1]$. 

Here, given the bases $\big\{\hat{\mathbf{f}}^{(i)}\big\}_{i=1}^{n}$ and $\big\{\hat{\mathbf{g}}^{(i)}\big\}_{i=1}^{n}$, let us define
\begin{equation*}
\begin{aligned}
Y^{(n)}&:=\tensor[_{\hat{\mathbf{g}}^{(n)}}]{ \big[\, T_{\lambda}^{(n)}\,\big] }{_{\hat{\mathbf{f}}^{(n)}}}\in \mathrm{M}(n,\mathbb{K})\, ,  \\[2pt]
Y^{(i)}&:=\tensor[_{\hat{\mathbf{g}}^{(i)}}]{ \big[\, T_{\lambda}^{(i)}\,\big] }{_{\hat{\mathbf{f}}^{(i)}}}\in \mathrm{M}(i,\mathbb{K})\, ,   \quad \text{for all $i\in [1,n-1]$}\, ,\\[2pt]
Z^{(k-1)} &:=\tensor[_{\hat{\mathbf{g}}^{(k)}}]{ \big[\, \gamma^{(k-1)} \,\big] }{_{\hat{\mathbf{f}}^{(k-1)}}}\in \mathrm{M}(k, k-1,\mathbb{K})\, ,  \\[2pt]
Z^{(k)} &:=\tensor[_{\hat{\mathbf{g}}^{(k+1)}}]{ \big[\, \gamma^{(k)} \,\big] }{_{\hat{\mathbf{f}}^{(k)}}}\in \mathrm{M}(k+1, k,\mathbb{K})\, .  \\[2pt]
\end{aligned}
\end{equation*}
In particular, by \emph{Assumption 2}, we have that $Y^{(n)}=D(\vec{u})$, and $Y^{(i)}$ is given by the principal $i\times i$ submatrix of $D(\vec{u})$, for all $i\in[1,n-1]$. 

Also, consider the bases $\big\{\hat{\mathbf{f}}^{(i)}[\sigma_{k}, z]\big\}_{i=1}^{n}$ and $\big\{\hat{\mathbf{g}}^{(i)}[\sigma_{k},z']\big\}_{i=1}^{n}$, and set
\begin{equation*}
\begin{aligned}
\widetilde{Y}^{(k)}&:= \tensor[_{\hat{\mathbf{g}}^{(k)}[\sigma_{k}, z']}]{ \big[\, \widetilde{T}^{\,(k)}_{\lambda} \,\big] }{_{\hat{\mathbf{f}}^{(k)}[\sigma_{k}, z]}}\in \mathrm{M}(k,\mathbb{K})\, , \\[2pt]
\widetilde{X}^{(k-1)} &:=\tensor[_{\hat{\mathbf{g}}^{(k)}[\sigma_{k}, z']}]{ \big[\, \widetilde{\phi}^{\,(k-1)}_{\sh{H}'} \,\big] }{_{\hat{\mathbf{f}}^{(k-1)}[\sigma_{k}, z]}}\in \mathrm{M}(k, k-1,\mathbb{K})\, ,  \\[2pt]
\widetilde{X}^{(k)} &:=\tensor[_{\hat{\mathbf{g}}^{(k+1)}[\sigma_{k}, z']}]{ \big[\, \widetilde{\phi}^{\,(k)}_{\sh{H}'} \,\big] }{_{\hat{\mathbf{f}}^{(k)}[\sigma_{k}, z]}}\in \mathrm{M}(k+1, k,\mathbb{K})\, ,  \\[2pt]
\widetilde{X}'^{\,(k-1)} &:=\tensor[_{\hat{\mathbf{g}}^{(k)}[\sigma_{k}, z']}]{ \big[\, \widetilde{\phi}^{\,(k-1)}_{\sh{H}'} \,\big] }{_{\hat{\mathbf{f}}^{(k-1)}[\sigma_{k}, z]}}\in \mathrm{M}(k, k-1,\mathbb{K})\, ,  \\[2pt]
\widetilde{X}'^{\,(k)} &:=\tensor[_{\hat{\mathbf{g}}^{(k+1)}[\sigma_{k}, z']}]{ \big[\, \widetilde{\phi}^{\,(k)}_{\sh{H}'} \,\big] }{_{\hat{\mathbf{f}}^{(k)}[\sigma_{k}, z]}}\in \mathrm{M}(k+1, k,\mathbb{K})\, .  \\[2pt]    
\end{aligned}    
\end{equation*}

Next, observe that, with respect to the bases $\big\{\hat{\mathbf{f}}^{(i)}[\sigma_{k}, z]\big\}_{i=1}^{n}$ and $\big\{\hat{\mathbf{g}}^{(i)}[\sigma_{k},z']\big\}_{i=1}^{n}$, the equivalence relations~\eqref{Eq: equivalence relations for R_j with k geq 2 b} and~\eqref{Eq: equivalence relations for R_j with k geq 2 b} translate into the system of equations
\begin{equation*}
\begin{aligned}
\widetilde{X}'^{\,(k-1)} &= \widetilde{X}^{(k-1)} + \widetilde{Y}^{(k)}\cdot\iota^{(k,k-1)} - \iota^{(k,k-1)}\cdot Y^{(k-1)}\,, \\[4pt]
\widetilde{X}'^{\,(k)}   &= \widetilde{X}^{(k)} + \big(B^{(k+1)}_{k}(z')\big)^{-1}\cdot Y^{(k+1)}\cdot B^{(k+1)}_k(z)\cdot\iota^{(k+1,k)} - \iota^{(k+1,k)}\cdot \widetilde{Y}^{(k)}\, .     
\end{aligned}     
\end{equation*}
It follows from Lemma~\eqref{Lemma: system of equations for equivalent extensions at a crossing sigma_k}(a) that for any choice of $\widetilde{X}'^{\,(k-1)}$ and the $k$-th column of $\widetilde{X}'^{\,(k)}$ expect its very last entry, the above system of equations uniquely determines $\widetilde{Y}^{(k)}$ and the remaining entries of $\widetilde{X}'^{\,(k)}$, and hence, $\sh{H}$ is equivalent to a representative $\sh{H}'$ such that with respect to the bases $\big\{\hat{\mathbf{f}}^{(i)}[\sigma_{k}, z]\big\}_{i=1}^{n}$ and $\big\{\hat{\mathbf{g}}^{(i)}[\sigma_{k}, z']\big\}_{i=1}^{n}$, the matrices representing $\widetilde{\phi}^{\,(k-1)}_{\sh{H}'}$ and $\widetilde{\phi}^{\,(k)}_{\sh{H}'}$ are given by 
\begin{equation}\label{Eq: choice for an auxiliar extension for k geq 2}
\widetilde{X}'^{\,(k-1)}=\mathbf{0}_{k\times (k-1)}\, , \quad \text{and} \quad \widetilde{X}'^{\,(k)}=\begin{bmatrix}
\tilde{x}'_{1,1} & \cdots & \widetilde{x}'_{1,k-1} & 0\\
\vdots     & \ddots &   \vdots     & \vdots\\
\tilde{x}'_{k,1} & \cdots & \widetilde{x}'_{k, k-1} & 0 \\
\tilde{x}'_{k+1,1} & \cdots & \widetilde{x}'_{k+1, k-1} & \widetilde{x}'_{k+1,k}
\end{bmatrix}\, .
\end{equation}

Finally, note that, with respect to the bases $\big\{\hat{\mathbf{f}}^{(i)}[\sigma_{k}, z]\big\}_{i=1}^{n}$ and $\big\{\hat{\mathbf{g}}^{(i)}[\sigma_{k},z']\big\}_{i=1}^{n}$, the crossing constraint for the extensions $\sh{H}'$ translate into the system of equations
\begin{equation}\label{Eq: system of equations for the crossing contrain of an auxiliar extension}
\begin{aligned}
\big(B^{(k+1)}_{k}(z')\big)^{-1}\cdot \Big( \iota^{(k+1,k)}\cdot Z^{(k-1)}+ Z^{(k)}\cdot \iota^{(k,k-1)} \Big)  =   \iota^{(k+1,k)}\cdot \widetilde{X}'^{\,(k-1)}+ \widetilde{X}'^{\,(k)}\cdot \iota^{(k,k-1)}\, .    
\end{aligned}
\end{equation}
In particular, given our choice of representative $\sh{H}'$ in~\eqref{Eq: choice for an auxiliar extension for k geq 2}, Lemma~\eqref{Lemma: system of equations for the extension crossing constraint} ensures that 
the first $k-1$ columns of $\widetilde{X}'^{\,(k)}$ are uniquely determined by the system of equations in~\eqref{Eq: system of equations for the crossing contrain of an auxiliar extension}, and in this case, these columns are uniquely determined by the data associated with $\phi^{(k)}_{\sh{G}}$, $\gamma^{(k-1)}$, $\gamma^{(k)}$, and $\phi^{(k-1)}_{\sh{F}}$. Thus, the result follows once we identify $a'=\tilde{x}'_{k+1,k}$. 
\end{proof}

\begin{lemma}\label{Lemma: Equivalence conditions for Ext^{1} on R_j with k geq 2}
Consider the assumptions of Setup~\eqref{Setup: Objects on R_j for k geq 2 and adapted bases} and fix an equivalence class $\xi\in \mathrm{Ext}^{1}(\sh{F},\sh{G})$. Let $\sh{H}$ and $\sh{H}'$ be two equivalent extensions of $\sh{F}$ by $\sh{G}$ representing $\xi$, and let $\lambda:\sh{H}\to \sh{H}'$ be a sheaf isomorphism realizing their equivalence. 

On $R_{j}$, $\sh{H}$ and $\sh{H}'$ are specified by two collections of $n+1$ characteristic maps
\begin{equation*}
\begin{aligned}
\big\{ \phi^{(i)}_{\sh{H}}:\mathbb{K}^{i}\to \mathbb{K}^{i+1} \big\}_{i=1}^{n-1}&\,\cup\,\big\{ \widetilde{\phi}^{\,(k-1)}_{\sh{H}}:\mathbb{K}^{k-1}\to \mathbb{K}^{k}\,, ~ \widetilde{\phi}^{\,(k)}_{\sh{H}}:\mathbb{K}^{k}\to \mathbb{K}^{k+1} \big\}\, , \\[2pt]
\big\{ \phi^{(i)}_{\sh{H}'}:\mathbb{K}^{i}\to \mathbb{K}^{i+1} \big\}_{i=1}^{n-1}&\,\cup\,\big\{ \widetilde{\phi}^{\,(k-1)}_{\sh{H}'}:\mathbb{K}^{k-1}\to \mathbb{K}^{k}\,, ~ \widetilde{\phi}^{\,(k)}_{\sh{H}'}:\mathbb{K}^{k}\to \mathbb{K}^{k+1} \big\}\, ,  
\end{aligned}
\end{equation*}
while $\lambda$ is characterized by a collection of $n+1$ linear maps 
\begin{equation*}
\big\{ T^{(i)}_{\lambda}: \mathbb{K}^{i}\to \mathbb{K}^{i} \big\}_{i=1}^{n}\;\cup\; \big\{ \widetilde{T}^{\,(k)}:\mathbb{K}^{k}\to \mathbb{K}^{k} \big\}    
\end{equation*}
such that:
\begin{equation*}
\begin{aligned}
\phi^{(i)}_{\sh{H}'} &= \phi^{(i)}_{\sh{H}} + T^{(i+1)}_{\lambda}\circ \phi^{(i)}_{\sh{F}}- \phi^{(i)}_{\sh{G}}\circ T^{(i)}_{\lambda} \, , \qquad \text{for all $i\in [1,n-1]$}\, ,\\[4pt] 
\widetilde{\phi}^{\,(k)}_{\sh{H}'} &= \widetilde{\phi}^{\,(k)}_{\sh{H}} + T^{(k+1)}_{\lambda}\circ \widetilde{\phi}^{\,(k)}_{\sh{F}}- \widetilde{\phi}^{\,(k)}_{\sh{G}}\circ \widetilde{T}^{\,(k)}_{\lambda} \, , \\[4pt]
\widetilde{\phi}^{\,(k-1)}_{\sh{H}'} &= \widetilde{\phi}^{\,(k-1)}_{\sh{H}} + \widetilde{T}^{\,(k)}_{\lambda}\circ \widetilde{\phi}^{\,(k-1)}_{\sh{F}}- \widetilde{\phi}^{\,(k-1)}_{\sh{G}}\circ T^{\,(k-1)}_{\lambda} \, .    
\end{aligned}    
\end{equation*}

\noindent
$\star$ \emph{Assumption 1}: Suppose that there are linear maps $\big\{\gamma^{(i)}: \mathbb{K}^{i}\to \mathbb{K}^{i+1} \big\}_{i=1}^{n-1}$ that fully characterize $\xi$, up to equivalence, on the left of the crossing in $R_{j}$. Building on this assumption, we choose representatives $\sh{H}$ and $\sh{H}'$ such that:
\begin{equation*}
\phi^{(i)}_{\sh{H}}=\phi^{(i)}_{\sh{H}'}=\gamma^{(i)}\, , \qquad \text{for all $i\in [1,n-1]$}\, .    
\end{equation*}

\noindent
$\star$ \emph{Assumption 2}: Suppose that, with respect to the bases $\big\{\hat{\mathbf{f}}^{(i)}\big\}_{i=1}^{n}$ and $\big\{\hat{\mathbf{g}}^{(i)}\big\}_{i=1}^{n}$, the matrices representing the maps $\big\{ T^{(i)}_{\lambda} \big\}_{i=1}^{n}$, which characterize $\lambda$ on the left of the crossing in $R_{j}$, satisfy: 
\begin{equation*}
\begin{aligned}
 \tensor[_{\hat{\mathbf{g}}^{(n)}}]{ \big[\, T_{\lambda}^{(n)}\,\big] }{_{\hat{\mathbf{f}}^{(n)}}} &\equiv D(\vec{u}),~\text{ for some $\vec{u}=(u_{1},\dots, u_{n})\in \mathbb{K}^{n}_{\mathrm{std}}$}\, ,\\[2pt]
\tensor[_{\hat{\mathbf{g}}^{(i)}}]{ \big[\, T_{\lambda}^{(i)}\,\big] }{_{\hat{\mathbf{f}}^{(i)}}}&\equiv \,\text{principal $i\times i$ submatrix of } D(\vec{u})\, , \quad\text{for all $i\in[1,n-1]$}\, .
\end{aligned}    
\end{equation*}

Under this assumption, Lemma~\eqref{Lemma: Ext1 on R_j for k geq 2} allows us to choose representatives $\sh{H}$ and $\sh{H}'$ such that with respect to the bases $\big\{\hat{\mathbf{f}}^{(i)}[\sigma_{k}, z]\big\}_{i=1}^{n}$ and $\big\{\hat{\mathbf{g}}^{(i)}[\sigma_{k}, z']\big\}_{i=1}^{n}$, the matrices representing $\widetilde{\phi}^{\,(k-1)}_{\sh{H}}$, $\widetilde{\phi}^{\,(k)}_{\sh{H}}$, $\widetilde{\phi}^{\,(k-1)}_{\sh{H}'}$ and $\widetilde{\phi}^{\,(k)}_{\sh{H}'}$ are given by 
\begin{equation*}
\begin{aligned}
\tensor[_{\hat{\mathbf{g}}^{(k)}[\sigma_{k}, z']}]{ \big[\, \widetilde{\phi}^{\,(k-1)}_{\sh{H}} \,\big] }{_{\hat{\mathbf{f}}^{(k-1)}[\sigma_{k}, z]}}&= \mathbf{0}_{k\times (k-1)} \, , \quad \text{and} \quad   \tensor[_{\hat{\mathbf{g}}^{(k+1)}[\sigma_{k}, z']}]{ \big[\, \widetilde{\phi}^{\,(k)}_{\sh{H}} \,\big] }{_{\hat{\mathbf{f}}^{(k)}[\sigma_{k}, z]}}= \begin{bmatrix}
\star & \cdots & \star  & 0\\
     \vdots     & \ddots &         \vdots        & \vdots\\     
\star  & \cdots & \star  & 0\\
\star  & \cdots & \star & a
\end{bmatrix}\, , \\[8pt]    
\tensor[_{\hat{\mathbf{g}}^{(k)}[\sigma_{k}, z']}]{ \big[\, \widetilde{\phi}^{\,(k-1)}_{\sh{H}'} \,\big] }{_{\hat{\mathbf{f}}^{(k-1)}[\sigma_{k}, z]}} &= \mathbf{0}_{k\times (k-1)} \, , \quad \text{and} \quad   \tensor[_{\hat{\mathbf{g}}^{(k+1)}[\sigma_{k}, z']}]{ \big[\, \widetilde{\phi}^{\,(k)}_{\sh{H}'} \,\big] }{_{\hat{\mathbf{f}}^{(k)}[\sigma_{k}, z]}}= \begin{bmatrix}
\star & \cdots & \star  & 0\\
     \vdots     & \ddots &         \vdots        & \vdots\\     
\star  & \cdots & \star  & 0\\
\star  & \cdots & \star & a'
\end{bmatrix}\, , 
\end{aligned}
\end{equation*}
for some $a, a' \in \mathbb{K}$. 

\noindent
$\star$ \emph{Main Conclusion}: Under the given assumptions, the following statements hold: 
\begin{itemize}
\justifying
\item \emph{Equivalence Condition}: $a'=a+d(z', \vec{u, z})$, where the residual parameter controlling the equivalence class in $\mathrm{Ext}^{1}(\sh{F},\sh{G})$ after the crossing in $R_{j}$ is given by
\begin{equation*}
d(z', \vec{u}, z):= \text{$(k+1,k)$-entry of}~~\big(B^{(n)}_{k}(z')\big)^{-1} \cdot D(\vec{u}) \cdot B^{(n)}_{k}(z)\, .   
\end{equation*}
\item With respect to the bases $\big\{\hat{\mathbf{f}}^{(i)}[\sigma_{k},z]\big\}_{i=1}^{n}$ and $\big\{\hat{\mathbf{g}}^{(i)}[\sigma_{k},z']\big\}_{i=1}^{n}$, the matrices representing the collection of maps $\big\{ T^{(i)}_{\lambda}\big\}_{i\in [1,n]\setminus \{k\} }\,\cup\, \big\{\widetilde{T}^{\,(k)}_{\lambda}\big\}$, which characterize $\lambda$ on the right of the crossing in $R_{j}$, satisfy:  
\begin{equation*}
\begin{aligned}
\tensor[_{\hat{\mathbf{g}}^{(n)}[\sigma_{k},z'] }]{ \big[\, T_{\lambda}^{\,(n)}\,\big] }{_{\hat{\mathbf{f}}^{(n)}[\sigma_{k},z]}} &\equiv D(\pi_{k}(\vec{u}))\,, \\[2pt]
\tensor[_{\hat{\mathbf{g}}^{(k)}[\sigma_{k},z'] }]{ \big[\, \widetilde{T}_{\lambda}^{\,(k)}\,\big] }{_{\hat{\mathbf{f}}^{(k)}[\sigma_{k},z]}}&\equiv \,\text{principal $k\times k$ submatrix of } D(\pi_{k}(\vec{u}))\, ,\\[2pt]
\tensor[_{\hat{\mathbf{g}}^{(i)}[\sigma_{k},z'] }]{ \big[\, T_{\lambda}^{\,(i)}\,\big] }{_{\hat{\mathbf{f}}^{(i)}[\sigma_{k},z]}}& \equiv \,\text{principal $i\times i$ submatrix of } D(\pi_{k}(\vec{u}))\, , \quad \text{for all $i\in[1,n-1]\setminus \{k\}$}\, ,\\[4pt]
\end{aligned}
\end{equation*}
where $\pi_{k}$ denotes the permutation associated with $\sigma_{k}$. 
\end{itemize}
\end{lemma}
\begin{proof}
Consider the bases $\big\{\hat{\mathbf{f}}^{(i)}\big\}_{i=1}^{n}$ and $\big\{\hat{\mathbf{g}}^{(i)}\big\}_{i=1}^{n}$, and set
\begin{equation*}
\begin{aligned}
Y^{(n)}&:=\tensor[_{\hat{\mathbf{g}}^{(n)}}]{ \big[\, T_{\lambda}^{(n)}\,\big] }{_{\hat{\mathbf{f}}^{(n)}}}\in \mathrm{M}(n,\mathbb{K})\, ,  \\[2pt]
Y^{(i)}&:=\tensor[_{\hat{\mathbf{g}}^{(i)}}]{ \big[\, T_{\lambda}^{(i)}\,\big] }{_{\hat{\mathbf{f}}^{(i)}}}\in \mathrm{M}(i,\mathbb{K})\, ,   \qquad \text{for all $i\in [1,n-1]$}.
\end{aligned}
\end{equation*}
By \emph{Assumption 2}, we have that $Y^{(n)}=D(\vec{u})$, and $Y^{(i)}$ is given by the principal $i\times i$ submatrix of $D(\vec{u})$, for all $i\in[1,n-1]$. 

Next, given the bases $\big\{\hat{\mathbf{f}}^{(i)}[\sigma_{k}, z]\big\}_{i=1}^{n}$ and $\big\{\hat{\mathbf{g}}^{(i)}[\sigma_{k},z']\big\}_{i=1}^{n}$, let us define
\begin{equation*}
\begin{aligned}
\widetilde{X}^{(k-1)}&:=\tensor[_{\hat{\mathbf{g}}^{(k)}[\sigma_{k}, z']}]{ \big[\, \widetilde{\phi}^{\,(k-1)}_{\sh{H}} \,\big] }{_{\hat{\mathbf{f}}^{(k-1)}[\sigma_{k}, z]}}= \mathbf{0}_{k\times (k-1)} \, , \qquad  \widetilde{X}^{(k)}:=\tensor[_{\hat{\mathbf{g}}^{(k+1)}[\sigma_{k}, z']}]{ \big[\, \widetilde{\phi}^{\,(k)}_{\sh{H}} \,\big] }{_{\hat{\mathbf{f}}^{(k)}[\sigma_{k}, z]}}= \begin{bmatrix}
\star & \cdots & \star  & 0\\
     \vdots     & \ddots &         \vdots        & \vdots\\     
\star  & \cdots & \star  & 0\\
\star  & \cdots & \star & a
\end{bmatrix}\, , \\[8pt]    
\widetilde{X}'^{\,(k-1)}&:=\tensor[_{\hat{\mathbf{g}}^{(k)}[\sigma_{k}, z']}]{ \big[\, \widetilde{\phi}^{\,(k-1)}_{\sh{H}'} \,\big] }{_{\hat{\mathbf{f}}^{(k-1)}[\sigma_{k}, z]}}= \mathbf{0}_{k\times (k-1)} \, , \qquad  \widetilde{X}'^{\,(k)}:=\tensor[_{\hat{\mathbf{g}}^{(k+1)}[\sigma_{k}, z']}]{ \big[\, \widetilde{\phi}^{\,(k)}_{\sh{H}'} \,\big] }{_{\hat{\mathbf{f}}^{(k)}[\sigma_{k}, z]}}= \begin{bmatrix}
\star & \cdots & \star  & 0\\
     \vdots     & \ddots &         \vdots        & \vdots\\     
\star  & \cdots & \star  & 0\\
\star  & \cdots & \star & a'
\end{bmatrix}\, .\\   
\end{aligned}
\end{equation*}
Then, with respect to the braid-transformed bases, the equivalence conditions for the pairs of maps $\big(\widetilde{\phi}^{\,(k-1)}_{\sh{H}}, \widetilde{\phi}^{\,(k-1)}_{\sh{H}'}\big)$ and $\big(\widetilde{\phi}^{\,(k)}_{\sh{H}}, \widetilde{\phi}^{\,(k)}_{\sh{H}'}\big)$ translate into the system of equations
\begin{equation}\label{Eq: special system of equations on R_j with k geq 2}
\begin{aligned}
\widetilde{X}'^{\,(k-1)} &= \widetilde{X}^{(k-1)} + \widetilde{Y}^{(k)}\cdot\iota^{(k,k-1)} - \iota^{(k,k-1)}\cdot Y^{(k-1)}\,, \\[4pt]
\widetilde{X}'^{\,(k)}   &= \widetilde{X}^{(k)} + \big(B^{(k+1)}_{k}(z')\big)^{-1}\cdot Y^{(k+1)}\cdot B^{(k+1)}_k(z)\cdot\iota^{(k+1,k)} - \iota^{(k+1,k)}\cdot \widetilde{Y}^{(k)}\, .     
\end{aligned}     
\end{equation}

It follows from Lemma~\eqref{Lemma: system of equations for equivalent extensions at a crossing sigma_k}(b) that the system of equations in~\eqref{Eq: special system of equations on R_j with k geq 2} has a unique solution:
\begin{equation*}
\begin{aligned}
~\widetilde{Y}^{(k)}&= \,\text{principal $k\times k$ submatrix of } \,\big(B^{(n)}_{k}(z')\big)^{-1}\cdot D(\vec{u})\cdot B^{(n)}_{k}(z)\,  ,\\[4pt]  
\star&=\star\, ,\\[4pt]
~a'&=a + d(z',\vec{u}, z) \, ,
\end{aligned}    
\end{equation*}
where
\begin{equation*}
d(z',\vec{u}, z)= \,\text{the $(k+1,k)$-entry of } \,\big(B^{(n)}_{k}(z')\big)^{-1}\cdot D(\vec{u})\cdot B^{(n)}_{k}(z)\, .   
\end{equation*}
Specifically, we have that $d(z', \vec{u}, z)$ measures the redundancy in the choice of representatives for the equivalence classes in $\mathrm{Ext}^{1}(\sh{F},\sh{G})$ on the right of the crossing in $R_{j}$, and since redundancies vanish when restricting to $\mathrm{Ext}^{1}(\sh{F},\sh{G})$, we deduce that $d(z', \vec{u}, z)\equiv 0 $.

In particular, observe that
\begin{equation*}
\big(B^{(n)}_{k}(z')\big)^{-1}\cdot D(\vec{u})\cdot B^{(n)}_{k}(z)=D(\pi_{k}(\vec{u}))+d(z',\vec{u}, z)\,E_{k+1,k}\, ,    
\end{equation*}
where $E_{k+1,k}\in \mathrm{M}(n,\mathbb{K})$ denotes the elementary matrix with 1 in the position $(k+1,k)$ and zeros everywhere else. Moreover, since $Y^{(i)}$ is given by the principal $i\times i$ submatrix of $D(\vec{u})$ for all $i\in[1,n-1]$, the block structure of the braid matrices guarantees that:
\begin{itemize}
\justifying
\item For all $i\in [1, k-1]$, 
\begin{equation*}
Y^{(i)}= \text{ principal $i\times i$ submatrix of } ~\big(B^{(n)}_{k}(z')\big)^{-1}\cdot D(\vec{u})\cdot B^{(n)}_{k}(z)\, .    
\end{equation*}
\item For all $i\in [k+1, n-1]$, 
\begin{equation*}
\big(B^{(i)}_{k}(z')\big)^{-1} \cdot Y^{(i)} \cdot B^{(i)}_{k}(z)= \text{ principal $i\times i$ submatrix of } ~\big(B^{(n)}_{k}(z')\big)^{-1}\cdot D(\vec{u})\cdot B^{(n)}_{k}(z)\,  .    
\end{equation*}
\end{itemize}

Finally, note that, with respect to the bases $\big\{\hat{\mathbf{f}}^{(i)}[\sigma_{1},z]\big\}_{i=1}^{n}$ and $\big\{\hat{\mathbf{g}}^{(i)}[\sigma_{1},z']\big\}_{i=1}^{n}$, 
\begin{equation*}
\begin{aligned}
\tensor[_{\hat{\mathbf{g}}^{(i)}[\sigma_{k},z'] }]{ \big[\, T_{\lambda}^{\,(i)}\,\big] }{_{\hat{\mathbf{f}}^{(i)}[\sigma_{k},z]}}&= \tensor[_{\hat{\mathbf{g}}^{(i)}}]{ \big[\, T_{\lambda}^{(i)}\,\big] }{_{\hat{\mathbf{f}}^{(i)}}}\, ,    \quad \text{for all $i\in [1,k-1]$}\, ,\\[4pt]
\tensor[_{\hat{\mathbf{g}}^{(i)}[\sigma_{k},z'] }]{ \big[\, T_{\lambda}^{\,(i)}\,\big] }{_{\hat{\mathbf{f}}^{(i)}[\sigma_{k},z]}}&= \big(B^{(i)}_{k}(z')\big)^{-1} \cdot \tensor[_{\hat{\mathbf{g}}^{(i)}}]{ \big[\, T_{\lambda}^{(i)}\,\big] }{_{\hat{\mathbf{f}}^{(i)}}} \cdot B^{(i)}_{k}(z)\, ,    \quad \text{for all $i\in [k+1,n]$}\, .\\[4pt]  
\end{aligned}
\end{equation*}
Hence, when restricting to $\mathrm{Ext}^{1}(\sh{F},\sh{G})$, we obtain that
\begin{equation*}
\begin{aligned}
\tensor[_{\hat{\mathbf{g}}^{(n)}[\sigma_{k},z'] }]{ \big[\, T_{\lambda}^{\,(n)}\,\big] }{_{\hat{\mathbf{f}}^{(n)}[\sigma_{k},z]}} &\equiv D(\pi_{k}(\vec{u}))\,, \\[2pt]
\tensor[_{\hat{\mathbf{g}}^{(k)}[\sigma_{k},z'] }]{ \big[\, \widetilde{T}_{\lambda}^{\,(k)}\,\big] }{_{\hat{\mathbf{f}}^{(k)}[\sigma_{k},z]}}&\equiv \,\text{principal $k\times k$ submatrix of } D(\pi_{k}(\vec{u}))\, ,\\[2pt]
\tensor[_{\hat{\mathbf{g}}^{(i)}[\sigma_{k},z'] }]{ \big[\, T_{\lambda}^{\,(i)}\,\big] }{_{\hat{\mathbf{f}}^{(i)}[\sigma_{k},z]}}& \equiv \,\text{principal $i\times i$ submatrix of } D(\pi_{k}(\vec{u}))\, , \qquad \text{for all $i\in[1,n-1]\setminus \{k\}$}\, ,\\[4pt]
\end{aligned}
\end{equation*}
This completes the proof. 
\end{proof}

Having established the preceding lemmas, we are now in a position to state and prove the second part of our first main theorem.

\begin{theorem}\label{Theorem: Ext1 as the cokernel of delta_F,G}
\textbf{Setup}: Let $\beta=\sigma_{i_{1}}\cdots\sigma_{i_{\ell}}\in\mathrm{Br}^{+}_{n}$ be a positive braid word, $\mathcal{U}_{\Lambda(\beta)}:=\big\{U_{0}, U_{\mathrm{B}}, U_{\mathrm{L}}, U_{\mathrm{R}}, U_{\mathrm{T}}\big\}$ the open cover of $\mathbb{R}^{2}$ from Construction~\eqref{Cons: Finite open cover for R^2}, and $\sh{F}$ and $\sh{G}$ objects of the category $\ccs{1}{\beta}$. 

\noindent
$\star$ \emph{Local descriptions on $U_{\mathrm{T}}$ and $U_{\mathrm{L}}$}: According to Lemma~\eqref{Lemma: linear map description of an object on the regions U_T, U_L, and U_R}, $\sh{F}$ and $\sh{G}$ have the following local descriptions: 
\begin{itemize}
\justifying
\item  On $U_{\mathrm{T}}$, $\sh{F}$ and $\sh{G}$ are specified by two collections of $n-1$ surjective linear maps 
\begin{equation*}
\big\{\psi_{\sh{F}}^{(i)}:\mathbb{K}^{i+1}\to \mathbb{K}^{i}\big\}_{i=1}^{n-1}\, , \quad \text{and} \quad \big\{\psi_{\sh{G}}^{(i)}:\mathbb{K}^{i+1}\to \mathbb{K}^{i}\big\}_{i=1}^{n-1}\, ,    
\end{equation*}
respectively. For a schematic illustration of a generic representative of one of these sheaves on $U_{\mathrm{T}}$, see Figure~\eqref{Fig: an object F in the region U_T}.

\item On $U_{\mathrm{L}}$, $\sh{F}$ and $\sh{G}$ are specified by two collections of $n-1$ injective linear maps 
\begin{equation*}
\big\{\phi_{\sh{F}}^{(i)}:\mathbb{K}^{i}\to \mathbb{K}^{i+1}\big\}_{i=1}^{n-1}\, , \quad \text{and} \quad \big\{\phi_{\sh{G}}^{(i)}:\mathbb{K}^{i}\to \mathbb{K}^{i+1}\big\}_{i=1}^{n-1} \, ,    
\end{equation*}
respectively. For a schematic illustration of a generic representative of one of these sheaves on $U_{\mathrm{L}}$, see Figure~\eqref{Fig: an object F in the region U_L}.

\item \textbf{Compatibility conditions}: For each $i\in[1,n-1]$, 
\begin{equation*}
\psi^{(i)}_{\sh{F}}\circ \phi^{(i)}_{\sh{F}}=\mathrm{id}_{\mathbb{K}^{i}}\, , \quad \text{and} \quad \psi^{(i)}_{\sh{G}}\circ \phi^{(i)}_{\sh{G}}=\mathrm{id}_{\mathbb{K}^{i}}\, .    
\end{equation*}
\end{itemize}

\noindent
$\star$ \emph{Global flag data}: By Theorem~\eqref{Flags and constructible sheaves}, $\sh{F}$ and $\sh{G}$ are geometrically characterized by two sequence of complete flags $\big\{\fl{F}_{j}\big\}_{j=0}^{\ell+1}$ and $\big\{\fl{G}_{j}\big\}_{j=0}^{\ell+1}$ in $\mathbb{K}^{n}$, respectively, such that: 
\begin{itemize}
\item $\fl{F}_{0}$ is completely opposite to both $\fl{F}_{1}$ and $\fl{F}_{\ell+1}$, and for each $j\in[1,\ell]$, $\fl{F}_{j}$ is in $s_{i_{j}}$-relative position with respect to $\fl{F}_{j+1}$.
\item $\fl{G}_{0}$ is completely opposite to both $\fl{G}_{1}$ and $\fl{G}_{\ell+1}$, and for each $j\in[1,\ell]$, $\fl{G}_{j}$ is in $s_{i_{j}}$-relative position with respect to $\fl{G}_{j+1}$.
\end{itemize}
In particular, by Lemma~\eqref{lemma for F in the region U_{R}}, we know that:
\begin{equation*}
\begin{aligned}
\fl{F}_{0}:=\prescript{}{\mathcal{K}\,}{\fl{F}}\big(\psi_{\sh{F}}^{(1)},\dots, \psi_{\sh{F}}^{(n-1)}\big)\, , \quad \text{and} \quad
\fl{F}_{1}:=\prescript{}{\mathcal{I}\,}{\fl{F}}\big(\phi_{\sh{F}}^{(1)},\dots, \phi_{\sh{F}}^{(n-1)}\big)\, ,\\[6pt]    
\fl{G}_{0}:=\prescript{}{\mathcal{K}\,}{\fl{F}}\big(\psi_{\sh{G}}^{(1)},\dots, \psi_{\sh{G}}^{(n-1)}\big)\, , \quad \text{and} \quad
\fl{G}_{1}:=\prescript{}{\mathcal{I}\,}{\fl{F}}\big(\phi_{\sh{G}}^{(1)},\dots, \phi_{\sh{G}}^{(n-1)}\big)\, ,\\
\end{aligned}
\end{equation*}
are the type $\mathcal{K}$ and type $\mathcal{I}$ flags in $\mathbb{K}^{n}$ associated with $\big\{\psi_{\sh{F}}^{(i)}\big\}_{i=1}^{n-1}$, $\big\{\phi_{\sh{F}}^{(i)}\big\}_{i=1}^{n-1}$, $\big\{\psi_{\sh{G}}^{(i)}\big\}_{i=1}^{n-1}$, and $\big\{\phi_{\sh{G}}^{(i)}\big\}_{i=1}^{n-1}$, respectively (cf. Definition~\eqref{Def:flags and adapted bases}--\eqref{Def: type I flag}--\eqref{Def: type K flag}). 

\vspace{4pt}
\noindent
$\star$ \emph{Main assumption} (Adapted bases): Let \,$\hat{\mathbf{f}}^{(n)}, \, \hat{\mathbf{g}}^{(n)} $ be bases for $\mathbb{K}^{n}$, and let $\vec{z}=(z_{1},\dots, z_{\ell}), \; \vec{z}\,'=(z'_{1},\dots, z'_{\ell}) \in X(\beta,\mathbb{K})$ be points such that the pairs $(\,\hat{\mathbf{f}}^{(n)}, \,\vec{z}\, )$ and $(\,\hat{\mathbf{g}}^{(n)}, \,\vec{z}\,'\, )$ algebraically characterizes $\sh{F}$ and $\sh{G}$ according to Theorem~\eqref{Prop. for sheaves and braid matrices}, respectively. In this setting, we have that: 
\begin{itemize}
\item Relative to the basis $\hat{\mathbf{f}}^{(n)}$ for $\mathbb{K}^{n}$: $\fl{F}_{0}$ and $\fl{F}_{1}$ are the anti-standard and standard flags, respectively.  For each $j\in[1,\ell]$, the flag $\fl{F}_{j+1}$ is represented by the path matrix $P_{\beta_{j}}(\vec{z}_{j})=B^{(n)}_{i_{1}}(z_{1})\cdots B^{(n)}_{i_{j}}(z_{j})\,\in\mathrm{GL}(n,\mathbb{K})$ associated with the truncated braid word $\beta_{j}=\sigma_{i_{1}}\cdots \sigma_{i_{j}}\in\mathrm{Br}^{+}_{n}$ and the truncated tuple $\vec{z}_{j}=(z_{1},\dots, z_{j})\in\mathbb{K}^{j}_{\mathrm{std}}$.  

\item Relative to the basis $\hat{\mathbf{g}}^{(n)}$ for $\mathbb{K}^{n}$: $\fl{G}_{0}$ and $\fl{G}_{1}$ are the anti-standard and standard flags, respectively. For each $j\in[1,\ell]$, the flag $\fl{G}_{j+1}$ is represented by the path matrix $P_{\beta_{j}}(\vec{z}\,'_{j})=B^{(n)}_{i_{1}}(z'_{1})\cdots B^{(n)}_{i_{j}}(z'_{j})\,\in\mathrm{GL}(n,\mathbb{K})$ associated with the truncated braid word $\beta_{j}$ and the truncated tuple $\vec{z}\,'_{j}=(z'_{1},\dots, z'_{j})\in\mathbb{K}^{j}_{\mathrm{std}}$.  
\end{itemize}

Furthermore, under this assumption, the \textbf{compatibility conditions} and an inductive argument ensure that, for each $i\in [1,n-1]$, there are unique bases $\hat{\mathbf{f}}^{(i)}:=\big\{ \hat{f}^{(i)}_{j}\big\}_{j=1}^{i}$ and $\hat{\mathbf{g}}^{(i)}:=\big\{ \hat{g}^{(i)}_{j}\big\}_{j=1}^{i}$ for $\mathbb{K}^{i}$ such that (cf. Definition~\eqref{Def:flags and adapted bases}--\eqref{Def: adapted bases I}--\eqref{Def: adapted bases II}): 
\begin{itemize}
\item The collection $\big\{\hat{\mathbf{f}}^{(i)}\big\}_{i=1}^{n}$ is a system of bases adapted to both $\big\{\psi_{\sh{F}}^{(i)}\big\}_{i=1}^{n-1}$ and $\big\{\phi_{\sh{F}}^{(i)}\big\}_{i=1}^{n-1}$.
\item The collection $\big\{\hat{\mathbf{g}}^{(i)}\big\}_{i=1}^{n}$ is a system of bases adapted to both $\big\{\psi_{\sh{G}}^{(i)}\big\}_{i=1}^{n-1}$ and $\big\{\phi_{\sh{G}}^{(i)}\big\}_{i=1}^{n-1}$.
\end{itemize}
Building on this, for each $j\in[1,\ell]$, we denote by $\big\{\hat{\mathbf{f}}^{(i)}[\beta_{j},\vec{z}_{j} ]\big\}_{i=1}^{n}$ and $\big\{\hat{\mathbf{g}}^{(i)}[\beta_{j},\vec{z}\,'_{j} ]\big\}_{i=1}^{n}$ the braid-transformed bases obtained from $\big\{\hat{\mathbf{f}}^{(i)} \big\}_{i=1}^{n}$ and $\big\{\hat{\mathbf{g}}^{(i)} \big\}_{i=1}^{n}$ via the truncated braid word $\beta_{j}$ and the truncated tuples $\vec{z}_{j}$ and $\vec{z}\,'_{j}$, respectively (cf. Definition~\eqref{Def: braid transformation of bases}). Accordingly, by Theorem~\eqref{Theorem: adapted bases for an object at a region R_j}, we have that: 
\begin{itemize}
\justifying
\item $\big(\big\{\hat{\mathbf{f}}^{(i)} \big\}_{i=1}^{n}, z_{1}\big)$ is a system of bases adapted to $\sh{F}$ on $R_{1}$.
\item For each $j\in[1,\ell-1]$, $\big(\big\{\hat{\mathbf{f}}^{(i)}[\beta_{j},\vec{z}_{j} ]\big\}_{i=1}^{n}, z_{j+1}\big)$ is a system of bases adapted to $\sh{F}$ on $R_{j+1}$.
\item $\big(\big\{\hat{\mathbf{g}}^{(i)} \big\}_{i=1}^{n}, z'_{1}\big)$ is a system of bases adapted to $\sh{G}$ on $R_{1}$.
\item For each $j\in[1,\ell-1]$, $\big(\big\{\hat{\mathbf{g}}^{(i)}[\beta_{j},\vec{z}\,'_{j} ]\big\}_{i=1}^{n}, z'_{j+1}\big)$ is a system of bases adapted to $\sh{G}$ on $R_{j+1}$.
\end{itemize}

\vspace{4pt}
\noindent
$\star$ \emph{Main Conclusion}: Following Definition~\eqref{Def: linear map delta}, let \,$\delta_{\sh{F},\sh{G}}: \mathbb{K}^{n}_{\mathrm{std}}\to \mathbb{K}^{\ell}_{\mathrm{std}}$\, be the linear map associated with the pair $(\sh{F}, \sh{G})$. Then, under the given \textbf{setup}, there is an isomorphism of vector spaces 
\begin{equation*}
\mathrm{Ext}^{1}(\sh{F},\sh{G})\cong \mathrm{coker} \, \delta_{\sh{F},\sh{G}}\, .
\end{equation*}
\end{theorem}
\begin{proof}
To begin, let $\mathcal{S}_{\Lambda(\beta)}$ be the stratification of $\mathbb{R}^{2}$ induced by $\Lambda(\beta)\subset \big(\mathbb{R}^{3},\xi_{\mathrm{std}}\big)$, $\mathcal{U}_{\Lambda(\beta)}=\big\{U_{0}, U_{\mathrm{B}}, U_{\mathrm{L}}, U_{\mathrm{R}}, U_{\mathrm{T}}\big\}$ the open cover of $\mathbb{R}^{2}$ from Construction~\eqref{Cons: Finite open cover for R^2}, and $\mathcal{R}_{\Lambda(\beta)}=\big\{R_{j}\big\}_{j=1}^{\ell}$ the partition of $U_{\mathrm{B}}$ into $\ell$ open vertical straps from Construction~\eqref{Cons: Definition of the vertical straps}.    

Fix an equivalence class $\xi\in\mathrm{Ext}^{1}(\sh{F},\sh{G})$. Let $\sh{H}$ and $\sh{H}'$ be two equivalent extensions of $\sh{F}$ by $\sh{G}$ representing $\xi$, and let $\lambda:\sh{H}\to \sh{H}'$ be a sheaf isomorphism realizing their equivalence. 

Then, building on Observation~\eqref{Obs: elements of Ext1 at the local models}, we have that:     
\begin{itemize}
\justifying
\item On $U_{\mathrm{T}}$, $\sh{H}$ and $\sh{H}'$ are determined by two collections of $n-1$ characteristic maps
\begin{equation*}
\big\{ \psi^{(i)}_{\sh{H}}: \mathbb{K}^{i+1}\to \mathbb{K}^{i}\big\}_{i=1}^{n-1}\, , \quad \text{and} \quad \big\{ \psi^{(i)}_{\sh{H}'}: \mathbb{K}^{i+1}\to \mathbb{K}^{i}\big\}_{i=1}^{n-1}\, ,
\end{equation*}
while $\lambda$ is characterized by a collection of $n$ linear maps $\big\{ S^{(i)}_{\lambda}: \mathbb{K}^{i}\to \mathbb{K}^{i}\big\}_{i=1}^{n}$ such that
\begin{equation*}
 \psi^{(i)}_{\sh{H}'}= \psi^{(i)}_{\sh{H}} + S_{\lambda}^{(i)}\circ\psi^{(i)}_{\sh{F}}-\psi^{(i)}_{\sh{G}}\circ S^{(i+1)}_{\lambda}\, ,    
\end{equation*}
for all $i\in[1,n-1]$. 

\item On $U_{\mathrm{B}}$, $\sh{H}$ and $\sh{H}'$ are determined by two collections of $n-1$ characteristic maps
\begin{equation*}
\big\{ \phi^{(i)}_{\sh{H}}: \mathbb{K}^{i}\to \mathbb{K}^{i+1}\big\}_{i=1}^{n-1}\, , \quad \text{and} \quad \big\{ \phi^{(i)}_{\sh{H}'}: \mathbb{K}^{i}\to \mathbb{K}^{i+1}\big\}_{i=1}^{n-1}\, ,
\end{equation*}
while $\lambda$ is characterized by a collection of $n$ linear maps $\big\{ T^{(i)}_{\lambda}: \mathbb{K}^{i}\to \mathbb{K}^{i}\big\}_{i=1}^{n}$ such that
\begin{equation*}
 \phi^{(i)}_{\sh{H}'}= \phi^{(i)}_{\sh{H}} + T_{\lambda}^{(i+1)}\circ\phi^{(i)}_{\sh{F}}-\phi^{(i)}_{\sh{G}}\circ T^{(i)}_{\lambda}\, ,    
\end{equation*}
for all $i\in[1,n-1]$.  
\end{itemize}

According to Lemma~\eqref{Lemma: Ext1 on U_T and U_L}, any extension of $\sh{F}$ by $\sh{G}$ is locally equivalent to the trivial extension $\mathbf{0}_{\mathrm{Ext}}$ on $U_{\mathrm{T}}\;\cup\;U_{\mathrm{L}}$, and hence we choose representatives $\sh{H}$ and $\sh{H}'$ such that $\phi^{(i)}_{\sh{H}}= \phi^{(i)}_{\sh{H}'}= 0$ and $\psi^{(i)}_{\sh{H}}= \psi^{(i)}_{\sh{H}'}= 0$, for all $i\in [1,n-1]$. Under these conditions, Lemma~\eqref{Lemma: Ext1 on U_T and U_L} further asserts that the maps characterizing $\lambda$ on $U_{\mathrm{T}}\;\cup\;U_{\mathrm{L}}$ satisfy the following properties: 
\begin{itemize}
\item For all $i\in [1,n]$, $S^{(i)}_{\lambda}=T^{(i)}_{\lambda}$.
\item With respect to the bases $\big\{\hat{\mathbf{f}}^{(i)}\big\}_{i=1}^{n}$ and $\big\{\hat{\mathbf{g}}^{(i)}\big\}_{i=1}^{n}$, 
\begin{equation*}
\begin{aligned}
 \tensor[_{\hat{\mathbf{g}}^{(n)}}]{ \big[\, T_{\lambda}^{(n)}\,\big] }{_{\hat{\mathbf{f}}^{(n)}}} &\equiv D(\vec{u}),~\text{ for some $\vec{u}=(u_{1},\dots, u_{n})\in \mathbb{K}^{n}_{\mathrm{std}}$}\, ,\\[2pt]
\tensor[_{\hat{\mathbf{g}}^{(i)}}]{ \big[\, T_{\lambda}^{(i)}\,\big] }{_{\hat{\mathbf{f}}^{(i)}}}&\equiv \,\text{principal $i\times i$ submatrix of } D(\vec{u})\, , \quad\text{for all  $i\in[1,n-1]$}\, .
\end{aligned}    
\end{equation*}
\end{itemize}

Next, consider $R_{1}$, the vertical strap in $\mathbb{R}^{2}$ containing the first crossing of $\beta$, namely $\sigma_{i_{1}}$, and denote by $k=i_{1}\in [1,n-1]$ the index of $\sigma_{i_{1}}$. By construction, $R_{1}$ is the vertical strap immediately to the right of $U_{\mathrm{L}}$, and hence, the constructibility of $\sh{F}$ and $\sh{G}$ with respect to $\mathcal{S}_{\Lambda(\beta)}$ leads to one of the following two cases:
\begin{itemize}
\justifying
\item \textit{Case 1}: Suppose that $k=1$. Then, on $R_{1}$, $\sh{F}$ and $\sh{G}$ are characterized by two collections of $n$ injective linear maps: 
\begin{equation*}
\begin{aligned}
&\big\{\phi^{\,(i)}_{\sh{F}}:\mathbb{K}^{i}\to \mathbb{K}^{i+1}\,\big\}_{i=1}^{n-1}\,\cup\,\big\{\widetilde{\phi}^{\,(1)}_{\sh{F}}:\mathbb{K}^{1}\to \mathbb{K}^{2}\, \big\}\, ,\\[6pt]
&\big\{\phi^{\,(i)}_{\sh{G}}:\mathbb{K}^{i}\to \mathbb{K}^{i+1}\,\big\}_{i=1}^{n-1}\,\cup\,\big\{\widetilde{\phi}^{\,(1)}_{\sh{G}}:\mathbb{K}^{1}\to \mathbb{K}^{2}\, \big\} \, .
\end{aligned}    
\end{equation*}
respectively. Thus, in this local configuration, Lemma~\eqref{Lemma: Ext1 on R_j for k=1} asserts that, on $R_{1}$, $\sh{H}$ and $\sh{H}'$ are specified by two collections of $n$ characteristic maps
\begin{equation*}
\begin{aligned}
\big\{\phi^{(i)}_{\sh{H}}:\mathbb{K}^{i}\to \mathbb{K}^{i+1}\,\big\}_{i=1}^{n-1}\,&\cup\,\big\{ \widetilde{\phi}^{\,(1)}_{\sh{H}}:\mathbb{K}^{1}\to \mathbb{K}^{2}\, \big\}\, , \\[2pt]
\big\{\phi^{(i)}_{\sh{H}'}:\mathbb{K}^{i}\to \mathbb{K}^{i+1}\,\big\}_{i=1}^{n-1}\,&\cup\,\big\{ \widetilde{\phi}^{\,(1)}_{\sh{H}'}:\mathbb{K}^{1}\to \mathbb{K}^{2}\, \big\}\, , 
\end{aligned}
\end{equation*}
while $\lambda$ is characterized by a collection of $n+1$ linear maps $\big\{ T^{(i)}_{\lambda}: \mathbb{K}^{i}\to \mathbb{K}^{i} \big\}\;\cup\; \big\{ \widetilde{T}^{(1)}:\mathbb{K}^{1}\to \mathbb{K}^{1} \big\}$ such that
\begin{equation*}
\begin{aligned}
\phi^{(i)}_{\sh{H}'} &= \phi^{(i)}_{\sh{H}} + T^{(i+1)}_{\lambda}\circ \phi^{(i)}_{\sh{F}}- \phi^{(i)}_{\sh{G}}\circ T^{(i)}_{\lambda} \, , \quad \text{for all $i\in [1,n-1]$,} \\[2pt] 
\widetilde{\phi}^{\,(1)}_{\sh{H}'} &= \widetilde{\phi}^{\,(1)}_{\sh{H}} + T^{(2)}_{\lambda}\circ \widetilde{\phi}^{\,(1)}_{\sh{F}}- \widetilde{\phi}^{\,(1)}_{\sh{G}}\circ \widetilde{T}^{\,(1)}_{\lambda} \, .
\end{aligned}    
\end{equation*}

\item \textit{Case 2}: Suppose that $k\geq 2$. Then, on $R_{1}$, $\sh{F}$ and $\sh{G}$ are characterized by two collections of injective linear maps: 
\begin{equation*}
\begin{aligned}
&\big\{ \phi^{\,(i)}_{\sh{F}}:\mathbb{K}^{i}\to \mathbb{K}^{i+1}\,\big\}_{i=1}^{n-1}\,\cup\,\big\{ \widetilde{\phi}^{\,(k-1)}_{\sh{F}}:\mathbb{K}^{k-1}\to\mathbb{K}^{k}\,, ~\widetilde{\phi}^{\,(k)}_{\sh{F}}:\mathbb{K}^{k}\to \mathbb{K}^{k+1}\, \big\}\, ,\\[6pt]
&\big\{ \phi^{\,(i)}_{\sh{G}}:\mathbb{K}^{i}\to \mathbb{K}^{i+1}\,\big\}_{i=1}^{n-1}\,\cup\,\big\{ \widetilde{\phi}^{\,(k-1)}_{\sh{G}}:\mathbb{K}^{k-1}\to \mathbb{K}^{k}\,, ~\widetilde{\phi}^{\,(k)}_{\sh{G}}:\mathbb{K}^{k}\to \mathbb{K}^{k+1}\, \big\}\, .
\end{aligned}    
\end{equation*}
Thus, in this local configuration, Lemma~\eqref{Lemma: Ext1 on R_j for k geq 2} asserts that, on $R_{1}$, $\sh{H}$ and $\sh{H}'$ are specified by two collections of $n+1$ characteristic maps
\begin{equation*}
\begin{aligned}
\big\{ \phi^{(i)}_{\sh{H}}:\mathbb{K}^{i}\to \mathbb{K}^{i+1} \big\}_{i=1}^{n-1}&\,\cup\,\big\{ \widetilde{\phi}^{\,(k-1)}_{\sh{H}}:\mathbb{K}^{k-1}\to \mathbb{K}^{k}\,, ~ \widetilde{\phi}^{\,(k)}_{\sh{H}}:\mathbb{K}^{k}\to \mathbb{K}^{k+1} \big\}\, , \\[2pt]
\big\{ \phi^{(i)}_{\sh{H}'}:\mathbb{K}^{i}\to \mathbb{K}^{i+1} \big\}_{i=1}^{n-1}&\,\cup\,\big\{ \widetilde{\phi}^{\,(k-1)}_{\sh{H}'}:\mathbb{K}^{k-1}\to \mathbb{K}^{k}\,, ~ \widetilde{\phi}^{\,(k)}_{\sh{H}'}:\mathbb{K}^{k}\to \mathbb{K}^{k+1} \big\}\, ,  
\end{aligned}
\end{equation*}
while $\lambda$ is characterized by a collection of $n+1$ linear maps 
\begin{equation*}
\big\{ T^{(i)}_{\lambda}: \mathbb{K}^{i}\to \mathbb{K}^{i} \big\}_{i=1}^{n}\;\cup\; \big\{ \widetilde{T}^{\,(k)}:\mathbb{K}^{k}\to \mathbb{K}^{k} \big\}    
\end{equation*}
such that:
\begin{equation*}
\begin{aligned}
\phi^{(i)}_{\sh{H}'} &= \phi^{(i)}_{\sh{H}} + T^{(i+1)}_{\lambda}\circ \phi^{(i)}_{\sh{F}}- \phi^{(i)}_{\sh{G}}\circ T^{(i)}_{\lambda} \, , \qquad \text{for all $i\in [1,n-1]$}\, ,\\[4pt] 
\widetilde{\phi}^{\,(k)}_{\sh{H}'} &= \widetilde{\phi}^{\,(k)}_{\sh{H}} + T^{(k+1)}_{\lambda}\circ \widetilde{\phi}^{\,(k)}_{\sh{F}}- \widetilde{\phi}^{\,(k)}_{\sh{G}}\circ \widetilde{T}^{\,(k)}_{\lambda} \, , \\[4pt]
\widetilde{\phi}^{\,(k-1)}_{\sh{H}'} &= \widetilde{\phi}^{\,(k-1)}_{\sh{H}} + \widetilde{T}^{\,(k)}_{\lambda}\circ \widetilde{\phi}^{\,(k-1)}_{\sh{F}}- \widetilde{\phi}^{\,(k-1)}_{\sh{G}}\circ T^{\,(k-1)}_{\lambda} \, .    
\end{aligned}    
\end{equation*}
\end{itemize}
In any of the above cases, we have that $\big\{\phi^{(i)}_{\sh{F}}\big\}_{i=1}^{n-1}$, $\big\{\phi^{(i)}_{\sh{G}}\big\}_{i=1}^{n-1}$, $\big\{\phi^{(i)}_{\sh{H}}\big\}_{i=1}^{n-1}$, $\big\{\phi^{(i)}_{\sh{H}'}\big\}_{i=1}^{n-1}$, and $\big\{T^{(i)}_{\lambda}\big\}_{i=1}^{n}$ are precisely the linear maps determining $\sh{F}$, $\sh{G}$, $\sh{H}'$, $\sh{H}$, and $\lambda$ on $U_{\mathrm{L}}$, respectively. In particular, given our choice of representatives $\sh{H}$ and $\sh{H}'$, we have that $\phi^{(i)}_{\sh{H}}= \phi^{(i)}_{\sh{H}'}=0$, for all $i\in [1,n-1]$.

Now, let $\vec{z}=(z_{1},\dots, z_{\ell})$ and $\vec{z}\,'=(z'_{1},\dots, z'_{\ell})$ be the points in the braid variety $X(\beta,\mathbb{K})$ that algebraically characterize $\sh{F}$ and $\sh{G}$ in accordance with Theorem~\eqref{Prop. for sheaves and braid matrices}, respectively. By Theorem~\eqref{Theorem: adapted bases for an object at a region R_j}, we know that: 
\begin{itemize}
\justifying
\item $\big(\big\{\hat{\mathbf{f}}^{(i)} \big\}_{i=1}^{n}, z_{1}\big)$ is a system of bases adapted to $\sh{F}$ on $R_{1}$.
\item $\big(\big\{\hat{\mathbf{g}}^{(i)} \big\}_{i=1}^{n}, z'_{1}\big)$ is a system of bases adapted to $\sh{G}$ on $R_{1}$.
\end{itemize}
Then, a direct application of either Lemma~\eqref{Lemma: Ext1 on R_j for k=1} or Lemma~\eqref{Lemma: Ext0 for R_j for k geq 2} depending on the value of $k$, considering that $k=i_{1}$, implies that:
\begin{itemize}
\justifying    
\item \textbf{Parametrization of $\sh{H}$ and $\sh{H}'$ at $\sigma_{i_{1}}$}: On $R_{1}$, $\sh{H}$ and $\sh{H}'$ are determined by a pair of parameters $a_{1}, a_{1}' \in \mathbb{K}$, respectively, such that $a'_{1}=a_{1}+\delta_{1}(z_{1}', \vec{u}, z)$, where the  parameter controlling the gauge freedom in $\mathrm{Ext}^{1}(\sh{F}, \sh{G})$ after the crossing in $R_{1}$ is given by 
\begin{equation*}
\delta_{1}(z'_{1}, \vec{u}, z_{1}):=\Big[\,\big(B^{(n)}_{i_{1}}(z'_{1})\big)^{-1}\cdot D(\vec{u}) \cdot B^{(n)}_{i_{1}}(z_{1})\,\Big]_{i_{1}+1,i_{1}}\, . 
\end{equation*}
Hence, when restricting to $\mathrm{Ext}^{1}(\sh{F},\sh{G})$, we have that $\delta_{1}(z'_{1}, \vec{u}, z_{1})\equiv 0$.

\item \text{Block diagonal properties of $\lambda$ at $\sigma_{i_{1}}$}: With respect to the bases $\big\{\hat{\mathbf{f}}^{(i)}[\sigma_{i_{1}},z_{1}] \big\}_{i=1}^{n}$ and $\big\{\hat{\mathbf{g}}^{(i)}[\sigma_{i_{1}},z'_{1}] \big\}_{i=1}^{n}$, 
\begin{equation*}
\tensor[_{\hat{\mathbf{g}}^{(n)}[\sigma_{i_{1}},z'_{1}]}]{ \big[\, T_{\lambda}^{(n)}\,\big] }{_{\hat{\mathbf{f}}^{(n)}[\sigma_{i_{1}},z_{1}]}}\equiv D(\pi_{i_{1}}(\vec{u}))\,,    
\end{equation*}
where $\pi_{i_{1}}(\vec{u})$ denotes the permutation of $\vec{u}$ associated with $\sigma_{i_{1}}$. 
\end{itemize}

Applying the same reasoning to $R_{2}$, the vertical strap in $\mathbb{R}^{2}$ containing $\sigma_{i_{2}}$---the second crossing of $\beta$---we obtain that: 
\begin{itemize}
\justifying    
\item \textbf{Parametrization of $\sh{H}$ and $\sh{H}'$ at $\sigma_{i_{2}}$}: On $R_{2}$, $\sh{H}$ and $\sh{H}'$ are determined by a pair of tuples $\vec{a}_{2}:=(a_{1}, a_{2})\in \mathbb{K}^{2}$ and $\vec{a}\,'_{2}:=(a'_{1}, a'_{2}) \in \mathbb{K}^{2}$, respectively, such that $a'_{2}=a_{2}+ \delta_{2}(z'_{2}, \vec{u},z_{2})$,  where: 
\begin{itemize}
\item $a_{1}$ and $a'_{1}$ are the scalar parameterizing $\sh{H}$ and $\sh{H}'$ in $R_{1}$. 
\item The additional parameter controlling the gauge freedom in $\mathrm{Ext}^{1}(\sh{F}, \sh{G})$ after the crossing in $R_{2}$ is given by
\begin{equation*}
\delta_{2}(z'_{2}, \vec{u}, z_{2}):=\Big[\,\big(B^{(n)}_{i_{2}}(z'_{2})\big)^{-1}\cdot D(\pi_{1}(\vec{u})) \cdot B^{(n)}_{i_{2}}(z_{2})\,\Big]_{i_{2}+1, i_{2}}\, .  
\end{equation*}
\end{itemize}
Hence, when restricting to $\mathrm{Ext}^{1}(\sh{F},\sh{G})$, we have that $\delta_{2}(z'_{2}, \vec{u}, z_{2})\equiv 0$. 

\item \text{Block diagonal properties of $\lambda$ at $\sigma_{i_{2}}$}: With respect to $\big\{\hat{\mathbf{f}}^{(i)}[\beta_{2},\vec{z}_{2}] \big\}_{i=1}^{n}$ and $\big\{\hat{\mathbf{g}}^{(i)}[\beta_{2},\vec{z}\,'_{2}] \big\}_{i=1}^{n}$, 
\begin{equation*}
\begin{aligned}
\tensor[_{\hat{\mathbf{g}}^{(n)}[\beta_{2},\vec{z}\,'_{2}]}]{ \big[\, T_{\lambda}^{(n)}\,\big] }{_{\hat{\mathbf{f}}^{(n)}[\beta_{2},\vec{z}_{2}]}}&\equiv D(\pi_{i_{2}}(\pi_{i_{1}}(\vec{u})))\,,   \\[2pt]
&\equiv D(\pi_{\beta_{2}}(\vec{u}))\,, 
\end{aligned} 
\end{equation*}
where $\pi_{\beta_{2}}(\vec{u})$ denotes the permutation of $\vec{u}$ associated with the truncated braid word $\beta_{2}=\sigma_{i_{1}}\cdot \sigma_{i_{2}}$. 
\end{itemize}

Proceeding iteratively through the all vertical straps $R_{j}$, we obtain that $\sh{H}$ and $\sh{H}'$ are globally parametrized by tuples $\vec{a}=(a_{1},\dots, a_{\ell})\in \mathbb{K}^{\ell}_{\mathrm{std}}$ and $\vec{a}\,'=(a'_{1},\dots, a'_{\ell})\in \mathbb{K}^{\ell}_{\mathrm{std}}$ such that $a'_{j}=a_{j}+\delta_{j}(z'_{j}, \vec{u}, z_{j})$, where the gauge parameters are defined by
\begin{equation*}
\delta_{j}(z'_{j}, \vec{u}, z_{j}):=\Big[\, \big(B^{(n)}_{i_{j}}(z'_{j})\big)^{-1}\cdot D(\pi_{\beta_{j-1}}(\vec{u})) \cdot B^{(n)}_{i_{j}}(z_{j})\,\Big]_{i_{j}+1, i_{j}} \,,   
\end{equation*}
with $\pi_{\beta_{j-1}}(\vec{u})$ denoting the permutation of $\vec{u}$ associated with the truncated braid word $\beta_{j-1}=\sigma_{i_{1}}\cdots \sigma_{i_{j-1}}\in \mathrm{Br}^{+}_{n}$, for all $j\in [1,\ell]$. 

Finally, building on Definition~\eqref{Def: linear map delta}, we observe that two extensions $\sh{H}(\vec{a})$ and $\sh{H}'(\vec{a}\,')$ represent the same equivalence class in $\mathrm{Ext}^{1}(\sh{F},\sh{G})$ if and only if the tuples $\vec{a},\, \vec{a}\,'\in \mathbb{K}^{\ell}_{\mathrm{std}}$ that algebraically parametrize them satisfy $\vec{a}\,'-\vec{a}\in \mathrm{im}\,\delta_{\sh{F}, \sh{G}}\subseteq \mathbb{K}^{\ell}_{\mathrm{std}}$. Hence, we conclude that
\begin{equation*}
\mathrm{Ext}^{1}(\sh{F},\sh{G})\cong \mathrm{coker}\,\delta_{\sh{F},\sh{G}}\, .    
\end{equation*}
\end{proof}

As an immediate yet striking consequence of Theorems~\eqref{Theorem: Ext0 as the kernel of delta_F,G} and~\eqref{Theorem: Ext1 as the cokernel of delta_F,G}, we obtain the following result, which not only provides a powerful computational tool but also offers profound insight into the rich topological information that the category $\ccs{1}{\beta}$ carries about the Legendrian link $\Lambda(\beta)\subset (\mathbb{R}^{3}, \xi_{\mathrm{std}})$. 

\begin{theorem}\label{Theorem: Euler characteristic for Ext groups}
Let $\beta=\sigma_{i_{1}}\cdots\sigma_{i_{\ell}}\in\mathrm{Br}^{+}_{n}$ be a positive braid word, and let $\sh{F}$, $\sh{G}$ be objects of the category $\ccs{1}{\beta}$. Let \,$\hat{\mathbf{f}}^{(n)}$, $\hat{\mathbf{g}}^{(n)}$ be bases for $\mathbb{K}^{n}$, and let $\vec{z},\,\vec{z}\,'\in X(\beta,\mathbb{K})$ be points such that the pairs $\big(\,\hat{\mathbf{f}}^{(n)},\, \vec{z}\, \big)$ and $\big(\,\hat{\mathbf{g}}^{(n)},\, \vec{z}\,'\,\big)$ algebraically parametrize $\sh{F}$ and $\sh{G}$ according to Theorem~\eqref{Prop. for sheaves and braid matrices}, respectively. Then, the Poincar\'e--Euler characteristic of the graded vector space $\mathrm{Ext}^{\bullet}(\sh{F},\sh{G})$ is given by
\begin{equation*}
\chi\left( \mathrm{Ext}^{\bullet} (\sh{F},\sh{G})\right)=-\mathrm{tb}(\Lambda(\beta))\, ,     
\end{equation*}
where $\mathrm{tb}(\Lambda(\beta)):=\ell-n$ denotes the Thurston--Bennequin number of the Legendrian link $\Lambda(\beta)\subset (\mathbb{R}^{3}, \xi_{\mathrm{std}})$. 
\end{theorem}
\begin{proof}
To begin, following Definition~\eqref{Def: linear map delta}, let \,$\delta_{\sh{F},\sh{G}}: \mathbb{K}^{n}_{\mathrm{std}}\to\mathbb{K}^{\ell}_{\mathrm{std}}$\, be the linear map associated with the pair $(\sh{F},\sh{G})$. In particular, by the rank--nullity theorem, we know that  
\begin{equation*}
\mathrm{dim}_{\,\mathbb{K}}\, \mathrm{ker}\,\delta_{\sh{F},\sh{G}} -  \mathrm{dim}_{\,\mathbb{K}}\, \mathrm{coker}\,\delta_{\sh{F},\sh{G}} =n-\ell \, .    
\end{equation*}

Moreover, by Theorems~\eqref{Theorem: Ext0 as the kernel of delta_F,G} and~\eqref{Theorem: Ext1 as the cokernel of delta_F,G}, we have that the lower-degree morphism spaces between $\sh{F}$ and $\sh{G}$ are given by 
\begin{equation*}
\begin{aligned}
\mathrm{Ext}^{0}(\sh{F},\sh{G})&\cong \mathrm{ker}\,\delta_{\sh{F},\sh{G}}\, ,  \\    
\mathrm{Ext}^{1}(\sh{F},\sh{G})&\cong \mathrm{coker}\,\delta_{\sh{F},\sh{G}}\, . 
\end{aligned} 
\end{equation*}

Hence, putting the above results together, we obtain that 
\begin{equation*}
\begin{aligned}
\chi\left( \mathrm{Ext}^{\bullet} (\sh{F},\sh{G})\right) & = \sum_{i\in \mathbb{Z}} (-1)^{i}   \mathrm{dim}_{\,\mathbb{K}}\, \mathrm{Ext}^{i} (\sh{F},\sh{G})\, ,\\
& = \mathrm{dim}_{\,\mathbb{K}}\, \mathrm{Ext}^{0} (\sh{F},\sh{G})- \mathrm{dim}_{\,\mathbb{K}}\, \mathrm{Ext}^{1} (\sh{F},\sh{G})\, ,\\[4pt]
& = \mathrm{dim}_{\,\mathbb{K}}\, \mathrm{ker}\,\delta_{\sh{F},\sh{G}} - \mathrm{dim}_{\,\mathbb{K}}\, \mathrm{coker}\,\delta_{\sh{F},\sh{G}}\, ,\\[4pt]
&= n-\ell\, .
\end{aligned}
\end{equation*}
Finally, since the Thurston--Bennequin number of the Legendrian link $\Lambda(\beta)$ is given by $\mathrm{tb}(\Lambda(\beta))=\ell-n$~\cite{KT2}, the result follows. 
\end{proof}  

Having analyzed the linear structure of the graded morphism spaces in the category $\ccs{1}{\beta}$, we now turn to the study of their compositions. 

\subsection{\texorpdfstring{Combinatorial Composition Rules in the Category $\ccs{1}{\beta}$}{Combinatorial Composition Rule in the Category HSh}} Let $\beta \in \mathrm{Br}^{+}_{n}$ be a positive braid word. This subsection is devoted to studying the composition of graded morphisms in the category $\ccs{1}{\beta}$. More precisely, our main goal is to establish some combinatorial rules that describe how these graded morphisms compose. Bearing this in mind, we begin by introducing some preliminaries. 

\subsubsection{Technical Background} Let $\beta \in \mathrm{Br}^{+}_{n}$ be a positive braid word. Next, we now collect some preliminaries that will play a fundamental role in the explicit computation of the composition of the graded morphisms in the category $\ccs{1}{\beta}$. In particular, we open the discussion with the following definition.

\begin{definition}
Let $\mathcal{M}$ be a smooth manifold. Let $\sh{F}$, $\sh{F}'$, and $\sh{G}$ be sheaves of $\mathbb{K}$-modules on $\mathcal{M}$, and let $\lambda:\sh{F}'\to \sh{F}$ be a sheaf homomorphism between $\sh{F}'$ and $\sh{F}$. Suppose $\sh{H}$ is an extension of $\sh{F}$ by $\sh{G}$. Then, we define the pull-back $\lambda^{*}\sh{H} $ of $\sh{H}$ via $\lambda$ to be an extension of $\sh{F}'$ by $\sh{G}$ such that there exists a sheaf homomorphism $\alpha:\lambda^{*}\sh{H} \to \sh{H}$ making each square in the diagram in Figure~\eqref{Fig: Pull-back of a sheaf extension H via a sheaf homomorphism} commute. 

\begin{figure}[ht]
\centering
\begin{tikzpicture}
\useasboundingbox (-2,-1.5) rectangle (2,1.5);
\scope[transform canvas={scale=1.2}]

\node at (-4,1) {\footnotesize $0$};
\node at (-2,1) {\footnotesize $\sh{G}$};
\node at (0,1) {\footnotesize $\displaystyle \lambda^{*}\sh{H}$};
\node at (2,1) {\footnotesize $\sh{F}'$};
\node at (4,1) {\footnotesize $0$};

\node at (-4,-1) {\footnotesize $0$};
\node at (-2,-1) {\footnotesize $\sh{G}$};
\node at (0,-1) {\footnotesize $\sh{H}$};
\node at (2,-1) {\footnotesize $\sh{F}$};
\node at (4,-1) {\footnotesize $0$};

\draw[->] (-4+0.3,1) -- (-2-0.30,1);
\draw[->] (-2+0.3,1) -- (-0-0.30-0.15,1);
\draw[->] (0+0.3+0.15,1) -- (2-0.30,1);
\draw[->] (2+0.3,1) -- (4-0.30,1);

\draw[->] (-4+0.3,-1) -- (-2-0.30,-1);
\draw[->] (-2+0.3,-1) -- (-0-0.30,-1);
\draw[->] (0+0.3,-1) -- (2-0.30,-1);
\draw[->] (2+0.3,-1) -- (4-0.30,-1);

\draw[-] (-2-0.05,1-0.30) -- (-2-0.05,-1+0.30);
\draw[-] (-2+0.05,1-0.30) -- (-2+0.05,-1+0.30);

\draw[->] (2,1-0.30) -- (2,-1+0.30);
\node[right] at (2,0) {\footnotesize $\lambda$};

\draw[->] (0,1-0.30) -- (0,-1+0.30);
\node[right] at (0,0) {\footnotesize $\alpha$};

\endscope
\end{tikzpicture}
\caption{The pull-back $\lambda^{*}\sh{H}$ of an extension $\sh{H}$ of $\sh{F}$ by $\sh{G}$ via a sheaf homomorphism $\lambda$ between $\sh{F}'$ and $\sh{F}$.}
\label{Fig: Pull-back of a sheaf extension H via a sheaf homomorphism}
\end{figure}
\end{definition}

\begin{definition}
Let $\mathcal{M}$ be a smooth manifold. Let $\sh{F}$, $\sh{G}$, and $\sh{G}'$ be sheaves of $\mathbb{K}$-modules on $\mathcal{M}$, and let $\lambda:\sh{G}\to \sh{G}'$ be a sheaf homomorphism between $\sh{G}$ and $\sh{G}'$. Suppose $\sh{H}$ is an extension of $\sh{F}$ by $\sh{G}$. Then, we define the push-out $\lambda_{*}\sh{H} $ of $\sh{H}$ via $\lambda$ to be an extension of $\sh{F}$ by $\sh{G}'$ such that there exists a sheaf homomorphism $\alpha:\sh{H} \to \lambda_{*}\sh{H}$ making each square in the diagram in Figure~\eqref{Fig: Push-out of a sheaf extension H via a sheaf homomorphism} commute. 

\begin{figure}[ht]
\centering
\begin{tikzpicture}
\useasboundingbox (-2,-1.5) rectangle (2,1.5);
\scope[transform canvas={scale=1.2}]

\node at (-4,1) {\footnotesize $0$};
\node at (-2,1) {\footnotesize $\sh{G}$};
\node at (0,1) {\footnotesize $\sh{H}$};
\node at (2,1) {\footnotesize $\sh{F}$};
\node at (4,1) {\footnotesize $0$};

\node at (-4,-1) {\footnotesize $0$};
\node at (-2,-1) {\footnotesize $\sh{G}'$};
\node at (0,-1) {\footnotesize $\displaystyle \lambda_{*}\sh{H}$};
\node at (2,-1) {\footnotesize $\sh{F}$};
\node at (4,-1) {\footnotesize $0$};

\draw[->] (-4+0.3,1) -- (-2-0.30,1);
\draw[->] (-2+0.3,1) -- (-0-0.30,1);
\draw[->] (0+0.3,1) -- (2-0.30,1);
\draw[->] (2+0.3,1) -- (4-0.30,1);

\draw[->] (-4+0.3,-1) -- (-2-0.30,-1);
\draw[->] (-2+0.3,-1) -- (-0-0.30-0.15,-1);
\draw[->] (0+0.3+0.15,-1) -- (2-0.30,-1);
\draw[->] (2+0.3,-1) -- (4-0.30,-1);

\draw[-] (2-0.05,1-0.30) -- (2-0.05,-1+0.30);
\draw[-] (2+0.05,1-0.30) -- (2+0.05,-1+0.30);

\draw[->] (-2,1-0.30) -- (-2,-1+0.30);
\node[right] at (-2,0) {\footnotesize $\lambda$};

\draw[->] (0,1-0.30) -- (0,-1+0.30);
\node[right] at (0,0) {\footnotesize $\alpha$};

\endscope
\end{tikzpicture}
\caption{The push-out $\lambda_{*}\sh{H}$ of an extension $\sh{H}$ of $\sh{F}$ by $\sh{G}$ via a sheaf homomorphism $\lambda$ between $\sh{G}$ and $\sh{G}'$.}
\label{Fig: Push-out of a sheaf extension H via a sheaf homomorphism}
\end{figure}
\end{definition}

With these definitions in place, we now present a formal and concrete definition of the composition of graded morphisms in the category $\ccs{1}{\beta}$. 

\begin{definition}
Let $\beta=\sigma_{i_{1}}\cdots \sigma_{i_{\ell}}\in \mathrm{Br}^{+}_{n}$ be a positive braid word, and let $\sh{F}$, $\sh{G}$, $\sh{Q}$ be objects of the category $\ccs{1}{\beta}$. For any $p, q\geq 0$, 
the graded composition
\begin{equation*}
\circ: \mathrm{Ext}^{p}(\sh{G},\sh{Q}) \times \mathrm{Ext}^{q}(\sh{F}, \sh{G})  \to  \mathrm{Ext}^{p+q}(\sh{F},\sh{Q})\, ,
\end{equation*}
is given by the Yoneda composition for \,$\mathrm{Ext}$ groups~\cite{HS1}, which in degrees $0$ and $1$ is defined as follows:
\begin{itemize}
\justifying
\item[(i)] Let $\lambda \in \mathrm{Ext}^{0}(\sh{F}, \sh{G})$ and $\lambda' \in \mathrm{Ext}^{0}(\sh{G}, \sh{Q})$. Then $\lambda'\circ \lambda \in \mathrm{Ext}^{0}(\sh{F},\sh{Q})$ is given by the usual composition between sheaf homomorphisms. 

\item[(ii)] Let $\lambda \in \mathrm{Ext}^{0}(\sh{F}, \sh{G})$, $\xi\,' \in \mathrm{Ext}^{1}(\sh{G}, \sh{Q})$, and $\sh{H}'$ be an extension of $\sh{G}$ by $\sh{Q}$ representing the equivalence class $\xi\,'$, namely $\big[\,\sh{H}'\,\big]=\xi\,'$. Then $\xi\,' \circ \lambda = \big[\,\lambda^{*}\sh{H}'\,\big]\in \mathrm{Ext}^{1}(\sh{F}, \sh{Q})$. 

\item[(iii)] Let $\xi \in \mathrm{Ext}^{1}(\sh{F}, \sh{G})$, $\lambda' \in \mathrm{Ext}^{0}(\sh{G}, \sh{Q})$, and $\sh{H}$ be an extension of $\sh{F}$ by $\sh{G}$ representing the equivalence class $\xi$, namely $\big[\,\sh{H}\,\big]=\xi$. Then $\lambda '\circ \xi= \big[\,{\lambda'}_{*}\sh{H}\,\big]\in \mathrm{Ext}^{1}(\sh{F}, \sh{Q})$. 
\end{itemize}
\end{definition}

Next, we present several technical lemmas that will enable us to explicitly describe the composition of graded morphisms in the category $\ccs{1}{\beta}$. To this end, we introduce the following definitions.  

\begin{definition}\label{Def: pull-back of an extension via a morphism for linear maps}
Let $A$, $B$, $C$, $C'$, $X$, $Y$, $Z$, $Z'$ be vector spaces over $\mathbb{K}$, and let $\gamma_{\sh{G}}:A \rightarrow X$, $\Gamma_{\sh{H}}:B\rightarrow Y$, $\gamma_{\sh{F}}:C\rightarrow Z$, and $\gamma_{\sh{F}'}:C'\rightarrow Z'$ be linear maps, with  $\Gamma_{\sh{H}}$ an extension of $\gamma_{\sh{F}}$ by $\gamma_{\sh{G}}$ (see Definition~\eqref{Def: Extensions of linear maps}). 

Let $\lambda_{\xi}:C' \rightarrow C$ and $\lambda_{\delta}: Z'\rightarrow Z$ be linear maps defining a morphism $\lambda_{\xi,\delta}:=(\lambda_{\xi}, \lambda_{\delta})$ between $\gamma_{\sh{F}'}$ and $\gamma_{\sh{F}}$ (see Definition~\eqref{Def: morphism between linear maps}). Then, the pull-back ${\lambda_{\xi,\delta}}^{*}\Gamma_{\sh{H}}$ of $\Gamma_{\sh{H}}$ via $\lambda_{\xi,\delta}$ is an extension of $\gamma_{\sh{F}'}$ by $\gamma_{\sh{G}}$. Concretely, it is a linear map ${\lambda_{\xi,\delta}}^{*}\Gamma_{\sh{H}}:B'\rightarrow Y'$ between vector spaces $B'$ and $Y'$ over $\mathbb{K}$, such that there exist linear maps $W_{1}:B' \rightarrow B$ and $W_{2}: Y' \rightarrow Y$ making each square of the diagram in Figure~\eqref{fig: Commutative diagram for the pull-back of extensions of linear maps via morphisms of linear maps} commute.  
\end{definition}

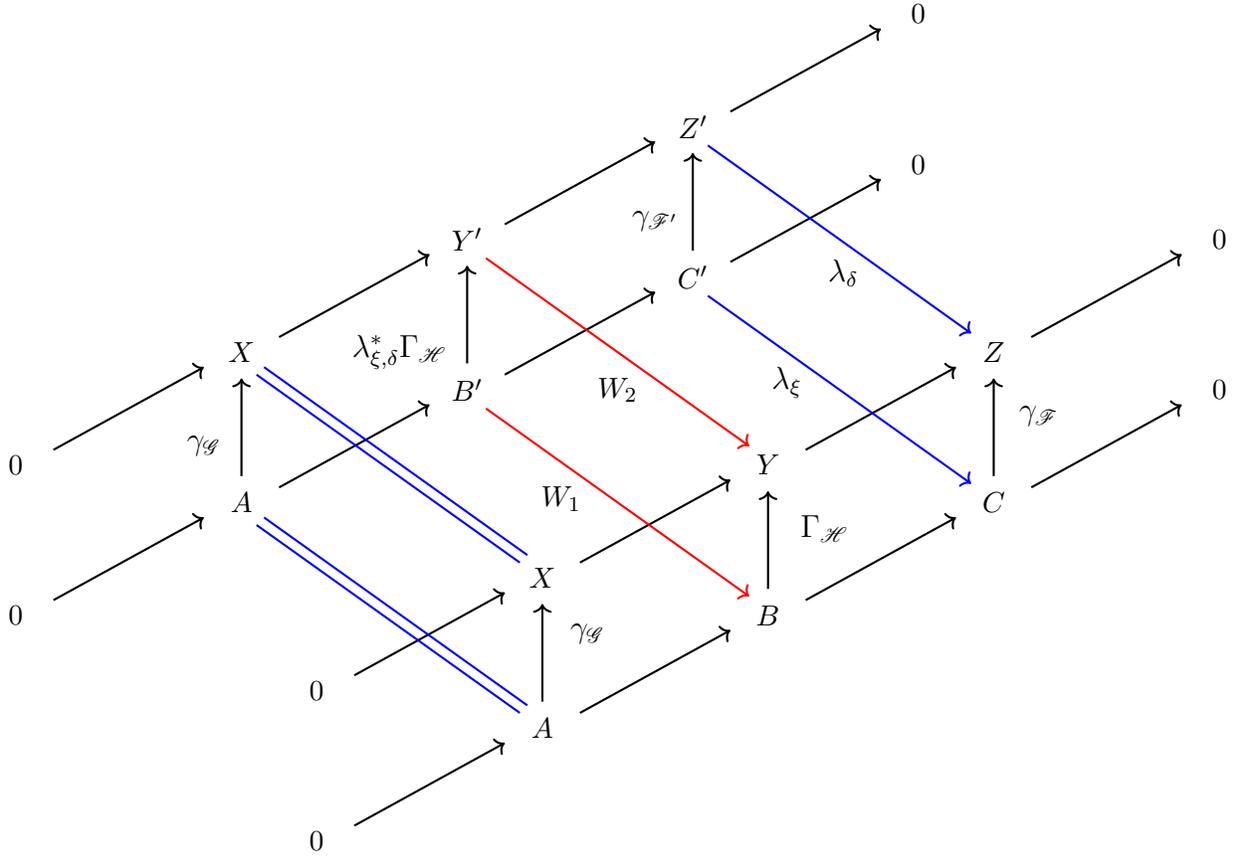
\begin{figure}[ht]
\centering
\begin{tikzpicture}
\useasboundingbox (-6,-5.5) rectangle (6,6.5);
\scope[transform canvas={scale=1}]

\node at (4,6) {\footnotesize\large $0$};
\node at (4,4) {\footnotesize\large $0$};

\node at (1,4.5) {\footnotesize\large $Z'$};
\node at (1,2.5) {\footnotesize\large $C'$};

\node at (-2,3) {\footnotesize\large $Y'$};
\node at (-2,1) {\footnotesize\large $B'$};

\node at (-5,1.5) {\footnotesize\large $X$};
\node at (-5,-0.5) {\footnotesize\large $A$};

\node at (-8,0) {\footnotesize\large $0$};
\node at (-8,-2) {\footnotesize\large $0$};

\node at (8,3) {\footnotesize\large $0$};
\node at (8,1) {\footnotesize\large $0$};

\node at (5,1.5) {\footnotesize\large $Z$};
\node at (5,-0.5) {\footnotesize\large $C$};

\node at (2,0) {\footnotesize\large $Y$};
\node at (2,-2) {\footnotesize\large $B$};

\node at (-1,-1.5) {\footnotesize\large $X$};
\node at (-1,-3.5) {\footnotesize\large $A$};

\node at (-4,-3) {\footnotesize\large $0$};
\node at (-4,-5) {\footnotesize\large $0$};

\draw[->, thick] (-8 +0.5,0+0.2) -- (-5-0.5,1.5-0.2);
\draw[->, thick] (-5 +0.5,1.5+0.2) -- (-2-0.5,3-0.2);
\draw[->, thick] (-2 +0.5,3+0.2) -- (1-0.5,4.5-0.2);
\draw[->, thick] (1 +0.5,4.5+0.2) -- (4-0.5,6-0.2);

\draw[->, thick] (-8 +0.5,-2+0.2) -- (-5-0.5,-0.5-0.2);
\draw[->, thick] (-5 +0.5,-0.5+0.2) -- (-2-0.5,1-0.2);
\draw[->, thick] (-2 +0.5,1+0.2) -- (1-0.5,2.5-0.2);
\draw[->, thick] (1 +0.5,2.5+0.2) -- (4-0.5,4-0.2);

\draw[->, thick] (-2,1+0.35) -- (-2,3-0.35);
\draw[->, thick] (1,2.5+0.35) -- (1,4.5-0.35);
\draw[->, thick] (-5,-0.5+0.35) -- (-5,1.5-0.35);

\draw[->, thick] (-4+0.5,-3+0.2) -- (-1-0.5,-1.5-0.2);
\draw[->, thick] (-1+0.5,-1.5+0.2) -- (2-0.5,0-0.2);
\draw[->, thick] (2+0.5,0+0.2) -- (5-0.5,1.5-0.2);
\draw[->, thick] (5+0.5,1.5+0.2) -- (8-0.5,3-0.2);

\draw[->, thick] (-4+0.5,-5+0.2) -- (-1-0.5,-3.5-0.2);
\draw[->, thick] (-1+0.5,-3.5+0.2) -- (2-0.5,-2-0.2);
\draw[->, thick] (2+0.5,-2+0.2) -- (5-0.5,-0.5-0.2);
\draw[->, thick] (5+0.5,-0.5+0.2) -- (8-0.5,1-0.2);

\draw[->, thick] (-1,-3.5+0.35) -- (-1,-1.5-0.35);
\draw[->, thick] (2,-2+0.35) -- (2,0-0.35);
\draw[->, thick] (5,-0.5+0.35) -- (5,1.5-0.35);

\draw[thick, blue] (-5+0.25-0.05, 1.5-0.25-0.05) -- (-1-0.25-0.05,-1.5+0.25-0.05);
\draw[thick, blue] (-5+0.25+0.05, 1.5-0.25+0.05) -- (-1-0.25+0.05,-1.5+0.25+0.05);

\draw[thick, blue] (-5+0.25-0.05, -0.5-0.25-0.05) -- (-1-0.25-0.05,-3.5+0.25-0.05);
\draw[thick, blue] (-5+0.25+0.05, -0.5-0.25+0.05) -- (-1-0.25+0.05,-3.5+0.25+0.05);

\draw[->, thick, red] (-2+0.25, 3-0.25) -- (2-0.25,0+0.25);
\draw[->, thick, red] (-2+0.25, 1-0.25) -- (2-0.25,-2+0.25);

\draw[->, thick, blue] (1+0.25-0.05, 4.5-0.25) -- (5-0.25-0.05,1.5+0.25);
\draw[->, thick, blue] (1+0.25-0.05, 2.5-0.25) -- (5-0.25-0.05,-0.5+0.25);

\node at (-0.75,-0.45) {\large $W_{1}$};
\node at (0,1) {\large $W_{2}$};

\node at (-0.75+3,-0.45+1.55) {\Large $\lambda_{\xi}$};
\node at (0+3,1+1.55) {\Large $\lambda_{\delta}$};

\node at (-5-0.5,0.5-0.25) {\Large $\gamma_{\sh{G}}$};
\node at (-2-0.9,2-0.5) {\Large $\lambda_{\xi,\delta}^{*}\Gamma_{\sh{H}}$};
\node at (1-0.5,3.5-0.25) {\Large $\gamma_{\sh{F}'}$};

\node at (-1+0.6,-2.5+0.25) {\Large $\gamma_{\sh{G}}$};
\node at (2+0.75,-1+0.15) {\Large $\Gamma_{\sh{H}}$};
\node at (5+0.6,0.5+0.15) {\Large $\gamma_{\sh{F}}$};

\endscope
\end{tikzpicture}
\caption{The pull-back $\lambda_{\xi,\delta}^{*}\Gamma_{\sh{H}}$ of the extension $\Gamma_{\sh{H}}$ of $\gamma_{\sh{F}}$ by $\gamma_{\sh{G}}$ via the morphism $\lambda_{\xi, \delta}$ between $\gamma_{\sh{F}'}$ and $\gamma_{\sh{F}}$.}
\label{fig: Commutative diagram for the pull-back of extensions of linear maps via morphisms of linear maps}
\end{figure}

\begin{definition}\label{Def: push-out of an extension via a morphism for linear maps}
Let $A$, $A'$, $B$, $C$, $X$, $X'$, $Y$, $Z$ be vector spaces over $\mathbb{K}$, and let $\gamma_{\sh{G}}:A \rightarrow X$, $\gamma_{\sh{G}'}:A' \rightarrow X'$, $\Gamma_{\sh{H}}:B\rightarrow Y$, and $\gamma_{\sh{F}}:C\rightarrow Z$ be linear maps, with  $\Gamma_{\sh{H}}$ an extension of $\gamma_{\sh{F}}$ by $\gamma_{\sh{G}}$ (see Definition~\eqref{Def: Extensions of linear maps}).

Let $\lambda_{\xi}:A \rightarrow A'$ and $\lambda_{\delta}: X\rightarrow X'$ be linear maps defining a morphism $\lambda_{\xi, \delta}:=(\lambda_{\xi},\lambda_{\delta})$ between $\gamma_{\sh{G}}$ and $\gamma_{\sh{G}'}$ (see Definition~\eqref{Def: morphism between linear maps}). Then, the push-out ${\lambda_{\xi,\delta}}_{*}\Gamma_{\sh{H}}$ of $\Gamma_{\sh{H}}$ via $\lambda_{\xi,\delta}$ is an extension of $\gamma_{\sh{F}}$ by $\gamma_{\sh{G}'}$. Concretely, it is a linear map ${\lambda_{\xi,\delta}}_{*}\Gamma_{\sh{H}}:B'\rightarrow Y'$ between vector spaces $B'$ and $Y'$ over $\mathbb{K}$, such that there exist linear maps $W_{1}:B \rightarrow B'$ and $W_{2}: Y \rightarrow Y'$ making each square of the diagram in Figure~\eqref{fig: Commutative diagram for the push-out of extensions of linear maps via morphisms of linear maps} commute. 
\end{definition}

\begin{figure}[ht]
\centering
\begin{tikzpicture}
\useasboundingbox (-6,-5.5) rectangle (6,6.5);
\scope[transform canvas={scale=1}]

\node at (4,6) {\footnotesize\large $0$};
\node at (4,4) {\footnotesize\large $0$};

\node at (1,4.5) {\footnotesize\large $Z$};
\node at (1,2.5) {\footnotesize\large $C$};

\node at (-2,3) {\footnotesize\large $Y$};
\node at (-2,1) {\footnotesize\large $B$};

\node at (-5,1.5) {\footnotesize\large $X$};
\node at (-5,-0.5) {\footnotesize\large $A$};

\node at (-8,0) {\footnotesize\large $0$};
\node at (-8,-2) {\footnotesize\large $0$};

\node at (8,3) {\footnotesize\large $0$};
\node at (8,1) {\footnotesize\large $0$};

\node at (5,1.5) {\footnotesize\large $Z$};
\node at (5,-0.5) {\footnotesize\large $C$};

\node at (2,0) {\footnotesize\large $Y'$};
\node at (2,-2) {\footnotesize\large $B'$};

\node at (-1,-1.5) {\footnotesize\large $X'$};
\node at (-1,-3.5) {\footnotesize\large $A'$};

\node at (-4,-3) {\footnotesize\large $0$};
\node at (-4,-5) {\footnotesize\large $0$};

\draw[->, thick] (-8 +0.5,0+0.2) -- (-5-0.5,1.5-0.2);
\draw[->, thick] (-5 +0.5,1.5+0.2) -- (-2-0.5,3-0.2);
\draw[->, thick] (-2 +0.5,3+0.2) -- (1-0.5,4.5-0.2);
\draw[->, thick] (1 +0.5,4.5+0.2) -- (4-0.5,6-0.2);

\draw[->, thick] (-8 +0.5,-2+0.2) -- (-5-0.5,-0.5-0.2);
\draw[->, thick] (-5 +0.5,-0.5+0.2) -- (-2-0.5,1-0.2);
\draw[->, thick] (-2 +0.5,1+0.2) -- (1-0.5,2.5-0.2);
\draw[->, thick] (1 +0.5,2.5+0.2) -- (4-0.5,4-0.2);

\draw[->, thick] (-2,1+0.35) -- (-2,3-0.35);
\draw[->, thick] (1,2.5+0.35) -- (1,4.5-0.35);
\draw[->, thick] (-5,-0.5+0.35) -- (-5,1.5-0.35);

\draw[->, thick] (-4+0.5,-3+0.2) -- (-1-0.5,-1.5-0.2);
\draw[->, thick] (-1+0.5,-1.5+0.2) -- (2-0.5,0-0.2);
\draw[->, thick] (2+0.5,0+0.2) -- (5-0.5,1.5-0.2);
\draw[->, thick] (5+0.5,1.5+0.2) -- (8-0.5,3-0.2);

\draw[->, thick] (-4+0.5,-5+0.2) -- (-1-0.5,-3.5-0.2);
\draw[->, thick] (-1+0.5,-3.5+0.2) -- (2-0.5,-2-0.2);
\draw[->, thick] (2+0.5,-2+0.2) -- (5-0.5,-0.5-0.2);
\draw[->, thick] (5+0.5,-0.5+0.2) -- (8-0.5,1-0.2);

\draw[->, thick] (-1,-3.5+0.35) -- (-1,-1.5-0.35);
\draw[->, thick] (2,-2+0.35) -- (2,0-0.35);
\draw[->, thick] (5,-0.5+0.35) -- (5,1.5-0.35);

\draw[->, thick, blue] (-5+0.25-0.05, 1.5-0.25) -- (-1-0.25-0.05,-1.5+0.25);
\draw[->, thick, blue] (-5+0.25-0.05, -0.5-0.25) -- (-1-0.25-0.05,-3.5+0.25);

\draw[->, thick, red] (-2+0.25, 3-0.25) -- (2-0.25,0+0.25);
\draw[->, thick, red] (-2+0.25, 1-0.25) -- (2-0.25,-2+0.25);

\draw[thick, blue] (1+0.25-0.05, 4.5-0.25-0.05) -- (5-0.25-0.05,1.5+0.25-0.05);
\draw[thick, blue] (1+0.25+0.05, 4.5-0.25+0.05) -- (5-0.25+0.05,1.5+0.25+0.05);

\draw[thick, blue] (1+0.25-0.05, 2.5-0.25-0.05) -- (5-0.25-0.05,-0.5+0.25-0.05);
\draw[thick, blue] (1+0.25+0.05, 2.5-0.25+0.05) -- (5-0.25+0.05,-0.5+0.25+0.05);

\node at (-0.75,-0.45) {\large $W_{1}$};
\node at (0,1) {\large $W_{2}$};

\node at (-0.75-3,-0.45-1.55) {\Large $\lambda_{\xi}$};
\node at (0-3,1-1.55) {\Large $\lambda_{\delta}$};

\node at (-5-0.5,0.5-0.25) {\Large $\gamma_{\sh{G}}$};
\node at (-2-0.5,2-0.2) {\Large $\Gamma_{\sh{H}}$};
\node at (1-0.5,3.5-0.25) {\Large $\gamma_{\sh{F}}$};

\node at (-1+0.6,-2.5+0.25) {\Large $\gamma_{\sh{G}'}$};
\node at (2+1,-1+0.2) {\Large ${\lambda_{\xi,\delta}}_{*}\Gamma_{\sh{H}}$};
\node at (5+0.6,0.5+0.15) {\Large $\gamma_{\sh{F}}$};

\endscope
\end{tikzpicture}
\caption{The push-out ${\lambda_{\xi,\delta}}_{*}\Gamma_{\sh{H}}$ of the extension $\Gamma_{\sh{H}}$ of $\gamma_{\sh{F}}$ by $\gamma_{\sh{G}}$ via the morphism $\lambda_{\xi, \delta}$ between $\gamma_{\sh{G}}$ and $\gamma_{\sh{G}'}$.}
\label{fig: Commutative diagram for the push-out of extensions of linear maps via morphisms of linear maps}
\end{figure}
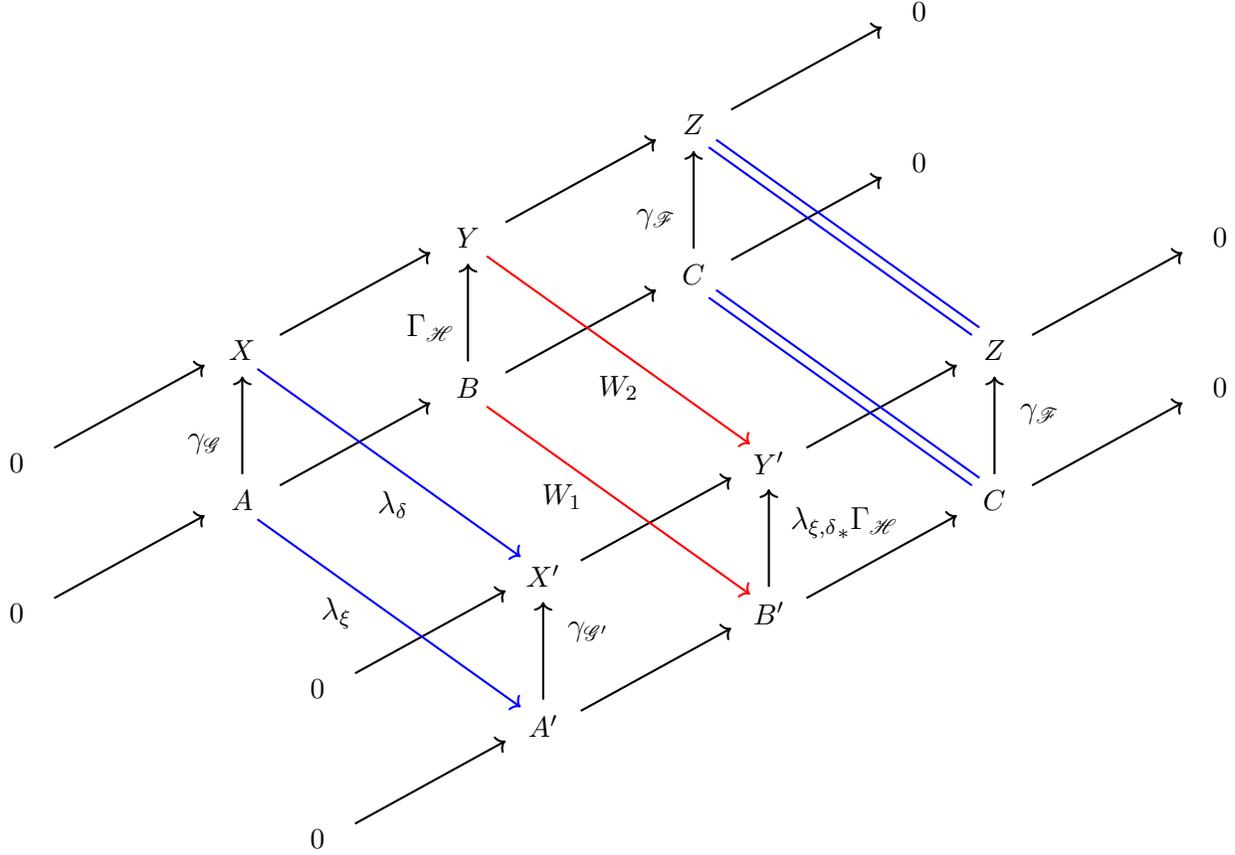

\begin{lemma}\label{Lemma: pull-back of an extension via a morphism for linear maps}
Let $A$, $C$, $C'$, $X$, $Z$, $Z'$ be vector spaces over $\mathbb{K}$, and let $\gamma_{\sh{G}}:A\rightarrow X$, $\gamma_{\sh{H}}:C\rightarrow X$, $\gamma_{\sh{F}}:C\rightarrow Z$, and $\gamma_{\sh{F}'}:C'\rightarrow Z'$ be linear maps. 

Let $\lambda_{\xi}:C' \rightarrow C$ and $\lambda_{\delta}: Z'\rightarrow Z$ be linear maps defining a morphism $\lambda_{\xi, \delta}:=(\lambda_{\xi},\lambda_{\delta})$ between $\gamma_{\sh{F}'}$ and $\gamma_{\sh{F}}$, and let $\elmG:A\oplus C\rightarrow X\oplus Z$ be the block extension of $\gamma_{\sh{F}}$ by $\gamma_{\sh{G}}$ associated with $\gamma_{\sh{H}}$ (see Lemma~\eqref{Lemma: block extensions of linear maps}). 

Finally, let $\lambda_{\xi,\delta}^{*}\,\elmG:B\rightarrow Y$ be the pull-back of $\elmG$ via $\lambda_{\xi,\delta}$, see Definition~\eqref{Def: pull-back of an extension via a morphism for linear maps}, and let $\elmg{\sh{F} }{\sh{G} }{\sh{Q}}:A\oplus C'\rightarrow X\oplus Z'$ be the block extension of $\gamma_{\sh{F}'}$ by $\gamma_{\sh{G}}$ associated with the linear map $\gamma_{\sh{Q}}:=\gamma_{\sh{H}}\circ \lambda_{\xi}: C'\to X$. Then, $\lambda_{\xi,\delta}^{*}\elmG$ and $\elmg{\sh{F} }{\sh{G} }{\sh{Q}}$ are equivalent extensions of $\gamma_{\sh{F}'}$ by $\gamma_{\sh{G}}$. 
\end{lemma}

\begin{proof}
By definition, $\lambda_{\xi,\delta}^{*}\elmG$ is an extension of $\gamma_{\sh{F}'}$ by $\gamma_{\sh{G}}$, and hence, Lemma~\eqref{Lemma: equivalence of extensions and block extensions} asserts that there exists a linear map $\gamma_{\sh{R}}:C'\rightarrow X$ such that $\lambda_{\xi,\delta}^{*}\elmG$ and $\elmg{\sh{F}'}{\sh{G}}{\sh{R}}$ are equivalent extensions of $\gamma_{\sh{F}'}$ by $\gamma_{\sh{G}}$, where $\elmg{\sh{F}'}{\sh{G}}{\sh{R}}:A\oplus C'\rightarrow X\oplus Z'$ denotes the block extension of $\gamma_{\sh{F}'}$ by $\gamma_{\sh{G}}$ associated with $\gamma_{\sh{R}}$. 

By Definitions~\eqref{Def: equivalent extensions of linear maps} and~\eqref{Def: pull-back of an extension via a morphism for linear maps}, there exist linear maps $N_{1}:A\oplus C'\rightarrow A\oplus C$ and $N_{2}:X\oplus Z'\rightarrow X\oplus Z$ making each square of the diagram in Figure~\eqref{fig: Commutative diagram for the equivalence of the pull-back of extensions via morphisms for linear maps} commute. In particular, observe that there exist linear maps $\lambda^{(1)}_{N_{1}}:A\rightarrow A$,  $\lambda^{(2)}_{N_{1}}:C'\rightarrow A$,  $\lambda^{(3)}_{N_{1}}:A\rightarrow C$, and $\lambda^{(4)}_{N_{1}}:C'\rightarrow C$ such that $N_{1}$ can be decomposed as
\begin{equation*}
N_{1}=\begin{bmatrix}
\lambda^{(1)}_{N_{1}} & \lambda^{(2)}_{N_{1}}\\
\lambda^{(3)}_{N_{1}} & \lambda^{(4)}_{N_{1}}
\end{bmatrix}\, .
\end{equation*}
Similarly, there exist linear maps $\lambda^{(1)}_{N_{2}}:X\rightarrow X$,  $\lambda^{(2)}_{N_{2}}:Z'\rightarrow X$,  $\lambda^{(3)}_{N_{2}}:X\rightarrow Z$, and $\lambda^{(4)}_{N_{2}}:Z'\rightarrow Z$ such that $N_{2}$ can be expressed as
\begin{equation*}
N_{2}=\begin{bmatrix}
\lambda^{(1)}_{N_{2}} & \lambda^{(2)}_{N_{2}}\\
\lambda^{(3)}_{N_{2}} & \lambda^{(4)}_{N_{2}}
\end{bmatrix}\, .  
\end{equation*}

It then follows from the commutativity of the diagram in Figure~\eqref{fig: Commutative diagram for the equivalence of the pull-back of extensions via morphisms for linear maps} that
\begin{align*}
\lambda^{(1)}_{N_{1}}&=\mathrm{id}_{A}\, , \\
\lambda^{(3)}_{N_{1}}&=0\, , \\
\lambda^{(4)}_{N_{1}}&=\lambda_{\xi}\, , \\
\lambda^{(1)}_{N_{2}}&=\mathrm{id}_{X}\, , \\
\lambda^{(3)}_{N_{2}}&=0\, , \\
\lambda^{(4)}_{N_{2}}&=\lambda_{\delta}\, , \\
\lambda_{\delta}\circ \gamma_{\sh{F}'}&=\gamma_{\sh{F}}\circ \lambda_{\xi}\, , \\
\gamma_{\sh{R}}+\lambda^{(2)}_{N_{2}}\circ \gamma_{\sh{F}'}&=\gamma_{\sh{H}}\circ \lambda_{\xi}+\gamma_{\sh{G}}\circ \lambda^{(2)}_{N_{1}}\, .
\end{align*}    
Thus, defining $\gamma_{\sh{Q}}=\gamma_{\sh{H}}\circ \lambda_{\xi}:C'\rightarrow X$, the above relations show that there exist linear maps $\lambda^{(2)}_{N_{1}}:C'\rightarrow A$ and $\lambda^{(2)}_{N_{2}}:Z'\rightarrow X$ such that
\begin{equation*}
\gamma_{\sh{Q}} = \gamma_{\sh{R}}+\lambda^{(2)}_{N_{2}}\circ \gamma_{\sh{F}'} - \gamma_{\sh{G}}\circ \lambda^{(2)}_{N_{1}}.    
\end{equation*}
Hence, building on Lemma~\eqref{Lemma: equivalence of extensions and block extensions}, we conclude that $\elmg{\sh{F}'}{\sh{G}}{\sh{R}}$ and $\elmg{\sh{F}'}{\sh{G}}{\sh{Q}}$ are equivalent extensions of $\gamma_{\sh{F}'}$ by $\gamma_{\sh{G}}$, where $\elmg{F'}{G}{Q}:A\oplus C' \rightarrow X\oplus Z'$ denotes the block extension of $\gamma_{\sh{F}'}$ by $\gamma_{\sh{G}}$ associated with $\gamma_{\sh{Q}}$. Finally, by transitivity of equivalence, we deduce that $\lambda_{\xi,\delta}^{*}\elmG$ and $\elmg{\sh{F}'}{\sh{G}}{\sh{Q}}$ are equivalent extensions of $\gamma_{\sh{F}'}$ by $\gamma_{\sh{G}}$. This completes the proof. 

\begin{figure}[ht]
\centering
\begin{tikzpicture}
\useasboundingbox (-6,-5.5) rectangle (6,6.5);
\scope[transform canvas={scale=1}]

\node at (8-4,3+3) {\footnotesize\large $0$};
\node at (8-4,1+3) {\footnotesize\large $0$};

\node at (5-4,1.5+3) {\footnotesize\large $Z'$};
\node at (5-4,-0.5+3) {\footnotesize\large $C'$};

\node at (2-4,0+3) {\footnotesize\large $X\oplus Z'$};
\node at (2-4,-2+3) {\footnotesize\large $A\oplus C'$};

\node at (-1-4,-1.5+3) {\footnotesize\large $X$};
\node at (-1-4,-3.5+3) {\footnotesize\large $A$};

\node at (-4-4,-3+3) {\footnotesize\large $0$};
\node at (-4-4,-5+3) {\footnotesize\large $0$};

\node at (8,3) {\footnotesize\large $0$};
\node at (8,1) {\footnotesize\large $0$};

\node at (5,1.5) {\footnotesize\large $Z$};
\node at (5,-0.5) {\footnotesize\large $C$};

\node at (2,0) {\footnotesize\large $X\oplus Z$};
\node at (2,-2) {\footnotesize\large $A\oplus C$};

\node at (-1,-1.5) {\footnotesize\large $X$};
\node at (-1,-3.5) {\footnotesize\large $A$};

\node at (-4,-3) {\footnotesize\large $0$};
\node at (-4,-5) {\footnotesize\large $0$};

\draw[->, thick] (-8 +0.5,0+0.2) -- (-5-0.5,1.5-0.2);
\draw[->, thick] (-5 +0.5,1.5+0.15) -- (-2-0.5-0.15,3-0.2-0.13);
\draw[->, thick] (-2 +0.5+0.2,3+0.2+0.125) -- (1-0.5,4.5-0.2);
\draw[->, thick] (1 +0.5,4.5+0.2) -- (4-0.5,6-0.2);

\draw[->, thick] (-8 +0.5,-2+0.2) -- (-5-0.5,-0.5-0.2);
\draw[->, thick] (-5 +0.5,-0.5+0.15) -- (-2-0.5-0.15,1-0.2-0.11);
\draw[->, thick] (-2 +0.5+0.2,1+0.2+0.125) -- (1-0.5,2.5-0.2);
\draw[->, thick] (1 +0.5,2.5+0.2) -- (4-0.5,4-0.2);

\draw[->, thick] (-2,1+0.35) -- (-2,3-0.35);
\draw[->, thick] (1,2.5+0.35) -- (1,4.5-0.35);
\draw[->, thick] (-5,-0.5+0.35) -- (-5,1.5-0.35);

\draw[->, thick] (-4+0.5,-3+0.2) -- (-1-0.5,-1.5-0.2);
\draw[->, thick] (-1+0.5,-1.5+0.15) -- (2-0.5-0.15,0-0.2-0.13);
\draw[->, thick] (2+0.5+0.2,0+0.2+0.125) -- (5-0.5,1.5-0.2);
\draw[->, thick] (5+0.5,1.5+0.2) -- (8-0.5,3-0.2);

\draw[->, thick] (-4+0.5,-5+0.2) -- (-1-0.5,-3.5-0.2);
\draw[->, thick] (-1+0.5,-3.5+0.15) -- (2-0.5-0.2,-2-0.2-0.125);
\draw[->, thick] (2+0.5+0.2,-2+0.2+0.125) -- (5-0.5,-0.5-0.2);
\draw[->, thick] (5+0.5,-0.5+0.2) -- (8-0.5,1-0.2);

\draw[->, thick] (-1,-3.5+0.35) -- (-1,-1.5-0.35);
\draw[->, thick] (2,-2+0.35) -- (2,0-0.35);
\draw[->, thick] (5,-0.5+0.35) -- (5,1.5-0.35);

\draw[thick, blue] (-5+0.25-0.035, 1.5-0.25-0.035) -- (-1-0.25-0.035,-1.5+0.25-0.035);
\draw[thick, blue] (-5+0.25+0.035, 1.5-0.25+0.035) -- (-1-0.25+0.035,-1.5+0.25+0.035);

\draw[thick, blue] (-5+0.25-0.035, -0.5-0.25-0.035) -- (-1-0.25-0.035,-3.5+0.25-0.035);
\draw[thick, blue] (-5+0.25+0.035, -0.5-0.25+0.035) -- (-1-0.25+0.035,-3.5+0.25+0.035);

\draw[->, thick, red] (-2+0.25+0.25-0.3, 3-0.25-0.0) -- (2-0.25-0.1+0.3,0+0.25+0.15-0.15);
\draw[->, thick, red] (-2+0.25+0.25-0.3, 1-0.25-0.0) -- (2-0.25-0.1+0.2,-2+0.25+0.15-0.15);

\draw[->, thick, blue] (1+0.25+0.05, 4.5-0.25) -- (5-0.25+0.05,1.5+0.25);
\draw[->, thick, blue] (1+0.25+0.05, 2.5-0.25) -- (5-0.25+0.05,-0.5+0.25);

\node at (-0.75,-0.45) {\normalsize $N_{1}$};
\node at (0,1) {\normalsize$N_{2}$};

\node at (-5-0.5,0.5-0.25) {\normalsize $\gamma_{\sh{G}}$};
\node at (-2-0.75-0.1,2-0.25) {\normalsize $\elmg{F'}{G}{R}$};
\node at (1-0.5,3.5-0.25) {\normalsize$\gamma_{\sh{F'}}$};

\node at (-1+0.6,-2.5+0.25) {\normalsize $\gamma_{\sh{G}}$};
\node at (2+0.8+0.1,-1+0.25 ) {\normalsize $\elmG$};
\node at (5+0.6,0.5+0.15) {\normalsize $\gamma_{\sh{F}}$};

\node at (-4+0.6,0.5+0.45) {\normalsize $\iota_{A\oplus C'}$};
\node at (-1+0.6,1.95+0.45) {\normalsize $\rho_{A\oplus C'}$};
\node at (-4+0.3,2.5+0.3) {\normalsize $\iota_{X\oplus Z'}$};
\node at (-1+0.3,3.95+0.3) {\normalsize $\rho_{X\oplus Z'}$};

\node at (-4+0.6+4.25,0.5+0.05-3.75) {\normalsize $\iota_{A\oplus C}$};
\node at (-1+0.6+4.25,1.95+0.2-3.75) {\normalsize $\rho_{A\oplus C}$};
\node at (-4+0.6+3.75,0.5+0.05-1.95) {\normalsize $\iota_{X\oplus Z}$};
\node at (-1+0.6+3.75,1.95+0.2-1.95) {\normalsize $\rho_{X\oplus Z}$};

\node at (-0.75+3,-0.45+1.55) {\normalsize $\lambda_{\xi}$};
\node at (0+3,1+1.55) {\normalsize $\lambda_{\delta}$};

\endscope
\end{tikzpicture}
\caption{Commutative diagram between the block extensions $\elmg{F'}{G}{R}$ and $\elmG$.}
\label{fig: Commutative diagram for the equivalence of the pull-back of extensions via morphisms for linear maps}
\end{figure}
\end{proof}

\begin{lemma}\label{Lemma: push-out of an extension via a morphism for linear maps}
Let $A$, $A'$, $C$, $X$, $X'$, $Z$ be vector spaces over $\mathbb{K}$, and let $\gamma_{\sh{G}}:A\rightarrow X$, $\gamma_{\sh{G}'}:A'\rightarrow X'$, $\gamma_{\sh{H}}:C\rightarrow X$, and $\gamma_{\sh{F}}:C\rightarrow Z$ be linear maps. 

Let $\lambda_{\xi}:A \rightarrow A'$ and $\lambda_{\delta}: X\rightarrow X'$ be linear maps defining a morphism $\lambda_{\xi, \delta}=(\lambda_{\xi},\lambda_{\delta})$ between $\gamma_{\sh{G}}$ and $\gamma_{\sh{G}'}$, and let $\elmG:A\oplus C\rightarrow X\oplus Z$ be the block extension of $\gamma_{\sh{F}}$ by $\gamma_{\sh{G}}$ associated with $\gamma_{\sh{H}}$ (see Lemma~\eqref{Lemma: block extensions of linear maps}). 

Finally, let ${\lambda_{\xi,\delta}}_{*}\elmG:B\rightarrow Y$ be the push-out of $\elmG$ via $\lambda_{\xi,\delta}$, see Definition~\eqref{Def: push-out of an extension via a morphism for linear maps}, and let $\elmg{\sh{F}}{\sh{G}'}{\sh{Q}}:A'\oplus C\rightarrow X'\oplus Z$ be the block extension of $\gamma_{\sh{F}}$ by $\gamma_{\sh{G}'}$ associated with the linear map $\gamma_{\sh{Q}}:=\lambda_{\delta}\circ \gamma_{\sh{H}}: C \to X'$. Then,  ${\lambda_{\xi,\delta}}_{*}\elmG$ and $\elmg{\sh{F}}{\sh{G}'}{\sh{Q}}$ are equivalent extensions of $\gamma_{\sh{F}}$ by $\gamma_{\sh{G}'}$. 
\end{lemma}

\begin{proof}
By definition, ${\lambda_{\xi,\delta}}_{*}\elmG$ is an extension of $\gamma_{\sh{F}}$ by $\gamma_{\sh{G}'}$, and hence, Lemma~\eqref{Lemma: equivalence of extensions and block extensions} asserts that there exists a linear map $\gamma_{\sh{R}}:C\rightarrow X'$ such that ${\lambda_{\xi,\delta}}_{*}\elmG$ and $\elmg{\sh{F}}{\sh{G}'}{\sh{R}}$ are equivalent extensions of $\gamma_{\sh{F}}$ by $\gamma_{\sh{G}'}$, where $\elmg{\sh{F}}{\sh{G}'}{\sh{R}}:A'\oplus C\rightarrow X'\oplus Z$ denotes the block extension of $\gamma_{\sh{F}}$ by $\gamma_{\sh{G}'}$ associated with $\gamma_{\sh{R}}$.

By Definitions~\eqref{Def: equivalent extensions of linear maps} and~\eqref{Def: push-out of an extension via a morphism for linear maps}, there exist linear maps $N_{1}:A\oplus C\rightarrow A'\oplus C$ and $N_{2}:X\oplus Z\rightarrow X'\oplus Z$ making each square of the diagram in Figure~\eqref{fig: Commutative diagram for the equivalence of the push-out of extensions via morphisms for linear maps} commute. In particular, observe that there exist linear maps $\lambda^{(1)}_{N_{1}}:A\rightarrow A'$,  $\lambda^{(2)}_{N_{1}}:C\rightarrow A'$,  $\lambda^{(3)}_{N_{1}}:A\rightarrow C$, and $\lambda^{(4)}_{N_{1}}:C\rightarrow C$ such that $N_{1}$ can be decomposed as
\begin{equation*}
N_{1}=\begin{bmatrix}
\lambda^{(1)}_{N_{1}} & \lambda^{(2)}_{N_{1}}\\
\lambda^{(3)}_{N_{1}} & \lambda^{(4)}_{N_{1}}
\end{bmatrix}\, .
\end{equation*}
Similarly, there exist linear maps $\lambda^{(1)}_{N_{2}}:X\rightarrow X'$,  $\lambda^{(2)}_{N_{2}}:Z\rightarrow X'$,  $\lambda^{(3)}_{N_{2}}:X\rightarrow Z$, and $\lambda^{(4)}_{N_{2}}:Z\rightarrow Z$  such that $N_{2}$ can be expressed as
\begin{equation*}
N_{2}=\begin{bmatrix}
\lambda^{(1)}_{N_{2}} & \lambda^{(2)}_{N_{2}}\\
\lambda^{(3)}_{N_{2}} & \lambda^{(4)}_{N_{2}}
\end{bmatrix}\, .
\end{equation*}

It then follows from the commutativity of the diagram in Figure~\eqref{fig: Commutative diagram for the equivalence of the push-out of extensions via morphisms for linear maps} that 
\begin{align*}
\lambda^{(1)}_{N_{1}}&=\lambda_{\xi}\, , \\
\lambda^{(3)}_{N_{1}}&=0\, , \\
\lambda^{(4)}_{N_{1}}&=\mathrm{id}_{C}\, , \\
\lambda^{(1)}_{N_{2}}&=\lambda_{\delta}\, , \\
\lambda^{(3)}_{N_{2}}&=0\, , \\
\lambda^{(4)}_{N_{2}}&=\mathrm{id}_{Z}\, , \\
\lambda_{\delta}\circ \gamma_{\sh{G}}&=\gamma_{\sh{G}'}\circ \lambda_{\xi}\, , \\
\gamma_{\sh{R}} + \gamma_{\sh{G}'}\circ \lambda^{(2)}_{N_{1}} &= \lambda_{\delta}\circ\gamma_{\sh{H}}+\lambda^{(2)}_{N_{2}}\circ \gamma_{\sh{F}}\, .
\end{align*}    
Thus, defining $\gamma_{\sh{Q}}:=\lambda_{\delta}\circ\gamma_{\sh{H}}:C\rightarrow X'$, the above relations show that there exist linear maps $\lambda^{(2)}_{N_{1}}:C\rightarrow A'$ and $\lambda^{(2)}_{N_{2}}:Z\rightarrow X'$ such that
\begin{equation*}
\gamma_{\sh{R}}= \gamma_{\sh{Q}}+\lambda^{(2)}_{N_{2}}\circ \gamma_{\sh{F}} - \gamma_{\sh{G}'}\circ \lambda^{(2)}_{N_{1}}.    
\end{equation*}
Hence, building on Lemma~\eqref{Lemma: equivalence between block extensions}, we conclude that $\elmg{\sh{F}}{\sh{G}'}{\sh{Q}}$ and $\elmg{\sh{F}}{\sh{G}'}{\sh{R}}$ are equivalent extensions of $\gamma_{\sh{F}}$ by $\gamma_{\sh{G}'}$, where $\elmg{\sh{F}}{\sh{G}'}{\sh{Q}}:A'\oplus C \rightarrow X'\oplus Z$ denotes the block extension of $\gamma_{\sh{F}}$ by $\gamma_{\sh{G}'}$ associated with $\gamma_{\sh{Q}}$. Finally, by transitivity of equivalence, we deduce that ${\lambda_{\xi,\delta}}_{*}\elmG$ and $\elmg{\sh{F}}{\sh{G}'}{\sh{Q}}$ are equivalent extensions of $\gamma_{\sh{F}}$ by $\gamma_{\sh{G}'}$.   

\begin{figure}[ht]
\centering
\begin{tikzpicture}
\useasboundingbox (-6,-5.5) rectangle (6,6.5);
\scope[transform canvas={scale=1}]

\node at (8-4,3+3) {\footnotesize\large $0$};
\node at (8-4,1+3) {\footnotesize\large $0$};

\node at (5-4,1.5+3) {\footnotesize\large $Z$};
\node at (5-4,-0.5+3) {\footnotesize\large $C$};

\node at (2-4,0+3) {\footnotesize\large $X\oplus Z$};
\node at (2-4,-2+3) {\footnotesize\large $A\oplus C$};

\node at (-1-4,-1.5+3) {\footnotesize\large $X$};
\node at (-1-4,-3.5+3) {\footnotesize\large $A$};

\node at (-4-4,-3+3) {\footnotesize\large $0$};
\node at (-4-4,-5+3) {\footnotesize\large $0$};

\node at (8,3) {\footnotesize\large $0$};
\node at (8,1) {\footnotesize\large $0$};

\node at (5,1.5) {\footnotesize\large $Z$};
\node at (5,-0.5) {\footnotesize\large $C$};

\node at (2,0) {\footnotesize\large $X'\oplus Z$};
\node at (2,-2) {\footnotesize\large $A'\oplus C$};

\node at (-1,-1.5) {\footnotesize\large $X'$};
\node at (-1,-3.5) {\footnotesize\large $A'$};

\node at (-4,-3) {\footnotesize\large $0$};
\node at (-4,-5) {\footnotesize\large $0$};

\draw[->, thick] (-8 +0.5,0+0.2) -- (-5-0.5,1.5-0.2);
\draw[->, thick] (-5 +0.5,1.5+0.15) -- (-2-0.5-0.15,3-0.2-0.13);
\draw[->, thick] (-2 +0.5+0.2,3+0.2+0.125) -- (1-0.5,4.5-0.2);
\draw[->, thick] (1 +0.5,4.5+0.2) -- (4-0.5,6-0.2);

\draw[->, thick] (-8 +0.5,-2+0.2) -- (-5-0.5,-0.5-0.2);
\draw[->, thick] (-5 +0.5,-0.5+0.15) -- (-2-0.5-0.15,1-0.2-0.11);
\draw[->, thick] (-2 +0.5+0.2,1+0.2+0.125) -- (1-0.5,2.5-0.2);
\draw[->, thick] (1 +0.5,2.5+0.2) -- (4-0.5,4-0.2);

\draw[->, thick] (-2,1+0.35) -- (-2,3-0.35);
\draw[->, thick] (1,2.5+0.35) -- (1,4.5-0.35);
\draw[->, thick] (-5,-0.5+0.35) -- (-5,1.5-0.35);

\draw[->, thick] (-4+0.5,-3+0.2) -- (-1-0.5,-1.5-0.2);
\draw[->, thick] (-1+0.5,-1.5+0.15) -- (2-0.5-0.15,0-0.2-0.13);
\draw[->, thick] (2+0.5+0.2,0+0.2+0.125) -- (5-0.5,1.5-0.2);
\draw[->, thick] (5+0.5,1.5+0.2) -- (8-0.5,3-0.2);

\draw[->, thick] (-4+0.5,-5+0.2) -- (-1-0.5,-3.5-0.2);
\draw[->, thick] (-1+0.5,-3.5+0.15) -- (2-0.5-0.2,-2-0.2-0.125);
\draw[->, thick] (2+0.5+0.2,-2+0.2+0.125) -- (5-0.5,-0.5-0.2);
\draw[->, thick] (5+0.5,-0.5+0.2) -- (8-0.5,1-0.2);

\draw[->, thick] (-1,-3.5+0.35) -- (-1,-1.5-0.35);
\draw[->, thick] (2,-2+0.35) -- (2,0-0.35);
\draw[->, thick] (5,-0.5+0.35) -- (5,1.5-0.35);

\draw[->, thick, blue] (-5+0.25-0.05, 1.5-0.25) -- (-1-0.25-0.05,-1.5+0.25);
\draw[->, thick, blue] (-5+0.25-0.05, -0.5-0.25) -- (-1-0.25-0.05,-3.5+0.25);

\draw[->, thick, red] (-2+0.25+0.25-0.3, 3-0.25-0.0) -- (2-0.25-0.1+0.3,0+0.25+0.15-0.15);
\draw[->, thick, red] (-2+0.25+0.25-0.3, 1-0.25-0.0) -- (2-0.25-0.1+0.2,-2+0.25+0.15-0.15);

\draw[thick, blue] (1+0.25-0.035, 4.5-0.25-0.035) -- (5-0.25-0.035,1.5+0.25-0.035);
\draw[thick, blue] (1+0.25+0.035, 4.5-0.25+0.035) -- (5-0.25+0.035,1.5+0.25+0.035);

\draw[thick, blue] (1+0.25-0.035, 2.5-0.25-0.035) -- (5-0.25-0.035,-0.5+0.25-0.035);
\draw[thick, blue] (1+0.25+0.035, 2.5-0.25+0.035) -- (5-0.25+0.035,-0.5+0.25+0.035);

\node at (-0.75,-0.45) {\normalsize $N_{1}$};
\node at (0,1) {\normalsize $N_{2}$};

\node at (-5-0.5,0.5-0.25) {\normalsize $\gamma_{\sh{G}}$};
\node at (-2-0.75-0.1,2-0.3) {\normalsize $\elmG$};
\node at (1-0.5,3.5-0.25) {\normalsize $\gamma_{\sh{F}}$};

\node at (-1+0.6,-2.5+0.25) {\normalsize$\gamma_{\sh{G}'}$};
\node at (2+0.8+0.1,-1+0.15) {\normalsize $\elmg{F}{G'}{R}$};
\node at (5+0.6,0.5+0.15) {\normalsize $\gamma_{\sh{F}}$};

\node at (-4+0.6,0.5+0.45) {\normalsize $\iota_{A\oplus C}$};
\node at (-1+0.6,1.95+0.45) {\normalsize $\rho_{A\oplus C}$};
\node at (-4+0.3,2.5+0.3) {\normalsize $\iota_{X\oplus Z}$};
\node at (-1+0.3,3.95+0.3) {\normalsize $\rho_{X\oplus Z}$};

\node at (-4+0.6+4.25,0.5+0.05-3.75) {\normalsize $\iota_{A'\oplus C}$};
\node at (-1+0.6+4.25,1.95+0.2-3.75) {\normalsize $\rho_{A'\oplus C}$};
\node at (-4+0.6+3.75,0.5+0.05-1.95) {\normalsize $\iota_{X'\oplus Z}$};
\node at (-1+0.6+3.75,1.95+0.2-1.95) {\normalsize $\rho_{X'\oplus Z}$};

\node at (-0.75-3,-0.45-1.55) {\normalsize $\lambda_{\xi}$};
\node at (0-3,1-1.55) {\normalsize $\lambda_{\delta}$};

\endscope
\end{tikzpicture}
\caption{Commutative diagram between the block extensions $\elmG$ and $\elmg{F}{G'}{R}$.}
\label{fig: Commutative diagram for the equivalence of the push-out of extensions via morphisms for linear maps}
\end{figure}
\end{proof}

Having established the above technical lemmas, we now present an observation that will elucidate the structure of the composition of mixed graded morphisms in the category $\ccs{1}{\beta}$. 

\begin{observation}
Let $\beta=\sigma_{i_{1}}\dots \sigma _{i_{\ell}}\in \mathrm{Br}^{+}_{n}$ be a positive braid word, and let $\mathcal{S}_{\Lambda(\beta)}$ denote the stratification of $\mathbb{R}^{2}$ induced by $\Lambda(\beta)$. Next, we analyze the composition of graded morphisms in the category $\ccs{1}{\beta}$ near an arbitrary arc in $\mathcal{S}_{\Lambda(\beta)}$. 

To begin, let $a$ be an arc in $\mathcal{S}_{\Lambda(\beta)}$. As illustrated in Sub-figure~\eqref{sub-fig: Strata near an arc}, near $a$, the stratification $\mathcal{S}_{\Lambda(\beta)}$ consists of $a$, an upper $2$-dimensional stratum $U$, and a lower $2$-dimensional stratum $D$. In particular, given this local configuration, we choose two arbitrary points $p \in U$ and $q \in D$. 

Let $\sh{F}$, $\sh{F}'$, $\sh{G}$, and $\sh{G}'$ be objects of the category $\ccs{1}{\beta}$. In addition, fix $\lambda\in \mathrm{Ext}^{0}(\sh{F}', \sh{F})$, $\xi\in \mathrm{Ext}^{1}(\sh{F},\sh{G})$, and $\lambda'\in \mathrm{Ext}^{0}(\sh{G}, \sh{G}')$, and let $\sh{H}$ be an extension of $\sh{F}$ by $\sh{G}$ representing $\xi$, namely $\xi=\big[\, \sh{H}\, \big]$. Bearing this in mind, our goal is to provide local models for the compositions $\lambda\circ \xi=\big[\, \lambda^{*}\sh{H} \,\big]\in \mathrm{Ext}^{1}(\sh{F}',\sh{G})$ and $\xi\circ \lambda'=\big[\, \lambda_{*}\sh{H} \,\big]\in \mathrm{Ext}^{1}(\sh{F},\sh{G}')$ near $a$, where $\lambda^{*}\sh{H}$ and ${\lambda'}_{*}\sh{H}$ denote the pull-back and push-out of $\sh{H}$ via $\lambda$ and $\lambda'$, respectively.  

Now, consider the points $p$ and $q$, and denote the stalks of the sheaves $\sh{G}$, $\sh{G}'$, $\sh{H}$, $\sh{F}$, and $\sh{F}'$ at these points by
\begin{equation*}
\begin{aligned}
X&=\sh{G}_{p}\, ,  \qquad X'={\sh{G}'}_{p}\, ,  \qquad   Y=\sh{H}_{p}\, ,  \qquad Z=\sh{F}_{p} \, , \qquad Z'=\sh{F}'_{p} \, , \\[2pt]     
A&=\sh{G}_{q}\, ,  \qquad A'={\sh{G}'}_{q}\, ,  \qquad   B=\sh{H}_{q}\, ,  \qquad C=\sh{F}_{q}\,  , \qquad C'=\sh{F}'_{q} \, .      
\end{aligned}
\end{equation*}
Then, according to the microlocal support conditions, near $a$, the sheaves $\sh{G}$, $\sh{G}'$, $\sh{H}$, $\sh{F}$, and $\sh{F}'$ are specified by linear maps $\gamma_{\sh{G}}: A \to X$, $\gamma_{\sh{G}'}: A' \to X'$, $\Gamma_{\sh{H}}:B \to Y$, $\gamma_{\sh{F}}: C \to Z$, and $\gamma_{\sh{F}'}:C' \to Z'$, respectively, with $\Gamma_{\sh{H}}$ an extension of $\gamma_{\sh{F}}$ by $\gamma_{\sh{G}}$ (see Definition~\eqref{Def: Extensions of linear maps}). In particular, by Lemma~\eqref{Lemma: equivalence of extensions and block extensions}, we know that there exists a linear map $\gamma_{\sh{H}}: C\to X$ such that $\Gamma_{\sh{H}}$ is equivalent to the block extension $\elmG$ of $\gamma_{\sh{F}}$ by $\gamma_{\sh{G}}$ associated with $\gamma_{\sh{H}}$. Hence, for the purpose of studying $\xi$, we identify $\sh{H}$ with the equivalent extension determined near $a$ by the characteristic map $\gamma_{\sh{H}}$.  

In addition, observe that near $a$, the sheaf homomorphisms $\lambda$ and $\lambda'$ are characterized by linear maps $\mu_{\lambda}:C' \to C$, $\nu_{\lambda}: Z' \to Z$, $\mu_{\lambda'}:A \to A'$, and  $\nu_{\lambda'}: X \to X'$ such that the pairs $\lambda_{\mu, \nu}:=(\mu_{\lambda}, \nu_{\lambda})$ and ${\lambda'}_{\mu, \nu}:=(\mu_{\lambda'}, \nu_{\lambda'})$ define morphisms between $\gamma_{\sh{F}'}$ and $\gamma_{\sh{F}}$ and between $\gamma_{\sh{G}}$ and $\gamma_{\sh{G}'}$, respectively (see Definition~\eqref{Def: morphism between linear maps}).

By definition, the pull-back $\lambda^{*}\sh{H}$ is an extension of $\sh{F}'$ by $\sh{G}$, and hence, near $a$, it is determined by the pull-back ${\lambda_{\mu,\nu}}^{*}\,\elmG$ of $\elmG$ via $\lambda_{\mu, \nu}$, which is an extension of $\gamma_{\sh{F}'}$ by $\gamma_{\sh{G}}$ (see Definition~\eqref{Def: pull-back of an extension via a morphism for linear maps}). Moreover, by Lemma~\eqref{Lemma: pull-back of an extension via a morphism for linear maps}, we know that ${\lambda_{\mu,\nu}}^{*}\,\elmG$ is equivalent to the block extension of $\gamma_{\sh{F}'}$ by $\gamma_{\sh{G}}$ associated with the composition $\gamma_{\sh{H}}\circ\mu_{\lambda}:C'\to X$. Consequently, we deduce that the equivalence class $\lambda\circ \xi=\big[\, \lambda^{*}\sh{H} \,\big]$ can be represented, near $a$, by an extension of $\sh{F}'$ by $\sh{G}$ whose characteristic map is $\gamma_{\sh{H}}\circ\mu_{\lambda}$. 

Similarly, the push-out ${\lambda'}_{*}\sh{H}$ is an extension of $\sh{F}$ by $\sh{G}'$, and hence, near $a$, it is determined by the push-out ${{\lambda'}_{\mu,\nu}}_{*}\elmG $ of $\elmG$ via ${\lambda'}_{\mu,\nu}$, which is an extension of $\gamma_{\sh{F}}$ by $\gamma_{\sh{G}'}$ (see Definition~\eqref{Def: push-out of an extension via a morphism for linear maps}). Furthermore, by Lemma~\eqref{Lemma: push-out of an extension via a morphism for linear maps}, we have that ${{\lambda'}_{\mu,\nu}}_{*}\,\elmG$ is equivalent to the block extension of $\gamma_{\sh{F}}$ by $\gamma_{\sh{G}'}$ associated with the composition $\nu_{\lambda'}\circ \gamma_{\sh{H}}:C\to X'$. Therefore, we conclude that near $a$, the equivalence class $\xi\circ \lambda' =\big[\, {\lambda'}_{*}\sh{H} \,\big]$ can be represented by an extension of $\sh{F}$ by $\sh{G}'$ whose characteristic map is $\nu_{\lambda'}\circ \gamma_{\sh{H}}$. 
\end{observation}

With the above observation at hand, we now turn to the explicit analysis of the composition of graded morphisms in the category $\ccs{1}{\beta}$. 

\subsubsection{Combinatorial Composition Rules}
Let $\beta=\sigma_{i_{1}}\cdots \sigma_{i_{\ell}}\in \mathrm{Br}^{+}_{n}$ be a positive braid word. Next, we establish combinatorial rules for the composition of graded morphisms in the category $\ccs{1}{\beta}$. To this end, we begin by introducing a definition and a related lemma, which ensure that our combinatorial formulas yield a well-defined composition of graded morphisms.

\begin{definition}\label{Def: braided graded composition}
Let $\beta = \sigma_{i_1} \cdots \sigma_{i_\ell} \in \mathrm{Br}_n^+$ be a positive braid word. We introduce three bilinear operations associated with $\beta$: the \textbf{Hadamard} $\odot: \mathbb{K}^n_{\mathrm{std}} \times \mathbb{K}^n_{\mathrm{std}} \to \mathbb{K}^n_{\mathrm{std}}$, the \textbf{left braided} $\circ_{\beta_{\mathrm{L}}}: \mathbb{K}^n_{\mathrm{std}} \times \mathbb{K}^\ell_{\mathrm{std}} \to \mathbb{K}^\ell_{\mathrm{std}}$, and the \textbf{right braided} $\circ_{\beta_{\mathrm{R}}}: \mathbb{K}^\ell_{\mathrm{std}} \times \mathbb{K}^n_{\mathrm{std}} \to \mathbb{K}^\ell_{\mathrm{std}}$ compositions.

Specifically, for any $\vec{u} = (u_1, \dots, u_n),\; \vec{v} = (v_1, \dots, v_n) \in \mathbb{K}^n_{\mathrm{std}}$, and any $\vec{p} = (p_1, \dots, p_\ell),\; \vec{q} = (q_1, \dots, q_\ell) \in \mathbb{K}^\ell_{\mathrm{std}}$, we define:
\begin{equation*}
\begin{array}{rcl}
\vec{v} \odot \vec{u} &:=& (v_1 u_1, \dots, v_n u_n)\in \mathbb{K}^{n}_{\mathrm{std}}\,, \\[6pt]
\vec{v}\, \circ_{\beta_{\mathrm{L}}} \vec{p} &:=& (v_{\pi_{\beta_1}(i_1+1)} p_1, \dots, v_{\pi_{\beta_\ell}(i_\ell+1)} p_\ell)\in \mathbb{K}^{\ell}_{\mathrm{std}}\, ,\\[6pt]
\vec{q}\, \circ_{\beta_{\mathrm{R}}} \vec{u} &:=& (q_1 u_{\pi_{\beta_1}(i_1)}, \dots, q_\ell u_{\pi_{\beta_\ell}(i_\ell)})\in \mathbb{K}^{\ell}_{\mathrm{std}}\,.     
\end{array}    
\end{equation*}
\end{definition}

\begin{lemma}\label{Lemma: auxiliar for composition rules}
Let $\beta=\sigma_{i_{1}}\cdots \sigma_{i_{\ell}}\in\mathrm{Br}^{+}_{n}$ be a positive braid word, and let $\sh{F}$, $\sh{G}$, $\sh{Q}$ be objects of the category $\ccs{1}{\beta}$. Let \,$\hat{\mathbf{f}}^{(n)}$, $\hat{\mathbf{g}}^{(n)}$, $\hat{\mathbf{q}}^{(n)}$ be bases for $\mathbb{K}^{n}$, and let $\vec{x},\,\vec{y},\,\vec{z}\in X(\beta,\mathbb{K})$ be points such that the pairs $(\,\hat{\mathbf{f}}^{(n)}, \vec{x}\,)$, $(\,\hat{\mathbf{g}}^{(n)}, \vec{y}\,)$, $(\,\hat{\mathbf{q}}^{(n)}, \vec{z}\,)$ algebraically characterize $\sh{F}$, $\sh{G}$, $\sh{Q}$ according to Theorem~\eqref{Prop. for sheaves and braid matrices}, respectively.

Following Definition~\eqref{Def: linear map delta}, let $\delta_{\sh{F},\sh{G}}\,,\delta_{\sh{G},\sh{Q}}\,,\delta_{\sh{F},\sh{Q}}: \mathbb{K}^{n}_{\mathrm{std}} \to \mathbb{K}^{\ell}_{\mathrm{std}}$ be the linear maps associated with the pairs $(\sh{F},\sh{G})$, $(\sh{G}, \sh{Q})$, $(\sh{F}, \sh{Q})$, respectively. Then the following statements hold:  
\begin{itemize}
\item[(i)] Let $\vec{u}\in \mathrm{ker}\,\delta_{\sh{F},\sh{G}}$, and $\vec{v}\in \mathrm{ker}\,\delta_{\sh{G},\sh{Q}}$. Then the Hadamard product $\vec{v}\,\odot\, \vec{u}\in \mathbb{K}^{n}_{\mathrm{std}}$ satisfies
\begin{equation*}
\vec{v}\,\odot\, \vec{u} \in \mathrm{ker}\, \delta_{\sh{F},\sh{Q}}\, .    
\end{equation*}

\item[(ii)] Let $\vec{u}\in \mathrm{ker}\,\delta_{\sh{F},\sh{G}}$. Then the action of $\vec{u}$ on $\mathrm{im}\,\delta_{\sh{G},\sh{Q}}$ induced by the right braided graded composition satisfies
\begin{equation*}
\mathrm{im}\,\delta_{\sh{G},\sh{Q}} \,\circ_{\beta_{\mathrm{R}}} \vec{u} \;\subseteq\; \mathrm{im}\,\delta_{\sh{F},\sh{Q}}\, .    
\end{equation*}

\item[(iii)] Let $\vec{v}\in \mathrm{ker}\,\delta_{\sh{G},\sh{Q}}$. Then the action of $\vec{v}$ on $\mathrm{im}\,\delta_{\sh{F},\sh{G}}$ induced by the left braided graded composition satisfies 
\begin{equation*}
\vec{v} \,\circ_{\beta_{\mathrm{L}}} \mathrm{im}\,\delta_{\sh{F},\sh{G}} \;\subseteq\; \mathrm{im}\,\delta_{\sh{F},\sh{Q}}\, .    
\end{equation*}
\end{itemize}
\end{lemma}
\begin{proof}
We prove parts (i) and (ii); part (iii) is analogous to (ii).

\noindent
Part (i): By Definition~\eqref{Def: linear map delta}, we know that $\vec{u}=(u_{1}, \dots, u_{n} )\in \mathrm{ker}\, \delta_{\sh{F},\sh{G}}\subseteq \mathbb{K}^{n}_{\mathrm{std}}$ and $\vec{v}=(v_{1}, \dots, v_{n} )\in \mathrm{ker}\, \delta_{\sh{G},\sh{Q}}\subseteq \mathbb{K}^{n}_{\mathrm{std}}$ if and only if
\begin{equation*}
\begin{aligned}
D(\pi_{\beta_{j-1}}(\vec{u}))\cdot B^{(n)}_{i_{j}}(x_{j})&=B^{(n)}_{i_{j}}(y_{j})\cdot D(\pi_{\beta_{j}}(\vec{u}))\, ,    \\
D(\pi_{\beta_{j-1}}(\vec{v}))\cdot B^{(n)}_{i_{j}}(y_{j})&=B^{(n)}_{i_{j}}(z_{j})\cdot D(\pi_{\beta_{j}}(\vec{v}))\, ,    
\end{aligned}
\end{equation*}
for all $j\in [1,\ell]$, with $\beta_{0}=e_{n}\in \mathrm{Br}^{+}_{n}$ and $\pi_{\beta_{0}}=e_{n}\in \mathcal{S}_{n}$ denoting the trivial braid word on $n$ strands and the trivial permutation on $n$ elements, respectively. Hence, since $\vec{v}\odot \vec{u}=(v_{1}u_{1},\dots, v_{n}u_{n})\in \mathbb{K}^{n}_{\mathrm{std}}$, we immediately observe that 
\begin{equation*}
\begin{aligned}
D(\pi_{\beta_{j-1}}(\vec{v}\,\odot\vec{u}))\cdot B^{(n)}_{i_{j}}(x_{j}) & =  D(\pi_{\beta_{j-1}}(\vec{v}))\cdot D(\pi_{\beta_{j-1}}(\vec{u}))\cdot B^{(n)}_{i_{j}}(x_{j})\, , \\
& = D(\pi_{\beta_{j-1}}(\vec{v}))\cdot B^{(n)}_{i_{j}}(y_{j})\cdot D(\pi_{\beta_{j}}(\vec{u}))\, , \\
& = B^{(n)}_{i_{j}}(z_{j})\cdot D(\pi_{\beta_{j}}(\vec{v})) \cdot D(\pi_{\beta_{j}}(\vec{u}))\, , \\
& = B^{(n)}_{i_{j}}(z_{j}) \cdot D(\pi_{\beta_{j}}(\vec{v}\,\odot\vec{u}))\, ,
\end{aligned}
\end{equation*}
for all $j\in [1,\ell]$. It then follows from Definition~\eqref{Def: linear map delta} that $\vec{v}\,\odot \, \vec{u}\in \mathrm{ker}\, \delta_{\sh{F},\sh{Q}}$. 

\noindent
Part (ii): By Definition~\eqref{Def: linear map delta}, we know that $\vec{u}=(u_{1}, \dots, u_{n})\in \mathrm{ker}\, \delta_{\sh{F},\sh{G}}\subseteq \mathbb{K}^{n}_{\mathrm{std}}$ if and only if
\begin{equation}\label{Eq: relations characterizing the kernel of delta_F,G}
D(\pi_{\beta_{j-1}}(\vec{u}))\cdot B^{(n)}_{i_{j}}(x_{j})=B^{(n)}_{i_{j}}(y_{j})\cdot D(\pi_{\beta_{j}}(\vec{u}))\, ,    
\end{equation}
for all $j\in [1,\ell]$.

Now, let $\vec{v} \in \mathbb{K}^{n}_{\mathrm{std}}$. By Definition~\eqref{Def: linear map delta}, we have that $\delta_{\sh{G},\sh{Q}}(\vec{v})=(\delta_{1},\dots, \delta_{\ell})\in \mathbb{K}^{\ell}_{\mathrm{std}}$, where
\begin{equation*}
\delta_{j}(\vec{z}, \vec{v}, \vec{y}) := \Big[\,\big(B^{(n)}_{i_{j}}(z_{j})\big)^{-1}\cdot D(\pi_{\beta_{j-1}}(\vec{v})) \cdot B^{(n)}_{i_{j}}(y_{j})\, \Big]_{i_{j}+1, i_{j}} \, ,  
\end{equation*}
for all $j\in [1,\ell]$. Then, in light of Definition~\eqref{Def: braided graded composition}, we observe that $\delta_{\sh{G},\sh{Q}}(\vec{v})\circ_{\beta_{\mathrm{R}}} \vec{u}=(\delta'_{1},\dots, \delta'_{\ell})\in \mathbb{K}^{\ell}_{\mathrm{std}}$, where 
\begin{equation*}
\delta'_{j}(\vec{z}, \vec{v}, \vec{y}) := \Big[\,\big(B^{(n)}_{i_{j}}(z_{j})\big)^{-1}\cdot D(\pi_{\beta_{j-1}}(\vec{v})) \cdot B^{(n)}_{i_{j}}(y_{j})\cdot D(\pi_{\beta_{j}}(\vec{u}))\,\Big]_{i_{j}+1, i_{j}}\, ,      
\end{equation*}
for all $j\in [1,\ell]$. Bearing this in mind, the kernel constraints~\eqref{Eq: relations characterizing the kernel of delta_F,G} for $\vec{u}$ imply that
\begin{equation*}
\begin{aligned}
\delta'_{j}(\vec{z}, \vec{v}, \vec{y}) &= \Big[\,\big(B^{(n)}_{i_{j}}(z_{j})\big)^{-1}\cdot D(\pi_{\beta_{j-1}}(\vec{v})) \cdot D(\pi_{\beta_{j-1}}(\vec{u}))\cdot B^{(n)}_{i_{j}}(x_{j})\,\Big]_{i_{j}+1, i_{j}}\, , \\[6pt]
&= \Big[\,\big(B^{(n)}_{i_{j}}(z_{j})\big)^{-1}\cdot D(\pi_{\beta_{j-1}}(\vec{v}\odot \vec{u})) \cdot B^{(n)}_{i_{j}}(x_{j})\,\Big]_{i_{j}+1, i_{j}} \, , \\
\end{aligned}
\end{equation*}
for all $j\in [1,\ell]$. It then follows from Definition~\eqref{Def: linear map delta} that $\delta_{\sh{G},\sh{Q}}(\vec{v})\circ_{\beta_{\mathrm{R}}} \vec{u} = \delta_{\sh{F},\sh{Q}}(\vec{v}\,\odot\, \vec{u})$. Finally, since $\vec{v}$ is arbitrary, we conclude that $\mathrm{im}\,\delta_{\sh{G},\sh{Q}} \circ_{\beta_{\mathrm{R}}} \vec{u} \;\subseteq\; \mathrm{im}\,\delta_{\sh{F},\sh{Q}}$. This completes the proof. 
\end{proof}

With this result at hand, we now proceed to prove the first part of our second main result: a theorem providing a combinatorial rule for the composition of zero-degree morphisms in the category $\ccs{1}{\beta}$.

\begin{theorem}\label{Theorem: Graded composition for Ext0 and Ext0}
Let $\beta=\sigma_{i_{1}}\dots \sigma_{i_{\ell}}\in \mathrm{Br}^{+}_{n}$ be a positive braid word, and let $\sh{F}$, $\sh{G}$, $\sh{Q}$ be objects of the category $\ccs{1}{\beta}$. Let \,$\hat{\mathbf{f}}^{(n)}$, $\hat{\mathbf{g}}^{(n)}$, $\hat{\mathbf{q}}^{(n)}$ be bases for $\mathbb{K}^{n}$, and let $\vec{x},\, \vec{y},\, \vec{z}\in X(\beta, \mathbb{K})$ be points such that the pairs $(\,\hat{\mathbf{f}}\,, \vec{x}\,)$, $(\,\hat{\mathbf{g}}\,, \vec{y}\, )$, $(\,\hat{\mathbf{q}}\,, \vec{z}\, )$ algebraically characterize $\sh{F}$, $\sh{G}$, $\sh{Q}$ according to Theorem~\eqref{Prop. for sheaves and braid matrices}, respectively. 

Following Definition~\eqref{Def: linear map delta}, let $\delta_{\sh{F},\sh{G}}, \, \delta_{\sh{G},\sh{Q}},\, \delta_{\sh{F},\sh{Q}}: \mathbb{K}^{n}_{\mathrm{std}}\to \mathbb{K}^{\ell}_{\mathrm{std}}$ be the linear maps associated with the pairs $(\sh{F}, \sh{G})$, $(\sh{G}, \sh{Q})$, $(\sh{F}, \sh{Q})$, respectively. In light of Theorem~\eqref{Theorem: Ext0 as the kernel of delta_F,G}, let $\mu\in \mathrm{Ext}^{0}(\sh{F},\sh{G})$ and $\nu\in \mathrm{Ext}^{0}(\sh{G},\sh{Q})$, and suppose that the vectors $\vec{u}\in \mathrm{ker}\,\delta_{\sh{F},\sh{G}}\subseteq \mathbb{K}^{n}_{\mathrm{std}}$ and $\vec{v}\in \mathrm{ker}\,\delta_{\sh{G},\sh{Q}}\subseteq \mathbb{K}^{n}_{\mathrm{std}}$ determine $\mu$ and $\nu$ under the isomorphisms 
\begin{equation*}
\mathrm{Ext}^{0}(\sh{F},\sh{G})\cong \mathrm{ker}\,\delta_{\sh{F},\sh{G}}\, , \qquad\quad   \mathrm{Ext}^{0}(\sh{G},\sh{Q})\cong \mathrm{ker}\,\delta_{\sh{G},\sh{Q}}\, .   
\end{equation*}
Then, building on Definition~\eqref{Def: braided graded composition}, the composition $\nu\circ \mu\in \mathrm{Ext}^{0}(\sh{F},\sh{Q})$ is determined, under the isomorphism $\mathrm{Ext}^{0}(\sh{F},\sh{Q})\cong \mathrm{ker}\,\delta_{\sh{F},\sh{Q}}$, by the vector $\vec{v}\,\odot\, \vec{u}\in \mathrm{ker}\,\delta_{\sh{F},\sh{Q}}\subseteq \mathbb{K}^{n}_{\mathrm{std}}$. 
\end{theorem}
\begin{proof}
To begin, observe that, given the bases $\hat{\mathbf{f}}^{(n)}$, $\hat{\mathbf{g}}^{(n)}$, and $\hat{\mathbf{q}}^{(n)}$, the sheaf homomorphisms $\mu$ and $\nu$ are fully determined by their characteristic maps $T^{(n)}_{\mu}: \mathbb{K}^{n}\to \mathbb{K}^{n}$ and $T^{(n)}_{\nu}: \mathbb{K}^{n}\to \mathbb{K}^{n}$, namely the linear maps whose matrix representations are given by
\begin{equation*}
\tensor[_{\hat{\mathbf{g}}^{(n)}}]{ \big[\, T_{\mu}^{(n)}\,\big] }{_{\hat{\mathbf{f}}^{(n)}}}=D(\vec{u})\,, \quad \text{and} \quad \tensor[_{\hat{\mathbf{q}}^{(n)}}]{ \big[\, T_{\nu}^{(n)}\,\big] }{_{\hat{\mathbf{g}}^{(n)}}}=D(\vec{v})\, .    
\end{equation*}

By definition, $\nu\circ \mu$ corresponds to the standard composition of sheaf homomorphisms $\nu$ and $\mu$, and hence, in our case, it is fully determined by the composition of the characteristic maps $T^{(n)}_{\nu} $ and $ T^{(n)}_{\mu}$, that is, by the linear map $T^{(n)}_{\nu}\circ T^{(n)}_{\mu}: \mathbb{K}^{n}\to \mathbb{K}^{n}$ whose matrix representation with respect to the bases $\hat{\mathbf{f}}^{(n)}$ and $\hat{\mathbf{q}}^{(n)}$ is given by
\begin{equation*}
\begin{aligned}
 \tensor[_{\hat{\mathbf{q}}^{(n)}}]{ \big[\, T_{\nu}^{(n)}\circ T_{\mu}^{(n)}\,\big] }{_{\hat{\mathbf{f}}^{(n)}}}&= D(\vec{v}) \cdot D(\vec{u})\, ,\\
&=D(\vec{v}\,\odot\, \vec{u})\, ,
\end{aligned}
\end{equation*}
where $\vec{v}\,\odot\,\vec{u}=(v_{1}u_{1},\dots, v_{n}u_{n})\in \mathbb{K}^{n}_{\mathrm{std}}$. Finally, observe that since $\vec{u}\in \mathrm{ker}\, \delta_{\sh{F},\sh{G}}$ and $\vec{v}\in \mathrm{ker}\, \delta_{\sh{G},\sh{Q}}$, Lemma~\eqref{Lemma: auxiliar for composition rules} ensures that $\vec{v}\,\odot\,\vec{u}\in \mathrm{ker}\, \delta_{\sh{F}, \sh{Q}}$, which shows that, under the isomorphism $\mathrm{Ext}^{0}(\sh{F},\sh{Q})\cong \mathrm{ker}\,\delta_{\sh{F},\sh{Q}}$, the composition $\nu\circ \mu$ is completely determined by the Hadamard product $\vec{v}\,\odot\,\vec{u}$, thereby providing a well-posed and explicit combinatorial rule for the composition of the graded morphisms under consideration. This completes the proof. 
\end{proof}

Next, we extend our combinatorial approach to give an explicit description of the composition of mixed-degree morphisms in the category $\ccs{1}{\beta}$, as made precise in the following theorems. 

\begin{theorem}\label{Theorem: Graded composition for Ext0 and Ext1}
Let $\beta=\sigma_{i_{1}}\dots \sigma_{i_{\ell}}\in \mathrm{Br}^{+}_{n}$ be a positive braid word, and let $\sh{F}$, $\sh{G}$, $\sh{Q}$ be objects of the category $\ccs{1}{\beta}$. Let \,$\hat{\mathbf{f}}^{(n)}$, $\hat{\mathbf{g}}^{(n)}$, $\hat{\mathbf{q}}^{(n)}$ be bases for $\mathbb{K}^{n}$, and let $\vec{x},\, \vec{y},\, \vec{z}\in X(\beta, \mathbb{K})$ be points such that the pairs $(\,\hat{\mathbf{f}}\,, \vec{x}\,)$, $(\,\hat{\mathbf{g}}\,, \vec{y}\, )$, $(\,\hat{\mathbf{q}}\,, \vec{z}\, )$ algebraically characterize $\sh{F}$, $\sh{G}$, $\sh{Q}$ according to Theorem~\eqref{Prop. for sheaves and braid matrices}, respectively. 

Following Definition~\eqref{Def: linear map delta}, let $\delta_{\sh{F},\sh{G}}, \, \delta_{\sh{G},\sh{Q}},\, \delta_{\sh{F},\sh{Q}}: \mathbb{K}^{n}_{\mathrm{std}}\to \mathbb{K}^{\ell}_{\mathrm{std}}$ be the linear maps associated with the pairs $(\sh{F}, \sh{G})$, $(\sh{G}, \sh{Q})$, $(\sh{F}, \sh{Q})$, respectively. In light of Theorems~\eqref{Theorem: Ext0 as the kernel of delta_F,G} and~\eqref{Theorem: Ext1 as the cokernel of delta_F,G}, let $\mu\in \mathrm{Ext}^{0}(\sh{F},\sh{G})$ and $\Phi\in \mathrm{Ext}^{1}(\sh{G},\sh{Q})$, and suppose that, for some representative $\vec{q}\in \mathbb{K}^{\ell}_{\mathrm{std}}$, the vector $\vec{u}\in \mathrm{ker}\,\delta_{\sh{F},\sh{G}}\subseteq \mathbb{K}^{n}_{\mathrm{std}}$ and the class $[\,\vec{q}\,]\in \mathrm{coker}\,\delta_{\sh{G},\sh{Q}}$ determine $\mu$ and $\Phi$ under the isomorphisms 
\begin{equation*}
\mathrm{Ext}^{0}(\sh{F},\sh{G})\cong \mathrm{ker}\,\delta_{\sh{F},\sh{G}}\, , \qquad\quad   \mathrm{Ext}^{1}(\sh{G},\sh{Q})\cong \mathrm{coker}\,\delta_{\sh{G},\sh{Q}}\, .   
\end{equation*}
Then, building on Definition~\eqref{Def: braided graded composition}, the composition $\Phi\circ \mu\in \mathrm{Ext}^{1}(\sh{F},\sh{Q})$ is determined, under the isomorphism $\mathrm{Ext}^{1}(\sh{F},\sh{Q})\cong \mathrm{coker}\,\delta_{\sh{F},\sh{Q}}$, by the class $[\,\vec{q}\,\circ_{\beta_{\mathrm{R}}}\vec{u}\,]\in \mathrm{coker}\,\delta_{\sh{F},\sh{Q}}$. 
\end{theorem}
\begin{proof}
To begin, let $\mathcal{U}_{\Lambda(\beta)}=\big\{U_{0}, U_{\mathrm{B}}, U_{\mathrm{L}}, U_{\mathrm{R}}, U_{\mathrm{T}}\big\}$ be the open cover of $\mathbb{R}^{2}$ from Construction~\eqref{Cons: Finite open cover for R^2}. By definition, $\Phi\circ \mu=\big[\, \mu^{*}\sh{H}' \,\big]$, for any extension $\sh{H}'$ of $\sh{G}$ by $\sh{Q}$ representing $\Phi$, where $\mu^{*}\sh{H}'$ denotes the pull-back of $\sh{H}'$ via $\mu$. Next, consider the open set $U_{\mathrm{B}}$, the region in $\mathbb{R}^{2}$ whose intersection with the front projection $\Pi_{x,z}(\Lambda(\beta))$ comprises the braid diagram on $n$ strands associated with $\beta$, as illustrated in Figure~\eqref{Front diagram decomposed into two regions}. By Theorem~\eqref{Theorem: Ext1 as the cokernel of delta_F,G}, any extension of $\sh{G}$ by $\sh{Q}$ is equivalent to a simple extension---namely, an extension determined by a tuple $\vec{q}=(q_{1}, \dots, q_{\ell})\in \mathbb{K}^{\ell}_{\mathrm{std}}$ on $U_{\mathrm{B}}$ and equivalent to the trivial extension away from this region, where two simple extensions are equivalent if and only if their characteristic tuples, say $\vec{q}, \vec{q}\,'\in \mathbb{K}^{\ell}_{\mathrm{std}}$, satisfy $\vec{q}\,' -\vec{q}\in \mathrm{im}\, \delta_{\sh{G}, \sh{Q}}$. Consequently, for the purpose of studying $\Phi\circ \mu$, it suffices to choose as representative for $\Phi$ a simple extension $\sh{H}'$ determined by a tuple $\vec{q}\in \mathbb{K}^{\ell}_{\mathrm{std}}$, compute the pull-back $\mu^{*}\sh{H}'$ on $\mathrm{U}_{\mathrm{B}}$, and then pass to its equivalence class.  

Let $\mathcal{R}_{\Lambda(\beta)}=\big\{R_{j}\big\}_{j=1}^{\ell}$ be the partition of $U_{\mathrm{B}}$ into $\ell$ open vertical straps from Construction~\eqref{Cons: Definition of the vertical straps}. Observe that, as we move from left to right through the regions $R_{j}$, the data characterizing the extension $\sh{H}'$ accumulates. In particular, in each region $R_{j}$, the new contribution is encoded in the characteristic map specifying the extension along the $(i_{j}+1)$-st strand immediately after the crossing $\sigma_{i_{j}}$, the $j$-th and only crossing of $\beta$ contained in $R_{j}$. Thus, to analyze the pull-back $\mu^{*}\sh{H}'$, it suffices to extract, in each region $R_{j}$, the data that determines the pull-back along the $(i_{j}+1)$-st strand, for all $j\in [1,\ell]$.   

Fix a region $R_{j}$ for some $j\in [1,\ell]$, and denote by $k=i_{j}\in [1,n-1]$ the index of $\sigma_{i_{j}}$. Then, near the $(k+1)$-st strand immediately after the crossing $\sigma_{k}$, we have that: 
\begin{itemize}
\item $\sh{F}$ and $\sh{G}$ are specified by a pair of linear maps $\widetilde{\phi}^{\,(k)}_{\sh{F}}:\mathbb{K}^{k}\to \mathbb{K}^{k+1}$ and $\widetilde{\phi}^{\,(k)}_{\sh{G}}:\mathbb{K}^{k}\to \mathbb{K}^{k+1}$, respectively. 
\item The extension $\sh{H}'$ is specified by a characteristic map $\widetilde{\phi}^{(k)}_{\sh{H}'}: \mathbb{K}^{k}\to \mathbb{K}^{k+1}$. 
\item The morphism $\mu$ is characterized by two linear maps $T^{(k+1)}_{\mu}:\mathbb{K}^{k+1}\to \mathbb{K}^{k+1}$ and $\widetilde{T}^{(k)}_{\mu}:\mathbb{K}^{k}\to \mathbb{K}^{k}$ such that the pair $(T^{(k+1)}_{\mu}, \widetilde{T}^{(k)}_{\mu})$ defines a morphism between $\widetilde{\phi}^{\,(k)}_{\sh{F}}:\mathbb{K}^{k}\to \mathbb{K}^{k+1}$ and $\widetilde{\phi}^{\,(k)}_{\sh{G}}:\mathbb{K}^{k}\to \mathbb{K}^{k+1}$. 
\end{itemize}
Consequently, we obtain that, near the $(k+1)$-st strand immediately after the crossing $\sigma_{k}$, the pull-back $\mu^{*}\sh{H}'$ is determined by the characteristic map $\widetilde{\phi}^{\,(k)}_{\sh{H}'}\circ \widetilde{T}^{\,(k)}_{\mu}: \mathbb{K}^{k}\to \mathbb{K}^{k+1}$.  

Now, for each $i\in [1,n-1]$, let $\hat{\mathbf{f}}^{(i)}$, $\hat{\mathbf{g}}^{(i)}$, and $\hat{\mathbf{q}}^{(i)}$ be bases for $\mathbb{K}^{i}$ induced by the bases $\hat{\mathbf{f}}^{(n)}$, $\hat{\mathbf{g}}^{(n)}$, and $\hat{\mathbf{q}}^{(n)}$ for $\mathbb{K}^{n}$, so that $\big\{ \hat{\mathbf{f}}^{(i)} \big\}_{i=1}^{n}$, $\big\{ \hat{\mathbf{g}}^{(i)} \big\}_{i=1}^{n}$, and $\big\{ \hat{\mathbf{q}}^{(i)} \big\}_{i=1}^{n}$ are systems of bases adapted to the linear maps characterizing $\sh{F}$, $\sh{G}$, and $\sh{Q}$ on the left of the crossing in $R_{1}$, respectively. 

Next, consider the braid-transformed bases $\hat{\mathbf{f}}^{(i)}[\beta_{j},\vec{x}_{j} ]$, $\hat{\mathbf{g}}^{(i)}[\beta_{j},\vec{y}_{j} ]$, and $\hat{\mathbf{g}}^{(i)}[\beta_{j},\vec{z}_{j} ]$ associated with the truncated braid word $\beta_{j}$ and the truncated tuples $\vec{x}_{j}$, $\vec{y}_{j}$, ${\vec{z}_{j}}\in \mathbb{K}^{j}_{\mathrm{std}}$. Then, with respect to these bases, we have that
\begin{itemize}
\item The matrix representing $\widetilde{\phi}^{(k)}_{\sh{H}'}$ is given by
\begin{equation*}
\tensor[_{\hat{\mathbf{q}}^{(k+1)}[\beta_{j}, \,\vec{z}_{j}] }]{ \big[\, \widetilde{\phi}^{\,(k)}_{\sh{H}'} \,\big] }{_{\hat{\mathbf{g}}^{(k)}[\beta_{j}, \,\vec{y}_{j}]}}= \begin{bmatrix}
\star & \cdots & \star & 0\\
\vdots & \ddots & \vdots & \vdots\\
\star & \cdots & \star & 0\\
\star & \cdots & \star & q_{j}\\
\end{bmatrix} _{(k+1, k)}\, ,
\end{equation*}
where $q_{j}$ denotes the $j$-th entry of the tuple $\vec{q}=(q_{1},\dots, q_{\ell})\in \mathbb{K}^{\ell}_{\mathrm{std}}$ characterizing $\sh{H}'$, and the starred entries depend on the data from prior crossings before $\sigma_{i_{j}}$. 
\item The matrix representing $\widetilde{T}^{\,(k)}_{\mu}$ is given by
\begin{equation*}
\tensor[_{\hat{\mathbf{g}}^{(k)}[\beta_{j}, \,\vec{y}_{j} ]}]{ \big[\, \widetilde{T}^{\,(k)}_{\mu}\,\big] }{_{\hat{\mathbf{f}}^{(k)}[\beta_{j}, \,\vec{x}_{j} ]}}
= \,\text{principal $k\times k$ submatrix of }\, D(\pi_{\beta_{j}}(\vec{u}))\, .   
\end{equation*}
\end{itemize}
Hence, a direct calculation shows that
\begin{equation*}
\tensor[_{\hat{\mathbf{q}}^{(k+1)}[\beta_{j}, \,\vec{z}_{j}] }]{ \big[\, \widetilde{\phi}^{\,(k)}_{\sh{H}'}\circ \widetilde{T}^{\,(k)}_{\mu} \,\big] }{_{\hat{\mathbf{f}}^{(k)}[\beta_{j}, \,\vec{x}_{j}]}} = \begin{bmatrix}
\star & \cdots & \star & 0\\
\vdots & \ddots & \vdots & \vdots\\
\star & \cdots & \star & 0\\
\star & \cdots & \star & q_{j}\cdot u_{\beta_{j}(k)} \\
\end{bmatrix} _{(k+1, k)}\, ,
\end{equation*}
where $u_{\pi_{\beta_{j}}(k)}$ is the $(k,k)$-entry of $D(\pi_{\beta_{j}}(\vec{u}))$, namely the $k$-th entry of the permutation of $\vec{u}$ by $\pi_{\beta_{j}}$, the permutation associated with the truncated braid word $\beta_{j}$. Moreover, since $k=i_{j}$, we deduce that the new contribution to the data characterizing the pull-back $\mu^{*}\sh{H}'$ after the crossing $\sigma_{i_{j}}$ is given by the parameter $q'_{j}=q_{j}\cdot u_{\beta_{j}(i_{j})}$. 

By systematically applying the above argument to all the regions $R_{j}$,  we obtain that the pull-back $\mu^{*}\sh{H}'$ is characterized by a tuple $\vec{q}\,':=(q'_{1}, \dots, q'_{\ell})\in \mathbb{K}^{\ell}_{\mathrm{std}}$, which entrywise is given by 
\begin{equation*}
q'_{j} = q_{j} \cdot u_{\pi_{\beta_{j}}(i_{j})}\, ,  
\end{equation*}
for all $j\in [1, \ell]$, and in light of Definition~\eqref{Def: braided graded composition}, we immediately recognize that $\vec{q}\,'=\vec{q}\,\circ_{\beta_{\mathrm{R}}} \vec{u}$. Finally, observe that, since $\vec{u}\in \mathrm{ker}\,\delta_{\sh{F},\sh{G}}$, Lemma~\eqref{Lemma: auxiliar for composition rules} ensures that $\mathrm{im}\, \delta_{\sh{G},\sh{Q}} \circ_{\beta_{\mathrm{R}}} \vec{u}\subseteq \mathrm{im}\,\delta_{\sh{F},\sh{Q}}$, and hence, passing to the equivalence class of the pull-back $\mu^{*}\sh{H}'$ in $\mathrm{Ext}^{1}(\sh{F},\sh{Q})$ allows us to conclude that the composition $\Phi\circ \mu$ is completely determined, under the isomorphism $\mathrm{Ext}^{1}(\sh{F},\sh{Q})\cong \mathrm{coker}\,\delta_{\sh{F},\sh{Q}}$, by the class $[\, \vec{q}\circ_{\beta_{\mathrm{R}}} \vec{u}  \,]\in \mathrm{coker}\, \delta_{\sh{F}, \sh{Q}}$, thereby providing a well-posed and explicit combinatorial rule for the composition of the graded morphisms under consideration. This completes the proof.
\end{proof}

\begin{theorem}\label{Theorem: Graded composition for Ext1 and Ext0}
Let $\beta=\sigma_{i_{1}}\dots \sigma_{i_{\ell}}\in \mathrm{Br}^{+}_{n}$ be a positive braid word, and let $\sh{F}$, $\sh{G}$, $\sh{Q}$ be objects of the category $\ccs{1}{\beta}$. Let \,$\hat{\mathbf{f}}^{(n)}$, $\hat{\mathbf{g}}^{(n)}$, $\hat{\mathbf{q}}^{(n)}$ be bases for $\mathbb{K}^{n}$, and let $\vec{x},\, \vec{y},\, \vec{z}\in X(\beta, \mathbb{K})$ be points such that the pairs $(\,\hat{\mathbf{f}}\,, \vec{x}\,)$, $(\,\hat{\mathbf{g}}\,, \vec{y}\, )$, $(\,\hat{\mathbf{q}}\,, \vec{z}\, )$ algebraically characterize $\sh{F}$, $\sh{G}$, $\sh{Q}$ according to Theorem~\eqref{Prop. for sheaves and braid matrices}, respectively. 

Following Definition~\eqref{Def: linear map delta}, let $\delta_{\sh{F},\sh{G}}, \, \delta_{\sh{G},\sh{Q}},\, \delta_{\sh{F},\sh{Q}}: \mathbb{K}^{n}_{\mathrm{std}}\to \mathbb{K}^{\ell}_{\mathrm{std}}$ be the linear maps associated with the pairs $(\sh{F}, \sh{G})$, $(\sh{G}, \sh{Q})$, $(\sh{F}, \sh{Q})$, respectively. In light of Theorems~\eqref{Theorem: Ext0 as the kernel of delta_F,G} and~\eqref{Theorem: Ext1 as the cokernel of delta_F,G}, let $\nu\in \mathrm{Ext}^{0}(\sh{G},\sh{Q})$ and $\Theta\in \mathrm{Ext}^{1}(\sh{F},\sh{G})$, and suppose that, for some representative $\vec{p}\in \mathbb{K}^{\ell}_{\mathrm{std}}$, the vector $\vec{v}\in \mathrm{ker}\,\delta_{\sh{G},\sh{Q}}\subseteq \mathbb{K}^{n}_{\mathrm{std}}$ and the class $[\,\vec{p}\,]\in \mathrm{coker}\,\delta_{\sh{F},\sh{G}}$ determine $\nu$ and $\Theta$ under the isomorphisms 
\begin{equation*}
\mathrm{Ext}^{0}(\sh{G},\sh{Q})\cong \mathrm{ker}\,\delta_{\sh{G},\sh{Q}}\, , \qquad\quad   \mathrm{Ext}^{1}(\sh{F},\sh{G})\cong \mathrm{coker}\,\delta_{\sh{F},\sh{G}}\, .   
\end{equation*}
Then, building on Definition~\eqref{Def: braided graded composition}, the composition $\nu\circ \Theta\in \mathrm{Ext}^{1}(\sh{F},\sh{Q})$ is determined, under the isomorphism $\mathrm{Ext}^{1}(\sh{F},\sh{Q})\cong \mathrm{coker}\,\delta_{\sh{F},\sh{Q}}$, by the class $[\,\vec{v}\,\circ_{\beta_{\mathrm{L}}}\vec{p}\,]\in \mathrm{coker}\,\delta_{\sh{F},\sh{Q}}$. 
\end{theorem}
\begin{proof}
To begin, let $\mathcal{U}_{\Lambda(\beta)}=\big\{U_{0}, U_{\mathrm{B}}, U_{\mathrm{L}}, U_{\mathrm{R}}, U_{\mathrm{T}}\big\}$ be the open cover of $\mathbb{R}^{2}$ from Construction~\eqref{Cons: Finite open cover for R^2}. By definition, $\nu\circ \Theta=\big[\, \nu_{*}\sh{H}\,\big]$, for any extension $\sh{H}$ of $\sh{F}$ by $\sh{G}$ representing $\Theta$, where $\nu_{*}\sh{H}$ denotes the push-out of $\sh{H}$ via $\nu$. Next, consider the open set $U_{\mathrm{B}}$, the region in $\mathbb{R}^{2}$ whose intersection with the front projection $\Pi_{x,z}(\Lambda(\beta))$ comprises the braid diagram on $n$ strands associated with $\beta$, as illustrated in Figure~\eqref{Front diagram decomposed into two regions}. By Theorem~\eqref{Theorem: Ext1 as the cokernel of delta_F,G}, any extension of $\sh{F}$ by $\sh{G}$ is equivalent to a simple extension---namely, an extension determined by a tuple $\vec{p}=(p_{1}, \dots, p_{\ell})\in \mathbb{K}^{\ell}_{\mathrm{std}}$ on $U_{\mathrm{B}}$ and equivalent to the trivial extension away from this region, where two simple extensions are equivalent if and only if their characteristic tuples, say $\vec{p}, \vec{p}\,'\in \mathbb{K}^{\ell}_{\mathrm{std}}$, satisfy ${\vec{p}\,'} -\vec{p}\in \mathrm{im}\, \delta_{\sh{F}, \sh{G}}$. Consequently, for the purpose of studying $\nu\circ \Theta$, it suffices to choose as representative for $\Theta$ a simple extension $\sh{H}$ determined by a tuple $\vec{p}\in \mathbb{K}^{\ell}_{\mathrm{std}}$, compute the push-out $\nu_{*}\sh{H}$ on $\mathrm{U}_{\mathrm{B}}$, and then pass to its equivalence class.  

Let $\mathcal{R}_{\Lambda(\beta)}=\big\{R_{j}\big\}_{j=1}^{\ell}$ be the partition of $U_{\mathrm{B}}$ into $\ell$ open vertical straps from Construction~\eqref{Cons: Definition of the vertical straps}. Observe that, as we move from left to right through the regions $R_{j}$, the data characterizing the extension $\sh{H}$ accumulates. In particular, in each region $R_{j}$, the new contribution is encoded in the characteristic map specifying the extension along the $(i_{j}+1)$-st strand immediately after the crossing $\sigma_{i_{j}}$, the $j$-th and only crossing of $\beta$ contained in $R_{j}$. Thus, to analyze the push-out $\nu_{*}\sh{H}$, it suffices to extract, in each region $R_{j}$, the data that determines the push-out along the $(i_{j}+1)$-st strand, for all $j\in [1,\ell]$.   

Fix a region $R_{j}$ for some $j\in [1,\ell]$, and denote by $k=i_{j}\in [1,n-1]$ the index of $\sigma_{i_{j}}$. Then, near the $(k+1)$-st strand immediately after the crossing $\sigma_{k}$, we have that: 
\begin{itemize}
\item $\sh{G}$ and $\sh{Q}$ are specified by a pair of linear maps $\widetilde{\phi}^{\,(k)}_{\sh{G}}:\mathbb{K}^{k}\to \mathbb{K}^{k+1}$ and $\widetilde{\phi}^{\,(k)}_{\sh{Q}}:\mathbb{K}^{k}\to \mathbb{K}^{k+1}$, respectively. 
\item The extension $\sh{H}$ is specified by a characteristic map $\widetilde{\phi}^{(k)}_{\sh{H}}: \mathbb{K}^{k}\to \mathbb{K}^{k+1}$. 
\item The morphism $\nu$ is characterized by two linear maps $T^{(k+1)}_{\nu}:\mathbb{K}^{k+1}\to \mathbb{K}^{k+1}$ and $\widetilde{T}^{(k)}_{\nu}:\mathbb{K}^{k}\to \mathbb{K}^{k}$ such that the pair $(T^{(k+1)}_{\nu}, \widetilde{T}^{(k)}_{\nu})$ defines a morphism between $\widetilde{\phi}^{\,(k)}_{\sh{G}}:\mathbb{K}^{k}\to \mathbb{K}^{k+1}$ and $\widetilde{\phi}^{\,(k)}_{\sh{Q}}:\mathbb{K}^{k}\to \mathbb{K}^{k+1}$. 
\end{itemize}
Consequently, we obtain that, near the $(k+1)$-st strand immediately after the crossing $\sigma_{k}$, the push-out $\nu_{*}\sh{H}$ is determined by the characteristic map $T^{(k+1)}_{\nu}\circ\widetilde{\phi}^{\,(k)}_{\sh{H}}: \mathbb{K}^{k}\to \mathbb{K}^{k+1}$.  

Now, for each $i\in [1,n-1]$, let $\hat{\mathbf{f}}^{(i)}$, $\hat{\mathbf{g}}^{(i)}$, and $\hat{\mathbf{q}}^{(i)}$ be bases for $\mathbb{K}^{i}$ induced by the bases $\hat{\mathbf{f}}^{(n)}$, $\hat{\mathbf{g}}^{(n)}$, and $\hat{\mathbf{q}}^{(n)}$ for $\mathbb{K}^{n}$, so that $\big\{ \hat{\mathbf{f}}^{(i)} \big\}_{i=1}^{n}$, $\big\{ \hat{\mathbf{g}}^{(i)} \big\}_{i=1}^{n}$, and $\big\{ \hat{\mathbf{q}}^{(i)} \big\}_{i=1}^{n}$ are systems of bases adapted to the linear maps characterizing $\sh{F}$, $\sh{G}$, and $\sh{Q}$ on the left of the crossing in $R_{1}$, respectively. 

Next, consider the braid-transformed bases $\hat{\mathbf{f}}^{(i)}[\beta_{j},\vec{x}_{j} ]$, $\hat{\mathbf{g}}^{(i)}[\beta_{j},\vec{y}_{j} ]$, and $\hat{\mathbf{g}}^{(i)}[\beta_{j},\vec{z}_{j} ]$ associated with the truncated braid word $\beta_{j}$ and the truncated tuples $\vec{x}_{j}$, $\vec{y}_{j}$, ${\vec{z}_{j}}\in \mathbb{K}^{j}_{\mathrm{std}}$. Then, with respect to these bases, we have that
\begin{itemize}
\item The matrix representing $\widetilde{\phi}^{\,(k)}_{\sh{H}}$ is given by
\begin{equation*}
\tensor[_{\hat{\mathbf{g}}^{(k+1)}[\beta_{j}, \,\vec{y}_{j}] }]{ \big[\, \widetilde{\phi}^{\,(k)}_{\sh{H}} \,\big] }{_{\hat{\mathbf{f}}^{(k)}[\beta_{j}, \,\vec{x}_{j}]}}= \begin{bmatrix}
\star & \cdots & \star & 0\\
\vdots & \ddots & \vdots & \vdots\\
\star & \cdots & \star & 0\\
\star & \cdots & \star & p_{j}\\
\end{bmatrix} _{(k+1, k)}\, ,
\end{equation*}
where $p_{j}$ denotes the $j$-th entry of the tuple $\vec{p}=(p_{1},\dots, p_{\ell})\in \mathbb{K}^{\ell}_{\mathrm{std}}$ characterizing $\sh{H}$, and the starred entries depend on the data from prior crossings before $\sigma_{i_{j}}$. 
\item The matrix representing $T^{(k+1)}_{\nu}$ is given by
\begin{equation*}
\tensor[_{\hat{\mathbf{q}}^{(k+1)}[\beta_{j}, \,\vec{z}_{j} ]}]{ \big[\, T^{(k+1)}_{\nu}\,\big] }{_{\hat{\mathbf{g}}^{(k+1)}[\beta_{j}, \,\vec{y}_{j} ]}}
= \,\text{principal $(k+1)\times (k+1)$ submatrix of }\, D(\pi_{\beta_{j}}(\vec{v}))\, .   
\end{equation*}
\end{itemize}
Hence, a direct calculation shows that
\begin{equation*}
\tensor[_{\hat{\mathbf{q}}^{(k+1)}[\beta_{j}, \,\vec{z}_{j}] }]{ \big[\, T^{(k+1)}_{\nu} \circ \widetilde{\phi}^{\,(k)}_{\sh{H}} \,\big] }{_{\hat{\mathbf{f}}^{(k)}[\beta_{j}, \,\vec{x}_{j}]}} = \begin{bmatrix}
\star & \cdots & \star & 0\\
\vdots & \ddots & \vdots & \vdots\\
\star & \cdots & \star & 0\\
\star & \cdots & \star &  v_{\beta_{j}(k+1)}\cdot p_{j} \\
\end{bmatrix} _{(k+1, k)}\, ,
\end{equation*}
where $v_{\pi_{\beta_{j}}(k+1)}$ is the $(k+1,k+1)$-entry of $D(\pi_{\beta_{j}}(\vec{v}))$, namely the $k+1$-th entry of the permutation of $\vec{v}$ by $\pi_{\beta_{j}}$, the permutation associated with the truncated braid word $\beta_{j}$. Moreover, since $k=i_{j}$, we deduce that the new contribution to the data characterizing the push-out $\nu_{*}\sh{H}$ after the crossing $\sigma_{i_{j}}$ is given by the parameter $p'_{j}=v_{\beta_{j}(i_{j}+1)}\cdot p_{j}$. 

By systematically applying the above argument to all the regions $R_{j}$, we obtain that the push-out $\nu_{*}\sh{H}$ is characterized by a tuple $\vec{p}\,':=(p'_{1}, \dots, p'_{\ell})\in \mathbb{K}^{\ell}_{\mathrm{std}}$, which entrywise is given by 
\begin{equation*}
p'_{j}= v_{\pi_{\beta_{j}}(i_{j}+1)} \cdot  p_{j}\, ,  
\end{equation*}
for all $j\in [1, \ell]$, and in light of Definition~\eqref{Def: braided graded composition}, we immediately recognize that $\vec{p}\,'=\vec{v}\,\circ_{\beta_{\mathrm{L}}} \vec{p}$. Finally, observe that, since $\vec{v}\in \mathrm{ker}\,\delta_{\sh{G},\sh{Q}}$, Lemma~\eqref{Lemma: auxiliar for composition rules} ensures that $\vec{v}\, \circ_{\beta_{\mathrm{L}}} \mathrm{im}\, \delta_{\sh{F},\sh{G}} \subseteq \mathrm{im}\,\delta_{\sh{F},\sh{Q}}$, and hence, passing to the equivalence class of the push-out $\nu_{*}\sh{H}$ in $\mathrm{Ext}^{1}(\sh{F},\sh{Q})$ allows us to conclude that the composition $\nu\circ \Theta$ is completely determined, under the isomorphism $\mathrm{Ext}^{1}(\sh{F},\sh{Q})\cong \mathrm{coker}\, \delta_{\sh{F}, \sh{Q}}$, by the class $[\,\vec{v}\,\circ_{\beta_{\mathrm{L}}} \vec{p} \,]\in \mathrm{coker}\, \delta_{\sh{F}, \sh{Q}}$, thereby providing a well-posed and explicit combinatorial rule for the composition of the graded morphisms under consideration. This completes the proof.
\end{proof}

Having established the preceding results,  we conclude the explicit algebraic characterization of the category $\ccs{1}{\beta}$. In the next section, we explore some applications of this characterization.

%% file: sec6.tex
\section{Applications}\label{sec:applications}
\noindent
Let $\beta\in \mathrm{Br}^{+}_{n}$ be a positive braid word. Next, we present some applications of the explicit description of the category $\ccs{1}{\beta}$ that we developed in the preceding sections. More precisely, we first consider the case where $\Lambda(\beta)$ is a Legendrian knot---a single-component Legendrian link. In this setting, we establish some structural results concerning the category $\ccs{1}{\beta}$ over a generic field $\mathbb{K}$ and provide a simple characterization of the category in the case where $\beta=\sigma_{1}\sigma_{2}\sigma_{1}\sigma_{2}$ and $\mathbb{K}=\mathbb{Z}_{2}$, which corresponds to the Legendrian trefoil knot realized as the rainbow closure of a positive braid on three strands. Furthermore, to illustrate how our theoretical framework works for multi-component Legendrian links, we explicitly analyze the category $\ccs{1}{\beta}$ in the case where $\beta=\sigma_{1}\sigma_{2}\sigma_{1}$ and $\mathbb{K}=\mathbb{Z}_{2}$, thereby covering the case of the Legendrian Hopf link realized as the rainbow closure of a positive braid word on three strands.  

\subsection{The Knot Case} 
Let $\beta = \sigma_{i_{1}}\cdots \sigma_{i_{\ell}} \in \mathrm{Br}^{+}_{n}$ be a positive braid word such that $\Lambda(\beta)$ is a Legendrian knot. Next, we provide some structural results concerning the linear structure of the graded morphism spaces and the algebraic properties of the graded endomorphism spaces in the category $\ccs{1}{\beta}$. Moreover, we present a detailed analysis of some aspects of the category $\ccs{1}{\beta}$ in the case where $\beta=\sigma_{1}\sigma_{2}\sigma_{1}\sigma_{2}$ and $\mathbb{K}=\mathbb{Z}_{2}$.

\subsubsection{Some Structural Results} Let $\beta = \sigma_{i_{1}}\cdots \sigma_{i_{\ell}} \in \mathrm{Br}^{+}_{n}$ be a positive braid word such that $\Lambda(\beta)$ is a Legendrian knot. In particular, recall that $\Lambda(\beta)$ defines a knot if and only if the permutation $\pi_{\beta} \in \mathrm{S}_{n}$ associated with $\beta$ consists of a single cycle in its cycle decomposition in $\mathrm{S}_{n}$~\cite{CGGS1}. Bearing this in mind, we now proceed to exploit the stated property of $\pi_{\beta}$ to establish some results revealing the linear structure of the graded morphism spaces and the algebraic properties of the graded endomorphism spaces in the category $\ccs{1}{\beta}$.

Let $X(\beta,\mathbb{K})$ be the braid variety associated with $\beta$, and let $\mathcal{T}$ be the $(n-1)$-dimensional torus given by
\begin{equation*}
\mathcal{T}:=\big\{\,\vec{t}\in\mathbb{K}^{n}_{\mathrm{std}} \,|\, D(\,\vec{t}\;)\in \mathrm{SL}(n,\mathbb{K})\,\big\} \, .    
\end{equation*}
As shown in~\cite{CGGS1}, $\mathcal{T}$ acts naturally on $X(\beta,\mathbb{K})$. The following definition formalizes this group action.  

\begin{definition}\label{Def: Torus action}
Let $\beta = \sigma_{i_{1}}\cdots \sigma_{i_{\ell}} \in \mathrm{Br}^{+}_{n}$ be a positive braid word. The torus action on $X(\beta,\mathbb{K})$ corresponds to the map 
\begin{equation*}
\begin{array}{cccccc}
&\star:& \mathcal{T}\times X(\beta,\mathbb{K}) &\longrightarrow& X(\beta,\mathbb{K})\,,           &   \\[2pt]
&      & (\,\vec{t}\,, \vec{z}\,)              &\longmapsto& \vec{t} \,\star\, \vec{z}\,, &  
\end{array}
\end{equation*}
where, for any $\vec{z}=(z_{1},\dots, \vec{z}_{\ell})\in X(\beta,\mathbb{K})$, the point $\,\vec{t}\,\star\,\vec{z}:=(z'_{1}, \dots, z'_{\ell})\in X(\beta,\mathbb{K})$ is uniquely determined by the system of equations
\begin{equation*}
D(\pi_{\beta_{j-1}}(\vec{t}\,))\cdot B^{(n)}_{i_{j}}(z_{j})=B^{(n)}_{i_{j}}(z'_{j})\cdot D(\pi_{\beta_{j}}(\vec{t}\,))\, ,  \qquad\text{for all $j\in [1,\ell]$}\, .
\end{equation*}
Accordingly, we say that two points $\vec{z}_{1}, \vec{z}_{2} \in X(\beta, \mathbb{K})$ are \emph{equivalent under the torus action} if there exists $\vec{t}\in \mathcal{T}$ such that $\vec{z}_{2} = \vec{t}\,\star\,\vec{z}_{1}$, and we denote this relation by $\vec{z}_{2} \cong \vec{z}_{1}$; otherwise, we write $\vec{z}_{2} \ncong \vec{z}_{1}$.
\end{definition}

A notable property of the torus action on $X(\beta,\mathbb{K})$ is that, when $\Lambda(\beta)$ is a Legendrian knot, it is free~\cite{CGGS1}. In view of this property, we now introduce a definition that will play a fundamental role in the subsequent discussion.

\begin{definition}
Let $\beta = \sigma_{i_{1}}\cdots \sigma_{i_{\ell}} \in \mathrm{Br}^{+}_{n}$ be a positive braid word. We define the subvariety $\widetilde{X}(\beta,\mathbb{K}) \subset X(\beta,\mathbb{K})$ by
\begin{equation*}
\widetilde{X}(\beta,\mathbb{K}) := \big\{ \vec{z} \in \mathbb{K}^{\ell}_{\mathrm{std}} \,\big|\, \text{all leading principal minors of $P_{\beta}(\vec{z}\,)$ are equal to $1$} \big\} \, .
\end{equation*}
\end{definition}

The subvariety $\widetilde{X}(\beta,\mathbb{K})\subset X(\beta,\mathbb{K})$ plays a particularly important role in our current setting. Specifically, when $\Lambda(\beta)$ is a Legendrian knot, the defining properties of $\widetilde{X}(\beta,\mathbb{K})$ together with the freeness of the torus action on $X(\beta,\mathbb{K})$ imply that the induced map 
\begin{equation*}
\begin{array}{ccc}
   \mathcal{T}\times \widetilde{X}(\beta,\mathbb{K}) &\to& X(\beta,\mathbb{K})\, ,  \\
   (\,\vec{t}\,,\, \vec{z}\,) &\mapsto & \vec{t}\,\star\, \vec{z}\, ,
\end{array}    
\end{equation*}
is bijective. Equivalently, for any $\vec{z}\in X(\beta, \mathbb{K})$ there exist unique $\vec{t}\in \mathcal{T}$ and $\vec{z}\,'\in \widetilde{X}(\beta,\mathbb{K})$ such that $\vec{z}=\vec{t}\,\star\, \vec{z}\,'$. 

Next, we present a result that, exploiting the torus action on $X(\beta,\mathbb{K})$, provides an explicit description of the linear structure of the graded morphism spaces in the category $\ccs{1}{\beta}$ in the case where $\Lambda(\beta)$ is a Legendrian knot.

\begin{theorem}\label{Theorem: Graded morphism spaces for the knot case}
Let $\beta=\sigma_{i_{1}}\cdots \sigma_{i_{\ell}}\in\mathrm{Br}^{+}_{n}$ be a positive braid word such that $\Lambda(\beta)$ is a Legendrian knot, and let $\sh{F}$, $\sh{G}$ be objects of the category $\ccs{1}{\beta}$. Let \,$\hat{\mathbf{f}}^{n}$, $\hat{\mathbf{g}}^{n}$ be bases for $\mathbb{K}^{n}$, and let $\vec{z}_{1},\,\vec{z}_{2}\in X(\beta, \mathbb{K})$ be points such that the pairs $(\,\hat{\mathbf{f}}, \vec{z}\,)$ and $(\,\hat{\mathbf{g}}, \vec{z}\,'\,)$ algebraically characterize $\sh{F}$ and $\sh{G}$ according to Theorem~\eqref{Prop. for sheaves and braid matrices}, respectively. Then there is an isomorphism of graded vector spaces
\begin{equation*}
\mathrm{Ext}^{\bullet}(\sh{F},\sh{G})\cong \begin{cases}
\;\mathbb{K}\left[0\right] \oplus \mathbb{K}^{\mathrm{tb}(\Lambda(\beta))+1}\left[1\right]\,, & \text{if $\,\vec{z}\,' \cong \vec{z}$}\\[2pt]  
\;0\left[0\right] \oplus \mathbb{K}^{\mathrm{tb}(\Lambda(\beta))}\left[1\right]\,, & \text{if $\,\vec{z}\,' \ncong \vec{z}$}   
\end{cases}\, ,    
\end{equation*}
where $\mathrm{tb}(\Lambda(\beta)):=\ell-n$ corresponds to the Thurston--Bennequin number of $\Lambda(\beta)$.     
\end{theorem}
\begin{proof}
To begin, let us consider the points $\vec{z},\,\vec{z}\,'\in X(\beta, \mathbb{K})$. By our previous discussion, there exist unique $\vec{s}, \, \vec{t}\in \mathcal{T}$ and $\vec{x},\, \vec{y}\in \widetilde{X}(\beta, \mathbb{K})$ such that $\vec{z}=\vec{s}\,\star\, \vec{x}$ and $\vec{z}\,'=\vec{t}\,\star\, \vec{y}$. More precisely, building on Definition~\eqref{Def: Torus action}, we have that
\begin{equation*}
\begin{aligned}
D(\pi_{\beta_{j-1}}(\,\vec{s}\;))\cdot B^{(n)}_{i_{j}}(x_{j})=B^{(n)}_{i_{j}}(z_{j})\cdot D(\pi_{\beta_{j}}(\,\vec{s}\;))\, , \quad \text{for all $j\in [1,\ell]$}\, , \\[2pt]
D(\pi_{\beta_{j-1}}(\,\vec{t}\;))\cdot B^{(n)}_{i_{j}}(y_{j})=B^{(n)}_{i_{j}}(z'_{j})\cdot D(\pi_{\beta_{j}}(\,\vec{t}\;))\, , \quad \text{for all $j\in [1,\ell]$}\, ,
\end{aligned}  
\end{equation*}
where $\vec{z}=(z_{1}, \dots, z_{\ell})$, $\vec{z}\,'=(z'_{1}, \dots, z'_{\ell})$, $\vec{x}=(x_{1}, \dots, x_{\ell})$, and $\vec{y}=(y_{1}, \dots, y_{\ell})$.  

Now, let us consider the linear map $\delta_{\sh{F},\sh{G}}:\mathbb{K}^{n}_{\mathrm{std}}\to \mathrm{K}^{\ell}_{\mathrm{std}}$ associated with $\sh{F}$ and $\sh{G}$, as introduced in Definition~\eqref{Def: linear map delta}, and let $\vec{u}\in \ker\,\delta_{\sh{F},\sh{G}}$. By definition, we know that
\begin{equation*}
D(\pi_{\beta_{j-1}}(\vec{u}\,))\cdot B^{(n)}_{i_{j}}(z_{j})=B^{(n)}_{i_{j}}(z'_{j})\cdot D(\pi_{\beta_{j}}(\vec{u}\,))\, , \quad \text{for all $j\in [1,\ell]$}\, .
\end{equation*}
Hence, putting the above relations together yields that
\begin{equation}\label{Eq: useful identities for Ext0 in the knot case}
D(\pi_{\beta_{j-1}}(\vec{v}\,))\cdot B^{(n)}_{i_{j}}(x_{j})=B^{(n)}_{i_{j}}(y_{j})\cdot D(\pi_{\beta_{j}}(\vec{v}\,))\, , \quad \text{for all $j\in [1,\ell]$}\, ,
\end{equation}
where $\vec{v}=\big(\,\vec{t}\;\big)^{-1}\odot\, \vec{u}\, \odot\, \vec{s}$. Furthermore, implementing relations~\eqref{Eq: useful identities for Ext0 in the knot case} iteratively throughout $j\in [1,\ell]$, we obtain that
\begin{equation}\label{Eq: useful global matrix constraint}
D(\vec{v}\,)\cdot P_{\beta}(\,\vec{x}\,)=P_{\beta}(\,\vec{y}\,)\cdot D(\pi_{\beta}(\vec{v}\,))\,,
\end{equation}
where $P_{\beta}(\vec{x})=B^{(n)}_{i_{1}}(x_{1})\cdots B^{(n)}_{i_{\ell}}(x_{\ell})$ and  $P_{\beta}(\vec{y})=B^{(n)}_{i_{1}}(y_{1})\cdots B^{(n)}_{i_{\ell}}(y_{\ell})$. In particular, observe that, since $\vec{x},\, \vec{y}\in \widetilde{X}(\beta,\mathbb{K})$, the defining properties of the subvariety $\widetilde{X}(\beta,\mathbb{K})$ and the matrix equation~\eqref{Eq: useful global matrix constraint} imply that 
\begin{equation}\label{Eq: vector identity and permutation}
\pi_{\beta}(\,\vec{v}\,)=\vec{v}\, .
\end{equation}
Here, recall that, by assumption, $\Lambda(\beta)$ is a Legendrian knot, and hence the permutation $\pi_{\beta}$ is a single-cycle in its cycle decomposition in $\mathrm{S}_{n}$~\cite{CGGS1}. Bearing this in mind, identity~\eqref{Eq: vector identity and permutation} guarantees that there exists $u\in \mathbb{K}$ such that $\vec{v}=u\,\vec{1}_{n}$, where $\vec{1}_{n}=(1,\dots, 1)\in \mathbb{K}^{n}_{\mathrm{std}}$. Consequently, relations~\eqref{Eq: useful identities for Ext0 in the knot case} reduce to the vector equation
\begin{equation*}
u\,(\,\vec{y}-\vec{x}\,)=0 \, ,  
\end{equation*}
which yields the following two cases: 
\begin{itemize}
\justifying
\item \textbf{Case 1}: Suppose $\vec{y}=\vec{x}$. In this case, $u$ is a free parameter, and as a result $\vec{u}=u\,\vec{a}$, where $\vec{a}=\vec{t}\odot \big(\,\vec{s}\;\big)^{-1}\in \ker\,\delta_{\sh{F},\sh{G}}$. Hence, since since $\vec{u}$ is arbitrary, we obtain that $\ker\,\delta_{\sh{F},\sh{G}}=\mathrm{Span}_{\,\mathbb{K}}\{\,\vec{a}\;\}\cong \mathbb{K}$, and by Theorem~\eqref{Theorem: Ext0 as the kernel of delta_F,G}, we conclude that $\mathrm{Ext}^{1}(\sh{F},\sh{G})\cong \mathbb{K}$. Here, it is important to emphasize that, if $\vec{z}\,'=\vec{z}$ or $\vec{z},\,\vec{z}\,'\in \widetilde{X}(\beta,\mathbb{K})$, then $\vec{a}=\vec{1}_{n}\in \mathbb{K}^{n}_{\mathrm{std}}$, and in any of these particular situations $\vec{u}=u\,\vec{1}_{n}$. 

\item \textbf{Case 2}: Suppose $\vec{y}\neq\vec{x}$. In this case, we have that $u=0$, and as a result $\vec{u}=u\,\vec{a}=0$. Hence, since $\vec{u}$ is arbitrary, we conclude that $\ker\,\delta_{\sh{F},\sh{G}}\cong 0$, and by Theorem~\eqref{Theorem: Ext0 as the kernel of delta_F,G}, we deduce that $\mathrm{Ext}^{1}(\sh{F},\sh{G})\cong 0$\, .
\end{itemize}

Next, by Theorem~\eqref{Theorem: Euler characteristic for Ext groups}, we know that
\begin{equation*}
\begin{aligned}
\chi (\mathrm{Ext}^{\bullet}(\sh{F},\sh{G}))&=\mathrm{dim}_{\,\mathbb{K}} \mathrm{Ext}^{0}(\sh{F},\sh{G})- \mathrm{dim}_{\,\mathbb{K}} \mathrm{Ext}^{1}(\sh{F},\sh{G})\, ,\\
&=-\mathrm{tb}(\Lambda(\beta))\, ,
\end{aligned}
\end{equation*}
which implies that
\begin{equation*}
\mathrm{dim}_{\,\mathbb{K}} \mathrm{Ext}^{1}(\sh{F},\sh{G})= \mathrm{tb}(\Lambda(\beta))+\mathrm{dim}_{\,\mathbb{K}} \mathrm{Ext}^{0}(\sh{F},\sh{G})\, . 
\end{equation*}
Bearing this in mind, we arrive to the following two cases: 
\begin{itemize}
\justifying
\item \textbf{Case 1}: Suppose that $\vec{y}=\vec{x}$. In this case, $\mathrm{Ext}^{1}(\sh{F},\sh{G})\cong \mathbb{K}$, and hence $\mathrm{dim}_{\,\mathbb{K}} \mathrm{Ext}^{1}(\sh{F},\sh{G})= \mathrm{tb}(\Lambda(\beta))+1$, which shows that $\mathrm{Ext}^{1}(\sh{F},\sh{G}) \cong \mathbb{K}^{\mathrm{tb}(\Lambda(\beta))+1}$. 
\item \textbf{Case 2}: Suppose that $\vec{y}\neq \vec{x}$. In this case, $\mathrm{Ext}^{1}(\sh{F},\sh{G})\cong 0$, and therefore $\mathrm{dim}_{\,\mathbb{K}} \mathrm{Ext}^{1}(\sh{F},\sh{G})= \mathrm{tb}(\Lambda(\beta))$, which implies that $\mathrm{Ext}^{1}(\sh{F},\sh{G}) \cong \mathbb{K}^{\mathrm{tb}(\Lambda(\beta))}$.
\end{itemize}

Finally, observe that, if $\vec{y}=\vec{x}$, then $\vec{z}\,'\cong \vec{z}$, whereas if $\vec{y}\neq \vec{x}$, then $\vec{z}\,'\ncong \vec{z}$. The result follows. 
\end{proof}

\noindent
From our point of view, Theorem~\eqref{Theorem: Graded morphism spaces for the knot case} is particularly enlightening. It establishes that, in the knot case, the linear structure of the graded morphism spaces in the category $\ccs{1}{\beta}$ is remarkably simple: it is largely determined by intrinsic topological data associated with $\Lambda(\beta)$, namely its Thurston-Bennequin number. 

Next, we present a result that, in the spirit of Theorem~\eqref{Theorem: Graded morphism spaces for the knot case}, provides a simple description of the unital graded ring structure of the graded endomorphism spaces in the category $\ccs{1}{\beta}$. 

\begin{theorem}\label{Theorem: Graded endomorphism spaces in the knot case}
Let $\beta=\sigma_{i_{1}}\cdots \sigma_{i_{\ell}}\in \mathrm{Br}^{+}_{n}$ be a positive braid word such that $\Lambda(\beta)$ is a Legendrian knot. Let $\sh{F}$ be an object of the category $\ccs{1}{\beta}$. Let $\hat{\mathbf{f}}^{(n)}$ a basis for $\mathbb{K}^{n}$, and let $\vec{z}\in X(\beta,\mathbb{K})$ be a point such that the pair $(\,\hat{\mathbf{f}}^{(n)},\, \vec{z}\,)$ algebraically characterizes $\sh{F}$ according to Theorem~\eqref{Prop. for sheaves and braid matrices}. Then: 
\begin{itemize}
\item[(1)] There is an isomorphism of graded vector spaces
\begin{equation*}
\mathrm{Ext}^{\bullet}(\sh{F}, \sh{F}) \cong \mathbb{K}\left[0\right] \oplus \mathbb{K}^{\mathrm{tb}(\Lambda(\beta))+1}\left[1\right]\, .   
\end{equation*}

\item[(2)] Under the isomorphism of graded vector spaces in \textit{(1)}, the unital graded ring structure of $\mathrm{Ext}^{\bullet}(\sh{F}, \sh{F})$ induced by the graded composition is given by
\begin{equation*}
\begin{array}{c c c}
\mathrm{Ext}^{\bullet}(\sh{F}, \sh{F}) \times \mathrm{Ext}^{\bullet}(\sh{F}, \sh{F}) & ~~\longrightarrow~~ & \mathrm{End}^{\bullet}(\sh{F}) \,,\\[6pt]
\big((u, \vec{x}\,) , (v, \vec{y}\,)\big)   & ~~\longmapsto~~ & (v,\vec{y}\,)\circ (u,\vec{x}) \;=\; (vu,\, v\vec{x} + u\vec{y}\,) \, ,
\end{array}
\end{equation*}
for all $u,\,v \in \mathbb{K}$ and $\vec{x},\, \vec{y} \in \mathbb{K}^{\mathrm{tb}(\Lambda(\beta)) +1}$.
\end{itemize}
Equivalently, there is an isomorphism of unital graded rings
\begin{equation*}
\mathrm{Ext}^{\bullet}(\sh{F}, \sh{F}) \cong H^{\bullet}(\Sigma_{g, 1}, \mathbb{K})\, ,     
\end{equation*}
where $H^{\bullet}(\Sigma_{g, 1}, \mathbb{K})$ denotes the unital cohomology ring over $\mathbb{K}$ of a genus-$g$ surface $\Sigma_{g,1}$ with one puncture and $2g=\mathrm{tb}(\Lambda(\beta))+1$.     
\end{theorem}
\begin{proof}
First, observe that part \text{(1)} follows directly from Theorem~\eqref{Theorem: Ext1 as the cokernel of delta_F,G}; nevertheless, we briefly outline the main arguments of its proof, as this will naturally set the stage for the proof of part \text{(2)}. 

To begin, consider the linear map $\delta_{\sh{F},\sh{F}}:\mathbb{K}^{n}_{\mathrm{std}}\to \mathbb{K}^{\ell}_{\mathrm{std}}$ associated with $\sh{F}$, as introduced in Definition~\eqref{Def: linear map delta}. Then, building on Theorems~\eqref{Theorem: Ext0 as the kernel of delta_F,G} and~\eqref{Theorem: Ext1 as the cokernel of delta_F,G}, we identify $\mathrm{Ext}^{0}(\sh{F},\sh{F})$ with $\mathrm{ker}\, \delta_{\sh{F},\sh{F}}$ and $\mathrm{Ext}^{1}(\sh{F},\sh{F})$ with $\mathrm{coker}\, \delta_{\sh{F},\sh{F}}$. 

Let $\vec{u}=(u_{1}, \dots,  u_{n} )\in \mathrm{ker}\, \delta_{\sh{F},\sh{F}}\subseteq \mathbb{K}^{n}_{\mathrm{std}}$. By an argument analogous to that in the proof of Theorem~\eqref{Theorem: Graded endomorphism spaces in the knot case}, there exists $u\in \mathbb{K}$ such that $u_{i}=u$ for all $i\in [1,n]$, that is, $\vec{u}=u\,\vec{1}_{n}$, where $\vec{1}_{n}=(1,\dots, 1)\in \mathbb{K}^{n}_{\mathrm{std}}$. Hence, since $\vec{u}$ is arbitrary, we deduce that $\mathrm{ker}\, \delta_{\sh{F},\sh{F}} = \mathrm{Span}_{\,\mathbb{K}}\{\,\vec{1}_{n}\,\}\cong \mathbb{K}$. It follows from Theorem~\eqref{Theorem: Euler characteristic for Ext groups} that $\mathrm{dim}\,\mathrm{Ext}^{1}(\sh{F},\sh{F})=\mathrm{tb}(\Lambda(\beta))+\mathrm{dim}\,\mathrm{Ext}^{0}(\sh{F},\sh{F})$, and as a result, we conclude that $\mathrm{coker}\, \delta_{\sh{F},\sh{F}} \cong \mathbb{K}^{\mathrm{tb}(\Lambda(\beta))+1}$. Bearing this in mind, part \textit{(1)} follows:
\begin{equation*}
\mathrm{Ext}^{\bullet}(\sh{F},\sh{F}) \cong \mathbb{K}\left[0\right] \oplus \mathbb{K}^{\mathrm{tb}(\Lambda(\beta))+1}\left[1\right]\, .   
\end{equation*}

Next, we verify part \textit{(2)}. For this purpose, we introduce the following notation. Let $u\in \mathbb{K}$ and $\vec{x}\in \mathbb{K}^{\mathrm{tb}(\Lambda(\beta))+1}$. Then, we write 
\begin{equation*}
u\simeq \vec{u}\in \mathrm{ker}\, \delta_{\sh{F},\sh{F}}\, , \quad \text{and} \quad    \vec{x}\simeq [\,\vec{x}\,'\,]\in \mathrm{coker}\, \delta_{\sh{F},\sh{F}}\, , \quad \vec{x}\,'\in \mathbb{K}^{\ell}_{\mathrm{std}}
\end{equation*}
to denote that $\vec{u}$ and $[\,\vec{x}\,'\,]$ are the unique elements corresponding to $u$ and $\vec{x}$ under the isomorphism of graded vector spaces established in part~\textit{(1)}, respectively. With this notation at hand, we proceed to analyze the algebraic structure induced by the graded composition in $\mathrm{Ext}^{\bullet}(\sh{F},\sh{F})$ by considering the following three cases: 

\noindent
$\bullet$ \textit{Case 1} ($\mathrm{Ext}^{0}(\sh{F},\sh{F})\times \mathrm{Ext}^{0}(\sh{F},\sh{F})$): Let $u, v\in \mathbb{K}$, and consider $\vec{u}, \vec{v}\in  \mathrm{ker}\, \delta_{\sh{F},\sh{F}}$ such that  $u\simeq \vec{u}$ and $v\simeq \vec{v}$, namely, $\vec{u}=u\,\vec{1}_{n}$ and $\vec{v}=v\,\vec{1}_{n}$. By Theorem~\eqref{Theorem: Graded composition for Ext0 and Ext0}, we know that $\vec{v}\circ \vec{u}=\vec{v}\odot \vec{u} \in \mathrm{ker}\, \delta_{\sh{F},\sh{F}}$, which implies that $\vec{v}\circ \vec{u}=vu\,\vec{1}_{n}$. Hence, if we define the composition
\begin{equation*}
 \begin{array}{ccccc}
      \circ :& \mathbb{K}\times \mathbb{K} &\to& \mathbb{K} & \\[2pt]
      &  (v,u) &\mapsto& v \circ u &
 \end{array}   
\end{equation*}
via the assignment $v\circ u\simeq \vec{v}\circ \vec{u}$, we deduce that $v\circ u=vu$. 

\noindent
$\bullet$ \textit{Case 2} ($\mathrm{Ext}^{1}(\sh{F},\sh{F})\times \mathrm{Ext}^{0}(\sh{F},\sh{F})$): Let $u\in \mathbb{K}$ and $\vec{y}\in \mathbb{K}^{\mathrm{tb}(\Lambda(\beta))+1}$, and consider $\vec{u}\in  \mathrm{ker}\, \delta_{\sh{F},\sh{F}}$ and $[\,\vec{y}\,'\,]\in \mathrm{coker}\, \delta_{\sh{F},\sh{F}}$ such that $u\simeq \vec{u}$ and $\vec{y}\simeq [\,\vec{y}\,'\,]$ for some $\vec{y}\,'=(y'_{1}, \dots, y'_{\ell})\in \mathbb{K}^{\ell}_{\mathrm{std}}$; in particular, $\vec{u}=(u_{1}, \dots, u_{n})\in \mathbb{K}^{n}_{\mathrm{std}}$ with $u_{i}=u$ for all $i\in [1,n]$, namely $\vec{u}=u\,\vec{1}_{n}$. By Theorem~\eqref{Theorem: Graded composition for Ext0 and Ext1}, we know that $[\,\vec{y}\,']\circ \vec{u}= [\,\vec{y}\,'\circ_{\beta_{\mathrm{R}}} \vec{u}\,]$, where $\vec{y}\,'\circ_{\beta_{\mathrm{R}}} \vec{u} =(\tilde{y}_{1},\dots, \tilde{y}_{\ell})\in \mathbb{K}^{\ell}_{\mathrm{std}}$ is entrywise given by
\begin{equation*}
\tilde{y}_{j}=y'_{j}\cdot u_{\pi_{\beta_{j}}(i_{j})}\, ,    
\end{equation*}
for all $j\in [1,\ell]$. Thus, since $u_{\pi_{\beta_{j}}(i_{j})}=u$ for all $j\in [1,\ell]$, we obtain that $\vec{y}\,'\circ_{\beta_{\mathrm{R}}}\vec{u}=u\,\vec{y}\,'$, and as a result, we also deduce that $u\,\vec{y} \simeq u\,[\,\vec{y}\,'\,]=[\, \vec{y}\,']\circ\,\vec{u} $. Hence, if we define the composition
\begin{equation*}
 \begin{array}{ccccc}
      \circ :& \mathbb{K}^{\mathrm{tb}(\Lambda(\beta))+1}\times \mathbb{K} &\to& \mathbb{K}^{\mathrm{tb}(\Lambda(\beta))+1} & \\[2pt]
      &  (\,\vec{y}\,,u\,) &\mapsto& \vec{y}\, \circ u &
 \end{array}   
\end{equation*}
via the assignment $\vec{y}\,\circ u\simeq [\,\vec{y}\,']\circ\, \vec{u}$, we conclude that $\vec{y}\,\circ u = u \vec{y}$. 

\noindent
$\bullet$ \textit{Case 3} ($\mathrm{Ext}^{0}(\sh{F},\sh{F})\times \mathrm{Ext}^{1}(\sh{F},\sh{F})$): Let $v\in \mathbb{K}$ and $\vec{x}\in \mathbb{K}^{\mathrm{tb}(\Lambda(\beta))+1}$, and consider  $\vec{v}\in  \mathrm{ker}\, \delta_{\sh{F},\sh{F}} $ and $[\,\vec{x}\,'\,]\in \mathrm{coker}\, \delta_{\sh{F},\sh{F}}$ such that $v\simeq \vec{v}$ and  $\vec{x}\simeq[\,\vec{x}\,'\,]$ for some $\vec{x}\,'=(x'_{1}, \dots, x'_{\ell})\in \mathbb{K}^{\ell}_{\mathrm{std}}$; in particular, $\vec{v}=(v_{1}, \dots, v_{n})\in \mathbb{K}^{n}_{\mathrm{std}}$ with $v_{i}=v$ for all $i\in [1,n]$, namely $\vec{v}=v\,\vec{1}_{n}$. By Theorem~\eqref{Theorem: Graded composition for Ext1 and Ext0}, we know that $\vec{v}\,\circ[\,\vec{x}\,']=[\,\vec{v}\,\circ_{\beta_{\mathrm{L}}}\vec{x}\,']$, where $\vec{v}\,\circ_{\beta_{\mathrm{L}}} \vec{x}\,'=(\tilde{x}_{1}, \dots, \tilde{x}_{\ell})\in \mathbb{K}^{\ell}_{\mathrm{std}}$ is entrywise given by
\begin{equation*}
\tilde{x}_{j}=v_{\pi_{\beta_{j}}(i_{j}+1)}\cdot x'_{j} \, ,    
\end{equation*}
for all $j\in [1,\ell]$. Thus, since $v_{\pi_{\beta_{j}}(i_{j}+1)}=v$ for all $j\in [1,\ell]$, we obtain that $\vec{v}\,\circ_{\beta_{\mathrm{L}}}\vec{x}\,'=v\,\vec{x}\,'$, and as a result, we also deduce that $v\,\vec{x} \simeq v\,[\,\vec{x}\,'\,]=\vec{v}\,\circ[\,\vec{x}\,']$. Hence, if we define the composition
\begin{equation*}
 \begin{array}{ccccc}
      \circ :& \mathbb{K}\times \mathbb{K}^{\mathrm{tb}(\Lambda(\beta))+1} &\to& \mathbb{K}^{\mathrm{tb}(\Lambda(\beta))+1} & \\[2pt]
      &  (\, v, \vec{x}\,) &\mapsto& v\, \circ\,\vec{x} &
 \end{array}   
\end{equation*}
via the assignment $v\,\circ \vec{x} \simeq \vec{v}\,\circ[\,\vec{x}\,']$, we conclude that $v\circ\vec{x}= v\,\vec{x}$.

In particular, combining the above cases, we obtain the graded composition rule
\begin{equation*}
\begin{array}{c c c}
\mathrm{Ext}^{\bullet}(\sh{F},\sh{F}) \times \mathrm{Ext}^{\bullet}(\sh{F},\sh{F}) & ~~\longrightarrow~~ & \mathrm{Ext}^{\bullet}(\sh{F},\sh{F}) \,,\\[6pt]
 \big((u, \vec{x}\,) , (v, \vec{y}\,)\big)   & ~~\longmapsto~~ & (v,\vec{y}\,)\circ (u,\vec{x}) \;=\; (vu,\, v\vec{x} + u\vec{y}\,) \, ,
\end{array}
\end{equation*}
for all $u,\,v \in \mathbb{K}$ and $\vec{x},\, \vec{y} \in \mathbb{K}^{\mathrm{tb}(\Lambda(\beta))+1}$. Bearing this in mind, part \textit{(2)} follows.

Now, let $\Sigma_{g,1}$ be a genus $g$-surface with one puncture. As is well known, the unital graded cohomological ring $H^{\bullet}(\Sigma_{g,1}, \mathbb{K})$ of $\Sigma_{g,1}$ over $\mathbb{K}$ is given by 
\begin{equation*}
H^{\bullet}(\Sigma_{g,1}, \mathbb{K})= \mathbb{K}[0] \oplus \mathbb{K}^{2g}[1]\, ,   
\end{equation*}
where the graded composition rule induced by the cup product is defined by
\begin{equation*}
\begin{array}{c c c}
H^{\bullet}(\Sigma_{g,1}, \mathbb{K}) \times H^{\bullet}(\Sigma_{g,1}, \mathbb{K})& ~\longrightarrow~ &H^{\bullet}(\Sigma_{g,1}, \mathbb{K}) \,,\\[4pt]
\big((u, \vec{x}\,) , (v, \vec{y}\,)\big)   & ~\longmapsto~ & (v,\vec{y}\,)\circ (u,\vec{x}) \;=\; (vu,\, v\vec{x} + u\vec{y}\,) \, ,
\end{array}    
\end{equation*}
for all $u,\,v \in \mathbb{K}$ and $\vec{x},\, \vec{y} \in \mathbb{K}^{2g}$.

Finally, recall that since by definition $\Lambda(\beta)$ is a Legendrian knot with maximal Thurston--Bennequin number, arising as the rainbow closure of a positive braid word $\beta$, $\mathrm{tb}(\Lambda(\beta))$ is an odd integer. Thus, if we consider $\Sigma_{g,1}$ such that $2g=\mathrm{tb}(\Lambda(\beta)) +1$, we conclude that there exists an isomorphism of unital graded rings
\begin{equation*}
\mathrm{Ext}^{\bullet}(\sh{F},\sh{F}) \cong H^{\bullet}(\Sigma_{g, 1}, \mathbb{K})\, . 
\end{equation*}
This completes the proof.
\end{proof}

\noindent
Conceptually, the previous result may be viewed as a sheaf-theoretic analogue of the \textit{Seidel isomorphism theorem}~\cite{R1}. From another perspective, and building on the insightful work of Chantraine~\cite{C1}, Theorem~\eqref{Theorem: Graded endomorphism spaces in the knot case} shows that, when $\Lambda(\beta)$ is a Legendrian knot, the unital graded ring structure of the graded endomorphism spaces in the category $\ccs{1}{\beta}$ is fully determined by the smooth topology of the possible exact Lagrangian fillings of $\Lambda(\beta)$. This reveals the rich algebraic, geometric, and topological structures encoded by the categorical invariant under consideration and illustrates the power of the sheaf-theoretic framework we have developed in this paper. 

With the above result at hand, we conclude our discussion compiling some structural aspects of the category $\ccs{1}{\beta}$ in the Legendrian knot case. Building on this, we proceed to analyze in detail some aspects of the category $\ccs{1}{\beta}$ in the case where $\mathbb{K}=\mathbb{Z}_{2}$ and $\Lambda(\beta)$ is a concrete Legendrian knot, namely the Legendrian trefoil knot on three strands.

\subsubsection{The Case of the Trefoil Knot} Let $\beta_{\mathrm{T}}:=\sigma_{1}\sigma_{2}\sigma_{1}\sigma_{2}\in \mathrm{Br}^{+}_{3}$. In this case, $\Lambda(\beta_{\mathrm{T}})\subset (\mathbb{R}^{3}, \xi_{\mathrm{std}})$ is Legendrian trefoil knot presented as the rainbow closure of a positive braid word on three strands. Next, we apply the framework we developed in the previous sections to provide a detailed description of certain aspects of the category $H^{\bullet}(\mathcal{S}h_{1}(\Lambda(\beta_{\mathrm{T}}), \mathbb{Z}_{2})_{0})$. In particular, since this category is a Legendrian isotopy invariant, our results agree with those in~\cite{STZ1}, where the Legendrian trefoil knot is studied as the rainbow closure of a positive braid on two strands. What is new here is the explicit description of the compositions of graded morphisms, which is not discussed in~\cite{STZ1}.

\noindent
$\bullet$ \textit{Setup}: To begin, observe that, for any $\vec{z}=(z_{1}, z_{2}, z_{3}, z_{4})\in \mathbb{Z}^{4}_{2}$, the path matrix associated with $\beta_{\mathrm{T}}$ is given by
\begin{equation*}
\displaystyle P_{\beta_{\mathrm{T}}}(\vec{z}\,)= \left[\;\begin{array}{c@{\hspace{12pt}}c@{\hspace{12pt}}c}
z_{1}z_{3}+z_{2} & z_{1}z_{4}+1 & z_{1}\\
z_{3}            &   z_{4}   & 1\\
1                &   0   & 0
\end{array}\;\right] \in \mathrm{GL}(3,\mathbb{Z}_{2})\, .    
\end{equation*}
Consequently, the braid variety associated with $\beta_{\mathrm{T}}$ is defined by
\begin{equation*}
X(\beta_{\mathrm{T}}, \mathbb{Z}_{2})=\left\{ \vec{z}=(z_{1}, z_{2}, z_{3}, z_{4})\in \mathbb{Z}^{4}_{2}\;|\; z_{1}z_{3}+z_{2}=1\,, ~ z_{2}z_{4}-z_{3}=1  \right\}\subset \mathbb{Z}^{4}_{2} \, . 
\end{equation*}
Accordingly, a straightforward calculation shows that $X(\beta_{\mathrm{H}}, \mathbb{Z}_{2})$ consists of five points:
\begin{equation*}
X(\beta_{\mathrm{T}}, \mathbb{Z}_{2})=\left\{ \vec{z}_{1}=(0,1,0,1)\,, \quad \vec{z}_{2}=(0,1,1,0)\, , \quad \vec{z}_{3}=(1,0,1,0)\, , \quad \vec{z}_{4}=(1,0,1,1)\, , \quad \vec{z}_{5}=(1,1,0,1)  \right\}\, .  
\end{equation*}

By Theorem~\eqref{Prop. for sheaves and braid matrices}, the objects of the category $H^{*}(\mathcal{S}h_{1}(\Lambda(\beta_{\mathrm{T}}), \mathbb{Z}_{2})_{0})$ are parametrized by a basis for $\mathbb{Z}^{3}_{2}$ together with a point in the braid variety $X(\beta_{\mathrm{T}}, \mathbb{Z}_{2})$. Bearing this in mind, we write
\begin{equation*}
\mathrm{Ob}(H^{*}(\mathcal{S}h_{1}(\Lambda(\beta_{\mathrm{T}}), \mathbb{Z}_{2})_{0}))= \left\{\,\left(\,\hat{\mathbf{f}}^{(3)}\,,\, \vec{z}\,\right)\,\Big|\, \text{$\hat{\mathbf{f}}^{(3)}$ is a basis for $\mathbb{Z}^{3}_{2}$,\, $\vec{z}\in X(\beta_{\mathrm{T}}, \mathbb{Z}_{2})$} \,\right\}\, .    
\end{equation*}

Under the preceding identification, the linear structure of the graded morphism spaces in the category $H^{\bullet}(\mathcal{S}h_{1}(\Lambda(\beta_{\mathrm{T}}), \mathbb{Z}_{2})_{0})$ is described as follows: let $\sh{F}=\left(\,\hat{\mathbf{f}}^{(3)}\,,\, \vec{x}\,\right)$ and $\sh{G}=\left(\,\hat{\mathbf{g}}^{(3)}\,,\, \vec{y}\,\right)$ be objects of the category $H^{*}(\mathcal{S}h_{1}(\Lambda(\beta_{\mathrm{T}}), \mathbb{Z}_{2})_{0})$, and let $\delta_{\sh{F},\sh{G}}: \mathbb{Z}^{3}_{2}\to \mathbb{Z}^{4}_{2}$ be the linear maps associated with these objects, as introduced in Definition~\eqref{Def: linear map delta}. Specifically,
\begin{equation*}
\delta_{\sh{F},\sh{G}}(\vec{u})=(u_{1}x_{1}-y_{1}u_{2}, u_{1}x_{2}-y_{2}u_{3}, u_{2}x_{3} - y_{3}u_{3}, u_{2}x_{4} - y_{4}u_{1})\, ,\qquad \vec{u}=(u_{1}, u_{2}, u_{3})\in \mathbb{Z}^{3}_{2}\,,    
\end{equation*}
where $\vec{x}=(x_{1}, x_{2}, x_{3}, x_{4})$ and $\vec{y}=(y_{1}, y_{2}, y_{3}, y_{4})$. Then, building on Theorems~\eqref{Theorem: Ext0 as the kernel of delta_F,G} and~\eqref{Theorem: Ext1 as the cokernel of delta_F,G}, we write
\begin{equation*}
\mathrm{Ext}^{0}(\sh{F},\sh{G}) = \mathrm{ker}\,\delta_{\sh{F},\sh{G}}\, ,\quad \text{and} \quad \mathrm{Ext}^{1}(\sh{F},\sh{G}) = \mathrm{coker}\,\delta_{\sh{F},\sh{G}}\, .    
\end{equation*}
Later, when dealing with concrete examples and in order to simplify our discussion, we will work under the identification $\mathrm{Ext}^{1}(\sh{F},\sh{G})=W$, where $W$ corresponds to some fixed complement of $\mathrm{im}\, \delta_{\sh{F},\sh{G}}$ in $\mathbb{Z}^{4}_{2}$, that is, a vector subspace such that $\mathbb{Z}^{4}_{2}=\mathrm{im}\, \delta_{\sh{F},\sh{G}}\,\oplus\,W$. 

With the above identifications at hand, the graded composition between the graded morphism spaces in the category $H^{\bullet}(\mathcal{S}h_{1}(\Lambda(\beta_{\mathrm{T}}), \mathbb{Z}_{2})_{0})$ is characterized as follows: let $\sh{F}=\left(\,\hat{\mathbf{f}}^{(3)}\,,\, \vec{x}\,\right)$, $\sh{G}=\left(\,\hat{\mathbf{g}}^{(3)}\,,\, \vec{y}\,\right)$, and $\sh{Q}=\left(\,\hat{\mathbf{q}}^{(3)}\,,\, \vec{z}\,\right)$ be objects of the category $H^{\bullet}(\mathcal{S}h_{1}(\Lambda(\beta_{\mathrm{T}}), \mathbb{Z}_{2})_{0})$. Fix $\vec{u}=(u_{1}, u_{2}, u_{3})\in \mathrm{Ext}^{0}(\sh{F},\sh{G})$, $\vec{v}=(v_{1}, v_{2}, v_{3})\in \mathrm{Ext}^{0}(\sh{G},\sh{Q})$, $\Sigma\in\mathrm{Ext}^{1}(\sh{F}, \sh{G})$, and $\Gamma\in \mathrm{Ext}^{1}(\sh{G}, \sh{Q})$, and let $\vec{p}=(p_{1}, p_{2}, p_{3}, p_{4})\in \mathbb{Z}^{4}_{2}$ and $\vec{q}=(q_{1}, q_{2}, q_{3}, q_{4})\in \mathbb{Z}^{4}_{2}$ be representatives such that $\Sigma=[\,\vec{p}\,]$ and $\Gamma=[\,\vec{q}\,]$. Then, by Theorems~\eqref{Theorem: Graded composition for Ext0 and Ext0}, ~\eqref{Theorem: Graded composition for Ext1 and Ext0}, and~\eqref{Theorem: Graded composition for Ext0 and Ext1}, the graded compositions between the morphisms under consideration are given by:
\begin{equation*}
\begin{aligned}
\vec{v}\,\circ\,\vec{u}&=\vec{v}\odot \vec{u}\in \mathrm{Ext}^{0}(\sh{F},\sh{Q})\, , \qquad \vec{v}\odot \vec{u}=(v_{1}u_{1}, v_{2}u_{2}, v_{3}u_{3}, v_{4}u_{4})\in \mathbb{Z}^{3}_{2}\, ,\\
\Gamma\,\circ\,\vec{u}& = [\, \,\vec{q}\,\circ_{\beta_{\mathrm{R}}}\vec{u} \,]\in \mathrm{Ext}^{1}(\sh{F},\sh{Q})\, ,  \qquad \vec{q}\,\circ_{\beta_{\mathrm{R}}}\vec{u}=(q_{1}u_{2}, q_{2}u_{3}, q_{3}u_{3}, q_{4}u_{1})\in \mathbb{Z}^{4}_{2}\, ,\\
\vec{v}\,\circ\,\Sigma& = [\, \,\vec{v}\,\circ_{\beta_{\mathrm{L}}}\vec{p} \,]\in \mathrm{Ext}^{1}(\sh{F},\sh{Q})\, , \qquad \vec{v}\,\circ_{\beta_{\mathrm{L}}}\vec{p}=(v_{1}p_{1}, v_{1}p_{2}, v_{2}p_{3}, u_{2}p_{4})\in \mathbb{Z}^{4}_{2}\, .
\end{aligned}    
\end{equation*}
Later, when working with concrete examples and in order to simplify our discussion, we identify $\mathrm{Ext}^{1}(\sh{F},\sh{G})$, $\mathrm{Ext}^{1}(\sh{G},\sh{Q})$, $\mathrm{Ext}^{1}(\sh{F},\sh{Q})$ with some fixed complements $W_{1}$, $W_{2}$, $W_{3}$ of the images of the corresponding linear maps $\delta_{\sh{F},\sh{G}},\allowbreak\;\delta_{\sh{G},\sh{Q}},\allowbreak\;\delta_{\sh{F},\sh{Q}} : \mathbb{Z}^{3}_{2} \to \mathbb{Z}^{4}_{2}$ in $\mathbb{Z}^{4}_{2}$. In this setting, the mixed degree graded compositions are computed via the left $\circ_{\beta_{\mathrm{L}}}:\mathbb{Z}^{3}_{2}\times\mathbb{Z}^{4}_{2}\to \mathbb{Z}^{4}_{2} $ and right $\circ_{\beta_{\mathrm{R}}} :\mathbb{Z}^{4}_{2}\times\mathbb{Z}^{3}_{2}\to \mathbb{Z}^{4}_{2}$ braided compositions, and the resulting vectors in the ambient space $\mathbb{Z}^{4}_{2}$ are then projected onto $W_{3}$ for consistency. 

\noindent
$\bullet$ \textit{Some Concrete Computations}: To illustrate how to apply the setup we established above, we consider five explicit objects of the category $H^{\bullet}(\mathcal{S}h_{1}(\Lambda(\beta_{\mathrm{T}}), \mathbb{Z}_{2})_{0})$: 
\begin{equation*}
\left\{ \sh{F}_{i}= \left(\,\hat{\mathbf{f}}^{(3)}_{i}\,,\, \vec{z}_{i}\,\right) \; \middle| \; \begin{array}{c}
\text{${\mathbf{f}}^{(3)}_{i}$ is a fixed basis for $\mathbb{Z}^{3}_{2}$}\,,\\ \text{$\vec{z}_{i}$ is the $i$-th distinct point in $X(\beta_{\mathrm{T}}, \mathbb{Z}_{2})$}\,,\\ i\in [1,5]     
\end{array}   \right\}\, .  
\end{equation*}

Next, we analyze in detail the algebraic properties of some of the graded morphism spaces associated with the objects $\big\{\sh{F}_{i} \big\}_{i=1}^{n}$. In particular, we begin by examining the ring structure of the graded endomorphism space $\mathrm{End}^{\bullet}(\sh{F}_{1})$, summarized for clarity in Table~\ref{tab:endomorphisms1knot}. From this table, we see that $\mathrm{End}^{\bullet}(\sh{F}_{1})$ has a particularly simple ring structure, namely that of the cohomology ring of a genus-one surface with one puncture. This is not a coincidence: since $\mathrm{tb}(\Lambda(\beta_{\mathrm{T}}))=\ell-n=1$, Theorem~\eqref{Theorem: Graded endomorphism spaces in the knot case} asserts that, for any $i\in [1,5]$, there is an isomorphism of unital graded rings
\begin{equation*}
\mathrm{End}^{\bullet}(\sh{F}_{i}) \cong H^{\bullet}(\Sigma_{1,1}, \mathbb{Z}_{2})\, ,     
\end{equation*}
where $H^{\bullet}(\Sigma_{1,1}, \mathbb{Z}_{2})$ denotes the unital cohomology ring over $\mathbb{Z}_{2}$ of a genus-one surface with one puncture. It follows that the algebraic structure of the graded endomorphism spaces of the objects $\big\{\sh{F}_{i}\big\}_{i=1}^{5}$ is particularly simple and depends primarily on intrinsic data associated with the Legendrian knot $\Lambda(\beta_{\mathrm{T}})$.   

\begin{table}[ht!]
\centering
\renewcommand{\arraystretch}{1.6} 
\setlength{\tabcolsep}{10pt}
\begin{tabular}{|c|}
\hline
Object: $\sh{F}_{1}$\\
\hline
Graded Endomorphism Spaces\\
\hline
$\mathrm{End}^{\,0}(\sh{F}_{1})=\mathrm{Span}_{\,\mathbb{Z}_{2}}\big\{ \hat{u}_{1}=\langle 1,1,1 \rangle\, \big\}$  \\
\hline
$\mathrm{End}^{\,1}(\sh{F}_{1})=\mathrm{Span}_{\,\mathbb{Z}_{2}}\big\{ \hat{\alpha}_{1}=\langle 1,0,0,0 \rangle\,, ~ \hat{\alpha}_{2}=\langle 0,0,1,0 \rangle\, \big\}$  \\
\hline
Graded Composition\\
\hline
$
\begin{array}{cccc}
\circ: &\mathrm{End}^{\,0}(\sh{F}_{1})\times \mathrm{End}^{\,0}(\sh{F}_{1})&\to& \mathrm{End}^{\,0}(\sh{F}_{1})\, ,\\[-3pt]
&\big(\,b_{1}\hat{u}_{1} \,, a_{1}\hat{u}_{1}\,\big)& \mapsto & b_{1}a_{1} \hat{u}_{1}\, .\\[4pt]
\end{array}
$\\
\hline
$
\begin{array}{cccc}
\circ: &\mathrm{End}^{\,1}(\sh{F}_{1})\times \mathrm{End}^{\,0}(\sh{F}_{1})&\to& \mathrm{End}^{\,1}(\sh{F}_{1})\, ,\\[-3pt]
&\big(\,q_{1}\hat{\alpha}_{1}+q_{2}\hat{\alpha}_{2} \,, a_{1}\hat{u}_{1}\,\big)& \mapsto & q_{1}a_{1} \hat{\alpha}_{1}+q_{2}a_{1}\hat{\alpha}_{2}\, .\\[4pt]
\end{array}
$\\
\hline
$
\begin{array}{cccc}
\circ: &\mathrm{End}^{\,0}(\sh{F}_{1})\times \mathrm{End}^{\,1}(\sh{F}_{1})&\to& \mathrm{End}^{\,1}(\sh{F}_{1})\, ,\\[-3pt]
&\big(\,b_{1}\hat{u}_{1}\,, p_{1}\hat{\alpha}_{1}+p_{2}\hat{\alpha}_{2}\,\big)& \mapsto & b_{1}p_{1}\hat{\alpha}_{1}+b_{1}p_{2}\hat{\alpha}_{2}\, .\\[4pt]
\end{array}
$\\
\hline
\end{tabular}
\caption{Unital ring structure of the graded endomorphism space $\mathrm{End}^{\bullet}(\sh{F}_{1})$.}
\label{tab:endomorphisms1knot}
\end{table}

We now extend our discussion by explicitly analyzing the graded morphism spaces and their compositions for the ordered triples $(\sh{F}_{1}, \sh{F}_{2}, \sh{F}_{3})$ and $(\sh{F}_{1}, \sh{F}_{1}, \sh{F}_{2})$.  
As summarized in Tables~\ref{tab:morphisms1knot} and~\ref{tab:morphisms2knot}, these examples illustrate two contrasting behaviors: when all objects are distinct, compositions are trivial, whereas when some objects coincide, nontrivial mixed-degree compositions appear.

\begin{table}[ht!]
\centering
\renewcommand{\arraystretch}{1.6} 
\setlength{\tabcolsep}{10pt}
\begin{tabular}{|c|}
\hline
Ordered Triple: ($\sh{F}_{1}, \sh{F}_{2}, \sh{F}_{3}$)\\
\hline
Graded morphism Spaces\\
\hline
$\begin{array}{c|c}
\mathrm{Ext}^{0}(\sh{F}_{1}, \sh{F}_{2})=0  ~&~ \mathrm{Ext}^{1}(\sh{F}_{1}, \sh{F}_{2})=\mathrm{Span}_{\,\mathbb{Z}_{2}}\big\{\,\hat{\alpha}_{1}=\langle 1,0,0,0 \rangle\,\big\} \\
\hline
\mathrm{Ext}^{0}(\sh{F}_{2}, \sh{F}_{3})=0  ~&~ \mathrm{Ext}^{1}(\sh{F}_{2}, \sh{F}_{3})=\mathrm{Span}_{\,\mathbb{Z}_{2}}\big\{\,\hat{\beta}_{1}=\langle 0,0,0,1 \rangle\,\big\} \\
\hline
\mathrm{Ext}^{0}(\sh{F}_{1}, \sh{F}_{3})=0 ~&~ \mathrm{Ext}^{1}(\sh{F}_{1}, \sh{F}_{3})=\mathrm{Span}_{\,\mathbb{Z}_{2}}\big\{\,\hat{\gamma}_{1}=\langle 1,0,0,1 \rangle\,\big\} \\
\end{array}$\\
\hline
Graded Composition\\
\hline
$
\begin{array}{cccc}
\circ: &\mathrm{Ext}^{0}(\sh{F}_{2},\sh{F}_{3})\times \mathrm{Ext}^{0}(\sh{F}_{1},\sh{F}_{2})&\to& \mathrm{Ext}^{0}(\sh{F}_{1},\sh{F}_{3})\, ,\\[-3pt]
&\big(\,0\,, 0\,\big)& \mapsto & 0\, .\\[4pt]
\end{array}
$\\
\hline
$
\begin{array}{cccc}
\circ: &\mathrm{Ext}^{1}(\sh{F}_{2},\sh{F}_{3})\times \mathrm{Ext}^{0}(\sh{F}_{1},\sh{F}_{2})&\to& \mathrm{Ext}^{1}(\sh{F}_{1},\sh{F}_{3})\, ,\\[-3pt]
&\big(\,q_{1}\hat{\beta}_{1} \,, 0\,\big)& \mapsto & 0\, .\\[4pt]
\end{array}
$\\
\hline
$
\begin{array}{cccc}
\circ: &\mathrm{Ext}^{0}(\sh{F}_{2},\sh{F}_{3})\times \mathrm{Ext}^{1}(\sh{F}_{1},\sh{F}_{2})&\to& \mathrm{Ext}^{1}(\sh{F}_{1},\sh{F}_{3})\, ,\\[-3pt]
&\big(\,0 \,, p_{1}\hat{\alpha}_{1}\,\big)& \mapsto & 0\, .\\[4pt]
\end{array}
$\\
\hline
\end{tabular}
\caption{Graded morphism spaces and their compositions for the ordered triple $(\sh{F}_{1}, \sh{F}_{2}, \sh{F}_{3})$.}
\label{tab:morphisms1knot}
\end{table}

\begin{table}[ht!]
\centering
\renewcommand{\arraystretch}{1.6} 
\setlength{\tabcolsep}{10pt}
\begin{tabular}{|c|}
\hline
Ordered Triple: ($\sh{F}_{1}, \sh{F}_{1}, \sh{F}_{2}$)\\
\hline
Graded morphism Spaces\\
\hline
$\begin{array}{c|c}
\mathrm{Ext}^{0}(\sh{F}_{1}, \sh{F}_{1})=\mathrm{Span}_{\,\mathbb{Z}_{2}}\big\{ \hat{u}_{1}=\langle 1,1,1 \rangle\, \big\}  ~&~ \mathrm{Ext}^{1}(\sh{F}_{1},\sh{F}_{1} )=\mathrm{Span}_{\,\mathbb{Z}_{2}}\big\{ \hat{\alpha}_{1}=\langle 1,0,0,0 \rangle\,, ~ \hat{\alpha}_{2}=\langle 0,0,1,0 \rangle\, \big\}\\
\hline
\mathrm{Ext}^{0}(\sh{F}_{1}, \sh{F}_{2})=0  ~&~ \mathrm{Ext}^{1}(\sh{F}_{1}, \sh{F}_{2})=\mathrm{Span}_{\,\mathbb{Z}_{2}}\big\{\,\hat{\beta}_{1}=\langle 1,0,0,0 \rangle\,\big\} \\
\end{array}$\\
\hline
Graded Composition\\
\hline
$
\begin{array}{cccc}
\circ: &\mathrm{Ext}^{0}(\sh{F}_{1},\sh{F}_{2})\times \mathrm{Ext}^{0}(\sh{F}_{1},\sh{F}_{1})&\to& \mathrm{Ext}^{0}(\sh{F}_{1},\sh{F}_{2})\, ,\\[-3pt]
&\big(\,0\,, a_{1}\hat{u}_{1}\,\big)& \mapsto & 0\, .\\[4pt]
\end{array}
$\\
\hline
$
\begin{array}{cccc}
\circ: &\mathrm{Ext}^{1}(\sh{F}_{1},\sh{F}_{2})\times \mathrm{Ext}^{0}(\sh{F}_{1},\sh{F}_{1})&\to& \mathrm{Ext}^{1}(\sh{F}_{1},\sh{F}_{2})\, ,\\[-3pt]
&\big(\,q_{1}\hat{\beta}_{1} \,, a_{1}\hat{u}_{1}\,\big)& \mapsto & q_{1}a_{1}\hat{\beta}_{1}\, .\\[4pt]
\end{array}
$\\
\hline
$
\begin{array}{cccc}
\circ: &\mathrm{Ext}^{0}(\sh{F}_{1},\sh{F}_{2})\times \mathrm{Ext}^{1}(\sh{F}_{1},\sh{F}_{1})&\to& \mathrm{Ext}^{1}(\sh{F}_{1},\sh{F}_{2})\, ,\\[-3pt]
&\big(\,0 \,, p_{1}\hat{\alpha}_{1}+p_{2}\hat{\alpha}_{2}\,\big)& \mapsto & 0\, .\\[4pt]
\end{array}
$\\
\hline
\end{tabular}
\caption{Graded morphism spaces and their compositions for the ordered triple $(\sh{F}_{1}, \sh{F}_{1}, \sh{F}_{2})$.}
\label{tab:morphisms2knot}
\end{table}

Finally, we conclude our discussion by examining the linear algebraic structure of the graded morphism spaces associated with the objects $\big\{\sh{F}_{i}\big\}_{i=1}^{5}$.
For this purpose, recall that $\mathrm{tb}(\Lambda(\beta_{\mathrm{T}}))=1$. Then, since the points $\vec{z}_{i}$ are all distinct and the torus action on $X(\beta_{\mathrm{T}}, \mathbb{Z}_{2})$ is trivial, Theorem~\eqref{Theorem: Graded morphism spaces for the knot case} yields an isomorphism of graded vector spaces
\begin{equation*}
\mathrm{Ext}^{\bullet}(\sh{F}_{i}, \sh{F}_{j}) \cong
\begin{cases}
\mathbb{Z}_{2}[0] \oplus \mathbb{Z}^{2}_{2}[1], & \text{if } i=j,\\[2pt]
\hfil 0[0] \oplus \mathbb{Z}_{2}[1], & \text{if } i\neq j,
\end{cases}
\end{equation*}
which shows that the linear structure of the graded morphism spaces associated with the objects $\big\{\sh{F}_{i}\big\}_{i=1}^{5}$ is particularly simple and largely determined by intrinsic data associated with the Legendrian knot $\Lambda(\beta_{\mathrm{T}})$.

Having established this, we conclude our study of the category $\ccs{1}{\beta}$ in the case where $\Lambda(\beta)$ is a Legendrian knot. The general case of Legendrian links is richer in structure and depends more heavily on the choice of the positive braid word $\beta$. In the next subsection, we apply the machinery we developed in the preceding sections to a simple yet illuminating example that offers significant insight into the structure of the category $\ccs{1}{\beta}$ in the case where $\mathbb{K}=\mathbb{Z}_{2}$ and $\Lambda(\beta)$ is a Legendrian link, namely the Legendrian Hopf link on three strands.

\subsection{The Case of the Hopf Link}
Let $\beta_{\mathrm{H}}:=\sigma_{1}\sigma_{2}\sigma_{1}\in \mathrm{Br}^{+}_{3}$. In this case, $\Lambda(\beta_{\mathrm{H}})\subset (\mathbb{R}^{3}, \xi_{\mathrm{std}})$ is the Legendrian Hopf link presented as the rainbow closure of a positive braid word on three strands. Next, we apply the framework we developed in the previous sections to provide a detailed description of certain aspects of the category $H^{\bullet}(\mathcal{S}h_{1}(\Lambda(\beta_{\mathrm{H}}), \mathbb{Z}_{2})_{0})$. 

\noindent
$\bullet$ \textit{Setup}: To begin, observe that, for any $\vec{z}=(z_{1}, z_{2}, z_{3})\in \mathbb{Z}^{3}_{2}$, the path matrix associated with $\beta_{\mathrm{H}}$ is given by
\begin{equation*}
\displaystyle P_{\beta_{\mathrm{H}}}(\vec{z}\,)= \left[\;\begin{array}{c@{\hspace{12pt}}c@{\hspace{12pt}}c}
z_{1}z_{3}+z_{2} & z_{1} & 1\\
z_{3}            &   1   & 0\\
1                &   0   & 0
\end{array}\;\right] \in \mathrm{GL}(3,\mathbb{Z}_{2})\, .    
\end{equation*}
Consequently, the braid variety associated with $\beta_{\mathrm{H}}$ is defined by
\begin{equation*}
X(\beta_{\mathrm{H}}, \mathbb{Z}_{2})=\left\{ \vec{z}=(z_{1}, z_{2}, z_{3})\in \mathbb{Z}^{3}_{2}\;|\; z_{1}z_{3}+z_{2}=1\,, ~ z_{2}=1  \right\}\subset \mathbb{Z}^{3}_{2} \, . 
\end{equation*}
Accordingly, a straightforward calculation shows that $X(\beta_{\mathrm{H}}, \mathbb{Z}_{2})$ consists of three points:
\begin{equation*}
X(\beta_{\mathrm{H}}, \mathbb{Z}_{2})=\left\{ \vec{z}_{1}=(0,1,0)\,, \quad \vec{z}_{2}=(0,1,1)\, , \quad \vec{z}_{3}=(1,1,0)  \right\}\, .  
\end{equation*}

By Theorem~\eqref{Prop. for sheaves and braid matrices}, the objects of the category $H^{\bullet}(\mathcal{S}h_{1}(\Lambda(\beta_{\mathrm{H}}), \mathbb{Z}_{2})_{0})$ are parametrized by a choice of basis for $\mathbb{Z}^{3}_{2}$ together with a point in the braid variety $X(\beta_{\mathrm{H}}, \mathbb{Z}_{2})$. Bearing this in mind, we write
\begin{equation*}
\mathrm{Ob}(H^{\bullet}(\mathcal{S}h_{1}(\Lambda(\beta_{\mathrm{H}}), \mathbb{Z}_{2})_{0}))= \left\{\,\left(\,\hat{\mathbf{f}}^{(3)}\,,\, \vec{z}\,\right)\,\Big|\, \text{$\hat{\mathbf{f}}^{(3)}$ is a basis for $\mathbb{Z}^{3}_{2}$,\, $\vec{z}\in X(\beta_{\mathrm{H}}, \mathbb{Z}_{2})$} \,\right\}\, .    
\end{equation*}

Under the preceding identification, the linear structure of the graded morphism spaces in the category $H^{\bullet}(\mathcal{S}h_{1}(\Lambda(\beta_{\mathrm{H}}), \mathbb{Z}_{2})_{0})$ is described as follows: let $\sh{F}=\left(\,\hat{\mathbf{f}}^{(3)}\,,\, \vec{x}\,\right)$ and $\sh{G}=\left(\,\hat{\mathbf{g}}^{(3)}\,,\, \vec{y}\,\right)$ be objects of the category $H^{\bullet}(\mathcal{S}h_{1}(\Lambda(\beta_{\mathrm{H}}), \mathbb{Z}_{2})_{0})$, and let $\delta_{\sh{F},\sh{G}}: \mathbb{Z}^{3}_{2}\to \mathbb{Z}^{3}_{2}$ be the linear maps associated with these objects, as introduced in Definition~\eqref{Def: linear map delta}. Specifically,
\begin{equation*}
\delta_{\sh{F},\sh{G}}(\vec{u})=(u_{1}x_{1}-y_{1}u_{2}, u_{1}x_{2}-y_{2}u_{3}, u_{2}x_{3} - y_{3}u_{3})\, ,\qquad \vec{u}=(u_{1}, u_{2}, u_{3})\in \mathbb{Z}^{3}_{2}\,,    
\end{equation*}
where $\vec{x}=(x_{1}, x_{2}, x_{3})$ and $\vec{y}=(y_{1}, y_{2}, y_{3})$. Then, building on Theorems~\eqref{Theorem: Ext0 as the kernel of delta_F,G} and~\eqref{Theorem: Ext1 as the cokernel of delta_F,G}, we write
\begin{equation*}
\mathrm{Ext}^{0}(\sh{F},\sh{G}) = \mathrm{ker}\,\delta_{\sh{F},\sh{G}}\, ,\quad \text{and} \quad \mathrm{Ext}^{1}(\sh{F},\sh{G}) = \mathrm{coker}\,\delta_{\sh{F},\sh{G}}\, .    
\end{equation*}
Later, when dealing with concrete examples and in order to simplify our discussion, we will work under the identification $\mathrm{Ext}^{1}(\sh{F},\sh{G})=W$, where $W$ corresponds to some fixed complement of $\mathrm{im}\, \delta_{\sh{F},\sh{G}}$ in $\mathbb{Z}^{3}_{2}$, that is, a vector subspace such that $\mathbb{Z}^{3}_{2}=\mathrm{im}\, \delta_{\sh{F},\sh{G}}\,\oplus\,W$. 

In particular, with the preceding identifications at hand, the composition of graded morphisms in the category $H^{\bullet}(\mathcal{S}h_{1}(\Lambda(\beta_{\mathrm{H}}), \mathbb{Z}_{2})_{0})$ is characterized as follows: let $\sh{F}=\left(\,\hat{\mathbf{f}}^{(3)}\,,\, \vec{x}\,\right)$, $\sh{G}=\left(\,\hat{\mathbf{g}}^{(3)}\,,\, \vec{y}\,\right)$, and $\sh{Q}=\left(\,\hat{\mathbf{q}}^{(3)}\,,\, \vec{z}\,\right)$ be objects of the category $H^{\bullet}(\mathcal{S}h_{1}(\Lambda(\beta_{\mathrm{H}}), \mathbb{Z}_{2})_{0})$. Fix $\vec{u}=(u_{1}, u_{2}, u_{3})\in \mathrm{Ext}^{0}(\sh{F},\sh{G})$, $\vec{v}=(v_{1}, v_{2}, v_{3})\in \mathrm{Ext}^{0}(\sh{G},\sh{Q})$, $\Sigma\in\mathrm{Ext}^{1}(\sh{F}, \sh{G})$, and $\Gamma\in \mathrm{Ext}^{1}(\sh{G}, \sh{Q})$, and let $\vec{p}=(p_{1}, p_{2}, p_{3})\in \mathbb{Z}^{3}_{2}$ and $\vec{q}=(q_{1}, q_{2}, q_{3})\in \mathbb{Z}^{3}_{2}$ be representatives such that $\Sigma=[\,\vec{p}\,]$ and $\Gamma=[\,\vec{q}\,]$. Then, by Theorems~\eqref{Theorem: Graded composition for Ext0 and Ext0},~\eqref{Theorem: Graded composition for Ext1 and Ext0}, and~\eqref{Theorem: Graded composition for Ext0 and Ext1}, the graded compositions between the morphisms under consideration are given by:
\begin{equation*}
\begin{aligned}
\vec{v}\,\circ\,\vec{u}&=\vec{v}\odot \vec{u}\in \mathrm{Ext}^{0}(\sh{F},\sh{Q})\, , \qquad \vec{v}\odot \vec{u}=(v_{1}u_{1}, v_{2}u_{2}, v_{3}u_{3})\in \mathbb{Z}^{3}_{2}\, ,\\
\Gamma\,\circ\,\vec{u}& = [\, \,\vec{q}\,\circ_{\beta_{\mathrm{R}}}\vec{u} \,]\in \mathrm{Ext}^{1}(\sh{F},\sh{Q})\, ,  \qquad \vec{q}\,\circ_{\beta_{\mathrm{R}}}\vec{u}=(q_{1}u_{2}, q_{2}u_{3}, q_{3}u_{3})\in \mathbb{Z}^{3}_{2}\, ,\\
\vec{v}\,\circ\,\Sigma& = [\, \,\vec{v}\,\circ_{\beta_{\mathrm{L}}}\vec{p} \,]\in \mathrm{Ext}^{1}(\sh{F},\sh{Q})\, , \qquad \vec{v}\,\circ_{\beta_{\mathrm{L}}}\vec{p}=(v_{1}p_{1}, v_{1}p_{2}, v_{2}p_{3})\in \mathbb{Z}^{3}_{2}\, .
\end{aligned}    
\end{equation*}
Later, when working with concrete examples and in order to simplify our discussion, we identify $\mathrm{Ext}^{1}(\sh{F},\sh{G})$, $\mathrm{Ext}^{1}(\sh{G},\sh{Q})$, $\mathrm{Ext}^{1}(\sh{F},\sh{Q})$ with some fixed complements $W_{1}$, $W_{2}$, $W_{3}$ of the corresponding linear maps $\delta_{\sh{F},\sh{G}},\allowbreak\;\delta_{\sh{G},\sh{Q}},\allowbreak\;\delta_{\sh{F},\sh{Q}} : \mathbb{Z}^{3}_{2} \to \mathbb{Z}^{3}_{2}$ in $\mathbb{Z}^{3}_{2}$. In this setting, the mixed degree graded compositions are computed via the left $\circ_{\beta_{\mathrm{L}}}:\mathbb{Z}^{3}_{2}\times\mathbb{Z}^{3}_{2}\to \mathbb{Z}^{3}_{2} $ and right $\circ_{\beta_{\mathrm{R}}} :\mathbb{Z}^{3}_{2}\times\mathbb{Z}^{3}_{2}\to \mathbb{Z}^{3}_{2}$ braided compositions, and the resulting vectors in the ambient space $\mathbb{Z}^{3}_{2}$ are then projected onto $W_{3}$ for consistency. 

\noindent
$\bullet$ \textit{Some Concrete Computations}: To illustrate how to apply the framework we developed in the previous sections, we consider three concrete objects of the category $H^{\bullet}(\mathcal{S}h_{1}(\Lambda(\beta_{\mathrm{H}}), \mathbb{Z}_{2})_{0})$: 
\begin{equation*}
\sh{F}_{1}= \left(\,\hat{\mathbf{f}}^{(3)}\,,\, \vec{z}_{1}\,\right)\, , \qquad \sh{F}_{2}= \left(\,\hat{\mathbf{g}}^{(3)}\,,\, \vec{z}_{2}\,\right)\, , \qquad \sh{F}_{3}= \left(\,\hat{\mathbf{q}}^{(3)}\,,\, \vec{z}_{3}\,\right)\, ,
\end{equation*}
where $\hat{\mathbf{f}}^{(3)}$, $\hat{\mathbf{g}}^{(3)}$, $\hat{\mathbf{q}}^{(3)}$ are fixed bases for $\mathbb{Z}^{3}_{2}$, and $\vec{z}_{1}$, $\vec{z}_{2}$, $\vec{z}_{3}$ are the three distinct points in the braid variety $X(\beta_{\mathrm{H}},\mathbb{Z}_{2})$.  

Next, we explicitly describe the algebraic properties of some of the graded morphism spaces associated with $\sh{F}_{1}$, $\sh{F}_{2}$, $\sh{F}_{3}$. To begin, note that the ring structure of the graded endomorphism space $\mathrm{End}^{\bullet}(\sh{F}_{1})$ is summarized in Table~\ref{tab:endomorphisms1}, which reveals that $\mathrm{End}^{0\,}(\sh{F}_{1})$ and $\mathrm{End}^{\,1}(\sh{F}_{1})$ are 2-dimensional over $\mathbb{Z}_{2}$, with the graded composition exhibiting a non-trivial, rich algebraic structure.

\begin{table}[ht!]
\centering
\renewcommand{\arraystretch}{1.6} 
\setlength{\tabcolsep}{10pt}
\begin{tabular}{|c|}
\hline
Object: $\sh{F}_{1}$\\
\hline
Graded Endomorphism Spaces\\
\hline
$\mathrm{End}^{\,0}(\sh{F}_{1})=\mathrm{Span}_{\,\mathbb{Z}_{2}}\big\{ \hat{u}_{1}=\langle 0,1,0 \rangle\,, ~ \hat{u}_{2}=\langle 1,0,1 \rangle\, \big\}$  \\
\hline
$\mathrm{End}^{\,1}(\sh{F}_{1})=\mathrm{Span}_{\,\mathbb{Z}_{2}}\big\{ \hat{\alpha}_{1}=\langle 1,0,0 \rangle\,, ~ \hat{\alpha}_{2}=\langle 0,0,1 \rangle\, \big\}$  \\
\hline
Graded Composition\\
\hline
$
\begin{array}{cccc}
\circ: &\mathrm{End}^{\,0}(\sh{F}_{1})\times \mathrm{End}^{\,0}(\sh{F}_{1})&\to& \mathrm{End}^{\,0}(\sh{F}_{1})\, ,\\[-3pt]
&\big(\,b_{1}\hat{u}_{1}+b_{2}\hat{u}_{2} \,, a_{1}\hat{u}_{1}+a_{2}\hat{u}_{2}\,\big)& \mapsto & b_{1}a_{1} \hat{u}_{1}+b_{2}a_{2}\hat{u}_{2}\, .\\[4pt]
\end{array}
$\\
\hline
$
\begin{array}{cccc}
\circ: &\mathrm{End}^{\,1}(\sh{F}_{1})\times \mathrm{End}^{\,0}(\sh{F}_{1})&\to& \mathrm{End}^{\,1}(\sh{F}_{1})\, ,\\[-3pt]
&\big(\,q_{1}\hat{\alpha}_{1}+q_{2}\hat{\alpha}_{2} \,, a_{1}\hat{u}_{1}+a_{2}\hat{u}_{2}\,\big)& \mapsto & q_{1}a_{1} \hat{\alpha}_{1}+q_{2}a_{2}\hat{\alpha}_{2}\, .\\[4pt]
\end{array}
$\\
\hline
$
\begin{array}{cccc}
\circ: &\mathrm{End}^{\,0}(\sh{F}_{1})\times \mathrm{End}^{\,1}(\sh{F}_{1})&\to& \mathrm{End}^{\,1}(\sh{F}_{1})\, ,\\[-3pt]
&\big(\,b_{1}\hat{u}_{1}+b_{2}\hat{u}_{2} \,, p_{1}\hat{\alpha}_{1}+p_{2}\hat{\alpha}_{2}\,\big)& \mapsto & b_{2}p_{1}\hat{\alpha}_{1}+b_{1}p_{2}\hat{\alpha}_{2}\, .\\[4pt]
\end{array}
$\\
\hline
\end{tabular}
\caption{Unital ring structure of the graded endomorphism space $\mathrm{End}^{\bullet}(\sh{F}_{1})$.}
\label{tab:endomorphisms1}
\end{table}

Similarly, the ring structure of the graded endomorphism spaces $\mathrm{End}^{\bullet}(\sh{F}_{2})$ and $\mathrm{End}^{\bullet}(\sh{F}_{3})$ is summarized in Tables~\eqref{tab:endomorphisms2} and~\eqref{tab:endomorphisms3}. In particular, these tables reveal that, compared to $\mathrm{End}^{\bullet}(\sh{F}_{1})$, the graded endomorphism spaces associated with $\sh{F}_{2}$ and $\sh{F}_{3}$ have a slightly simpler ring structure, namely that of the cohomology ring of an annulus.

\begin{table}[ht!]
\centering
\renewcommand{\arraystretch}{1.6} 
\setlength{\tabcolsep}{10pt}
\begin{tabular}{|c|}
\hline
Object: $\sh{F}_{2}$\\
\hline
Graded Endomorphism Spaces\\
\hline
$\mathrm{End}^{\,0}(\sh{F}_{2})=\mathrm{Span}_{\,\mathbb{Z}_{2}}\big\{\,\hat{u}_{1}=\langle 1,1,1 \rangle\,\big\}$  \\
\hline
$\mathrm{End}^{\,1}(\sh{F}_{2})=\mathrm{Span}_{\,\mathbb{Z}_{2}}\big\{\,\hat{\alpha}_{1}=\langle 1,0,0 \rangle\,\big\}$  \\
\hline
Graded Composition\\
\hline
$
\begin{array}{cccc}
\circ: &\mathrm{End}^{\,0}(\sh{F}_{2})\times \mathrm{End}^{\,0}(\sh{F}_{2})&\to& \mathrm{End}^{\,0}(\sh{F}_{2})\, ,\\[-3pt]
&\big(\,b_{1}\hat{u}_{1} \,, a_{1}\hat{u}_{1}\,\big)& \mapsto & b_{1}a_{1} \hat{u}_{1}\, .\\[4pt]
\end{array}
$\\
\hline
$
\begin{array}{cccc}
\circ: &\mathrm{End}^{\,1}(\sh{F}_{2})\times \mathrm{End}^{\,0}(\sh{F}_{2})&\to& \mathrm{End}^{\,1}(\sh{F}_{2})\, ,\\[-3pt]
&\big(\,q_{1}\hat{\alpha}_{1} \,, a_{1}\hat{u}_{1}\,\big)& \mapsto & q_{1}a_{1}\hat{\alpha}_{1}\, .\\[4pt]
\end{array}
$\\
\hline
$
\begin{array}{cccc}
\circ: &\mathrm{End}^{\,0}(\sh{F}_{2})\times \mathrm{End}^{\,1}(\sh{F}_{2})&\to& \mathrm{End}^{\,1}(\sh{F}_{2})\, ,\\[-3pt]
&\big(\,b_{1}\hat{u}_{1} \,, p_{1}\hat{\alpha}_{1}\,\big)& \mapsto & b_{1}p_{1}\hat{\alpha}_{1}\, .\\[4pt]
\end{array}
$\\
\hline
\end{tabular}
\caption{Unital ring structure of the graded endomorphism space $\mathrm{End}^{\bullet}(\sh{F}_{2})$.}
\label{tab:endomorphisms2}
\end{table}

\begin{table}[ht!]
\centering
\renewcommand{\arraystretch}{1.6} 
\setlength{\tabcolsep}{10pt}
\begin{tabular}{|c|}
\hline
Object: $\sh{F}_{3}$\\
\hline
Graded Endomorphism Spaces\\
\hline
$\mathrm{End}^{\,0}(\sh{F}_{3})=\mathrm{Span}_{\,\mathbb{Z}_{2}}\big\{\,\hat{u}_{1}=\langle 1,1,1 \rangle\,\big\}$  \\
\hline
$\mathrm{End}^{\,1}(\sh{F}_{3})=\mathrm{Span}_{\,\mathbb{Z}_{2}}\big\{\,\hat{\alpha}_{1}=\langle 0,0,1 \rangle\,\big\}$  \\
\hline
Graded Composition\\
\hline
$
\begin{array}{cccc}
\circ: &\mathrm{End}^{\,0}(\sh{F}_{3})\times \mathrm{End}^{\,0}(\sh{F}_{3})&\to& \mathrm{End}^{\,0}(\sh{F}_{3})\, ,\\[-3pt]
&\big(\,b_{1}\hat{u}_{1} \,, a_{1}\hat{u}_{1}\,\big)& \mapsto & b_{1}a_{1} \hat{u}_{1}\, .\\[4pt]
\end{array}
$\\
\hline
$
\begin{array}{cccc}
\circ: &\mathrm{End}^{\,1}(\sh{F}_{3})\times \mathrm{End}^{\,0}(\sh{F}_{3})&\to& \mathrm{End}^{\,1}(\sh{F}_{3})\, ,\\[-3pt]
&\big(\,q_{1}\hat{\alpha}_{1} \,, a_{1}\hat{u}_{1}\,\big)& \mapsto & q_{1}a_{1}\hat{\alpha}_{1}\, .\\[4pt]
\end{array}
$\\
\hline
$
\begin{array}{cccc}
\circ: &\mathrm{End}^{\,0}(\sh{F}_{3})\times \mathrm{End}^{\,1}(\sh{F}_{3})&\to& \mathrm{End}^{\,1}(\sh{F}_{3})\, ,\\[-3pt]
&\big(\,b_{1}\hat{u}_{1} \,, p_{1}\hat{\alpha}_{1}\,\big)& \mapsto & b_{1}p_{1}\hat{\alpha}_{1}\, .\\[4pt]
\end{array}
$\\
\hline
\end{tabular}
\caption{Unital ring structure of the graded endomorphism space $\mathrm{End}^{\bullet}(\sh{F}_{3})$.}
\label{tab:endomorphisms3}
\end{table}

Next, we extend our discussion by explicitly analyzing the graded morphism spaces and their compositions for the ordered triple $(\sh{F}_{3}, \sh{F}_{1}, \sh{F}_{2})$, which for clarity is summarized in Table~\eqref{tab:morphisms1}. In this concrete example, the graded composition is largely trivial, with the only non-trivial contributions arising from the degree zero morphism spaces.

\begin{table}[ht!]
\centering
\renewcommand{\arraystretch}{1.6} 
\setlength{\tabcolsep}{10pt}
\begin{tabular}{|c|}
\hline
Ordered Triple: ($\sh{F}_{3}, \sh{F}_{1}, \sh{F}_{2}$)\\
\hline
Graded morphism Spaces\\
\hline
$\begin{array}{c|c}
\mathrm{Ext}^{0}(\sh{F}_{3}, \sh{F}_{1})=\mathrm{Span}_{\,\mathbb{Z}_{2}}\big\{\,\hat{u}_{1}=\langle 0,1,0 \rangle\,\big\}  ~&~ \mathrm{Ext}^{1}(\sh{F}_{3}, \sh{F}_{1})=\mathrm{Span}_{\,\mathbb{Z}_{2}}\big\{\,\hat{\alpha}_{1}=\langle 0,0,1 \rangle\,\big\} \\
\hline
\mathrm{Ext}^{0}(\sh{F}_{1}, \sh{F}_{2})=\mathrm{Span}_{\,\mathbb{Z}_{2}}\big\{\,\hat{v}_{1}=\langle 0,1,0 \rangle\,\big\}  ~&~ \mathrm{Ext}^{1}(\sh{F}_{1}, \sh{F}_{2})=\mathrm{Span}_{\,\mathbb{Z}_{2}}\big\{\,\hat{\beta}_{1}=\langle 1,0,0 \rangle\,\big\} \\
\hline
\mathrm{Ext}^{0}(\sh{F}_{3}, \sh{F}_{2})=\mathrm{Span}_{\,\mathbb{Z}_{2}}\big\{\,\hat{w}_{1}=\langle 0,1,0 \rangle\,\big\}  ~&~ \mathrm{Ext}^{1}(\sh{F}_{3}, \sh{F}_{2})=\mathrm{Span}_{\,\mathbb{Z}_{2}}\big\{\,\hat{\gamma}_{1}=\langle 1,1,1 \rangle\,\big\} \\
\end{array}$\\
\hline
Graded Composition\\
\hline
$
\begin{array}{cccc}
\circ: &\mathrm{Ext}^{0}(\sh{F}_{1},\sh{F}_{2})\times \mathrm{Ext}^{0}(\sh{F}_{3},\sh{F}_{1})&\to& \mathrm{Ext}^{0}(\sh{F}_{3},\sh{F}_{2})\, ,\\[-3pt]
&\big(\,b_{1}\hat{v}_{1} \,, a_{1}\hat{u}_{1}\,\big)& \mapsto & b_{1}a_{1} \hat{w}_{1}\, .\\[4pt]
\end{array}
$\\
\hline
$
\begin{array}{cccc}
\circ: &\mathrm{Ext}^{1}(\sh{F}_{1},\sh{F}_{2})\times \mathrm{Ext}^{0}(\sh{F}_{3},\sh{F}_{1})&\to& \mathrm{Ext}^{1}(\sh{F}_{3},\sh{F}_{2})\, ,\\[-3pt]
&\big(\,q_{1}\hat{\beta}_{1} \,, a_{1}\hat{u}_{1}\,\big)& \mapsto & 0\, .\\[4pt]
\end{array}
$\\
\hline
$
\begin{array}{cccc}
\circ: &\mathrm{Ext}^{0}(\sh{F}_{1},\sh{F}_{2})\times \mathrm{Ext}^{1}(\sh{F}_{3},\sh{F}_{1})&\to& \mathrm{Ext}^{1}(\sh{F}_{3},\sh{F}_{2})\, ,\\[-3pt]
&\big(\,b_{1}\hat{v}_{1} \,, p_{1}\hat{\alpha}_{1}\,\big)& \mapsto & 0\, .\\[4pt]
\end{array}
$\\
\hline
\end{tabular}
\caption{Graded morphism spaces and their compositions for the ordered triple $(\sh{F}_{3}, \sh{F}_{1}, \sh{F}_{2})$.}
\label{tab:morphisms1}
\end{table}

Finally, we close our discussion by summarizing in Table~\eqref{tab:morphisms} the linear algebraic structure of the graded morphism spaces associated with $\sh{F}_{1}$, $\sh{F}_{2}$, $\sh{F}_{3}$. From this table, we observe that $\sh{F}_{1}$ stands out, as its graded morphism spaces exhibit a richer structure compared to the others.
 
\begin{table}[ht!]
\centering
\renewcommand{\arraystretch}{1.6} 
\setlength{\tabcolsep}{10pt}
\begin{tabular}{|c|c|c|c|}
\hline
$\mathrm{Ext}^\bullet(\sh{F}_{i},\sh{F}_{j})$ & $\sh{F}_1$ & $\sh{F}_2$ & $\sh{F}_3$ \\
\hline
$\sh{F}_1$ & $\mathbb{Z}^{2}_{2}[0]\oplus \mathbb{Z}^{2}_{2}[1]$ & $\mathbb{Z}_{2}[0]\oplus \mathbb{Z}_{2}[1]$ & $\mathbb{Z}_{2}[0]\oplus \mathbb{Z}_{2}[1]$  \\
\hline
$\sh{F}_2$ & $\mathbb{Z}_{2}[0]\oplus \mathbb{Z}_{2}[1]$  & $\mathbb{Z}_{2}[0]\oplus \mathbb{Z}_{2}[1]$  & $\mathbb{Z}_{2}[0]\oplus \mathbb{Z}_{2}[1]$  \\
\hline
$\sh{F}_3$ & $\mathbb{Z}_{2}[0]\oplus \mathbb{Z}_{2}[1]$  & $\mathbb{Z}_{2}[0]\oplus \mathbb{Z}_{2}[1]$  & $\mathbb{Z}_{2}[0]\oplus \mathbb{Z}_{2}[1]$  \\
\hline
\end{tabular}
\caption{Linear structure of the graded morphism spaces associated with $\sh{F}_{1}$, $\sh{F}_{2}$, $\sh{F}_{3}$.} 
\label{tab:morphisms}
\end{table}

%% file: appendix.tex
\section{Auxiliary Results on Braid Matrices}\label{appendix:a}
\noindent
In this appendix, we record several auxiliary results concerning the braid matrices (see Definition~\eqref{Def: braid matrices}). Let $n\geq2$ be an integer. Given $M\in \mathrm{M}(n,\mathbb{K})$, $k\in[1, n-1]$, and $z\in \mathbb{K}$, the following claims provide explicit formulas for the products $MB^{(n)}_{k}(z)$, $B^{(n)}_{k}(z)M$, $M(B^{(n)}_{k}(z))^{-1}$, and $(B^{(n)}_{k}(z))^{-1}M$. 

\begin{claim}\label{braid matrix product from right}
Let $n\geq2$ be an integer, and let $M\in \mathrm{M}(n,\mathbb{K})$ be a matrix given by
\begin{equation*}
M=\big[\vec{c}_{1},\dots, \vec{c}_{n} \big]\, ,
\end{equation*}
where $\vec{c}_{i}$ denotes the $i$-th column vector of $M$ for each $i\in[1, n]$. Then, for any $k\in[1, n-1]$ and $z\in \mathbb{K}$, we have that 
\begin{equation*}
MB^{(n)}_{k}(z)=\big[\vec{c}_{1},\dots, \vec{c}_{k-1}, \vec{c}_{k+1}+\vec{c}_{k}\cdot z,\vec{c}_{k}, \vec{c}_{k+2},\dots, \vec{c}_{n} \big]\,.
\end{equation*}
\end{claim}
\begin{proof}
Multiplying the matrix $M$ on the right by the braid matrix $B^{(n)}_k(z)$ only changes the $k$-th and $(k+1)$-th columns of $M$. Specifically,
\begin{equation*}
\vec{c}_k \mapsto \vec{c}_k\cdot z + \vec{c}_{k+1}, \qquad 
\vec{c}_{k+1} \mapsto \vec{c}_k,    
\end{equation*}
while all other columns of $M$ remain unchanged. The claim follows.
\end{proof}

\begin{claim}\label{braid matrix product from left}
Let $n\geq2$ be an integer, and let $M\in \mathrm{M}(n,\mathbb{K})$ be a matrix given by
\begin{equation*}
M=\left[ \begin{array}{c}
     \vec{r}_{1} \\
     \vdots\\
     \vec{r}_{n}
\end{array}\right]\, ,    
\end{equation*}
where $\vec{r}_{i}$ denotes the $i$-th row vector of $M$ for each $i\in[1, n]$. Then, for any $k\in[1, n-1]$ and $z\in \mathbb{K}$, we have that 
\begin{equation*}
B^{(n)}_{k}(z)M= \left[ \begin{array}{c}
     \vec{r}_{1}  \\
     \vdots\\
     \vec{r}_{k-1}\\
     \vec{r}_{k+1}+z\cdot \vec{r}_{k}\\
     \vec{r}_{k}\\
     \vec{r}_{k+2}\\
     \vdots\\
     \vec{r}_{n}
\end{array}    \right]\, .
\end{equation*}  
\end{claim}
\begin{proof}
Multiplying the matrix $M$ on the left by the braid matrix $B^{(n)}_k(z)$ only changes the $k$-th and $(k+1)$-th rows of $M$. Explicitly,
\begin{equation*}
\vec{r}_k \mapsto \vec{r}_{k+1} + z\cdot \vec{r}_k, \qquad 
\vec{r}_{k+1} \mapsto \vec{r}_k,    
\end{equation*}
while all other rows of $M$ remain unchanged. The claim follows.
\end{proof}

\begin{claim}\label{inverse braid matrix product from right}
Let $n\geq2$ be an integer, and let $M\in\mathrm{M}(n,\mathbb{K})$ be a matrix given by 
\begin{equation*}
M=\big[\vec{c}_{1},\dots, \vec{c}_{n} \big]\, ,
\end{equation*}
where $\vec{c}_{i}$ denotes the $i$-th column vector of $M$ for each $i\in[1, n]$. Then, for any $k\in[1, n-1]$ and $z\in\mathbb{K}$, we have that
\begin{equation*}
M(B^{(n)}_{k}(z))^{-1}=\big[\vec{c}_{1},\dots, \vec{c}_{k-1}, \vec{c}_{k+1},\vec{c}_{k}-\vec{c}_{k+1}\cdot z, \vec{c}_{k+2},\dots, \vec{c}_{n} \big]\,.
\end{equation*}
\end{claim}
\begin{proof}
Multiplying the matrix $M$ on the right by the inverse braid matrix $\big(B^{(n)}_k(z)\big)^{-1}$ only changes the $k$-th and $(k+1)$-th columns of $M$. Specifically,
\begin{equation*}
\vec{c}_k \mapsto \vec{c}_{k+1}, \qquad 
\vec{c}_{k+1} \mapsto \vec{c}_{k} - \vec{c}_{k+1}\cdot z,    
\end{equation*}
while all other columns of $M$ remain unchanged. The claim follows.
\end{proof}

\begin{claim}\label{inverse braid matrix product from left}
Let $n\geq2$ be an integer, and let $M\in \mathrm{M}(n,\mathbb{K})$ be a matrix given by
\begin{equation*}
M=\left[ \begin{array}{c}
     \vec{r}_{1} \\
     \vdots\\
     \vec{r}_{n}
\end{array}\right]\, ,    
\end{equation*}
where $\vec{r}_{i}$ denotes the $i$-th row vector of $M$ for each $i\in[1, n]$. Then, for any $k\in[1, n-1]$ and $z\in \mathbb{K}$, we have that
\begin{equation*}
(B^{(n)}_{k}(z))^{-1}M= \left[ \begin{array}{c}
     \vec{r}_{1}  \\
     \vdots\\
     \vec{r}_{k-1}\\
     \vec{r}_{k+1}\\
     \vec{r}_{k}-z\cdot \vec{r}_{k+1}\\
     \vec{r}_{k+2}\\
     \vdots\\
     \vec{r}_{n}
\end{array}    \right]\, .
\end{equation*}  
\end{claim}
\begin{proof}
Multiplying the matrix $M$ on the left by the inverse braid matrix $\big(B^{(n)}_k(z)\big)^{-1}$ only changes the $k$-th and $(k+1)$-th rows of $M$. Explicitly,
\begin{equation*}
\vec{r}_k \mapsto \vec{r}_{k+1}, \qquad 
\vec{r}_{k+1} \mapsto \vec{r}_{k} - z\cdot \vec{r}_{k+1},    
\end{equation*}
while all other rows of $M$ remain unchanged. The claim follows.
\end{proof}

We conclude this appendix with the following lemma, which describes the effect of a generalized conjugation of a diagonal matrix by braid matrices.

\begin{lemma}\label{lemma for braid matrices and diagonal matrices}
Let $\vec{u}=(u_{1}, \dots, u_{n})\in \mathbb{K}^{n}_{\mathrm{std}}$ be a fixed tuple. Then, for any $k\in[1, n-1]$ and any $z,\,z'\in \mathbb{K}$, we have that

\begin{equation*}
\big(B^{(n)}_{k}(z')\big)^{-1} D(\vec{u})\; B^{(n)}_{k}(z)
=
\begin{bmatrix}
u_1 & & & & & & & \\
& \ddots & & & & & & \\
& & u_{k-1} & & & & & \\
& & & u_{k+1} & ~~0~~ & & & \\
& & & u_{k}\,z - z'\,u_{k+1} & ~~~u_k~~ & & & \\
& & & & & u_{k+2} & & \\
& & & & & & \ddots & \\
& & & & & & & u_n
\end{bmatrix}.
\end{equation*}
\end{lemma}
\begin{proof}
The result follows from a straightforward application of claims \eqref{braid matrix product from right} and \eqref{inverse braid matrix product from left}. 
\end{proof}